\documentclass[12pt]{amsart}

\usepackage{amsmath,amssymb,amsthm,graphicx,hyperref,pinlabel}
\usepackage{enumitem}
\usepackage{tikz}
\usetikzlibrary{arrows}
\usepackage[margin=1.25in]{geometry}

\numberwithin{equation}{section}
\usepackage{hyperref}
\hypersetup{bookmarksdepth=3}

\theoremstyle{plain}
\newtheorem{theorem}[equation]{Theorem}

\newtheorem{proposition}[equation]{Proposition}
\newtheorem{lemma}[equation]{Lemma}

\newtheorem{corollary}[equation]{Corollary}

\newtheorem{claim}{Claim}[subsection]

\newtheorem*{claim*}{Claim}

\newtheorem*{addendum-Thm15}{Addendum to Theorem~\ref{Thm_comparing_RF_weak_form}}

\theoremstyle{remark}

\theoremstyle{definition}
\newtheorem{definition}[equation]{Definition}

\newcommand{\si}{\sigma}
\newcommand{\wh}{\widehat}

\newcommand{\C}{{\mathcal C}}

\newcommand{\D}{\partial}
\newcommand{\DD}{\mathcal{D}}
\newcommand{\cald}{\mathcal{D}}

\newcommand{\M}{{\mathcal M}}
\newcommand{\N}{{\mathcal{N}}}

\newcommand{\R}{\mathbb R}

\renewcommand{\S}{{\mathcal S}}
\renewcommand{\t}{\mathfrak{t}}

\newcommand{\Z}{\mathbb Z}

\newcommand{\lb}{\linebreak[1]}

\newcommand{\al}{\alpha}

\newcommand{\ben}{\begin{enumerate}}

\newcommand{\de}{\delta}
\newcommand{\diam}{\operatorname{diam}}

\newcommand{\een}{\end{enumerate}}

\newcommand{\eps}{\epsilon}
\newcommand{\ga}{\gamma}

\newcommand{\id}{\operatorname{id}}
\newcommand{\im}{\operatorname{Im}}
\newcommand{\injrad}{\operatorname{InjRad}}
\newcommand{\la}{\lambda}

\newcommand{\lra}{\longrightarrow}

\newcommand{\ol}{\overline}

\newcommand{\ra}{\rightarrow}

\newcommand{\Rm}{\operatorname{Rm}}

\newcommand{\ul}{\underline}

\DeclareMathOperator{\Bry}{Bry}
\DeclareMathOperator{\CAP}{cap}

\DeclareMathOperator{\Image}{Im}
\DeclareMathOperator{\Int}{Int}
\DeclareMathOperator{\can}{can}
\DeclareMathOperator{\lin}{lin}

\DeclareMathOperator{\cut}{cut}
\DeclareMathOperator{\Cut}{Cut}
\DeclareMathOperator{\comp}{comp}
\DeclareMathOperator{\Sym}{Sym}
\DeclareMathOperator{\Ric}{Ric}
\DeclareMathOperator{\tr}{tr}
\DeclareMathOperator{\End}{End}
\DeclareMathOperator{\nn}{n}
\DeclareMathOperator{\bb}{b}

\DeclareMathOperator{\length}{length}

\DeclareMathOperator{\sd}{SD}
\DeclareMathOperator{\proj}{proj}

\newcommand{\ov}[1]{\overline{#1}}
\newcommand{\td}[1]{\widetilde{#1}}

\def\XXint#1#2#3{{\setbox0=\hbox{$#1{#2#3}{\int}$}
     \vcenter{\hbox{$#2#3$}}\kern-.5\wd0}}

\begin{document}

\author{Richard H. Bamler}
\address{Department of Mathematics, University of California, Berkeley, Berkeley, CA 94720}
\email{rbamler@berkeley.edu}
\author{Bruce Kleiner}
\address{Courant Institute of Mathematical Sciences, New York University,  251 Mercer St., New York, NY 10012}
\email{bkleiner@cims.nyu.edu}
\thanks{The first author was supported by a Sloan Research Fellowship and NSF grant DMS-1611906.
The research was in part conducted while the first author was in residence at the Mathematical Sciences Research Institute in Berkeley, California, during the Spring 2016 semester. \\
\hspace*{2.38mm} The second author was supported by NSF grants DMS-1405899 and DMS-1406394, and a Simons Collaboration grant.}

\title[Uniqueness and stability of Ricci flow]{Uniqueness and stability of Ricci flow through singularities}

\date{\today}

\begin{abstract}
We verify a conjecture of Perelman, which states that there exists a canonical Ricci flow through singularities starting from an arbitrary compact Riemannian $3$-manifold.  
Our main result is a uniqueness theorem for such flows, which, together with an earlier existence theorem of Lott and the second named author, implies Perelman's conjecture.  We also show that this flow through singularities depends continuously on its initial condition and that it may be obtained as a limit of Ricci flows with surgery.  

Our results have applications to the study of diffeomorphism groups of three manifolds --- in particular to the Generalized Smale Conjecture --- which will appear in a subsequent paper.

\end{abstract}

\maketitle

\tableofcontents

\section{Introduction}
\subsection{Overview} \label{subsec_motivation}

The understanding of many aspects of Ricci flow has advanced dramatically in the last 15 years.  
This has led to numerous applications, the most notable being Perelman's landmark proof of the Geometrization and Poincar\'e Conjectures.  Nonetheless, from an analytical viewpoint, a number of fundamental questions remain, even for $3$-dimensional Ricci flow.  One of these concerns the nature of Ricci flow with surgery, a modification of Ricci flow that was central to Perelman's proof.  Surgery, an idea initially  developed by Hamilton, removes singularities as they form, allowing one to continue the flow.  While Perelman's construction of Ricci flow with surgery was spectacularly successful, it is not entirely satisfying due to its ad hoc character and the fact that it depends on a number of non-canonical choices.  Furthermore, from a PDE viewpoint, Ricci flow with surgery does not provide a theory of solutions to the Ricci flow PDE  itself, since surgery violates the equation.  In fact, Perelman himself was aware of these drawbacks and drew attention to them in both of his Ricci flow preprints:

\begin{quote} \it
``It is likely that by passing to the limit in this construction [of Ricci flow with surgery] one would get a canonically defined Ricci flow through singularities, but at the moment I don't have a proof of that.'' --- \cite[p.37]{Perelman:2002um}
\end{quote}
\begin{quote} \it
``Our approach \ldots is aimed at eventually constructing a canonical Ricci flow \ldots a goal, that has not been achieved yet in the present work.'' --- \cite[p.1]{Perelman:2003tka}
\end{quote}
Motivated by the above, the paper \cite{Kleiner:2014wz} introduced a new notion of weak (or generalized) solutions to Ricci flow in dimension $3$ and proved the existence within this class of solutions for arbitrary initial data, as well as a number of results about their geometric and analytical properties.

In this paper we show that the weak solutions of \cite{Kleiner:2014wz} are uniquely determined by their initial data (see Theorem~\ref{thm_uniquness} below).  
In combination with  \cite{Kleiner:2014wz}, this implies that the associated initial value problem has a \emph{canonical} weak solution, thereby proving Perelman's conjecture (see Corollary~\ref{cor_existence_uniqueness}).
We also show that this weak solution depends continuously on its initial data, and that it is a limit of Ricci flows with surgery (see Corollary~\ref{cor_convergence_corollary}).
In summary, our results provide an answer to the long-standing problem of finding a satisfactory theory of weak solutions to the Ricci flow equation in the $3$-dimensional case.

From a broader perspective, it is interesting to compare the results in this paper  with work on weak solutions to other geometric PDEs.  

The theory of existence and partial regularity of such weak solutions has been studied extensively. As with PDEs in general, proving existence of solutions requires a choice of objects and a topology that is strong enough to respect the equation, but weak enough to satisfy certain compactness properties. Establishing the finer structure of solutions (e.g. partial regularity) requires, generally speaking, a mechanism for restricting  blow-ups.  For  minimal surfaces, harmonic maps and harmonic map heat flow, good notions of weak solutions with accompanying existence and partial regularity theorems were developed   long ago  \cite{almgren_interior_regularity,simons_minimal_varieties,schoen_uhlenbeck_regularity_harmonic_maps,chen_struwe_hmhf_partial_regularity}.  By contrast, the theory of weak solutions to mean curvature flow,  the Einstein equation and Ricci flow, are at earlier stages of development.  
For mean curvature flow, for instance, different approaches to weak solutions  (e.g. (enhanced) Brakke flows and level set flow) were introduced over the last 40 years  \cite{brakke_motion,evans_spruck,chen_giga_goto,ilmanen_elliptic_regularization}.  
Yet, in spite of deep results for the cases of mean convex or generic initial conditions \cite{white_size,white_nature,white_regularity,colding_minicozzi_generic}, to our knowledge, the best results known for flows starting from a general compact smooth surface in $\R^3$ are essentially those of \cite{brakke_motion}, which are presumably far from optimal.  
For the (Riemannian) Einstein equation many results have been obtained in the K\"ahler case and on limits of smooth Einstein manifolds, but otherwise progress toward even a viable definition of weak solutions has been rather limited.  
 Progress on Ricci flow has been limited to the study of specific models for an isolated singularity \cite{feldman_ilmanen_knopf,angenent_knopf_precise,angenent_caputo_knopf_minimally} and the K\"ahler case,  which has advanced rapidly in the last 10 years after the appearance of  \cite{song_tian}.

Regarding uniqueness of weak solutions, our focus in this paper, much less is known.    The paper \cite{ilmanen_lectures} describes a mechanism for non-uniqueness, stemming from the dynamical instability of cones, which is applicable to a number of geometric flows.
For example, for mean curvature flow of hypersurfaces in $\R^n$ this mechanism provides examples of non-uniqueness in high dimensions.   
Ilmanen and White \cite{white_icm} found  examples of non-uniqueness starting from compact smooth surfaces in $\R^3$.  Examples for harmonic map heat flow are constructed in \cite{germain_rupflin,germain_ghoul_miura}, and for Ricci flow in higher dimensions there are examples in \cite{feldman_ilmanen_knopf}, which suggest non-uniqueness.   Since any discussion of uniqueness must refer to a particular class of admissible solutions, the interpretation of some of the above examples is not entirely clear, especially in the case of higher dimensional Ricci flow, where a definition of weak solutions is lacking.   
In the other direction, uniqueness has been proven to hold in only a few cases:  harmonic map heat flow with $2$-dimensional domain \cite{struwe_hmhf_surfaces}, mean convex mean curvature flow \cite{white_nature} and K\"ahler-Ricci flow \cite{song_tian,eyssidieux_guedj_zeriahi}.  The proofs of these theorems rely on special features of these flows.  In \cite{struwe_hmhf_surfaces}, the flow develops singularities only at a finite set of times, and at isolated points.  The striking  proof of uniqueness in \cite{white_nature}  is based on comparison techniques for scalar equations and a geometric monotonicity property specific to mean convex flow (see also the recent paper \cite{Hershkovits_White}, which localizes the mean convexity assumption).  
Lastly, K\"ahler-Ricci flow has many remarkable features that play a crucial role in its uniqueness argument: the singularities, whose form is quite rigid, arise at a finite set of times determined by the evolution of the K\"ahler class; also, techniques specific to scalar equations play an important role.     

The method of proving uniqueness used in this paper is completely different in spirit from earlier work.  
Uniqueness is deduced by comparing two flows with nearby initial condition and estimating the rate at which they diverge from one another.  
Due to the nature of the singularities, which might in principle occur at a Cantor set of times, the flows can only be compared after the removal of their almost singular regions.  
Since one knows nothing about the correlation between the almost singular parts of the two flows, the crux of the proof is to control the  influence of effects emanating from the boundary of the truncated flows.  
This control implies a strong stability property, which roughly speaking states that both flows are close away from their almost singular parts if they are sufficiently close initially.
A surprising consequence of our analysis is that this strong stability result applies not just to Ricci flows with surgery and the weak solutions of \cite{Kleiner:2014wz}, but to flows whose almost singular parts are allowed to evolve in an arbitrary fashion, possibly violating the Ricci flow equation at small scales.

The main ideas of our proof may throw light on uniqueness problems in general.  
When distilled down to its essentials, our proof is based on the following  ingredients:
\ben
\item A structure theory for the almost singular part of the flow, which is based on a classification of all blow-ups, not just shrinking solitons.
\item Uniform strict stability for solutions to the linearized equation,  for all blow-ups.
\item An additional quantitative rigidity property for blow-ups that makes it possible to fill in missing data to the evolution problem, after recently resolved singularities.
\een 
This list, which is not specific to Ricci flow, suggests a tentative criterion for when one might expect, and possibly prove, uniqueness for weak solutions to a given geometric flow.
From a philosophical point of view, it is natural to expect (1) and (2) to be necessary conditions for uniqueness.  
However, implementation of even (1) can be quite difficult. 
 Indeed,  to date there are few situations where such a classification is known.  
 It turns out that (3) is by far the most delicate part of  the proof in our setting and it is responsible for much of the complexity in the argument (see the overview of the proof in Section~\ref{sec_overview} for more discussion of this point).   
Another context where the above criteria may be satisfied is  the case of mean curvature flow of $2$-spheres in $\R^3$, where uniqueness is conjectured to hold \cite{white_icm}.

We mention that our main result implies that weak solutions to Ricci flow behave well even when one considers continuous families of initial conditions.   
This continuous dependence leads to new results for diffeomorphism groups of $3$-manifolds, in particular for the Generalized Smale Conjecture, which will be discussed elsewhere \cite{BK2018}.

\subsection{Background and setup} \label{subsec_status_quo}
In preparation  for the statements of our main results, which will be presented in the next subsection, we now recall in greater detail some facts about Perelman's Ricci flow with surgery \cite{Perelman:2003tka,Kleiner:2008fh,morgan_tian_pc,besson_et_al} and the weak solutions from \cite{Kleiner:2014wz}, which will be needed for our setup.
As these constructions are generally very technical, we will continue in a relatively informal style.
The reader who is already familiar with this material may skip this subsection and proceed to the presentation of the main results in Subsection~\ref{subsec_statement_main_results}.

In his seminal paper \cite{Hamilton:1982ts}, Hamilton introduced the
Ricci flow equation 
\[ \partial_t g(t) = - 2 \Ric (g(t)), \qquad g(0) = g_0 \]
and showed that any Riemannian metric $g_0$ on a compact manifold can be evolved into a unique solution $(g(t))_{t \in [0,T)}$. 
This solution may, however, develop a singularity in finite time. In \cite{Perelman:2002um}, Perelman analyzed such finite-time singularities in the 3-di\-men\-sion\-al case and showed that those are essentially caused by two behaviors: 
\begin{itemize}
\item \emph{Extinction} (e.g. the flow becomes asymptotic to a shrinking round sphere). 
\item The development of \emph{neck pinches} (i.e. there are regions of the manifold that become more and more cylindrical, $\approx S^2 \times \R$, modulo rescaling, while the diameter of the cross-sectional 2-sphere shrinks to zero).
\end{itemize}
Based on this knowledge, and inspired by a program suggested by Hamilton, Perelman specified a surgery process in which the manifold is cut open along small cross-sectional 2-spheres,  the high  curvature part of the manifold and extinct components are removed, and the resulting spherical boundary components are filled in with $3$-disks endowed with a standard cap metric.  
This produces a new smooth metric on a closed manifold, from which the Ricci flow can be restarted. The process may then be iterated to yield a \emph{Ricci flow with surgery}.  
More specifically, a Ricci flow with surgery is a sequence of conventional Ricci flows $(g_1(t))_{t \in [0,T_1]}, (g_2(t))_{t \in [T_1,T_2]}, (g_3(t))_{t \in [T_2,T_3]}, \ldots$ on compact manifolds $M_1, M_2, M_3, \ldots$, where $(M_{i+1}, g_{i+1} (T_{i}))$ arises from $(M_i, g_i (T_i))$ by a surgery process, as described before.

\begin{figure}
\labellist
\small\hair 2pt
\pinlabel {surgery} at 460 295
\endlabellist
\centering
\includegraphics[width=145mm]{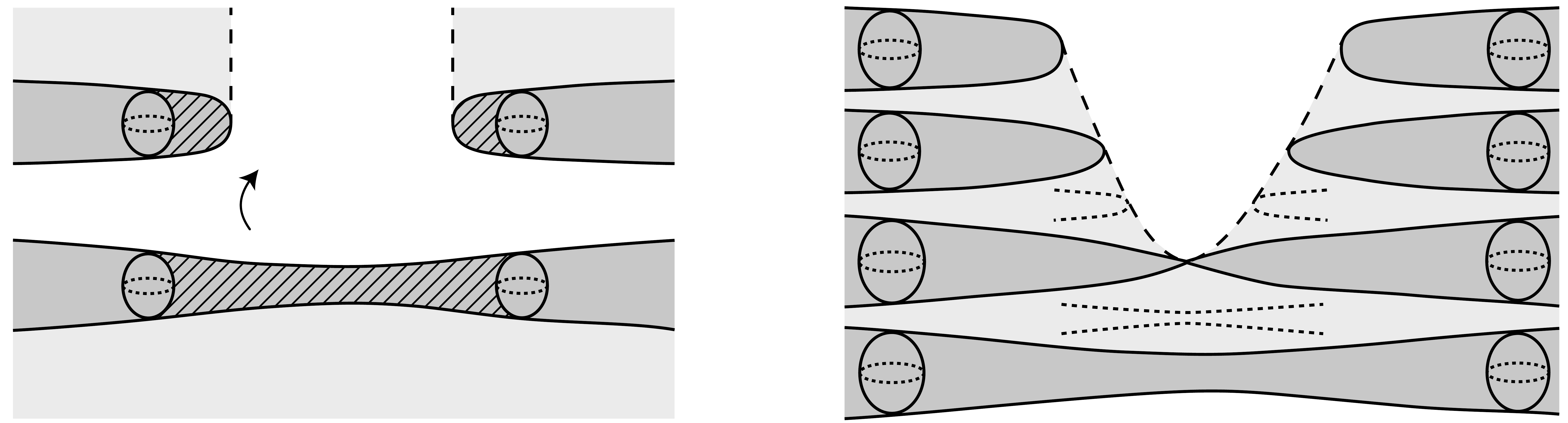}
\caption{In a Ricci flow with surgery (left figure) surgeries are performed at a positive scale, whereas a singular Ricci flow (right figure) ``flows through'' a singularity at an infinitesimal scale.
The hatched regions in the left figure mark the surgery points, i.e. the points that are removed or added during a surgery.
\label{fig_RF_surg_vs_spacetime}}
\end{figure}
As mentioned in Subsection~\ref{subsec_motivation}, the construction of a Ricci flow with surgery depends on a variety of auxiliary parameters, for which there does not seem to be a canonical choice, such as:
\begin{itemize}
\item The scale of the cross-sectional 2-sphere along which a neck pinch sin\-gu\-lar\-i\-ty is excised; this scale is often called the \emph{surgery scale}.
\item The precise position and number of these 2-spheres.
\item The standard cap metric that is placed on the 3-disks which are glued into the 2-sphere boundary components.
\item The method used to interpolate between this metric and the metric on the nearby necks.
\end{itemize}
Different choices of these parameters may influence the future development of the flow significantly (as well as the space of future surgery parameters).
Hence a Ricci flow with surgery cannot be constructed in a \emph{canonical} way or, in other words, a Ricci flow with surgery is not \emph{uniquely} determined by its initial metric.

It is therefore a natural question whether a Ricci flow with surgery can be replaced by a more canonical object, which one may hope is uniquely determined by its initial data.
This question was first addressed in \cite{Kleiner:2014wz}, where the notion of a \emph{singular Ricci flow}, a kind of weak solution to the Ricci flow equation, was introduced.  
In these flows, surgeries have been replaced by singular structure, i.e. regions with unbounded curvature, which may be thought of as ``surgery at an infinitesimal scale'' (see Figure~\ref{fig_RF_surg_vs_spacetime}).

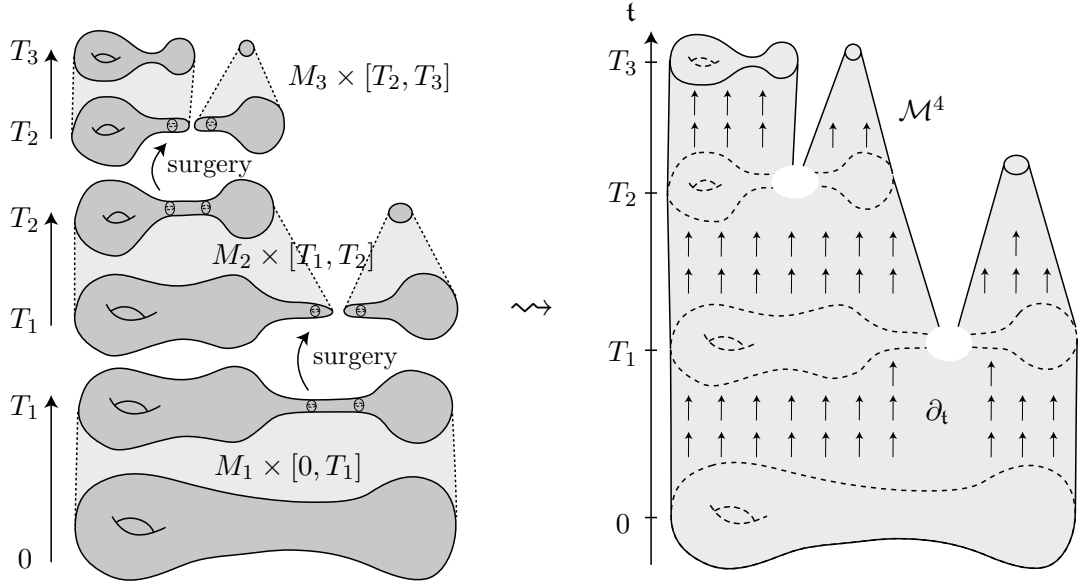
\begin{figure}
\labellist
\small\hair 2pt
\pinlabel $0$ at 50 60
\pinlabel $T_1$ at 50 360
\pinlabel $T_1$ at 50 530
\pinlabel $T_2$ at 50 730
\pinlabel $T_2$ at 50 900
\pinlabel $T_3$ at 50 1060
\pinlabel {$M_1 \times [0,T_1]$} at 570 240
\pinlabel {$M_2 \times [T_1,T_2]$} at 580 657
\pinlabel {$M_3 \times [T_2,T_3]$} at 730 1010
\pinlabel {\smaller surgery} at 702 450
\pinlabel {\smaller surgery} at 415 830
\pinlabel $\partial_\t$ at 1850 350
\pinlabel $0$ at 1230 130
\pinlabel $T_1$ at 1230 465
\pinlabel $T_2$ at 1230 780
\pinlabel $T_3$ at 1230 1040
\pinlabel $\t$ at 1245 1140
\pinlabel {\Large $\rightsquigarrow$} at 1050 550
\pinlabel $\M^4$ at 1820 950
\endlabellist
\centering
\includegraphics[width=145mm]{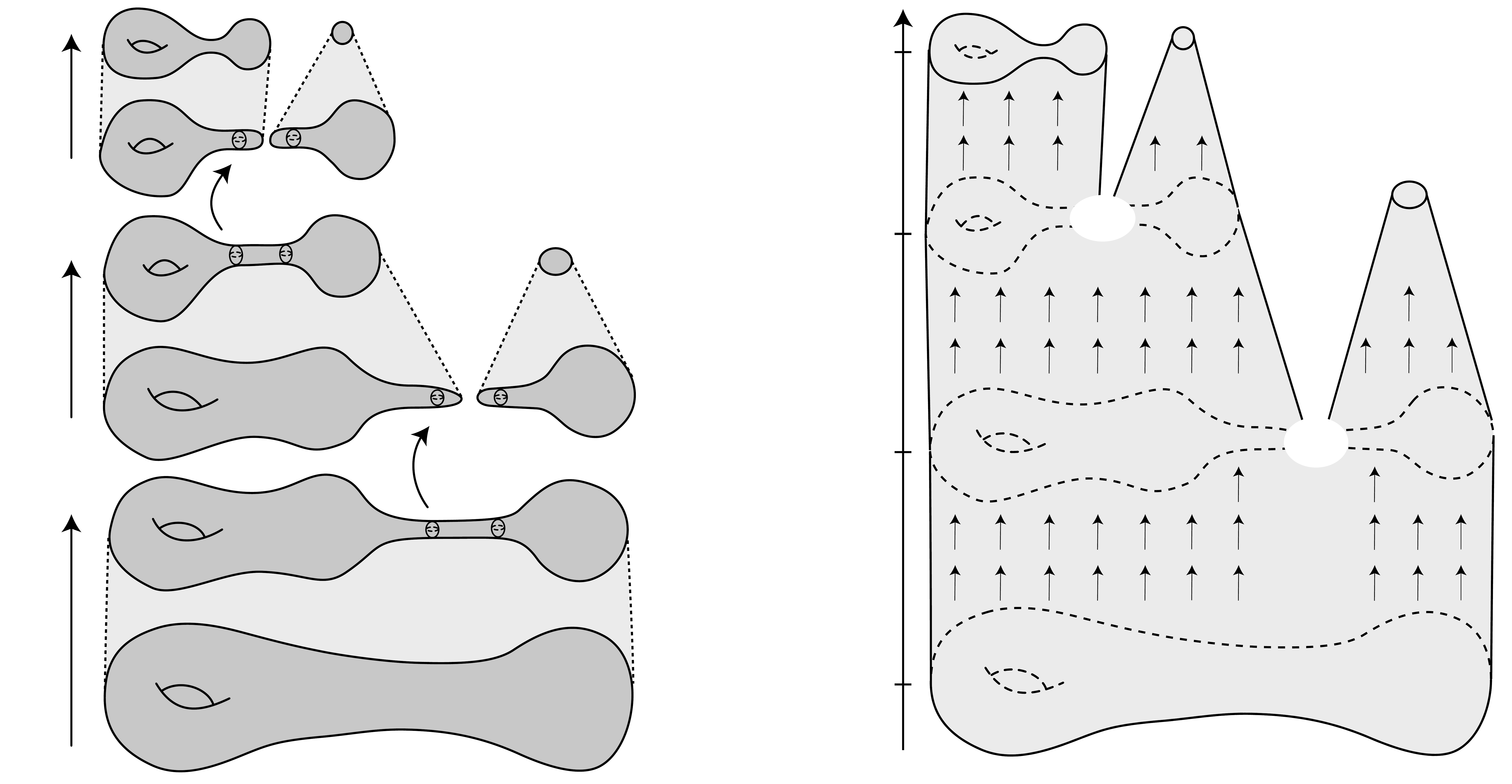}
\caption{A Ricci flow with surgery (left figure) can be converted to a Ricci flow spacetime (right figure) by identifying pre and post-surgery time-slices and removing surgery points.
The white circles in the right figure indicate that surgery points were removed at times $T_1$ and $T_2$.\label{fig_spacetime_construction}}
\end{figure}
In order to present the definition and summarize the construction of a singular Ricci flow, we need to introduce the spacetime picture of a Ricci flow or a Ricci flow with surgery.
For this purpose, consider a Ricci flow with surgery consisting of the conventional Ricci flows $(M_1, (g_1(t))_{t \in [0,T_1]})$, $(M_2, (g_2(t))_{t \in [T_1,T_2]}), \ldots$ and form the following 4-dimensional spacetime manifold (see Figure~\ref{fig_spacetime_construction} for an illustration):
\begin{equation} \label{eq_MM_construction_from_RFws}
 \M := \big( M_1 \times [0,T_1] \cup_{\phi_1} M_2 \times [T_1, T_2] \cup_{\phi_2} M_3 \times [T_2, T_3] \cup_{\phi_3} \ldots \big) \setminus \mathcal{S} 
\end{equation}
Here $\mathcal{S}$ denotes the set of surgery points, i.e. the set of points that are removed or added during a surgery step and $\phi_i : M_i \supset U_i \to U_{i+1} \subset M_{i+1}$ are isometric gluing maps, which are defined on the complement of the surgery points in $M_i \times \{T_i\}$ and $M_{i+1} \times \{ T_{i} \}$.
The above construction induces a natural \emph{time-function} $\t : \M \to [0,\infty)$, whose level-sets are called \emph{time-slices}, as well as a \emph{time-vector field} $\partial_{\t}$ on $\M$ with $\partial_{\t} \cdot \t = 1$.
The Ricci flows $(g_1(t))_{t \in [0,T_1]}, (g_2 (t))_{t \in [T_1,T_2]}, \ldots$ induce a metric $g$ on the horizontal distribution $\{ d \t = 0 \} \subset T\M$, which satisfies the Ricci flow equation
\[ \mathcal{L}_{\partial_{\t}} g = - 2 \Ric (g). \]
The tuple $(\M, \t, \partial_\t, g)$ is called a \emph{Ricci flow spacetime} (see Definition~\ref{def_RF_spacetime} for further details).
We will often abbreviate this tuple by $\M$.

Note that a Ricci flow spacetime $\M$ that is constructed from a Ricci flow with surgery by the procedure above is incomplete (see Definition~\ref{def_completeness} for more details).
More specifically, the time-slices corresponding to surgery times are incomplete Riemannian manifolds, because surgery points, consisting of necks near neck pinches or standard caps are not included in $\M$.
So these time-slices have ``holes'' whose ``diameters'' are $\leq C \delta$, where $\delta$ is the surgery scale and $C$ is a universal constant.
A Ricci flow with this property is called \emph{$C\delta$-complete} (see again Definition~\ref{def_completeness} for further details).

\begin{figure}
\labellist
\small\hair 2pt
\pinlabel $0$ at 430 80
\pinlabel {$\M^4$} at 1260 1005
\pinlabel $\t$ at 420 1050
\pinlabel $\partial_\t$ at 1350 340
\pinlabel $(M,g_0)$ at 1200 165
\endlabellist
\centering
\includegraphics[width=145mm]{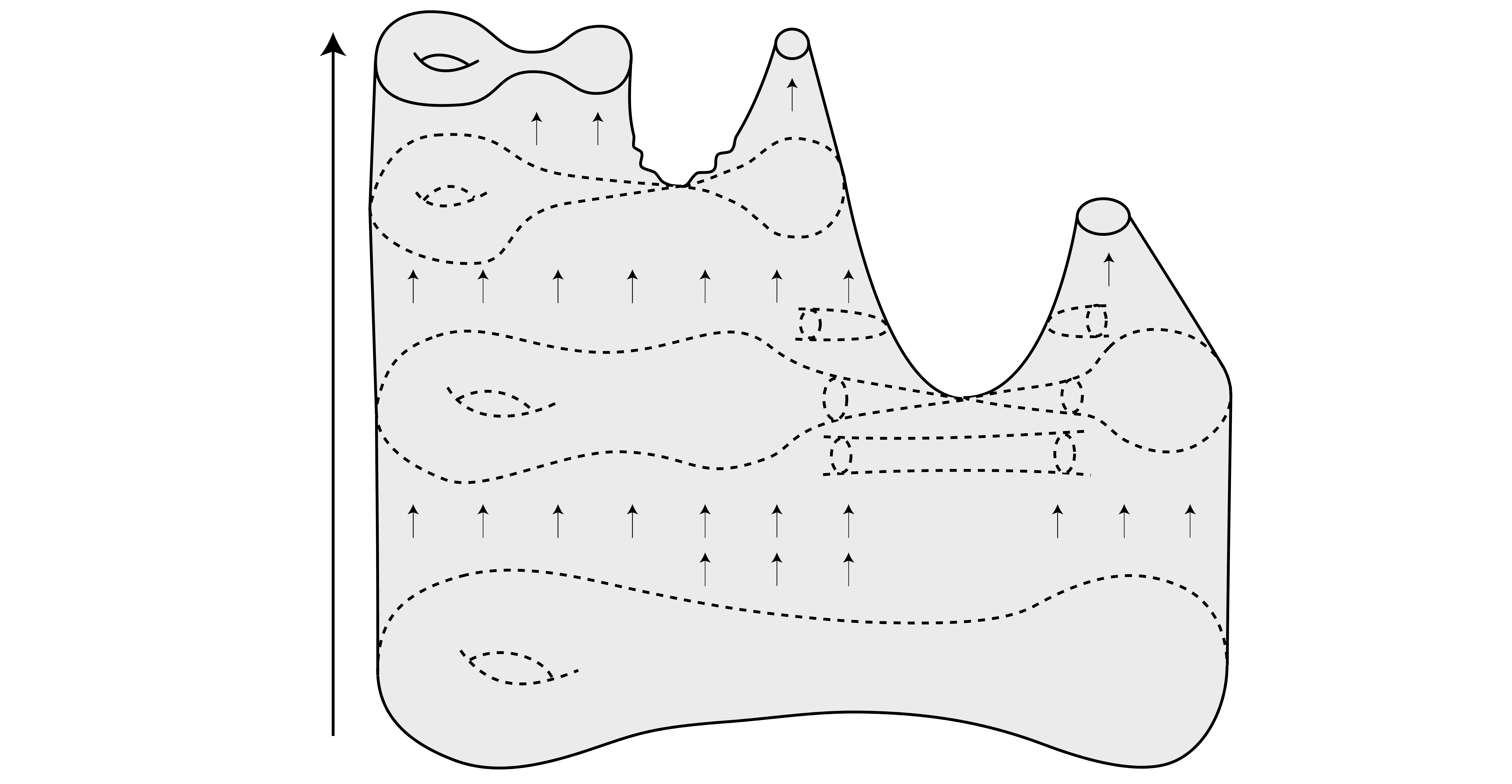}
\caption{Example of a $0$-complete Ricci flow spacetime with initial time-slice $(M, g_0)$.\label{fig_RF_spacetime_complete}}
\end{figure}
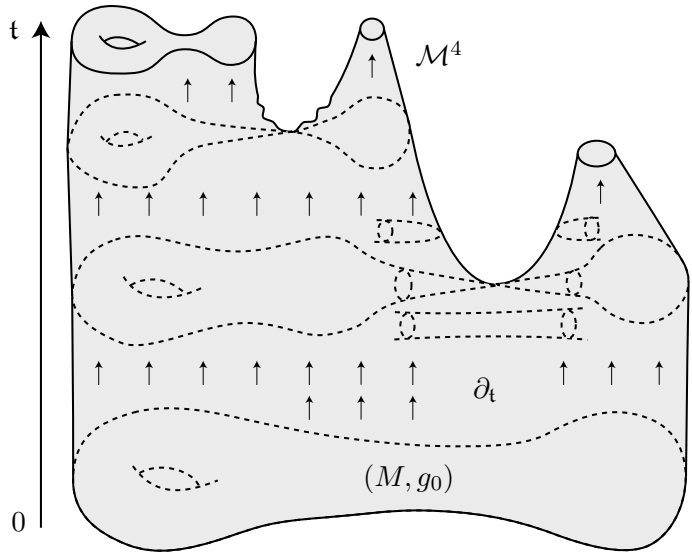
In \cite{Kleiner:2014wz} it was shown that every Riemannian manifold is the initial time-slice of a Ricci flow spacetime $\M$ whose time-slices are \emph{0-complete}, which we also refer to as \emph{complete} (see Figure~\ref{fig_RF_spacetime_complete} for an illustration).
This means that the time-slices of $\M$ may be incomplete, but each time-slice can be completed as a metric space by adding a countable set of points.
Note that since the curvature after a singularity is not uniformly bounded, we cannot easily control the time until a subsequent singularity arises.
In fact, it is possible --- although not known at this point --- that the set of singular times on a finite time-interval is infinite or even uncountable.
See \cite{kleiner_lott_singular_ricci_flows_ii} for a proof that this set has Minkowski dimension $\leq \frac12$.

We briefly review the construction of the (0-complete) Ricci flow spacetime $\M$ in \cite{Kleiner:2014wz}.
Consider a sequence of Ricci flows with surgery  with surgery scale $\delta_i \to 0$, starting from the same given initial metric, and construct the corresponding Ricci flow spacetimes $\M_{\delta_i}$ as in (\ref{eq_MM_construction_from_RFws}).
Using a compactness argument, it was shown in \cite{Kleiner:2014wz} that, after passing to a subsequence, we have convergence
\begin{equation} \label{eq_MM_delta_i_to_MM}
 \M_{\delta_i} \longrightarrow \M 
\end{equation}
in a certain sense.
The Ricci flow spacetime $\M$ can then be shown to be $0$-complete.

We remark that even though the surgery scale in this flow is effectively 0, which seems more canonical than in a Ricci flow with surgery, the entire flow may a priori not be canonical; i.e. the flow is a priori not uniquely determined by its initial data.

We also remind the reader that, while a Ricci flow spacetime describes a singular flow, the metric tensor field $g$ on $\M$ is not singular itself, since the spacetime manifold $\M$ does not ``contain the singular points''.
In other words, $\M$ describes the flow only on its regular part.
A flow that includes singular points can be obtained, for example, by taking the metric completion of the time-slices.
However, we do not take this approach, in order to avoid having to formulate the Ricci flow equation at the added singular points.
This is in contrast to weak forms of other geometric flows, such as the Brakke flow (generalizing mean curvature flow), which is defined at singular points and therefore not smooth everywhere.  

In lieu of an interpretation of the Ricci flow equation at the (nonexistent) singular points of a Ricci flow spacetime, it becomes necessary to characterize the asymptotic geometry in its almost singular regions.
This is achieved via the  \emph{canonical neighborhood assumption}, which states that regions of high curvature are geometrically close to model solutions --- \emph{$\kappa$-solutions} --- modulo rescaling (see Definition~\ref{def_canonical_nbhd_asspt} for more details).
Roughly speaking, this implies that these regions are either spherical, neck-like or cap-like. $\kappa$-solutions (see Definition~\ref{def_kappa_solution} for more details) arise naturally as blow-up limits of conventional 3-dimensional Ricci flows and have also been shown to characterize high curvature regions in Ricci flows with surgery.  
Moreover, the Ricci flow spacetimes constructed in \cite{Kleiner:2014wz} also satisfy the canonical neighborhood assumption in an even stronger sense (for more details see the discussion after Definition~\ref{def_canonical_nbhd_asspt}).

\subsection{Statement of the main results} \label{subsec_statement_main_results}
We now state the main results of this paper in their full generality.
Some of the terminology used in the following was informally introduced in the previous subsection.
For precise definitions and further discussions we refer the reader to Section~\ref{sec_preliminaries_I}.

Our first main result is the uniqueness of complete Ricci flow spacetimes that satisfy the canonical neighborhood assumptions.
These spacetimes were also sometimes called ``weak Ricci flows'' in the previous two subsections.

\begin{theorem}[Uniqueness of Ricci flow spacetimes, general form] \label{thm_uniquness}
There is a universal constant $\eps_{\can} > 0$ such that the following holds.

Let $(\M, \t, \partial_{\t}, g)$ and $(\M', \t', \partial_{\t'}, g')$ be two Ricci flow spacetimes that are both $(0, T)$-complete for some $T \in (0, \infty]$ and satisfy the $\eps_{\can}$-canonical neighborhood assumption at scales $(0,r)$ for some $r > 0$.
If the initial time-slices $(\M_0, g_0)$ and $(\M'_0, g'_0)$ are isometric, then the flows $(\M, \t, \partial_{\t}, g)$ and $(\M', \t', \partial_{\t'}, g')$ are isometric as well.

More precisely, assume that there is an isometry $\phi : (\M_0, g_0) \to (\M'_0, g'_0)$.
Then there is a unique smooth diffeomorphism $\wh\phi : \M_{[0,T]} \to \M'_{[0,T]}$ such that
\[ \wh\phi^* g' = g, \qquad
\wh\phi \big|_{\M_0} = \phi, \qquad
\wh\phi_* \partial_{\t} = \partial_{\t'}, \qquad
\t' \circ \wh\phi = \t. \]
\end{theorem}

A Ricci flow spacetime is ``$(0, T)$-complete'' if the 0-completeness property holds up to time $T$ (see Definition~\ref{def_completeness}).

Both properties that are imposed on $\M$ and $\M'$ in Theorem~\ref{thm_uniquness} hold naturally for the Ricci flow spacetimes constructed in \cite{Kleiner:2014wz}.
So we obtain the following corollary.

\begin{corollary} \label{cor_existence_uniqueness}
There is a universal constant $\eps_{\can} > 0$ such that the following holds.

For every compact Riemannian manifold $(M, g)$ there is a unique (i.e. canonical) Ricci flow spacetime $(\M, \t, \partial_{\t}, g)$ whose initial time-slice $(\M_0, g_0)$ is isometric to $(M, g)$ and that is $0$-complete, and such that for every $T > 0$ the time-slab $\M_{[0,T)}$ satisfies the $\eps_{\can}$-canonical neighborhood assumption at scales $(0,r_T)$ for some $r_T > 0$.
\end{corollary}

While we will not discuss this here, we remark that it is possible to modify the arguments in \cite{Bamler_LT_0, Bamler_LT_A, Bamler_LT_B, Bamler_LT_C, Bamler_LT_D} to show that the flow $\M$ becomes non-singular past some time $T > 0$ and we have a curvature bound of the form $|{\Rm}| < C / t$.   So the scale $r_T$ in Corollary~\ref{cor_existence_uniqueness} can even be chosen independently of $T$.

Coming back  to Theorem~\ref{thm_uniquness}, we draw attention to the fact that the time-slices of $\M$ and $\M'$, including the initial time-slices, may have infinite diameter or volume.  Also, they may have unbounded curvature even in bounded subsets, for instance when the flow starts from a manifold with finite diameter cuspidal ends.
 We also emphasize that the constant $\eps_{\can}$ is universal and does not depend on any geometric quantities.

Theorem~\ref{thm_uniquness} will follow from a stability result for Ricci flow spacetimes.
We first present a slightly less general, but more accessible version of this stability result.   In the following theorem, we only require the completeness and the canonical neighborhood assumption to hold above some small scale $\eps$, i.e. where the curvature is $\lesssim \eps^{-2}$.
As such, the theorem can also be used to compare two Ricci flows with surgery or a Ricci flow with surgery and a Ricci flow spacetime, via the construction (\ref{eq_MM_construction_from_RFws}). 
Furthermore, we only require the initial time-slices of $\M$ and $\M'$ be close in the sense that there is a sufficiently precise bilipschitz map $\phi$, which may only be defined on regions where the curvature is not too large.
As a consequence, the two Ricci flow spacetimes $\M, \M'$ can only be shown to be geometrically close.
More specifically, the map $\wh\phi$ that compares $\M$ with $\M'$ can only shown to be  bilipschitz and may not be defined on high curvature regions.
The map $\wh\phi$ is also not necessarily $\partial_{\t}$-preserving (see Definition~\ref{def_d_t_preserving}), but it satisfies the harmonic map heat flow equation (see Definition~\ref{def_RF_spacetime_harm_map_hf}).

\begin{theorem}[Stability of Ricci flow spacetimes, weak form] \label{Thm_comparing_RF_weak_form}
For every $\delta > 0$ and $T < \infty$ there is an $\eps = \eps (\delta, T) > 0$ such that the following holds.

Consider two $(\eps, T)$-complete Ricci flow spacetimes $\M, \M'$ that each satisfy the $\eps$-canonical neighborhood assumption at scales $(\eps, 1)$.

Let $\phi : U \to U'$ be a diffeomorphism between two open subsets $U \subset \M_0$, $U' \subset \M'_0$.
Assume that $|{\Rm}| \geq \eps^{-2}$ on $\M_0 \setminus U$ and
\[ | \phi^* g'_0 - g_0 | \leq \eps . \]
Assume moreover that the $\eps$-canonical neighborhood assumption holds on $U'$ at scales $(0,1)$.

Then there is a time-preserving diffeomorphism $\wh\phi : \wh{U} \to \wh{U}'$ between two open subsets $\wh{U} \subset \M_{[0, T]}$ and $\wh{U}' \subset \M'_{[0, T]}$ that evolves by the harmonic map heat flow and that satisfies $\wh\phi = \phi$ on $U \cap \wh{U}$ and
\[ | \wh\phi^* g' - g | \leq \delta. \]
Moreover, $|{\Rm}| \geq \delta^{-2}$ on $\M_{[0, T]} \setminus \wh{U}$.
\end{theorem}

We remark that the condition that the $\eps$-canonical neighborhood assumption holds on $U'$ at scales $(0,1)$ is automatically satisfied if the curvature scale on $U'$ is $> \eps$, which is implied by a bound of the form $|{\Rm}| < c \eps^{-2}$ on $U'$ (see Definition~\ref{def_canonical_nbhd_asspt} for further details).

Theorem~\ref{Thm_comparing_RF_weak_form} is formulated using only $C^0$-bounds on the quantity $\phi^*g'-g$,  which measures the deviation from an isometry.
Using a standard argument involving local gradient estimates for non-linear parabolic equations, these bounds can be improved to higher derivative bounds as follows:

\begin{addendum-Thm15}
Let $m_0 \geq 1$ and $C < \infty$.
If in Theorem~\ref{Thm_comparing_RF_weak_form} we additionally require that
\[ \big| \nabla^m (\phi^* g'_0 - g_0 ) \big| \leq \eps \]
and
\[ |\nabla^m {\Rm}| \leq C \]
on $U$ for all $m = 0,\ldots, m_0 + 2$ and allow $\eps$ to depend on $m_0$, $C$, then
\[ \big| \nabla^m (\wh\phi^* g' - g ) \big| \leq \delta \]
on $\wh U$ for all $m = 0,\ldots, m_0$.
\end{addendum-Thm15}

A similar addendum applies to Theorem~\ref{Thm_comparing_RF_general_version} below.

Combining Theorem~\ref{Thm_comparing_RF_weak_form} with \cite[Thm. 1.2]{Kleiner:2014wz}  (see also \cite[p.6]{Kleiner:2014wz}) we obtain:

\begin{corollary}
\label{cor_convergence_corollary}
Let $(M,g)$ be a compact Riemannian manifold, and consider a sequence of Ricci flows with surgery starting from $(M,g)$, for a sequence  of surgery scales  $\delta_i\ra 0$.
Let $\M_{\delta_i}$ be the corresponding Ricci flow spacetimes, as defined in (\ref{eq_MM_construction_from_RFws}).
Then the $\M_{\delta_i}$ converge to a unique Ricci flow spacetime as in (\ref{eq_MM_delta_i_to_MM}).  
\end{corollary}

We remark that in the case of mean curvature flow a similar result holds: In \cite{head,lauer} it was  shown that the $2$-convex mean curvature flow with surgery constructed in \cite{huisken_sinestrari3} converges to the level set flow as the surgery parameter tends to zero.  
However, their proofs, which are remarkably elementary, are entirely different from ours: they use a quantitative variant of the barrier argument from White's uniqueness theorem \cite{white_nature}.  A similar convergence result holds for mean convex mean curvature flow with surgery in $\R^3$, as constructed in \cite{brendle_huisken,haslhofer_kleiner}.

Lastly, we state the stability theorem for Ricci flow spacetimes in its full generality.
The following theorem is an improvement of Theorem~\ref{Thm_comparing_RF_weak_form} for the following reasons:
\begin{itemize}
\item It provides additional information on the bilipschitz constant and establishes a polynomial dependence on the curvature.
\item It states that the precision of the canonical neighborhood assumption can be chosen independently of time and bilipschitz constant.
\item In provides a condition under which the map $\wh\phi$ is almost surjective.
\end{itemize}

\begin{theorem}[Strong Stability of Ricci flow spacetimes] \label{Thm_comparing_RF_general_version}
There is a constant $\underline{E} < \infty$ such that for every $\delta > 0$, $T < \infty$ and $ \underline{E} \leq E < \infty$ there are constants $\eps_{\can} = \eps_{\can} (E), \eps = \eps (\delta, T, E) > 0$ such that for all $0 < r \leq 1$ the following holds.

Consider two $(\eps r, T)$-complete Ricci flow spacetimes $\M, \M'$ that each satisfy the $\eps_{\can}$-canonical neighborhood assumption at scales $(\eps r, 1)$.

Let $\phi : U \to U'$ be a diffeomorphism between two open subsets $U \subset \M_0$, $U' \subset \M'_0$.
Assume that $|{\Rm}| \geq (\eps r)^{-2}$ on $\M_0 \setminus U$ and
\[ | \phi^* g'_0 - g_0 | \leq \eps \cdot  r^{2E} (|{\Rm}| + 1)^{E} \]
on $U$.
Assume moreover that the $\eps_{\can}$-canonical neighborhood assumption holds on $U'$ at scales $(0,1)$.

Then there is a time-preserving diffeomorphism $\wh\phi : \wh{U} \to \wh{U}'$ between two open subsets $\wh{U} \subset \M_{[0, T]}$ and $\wh{U}' \subset \M'_{[0, T]}$ that evolves by the harmonic map heat flow, satisfies $\wh\phi = \phi$ on $U \cap \wh{U}$ and that satisfies
\[ | \wh\phi^* g' - g | \leq \delta \cdot r^{2E} (|{\Rm}| + 1)^{E} \]
on $\wh{U}$.
Moreover, we have $|{\Rm}| \geq r^{-2}$ on $\M_{[0, T]} \setminus \wh{U}$.

If additionally $|{\Rm}| \geq (\eps r)^{-2}$ on $\M'_0 \setminus U'$, then we also have $|{\Rm}| \geq r^{-2}$ on $\M'_{[0,T]} \setminus \wh{U}'$.
\end{theorem}

\subsection{A brief sketch of the proof, and further discussion}
We now give a very brief and informal outline of the proof. 
See Section~\ref{sec_overview} for a more detailed overview.

Theorem~\ref{thm_uniquness}, the main uniqueness theorem, is obtained from the Strong Stability Theorem~\ref{Thm_comparing_RF_weak_form} or \ref{Thm_comparing_RF_general_version} via a limit argument.
In Theorems~\ref{Thm_comparing_RF_weak_form} and \ref{Thm_comparing_RF_general_version} we are given a pair of Ricci flow spacetimes $\M$, $\M'$, and an almost isometry $\phi:\M_0\supset U\ra U'\subset \M_0'$ between open subsets of their initial conditions, and our goal is to construct an almost isometry $\wh\phi:\M\supset \wh{U} \ra \wh{U'}\subset \M'$ that extends $\phi$ forward in time.  The construction of $\wh\phi$ involves a procedure for choosing the domain $\wh{U}$ of $\wh\phi$, and the map $\wh\phi$ on this domain.  
These two procedures interact in a complex way, and for this reason they are implemented by means of a simultaneous induction argument.

We now indicate some of the highlights in the two steps of the induction. 

The domain $\wh{U}$ is chosen to contain all points in $\M$ whose curvature $|{\Rm}|$ lies (roughly) below a certain threshold and is obtained  from $\M$ by means of a delicate truncation argument.  The truncation uses the fact that, roughly speaking, the part of $\M$ with large curvature looks locally either like a neck, or like a cap region.  We cut along neck regions so that the time-slices of $\wh{U}$ have spherical boundary.  A critical complication stems from the occurrence of moments in time when the presence of cap regions interferes with the need to cut along neck regions.  
This occurrence necessitates modification of the domain by either insertion or removal of cap regions.

The map $\wh\phi$ is constructed by solving the harmonic map heat flow equation for its inverse $\wh\phi^{-1}$.  There are many interrelated issues connected with this step, of which the three most important are:
\begin{itemize}
\item The distortion of the map $\wh\phi$ must be controlled under the harmonic map heat flow.  For this, our main tool is an interior decay estimate, which may be applied away from the spacetime boundary of $\wh{U}$.
\item The presence of boundary in  $\wh{U}$ introduces boundary effects, which must be controlled.  
It turns out that the geometry of shrinking necks implies that the neck boundary recedes rapidly, which helps to stabilize the construction.
\item The insertion of the cap regions alluded to above necessitates the extension of the map $\wh\phi$ over the newly added region.  
The implementation of this extension procedure relies on a delicate interpolation argument, in which the geometric models for the cap regions must be aligned with the existing comparison map $\wh\phi$  within tolerances fine enough to prolong the construction.  This step hinges on several ingredients and their precise compatibility ---   rigidity theorems for the models of the cap regions \cite{Hamilton:1993em,Brendle:2013dd}, quantitative asymptotics of the models \cite{Bryant:gVzfj4nx}, and
strong decay estimates for the distortion of the map $\wh\phi$.
\end{itemize}

\subsection{Acknowledgements}
The first named author would like to thank Bennet Chow for pointing out his paper with Greg Anderson (see \cite{Anderson:2005cf}).
Both authors would also like to thank Yongjia Zhang for many valuable corrections.

\section{Overview of  the proof}\label{sec_overview}
 In this section we will describe  the proof of the main theorem.  
 Our aim here is to cover the most important ideas in an informal way, with many technicalities omitted.  
 The first subsection of this overview  provides an initial glimpse of the argument.  
 It is intended to be accessible to readers outside the field who would like to gain some sense of how the proof goes.  
 The remaining subsections delve into the proof in greater detail and are primarily intended for people working in the area. 

The main part of the paper is concerned with the proof of the Strong Stability Theorem, Theorem~\ref{Thm_comparing_RF_weak_form} or \ref{Thm_comparing_RF_general_version}, which asserts that two Ricci flow spacetimes are geometrically close, given that their initial data are geometrically close and we have completeness as well as the canonical neighborhood assumption in both spacetimes above a sufficiently small scale. 
All other results of this paper will follow from this theorem; in particular, the Uniqueness Theorem, Theorem~\ref{thm_uniquness}, will follow from Theorem~\ref{Thm_comparing_RF_weak_form} or \ref{Thm_comparing_RF_general_version} via a limit argument.  

In the Strong Stability Theorem, we consider two Ricci flow spacetimes $\M$ and $\M'$, whose initial time-slices, $(\M_0, g_0)$ and $(\M'_0, g'_0)$, are geometrically close or even isometric.
Our goal is the construction of a map $\phi : \M \supset U \to   \M'$, defined on a sufficiently large domain $U$, whose bilipschitz constant is sufficiently close to 1.
In Theorems~\ref{Thm_comparing_RF_weak_form} and~\ref{Thm_comparing_RF_general_version}, this map is denoted by $\wh\phi$.
However, in the main part of this paper, as well as in this overview, the hat will be omitted.

Our basic method for constructing $\phi$, which goes back to DeTurck \cite{DeTurck:1983jp}, is to solve the harmonic map heat flow equation for the inverse $\phi^{-1}$.  
In the nonsingular case when both Ricci flow spacetimes $\M$ and $\M'$ may be represented by ordinary smooth Ricci flows on compact manifolds $(M, g(t))$ and $(M', g'(t))$, this reduces to finding a solution $\phi (t) : M \to M'$ to the equation $\D_t(\phi^{-1})=\Delta (\phi^{-1})$. As DeTurck observed,  the family of difference tensors $h(t) := (\phi(t))^* g'(t) - g(t)$, which quantify the deviation of $\phi (t)$ from being an isometry, then satisfies the \emph{Ricci-DeTurck perturbation equation}:
\begin{equation} \label{eq_RdT_outline}
 \partial_t h(t) = \triangle_{g(t)} h(t) + 2\Rm_{g(t)} (h(t)) + \nabla h(t) * \nabla h(t) + h(t) * \nabla^2 h(t)\,.
\end{equation}
If $\phi(0)$ is an isometry, then $h(0)\equiv 0$.  So by the uniqueness of solutions to the strictly parabolic equation (\ref{eq_RdT_outline}) one gets that $h(t)\equiv 0$ for all $t \geq 0$,  and hence the two given Ricci flows are isometric.  In our case we are given that $h(0)$ is small, and want to show that it remains small.   Equation (\ref{eq_RdT_outline}) has several properties that are  important for maintaining control over of the size of the perturbation $h$, as the construction proceeds.

\subsection{The construction process, an initial glimpse} \label{subsec_constr_process_I_outline}
In the general case, in which  $\M$ and $\M'$ may be singular, the domain of the map $\phi$ will be the part of $\M$ that is not too singular, i.e. the set of points whose curvature is not too large.  Note that this means that we will effectively be solving the harmonic map heat flow equation with a boundary condition.  

The main objects of our construction are a subset $\N \subset \M$, called the \emph{comparison domain}, and a time-preserving diffeomorphism onto its image $\phi : \N \to \M'$, called the \emph{comparison (map)}.    We  construct $\N$ and $\phi$ by a simultaneous induction argument using discrete time increments    $[t_{j-1},t_{j} ]$.  The domain $\N$   is the union
\begin{equation*} \label{eq_NN_form_outline}
 \N = \N^1 \cup \N^2 \cup \ldots \cup \N^{J}\,,
\end{equation*}
where $\N^{j}$ lies in the time-slab of $\M$ corresponding to the time-interval $[t_{j-1}, t_j]$.  
The restriction of $\phi$ to each time-slab $\N^j$ is denoted by $\phi^j : \N^j \to \M'$.  In the induction step, we enlarge $\N$ and $\phi$ in two stages: in the first we determine  $\N^{J+1}$, and in the second we define the map $\phi^{J+1}:\N^{J+1}\ra \M'$.

Before proceeding, we introduce  the  \emph{curvature scale} $\rho$, which will be  used throughout the paper.  The precise definition may be found in Subsection \ref{subsec_curvature_scale}, but for the purposes of this overview, $\rho$ can be any function that agrees up to a fixed factor with $R^{-1/2}$ wherever $|{\Rm}|$ is sufficiently large.  Here $R$ denotes the scalar curvature.  Note that $\rho$ has the dimension of length.

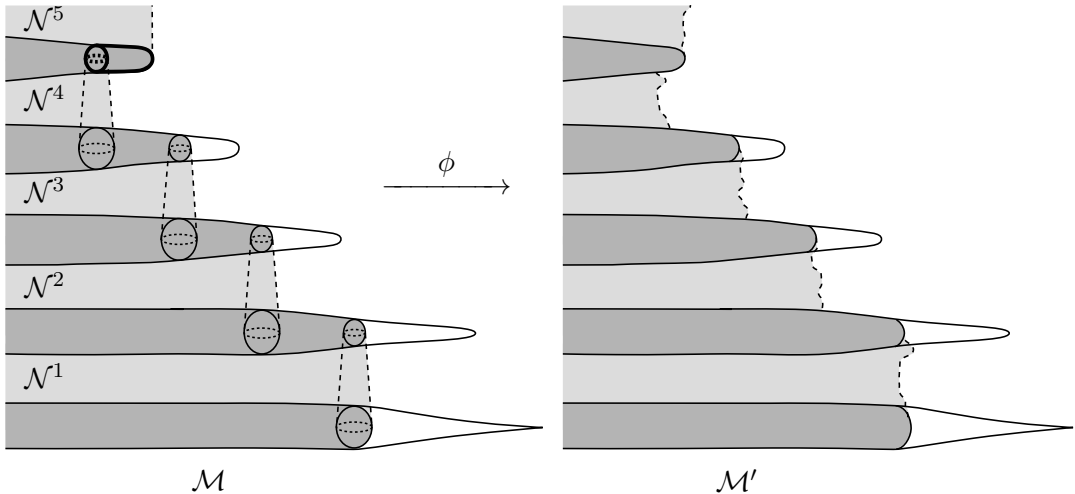
\begin{figure}
\labellist
\small\hair 2pt
\pinlabel $\M$ at 400 70
\pinlabel $\M'$ at 1400 70
\pinlabel $\phi$ at 850 670
\pinlabel $\xrightarrow{\hspace{15mm}}$ at 850 630
\pinlabel {\small $\N^1$} at 90 268
\pinlabel {\small $\N^2$} at 90 440
\pinlabel {\small $\N^3$} at 90 620
\pinlabel {\small $\N^4$} at 90 795
\pinlabel {\small $\N^5$} at 90 945
\endlabellist
\centering
\includegraphics[width=145mm]{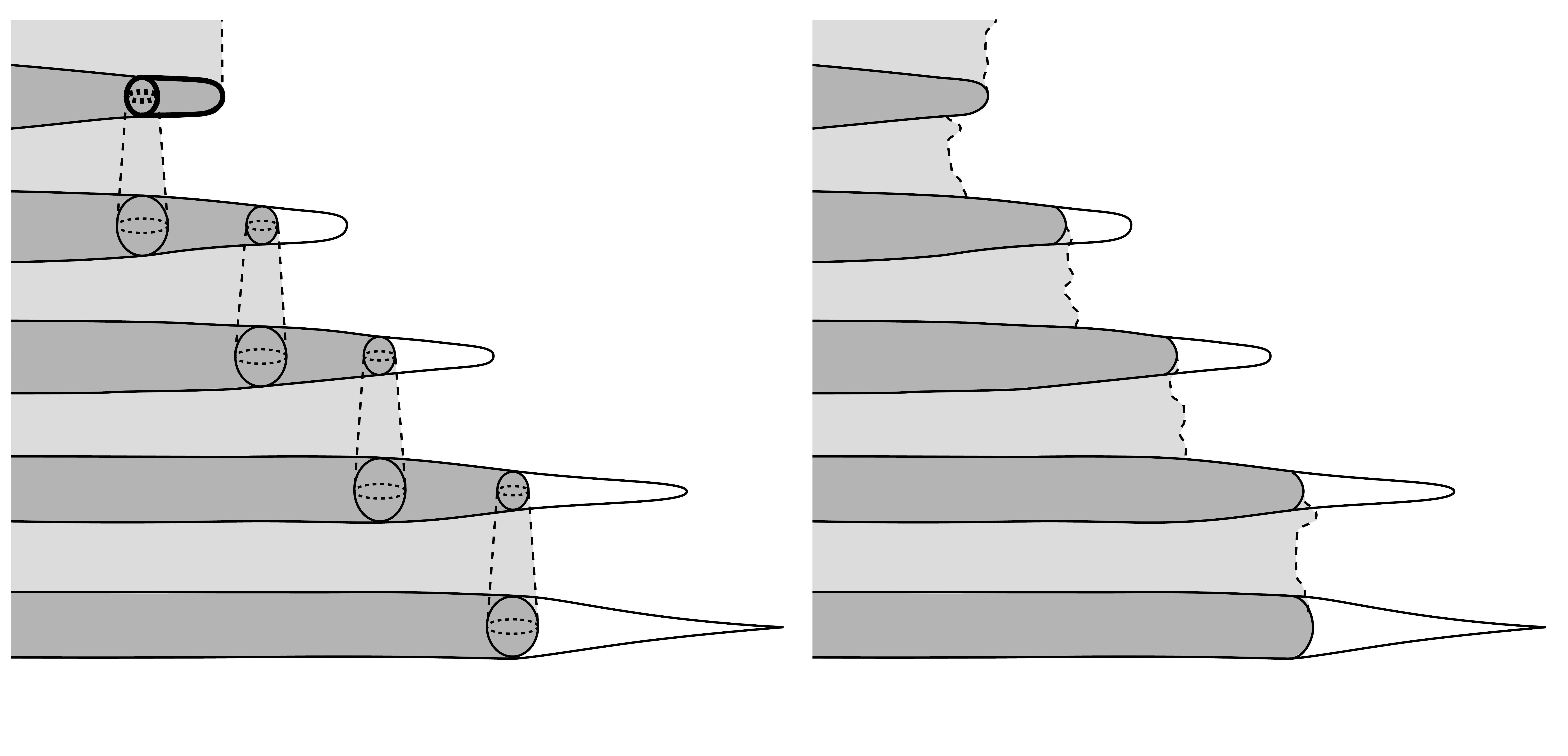}
\caption{Comparison domain $\N \subset \M$ and comparison $\phi$ between $\M$ and $\M'$.
The extension cap on the initial time-slice of $\N^5$ is outlined in bold.
\label{fig_comparison-schematic}}
\end{figure}

We will now provide further details on the geometry of $\N$ and $\phi$.

Fix a small \emph{comparison scale} $r_{\comp} > 0$.  Our goal is to choose the comparison domain $\N$ such that it roughly contains the points for which $\rho \gtrsim r_{\comp}$.  
So we will have $R \leq C r_{\comp}^{-2}$ on $\N$ and $R \geq c r_{\comp}^{-2}$ on $\M \setminus \N$ for some constants $C, c > 0$.  The constant $r_{\comp}$ will also determine the length of our time steps: we set $t_j=j r_{\comp}^2$, so that the time steps have duration $r_{\comp}^2$.

Each time-slab $\N^{j}$ will be chosen to be a product domain on the time-interval $[t_{j-1}, t_{j}]$.
That is, the flow restricted to $\N^{j}$ can be described by an ordinary Ricci flow parameterized by the time-interval $[t_{j-1}, t_{j}]$, on the initial time-slice of $\N^{j}$.  
We will sometimes denote this initial time-slice of $\N^{j}$ by $\N^{j}_{t_{j-1}}$ and the final time-slice by $\N^{j}_{t_{j}}$.  
Note that $\N^{j}_{t_{j-1}}$ and $\N^{j}_{t_{j}}$ are diffeomorphic, as $\N^{j}$ is a product domain.   
Each domain $\N^{j}$ will moreover be chosen in such a way that its time-slices $\N^{j}_t$ are bounded by 2-spheres of diameter $\approx r_{\comp}$ that are central 2-spheres of sufficiently precise necks (i.e. cylindrical regions) in $\M$. 

We now discuss the inductive construction of $\N$ and $\phi$. 
For this purpose, assume that $\N^1, \ldots, \N^J$ and $\phi^1, \ldots, \phi^J$ have already been constructed.
Our goal is now to construct $\N^{J+1}$ and $\phi^{J+1}$.

We first outline the construction of $\N^{J+1}$.
Our construction relies on  the canonical neighborhood assumption, which guarantees that the large curvature part of the Ricci flow looks, roughly speaking, locally either  neck-like or like a cap region diffeomorphic to a $3$-ball.  
Using this geometric characterization, the final time-slice $\N^{J+1}_{t_{J+1}}$  of $\N^{J+1}$ is obtained by truncating the time-$t_{J+1}$-slice $\M_{t_{J+1}}$ along a suitable collection of central 2-spheres of necks of scale $\approx r_{\comp}$.  
Due to the fact that a neck region shrinks substantially in a single time step and our neck regions have nearly constant scale, this process will ensure that the boundaries of successive time steps are separated by a distance $\gg r_{\comp}$. 
So our truncation process typically yields a rapidly receding ``staircase'' pattern (see Figure~\ref{fig_comparison-schematic}). 
However, it can happen that a cap region evolves in such a way that its scale increases slowly  over a time-interval of duration $\gg r_{\comp}^2$, so that at time $t_{J}$, this cap region is not contained in the final time-slice $\N^{J}_{t_J}$, but is contained in the initial time-slice $\N^{J+1}_{t_J}$.  
This behavior occurs, for instance, a short time after a generic neck pinch singularity.  
In such a situation, the comparison domain $\N$ is enlarged at time $t_J$ by a cap region, which we call an \emph{extension cap} (see again Figure~\ref{fig_comparison-schematic}).   
It then becomes necessary to extend the comparison map $\phi$ over the inserted region.

We now turn to the second stage of the induction step --- the construction of the comparison map $\phi^{J+1}:\N^{J+1}\ra \M'$.  

As mentioned above, we will construct $\phi^{J+1}$ by solving the harmonic map heat flow equation for the inverse diffeomorphism $(\phi^{J+1})^{-1}$.   For now, we will only provide a brief indication of a few of the obstacles that arise, leaving more detailed discussion to the subsequent subsections of this overview:
\begin{itemize}
\item \emph{(Controlling $h$, Subsection~\ref{subsec_controlling_the_perturbation})} \quad Since our objective is to produce a map that is almost an isometry, one of the key ingredients in our argument is a scheme for maintaining control on the size of the metric \emph{perturbation}  $h=\phi^*g'-g$ as the map $\phi$ evolves.  Our main tool for this is an interior decay estimate for $|h|$ with respect to a certain weight.
\item \emph{(Treatment of the boundary, Subsection~\ref{subsec_constr_process_II_outline})} \quad  The Ricci flow spacetime restricted to the product domain $\N^{J+1}$ is given by an ordinary Ricci flow on the manifold with boundary $\N^{J+1}_{t_{J+1}}$.  The process for solving the harmonic map heat flow equation must take this boundary into account and maintain control on any influence it may have on the rest of the evolution.  
\item \emph{(Extending the comparison, Subsections~\ref{subsec_constr_process_III_outline} and \ref{subsec_overview_bryant_extension_principle})} \quad  
As mentioned above, it may be necessary to extend the comparison map $\phi$ over an extension cap at time $t_J$.  This requires a careful analysis of the geometry of $\M$ and $\M'$ in neighborhoods of the cap and its image, showing that both are well approximated by rescaled Bryant solitons.  
Then the extension of $\phi$ is obtained by gluing the pre-existing comparison map with suitably normalized Bryant soliton ``charts''.  This gluing construction is particularly delicate, since it must maintain sufficient control over the quality of the comparison map.
\end{itemize}

The actual construction of the comparison map $\phi$ is implemented using a continuity argument.  The above issues interact with one another in a variety of different ways.  For instance, both the treatment of the boundary and the procedure for extending $\phi$ over cap regions are feasible only under  certain assumptions on the smallness of $h$, and both cause potential deterioration of $h$, which must be absorbed by the argument for controlling $h$. We defer further discussion of these interactions, and other points of a more technical nature, to Subsection~\ref{subsec_constr_process_IV_outline}.

\subsection{Controlling the perturbation \texorpdfstring{$h$}{h}}
\label{subsec_controlling_the_perturbation}
In order to control the perturbation $h = \phi^* g' -g$ in the inductive argument described above, we will consider the following weighted quantity:
\begin{equation}
\label{eqn_def_q_overview}
 Q \approx e^{- H  t} R^{-E/2} |h| \approx e^{- H  t} \rho^E |h|\,.
\end{equation}
Here $R$ denotes the scalar curvature and $H > 0$, $E > 2$. 
We will show that this quantity satisfies an \emph{interior decay estimate}, which may be thought of as a quantitative semi-local version of a maximum principle: rather than asserting that $Q$ cannot attain an interior maximum, it roughly states that $Q$, evaluated at a point $(x,t)$, must be a definite amount smaller than its maximum over a suitable parabolic neighborhood around $(x,t)$ (see below for a more precise statement). 
This interior decay estimate will allow us to promote, and sometimes improve, a bound of the form $Q \leq \ov{Q}$ forward in time. 
We emphasize that the presence of the factor $\rho^E$, and the fact that $E$ is strictly larger than $2$, are both essential for the interior decay estimate.  Moreover, the freedom to choose $E$ large ($>100$ say) will be of crucial importance at a later point in our proof (see Subsection~\ref{subsec_overview_bryant_extension_principle}).

Before providing further details on this estimate, we want to illustrate the function of the weights in the definition of $Q$.
The weight $e^{-Ht}$ serves a technical purpose, which we will neglect in this overview.
To appreciate the role  of the weight $R^{-E/2}$, consider for a moment a classical Ricci flow $(M, g(t))$ with a perturbation $h(t)$ that evolves by (\ref{eq_RdT_outline}).
Suppose that $h(0)$ is bounded and supported in a region of large scalar curvature.
So, due to the existence of the weight $R^{-E/2}$, the quantity $Q$ is small at time 0.
Our estimates will imply that $Q$ remains small throughout the flow.
Therefore, at any later time, the perturbation $h$ must be small at points where the curvature is controlled. In the following we will exploit this phenomenon, since, heuristically,  we are considering two Ricci flow spacetimes $\M$ and $\M'$ whose initial data is either equal or very similar away from the almost singular regions, where the scalar curvature is large.  
So even if $\M$ and $\M'$ were a priori significantly different at those almost singular scales --- resulting in a large perturbation $h(t)$ there --- then $Q$ would  still be small, initially.  
Thus the perturbation is expected to decay as we move forward in time and towards regions of bounded curvature, establishing an improved closeness there.
More specifically, as remarked in the previous subsection, $h$ may a priori only satisfy a rough bound near the neck-like boundary of each $\N^j$.
However, as $R \approx r_{\comp}^{-2}$ near such a boundary and $r_{\comp}$ is assumed to be small, our estimate suggests a significant improvement of this bound in regions where $R \approx 1$.

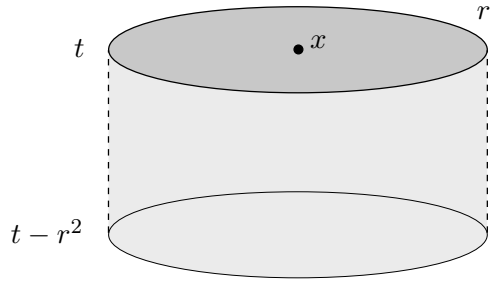
\begin{figure}
\labellist
\small\hair 2pt
\pinlabel {$t-r^2$} at 390 165
\pinlabel $t$ at 490 740
\pinlabel $r$ at 1750 850
\pinlabel $x$ at 1235 760
\endlabellist
\centering
\includegraphics[width=100mm]{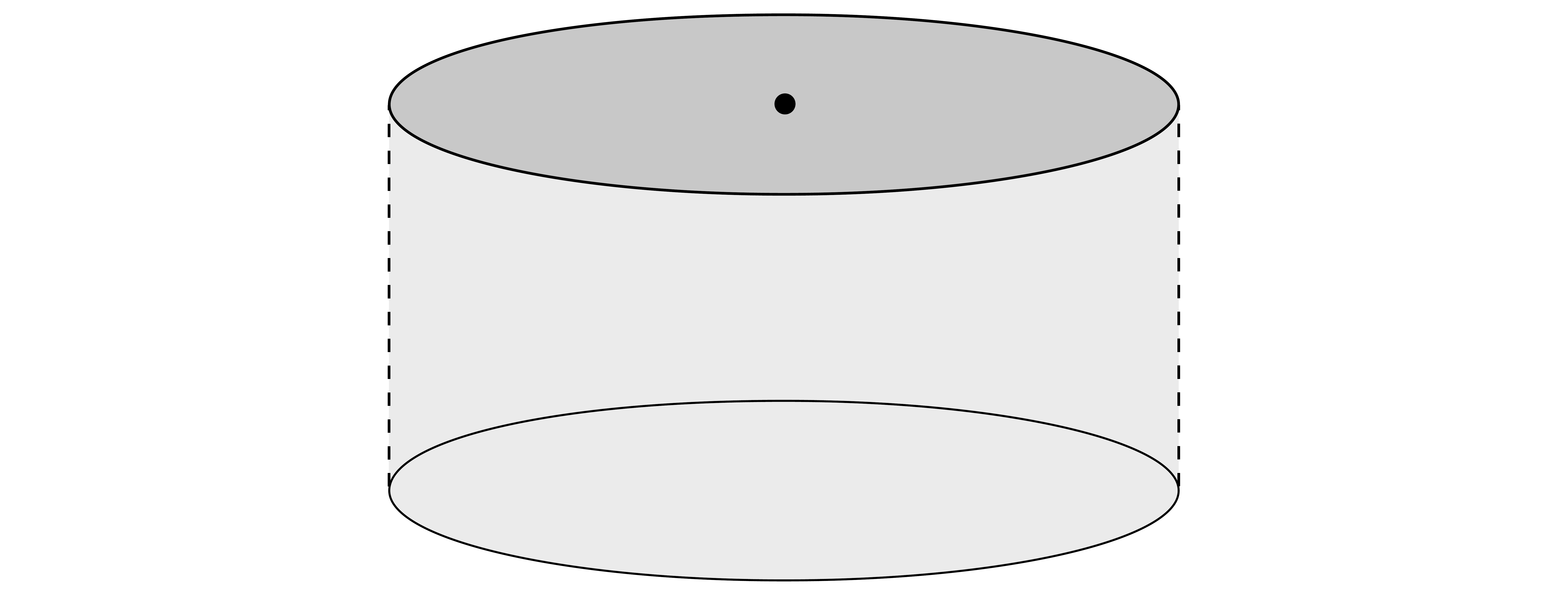}
\caption{A parabolic ball $B(x,t,r) \times [t-r^2,t]$ of radius $r$.
\label{fig_parabolic_ball}}
\end{figure}

We now explain the statement of the interior decay estimate in more detail, in the case of a classical Ricci flow on $M \times [0,T)$.
Assume that the perturbation $h$ is defined on a sufficiently large backwards parabolic region $P \subset M \times [0,T)$ around some point $(x,t)$.  
If $H$ is chosen sufficiently large and $|h| \leq \eta_{\lin}$ on $P$ for some sufficiently small $\eta_{\lin}$, where both $H$ and $\eta_{\lin}$ depend on $E$, then our estimate states that
\begin{equation} \label{eq_Q_110_sup_PQ_outline}
 Q(x,t) \leq \frac1{100} \sup_P Q. 
\end{equation}
Here ``$P$ sufficiently large'' means, roughly speaking, that the parabolic region $P$ contains a product domain of the form $B(x,t,r)\times [t-r^2,t]$ (a parabolic ball), where $B(x,t,r)$ is the $r$-ball  centered at $(x,t)$ in the time-$t$ slice $M\times\{t\}$ and $r$ is equal to a large constant times the scale $\rho(x,t)$ (see Figure~\ref{fig_parabolic_ball}).   

In fact, the choice of the factor $\frac1{100}$ in (\ref{eq_Q_110_sup_PQ_outline}) is arbitrary: for any $\alpha>0$  
we have the estimate
\begin{equation}
\label{eqn_stronger_interior_decay_estimate}
 Q(x,t) \leq \alpha \sup_P Q ,
\end{equation}
as long as we increase the size of the parabolic neighborhood $P$ accordingly.  An important detail here is that the constant $\eta_{\lin}$ in the bound $|h| \leq \eta_{\lin}$ can be chosen independently of $\alpha$.

The decay estimate (\ref{eq_Q_110_sup_PQ_outline}) will be used to propagate a bound of the form
\begin{equation} \label{eq_Q_Qbar_outline}
 Q \leq \ov{Q} 
\end{equation}
throughout most parts of the comparison domain $\N$.
Here we will choose the constant $\ov{Q}$ in such a way that (\ref{eq_Q_Qbar_outline}) holds automatically near the neck-like boundary of the $\N^j$ and such that (\ref{eq_Q_Qbar_outline}) implies $|h| \leq \eta_{\lin}$ wherever $\rho \geq r_{\comp}$.
Note  that at scales $\rho \gg r_{\comp}$, the bound (\ref{eq_Q_Qbar_outline}) implies a more precise bound on $|h|$, whose quality improves polynomially in $\rho$.

We will prove the interior decay estimate using a limit argument combined with a vanishing theorem for solutions of the linearized Ricci-DeTurck equation on $\kappa$-solutions, which uses an estimate of Anderson and Chow \cite{Anderson:2005cf}.  See Section~\ref{sec_semi_local_max} for more details.

\subsection{Treatment of the boundary} \label{subsec_constr_process_II_outline}
We now discuss aspects of the inductive construction of the map $\phi^{J+1}:\N^{J+1}\ra \M'$ (sketched in Subsection~\ref{subsec_constr_process_I_outline}) that are related to the presence of a boundary in the time-slices $\N^{J+1}_t$.  While the actual approach used in the body of the paper is guided by considerations that are beyond the scope of this overview, we will describe some of the main points in a form that is faithful to the spirit of the actual proof. 

Recall from Subsection~\ref{subsec_constr_process_I_outline} that  we wish to construct  $\phi^{J+1}$ by solving the harmonic map heat flow equation (for the inverse $(\phi^{J+1})^{-1}$), in such a way that  $\phi$ yields a perturbation $h=\phi^*g'-g$ satisfying the bound $Q\leq \ov Q$ near the neck-like boundary of $\N^j$, where $Q$ is as in Subsection~\ref{subsec_controlling_the_perturbation}.  Thus we need to specify boundary conditions so that the resulting evolution respects the bound $Q\leq \ov Q$.

Our strategy exploits the geometry of the boundary of $\N^{J+1}$.  
Recall from Subsection~\ref{subsec_constr_process_I_outline} that $\N^{J+1}$ is a product domain, and its boundary is collared by regions that look very close to shrinking round half-cylinders (half-necks) with scale comparable to $r_{\comp}$.  
Under a smallness condition on $h$ imposed in the vicinity of boundary components of $\N^{J+1}_{t_{J}}$, we argue that at time $t_{J}$, our map $\phi$ must map the  half-neck collar regions around the boundary to regions in the time-$t_{J}$-slice $\M'_{t_{J}}$ that are nearly isometric to half-necks. 
Moreover,  we will show that both half-necks evolve over the time-interval $[t_{J},t_{J+1}]$ nearly like round half-necks.  
We then use this characterization and a truncation procedure to find an approximate product domain $\N^{\prime J+1}\subset \M'$ that serves as the domain for the evolving inverse map $\phi^{-1}$.  It turns out that if the half-neck regions in $\M$ and $\M'$ are sufficiently cylindrical, and $\phi$ is initially (at time $t_{J}$) sufficiently close to an isometry near the collar regions, then the map $\phi^{J+1}$ produced by harmonic map heat flow remains sufficiently close to an isometry near the boundary of $\N^{J+1}$, in the sense that $Q\leq \ov Q$. 

\begin{figure}
\labellist
\small\hair 2pt
\pinlabel $\ldots$ at 790 710
\pinlabel $\ldots$ at 1340 520
\pinlabel $\ldots$ at 1830 320
\pinlabel { $\N^{J+1}$} at 120 620
\pinlabel { $\N^{J}$} at 95 425
\endlabellist
\centering
\includegraphics[width=145mm]{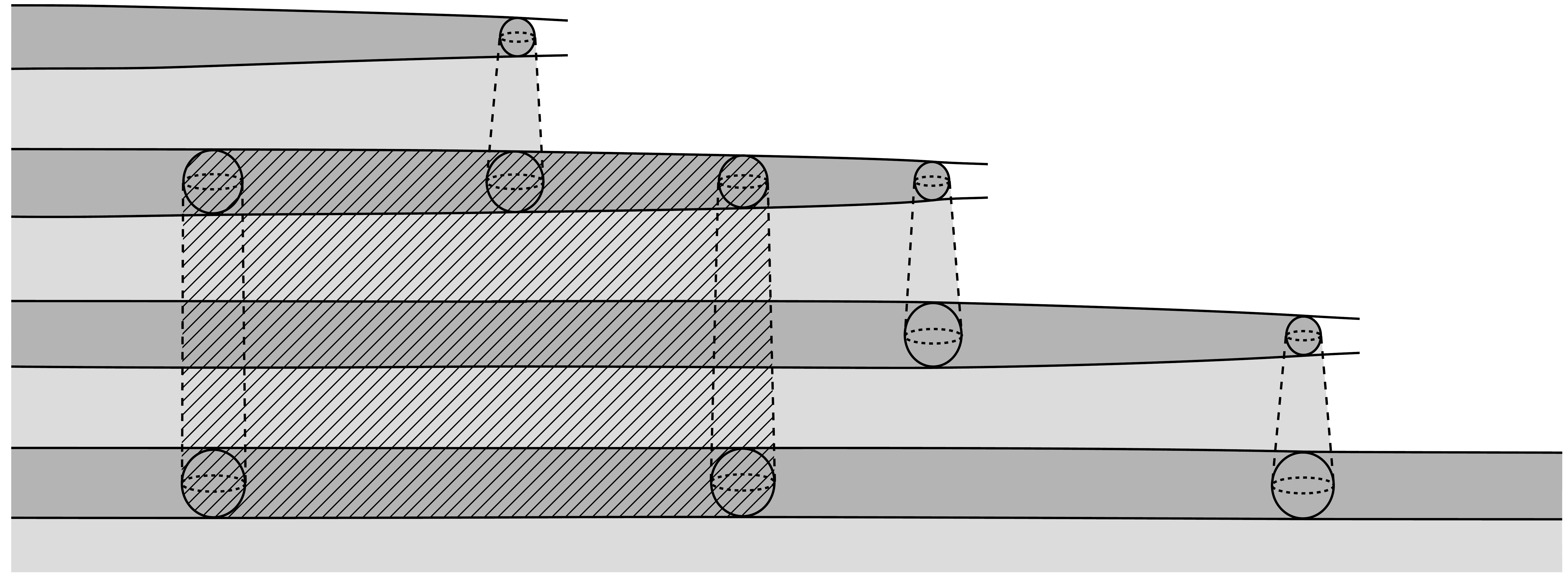}
\caption{A parabolic neighborhood (hatched region) inside a comparison domain (shaded region).
In order to apply the interior estimate at time $t_J$ near the boundary of $\N^{J+1}$, a large parabolic neighborhood must fit underneath the staircase pattern.
\label{fig_fit_under_staircase}}
\end{figure}
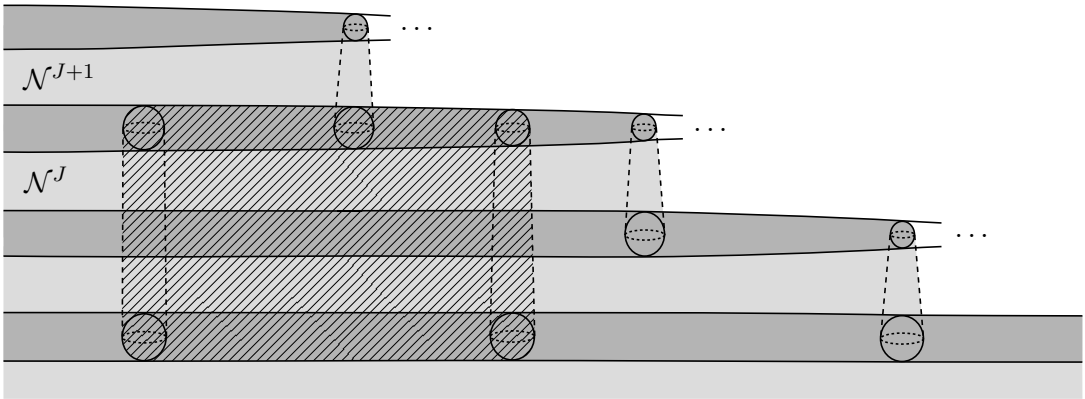
The above construction is feasible only under improved initial control on $|h|$, which necessitates an improved bound of the form $Q\leq \alpha \ov Q$, for some $\alpha \ll 1$, near the boundary components of $\N^{J+1}_{t_J}$.  To verify this improved bound, we apply the strong form of the interior decay estimate, (\ref{eqn_stronger_interior_decay_estimate}), using parabolic regions that are large depending on $\alpha$.  This requires the geometry of the staircase pattern of the comparison domain to be ``flat enough'' to create enough space for such a parabolic region ``under the staircase'' (see Figure~\ref{fig_fit_under_staircase}).  Such flatness can be guaranteed, provided the half-neck collars are sufficiently precise.

\subsection{Defining \texorpdfstring{$\phi$}{{\textbackslash}phi} on extension caps} \label{subsec_constr_process_III_outline}
We recall from Subsection~\ref{subsec_constr_process_I_outline} that in the inductive construction of the time slab $\N^{J+1}$, we sometimes encounter \emph{extension caps},  i.e. $3$-disks $\C$ in  the time-$t_{J}$-slice $\M_{t_{J}}$  of $\M$ that belong to time slab $\N^{J+1}$, but  that were not present in the preceding time slab $\N^{J}$.  In this and the next subsection, we discuss how these extension caps are handled in the second stage of the induction step, in which $\phi^{J+1}$ is defined on $\N^{J+1}$.

Recall that we assume inductively that the map $\phi^J$, as  constructed in the previous step,  restricts to an almost isometric map  of the final time-slice $\N^{J}_{t_{J}}$ of $\N^{J}$ into $\M'_{t_{J}}$.   
We would like to proceed with the construction of $\phi^{J+1}$ on $\N^{J+1}$ using harmonic map heat flow, as described in the previous subsection.  However, in order to do this, $\phi^{J+1}$ must be defined on the initial time-slice $\N^{J+1}_{t_{J}}$, whereas the previous induction step only determined $\phi^J$ on the complement of the extension caps.  
Thus we must first extend $\phi^J$ over the extension caps to an almost isometry defined on $\N^{J+1}_{t_{J}}$.

A priori, it is unclear why such an extension should exist;  after all, since $\phi$ has thus far only been defined on $\N^{J}_{t_{J}}$, one might not expect an extension cap $\C\subset\N^{J+1}_{t_{J}}$ to be nearly isometric to a corresponding $3$-ball region in $\M'$.   

To obtain such an extension, we will need to combine several ingredients. 
The first is the canonical neighborhood assumption, which asserts that the geometry of $\M$ and $\M'$ near any point of large curvature is well-approximated by a model Ricci flow ---  a $\kappa$-solution.  
For regions such as extension caps, the $\kappa$-soliton model is a Ricci flow on $\R^3$. 
Up to rescaling, the only known example of this type is the \emph{Bryant soliton}, a rotationally symmetric steady gradient soliton, which can be expressed as a warped product
\[ g_{\Bry} = dr^2 + a^2 (r) g_{S^2}, \]
where $a (r) \sim \sqrt{r}$ as $r \to \infty$.  Bryant solitons commonly occur as singularity models of Type-II blowups of singularity models, for example in the formation of a degenerate neck pinch \cite{gu_zhu_type_II,knopf_et_al_degenerate_neckpinches}.  Moreover, they also occur in Ricci flow spacetimes when a singularity resolves.  It is a well-known conjecture of Perelman that the Bryant soliton is the only $\kappa$-solution on $\R^3$, up to rescaling and isometry.  
This conjecture would imply that a Bryant soliton \emph{always}   describes singularity formation/resolution processes as above, and in particular the geometry of extension caps.  Although this conjecture remains open, by using a combination of rigidity results of Hamilton and Brendle \cite{Hamilton:1993em,Brendle:2013dd}, it is possible to show that Bryant solitons always describe the geometry at points where the curvature scale increases in time (i.e. where the scalar curvature decreases).  Such points are abundant near a resolution of a singularity, as the curvature scale increases from zero to a positive value.

The above observation will be central to our treatment of extension caps.
We will show that it is possible to choose the time slabs $\{\N^j\}$ so that each extension  cap arises ``at the right time'', meaning at a time when the geometry near the extension cap in $\M$ and its counterpart in $\M'$ is sufficiently close to the Bryant soliton --- at possibly different scales.   The main strategy behind this choice of time will be to choose two different thresholds for the curvature scale on $\N$, specifying when an extension cap \emph{may}, and when it \emph{must}, be constructed.
As curvature scales only grow slowly in time (with respect to the time scale corresponding to the curvature), this extra play will produce sufficiently many time-steps during which an extension cap may, but need not necessarily be constructed.
It can be shown that at one of these time-steps the geometry in both $\M$ and $\M'$ is in fact close to a Bryant soliton.
This time-step will then be chosen as the ``right time'' for the construction of the extension cap.

The fact that  the geometry near both the extension caps in $\M$ and the corresponding regions in $\M'$ can be described by the same singularity model (the Bryant soliton) is necessary in order to construct the initial time-slice of $\phi^{J+1}$.
However, it is not sufficient, as it is still not guaranteed that $\phi^{J}$ at time $t_J$ extends over the extension caps almost isometrically, due to the following reasons:
\begin{itemize}
\item The scales of the approximate Bryant soliton regions in $\M$ and $\M'$ may differ, so that they are not almost isometric.
\item Even if there is an almost isometry of the approximate Bryant soliton regions, in order to define a global map, there must be an almost isometry that is close enough to the existing almost isometry (given by $\phi^{J}$) on the overlap, so that the two maps may be glued together to form an almost isometry.  
\end{itemize}
These issues will be resolved by the \emph{Bryant Extension Principle}, which will be discussed in the next subsection.

\subsection{The Bryant Extension Principle}
\label{subsec_overview_bryant_extension_principle}

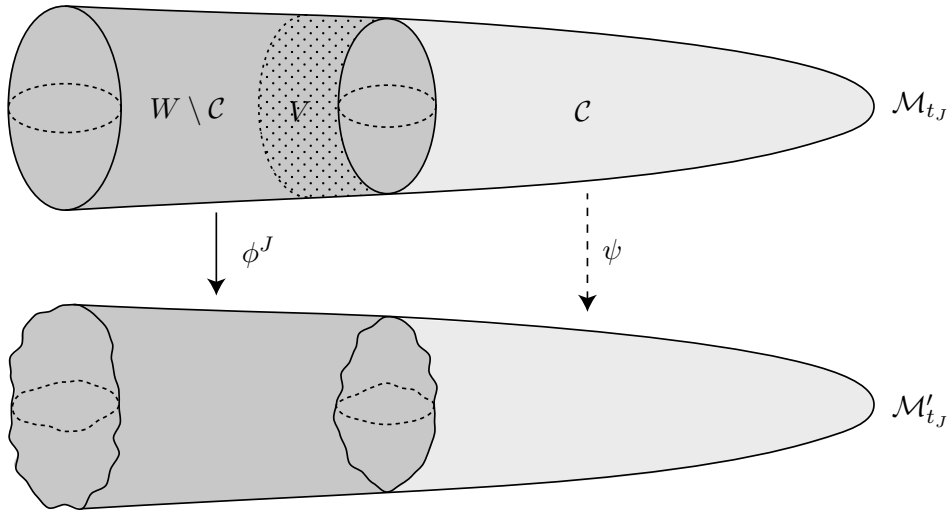
\begin{figure}
\labellist
\small\hair 2pt
\pinlabel $\M_{t_J}$ at 2210 980
\pinlabel $\M'_{t_J}$ at 2210 250
\pinlabel $\psi$ at 1470 630
\pinlabel $\phi^J$ at 620 630
\pinlabel {$W \setminus \C$} at 450 970
\pinlabel $V$ at 720 970
\pinlabel $\C$ at 1400 970
\endlabellist
\centering
\includegraphics[width=130mm]{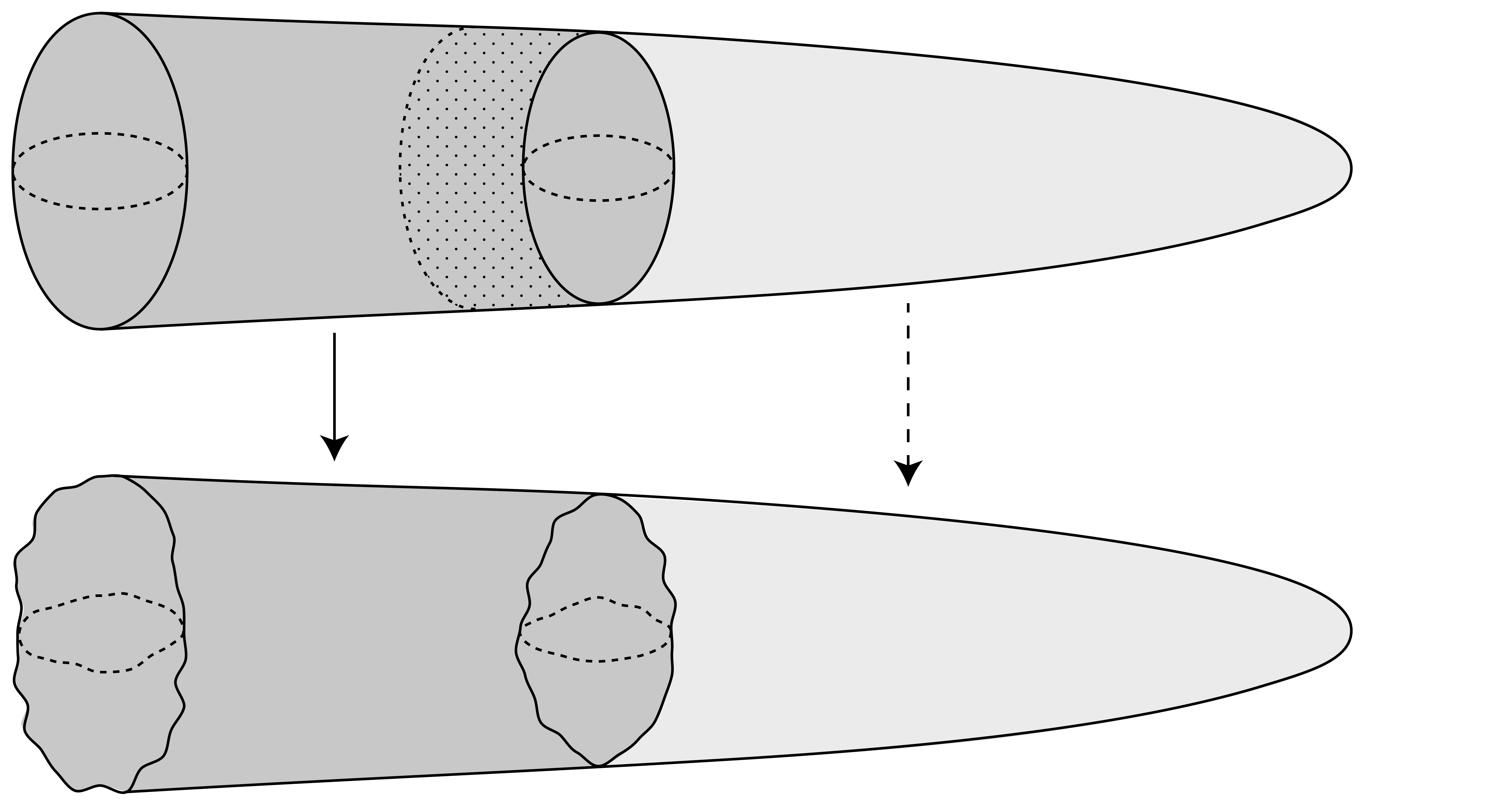}
\caption{Extending the map $\phi^J_{W \setminus \C}$ over the extension cap $\C$ to an almost isometry $\psi : W \to W'$.\label{fig_Bryant_Extension_Outline}}
\end{figure}

In the process of determining the initial data (at time $t_J$) of $\phi^{J+1}$ on or near the extension caps, as mentioned in the previous subsection, we are faced with the following task (see Figure~\ref{fig_Bryant_Extension_Outline} for an illustration).
We can find two regions $W \subset \M_{t_J}$ and $W' \subset \M'_{t_J}$ in the time-$t_J$-slices of $\M$ and $\M'$ that are each geometrically close to a Bryant soliton modulo rescaling by some constants $\lambda$ and $\lambda'$, respectively.
Moreover, the region $W$ contains an extension cap $\C \subset \M_{t_J}$.
The map $\phi^J$ restricted to $W \setminus \C$ is an almost isometric map $W \setminus \C \to W'$.
Our task is then to find another almost isometric map $\psi: W \to W'$, which is defined on the entire region $W$, and that coincides with $\phi^J$ away from some neighborhood of $\C$.  Although in this overview we have largely avoided any mention of quantitative features of the proof, we point out that this step hinges on  careful consideration of asymptotics, in order to make our construction independent of the diameters of $W$, $W'$, and $\C$.  In particular, it turns out to be of fundamental importance that we have the freedom to choose the exponent $E$ in the definition of $Q$ in (\ref{eqn_def_q_overview}) to be large.

We obtain $\psi$ as follows.   We use the fact that $W$ and $W'$ are approximate (rescaled) Bryant soliton regions to define an approximate homothety $\psi_0:W\ra W'$ that scales distances by the factor $\lambda'/\lambda$,  possibly after shrinking $W$, $W'$ somewhat.  
The map $\psi_0$ is unique up to pre/post-composition with almost isometries, i.e. approximate rotations around the respective tips.  
We then compare $\psi_0$ with $\phi^J$ on $W\setminus\C$, and argue that $\psi_0$ may be chosen so that it may be glued to $\phi^J$, to yield the desired map $\psi$. 
 To do this, we must show that: 
 \begin{itemize}
 \item $\psi_0$ is an approximate isometry not just an approximate homothety, i.e. the ratio of the scales $\lambda'/\lambda$  nearly  equals $1$.
 \item $\psi_0$ may be chosen to be sufficiently close to $\phi^J$ on a suitably chosen transition zone $V\subset W \setminus \C$.  
\end{itemize}
However, we are only given information on the map $\phi^J$ far away from tip of the extension cap $\C$, where the metric is close to a round cylinder. 
Using this information, we must determine to within small error the scale of the tips and the discrepancy between the two maps. 
This aspect makes our construction quite delicate, because the only means of detecting the scale of the tip is to measure the deviation from a cylindrical geometry near $V$, which is decaying polynomially in terms of the distance to the tip.  
The crucial point in our construction is that we can arrange things so that this deviation can be measured to within an error that decays at a faster polynomial rate. 

We now explain in some more detail the delicacy of the construction and our strategy for the case of showing that $\lambda' / \lambda$ is nearly equal to $1$.
The problem of matching $\phi^J$ and $\psi_0$ on $V$ will be handled similarly.
The only means to compare $\lambda$ and $\lambda'$ is the almost isometry $\phi^J : W \setminus \C \to W'$.
This almost isometry implies that the cross-sectional spheres of $W \setminus \C$ have approximately the same diameter as their images in $W'$.
Unfortunately, the closeness of these diameters does not imply a bound on $\lambda' / \lambda$, since these diameters vary --- and even diverge --- as we move away from the tips of $W$ and $W'$, and $\phi^J$ may map cross-sectional spheres of $W$ to other almost-cross-sectional spheres in $W'$ that are closer or farther from the tip.
This fact requires us to estimate the \emph{deviations} from the cylindrical geometry in $W$ and $W'$ by analyzing the precise asymptotics of the Bryant soliton.
If the precision of $\phi^J$ is smaller than these deviations, then $\phi^J$ can be used to compare further geometric quantities on $W$ and $W'$, not just the diameters of the cross-sectional spheres.
Combined with the almost preservation of the diameters of these spheres, this will imply that $\lambda' / \lambda \approx 1$.

The precision of the almost isometry $\phi^J$ is measured in terms of $|h|$.
Using the bound $Q \leq \ov{Q}$, as discussed in Subsection~\ref{subsec_controlling_the_perturbation}, we obtain a bound of the form
\begin{equation} 
\label{eq_h_RE_rhoE}
 |h| \lesssim R^{E/2}  \lesssim \rho^{-E}\,. 
\end{equation}
Since  $W$ is an approximate rescaled Bryant soliton region and  $\rho \to \infty$ as one goes to infinity on  the Bryant soliton, the bound (\ref{eq_h_RE_rhoE}) improves as we move further away from the tip of $W$.  
If the exponent $E$ is chosen large enough, then the precision of the almost isometry $\phi^J$ in the transition zone $V\subset W$ is good enough to compare the deviations from a cylinder in $V$ and its image, to very high accuracy.  
As mentioned before, this will imply that $\lambda' / \lambda \approx 1$.

For more details, we refer to Section~\ref{sec_extending_map_between_bryant_solitons}.

We mention that the mechanism that we are exploiting here can be illustrated  using a cantilever: the longer the cantilever, and the less rigid it is, the more its tip may wiggle.  However, the rigidity of a cantilever depends not only on its length, but also on the rigidity of the attachment at its base.  A longer cantilever may be more stable than a short one, as long as the attachment at its base is chosen rigid enough to compensate for the increase in length. (Here we are assuming the lever itself to be infinitely rigid.)

\subsection{Further discussion of the proof} \label{subsec_constr_process_IV_outline}
In this subsection we touch on a few additional features of the induction argument sketched in Subsection~\ref{subsec_constr_process_I_outline}.  Due to the complexity of the underlying issues, our explanations will be brief and relatively vague. For more details, we refer to Section~\ref{sec_apas}.

We recall the bound $Q \leq \ov{Q}$ from Subsection~\ref{subsec_controlling_the_perturbation}, which enabled us to guarantee a bound of the form $|h| \leq \eta_{\lin}$ in most parts of the comparison domain $\N$.
As discussed in that subsection, this bound is propagated forward in time using the interior decay estimate.
The bound $Q \leq \ov{Q}$, especially the factor $\rho^E$ in its definition, was also crucial in the Bryant Extension process.
In fact, it could be used to construct the initial time-slice of a new almost isometry $\phi^{J+1}$, defined on the extension caps, that is precise enough such that a bound of the form $|h| \leq \eta_{\lin}$ holds on or near each extension cap.

However, the bound $Q \leq \ov{Q}$, which is typically stronger than $|h| \leq \eta_{\lin}$, may not remain preserved during the Bryant Extension Process; it may deteriorate by a fixed factor.
In order to control the quality of the comparison map, measured by $|h|$, after the Bryant Extension process has been performed, we will  consider an additional bound of the form
\[ Q^* \approx e^{H (T-t)} \frac{|h|}{R^{3/2}} \approx e^{H (T-t)} \rho^3 |h| \leq \ov{Q}^*. \]
The constant $\ov{Q}^*$ will be chosen such that this bound implies the bound $|h| \leq \eta_{\lin}$ wherever it holds on $\N$.
Due to the small exponent $3 \ll E$, which makes the bound $Q^* \leq \ov{Q}^*$ weaker than $Q \leq \ov{Q}$ at large scales, this bound still holds on and near each extension cap after the Bryant Extension process has been carried out.

Both bounds, $Q \leq \ov{Q}$ and $Q^* \leq \ov{Q}^*$ will be propagated forward in time via the interior decay estimate from Subsection~\ref{subsec_controlling_the_perturbation}.
The bound $Q \leq \ov{Q}$ will hold at all points on the comparison domain $\N$ that are sufficiently far (in space and forward in time) from an extension cap, while the bound $Q^* \leq \ov{Q}^*$ will hold sufficiently far (in space) from the neck-like boundary of $\N$.
It will follow that, for a good choice of parameters, at least one of these bounds holds at each point of the comparison domain $\N$.
This fact will enable us to guarantee that $|h| \leq \eta_{\lin}$ everywhere on $\N$.

Even though the bound $Q \leq \ov{Q}$ may not hold in the near future of an extension cap, it may be important that it holds at \emph{some time} in the future,  thus allowing us to control the comparison map near a future neck-like boundary component, as described in Subsection~\ref{subsec_constr_process_II_outline}.
In order to guarantee this bound near such neck-like boundary components, in the future of extension caps, we first ensure that the neck-like boundary and the extension caps of the comparison domain are sufficiently separated (in space and time).
Then we use the strong form of the interior decay estimate to show that a weak bound of the form $Q \leq W \ov{Q}$, $W \gg 1$, which holds after a Bryant Extension process, improves as we move forward in time and eventually implies $Q \leq \ov{Q}$.
This interior decay estimate relies on the fact that $|h| \leq \eta_{\lin}$, which is guaranteed by the bound $Q^* \leq \ov{Q}^*$.

\section{Organization of the paper}  
The theorems stated in the introduction are proven in Section~\ref{sec_proofs_of_main_results}.  They are all consequences of a more technical stability theorem, Theorem~\ref{Thm_existence_comparison_domain_comparison}, which first appears in Section~\ref{sec_proofs_of_main_results}. This theorem asserts the existence of a comparison map between two Ricci flow spacetimes, satisfying a (large) number of geometric and analytic bounds.      As explained in the overview in the preceding section, Theorem~\ref{Thm_existence_comparison_domain_comparison} is proven using a simultaneous induction argument, in which the domain of this comparison map and the comparison map itself are constructed.  The induction step consists of two stages; the first one is concerned with the comparison domain, and the second with the comparison map.  These two stages are implemented in Sections \ref{sec_inductive_step_extension_comparison_domain} and \ref{sec_extend_comparison}, respectively, and the induction hypotheses are collected beforehand as a set of a priori assumptions, in Section~\ref{sec_apas}.  Both induction steps are formulated using objects and terminology that are introduced in preliminaries sections, Sections~\ref{sec_preliminaries_I} and \ref{sec_preliminaries_II}.  
The arguments in Section~\ref{sec_extend_comparison} rely on two main ingredients: the interior decay estimate, which is discussed in Section~\ref{sec_semi_local_max}, and the Bryant Extension Principle, which is presented in Section~\ref{sec_extending_map_between_bryant_solitons}. 
We will also make use of a number of technical tools, which  appear in Section~\ref{sec_preparatory_results}.

To facilitate readability and verifiability, we have made an effort to make the proof modular and hierarchical.  This eliminates unnecessary interdependencies, and minimizes the number of details the reader must bear in mind at any given stage of the proof.   For instance, the two stages of the induction argument are formulated so as to be completely logically independent of each other.  Also, within Section~\ref{sec_extend_comparison}, which constructs the inductive extension of the comparison map, the argument is split into several pieces, which have been made as independent of one another as possible.

\section{Conventions}\label{sec_notation_terminology}
\subsection{Orientability}
Throughout the paper we impose a blanket assumption that all $3$-manifolds are orientable.  The results remain true without this assumption --- for instance Theorem~\ref{thm_uniquness} can be deduced from the orientable case by passing to the orientation cover.  However, proving the main result without assuming orientability would complicate the exposition by increasing the number of special cases in many places.   It is fairly straightforward, albeit time consuming, to modify the argument to obtain this extra generality.

\subsection{Conventions regarding parameters}
The statements of the a priori assumptions in Section~\ref{sec_apas} involve a number of parameters, which will have to be chosen carefully.
We will not assume these parameters to be fixed throughout the paper; instead, in each theorem, lemma or proposition we will include a list of restrictions on these parameters that serve as conditions for the hypothesis to hold. These restrictions state that certain parameters must be bounded from below or above by  functions  depending on certain other parameters.  When we prove the main stability result, Theorem~\ref{Thm_existence_comparison_domain_comparison}, by combining our two main propositions, Propositions~\ref{prop_comp_domain_construction} and \ref{Prop_extend_comparison_by_one}, we will need to verify that these restrictions are compatible with one another.   This can be verified most easily via the parameter order, as introduced and discussed in Subsection~\ref{subsec_parameter_order}. 
This parameter order is chosen in such a way that the required bounds from below/above on each parameter, if any, are given by a function depending on parameters that precede it.
Hence, in order to verify the compatibility of all restrictions, it suffices to check that each parameter restriction is compatible with the parameter order in this way.

Throughout the entire paper, we will adhere to the convention that small (greek or arabic) letters stand for parameters that have to be chosen small enough and capital (greek or arabic) letters stand for parameters that have to be chosen sufficiently large.
When stating theorems, lemmas or propositions, we will often express restrictions on parameters in the form
\[ y \leq \ov{y} (x), \qquad Z \geq \underline{Z} ( x). \]
By this we mean that there are constants $\ov{y}$ and $\underline{Z}$, depending only on $x$ such that if $0 < y \leq \ov{y}$ and $Z \geq \underline{Z}$, then the subsequent statements hold.
Furthermore, in longer proofs, we will introduce a restriction on parameters in the same form as a displayed equation.
This makes it possible for the reader to check quickly that these restrictions are accurately reflected in the preamble of the theorem, lemma or proposition.
Therefore, she/he may direct their full attention to the remaining details of the proof during the first reading.

\section{Preliminaries I} \label{sec_preliminaries_I}
In the following we define most of the notions that are needed in the statement of the main results of this paper, as stated in Subsection~\ref{subsec_statement_main_results}.

\begin{definition}[Ricci flow spacetimes] \label{def_RF_spacetime}
A {\bf Ricci flow spacetime}  is a tuple $(\M, \lb \mathfrak{t}, \lb \partial_{\mathfrak{t}}, \lb g)$ with the following properties:
\begin{enumerate}[label=(\arabic*)]
\item $\M$ is a smooth $4$-manifold with (smooth) boundary $\partial \M$.
\item $\mathfrak{t} : \M \to [0, \infty)$ is a smooth function without critical points (called {\bf time function}).
For any $t \geq 0$ we denote by $\M_t := \mathfrak{t}^{-1} (t) \subset \M$ the {\bf time-$t$-slice} of $\M$.
\item $\M_0 = \mathfrak{t}^{-1} (0) = \partial \M$, i.e. the initial time-slice is equal to the boundary of $\M$.
\item $\partial_{\mathfrak{t}}$ is a smooth vector field (the {\bf time vector field}), which satisfies $\partial_{\mathfrak{t}} \mathfrak{t} \equiv 1$.
\item $g$ is a smooth inner product on the spatial subbundle $\ker (d \mathfrak{t} ) \subset T \M$.
For any $t \geq 0$ we denote by $g_t$ the restriction of $g$ to the time-$t$-slice $\M_t$ (note that $g_t$ is a Riemannian metric on $\M_t$).
\item $g$ satisfies the Ricci flow equation: $\mathcal{L}_{\partial_\mathfrak{t}} g = - 2 \Ric (g)$.
Here $\Ric (g)$ denotes the symmetric $(0,2)$-tensor on $\ker (d \mathfrak{t} )$ that restricts to the Ricci tensor of $(\M_t, g_t)$ for all $t \geq 0$.
\end{enumerate}
For any interval $I \subset [0, \infty)$ we also write $\M_{I} = \mathfrak{t}^{-1} (I)$ and call this subset the {\bf time-slab} of $\M$ over the time-interval $I$.  
Curvature quantities on $\M$, such as the Riemannian curvature tensor $\Rm$, the Ricci curvature $\Ric$, or the scalar curvature $R$ will refer to the corresponding quantities with respect to the metric $g_t$ on each time-slice.
Tensorial quantities will be imbedded using the splitting $T\M = \ker (d\mathfrak{t} ) \oplus \langle \partial_{\mathfrak{t}} \rangle$.

When there is no chance of confusion, we will sometimes abbreviate the tuple $(\M, \mathfrak{t}, \partial_{\mathfrak{t}}, g)$ by $\M$.
\end{definition}

Ricci flow spacetimes were introduced by Lott and the second author (see \cite{Kleiner:2008fh}).
The definition above is  almost verbatim that of \cite{Kleiner:2008fh} with the exception that we require Ricci flow spacetimes to have initial time-slice at time $0$ and no final time-slice.
This can always be achieved by applying a time-shift and removing the final time-slice from $\M$.
Ricci flows with surgery, as constructed by Perelman in \cite{Perelman:2003tka}, can be turned easily into Ricci flow spacetimes by removing a relatively small subset of surgery points.
See (\ref{eq_MM_construction_from_RFws}) in Subsection~\ref{subsec_status_quo} for further explanation.

We emphasize that, while a Ricci flow spacetime may have singularities --- in fact the sole purpose of our definition is to understand flows with singularities --- such singularities are not directly captured by a Ricci flow spacetime, as ``singular points'' are not contained in the spacetime manifold $\M$.
Instead, the idea behind the definition of a Ricci flow spacetime is to understand a possibly singular flow by analyzing its asymptotic behavior on its regular part. 

Any (classical) Ricci flow of the form $(g_t)_{t \in [0,T)}$, $0 < T \leq \infty$ on a $3$-manifold $M$ can be converted into a Ricci flow spacetime  by setting $\M = M \times [0,T)$, letting $\mathfrak{t}$ be the projection to the second factor and letting $\partial_{\mathfrak{t}}$ correspond to the unit vector field on $[0,T)$.
Vice versa, if $(\M, \mathfrak{t}, \partial_{\mathfrak{t}}, g)$ is a Ricci flow spacetime with $\mathfrak{t}(\M) = [0, T)$ for some $0 < T \leq \infty$ and the property that every trajectory of $\partial_{\mathfrak{t}}$ is defined on the entire time-interval $[0,T)$, then $\M$ comes from such a classical Ricci flow.

We now generalize some basic geometric notions to Ricci flow spacetimes.

\begin{definition}[Length, distance and metric balls in Ricci flow spacetimes]
Let $(\M, \mathfrak{t}, \partial_{\mathfrak{t}}, g)$ be a Ricci flow spacetime.
For any two points $x, y \in \M_t$ in the same time-slice of $\M$ we denote by $d(x,y)$ or $d_t (x,y)$ the {\bf distance} between $x, y$ within $(\M_t, g_t)$.
The distance between points in different time-slices is not defined.

Similarly, we define the {\bf length} $\length (\gamma)$ or $\length_t (\gamma)$ of a path $\gamma : [0,1] \to \M_t$ whose image lies in a single time-slice to be the length of this path when viewed as a path inside the Riemannian manifold $(\M_t, g_t)$.

For any $x \in \M_t$ and $r \geq 0$ we denote by $B(x,r) \subset \M_t$ the {\bf $r$-ball} around $x$ with respect to the Riemannian metric $g_t$.
\end{definition}

Our next goal is to characterize the (microscopic) geometry of a Ricci flow spacetime near a singularity or at an almost singular point.
For this purpose, we will introduce a \textbf{(curvature) scale function} $\rho : \M \to (0, \infty]$ with the property that
\begin{equation} \label{eq_rho_equivalent_Rm}
 C^{-1} \rho^{-2} \leq |{\Rm}| \leq C \rho^{-2} 
\end{equation}
for some universal constant $C < \infty$.
The quantity $\rho$ will be a (pointwise) function of the curvature tensor and therefore it can also be defined on (3-dimensional) Riemannian manifolds.
For the purpose of this section, it suffices to assume that $\rho = |{\Rm}|^{-1/2}$.
However, in order to simplify several proofs in subsequent sections, we will work with a slightly more complicated definition of $\rho$, which we will present in Subsection~\ref{subsec_curvature_scale} (see Definition~\ref{def_curvature_scale}).
Nonetheless, the discussion in the remainder of this subsection and the main results of the paper, as presented in Subsection~\ref{subsec_statement_main_results}, remain valid for any definition of $\rho$ that satisfies (\ref{eq_rho_equivalent_Rm}).

We now define what we mean by completeness for Ricci flow spacetimes.
Intuitively, a Ricci flow spacetime is called complete if its time-slices can be completed by adding countably many ``singular points'' and if no component appears or disappears suddenly without the formation of a singularity.

\begin{definition}[Completeness of Ricci flow spacetimes] \label{def_completeness}
We say that a Ricci flow spacetime $(\M,\mathfrak{t}, \partial_{\mathfrak{t}}, g)$ is {\bf $(r_0,t_0)$-complete}, for some $r_0, t_0 \geq 0$, if the following holds:
Consider a path $\gamma : [0, s_0) \to \M_{[0,t_0]}$ such that $\inf_{s \in [0,s_0)}  \rho (\gamma(s)) > r_0$ for all $s \in [0,s_0)$ and such that:
\begin{enumerate}
\item Its image $\gamma ([0,s_0))$ lies in a time-slice $\M_t$ and the time-$t$ length of $\gamma$ is finite or
\item $\gamma$ is a trajectory of $\partial_{\mathfrak{t}}$ or of $- \partial_{\mathfrak{t}}$.
\end{enumerate}
Then the limit $\lim_{s \nearrow s_0} \gamma (s)$ exists.

If $(\M, \mathfrak{t}, \partial_{\mathfrak{t}}, g)$ is $(r_0, t_0)$-complete for all $t_0 \geq 0$, then we also say that it is {\bf $r_0$-complete}.
Likewise, if $(\M, \mathfrak{t}, \partial_{\mathfrak{t}}, g)$ is {\bf $0$-complete,} then we say that it is {\bf complete}.
\end{definition}

Note that the Ricci flow spacetimes constructed \cite{Kleiner:2014wz} are $0$-complete, see \cite[Prop. 5.11(a), Def. 1.8]{Kleiner:2014wz}.
A Ricci flow with surgery and $\delta$-cutoff, as constructed by Perelman in \cite{Perelman:2003tka}, can be turned into a Ricci flow spacetime as in (\ref{eq_MM_construction_from_RFws}) that is $c \delta r$-complete for some universal constant $c > 0$, as long as the cutoff is performed in an appropriate way\footnote{As Perelman's objective was the characterization of the underlying topology, he allowed (but did not require) the removal of macroscopic spherical components during a surgery step. In contrast, Kleiner and Lott's version (cf \cite{Kleiner:2008fh}) of the cutoff process does not allow this. However, both cutoff approaches allow some flexibility on the choice of the cutoff spheres inside the $\eps$-horns. Some of these choices may result in the removal of points of scale larger than $c \delta r$; in such a case $c\delta r$-completeness cannot be guaranteed.
Nevertheless, in both approaches it is always possible to perform the cutoff in such a way that the resulting Ricci flow spacetime is $c \delta r$-complete.}, see \cite[Section 3]{Kleiner:2014wz}.

Lastly, we need to characterize the asymptotic geometry of a Ricci flow spacetime near its singularities.
This is done by the canonical neighborhood assumption, a notion which is inspired by Perelman's work (\cite{Perelman:2003tka}) and which appears naturally in the study of 3-dimensional Ricci flows.
The idea is to impose the same asymptotic behavior near singular points in Ricci flow spacetimes as is encountered in the singularity formation of a classical (smooth) 3-dimensional Ricci flow.
The same characterization also holds in high curvature regions of Perelman's Ricci flow with surgery that are far enough from ``man-made'' surgery points.
Furthermore, an even stronger asymptotic behavior was shown to hold on Ricci flow spacetimes as constructed by Lott and the second author in \cite{Kleiner:2014wz}.

The singularity formation in 3-dimensional Ricci flows is usually understood via singularity models called $\kappa$-solutions (see \cite[Sec. 11]{Perelman:2002um}).
The definition of a $\kappa$-solution consists of a list of properties that are known to be true for $3$-dimensional singularity models.
Interestingly, these properties are sufficient to allow a qualitative (and sometimes quantitative) analysis of $\kappa$-solutions.
We refer the reader to Appendix~\ref{app_kappa_solution_properties} and \cite{Perelman:2003tka,Kleiner:2008fh} for further details.

Let us recall the definition of a $\kappa$-solution.

\begin{definition}[$\kappa$-solution] \label{def_kappa_solution}
An ancient Ricci flow $(M, (g_t)_{t \in (-\infty, 0]} )$ on a $3$-dimensional manifold $M$ is called a \textbf{(3-dimensional) $\kappa$-solution}, for $\kappa > 0$, if the following holds:
\begin{enumerate}[label=(\arabic*)]
\item $(M, g_t)$ is complete for all $t \in (- \infty, 0]$,
\item $|{\Rm}|$ is bounded on $M \times I$ for all compact $I \subset ( - \infty, 0]$,
\item $\sec_{g_t} \geq 0$ on $M$ for all $t \in (- \infty, 0]$,
\item $R > 0$ on $M \times (- \infty, 0]$,
\item $(M, g_t)$ is $\kappa$-noncollapsed at all scales for all $t \in (- \infty, 0]$

(This means that for any $(x,t) \in M \times (- \infty, 0]$ and any $r > 0$ if $|{\Rm}| \leq r^{-2}$ on the time-$t$ ball $B(x,t,r)$, then we have $|B(x,t,r)| \geq \kappa r^n$ for its volume.)
\end{enumerate}
\end{definition}

We will compare the local geometry of a Ricci flow spacetime to the geometry of $\kappa$-solution using the following concept of pointed closeness.

\begin{definition}[Geometric closeness] \label{def_geometric_closeness_time_slice}
We say that a pointed Riemannian manifold $(M, g, x)$ is \textbf{$\eps$-close} to another pointed Riemannian manifold $(\ov{M}, \ov{g}, \ov{x})$ \textbf{at scale $\lambda > 0$} if there is a diffeomorphism onto its image
\[ \psi : B^{\ov{M}} (\ov{x}, \eps^{-1} ) \longrightarrow M \]
such that $\psi (\ov{x}) = x$ and
\[ \big\Vert \lambda^{-2} \psi^* g - \ov{g} \big\Vert_{C^{[\eps^{-1}]}(B^{\ov{M}} (\ov{x}, \eps^{-1} ))} < \delta. \]
Here the $C^{[\eps^{-1}]}$-norm  of a tensor $h$ is defined to be the sum of the $C^0$-norms of the tensors $h$, $\nabla^{\ov{g}} h$, $\nabla^{\ov{g},2} h$, \ldots, $\nabla^{\ov{g}, [\eps^{-1}]} h$ with respect to the metric $\ov{g}$.
\end{definition}

We can now define the canonical neighborhood assumption.
The main statement of this assumption is that regions of small scale (i.e. high curvature) are geometrically close to regions of $\kappa$-solutions.

\begin{definition}[Canonical neighborhood assumption] \label{def_canonical_nbhd_asspt}
Let $(M, g)$ be a (possibly incomplete) Riemannian manifold.
We say that $(M, g)$ satisfies the {\bf $\eps$-canonical neighborhood assumption} at some point $x$ if there is a $\kappa > 0$, a $\kappa$-solution $(\overline{M}, \linebreak[1] (\overline{g}_t)_{t \in (- \infty, 0]})$ and a point $\ov{x} \in \ov{M}$ such that $\rho (\overline{x}, 0) = 1$ and such that $(M, g, x)$ is $\eps$-close to $(\ov{M}, \ov{g}_0, \ov{x})$ at some (unspecified) scale $\lambda > 0$.

We say that $(M,g)$ {\bf satisfies the $\eps$-canonical neighborhood assumption at scales $(r_1, r_2)$,} for some $0 \leq  r_1 < r_2$, if every point $x \in M$ with $r_1 < \rho(x) < r_2$ satisfies the $\eps$-canonical neighborhood assumption.

We say that a Ricci flow spacetime $(\M, \mathfrak{t}, \partial_{\mathfrak{t}}, g)$ satisfies the \textbf{$\eps$-ca\-non\-i\-cal neighborhood assumption} at a point $x \in \M$ if the same is true at $x$ in the time-slice $(\M_{\mathfrak{t}(x)}, g_{\mathfrak{t}(x)})$.
Moreover, we say that $(\M, \mathfrak{t}, \partial_{\mathfrak{t}}, g)$ satisfies the {\bf $\eps$-canonical neighborhood assumption at scales $(r_1, r_2)$} if the same is true for all its time-slices.
Lastly, we say that a subset $X \subset \M$ satisfies the \textbf{$\eps$-canonical neighborhood assumption at scales $(r_1, r_2)$}, if the $\eps$-canonical neighborhood assumption holds at all $x \in X$ with $\rho(x) \in (r_1, r_2)$.
\end{definition}

Note that if $\M$ is a Ricci flow spacetime as constructed in \cite{Kleiner:2014wz}, then $\M_{[0,T]}$ satisfies the $\eps$-canonical neighborhood assumption at scales $(0,r)$, where $r=r(\eps,T) >0$ \cite[Thm. 1.3, Prop. 5.30]{Kleiner:2014wz}.
If $\M$ is the Ricci flow spacetime of a Ricci flow with surgery and $\delta$-cutoff, as constructed by Perelman in \cite{Perelman:2003tka}, then $\M$ satisfies the $\eps$-canonical neighborhood assumption at $x\in \M$, provided the scale of $x$ lies in the interval $(10h,r)$.  Here $h=h(\eps,t)$ and $r(\eps,t)$ are decreasing functions of time, which appear in Perelman's construction, $h\leq \delta^2 r$, and $\delta = \delta(\eps, t)$ may be chosen as small as desired.

Observe that we do not assume a global lower bound on $\kappa$ in Definition~\ref{def_canonical_nbhd_asspt}.
This slight generalization from other notions of the canonical neighborhood assumption does not create any serious issues, since by Perelman's work \cite{Perelman:2003tka}, every $3$-dimensional $\kappa$-solution is a $\kappa_0$-solution for some universal $\kappa_0 > 0$, unless it homothetic to a quotient of a round sphere (see assertion \ref{ass_C.1_a} of Lemma~\ref{lem_kappa_solution_properties_appendix} for further details).

We also remark that in Definition~\ref{def_geometric_closeness_time_slice} we have put extra care in describing how the $C^{[\eps^{-1}]}$-norm has to be understood.
The reason for this is that the model metric $\ov{g}$ in Definition~\ref{def_canonical_nbhd_asspt} is not fixed.
So it would be problematic, for example, to define the $C^{[\eps^{-1}]}$-norm using coordinate charts on $\ov{M}$, as the number and sizes of those coordinate charts may depend on the Riemannian manifold $(\ov{M}, \ov{g} )$.

It may seem more standard to require spacetime closeness to a $\kappa$-solution on a backwards parabolic neighborhood --- as opposed to closeness on a ball in a single time-slice --- in the definition of the canonical neighborhood assumption.
Such a condition would be stronger and, as our goal is to establish a uniqueness property, it would lead to a formally less general statement.
We point out that  spacetime closeness to a $\kappa$-solution is a rather straight forward consequence of time-slice closeness.
The main purpose of the use of time-slice closeness in our work is because our uniqueness property also applies to Ricci flow spacetime with singular initial data.
For this reason the canonical neighborhood assumption also has to be applicable to the initial time-slice $\M_0$ or to time-slices $\M_t$ for small $t$.

\section{Preliminaries II}\label{sec_preliminaries_II}
In this section we present basic definitions and concepts that will be important for the proofs of the main results of this paper.

\subsection{Curvature scale} \label{subsec_curvature_scale}
As mentioned in Section~\ref{sec_preliminaries_I}, we will now define a notion of a curvature scale $\rho$ that will be convenient for our proofs.
The main objective in our definition will be to ensure that $\rho = (\frac13 R)^{-1/2}$ wherever the sectional curvature is almost positive.
For this purpose, observe that there is a constant $c_0 > 0$ such that the following holds.
Whenever $\Rm$ is an algebraic curvature tensor with the property that its scalar curvature $R$ is positive and all its sectional curvatures are bounded from below by $-\frac1{10} R$, then $c_0 |{\Rm}| \leq \frac13 R$.
We will fix $c_0$ for the remainder of this paper. 

\begin{definition}[Curvature scale] \label{def_curvature_scale}
Let $(M, g)$ be a 3-dimensional Riemannian manifold and $x \in M$ a point.
We define the {\bf (curvature) scale} at $x$ to be 
\begin{equation} \label{eq_def_curvature_scale}
 \rho (x) = \min \big\{ \big( \tfrac13 R_+(x) \big)^{-1/2}, \big( c_0 |{\Rm}| (x) \big)^{-1/2}\big\}. 
\end{equation}
Here $R_+(x) := \max \{ R(x), 0 \}$ and we use the convention $0^{-1/2} = \infty$.

If $r_0 > 0$, then we set $\rho_{r_0} (x) := \min \{ \rho (x), r_0 \}$.
Lastly, if $(\M, \mathfrak{t}, \partial_{\mathfrak{t}}, g)$ is a Ricci flow spacetime, then we define $\rho, \rho_{r_0} : \M \to \R$ such that they restrict to the corresponding scale functions on the time-slices. 
\end{definition}

\begin{lemma} \label{lem_rho_Rm_R}
There is a universal constant $C < \infty$ such that
\begin{equation} \label{eq_equivalence_bound_rho_Rm}
 C^{-1} \rho^{-2} (x) \leq |{\Rm}|(x) \leq C \rho^{-2} (x). 
\end{equation}
Moreover, there is a universal constant $\eps_0 > 0$ such that if $x$ satisfies the $\eps_{\can}$-canonical neighborhood assumption for some $\eps_{\can} \leq \eps_0$, then $R(x) = 3\rho^{-2} (x)$.
\end{lemma}

\begin{proof}
The bound (\ref{eq_equivalence_bound_rho_Rm}) is obvious.
For the second part of the lemma observe that for sufficiently small $\eps_{\can}$ we have $R(x) > 0$ and $\sec \geq - \frac1{10} R(x)$ at $x$.
So $(\frac13 R_+ (x))^{-1/2} \leq (c_0 |{\Rm}| (x) )^{-1/2}$.
\end{proof}

The normalization constant $\frac13$ in front of the scalar curvature in (\ref{eq_def_curvature_scale}) is chosen purely for convenience.
More specifically, we will frequently consider the following round shrinking cylinder evolving by Ricci flow:
\[ \big( S^2 \times \R, (g_t = (\tfrac23 - 2t) g_{S^2} + g_{\R})_{t \in (- \infty, \frac13]} \big). \]
The scale of this cylinder and the normalization of the curvature scale have been chosen in such a way that  $\rho (\cdot, 0) \equiv 1$ and $\rho (\cdot, -1) \equiv 2$ hold, which can be remembered easily;
more generally, we have
\[ \rho(\cdot, t)  \equiv \sqrt{1-3t}. \]

\begin{definition}[(Weakly) thick and thin subsets]
Let $X$ be a subset of a Riemannian manifold $(M,g)$ or Ricci flow spacetime $(\M, \mathfrak{t}, \partial_{\mathfrak{t}}, g)$ and $r >0$ a number.
We say that $X$ is \textbf{$r$-thick} if $\rho(X) > r$ and \textbf{weakly $r$-thick} if $\rho (X) \geq r$.
Similarly, we say that $X$ is \textbf{$r$-thin} or \textbf{weakly $r$-thin} if $\rho(X) < r$ or $\rho (X) \leq r$, respectively.
\end{definition}

\subsection{Basic facts about the Bryant soliton} \label{subsec_basics_Bryant} 
In the following, we will denote by $(M_{\Bry}, (g_{\Bry, t})_{t \in \R})$ the Bryant soliton and with tip $x_{\Bry} \in M_{\Bry}$ normalized in such a way that $\rho (x_{\Bry}) = 1$.
The Bryant soliton was first constructed \cite{Bryant:gVzfj4nx}.
A more elementary construction can also be found in \cite{Appleton:2017vf}. 
Recall that $(M_{\Bry}, (g_{\Bry, t})_{t \in \R})$ is a steady gradient soliton all whose time-slices are rotationally symmetric with center $x_{\Bry}$.
More specifically, $(M_{\Bry}, g_{\Bry, t})$ can be expressed as a warped product of the form
\[ g_{\Bry, t} = d\sigma^2 + w_t^2 (\sigma) g_{S^2} \]
where $w_t (\sigma) \sim \sqrt{\sigma}$ for large $\sigma$.
We refer to Lemma~\ref{lem_bryant_geometry} for a more extensive list of properties of the Bryant soliton that are being used in this paper.
Note that, due to the normalization of $M_{\Bry}$, the definition of $\rho$ and (\ref{eqn_soliton_conserved}) of Lemma~\ref{lem_bryant_geometry}, we have $\rho \geq 1$ on $M_{\Bry}$.
This fact will be important in this paper.

We will set $g_{\Bry} := g_{\Bry, 0}$ for the time-$0$-slice of the Bryant soliton.
Furthermore, we will denote by $M_{\Bry} (r) := B(x_{\Bry},  r)$ the $r$-ball around the tip with respect to $g_{\Bry}$ and for $0 < r_1 < r_2$, we will denote by $M_{\Bry} (r_1, r_2)$ the open $(r_1, r_2)$-annulus around $x_{\Bry}$.

\subsection{Geometry of Ricci flow spacetimes}

The goal of this subsection is to introduce several notions that we will frequently use in order to describe points or subsets in Ricci flow spacetimes.

\begin{definition}[Points in Ricci flow spacetimes] \label{def_points_in_RF_spacetimes}
Let $(\M, \mathfrak{t}, \partial_{\mathfrak{t}}, g)$ be a Ricci flow spacetime and $x \in \M$ be a point.
Set $t := \mathfrak{t} (x)$.
Consider the maximal trajectory $\gamma_x : I \to \M$, $I \subset [0, \infty)$ of the time-vector field $\partial_{\mathfrak{t}}$ such that $\gamma_x (t) = x$.
Note that then $\mathfrak{t} (\gamma_x(t')) = t'$ for all $t' \in I$.
For any $t' \in I$ we say that $x$ \textbf{survives until time $t'$} and we write 
\[ x(t') := \gamma_x (t'). \]

Similarly, if $X \subset \M_t$ is a subset in the time-$t$ time-slice, then we say that $X$ \textbf{survives until time $t'$} if this is true for every $x \in X$ and we set $X(t') := \{ x(t') \;\; : \;\; x \in X \}$.
\end{definition}

We will also use the following two notions.

\begin{definition}[Time-slice of a subset]
Let $(\M, \mathfrak{t}, \partial_{\mathfrak{t}}, g)$ be a Ricci flow spacetime and let $X \subset \M$ be a subset.
For any time $t \in [0, \infty)$ we define the \textbf{time-$t$-slice} of $X$ to be $X_t := X \cap \M_t$ and for any interval $I \subset [0, \infty)$ we define the \textbf{$I$-time-slab} of $X$ to be $X_I := X \cap \M_I$.
\end{definition}

\begin{definition}[Product domain]
\label{def_product_domain}
Let $(\M, \mathfrak{t}, \partial_{\mathfrak{t}}, g)$ be a Ricci flow spacetime and let $X \subset \M$ be a subset.
We call $X$ a \emph{product domain} if there is an interval $I \subset [0, \infty)$ such that for any $t \in I$ any point $x \in X$ survives until time $t$ and $x(t) \in X$.
\end{definition}

Note that a product domain $X$ can be identified with the product $X_{t_0} \times I$ for an arbitrary $t_0 \in I$.
If $X_{t_0}$ is sufficiently regular (e.g. open or a domain with smooth boundary in $\M_{t_0}$), then the metric $g$ induces a classical Ricci flow $(g_t)_{t \in I}$ on $X_{t_0}$.
We will often use the metric $g$ and the Ricci flow $(g_t)_{t \in I}$ synonymously when our analysis is restricted to a product domain.

\begin{definition}[Parabolic neighborhood]
Let $(\M, \mathfrak{t}, \partial_{\mathfrak{t}}, g)$ be a Ricci flow spacetime.
For any $y \in \M$ let $I_y \subset [0, \infty)$ be the set of all times until which $y$ survives.
Now consider a point $x \in \M$ and two numbers $a \geq 0$, $b \in \R$.
Set $t := \mathfrak{t} (x)$.
Then we define the \textbf{parabolic neighborhood} $P(x, a, b) \subset \M$ as follows:
\[ P(x,a,b) := \bigcup_{y \in B(x,a)} \bigcup_{t' \in [t, t+b] \cap I_y} y(t'). \]
If $b < 0$, then we replace $[t,t+b]$ by $[t+b, t]$.
We call $P(x,a,b)$ \textbf{unscathed} if $B(x,a)$ is relatively compact in $\M_t$ and if $I_y \supset [t, t+b]$ or $I_y\supset [t +b, t] \cap [0, \infty)$ for all $y \in B(x,a)$.
Lastly, for any $r > 0$ we introduce the simplified notation
\[ P(x,r) := P(x,r,-r^2) \]
for the \textbf{(backward) parabolic ball} with center $x$ and radius $r$.
\end{definition}

Note that if $P(x,a,b)$ is unscathed, then it is a product domain of the form $B(x,a)  \times I_y$ for any $y \in B(x,a)$.
We emphasize that $P(x,a,b)$ can be unscathed even if $t+b < 0$, that is when it hits the initial time-slice earlier than expected.
So an unscathed parabolic neighborhood is not necessarily of the form $B(x,a) \times [t+b, t]$ if $b < 0$.

\subsection{Necks }
Borrowing from Definition~\ref{def_geometric_closeness_time_slice}, we will introduce the notion of a $\delta$-neck. 
\begin{definition}[$\delta$-neck]
Let $(M,g)$ be a Riemannian manifold and $U \subset M$ an open subset.
We say that $U$ is a {\bf $\delta$-neck at scale $\lambda > 0$} if there is a diffeomorphism
\[ \psi : S^2 \times \big( {- \delta^{-1}, \delta^{-1} }\big) \longrightarrow U \]
such that
\[ \big\Vert \lambda^{-2} \psi^* g - \big( \tfrac23 g_{S^2} + g_{\R} \big) \big\Vert_{C^{[\delta^{-1}]}(S^2 \times (- \delta^{-1}, \delta^{-1}))} < \delta. \]
We call the image $\psi ( S^2 \times \{ 0 \})$ a {\bf central 2-sphere of $U$} and every point on a central $2$-sphere a {\bf center of $U$}.
\end{definition}

Note that by our convention (see Definition~\ref{def_curvature_scale}) we have $\rho \equiv 1$ on $(S^2 \times \R,  \tfrac23 g_{S^2} + g_{\R})$.
So on a $\delta$-neck at scale $\lambda$ we have $\rho \approx \lambda$, where the accuracy depends on the smallness of $\delta$.
We also remark that a $\delta$-neck $U$ has infinitely many central $2$-spheres, as we may perturb $\psi$ slightly.
This is why we speak of \emph{a} central 2-sphere of $U$, as opposed to \emph{the} central 2-sphere.
Similarly, the centers of $U$ are not unique, but form an open subset of $U$.

\subsection{Ricci-DeTurck flow and harmonic map heat flow} \label{subsec_RdT_hmhf}
In this subsection we recall some of the basic facts about the harmonic map heat flow and the Ricci-DeTurck flow equation in the classical setting, which were first observed by DeTurck \cite{DeTurck:1983jp} and Hamilton \cite[Sec.6]{hamilton_formation}.
More details, including precise statements of short-time existence and regularity of these flows, can be found in Appendix~\ref{appx_Ricci_deT}.

Consider two $n$-dimensional manifolds $M, M'$, each equipped with a smooth family of Riemannian metrics $(g_t)_{t \in [0,T]}$, $(g'_t)_{t \in [0,T]}$.
Let moreover $(\chi_t)_{t \in [0,T]}$, $\chi_t : M' \to M$ be a smooth family of maps.

\begin{definition} \label{def_classical_hmh_flow}
We say that the family $(\chi_t)_{t \in [0,T]}$ moves by {\bf harmonic map heat flow between $(M', g'_t)$ and $(M,g_t)$} if it satisfies the following evolution equation:
\begin{equation} \label{eq_hmh_flow_def}
 \partial_t \chi_t = \triangle_{g'_t, g_t} \chi_t = \sum_{i=1}^n \big( \nabla^{g_t}_{d\chi_t (e_i)} d\chi_t (e_i) - d\chi_t ( \nabla^{g'_t}_{e_i} e_i ) \big),  
\end{equation}
where $\{ e_i \}_{i=1}^n$ is a local frame  on $M'$ that is orthonormal with respect to $g'_t$.
\end{definition}

Assume now for the remainder of this subsection that $(g_t)_{t \in [0,T]}$ and $(g'_t)_{t \in [0,T]}$ evolve by the Ricci flow equations
\[ \partial_t g_t = -2 \Ric_{g_t}, \qquad \partial_t g'_t = -2 \Ric_{g'_t}. \]
Furthermore, assume for the rest of this subsection that all the maps $\chi_t$ are diffeomorphisms and consider their inverses $\phi_t := \chi_t^{-1}$.
A basic calculation (see Appendix~\ref{appx_Ricci_deT} for more details) reveals that the pullback $g^*_t := \phi^*_t g'_t$ evolves by the {\bf Ricci-DeTurck flow} equation
\begin{equation} \label{eq_RdT_flow_def}
 \partial_t g^*_t = - 2 \Ric_{g^*_t} - \mathcal{L}_{X_{g_t} (g^*_t)} g^*_t, 
\end{equation}
where the vector field $X_{g_t} (g^*_t)$ is defined by
\begin{equation} \label{eq_def_X}
 X_{g_t} (g^*_t) :=  \triangle_{g^*_t, g_t} \id_M =  \sum_{i=1}^n \big( \nabla^{g_t}_{e_i} e_i - \nabla^{g^*_t}_{e_i} e_i \big), 
\end{equation}
for a local frame $\{ e_i \}_{i=1}^n$ that is orthonormal with respect to $g^*_t$.

The advantage of the Ricci-DeTurck flow equation over the Ricci flow equation is that it is a non-linear, strongly parabolic equation in the metric $g^*_t$.
More specifically, if we express $g^*_t$ in terms of the perturbation $h_t := g^*_t - g_t$, then (\ref{eq_RdT_flow_def}) becomes the {\bf Ricci-DeTurck flow equation for perturbations} 
\begin{equation} \label{eq_Ricci_de_Turck_introduction}
 \nabla_{\partial_t} h_t = \triangle_{g_t} h_t + 2\Rm_{g_t} (h_t) + \mathcal{Q}_{g_t} [ h_t]. 
\end{equation}
Here we view $g_t$ as a background metric.
All curvature quantities and covariant derivatives are taken with respect to $g_t$.
On the left-hand side of (\ref{eq_Ricci_de_Turck_introduction}), we moreover use Uhlenbeck's trick:
\[ (\nabla_{\partial_t} h_t)_{ij} = (\partial_t h_t)_{ij} + g^{pq}_t \big( {\Ric^{g_t}_{pj} ( h_t)_{iq} + \Ric^{g_t}_{ip} (h_t)_{qj} }\big) \]
The expressions on the right-hand side of (\ref{eq_Ricci_de_Turck_introduction}) are to be interpreted as follows:
\[ \big({ \Rm_{g_t} (h_t) }\big)_{ij} = g_t^{pq} R_{p ij}^{\quad u} (h_t)_{qu} \]
and $\mathcal{Q}_{g_t} [h_t]$ is an algebraic expression in $g_t$, $h_t$, $\nabla h_t$, $\nabla^2 h_t$ of the form
\begin{multline*}
 \mathcal{Q}_{g_t} [ h_t ] = (g_t + h_t)^{-1} * (g_t + h_t)^{-1} * \nabla h_t * \nabla h_t \\ 
 + (g_t+h_t)^{-1} * (g_t+h_t)^{-1} * \Rm_{g_t} *h_t * h_t +  ( g_t + h_t)^{-1} * ( g_t +h_t)^{-1} * h_t * \nabla^2 h_t. 
\end{multline*}
See (\ref{eq_QQ_formula}) in Appendix~\ref{appx_Ricci_deT} for an explicit formula for $\mathcal{Q}_{g_t}$.
The precise structure of the quantity $\mathcal{Q}_{g_t}$ will, however, not be of essence in this paper.
 
We remark that in the classical setting and in the compact case, the uniqueness of solutions to the Ricci flow equation follows from the existence of solutions to (\ref{eq_hmh_flow_def}) and the uniqueness of solutions to (\ref{eq_Ricci_de_Turck_introduction}).
More specifically, for any two Ricci flows $(g_t)_{t \in [0,T]}$ and $(g'_t)_{t \in [0, T]}$ on $M$ and $M'$ for which there is an isometry $\chi : M' \to M$ with $\chi^{*} g_0 = g'_0$ one first constructs a solution $(\chi_t)_{t \in [0,\tau)}$ of (\ref{eq_hmh_flow_def}), for some maximal $\tau < T$, with initial condition $\chi_0 = \chi$.
The resulting perturbation $h_t = \phi^*_t g'_t - g_t$, for $\phi_t = \chi_t^{-1}$, solves (\ref{eq_Ricci_de_Turck_introduction}), as long as it is well defined.
As $h_0 \equiv 0$, we obtain by uniqueness that $h_t \equiv 0$, as long as it is defined.
It then follows that $\chi_t$ is an isometry for all $t \in [0,T] = [0, \tau]$ and by (\ref{eq_hmh_flow_def}) that $\partial_t \chi_t \equiv 0$.

In this paper we will mostly analyze solutions $h_t$ to (\ref{eq_Ricci_de_Turck_introduction}) of small norm.
Via a limit argument, such solutions can be understood in terms the {\bf linearized Ricci-DeTurck} equation
\[ \nabla_{\partial_t} h'_t = \triangle_{g_t} h'_t + 2 \Rm_{g_t} (h'_t) . \]
For more details on this, see Section~\ref{sec_semi_local_max}.

\subsection{Maps between Ricci flow spacetimes}
In this subsection consider two Ricci flow spacetimes $(\M,\t,\D_\t,g)$ and $(\M',\t',\D_{\t'},g')$, which we will abbreviate in the following by $\M$ and $\M'$.
Our goal will be to characterize maps between subsets of these spacetimes.
Using the terminology introduced above, we will then generalize the notions introduced in the previous subsection to Ricci flow spacetimes.

\begin{definition}[Time-preserving and time-equivariant maps]
Let $X \subset \M$ be a subset and $\phi : X \to \M'$ be a map.
We say that $\phi$ is \textbf{time-preserving} if $\mathfrak{t}' (\phi (x)) = \mathfrak{t} (x)$ for all $x \in X$.
We say that $\phi$ is \textbf{$a$-time-equivariant}, for some $a \in \R$, if there is some $t_0 \in \R$ such that $\mathfrak{t}' (\phi (x)) = a \mathfrak{t} (x) + t_0$ for all $x \in X$.
\end{definition}

Observe that a time-preserving map is also $1$-time-equivariant.

\begin{definition}[Time-slices of a map]
If $\phi : X \subset \M \to \M'$ is time-equivariant and $t \in [0, \infty)$ such that $X_t = X \cap \M_t \neq \emptyset$, then we denote by
\[ \phi_t := \phi |_{X_t} : X_t \longrightarrow \M'_{t'} \subset \M' \]
the \textbf{time-$t$-slice} of $\phi$.
Here $t'$ is chosen such that $\phi (X_t) \subset \M'_{t'}$.
\end{definition}

\begin{definition}[$\partial_{\t}$-preserving maps]
\label{def_d_t_preserving}
Let $\phi : X \to \M'$ be a differentiable map defined on a sufficiently regular domain $X \subset \M$.
If $(d \phi)_* \partial_{\t} = \partial_{\t'}$, then we say that $\phi$ is \textbf{$\partial_{\t}$-preserving.}
\end{definition}

Note that the image of a product domain under a time-equivariant and $\partial_{\t}$-preserving map is again a product domain.

\begin{definition}[Harmonic map heat flow] \label{def_RF_spacetime_harm_map_hf}
Let $Y \subset \M'$ be a subset.
We say that a map $\chi : Y \to \M$ evolves by \textbf{harmonic map heat flow} if it is $1$-time-equi\-va\-ri\-ant and if at all times $t, t' \in [0, \infty)$ with $Y_t \neq \emptyset$ and $\chi (Y_{t'}) \subset \M'_{t}$ the identity
\begin{equation} \label{eq_hmhflow_RF_spacetime}
 d\chi (\partial_{\t'}) = \partial_{\t} + \triangle_{g'_t, g_{t}} \chi_t 
\end{equation}
holds on the interior of $Y$.
The last term in this equation denotes the Laplacian of the map $\psi_t : (\M'_{t'}, g'_{t'}) \to (\M_{t}, g_{t})$ (see (\ref{eq_hmh_flow_def}) for further details).
\end{definition}

It is not difficult to see that the notions of harmonic map heat flow in Definition~\ref{def_RF_spacetime_harm_map_hf} corresponds to Definition~\ref{def_classical_hmh_flow} in the case in which $\M$ and $\M'$ can be described in terms of classical Ricci flows $(M, (g_t)_{t \in I})$ and $(M', (g'_t)_{t \in I'})$, respectively.
The same is true in the case in which $\chi$ is the inverse of a diffeomorphism $\phi : X \to Y \subset \M'$, where $X$ is a product domain in $\M$ whose time-slices are domains with smooth boundary. 
In this case, which will be of main interest for us (see Definition~\ref{def_comparison}), the equation (\ref{eq_hmhflow_RF_spacetime}) makes sense and holds, by continuity, on all of $Y$.

Next, we generalize the concept of Ricci-DeTurck flow to the setting of Ricci flow spacetimes.

\begin{definition} \label{def_RdT_perturbation_on_MM}
Consider a smooth symmetric $(0,2)$-tensor field $h$ on the subbundle $\ker (d \t) \subset T\M$ over a sufficiently regular domain $N  \subset \M$ (in this paper we will only consider the case in which $N$ is a domain with smooth boundary or is a product domain whose time-slices are domains with smooth boundary).
We say that $h$ is a \textbf{Ricci-DeTurck perturbation (on $N$)} if
\begin{equation} \label{eq_RdT_on_spacetime_definition}
 \mathcal{L}_{\partial_{\t}} ( g + h) = - 2 \Ric ( g + h) - \mathcal{L}_{X_{g} (g+h)} ( g + h ), 
\end{equation}
where $X_{g} (g+h)$ is defined on each time-slice $X_t$ as in (\ref{eq_def_X}).
\end{definition}

If $X$ is a product domain of the form $X' \times I$, and if we identify $g$ and $h$ with smooth families of the form $(g_t)_{t \in I}$ and $(h_t)_{t \in I}$, then (\ref{eq_RdT_on_spacetime_definition}) is equivalent to the classical Ricci-DeTurck equation (\ref{eq_RdT_flow_def}).

The following lemma is an immediate consequence of our discussion from Subsection~\ref{subsec_RdT_hmhf}.

\begin{lemma} \label{lem_hmhf_RdTperturbation_spacetime}
Let $X \subset \M$ be open or a product domain whose time-slices are domains with smooth boundary and consider a diffeomorphism $\phi : X \to Y := \phi (X) \subset \M'$.
Assume that the inverse map $\phi^{-1} : Y \to X$ evolves by harmonic map heat flow.
Then the perturbation $h := \phi^* g' - g$ is a Ricci-DeTurck perturbation in the sense of Definition~\ref{def_RdT_perturbation_on_MM}. 
\end{lemma}

\section{A priori assumptions}
\label{sec_apas}

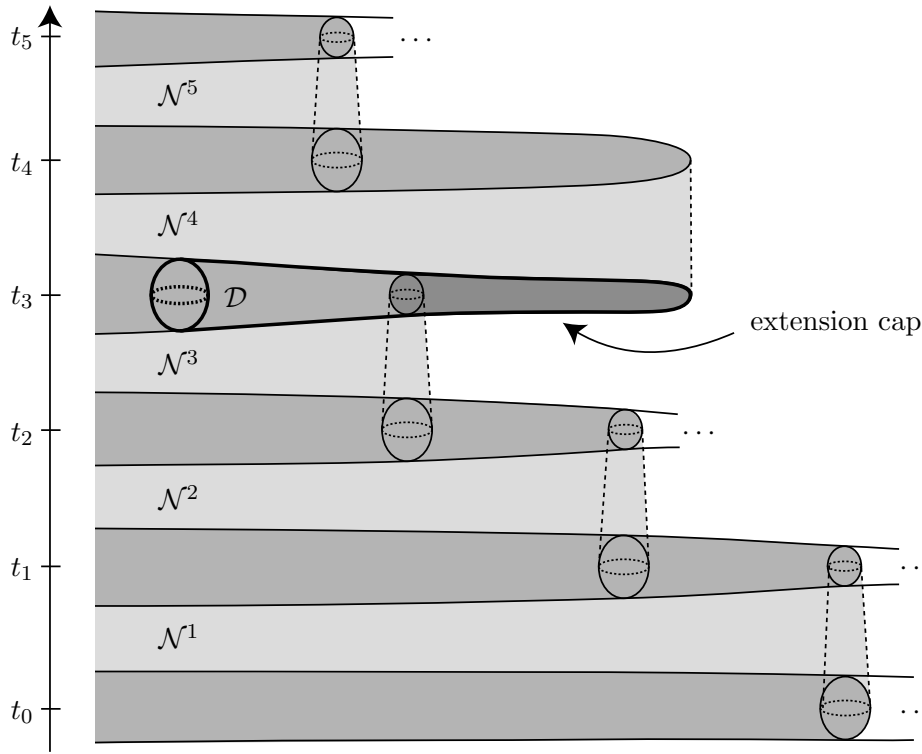
\begin{figure}
\labellist
\small\hair 2pt
\pinlabel $t_0$ at -10 95
\pinlabel $t_1$ at -10 425
\pinlabel $t_2$ at -10 735
\pinlabel $t_3$ at -10 1050
\pinlabel $t_4$ at -10 1355
\pinlabel $t_5$ at -10 1645
\pinlabel $\cdots$ at 900 1640
\pinlabel $\cdots$ at 1550 735
\pinlabel $\cdots$ at 2050 425
\pinlabel $\cdots$ at 2050 95
\pinlabel $\N^1$ at 350 270
\pinlabel $\N^2$ at 350 590
\pinlabel $\N^3$ at 350 900
\pinlabel $\N^4$ at 350 1220
\pinlabel $\N^5$ at 350 1520
\pinlabel $\mathcal{D}$ at 480 1050
\pinlabel {extension cap} at 1860 990
\endlabellist
\centering
\includegraphics[width=120mm]{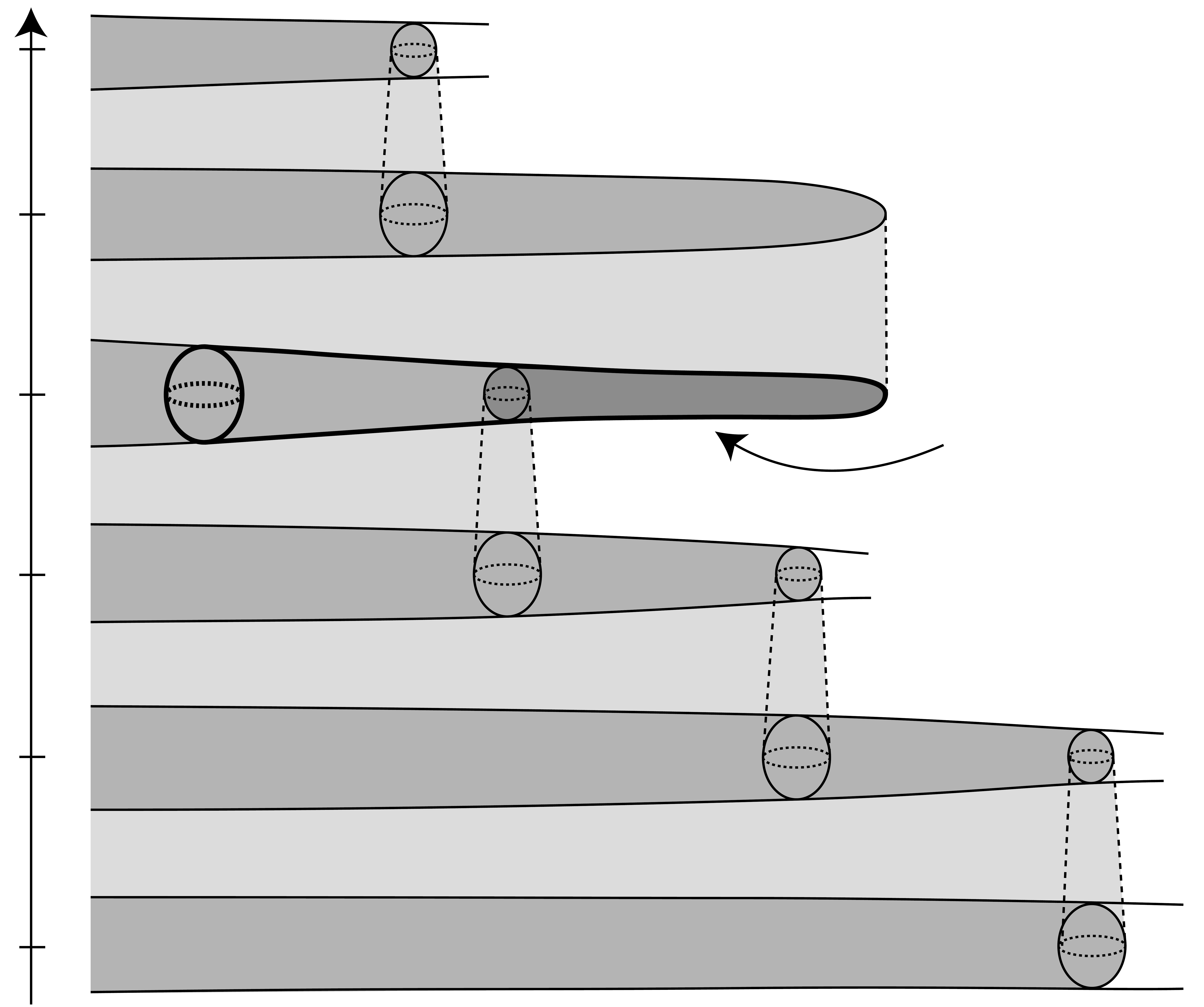}
\caption{Example of a comparison domain defined over the time-interval $[0,t_5]$ and a cut.
The dark shaded regions indicate the picture of the comparison domain at integral time-steps $t_0, \ldots, t_5$.
The extension cap at time $t_3$ is shaded very dark.
This extension cap is contained in a cut $\DD$, which is outlined in bold.
Note that cuts, such as $\DD$, occur in the definition of a comparison, not of a comparison domain.
\label{fig_comparison_domain}}
\end{figure}
In this section we introduce the objects and conditions that will be used to formulate and prove the main result, Theorem \ref{Thm_existence_comparison_domain_comparison}, which asserts the existence of a certain type of map between subsets of Ricci flow spacetimes.  The domain of the map will be called a {\it comparison domain} (Definition~\ref{def_comparison_domain}), and the map itself a {\it comparison} (Definition~\ref{def_comparison}).   The comparison and its domain will be subject to a number of {\it a priori assumptions} (Definitions~\ref{def_a_priori_assumptions_1_7} and \ref{def_a_priori_assumptions_7_13}).  These definitions have been tailored to facilitate an existence proof by induction over time steps.

We recommend reading the overview in Section~\ref{sec_overview} prior to reading this section, because it provides motivation for the structures defined here, and gives some indication of the role they play in the proof.  We refer the reader to Sections \ref{sec_preliminaries_I} and \ref{sec_preliminaries_II} for the definitions relevant to this section.

\subsection{Comparison domains} We begin with a definition that collects the qualitative features of the domain of our comparison map.  Additional assumptions of a quantitative nature are imposed later, in the a priori assumptions.  Loosely speaking, a comparison domain is a sequence of product domains  $\N^1,\ldots,\N^J$ defined on successive time-intervals, whose time-slices have spherical boundary (see Figure~\ref{fig_comparison_domain} for an illustration).  One observes two types of behavior near the boundary as one transitions from one product domain to the next:  boundary components can either ``recede'', or they can be filled in by $3$-balls.   In the main existence proof, the latter case corresponds to the situation when the comparison map is extended over a cap region lying in a subset that is approximated by a Bryant soliton;  for this reason, we call the closures of such $3$-balls extension caps.

\begin{definition}[Comparison domain]
\label{def_comparison_domain}
A {\bf comparison domain (defined over the time-interval $[0,t_J]$)} in a Ricci flow spacetime $\M$ is a triple $(\N, \lb \{\N^j\}_{1\leq j\leq J}, \lb \{ t_j \}_{j=0}^J)$, where:
\begin{enumerate}
\item \label{pr_7.1_1} The times $0=t_0<\ldots<t_J$ partition the time-interval $[0,t_J]$.
Each $\N^j$ (for $1 \leq j \leq J$) is a subset of  $\M_{[t_{j-1},t_j]}$, and $\N=\cup_{1\leq j\leq J}\N^j\subset \M_{[0,t_J]}$.
\item \label{pr_7.1_2} For all $1 \leq j \leq J$ the subset $\N^j$  is closed in $\M$, and is a product domain, in the sense of Definition~\ref{def_product_domain}. 
\item \label{pr_7.1_3} For all $1\leq j<J$, we have $\partial \N^{j+1}_{t_j} \subset \Int \N^{j}_{t_j}$.
Here $\Int \N^{j}_{t_j}$ denotes the interior of $\N^j_{t_j}$ inside $\M_{t_j}$.  Consequently, the difference $\N^{j+1}_{t_{j}} \setminus \Int \N^j_{t_{j}}$ is a closed subset of $\M_{t_j}$ that is a domain with smooth boundary, with boundary contained in $\N^j_{t_j}$.
\item \label{pr_7.1_4} For every $1\leq j<J$,
	\label{item_3_disk} the components of $\N^{j+1}_{t_{j}} \setminus \Int \N^j_{t_{j}}$ are $3$-disks, which are called {\bf  extension caps}. 
\end{enumerate}
For any $t<t_J$, we define the forward  time-$t$ slice $\N_{t+}$ of $\N$ to be the set  of accumulation points of $\N_{\bar t}$ as $\bar t\searrow t$, and  if $t=t_J$ we define $\N_{t+}=\N_t$.  
We define the backward time-slices $\N_{t-}$ similarly, but taking accumulation points as $\bar t\nearrow t$, and when $t=  0$, we put $\N_{0-}=\N_0$.   
Thus if $t\in (t_{j-1},t_j)$ then $\N_{t\pm}=\N^j_t$ and $\N_{t_j-}=\N^j_{t_j}$, $\N_{t_j+}=\N^{j+1}_{t_j}$ if $1\leq j<J$.  Observe that $\N_t=\N_{t-}\cup \N_{t+}$.

In the case $J = 0$ the comparison domain $(\N = \emptyset, \{ \}, \{ t_0 \} )$ is called the {\bf empty comparison domain}.

When there is no chance of confusion, we will sometimes abbreviate $(\N,\{\N^j\}_{j=1}^J, \lb \{ t_j \}_{j=0}^J)$ by $\N$.  
\end{definition}

\bigskip

\subsection{Comparisons}
Next, we collect the basic properties of our comparison maps between Ricci flow spacetimes.  Roughly speaking, a comparison is a map between Ricci flow spacetimes that is defined on a comparison domain.  
Away from the transition times, the inverse of this map  solves the harmonic map heat flow equation for the evolving metrics, or equivalently, the pullback metric satisfies the Ricci-DeTurck equation.  At a transition time, the comparison is extended over the extension caps.  In order to guarantee a good interpolation, it is necessary to adjust the comparison over a region that is much larger than the extension cap.  
As a consequence, the comparison, when viewed as a map between spacetimes, may have jump discontinuities near every extension cap.  The discontinuity locus is contained in a disjoint union of closed disks, which we  will call {\it cuts} (see  Figure~\ref{fig_comparison_domain} for an illustration).

In the following definition, we allow a comparison to be defined on a shorter time-interval than the comparison domain.   This is done for technical reasons having to do with a two part induction argument.  More specifically, in Section \ref{sec_inductive_step_extension_comparison_domain}, we will analyze a comparison that is defined on an entire comparison domain (over a time-interval $[0, t_J]$) and then extend the comparison domain by one time-step (to the time-interval $[0, t_{J+1}]$), without extending the comparison itself.  So we will end up with a comparison domain that is defined up to some time $t_{J+1}$, while the comparison itself still remains defined only up to time $t_J$.

\begin{definition}[Comparison] \label{def_comparison}
Let $\M$, $\M'$ be Ricci flow spacetimes and consider a comparison domain $(\N, \lb \{ \N^j \}_{j=1}^J, \lb \{ t_j \}_{j=0}^J )$ defined over the time-interval $[0,t_J]$ in $\M$.

A triple $(\Cut, \phi, \{ \phi^j \}_{j = 1}^{J^*} )$ is a {\bf comparison from $\M$ to $\M'$ defined on $(\N, \lb \{ \N^j \}_{j=1}^J, \lb \{ t_j \}_{j=0}^J)$ (over the time-interval $[0,t_{J^*}]$)} if:
\begin{enumerate}
\item \label{pr_7.2_1} $J^* \leq J$.
\item \label{pr_7.2_2} $\Cut = \Cut^1 \cup \ldots \cup \Cut^{J^*-1}$, where each $\Cut^j$ is a collection of pairwise disjoint $3$-disks inside $\Int \N_{t_j+}$.  
\item \label{pr_7.2_3} Each $\DD \in \Cut$ contains exactly one extension cap of the domain $(\N, \lb \{ \N^j \}_{j=1}^J, \lb \{ t_j \}_{j=0}^J)$ and every extension cap of $(\N, \lb \{ \N^j \}_{j=1}^J, \lb \{ t_j \}_{j=0}^J)$ that is contained in $\M_{[0,t_{J^*-1}]}$ is contained in one element of $\Cut$.
\item \label{pr_7.2_4} Each $\phi^j : \N^j  \to \M'$ is a  time-preserving diffeomorphism onto its image.  More precisely, $\phi^j$ may be extended to a diffeomorphism onto its image defined on an open neighborhood of $\N^j$ in the manifold with boundary $\M_{[t_{j-1},t_j]}$.
\item \label{pr_7.2_5} If $J^* \geq 1$, then $\phi : \cup_{j=1}^{J^*} \N^j \setminus \cup_{\DD \in \Cut}\DD \to \M'$ is a continuous map that is smooth on the interior of $ \cup_{j=1}^{J^*} \N^j  \setminus \cup_{\DD\in \Cut} \DD$. 
If $J^* = 0$, then we assume that $\phi : \emptyset \to \emptyset$ is the trivial map.
\item \label{pr_7.2_6} $\phi = \phi^j$ on the open time slab $\N^j_{(t_{j-1}, t_j)}$ for all $j = 1, \ldots, J^*$.
\item \label{pr_7.2_7} For all $j = 1, \ldots, J^*$, the inverse map $(\phi^j)^{-1} : \phi^j (\N^j) \to \N^j $ evolves by harmonic map heat flow (according to Definition~\ref{def_RF_spacetime_harm_map_hf}). 
\end{enumerate} 
We define $\phi_{t_j-}$ to be $\phi^j|_{\N^j_{t_j}}$ if $0< j\leq J^*$ and  $\phi^1_0$ if $j=0$.  Similarly, we define $\phi_{t_j+}$ to be $\phi^j|_{\N^{j+1}_{t_j}}$ if $0\leq j<J^*$ and  $\phi^{J^*}|_{\N^{J^*}_{t_{J^*}}}$ if $j=J^*$.
\end{definition}

We remark that Definition~\ref{def_comparison} implies that $\phi$ is injective, and that $\phi^{-1}$  satisfies the harmonic map heat flow equation everywhere it is defined.

Note that by Definition~\ref{def_comparison}, the only comparison in the case $J^* = 0$ is the trivial comparison $(\Cut = \emptyset, \phi : \emptyset \to \emptyset, \emptyset$).

As explained in Subsection~\ref{subsec_RdT_hmhf}, a map whose inverse is evolving by harmonic map heat flow induces a Ricci-DeTurck flow on its domain.
We will now use this fact to define the Ricci-DeTurck perturbation associated with a comparison.

\begin{definition}[Associated Ricci-DeTurck perturbation]
Consider a comparison domain $(\N, \lb \{ \N^j \}_{j=1}^J, \lb \{ t_j \}_{j=0}^J)$ in a Ricci flow spacetime $\M$ that is defined  over the time-interval $[0,t_J]$ and a comparison $(  \Cut, \linebreak[1] \phi, \linebreak[1] \{ \phi^j \}_{j=1}^{J^*})$ from $\M$ to $\M'$ defined on this domain  over the time-interval $[0, t_{J^*}]$ for some $J^* \leq J$.

Define $h := \phi^* g' - g$ on  $\N \setminus \cup_{\DD \in \Cut} \DD$ and $h^j := (\phi^j)^* g' - g$ on $\N^j$ for all $1 \leq j \leq J^*$.
Then we say that $(h, \{ h^j \}_{j=1}^{J^*})$ is the {\bf associated Ricci-DeTurck perturbation} for $(  \Cut, \linebreak[1] \phi, \linebreak[1] \{ \phi^j \}_{j=1}^{J^*})$.
Moreover, for  $1 \leq j \leq J^*$ we set $h_{t_j-} := h^j_{t_j}$, and define $h_{t_0-}:=h^1_0$.
Likewise, for $0 \leq j \leq J^* - 1$ we set $h_{t_j+} := h^{j+1}_{t_j}$ and $h_{t_{J^*}+}=h^{J*}_{t_{J^*}}$.
\end{definition}

Note that by Lemma~\ref{lem_hmhf_RdTperturbation_spacetime} the tensors $h$ and $h^j$ are Ricci-DeTurck perturbations in the sense of Definition~\ref{def_RdT_perturbation_on_MM}.

\bigskip

\subsection{A priori assumptions I: the geometry of the comparison domain}
Next, we introduce a priori assumptions for a comparison $(\Cut, \linebreak[1] \phi, \linebreak[1] \{ \phi^j \}_{j=1}^{J^*})$ defined on a comparison domain $(\N, \linebreak[1] \{ \N^j \}_{j=1}^J, \linebreak[1] \{ t_j \}_{j=0}^J)$.
We first state the first six a priori assumptions, \ref{item_time_step_r_comp_1}--\ref{item_eta_less_than_eta_lin_13}, which characterize the more geometric properties of the comparison domain and the comparison.  These are the only a priori assumptions needed to implement the first part of the main induction argument, in Section~\ref{sec_inductive_step_extension_comparison_domain}.

To make it easier to absorb the list of conditions, we make some informal preliminary remarks.
The construction of the comparison domain and comparison involves a comparison scale $r_{\comp}$.  Most of the a priori assumptions impose conditions at scales that are defined relative to $r_{\comp}$.  
For instance, the final time-slice of each product domain $\N^j$ of the comparison domain is assumed to have boundary components that are central 2-spheres of necks at scale $r_{\comp}$.
Moreover, we assume the comparison domain to be $\lambda r_{\comp}$-thick and to contain all $\Lambda r_{\comp}$-thick points at integral time-slices.   
These and similar characterizations will be made in a priori assumptions \ref{item_time_step_r_comp_1}--\ref{item_backward_time_slice_3}.

In addition,  we impose  two assumptions,  \ref{item_discards_not_too_thick_4} and \ref{item_geometry_cap_extension_5},   that restrict the situations when a component can be discarded or added,  respectively.  
To appreciate the role of these two conditions, the reader may wish to imagine a scenario when a Bryant-like cap region in $\M$  evolves through a range of scales, initially well below $\lambda r_{\comp}$, then well above $\Lambda r_{\comp}$,  possibly fluctuating between these over a time scale $\gg r_{\comp}^2$.  
Then initially the cap region will lie outside the comparison domain, because its scale is too small, and later it will necessarily lie in the comparison domain, because it has scale $>\Lambda r_{\comp}$.   
A priori assumptions   \ref{item_discards_not_too_thick_4} and \ref{item_geometry_cap_extension_5} ensure that these events occurs when the tip of the cap has scale in the range approximately $(\lambda r_{\comp},10\lambda r_{\comp})$, and that they do not occur unnecessarily too often.

Finally, a priori assumption \ref{item_eta_less_than_eta_lin_13} states that the comparison itself is an almost isometry of high enough precision.

We mention that a priori assumptions \ref{item_time_step_r_comp_1}--\ref{item_eta_less_than_eta_lin_13} depend on a number of parameters, which will be chosen in the course of this paper.
Also,  as with Definition~\ref{def_comparison}, in the following definition we do not require a comparison to be defined on the entire comparison domain (see the discussion before Definition~\ref{def_comparison}).

\begin{definition}[A priori assumptions \ref{item_time_step_r_comp_1}--\ref{item_eta_less_than_eta_lin_13}]
\label{def_a_priori_assumptions_1_7}
Let $(\N, \lb \{ \N^j \}_{j=1}^J, \lb \{ t_j \}_{j=0}^J)$ be a comparison domain in a Ricci flow spacetime $\M$ that is defined over  the time-interval $[0,t_J]$ and consider a comparison $(  \Cut, \linebreak[1] \phi, \linebreak[1] \{ \phi^j \}_{j=1}^{J^*})$ from $\M$ to $\M'$ on this domain to another Ricci flow spacetime $\M'$ that is defined over  the time-interval $[0, t_{J^*}]$ for some $J^* \leq J$.

We say that {\bf $(\N, \linebreak[1] \{ \N^j \}_{j=1}^J, \linebreak[1] \{ t_j \}_{j=0}^J)$  and $\linebreak[1] (\Cut, \linebreak[1] \phi, \linebreak[1] \{ \phi^j \}_{j=1}^{J^*})$  satisfy a priori assumptions \ref{item_time_step_r_comp_1}--\ref{item_eta_less_than_eta_lin_13}  with respect to the tuple of parameters $(\eta_{\lin}, \lb \delta_{\nn}, \lb \lambda, \lb D_{\CAP}, \lb \Lambda, \lb  \delta_{\bb}, \lb \eps_{\can}, \lb r_{\comp})$}  if the following holds:
\begin{enumerate}[label=(APA \arabic*), leftmargin=* ]
\item 
\label{item_time_step_r_comp_1}  We have $t_j = j \cdot  r_{\comp}^2$ for each $0 \leq j \leq J$.
\item 
\label{item_lambda_thick_2} All points in  $\N$ are $\lambda r_{\comp}$-thick.
\item 
\label{item_backward_time_slice_3} For every $1\leq j\leq J$, the backward time-slice $\N_{t_j-}=\N^j_{t_j}$ has the following properties: 
	\begin{enumerate}
	\item \label{item_central_2_spheres} The boundary components of $\N_{t_j-}$ are central $2$-spheres of $\delta_{\nn}$-necks at scale $r_{\comp}$.
	\item $\N_{t_j -}$ contains all $\Lambda r_{\comp}$-thick points of $\M_{t_j}$. 
	\item \label{item_contains_Lambda_thick_point} Each component of $\N_{t_j-}$ contains a $\Lambda r_{\comp}$-thick point. 
	\item \label{item_10_lambda_thick} Each component of $\M_{t_j} \setminus \Int \N_{t_j-}$ with non-empty boundary contains a $10\lambda r_{\comp}$-thin point.
	\item The points on each cut $\DD \in \Cut$ are $\Lambda r_{\comp}$-thin.
	\end{enumerate}

\item \label{item_discards_not_too_thick_4} \textit{(Discarded disks become thin)} \quad Suppose  $1 \leq j \leq J$,  and  $\mathcal{C}$ is a component of $\N_{t_{{j-1}} -} \setminus \Int \N_{t_{{j-1}}+}$ (if $j \geq 2$) or $\M_{0} \setminus \Int \N_{0+}$ (if $j=1$) such that:  	\ben
	\item  $\C$ is diffeomorphic to a $3$-disk.
	\item   $\D\C\subset \N_{t_{j-1}+}$.
	\een
Then  either $\C$ does not survive until time $t_j$ (as in Definition \ref{def_points_in_RF_spacetimes}) or for some time $t\in [t_{j-1},t_j]$ we can find a weakly $\lambda r_{\comp}$-thin point on $\mathcal{C}(t)$ (recall the notation $\C(t)$ from Definition~\ref{def_points_in_RF_spacetimes}.)

\item \label{item_geometry_cap_extension_5} \textit{(Geometry of extension caps)} \quad 
For each $1 \leq j \leq J^*$ and every component $\C$ of $\M_{t_j} \setminus \Int \N_{t_j-}$ the following holds.

$\C$ is an extension cap of $(\N, \{\N^j \}_{j=1}^J, \{ t_j \}_{j=0}^J)$ if and only if there is a component $\C'$ of $\M'_{t_j} \setminus \phi_{t_j-} ( \Int \N_{t_j -} )$ such that:
		\begin{enumerate}
		\item $\C$ and $\C'$ are $3$-disks.
		\item $\partial \C' = \phi_{t_j-} (\partial \C)$.
		\item   There is a point $x\in \C$ such that $(\M_{t_j},x)$ is $\de_{\bb}$-close to the pointed Bryant soliton $(M_{\Bry},g_{\Bry},x_{\Bry})$ at scale $10\lambda r_{\comp}$.
		\item There is a point $x'\in \M'_{t_j}$, at distance $\leq D_{\CAP} r_{\comp}$ from $\C'$, such that $(\M'_{t_j},x')$ is $\de_{\bb}$-close to the pointed Bryant soliton $(M_{\Bry},g_{\Bry},x_{\Bry})$  at some scale in the interval $[D_{\CAP}^{-1}r_{\comp},D_{\CAP}r_{\comp}]$.
		\item $\C$ and $\C'$ have diameter $\leq D_{\CAP}r_{\comp}$.
		\end{enumerate}

\item \label{item_eta_less_than_eta_lin_13} Consider the Ricci-DeTurck perturbation $(h, \{ h \}_{j=1}^{J^*})$ associated to the comparison $(  \Cut, \linebreak[1] \phi, \linebreak[1] \{ \phi^j \}_{j=1}^{J^*})$.
If $J^* \geq 1$, then $|h| \leq \eta_{\lin}$ on $ \cup_{j=1}^{J^*} \N^j  \setminus \cup_{\DD \in \Cut} \DD$.
Moreover, the $\eps_{\can}$-canonical neighborhood assumption holds at scales $(0,1)$ on $\cup_{j=1}^{J^*} \phi^j (\N^j)$. 
\end{enumerate}
\end{definition}

We point out that a priori assumptions \ref{item_time_step_r_comp_1}--\ref{item_discards_not_too_thick_4}  are conditions on the comparison domain only.  On the other hand, a priori assumption \ref{item_geometry_cap_extension_5} places restrictions on extension caps in terms of the comparison map and the local geometry of the image.  This is to ensure that extension caps arise only when the geometry of the domain and target are nice enough to allow an extension of a comparison on that is a precise enough almost isometry.

\begin{figure}
\labellist
\small\hair 2pt
\pinlabel {no $Q^*$-bound} at 1225 422
\pinlabel $\cdots$ at 1730 535
\pinlabel $\cdots$ at 2120 305
\pinlabel $\cdots$ at 2120 95
\pinlabel {no $Q$-bound} at 1150 870
\pinlabel {cut} at 1720 750
\endlabellist
\centering
\includegraphics[width=130mm]{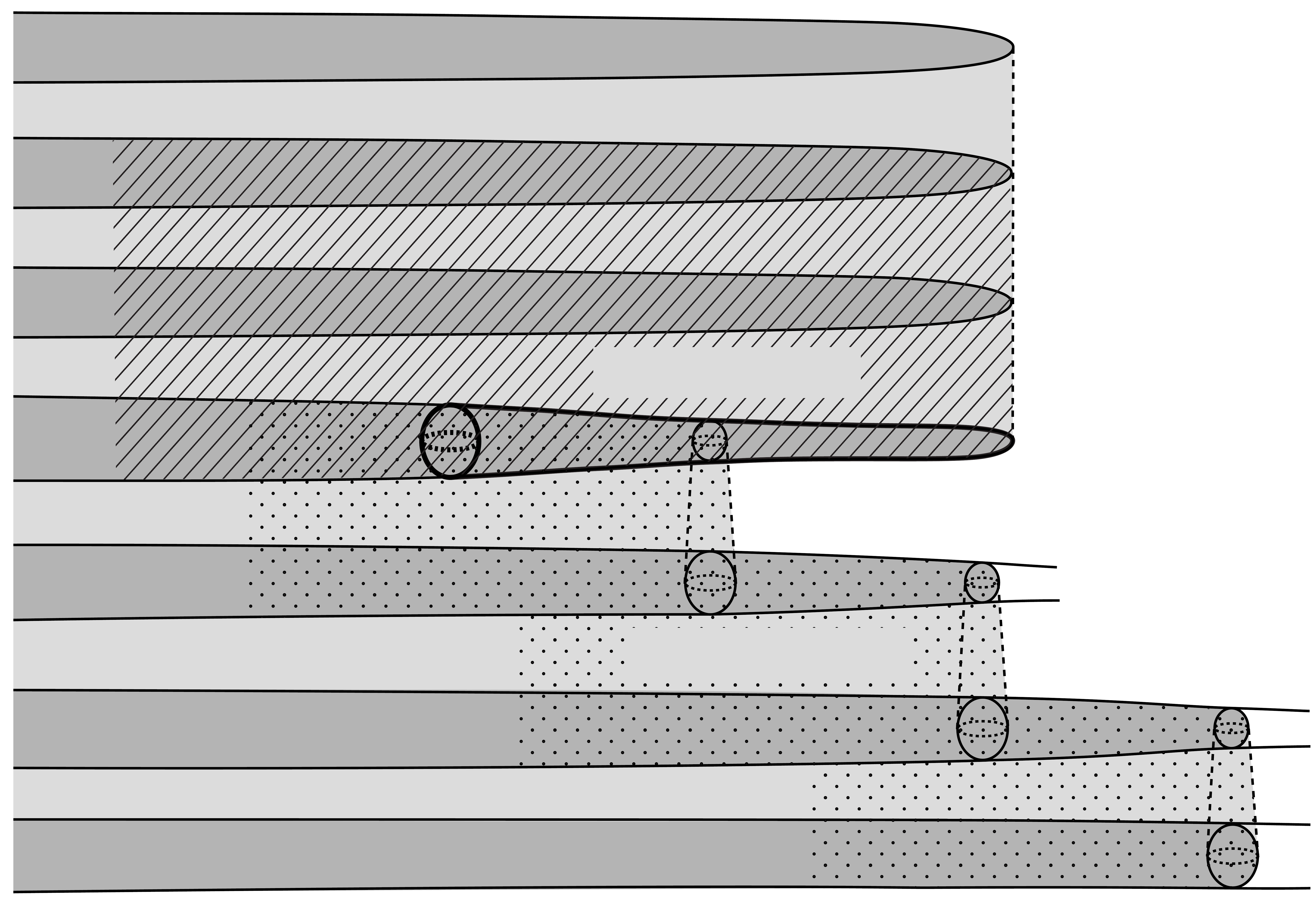}
\caption{The bound $Q \leq \ov{Q}$ in \ref{item_q_less_than_q_bar_6} holds on all of $\N$ except for the hatched region.
The bound $Q^* \leq \ov{Q}^*$ in \ref{item_q_star_less_than_q_star_bar} holds on all of $\N$ except for the dotted region.\label{fig_QQstar_regions}}
\end{figure}
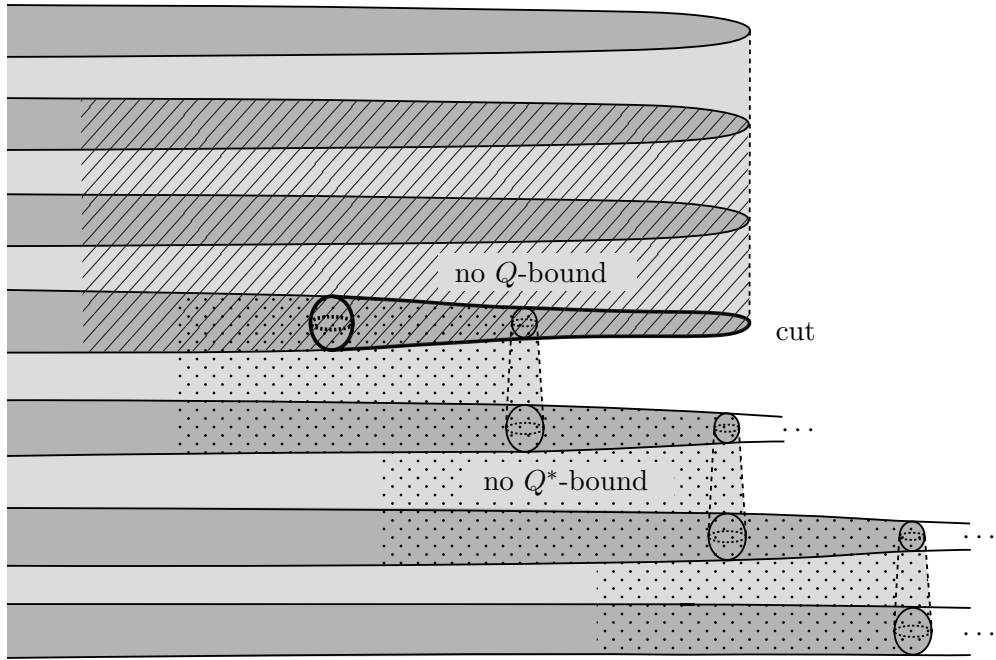

\subsection{A priori assumptions II: analytic conditions on the comparison}
Lastly, we introduce a further set of a priori assumptions, a priori assumptions \ref{item_q_less_than_q_bar_6}--\ref{item_apa_13}, which characterize the behavior of the perturbation $h$ and the geometry of the cuts more precisely.
These assumptions will become important in Section~\ref{sec_extend_comparison}, where we will extend the comparison by one time-step onto a larger comparison domain.

We now give a brief overview of a priori assumptions \ref{item_q_less_than_q_bar_6}--\ref{item_apa_13}.
A priori assumptions \ref{item_q_less_than_q_bar_6}--\ref{item_h_derivative_bounds_9} impose global bounds on the Ricci-DeTurck perturbation $h$ via two quantities $Q$ and $Q^*$.
These bounds essentially introduce a pointwise weight, which depends on the curvature scale $\rho$ and time.  
A priori assumption \ref{item_q_less_than_q_bar_6} imposes a bound on $Q$ on the comparison domain,  on the complement of forward parabolic neighborhoods of cuts.
Similarly, a priori assumption \ref{item_q_star_less_than_q_star_bar} imposes a bound on $Q^*$ at points of the comparison domain that are far enough away from its neck-like boundary.
For an illustration of the domains on which these bounds do or do not hold, see Figure~\ref{fig_QQstar_regions}.
A priori assumption \ref{item_q_less_than_w_q_bar_7} introduces a weaker bound on $Q$, which holds essentially everywhere on the comparison domain.
Note that the constant $W$ in this bound will be chosen to be large.
Therefore, a priori assumption \ref{item_q_less_than_w_q_bar_7} will not directly imply a priori assumption \ref{item_q_less_than_q_bar_6}.

A priori assumption \ref{item_h_derivative_bounds_9} states that $Q^*$ is small on each cut and a priori assumption \ref{item_q_less_than_nu_q_bar_12} guarantees a good bound on $Q$ and $Q^*$ on the initial time-slice.
A priori assumption \ref{item_cut_diameter_less_than_d_r_comp_11} controls the geometry of the cuts. Lastly, a priori assumption \ref{item_apa_13} imposes a bound on $t_J$.

\begin{definition}[A priori assumptions \ref{item_q_less_than_q_bar_6}--\ref{item_apa_13}]
\label{def_a_priori_assumptions_7_13}
Let $(\N, \lb \{ \N^j \}_{j=1}^J, \lb \{ t_j \}_{j=0}^J)$ be a comparison domain in a Ricci flow spacetime $\M$ that is defined on the time-interval $[0,t_J]$ and consider a comparison $(  \Cut, \linebreak[1] \phi, \linebreak[1] \{ \phi^j \}_{j=1}^{J^*})$ on this domain to another Ricci flow spacetime $\M'$ that is defined on the same time-interval $[0, t_J]$.

We say that {\bf $( \N, \linebreak[1] \{ \N^j \}_{j=1}^J, \linebreak[1] \{ t_j \}_{j=0}^J)$ and $(\Cut, \linebreak[1] \phi, \linebreak[1] \{ \phi^j \}_{j=1}^J)$ satisfy  a priori assumptions \ref{item_q_less_than_q_bar_6}--\ref{item_apa_13}  with respect to the tuple of parameters $( T, \lb E, \lb H, \lb \eta_{\lin}, \lb \nu, \lb \lambda, \lb \eta_{\cut},  \lb  D_{\cut}, \lb W, \lb A, \lb r_{\comp})$  if the following holds.}
Define the functions
\[ Q := e^{H (T-\mathfrak{t})} \rho_1^E |h| , \qquad Q^* := e^{ H(T- \mathfrak{t})} \rho_1^3 |h|  \]
on $\N \setminus \cup_{\DD \in \Cut} \DD$, where $\mathfrak{t} : \M \to [0, \infty)$ is the time-function.
On $\N_{t_j \pm}$ we denote by $Q_{\pm}$ and $Q^*_{\pm}$ the corresponding values for $h_{t_j \pm }$.
We also set $Q_{\pm} := Q$ on $\N \setminus \cup_{\DD \in \Cut} \DD$.
Set 
\[ \ov{Q} := 10^{-E-1} \eta_{\lin} r_{\comp}^{E}, \qquad \ov{Q}^*:=10^{-1} \eta_{\lin} (\lambda r_{\comp})^{3} .\]
Then
\begin{enumerate}[start= 7, label=(APA \arabic*), leftmargin=* ]
\item \label{item_q_less_than_q_bar_6}\textit{($Q\leq \ov{Q}$ if no cuts in nearby past)} \quad For all $x \in \N \setminus \cup_{\DD \in \Cut} \DD$ for which $P(x, A \rho_1 (x)) \cap \DD = \emptyset$ for all $\DD \in \Cut$, we have
\[ Q(x) \leq \ov{Q}. \]
Note that the bound is also required to hold if $P(x, A \rho_1 (x))  \not\subset \N$.

\item \label{item_q_less_than_w_q_bar_7} We have
\[ Q \leq W \cdot \ov{Q} \qquad \text{on} \qquad  \N \setminus \cup_{\DD \in \Cut} \DD.\]

\item \label{item_q_star_less_than_q_star_bar} \textit{($Q^*\leq\ov{Q}^*$ away from time-slice boundary)} \quad For all $x \in \N \setminus \cup_{\DD \in \Cut} \DD$ for which $B(x, A \rho_1 (x)) \subset \N_{\t(x)-}$, we have
\[ Q^*(x) \leq \ov{Q}^*. \]

\item \label{item_h_derivative_bounds_9}  On every cut $\DD \in \Cut$, we have  
\[ Q^*_+ \leq \eta_{\cut} \ov{Q}^*. \]
\item \label{item_cut_diameter_less_than_d_r_comp_11} For every cut $\DD \in \Cut$, $\DD \subset \M_{t_j}$, the following holds: The diameter of $\DD$ is less than   $D_{\cut} r_{\comp}$ and $\DD$ contains a $\frac1{10} D_{\cut} r_{\comp}$-neighborhood of the extension cap  $\C = \DD \setminus \Int\N_{t_j -}$.

\item  \label{item_q_less_than_nu_q_bar_12} We have $Q \leq \nu \ov{Q}$ and $Q^* \leq \nu \ov{Q}^*$ on $\N_0$ (i.e. at time $0$).

\item \label{item_apa_13} We have $t_J \leq T$.
\end{enumerate}
\end{definition}

Note that a priori assumptions \ref{item_q_less_than_q_bar_6}--\ref{item_apa_13} are vacuous if $J = 0$.

We briefly comment on the purpose of a priori assumptions \ref{item_q_less_than_q_bar_6}--\ref{item_q_star_less_than_q_star_bar}.

As explained in Section~\ref{sec_overview}, a priori assumption \ref{item_q_less_than_q_bar_6}, the bound $Q \leq \ov{Q}$, serves as a main ingredient for the Bryant Extension Principle, as long as $E$ is chosen large enough.
It will also be used to ensure that $|h| \leq \eta_{\lin}$ at most points of the comparison domain.

Note however that $\ov{Q}$ is chosen such that the bound $Q \leq \ov{Q}$ only implies $|h| \leq \eta_{\lin}$ when $\rho_1 \gtrsim r_{\comp}$.
So it does not imply $|h| \leq \eta_{\lin}$ everywhere on the comparison domain.
Unfortunately, we will not be able to remedy this issue by replacing $\ov{Q}$ in a priori assumption \ref{item_q_less_than_q_bar_6} by a smaller constant, as our solution of the harmonic map heat flow will introduce an error of magnitude depending on $\delta_{\nn}$ near the neck-like boundary of $\N$.

More specifically, assuming that the bound $Q\leq \ov Q$ holds near the neck-line boundary, which has scale $\approx r_{\comp}$, then errors   would force $ \ov Q\gtrsim r_{\comp}^E\eta'$ where $\eta'=\eta'(\delta_{\nn})$.  On the other hand, since we would want the inequality $Q\leq \ov Q$  to enforce  the bound $|h|\leq \eta_{\lin}$  everywhere in $\N$, and since $\N$ may contain points of scale $\approx\lambda r_{\comp}$, we would need to have $ \ov Q\lesssim (\lambda r_{\comp})^E\eta_{\lin}$.  Combining the two inequalities, we get  $\eta'(\delta_{\nn})\lesssim \lambda^E\eta_{\lin}$, so we end up with a condition of the form $\delta_{\nn}\leq \delta_{\nn}(\eta_{\lin},\lambda)$. However, to construct the comparison domain so that its boundary consists of (roughly)  $\delta_{\nn}$-necks, we need a condition of the form $\lambda\leq \lambda(\delta_{\nn})$, which is incompatible.  In summary, the constant $\ov{Q}$ cannot be chosen such that a priori assumption \ref{item_q_less_than_q_bar_6} is both weak enough to hold near the boundary of $\N$ and strong enough to imply $|h| \leq \eta_{\lin}$ at all points of scale $\gtrsim \lambda r_{\comp}$.

The bound $Q^* \leq \ov{Q}^*$ in a priori assumption \ref{item_q_star_less_than_q_star_bar}, on the other hand, automatically implies $|h| \leq \eta_{\lin}$ everywhere on $\N$.
However, we are not imposing it near the neck-like boundary of $\N$.

Lastly, note that the bound $Q \leq \ov{Q}$ may be violated after a Bryant Extension construction.
Therefore, we have not imposed it in a priori assumption \ref{item_q_less_than_q_bar_6} at points that lie in the near future of cuts.
At these points, the bound $Q^* \leq \ov{Q}^*$ will be used to guarantee $|h| \leq \eta_{\lin}$.
Moreover, the bound $Q \leq W \ov{Q}$ from a priori assumption \ref{item_q_less_than_w_q_bar_7} will be used to partially retain a priori assumption \ref{item_q_less_than_q_bar_6} in the future of a cut.
Using the interior decay from Subsection~\ref{sec_semi_local_max}, this bound can in turn be improved to the bound $Q \leq \ov{Q}$ from a priori assumption \ref{item_q_less_than_q_bar_6} after a sufficient time.

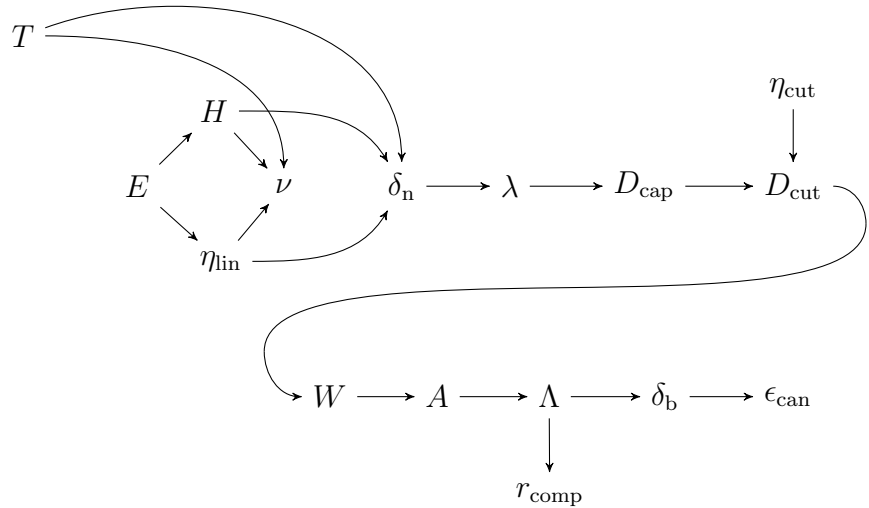
\begin{figure}
\begin{tikzpicture}[->,>=stealth',auto,node distance=3cm,main node/.style={circle,draw,font=\sffamily\Large\bfseries}]
  \node[anchor=west] at (0,2) (T) {$T$};
  \node[anchor=west] at (1.5,0) (E) {$E$};
  \node[anchor=west] at (2.5,1) (H) {$H$};
  \node[anchor=west] at (2.5,-1) (etalin) {$\eta_{\lin}$};
  \node[anchor=west] at (3.5,0) (nu) {$\nu$};
  \node[anchor=west] at (5,0) (deltan) {$\delta_{\nn}$};
  \node[anchor=west] at (6.5,0) (lambda) {$\lambda$};
  \node[anchor=west] at (8,0) (Dcap) {$D_{\CAP}$};
  \node[anchor=west] at (10,0) (Dcut) {$D_{\cut}$};
  \node[anchor=west] at (10.07,1.3) (etacut) {$\eta_{\cut}$};
  \node[anchor=west] at (4,-2.8) (W) {$W$};
  \node[anchor=west] at (5.5,-2.8) (A) {$A$};
  \node[anchor=west] at (7,-2.8) (Lambda) {$\Lambda$};
  \node[anchor=west] at (6.7,-4.1) (rcomp) {$r_{\comp}$};
   \node[anchor=west] at (8.5,-2.8) (deltab) {$\delta_{\bb}$};
   \node[anchor=west] at (10,-2.8) (epscan) {$\eps_{\can}$};
 \draw[->]  (E) edge (H)  (etalin) edge (nu) (E) edge (etalin) (H) edge (nu) (deltan) edge (lambda) (lambda) edge (Dcap) (Dcap) edge (Dcut) (etacut) edge (Dcut) (W) edge (A) (A) edge (Lambda) (Lambda) edge (deltab) (deltab) edge (epscan) (Lambda) edge (rcomp);
 \draw (T) edge[out=0,in=90,->] (nu);
 \draw (T) edge[out=20,in=90,->] (deltan);
 \draw (H) edge[out=0,in=120,->] (deltan);
 \draw (etalin) edge[out=0,in=240,->] (deltan);
 \draw[-] (Dcut) edge[out=0,in=90, looseness=1] (11.5,-0.5);
 \draw[-] (11.5,-0.5) edge[out=-90,in=90, looseness=0.5]  (3.5,-2.2) ;
 \draw[<-] (W) edge[in=-90,out=180, looseness=1] (3.5,-2.2);
\end{tikzpicture}
\caption{All restrictions on the parameters that will be imposed throughout this paper.
Each parameter in this graph can be chosen depending only on the parameters that can be reached by following the arrows backwards.
Note that the graph does not contain any oriented circles.\label{fig_parameter_dep}}
\end{figure}

\subsection{Parameter order}\label{subsec_parameter_order}
As mentioned earlier, the a priori assumptions, as introduced in the last two subsections, involve several parameters, which will need to be chosen carefully in the course of this paper.
Each step of our construction will require that certain parameters be chosen sufficiently small/large depending on certain other parameters.
In order to show that these parameters can eventually be chosen such that all restrictions are met, we need to ensure that these restrictions are not circular.
For this purpose, we introduce the following parameter order:
\[
 T, \hspace{1.2mm} E,  \hspace{1.2mm} H, \hspace{1.2mm} \eta_{\lin}, \hspace{1.2mm} \nu, \hspace{1.2mm} \delta_{\nn}, \hspace{1.2mm} \lambda, \hspace{1.2mm} D_{\CAP}, \hspace{1.2mm} \eta_{\cut}, \hspace{1.2mm} D_{\cut}, \hspace{1.2mm} W, \hspace{1.2mm} A, \hspace{1.2mm} \Lambda, \hspace{1.2mm} \delta_{\bb}, \hspace{1.2mm} \eps_{\can},  \hspace{1.2mm} r_{\comp} 
\]
In the entire paper, we will require each parameter to be chosen depending only on preceding parameters in this list.
So parameters can eventually be chosen successively in the order indicated by this list.

For a more detailed picture of all the parameter restrictions imposed in this paper see Figure~\ref{fig_parameter_dep}.
These restrictions also appear in the preamble of our main technical result, Theorem~\ref{Thm_existence_comparison_domain_comparison}.
Note that, as these restrictions are not completely linear, there are several admissible parameter orders.
We have chosen the above parameter order, because we found it to be most intuitive.

We advise the first-time reader that it is not necessary to follow all parameter restrictions in detail when going through the proofs of this paper.
Instead, it suffices to check that the above parameter order is obeyed in each step.

\section{Preparatory results}
\label{sec_preparatory_results}
In this section we collect a variety of technical results that will be needed in the proof of the main theorem.  
These are based on definitions from Sections~\ref{sec_preliminaries_I}--\ref{sec_apas}.  
The reader may wish to skim (or skip) this section on a first reading.   

\subsection{Consequences of the canonical neighborhood assumption}
The completeness and canonical neighborhood assumptions, as introduced in Definitions~\ref{def_completeness} and \ref{def_canonical_nbhd_asspt} lead in a straightforward way to local bounds on geometry, including local control on curvature and its derivatives, as well as control on neck and non-neck structure.  We begin this subsection with a few such results (Lemmas~\ref{lem_properties_kappa_solutions_cna}--\ref{lem_containment_parabolic_nbhd}), and then use them to deduce  control on scale distortion of bilipschitz maps (Lemma~\ref{lem_scale_distortion}), self-improvement of necks (Lemma~\ref{lem_C0_neck_smooth_neck}) and scale bounds near necks (Lemma~\ref{lem_thick_close_to_neck}).

Our first two results are direct consequences of the definition of the canonical neighborhood assumption, and properties of $\kappa$-solutions.

\begin{lemma}\label{lem_properties_kappa_solutions_cna}
The following hold:
\begin{enumerate}[label=(\alph*)]
\item 
\label{item_p_x_a_scale_curvature_bound} 
For every $A<\infty$ there is a constant $C=C(A)<\infty$ such that if
$$
\eps_{\can}\leq \ov\eps_{\can}(A)
$$
and a Riemannian manifold $M$ satisfies the $\eps_{\can}$-canonical neighborhood assumption at $x\in M$, then the following holds on the ball $B(x, A \rho(x))$ for all $0 \leq m \leq A$
\begin{gather*}
 \rho=\sqrt{3} R^{-\frac12}\,,\qquad C^{-1}\rho(x)\leq \rho\leq C\rho(x), \qquad
|\nabla^m \Rm|\leq C \rho^{-2-m}(x) .
\end{gather*} 

\item \label{item_dt_r_bound} There is a $C<\infty$ such that if
$$
\eps_{\can}\leq\ov\eps_{\can}
$$
and $\M$ is a Ricci flow spacetime that satisfies the $\eps_{\can}$-canonical neighborhood assumption at some point $x \in \M$, then 
$$
|\D_t\rho^2| (x) =3|\D_tR^{-1}| (x)\leq C\,.
$$
\item
\label{item_rho_squared_derivative_small} Given $\delta>0$,  if
$$
\eps_{\can}\leq\ov\eps_{\can}(\delta)
$$
and $\M$ is a Ricci flow spacetime that  satisfies the $\eps_{\can}$-canonical neighborhood assumption at some point $x \in \M$, then 
$$
\D_t\rho^2 (x) =3\D_tR^{-1} (x) \leq \delta\,.
$$

\end{enumerate}
\end{lemma}
\begin{proof}
Assertion \ref{item_p_x_a_scale_curvature_bound} follows from the definition of the canonical neighborhood assumption, 
assertions \ref{ass_C.1_c} and \ref{ass_C.1_d} of Lemma~\ref{lem_kappa_solution_properties_appendix} and Lemma~\ref{lem_rho_Rm_R}.

For assertions \ref{item_dt_r_bound} and \ref{item_rho_squared_derivative_small} we recall that in  a Ricci flow the time derivative $\D_tR(x)$ may be expressed as a universal continuous function $F(\Rm(x),\nabla\Rm(x), \nabla^2 \Rm (x))$ of  spatial curvature derivatives.  
Now assertions \ref{item_dt_r_bound} and \ref{item_rho_squared_derivative_small} follow from the definition of the canonical neighborhood assumption, and
assertion \ref{ass_C.1_e} of Lemma~\ref{lem_kappa_solution_properties_appendix}.
\end{proof}

\begin{lemma} \label{lem_neck_or_cap}
For every $\delta > 0$ there is a constant $C_0=C_0(\delta) < \infty$ such that if
\[ \eps_{\can} \leq \ov\eps_{\can} (\delta), \]
then the following holds.

Assume that $(M,g)$ is a Riemannian manifold that satisfies the $\eps_{\can}$-canonical neighborhood assumption at some point $x \in M$.
Then  one of the following hold:
\begin{enumerate}[label=(\alph*)]
\item \label{ass_8.2_a} $x$ is the center of a $\delta$-neck at scale $\rho(x)$.
\item \label{ass_8.2_b} There is a compact, connected domain $V \subset \M_t$ with connected (possibly empty boundary) such that the following hold:
\begin{enumerate}[label=(\arabic*)]
\item  $B(x,\delta^{-1} \rho(x)) \subset V$.
\item $\rho (y_1)  < C_0 \rho(y_2)$ for all $y_1, y_2 \in V$.
\item $\diam V < C_0 \rho(x)$.
\item If $\D V \neq \emptyset$, then: 
	\ben
	\item $\partial V$ is a central $2$-sphere of a $\delta$-neck.
	\item Either $V$ is  a 3-disk or is diffeomorphic  to a twisted interval bundle over $\R P^2$.
	\item Any two points $z_1, z_2 \in \partial V$ can be connected by a continuous path inside $\partial V$ whose length is less than \[ \min \{ d (z_1, x),  d (x, z_2) \} - 100 \rho (x). \]
	\een
\item If $V$ is diffeomorphic to a twisted interval bundle over $\R P^2$, then $\rho (y_1) < 2 \rho(y_2)$ for all $y_1, y_2 \in V$.
\end{enumerate}
\end{enumerate}

\end{lemma}

\begin{proof}
This follows immediately from Lemma \ref{lem_neck_or_cap_kappa_solution} using the definition of the canonical neighborhood assumption.
\end{proof}

\begin{lemma}\label{lem_unscathedness_criterion}
Suppose $\M$ is an $(r_0, t_0)$-complete Ricci flow spacetime.  If for some $r > r_0$ we have $\rho > r$ on  a parabolic  neighborhood $P(x,a,b) \subset \M_{[0,t_0]}$, then it is unscathed. 
\end{lemma}
\begin{proof}
Let $t=\t(x)$.  From the $(r_0,t_0)$-completeness of $\M$, any unit speed geodesic in $\M_t$ starting at $x$ can be extended up to a length of at least $a$.   
Therefore the exponential map $\exp_x:T_x\M_t\supset \ol{B(0,a)}\ra \M_t$ is well-defined, and has compact image $\exp_x(\ol{B(0,a)})=\ol{B(x,a)}$.  If $y\in B(x,a)$, then since $\rho> r$ on $P(x,a,b)$, it follows from $(r_0,t_0)$-completeness that $y(\bar t)$ is defined on $[t,t+b]$ if $b>0$ or $[t+b,t]\cap [0,\infty)$ if $b<0$. 
\end{proof}

Next, we  derive a few results based on the bounds in Lemma~\ref{lem_properties_kappa_solutions_cna}. 

\begin{lemma}[Scale nearly constant on small two-sided parabolic balls]
\label{lem_forward_backward_control}
If  $L>1$ and 
$$
\eta\leq\ov\eta(L)\,,\qquad \eps_{\can}\leq \ov\eps_{\can},
$$ 
then the following holds.

Suppose $0<r  \leq 1$ and $\M$ is an $(\eps_{\can}r, t_0)$-complete Ricci flow spacetime satisfying the $\eps_{\can}$-canonical neighborhood assumption at scales $(\eps_{\can}r,1)$.  If for some point $x\in \M_t$ with $t \in [0,t_0]$ we have $\rho_1(x)\geq r$, then the parabolic neighborhoods $P_{\pm} := P(x, \eta \rho_1 (x), \pm (\eta \rho_1 (x))^2) \cap \M_{[0,t_0]}$ are unscathed and
\begin{equation}
\label{eqn_rho_1_upper_lower_bound}
L^{-1}\rho_1(x)\leq \rho_1\leq L\rho_1(x)
\end{equation}
on $P_+ \cup P_-$.
\end{lemma}

\begin{proof}
If 
$$
\eps_{\can}\leq\ov\eps_{\can},
$$
then by Lemma \ref{lem_rho_Rm_R} and assertions \ref{item_p_x_a_scale_curvature_bound} and \ref{item_dt_r_bound} of Lemma \ref{lem_properties_kappa_solutions_cna} there is a constant $C_0<\infty$ such that near any point that satisfies the $\eps_{\can}$-canonical neighborhood assumption we have
\begin{equation}\label{eqn_rho_gradient_estimate}
 |\nabla \rho|,|\D_{\t} \rho^2|  \leq C_0\,.
\end{equation}
Now choose a point $y\in P_\pm$, and let $\ga : [0,a] \to \M$ be a curve from $x$ to $y$ that is a concatenation of curves $\ga_1, \ga_2$, where $\ga_1$ is a unit speed curve  from $x$ to $y(t)$ of length $<\eta\rho_1(x)$, and $\ga_2$ is the integral curve of $\pm\D_{\t}$ from $y(t)$ to $y$.   Then by (\ref{eqn_rho_gradient_estimate}), for $i=1,2$, we have 
\begin{equation}
\label{eqn_curve_derivative_estimate}
|(\rho_1\circ \ga_1)'(s)|\leq C_0\,,\qquad |(\rho_1^2\circ\ga_2)'(s)|\leq C_0
\end{equation}
wherever the derivatives are defined and $\rho_1\circ\ga_i(s)> \eps_{\can}r$.  Therefore if 
$$
\eta\leq\ov\eta(L),
$$
then (\ref{eqn_rho_1_upper_lower_bound}) follows by integrating the derivative  bound (\ref{eqn_curve_derivative_estimate}).
The fact that $P_{\pm}$ are unscathed follows from Lemma~\ref{lem_unscathedness_criterion}.
\end{proof}

\begin{lemma}[Backward survival control]
\label{lem_lambda_thick_survives_backward}
If $\delta>0$, $A<\infty$ and
\[\eps_{\can}\leq\ov\eps_{\can}(\delta,A)\,,\]
then the following holds.

Suppose $r>0$ and $\M$ is an $(\eps_{\can}r, t_0)$-complete Ricci flow spacetime satisfying the $\eps_{\can}$-canonical neighborhood assumption at scales $(\eps_{\can} r,r)$.  
Let $x \in \M_t$ with $t \in [0,t_0]$ and assume that $\rho(x) \geq r$.
Then $x(t')$ exists for all $\ov{t} \in [ t- Ar^2, t] \cap [0, \infty)$ and we have $\rho (x(\ov{t})) > (1-\delta) r$.
\end{lemma}

\begin{proof}
Set $t := \t (x)$ and let $\delta_\#>0$ be a constant whose value we will choose at the end of the proof.
Recall that $\rho_r = \min \{ \rho, r \}$.
By assertion \ref{item_rho_squared_derivative_small} of Lemma~\ref{lem_properties_kappa_solutions_cna}, and assuming
\[ \eps_{\can} \leq \ov\eps_{\can} (\delta_\#), \]
we have
\begin{equation} \label{eq_rho_r_dt_bound}
 \frac{d}{d\ov{t}} \big( \rho^2_r (x(\ov{t}) ) \big) \leq \delta_\#  
\end{equation}
for all $\ov{t} \leq t$ for which both $x (\ov{t})$ and the derivative exist and $\rho^2_r (x(\ov{t})) > (\eps_{\can} r)^2$.
Therefore, if 
\[ \delta_\# \leq \ov\delta (\delta, A), \]
then we may integrate (\ref{eq_rho_r_dt_bound}) to obtain that $\rho_r^2 (x ( \ov{t} )) > (1-\delta ) r$ for all $\ov{t} \leq t$ for which $t - \ov{t} \leq A r^2$ and $x (\ov{t})$ is defined.
Assuming
\[ \eps_{\can} < 1 - \delta, \]
we can use the $(\eps_{\can} r, t_0)$-completeness to show that $x (\ov{t})$ is defined for all $\ov{t} \in [ t- Ar^2, t] \cap [0, \infty)$.
\end{proof}

\begin{lemma}[Bounded curvature at bounded distance] \label{lem_bounded_curv_bounded_dist}
For every $A < \infty$ there is a constant $C = C (A) < \infty$ such that if
\[ \eps_{\can} \leq \ov\eps_{\can} (A), \]
then the following holds.

Let $0 < r \leq  1$ and consider an $(\eps_{\can} r, t_0)$-complete Ricci flow spacetime  $\M$ that satisfies the $\eps_{\can}$-canonical neighborhood assumption at scales $(\eps_{\can} r, 1)$.   If $x \in \M_{[0,t_0]}$ and  $\rho_1 (x) \geq    r$, then $P(x, \lb A \rho_1 (x))$ is unscathed and we have
\begin{equation}\label{eqn_bdded_curv_bdded_dist_lem} C^{-1} \rho_1(x) < \rho_1 < C \rho_1 (x) \qquad \text{on} \qquad P(x,  A \rho_1 (x)). 
\end{equation}
\end{lemma}

\begin{proof}
We claim that there is a constant $C_1=C_1(A) < \infty$ such that 
\begin{equation}\label{eqn_scale_control_p_x_a}
C_1^{-1}\rho_1(x)\leq \rho_1\leq C_1\rho_1(x) \qquad \text{on} \qquad B(x,A\rho_1(x))\,.
\end{equation} 
This is immediate if $\rho_1\equiv 1$ on $B(x,A\rho_1(x))$, so suppose $\rho_1(y)<1$ for some $y\in B(x,A\rho_1(x))$.  By the continuity of $\rho_1$, we may choose $y$ such that $\rho_1(y)  \in ( \frac12 \rho_1(x) , 1)$.  
Applying assertion \ref{item_p_x_a_scale_curvature_bound} of Lemma~\ref{lem_properties_kappa_solutions_cna} to the ball $B(y,4A\rho(y))\supset B(x,A\rho_1(x))$, we get (\ref{eqn_scale_control_p_x_a}).

If 
\[ \eps_{\can}\leq\ov\eps_{\can}(A) , \]
then using (\ref{eqn_scale_control_p_x_a}), we may apply Lemma~\ref{lem_lambda_thick_survives_backward} at any point $z\in B (x,A\rho_1(x))$ to conclude that $\ga(\bar t)$ is defined and 
\begin{equation*}\label{eqn_rho_1_lower_bound}
\rho_1(z(\bar t))\geq \tfrac12 C_1^{-1}\rho_1(x) > \eps_{\can} r
\end{equation*}
for all $\bar t\in [t-A\rho_1(x)^2,t]\cap [0,\infty)$. 
Thus $P(x, A \rho_1 (x))$ is unscathed by Lemma~\ref{lem_unscathedness_criterion}.

Next, by assertion \ref{item_p_x_a_scale_curvature_bound} of Lemma~\ref{lem_properties_kappa_solutions_cna} there is a universal constant $C_2< \infty$ such that if
\[\eps_{\can}\leq\ov\eps_{\can}(A), \]
then for all $\bar t\in [t-A\rho_1^2(x)]\cap [0,\infty)$ we have  $|\frac{d}{d\bar t} (\rho_1^2(z(\bar t))|<C_2$, provided the derivative is defined.  Integrating this bound yields $\rho^{2}_1(z(\bar t))  \leq C^{2} \rho_1^{2}(z)$ for $C=C(A) < \infty$.   Thus (\ref{eqn_bdded_curv_bdded_dist_lem}) holds.
\end{proof}

In the next result we combine the bounded curvature at bounded distance estimate (Lemma~\ref{lem_bounded_curv_bounded_dist}) with a distance distortion estimate to find a parabolic neighborhood centered at a point $x$ that contains all parabolic neighborhoods of the form $P(y,A_2\rho_1(y))$, where $y$ varies over some parabolic neighborhood $P(x,A_1\rho_1(x))$.

\begin{lemma}[Containment of parabolic neighborhoods] \label{lem_containment_parabolic_nbhd}
For any $A_1, \lb A_2 \lb < \lb \infty$ there is a constant $A' = A' (A_1, A_2) < \infty$ with $A' \geq A_1 + A_2$ such that if
\[ \eps_{\can} \leq \ov\eps_{\can} (A_1, A_2), \]
then the following holds.

Let $0 < r \leq  1$ and consider an $(\eps_{\can} r, t_0)$-complete Ricci flow spacetime $\M$  that satisfies the $\eps_{\can}$-canonical neighborhood assumption at scales $(\eps_{\can} r, 1)$.  
If $x \in \M_{[0,t_0]}$ and $\rho_1 (x) \geq   r$, then the parabolic neighborhood $P(x, A' \rho_1 (x))$ is unscathed and we have
\begin{equation}
\begin{aligned}
 \label{eq_P_containment_lemma}
 \qquad P(y, A_2 \rho_1(y)) \subset P(x, A' \rho_1 (x)) 
 \end{aligned}
\end{equation}
for all $y \in P(x, A_1 \rho_1 (x))$. 
\end{lemma}

\begin{proof}
We first use Lemma~\ref{lem_bounded_curv_bounded_dist}, assuming
\[ \eps_{\can} \leq \ov\eps_{\can} (A_1), \]
to argue that $P(x, A_1 \rho_1 (x))$ is unscathed and 
\begin{equation} \label{eq_rho_C1A1_containment}
\rho_1 < C_1 (A_1) \rho_1 (x) \qquad \text{on} \qquad P(x, A_1 \rho_1 (x))
\end{equation}
for some $C_1 = C_1 (A_1) < \infty$.

The constant $A'$ will be determined in the course of the proof.
Again, by Lemma~\ref{lem_bounded_curv_bounded_dist}, assuming
\begin{equation} \label{eq_eps_A_1_A_2_bound}
 \eps_{\can} \leq \ov\eps_{\can} (A'), 
\end{equation}
we find that $P(x, 2A' \rho_1 (x))$ is unscathed and that $\rho_1 > c_2 (A') \rho_1 (x) \lb \geq c_2 r \lb > \eps_{\can} r$ on it.
At any point $z \in P(x, 2A' \rho_1 (x))$ with $\rho_1 (z) < 1$ the curvature operator is close to that of a $\kappa$-solution.
Since $\kappa$-solutions have non-negative Ricci curvature, we can argue that
\[ \Ric \geq - c_2^2 (A') \rho^{-2} (z) \geq - \rho_1^{-2} (x) \]
at $z$  if we assume a bound of the form (\ref{eq_eps_A_1_A_2_bound}).
On the other hand, at any $z \in P(x, 2A' \rho_1 (x))$ with $\rho_1 (z) = 1$ we have $\rho (z) \geq 1$ and therefore $\Ric \geq - C_2 $ at $z$ for some universal constant $C_2$.
So, in summary, we have
\begin{equation} \label{eq_lower_Ric_bound_P}
 \Ric \geq - C_2 \rho_1^{-2}(x)  \qquad \text{on} \qquad P(x, 2A' \rho_1 (x)). 
\end{equation}

Now consider a point $y \in P(x, A_1 \rho_1 (x))$.
Set $t_x  := \t (x)$ and $ t_y := \t (y)$.
We first claim that for
\[ A' \geq \underline{A}' (A_1, A_2) \]
we have
\begin{equation} \label{eq_B_yt_P_x_A_prime}
 B(y, A_2 \rho_1 (y)) \subset P(x,  A' \rho_1 (x)). 
\end{equation}
Assume not and choose a smooth curve $\gamma : [0,1] \to (P(x, 2A' \rho_1(x))_{t_y}$ between $y$ and a point in $z \in P(x, \lb 2A' \rho_1(x)) \lb \setminus \lb P(x, \lb A' \rho_1 (x))$ such that $\ell_{t_y} (\gamma) <  A_2 \rho_1 (y)$.
Note that for all $t' \in  [t_y, t_x]$ the curve $\gamma_{t'} : [0,1] \to \M_t$ with $\gamma_{t'} (s) := (\gamma(s))(t')$ is defined and its image is contained in $P(x, 2A' \rho_1 (x))$.
So by (\ref{eq_lower_Ric_bound_P}) and (\ref{eq_rho_C1A1_containment}) we have
\begin{multline*}
 d_{t_x } (y(t_x ), z(t_x )) \leq \ell_{t_x } (\gamma_{t_x } ) 
  <   \exp \big(  C_2 \rho_1^{-2} (x)  A_1^2 \rho_1^2(x ) \big) \cdot A_2 \rho_1 (y) \\
   \leq C_1 A_2 \exp ( C_2 A_1^2  ) \rho_1(x)
\end{multline*}
So
\begin{equation*}
 A' \rho_1(x) \leq d_{t_x} (x, y) + d_{t_x } (y, z(t_x )) 
 < A_1 \rho_1(x) + C_1 A_2 \exp ( C_2 A_1^2  ) \rho_1(x) . 
\end{equation*}
Now set
\[ A' (A_1, A_2) := A_1  + C_1 A_2 \exp ( C_2 A_1^2  ) + \sqrt{A_1^2 + A_2^2} . \]
Then we obtain a contradiction and thus (\ref{eq_B_yt_P_x_A_prime}) holds.
Since $A^{\prime 2} \geq A_1^2 + A_2^2$, we obtain (\ref{eq_P_containment_lemma}).
\end{proof}

The next two results concern the behavior of the curvature scale $\rho$ under nearly isometric mappings.  
We begin with a convergence lemma that shows that an immersion between Riemannian manifolds must nearly preserve the scale, provided it is nearly an isometry, and we have sufficient control on the curvature and possibly curvature derivatives on the domain and target. 
The main point is that the map is only assumed to be an almost isometry in the $C^0$-sense.

\begin{lemma}
\label{lem_convergence_of_scales}
Suppose $\{(Z^1_k,g^1_k,z^1_k\}_{k=1}^\infty$, $\{(Z^2_k,g^2_k,z^2_k)\}_{k=1}^\infty$ are sequences of pointed smooth Riemannian manifolds such that for some $r_0 > 0$ and for each $i=1,2$  the  ball $B(z^i_k,r_0)\subset Z^i_k$ is relatively compact for all $k$, and one of the following holds:
\begin{enumerate}[label=(\roman*)]
\item \label{con_8.20_i} $\sup_{B(z^i_k,r_0)}|{\Rm}|_{g^i_k}\ra 0$ as $k\ra\infty$.
\item \label{con_8.20_ii} $\limsup_{k\ra\infty}\sup_{B(z^i_k,r_0)}|\nabla^j{\Rm}|_{g^i_k}<\infty$ for $0\leq j\leq 5$.
\end{enumerate}
Let $\{\phi_k:Z^1_k\ra Z^2_k\}_{k=1}^{\infty}$ be a sequence of smooth maps  such that $\phi_k(z^1_k)=z^2_k$ and 
\begin{equation}
\label{eqn_almost_isometry}
\sup_{B(z^1_k,r_0)} \big| (\phi_k^*g^2_k-g^1_k) \big|_{g^1_k}\ra 0 \qquad  \text{as} \qquad k \ra \infty\,.
\end{equation}
Then, after passing to a subsequence, the scale functions converge to the same limit: 
$$
\lim_{k\ra\infty}\rho(z^1_k)=\lim_{k\ra\infty}\rho(z^2_k)\in [0,\infty)\cup\{\infty\}\,.
$$
\end{lemma}

\begin{proof}
We first prove the lemma under the additional conditions that the $\phi_k$s are diffeomorphisms and the injectivity radii  at $z^1_k$ satisfy 
\begin{equation}
\label{eqn_injrad_z_1}
\liminf_{k\ra\infty}\injrad(Z^1_k,g^1_k,z^1_k)>0\,.
\end{equation}

Using standard injectivity radius estimates, conditions  \ref{con_8.20_i}, \ref{con_8.20_ii},  (\ref{eqn_almost_isometry}), and (\ref{eqn_injrad_z_1}) imply that for every $r<r_0$, and sufficiently large $k$, the injectivity radius is bounded uniformly from below on $B(z^i_k,r)\subset Z^i_k$.  By standard compactness arguments, after passing to a subsequence, the sequence of pointed balls $\{(B(z^i_k,r_0),g^i_k,z^i_k)\}_{k=1}^\infty$ converges to a pointed  $C^4$-Riemannian manifold $(Z^i_\infty,g^i_\infty,z^i_\infty)$ that is a proper $r_0$-ball (i.e. balls of radius $<r_0$ are relatively compact), and there is a basepoint preserving  map $\phi_\infty:(Z^1_\infty,z^1_\infty)\ra (Z^2_\infty,z^2_\infty)$ that is an isometry of the Riemannian  distance functions, where for each $i=1,2$:
\begin{itemize}
\item  If $\{g^i_k\}$ satisfies \ref{con_8.20_i}, then the pointed convergence $(B(z^i_k, \lb r_0), \lb g^i_k, \lb z^i_k) \lb \ra \lb (Z^i_\infty, \lb z^i_\infty)$ is with respect to the Gromov-Hausdorff topology and $Z^i_\infty$ is flat.
\item If $\{g^i_k\}$ satisfies \ref{con_8.20_ii}, then the pointed convergence $(B(z^i_k, \lb r_0), \lb g^i_k, \lb z^i_k) \lb \ra \lb (Z^i_\infty, \lb g^i_\infty, \lb z^i_\infty)$ is with respect to the $C^4$-topology.
\end{itemize}
In view of the above we have $\rho(z^i_k)\ra \rho(z^i_\infty)\in [0,\infty)\cup\{\infty\}$ as $k\ra\infty$ for $i=1,2$.  Since $\phi_\infty$ is an isometry (of distance functions) between $C^4$ Riemannian manifolds, it is a $C^3$-isometry of Riemannian manifolds, and hence it preserves curvature tensors: $\phi_\infty^*(\Rm(z^2_\infty))=\Rm(z^1_\infty)$.  It follows that $\rho(z^1_\infty)=\rho(z^2_\infty)$.

We now return to the general case.  We may assume after shrinking $r_0$ that the conjugate radius of $Z^1_k$ at $z^1_k$ is $\geq 2r_0$.  For $i=1,2$ let $(W^i_k,w^i_k)$ be the ball $B(0,2r_0)\subset T_{z^1_k}Z^1_k$ with basepoint $w^i_k=0\in B(0,2r_0)$, and let $h^1_k:=\exp_{z^1_k}^*g^i_k$, $h^2_k:=(\phi_k\circ\exp_{z^1_k})^*g^2_k$.  Then the injectivity radius at $w^1_k$ satisfies $\injrad(W^1_k,h^1_k,w^1_k)\geq r_0$, and $B(w^i_k,r_0)\subset W^i_k $ is relatively compact.  
Therefore, applying the above argument to the identity maps $W^1_k\ra W^2_k$, we obtain the lemma.
\end{proof}

\begin{lemma}[Scale distortion of bilipschitz maps] \label{lem_scale_distortion}
There is a constant $10^3 < C_{\sd} < \infty$   such that the following holds if
\[
  \eta_{\lin} \leq \ov\eta_{\lin},  \qquad
  \delta_{\nn} \leq \ov{\delta}_{\nn}, \qquad 
  \eps_{\can} \leq \ov\eps_{\can}, \qquad 
  r_{\comp} \leq \ov{r}_{\comp}.
\]

Let $\M, \M'$ be $(\eps_{\can} r_{\comp}, t_0)$-complete Ricci flow spacetimes.
Consider a closed product domain $X \subset \M_{[0,t_0]}$ on a time-interval of the form $[t - r_{\comp}^2, t]$, $t \geq r_{\comp}^2$, such that the following holds:
\begin{enumerate}[label=(\roman*)]
\item \label{con_8.23_i} $\partial X_{t}$ consists of embedded $2$-spheres that are each centers of $\delta_{\nn}$-necks at scale $r_{\comp}$.
\item \label{con_8.23_ii} Each connected component of $X_{t}$ contains a $2 r_{\comp}$-thick point.
\end{enumerate}

Let $\bar t \in [t - r_{\comp}^2, t]$, $t' \geq 0$ and consider a diffeomorphism onto its image $\phi : X_{\bar t} \to \M'_{t'}$ such that $| \phi^* g'_{t'} - g_{\bar t} | \leq \eta_{\lin}$.  
We assume that $\M$ satisfies the $\eps_{\can}$-canonical neighborhood assumption  at scales $(0,1)$ on $X_{\bar t}$, and that $\M'$ satisfies the $\eps_{\can}$-canonical neighborhood assumption at scales $(0,1)$  on $\phi(X_{\bar t})$.

Then for any $x \in X_{\ov{t}}$ we have
\begin{equation} \label{eq_rho_phi_rho_bound}
 C_{\sd}^{-1} \rho_1 (x) < \rho_1 (\phi (x)) < C_{\sd} \rho_1 (x). 
\end{equation}
\end{lemma}

This lemma will later be applied whenever a bound on the distortion of the scale function under a comparison (as defined in Definition~\ref{def_comparison}) is needed.
The product domain $X$ in this lemma will later be taken to be a time-slab $\N^j$ of a comparison domain (as defined in Definition~\ref{def_comparison_domain}) and $\phi$ will denote the time-slice of a comparison.
Assumptions \ref{con_8.23_i} and \ref{con_8.23_ii} correspond to a priori assumptions \ref{item_backward_time_slice_3}(a) and \ref{item_backward_time_slice_3}(d), respectively (see Definition~\ref{def_a_priori_assumptions_1_7}).

In order to avoid confusion, we point out that usually it is possible to derive stronger scale distortion bounds than (\ref{eq_rho_phi_rho_bound}), with $C_{\sd}$ replaced by a constant that can be chosen arbitrarily close to 1.
These stronger bounds follow simply via local gradient estimates, due to the parabolic nature of the comparison.
This approach, however, fails if the point $x$ lies close to the spatial or time-like boundary of $X$.
This is why we have to work with a larger constant $C_{\sd}$ in this paper.

\begin{proof}
Assume that the lemma was  false.  
Then there are sequences 
$\eta_{\lin,k}\ra 0$, $ \delta_{\nn,k}\ra 0$, $ \eps_{\can,k}\ra 0$,  $r_{\comp,k}\ra 0$, $\{\M^k\}$, $\{\M^{\prime k}\}$, $\{X^k\}$, $\{x_k\}$, $\{t_k\}$, $\{\bar t_k\}$, $\{t'_k\}$, $\phi_k:X^k_{\bar t_k}\ra\M^{\prime k}_{t'_k}$ satisfying the assumptions of the lemma, such that 
\begin{equation}
\label{eqn_ratio_blows_up}
\frac{\rho_1(x_k)}{\rho_1(\phi_k(x_k))}
\lra 0 \quad \text{or} \quad \infty \qquad \text{as} \qquad k\ra \infty\,.
\end{equation}

To simplify notation, we let $M_k:=\M^k_{\bar t_k}$ and $M_k':=\M^{'k}_{t'_k}$ denote the time-slices, with  metrics $g_k$ and $g_k'$, respectively, and let $Y_k:= X^k_{\bar t_k}\subset M_k$ be the relevant time-slice of the product domain $X^k$.

Let $r_k:=\min\{\rho_1(x_k),\rho_1(\phi_k(x_k))\}$.   In view of (\ref{eqn_ratio_blows_up}) we have $r_k\ra 0$.   
Note that by our assumptions, for each of $x_k$, $\phi_k(x_k)$, either the $\eps_{\can,k}$-canonical neighborhood assumption holds or we have $\rho_1 (x_k) = 1$ or $\rho_1 (\phi_k (x_k)) = 1$, respectively.
In the first case we may use the estimates on the derivatives of curvature in assertion \ref{item_p_x_a_scale_curvature_bound} of Lemma~\ref{lem_properties_kappa_solutions_cna}, and we have
\begin{equation} \label{eq_curv_deriv_bounded_in_proof}
 |\nabla^j {\Rm}| < C_1 r_k^{-2-j} \qquad \text{on} \qquad B(x_k, r_k) \quad \text{or} \quad B(\phi_k(x_k) , r_k), 
\end{equation}
respectively for some universal $C_1 < \infty$ and large $k$ and $0 \leq j \leq 5$, and in the second case we may apply Lemma~\ref{lem_bounded_curv_bounded_dist} to obtain
\begin{equation}
\label{eqn_curvature_bounded_by_c}
|{\Rm}|<C_2  \qquad \text{on} \qquad B(x_k,r_k) \quad \text{or} \quad B(\phi_k (x_k), r_k),
\end{equation}
respectively, for some universal $C_2 < \infty$ and large $k$.

\textit{Case 1: \quad $\liminf_{k\ra\infty}r_k^{-1} d(x_k,\D Y_k)>0$.}

If we let  $\wh g_k:=r_k^{-2}g_k$, $\wh g_k':=r_k^{-2}g_k'$, then the assumptions of Lemma \ref{lem_convergence_of_scales} hold for the sequence $\{\phi_k:(\Int Y_k,\wh{g}_k,x_k)\ra(M_k',\wh{g}_k',\phi_k(x_k))\}$ by (\ref{eqn_ratio_blows_up}), (\ref{eq_curv_deriv_bounded_in_proof}), (\ref{eqn_curvature_bounded_by_c}) and the fact that $r_k \to 0$.  
Hence, after passing to a subsequence,
$$
\lim_{k\ra\infty}\rho_{\wh{g}_k}(x_k)=\lim_{k\ra\infty}\rho_{\wh{g}_k'}(\phi_k(x_k))\,.
$$ 
Since for every $k$ the $\eps_{\can,k}$-canonical neighborhood assumption holds at  one of the points $x_k$, $\phi_k(x_k)$, the above limit must equal $1$.  This contradicts (\ref{eqn_ratio_blows_up}).

\textit{Case 2: \quad $\liminf_{k\ra\infty}r_k^{-1} d(x_k,\D Y_k)=0$.}

After passing to a subsequence, we may assume that 
\begin{equation}
\label{eqn_dist_x_k_boundary}
\lim_{k\ra\infty}r_k^{-1}  d(x_k, \lb \D Y_k)=0\,.
\end{equation} 
For each $k$ we may choose a boundary component $\Sigma_k\subset X^k_{t_k}$ such that $\lim_{k\ra\infty}r_k^{-1} d(x_k, \lb\Sigma_k(\bar t_k)) \lb=  \lb 0$.   
Let $U_k$ be the $10r_{\comp,k}$-neighborhood of $\Sigma_k$ in $X^k_{t_k}$.
If $k$ is large, then $\frac12 r_{\comp, k} < \rho < 2 r_{\comp,k}$ on $U_k$.
So by assumption \ref{con_8.23_ii} and the fact that $\delta_{\nn,k}\ra 0$, it follows that $U_k$ does not fully contain the component of $X^k_{t_k}$ in which it lies, and moreover it does not intersect any other boundary components of $X^k_{t_k}$.
Therefore, we can pick $ y_k \in X^k_{t_k}$ with $d(y_k ,\Sigma_k)=r_{\comp,k}$.
By Lemma~\ref{lem_bounded_curv_bounded_dist} there is a universal constant $C_3 < \infty$ such that for large $k$ 
 we have 
\begin{equation}
\label{eqn_rho_on_u_k_t_k_bar}
C_3^{-1}r_{\comp,k}\leq \rho\leq C_3 r_{\comp,k}
\end{equation}
on $U_k(\bar t_k)$, in particular on $\Sigma_k(\bar t_k)$.  By (\ref{eq_curv_deriv_bounded_in_proof}) or (\ref{eqn_curvature_bounded_by_c}) and the fact that $B(x_k,r_k)\cap \Sigma_k(\bar t_k)\neq\emptyset$ for large $k$, we get $r_k\leq C_4r_{\comp,k}$ for large $k$, where $C_4 < \infty$ is a universal constant.  By (\ref{eqn_rho_on_u_k_t_k_bar}), (\ref{eqn_dist_x_k_boundary}),  and a distance distortion estimate, we have $x_k\in U_k(\bar t_k)$, and therefore 
$r_k \leq \rho_1 (x_k) \leq C_3 r_{\comp,k}$ for large $k$.
Hence $\lim_{k \to \infty} r_{\comp, k}^{-1} d(x_k, \partial Y_k) = 0$.
By 
a distance distortion estimate, there is a universal constant $C_5 < \infty$ such that for large $k$
\begin{equation}\label{eqn_bar_x_k_bdy}
C_5^{-1}r_{\comp,k}\leq d( y_k (\bar t_k),x_k),d(  y_k (\bar t_k),\D Y_k)\leq C_5 r_{\comp,k}\,.
\end{equation}
So using (\ref{eqn_bar_x_k_bdy}) and Case 1, we can find a uniform $C_6 < \infty$ such that  
\[ C_3^{-1} C_6^{-1} r_{\comp, k} \leq C^{-1}_6 \rho_1(y_k  (\bar t_k))\leq\rho_1(\phi_k(y_k (\bar t_k)))\leq C_6 \rho_1(y_k  (\bar t_k)) \leq C_3 C_6 r_{\comp, k}. \]   
Since $d(\phi_k(y_k  (\bar t_k)),\phi_k(x_k))\leq 2C_3r_{\comp,k}$ for large $k$, Lemma~\ref{lem_bounded_curv_bounded_dist} gives $C_7^{-1}r_{\comp,k}\leq\rho_1(\phi_k(x_k))\leq C_7 r_{\comp,k}$ for some uniform $C_7 < \infty$ and large $k$.  
This contradicts (\ref{eqn_ratio_blows_up}).
\end{proof}

In the following lemma we show that a region that is bilipschitz close to a cylinder contains a smaller region on which we have closeness to a cylinder in the $C^m$-sense,  provided that the canonical neighborhood assumption holds.
So the smaller region is a neck of arbitrarily high accuracy, as long as the bilipschitz control on the  larger region is strong enough.

\begin{lemma}[Self-improvement of necks] \label{lem_C0_neck_smooth_neck}
If
\[ \delta_\# > 0, \qquad 
\delta \leq \ov\delta (\delta_\#), \qquad 
\eps_{\can} \leq \ov\eps_{\can} (\delta_\#), \]
then the following holds.

Let $(M, g)$ be a Riemannian manifold and $x \in M$ be a point that satisfies the $\eps_{\can}$-canonical neighborhood assumption.
Let $r> 0$ be a constant and $\psi : S^2 \times (-\delta^{-1}, \delta^{-1}) \to M$ be a diffeomorphism onto its image that satisfies $x \in \psi ( S^2 \times \{ 0 \} )$ and
\[ \big\Vert r^{-2} \psi^* g - g^{S^2 \times \R} \big\Vert_{C^0} < \delta, \]
where $g^{S^2 \times \R}$ denotes the round cylindrical metric with $\rho \equiv  1$ and the $C^0$-norm is taken over the domain of $\psi$.

Then $x$ is a center of a $\delta_\#$-neck in $M$ at scale $r$.
\end{lemma}

\begin{proof}
Without loss of generality we may assume that $r = 1$.  

Assume that the lemma was false for some $\delta_\# > 0$.  Then we can find sequences $\delta_k\ra 0$, $\eps_{\can,k}\ra 0$, as well as a sequence $\{(M_k,g_k, x_k)\}$ of pointed Riemannian manifolds and a sequence $\{\psi_k:S^2\times (-\delta^{-1}_k,\delta^{-1}_k)\ra M_k\}$ of diffeomorphisms onto their images such that for all $k$:
\begin{enumerate}[label=(\arabic*)]
\item  $(M_k,g_k)$ satisfies the $\eps_{\can,k}$ canonical neighborhood assumption at $x_k$.
\item $x_k\in\psi_k(S^2\times\{0\})$,
\item  $\|\psi_k^*g_{k}-g^{S^2\times\R}\|_{C^0}<\delta_k\ra 0$.
\item $x_k$ is not a  center of a $\delta_\#$-neck at scale $1$.
\end{enumerate}

Let $\wh{r}_k:=\rho_1(x_k)$.  Then letting 
\begin{align*}
(Z^1_k,g^1_k,z^1_k)&:=\big( S^2\times (-\delta_k^{-1},\delta_k^{-1}),\wh{r}_k^{-2}g^{S^2\times\R},\psi_k^{-1}(x_k) \big)\,,\\ 
(Z^2_k,g^2_k,z^2_k)&:=(M_k,\wh r_k^{-2}g_k,x_k)\,,
\end{align*}
and $\phi_k:=\psi_k$, the assumptions of Lemma~\ref{lem_convergence_of_scales} hold by (3) above and assertion \ref{item_p_x_a_scale_curvature_bound} of Lemma~\ref{lem_properties_kappa_solutions_cna} together with the choice of $\wh r_k$.   Therefore we have  $\rho(x_k)\ra \rho(\psi_k^{-1}(x_k))=1$ as $k\ra\infty$.   
It follows that $(M_k,g_k,x_k)$ is $\eps_{\can,k}$-close at scale tending to 1 to the final time-slice $(\wh M_k,\wh g_k,\wh x_k)$ of a $\kappa_k$-solution with $\rho (\wh x_k) = 1$, as $k\ra\infty$.
Hence  $\diam(\wh M_k,\wh g_k)\ra \infty$. 
Since $\rho(\wh x_k)=1$, it follows that $(\wh M_k,\wh g_k)$ cannot be a round metric for large $k$.  
Hence, by assertions \ref{ass_C.1_a} and \ref{ass_C.1_b} of Lemma~\ref{lem_kappa_solution_properties_appendix}, after passing to a subsequence, the sequence $\{(M_k,g_k,x_k)\}$ converges in the pointed smooth topology to the final time-slice $(M_\infty,g_\infty,x_\infty)$ of some $\kappa$-solution.  
However, by property  (3) above we conclude that  $(M_\infty,g_\infty)$ is isometric as a metric space to $(S^2\times \R,g^{S^2\times\R})$ equipped with the induced length metric.  
So $(M_\infty,g_\infty)$ is isometric as a Riemannian manifold to  $(S^2\times \R,g^{S^2\times\R})$.  
Thus $x_k$ is a  center of a $\delta_\#$-neck at scale $1$ for large $k$, contradicting (4).
\end{proof}

The next lemma gives control on the scale at bounded distance to a neck,  assuming the canonical neighborhood assumption.

\begin{lemma}[Scale bounds near necks] \label{lem_thick_close_to_neck} 
There is a constant $\delta_0 > 0$ such that for every $X < \infty$ there is a constant $Y = Y(X) < \infty$ such that if 
\[ \eps_{\can} \leq \ov{\eps}_{\can} (X), \]
then the following holds.

Let $(M,g)$ be a (possibly incomplete) Riemannian manifold and let $\Sigma \subset M$ be a central 2-sphere of a $\delta_0$-neck at scale $1$ in $M$.
Assume that $M$ satisfies the $\eps_{\can}$-canonical neighborhood assumption at some point in $\Sigma$.

Consider a point $x \in M \setminus \Sigma$  and let $\C$ be the component of $M\setminus\Sigma$ containing $x$.   
If $d (x, \Sigma) \leq X$ and $\diam\C\geq Y$, then $\rho_1 (x) > \frac1{10}$.  
Here the diameter is taken with respect to the distance function of $(M,g)$.
\end{lemma}

The proof uses the geometry of non-negatively curved manifolds to bound neck scales from below.
The argument is a variation on  part of Perelman's proof of compactness of $\kappa$-solutions (see \cite{Perelman:2002um}).

\begin{proof}
Fix $X < \infty$ and  some small constant  $\delta_0 > 0$.
The precise conditions on the smallness of $\delta_0$ will become clear in the course of the proof.

Assume that the statement of the lemma was false (for fixed $X$) and choose sequences $Y_k \to \infty$ and $\eps_{\can,k} \to 0$.
Then we can find counterexamples $(M_k, g_k)$, $\Sigma_k$, $x_k$, $\C_k \subset M_k \setminus \Sigma_k$ such that $(M_k, g_k)$ satisfies the $\eps_{\can, k}$-canonical neighborhood assumption at some point $y_k \in \Sigma_k$, $d(x_k, \Sigma_k) \leq X$, $\diam \C_k \geq Y_k$, but $\rho (x_k) \leq  \frac1{10}$.  

If 
$$\delta_0\leq\ov\delta_0,$$ 
then the injectivity radius at $y_k$ is uniformly bounded from below by a positive constant.
So, after passing to a subsequence, we may assume that:
\begin{itemize}
\item The sequence of pointed Riemannian manifolds $(M_k, g_k, y_k)$ converges to the pointed final time-slice $(M_\infty, g_\infty, y_\infty)$ of some $\kappa$-solution.
\item The 2-spheres $\Sigma_k \subset M_k$ converge to a central 2-sphere $\Sigma_\infty$ of a $2\delta_0$-neck $U_\infty \subset M_\infty$ at scale 1.  
\item The points $x_k$ converge to a point $x_\infty \in M_\infty$ such that $\rho (x_\infty) \leq \frac1{10}$.
\item $d(x_\infty,y_\infty)\geq \frac14 \delta_0^{-1}$, since we may assume that $\rho>\frac12$ on the $2\de_0$-neck $U_\infty$.
\end{itemize}

As $\diam \C_k \geq Y_k \to \infty$,  the $\kappa$-solution $M_\infty$ must be non-compact.   If 
$$\delta_0\leq\ov\delta_0\,,$$ 
then $M_\infty$ cannot be isometric to a quotient of a round cylinder, because  $U_\infty$ is a $2\delta_0$-neck of scale $1$, while $\rho(x_\infty)\leq \frac{1}{10}$.  
Therefore $M_\infty$ is diffeomorphic to $\R^3$, and the $2$-sphere $\Sigma_\infty$ bounds a compact domain, and a non-compact domain $Z$.  
We cannot have $x_\infty\in M_\infty\setminus Z$, since this  would imply that $\diam \C_k \leq 2\diam(M_\infty\setminus Z)$ for large $k$, contradicting the fact that $\diam \C_k \ra\infty$.   So $x_\infty \in Z$.

Let $\gamma\subset M_\infty$ be a minimizing geodesic ray starting from $y_\infty$, and pick $z\in \gamma \cap Z$, to be determined later.   Let $\ol{y_\infty z}$, $\ol{zx_\infty}$, and $\ol{x_\infty y_\infty}$ be minimizing geodesic segments between the corresponding pairs of points.  Assuming 
$$\delta_0\leq\ov\delta_0,$$ 
the segments $\ol{y_\infty z}$, $\ol{y_\infty x_\infty}$ may intersect $\Sigma_\infty$ at most once and are nearly parallel to the $\R$-factor of the neck $U_\infty$.
Therefore both segments are contained in $Z$ apart from the endpoint $y_\infty$, and they form an angle of at most  $\frac{\pi}{4}$ at $y_\infty$.   By Toponogov's theorem, this implies that the comparison angle $\td{\angle}_{y_\infty}(x_\infty,z)$ is at most $\frac{\pi}{4}$.  Provided that $d(z,y_\infty)$ is sufficiently large, we therefore have $\td{\angle}_{x_\infty}(y_\infty,z)>\frac{\pi}{4}$. 

Fix some small $\delta_1 > 0$ whose value we will determine later. If 
\[ \delta_0\leq \ov\delta_0 \]
and $d(z,y_\infty)$ is sufficiently large, then $\rho^{-1}(x_\infty) \min \{d(x_\infty, \lb z), \lb d(x_\infty, \lb y_\infty)\}$ is  large enough that we may apply \cite[Corollary~49.1]{Kleiner:2008fh} to conclude that $x_\infty$ is a center of a $\delta_1$-neck, with central $2$-sphere $\Sigma_{x_\infty} \subset M_\infty$.  If $\delta_1\leq\ov\delta_1$, then the segments $\ol{x_\infty z}$, $\ol{x_\infty,y_\infty}$ intersect $\Sigma_{x_\infty}$ only at $x_\infty$ and are nearly parallel to the $\R$-factor the neck at $x_\infty$.  Since their angle at $x_\infty$ is $>\frac{\pi}{4}$, it follows that $y_\infty$ and $z$ lie in distinct connected components of $M_\infty\setminus \Sigma_{x_\infty}$. 

Let $c_0$ be the diameter of a central $2$-sphere of a round cylinder of scale $1$.  
If $\delta_0\leq\ov\delta_0$, we may choose a point $y_\infty'\in \Sigma_\infty$ such that $d(y_\infty',y_\infty)\geq .99 c_0$.   
Now consider geodesic segments $\ol{y_\infty z}$, $\ol{y_\infty' z}$.  
If $\delta_0\leq\ov\delta_0$, both segments are contained in $Z$, and since $\Sigma_{x_\infty}$ separates $y_\infty$ from $z$, both segments intersect $\Sigma_{x_\infty}$.  
If $\delta_0\leq\ov\delta_0$ then $|d(z,y_\infty)-d(z,y_\infty')|<.01c_0 $, as follows by applying the triangle inequality to points on $\ol{y_\infty z}$, $\ol{y_\infty' z}$ at distance $\frac12\delta_0^{-1}$.  
Therefore, after swapping the labels of $y_\infty$ and $y'_\infty$ if necessary, we may assume without loss of generality that there is a point $y''_\infty\in\ol{y'_\infty z}$ such that $d(z,y''_\infty)=d(z,y_\infty)$ and $d(y_\infty,y''_\infty)>.98c_0$.  
Similarly, if $\delta_1\leq\ov\delta_1$, there are  points $w_\infty\in\ol{y_\infty z}$, $w'_\infty\in \ol{y''_\infty z}$ such that 
\[ d(w_\infty,w'_\infty)<1.01c_0\rho(x_\infty)< \tfrac15 c_0, \] $d(w_\infty,z)=d(w'_\infty,z)$ and one of $w_\infty$, $w'_\infty$ lies on $\Sigma_{x_\infty}$.   
By Toponogov's theorem (monotonicity of comparison angles) we have
$$
\frac{d(y_\infty,y''_\infty)}{d(y_\infty,z)} \leq \frac{d(w_\infty,w'_\infty)}{d(w_\infty,z)}\,.
$$
So if $d(z,y_\infty)$ is sufficiently large, then 
\[  .98 c_0 \leq d(y_\infty, y''_\infty) \leq 2 d(w_\infty,w'_\infty) < \tfrac25 c_0, \]
which is a contradiction.
\end{proof}

\subsection{Promoting time-slice models to spacetime models}
Our next two results show that under appropriate completeness and canonical neighborhood assumptions, if a time-slice of a Ricci flow spacetime is close to a neck or a Bryant soliton, then a parabolic region is also close to a neck or Bryant soliton, respectively.  
The proofs are standard convergence arguments based on a rigidity property of necks and Bryant solitons among $\kappa$-solutions.

\begin{lemma}[Time-slice necks imply spacetime necks] \label{lem_time_slice_neck_implies_space_time_neck}
If
\[ \delta_\# > 0, \qquad 
0 < \delta \leq  \ov{\delta} (\delta_\#), \qquad 0 < \eps_{\can} \leq \ov\eps_{\can} (\delta_\#), \qquad 
0 < r \leq \ov{r},  \]
then the following holds.

Assume that $\M$ is an $(\eps_{\can} r, t_0)$-complete Ricci flow spacetime that satisfies the $\eps_{\can}$-canonical neighborhood assumption at scales $(\eps_{\can} r, 1)$.
Let $a \in [-1, \frac14]$ and consider a time $t \geq  0$ such that $t + ar^2 \in [0,t_0]$.

Assume that $U \subset \M_{t + a r^2}$ is a $\delta$-neck at scale $\sqrt{1-3a}\, r$.
So there is a diffeomorphism
\begin{equation*} 
 \psi_1 : S^2 \times (- \delta^{-1}, \delta^{-1} ) \longrightarrow U 
\end{equation*}
such that
\begin{equation} \label{eq_psi_1_lem_necks}
 \big\Vert  r^{-2} \psi_1^* g_{t + ar^2} - g^{S^2 \times \R}_a \big\Vert_{C^{[\delta^{-1}]}} < \delta. 
\end{equation}
Here $(g^{S^2 \times \R}_t)_{t \in (-\infty, \frac13)}$ denotes the shrinking round cylinder with $\rho (\cdot, 0) = 1$ at time $0$ and the $C^{[\delta^{-1}]}$-norm is taken over the domain of $\psi_1$.

Then there is a product domain $U^* \subset \M_{[t-r^2, t + \frac14 r^2] \cap [0,t_0]}$ and an $r^2$-time-equivariant and $\partial_{\mathfrak{t}}$-preserving diffeomorphism
\[ \psi_2 : S^2 \times \big({ - \delta_\#^{-1}, \delta_\#^{-1} }\big) \times [ t^{*}, t^{**} ] \longrightarrow U^*, \]
with $t + t^* r^2 = \max \{ t - r^2, 0 \}$ and $t + t^{**} r^2 = \min \{ t + \frac14 r^2, t_0 \}$, such that
\[ \psi_2 \big|_{S^2 \times ( - \delta_\#^{-1}, \delta_\#^{-1} ) \times \{ a \} } = \psi_1 \big|_{S^2 \times (- \delta_\#^{-1}, \delta_\#^{-1} )} \]
and
\[ \big\Vert {r^{-2}  \psi_2^* g - g^{S^2 \times \R} } \big\Vert_{C^{[\delta^{-1}_\#]}} < \delta_\#. \]
Here the $C^{[\delta^{-1}_\#]}$-norm is taken over the domain of $\psi_2$.
\end{lemma}

Note that the lemma can be generalized to larger time-intervals.
We have omitted this aspect, as it will not be important for us later.   We also remark that one may prove a more general result to the effect that any parabolic region is close to a parabolic region in a $\kappa$-solution.

\begin{proof}
For the following proof, we may assume that $\ov{r}$ and $\ov\eps_{\can}$ are chosen small enough such that any point $x \in \M$ with $\frac1{10} r \leq \rho(x) \leq 10r$ satisfies the $\eps_{\can}$-canonical neighborhood assumption.

Assuming 
\[\delta\leq \ov\delta\,,\] we have the following bound on the image of $\psi_1$:
\begin{equation} \label{eq_rho_14_4}
 \tfrac14 r \leq \tfrac12 \sqrt{1-3a} \, r < \rho  < 2 \sqrt{1-3a} \, r \leq 4 r . 
\end{equation}

Assume now that the statement of the lemma was false for some fixed $\delta_\# > 0$.
So there are sequences $\eps_{\can, k} \to 0$, $\delta_k \to 0$, $r_k \leq \ov{r}$, $t_k \geq r_k^2$, $a_k \in [-1,\frac14]$, $t_{0,k} \geq 0$, $t^*_k \in [-1,0]$, $t^{**}_k \in [0, \frac14]$ with $t_k + t^*_k r_k^2 = \max \{ t_k - r_k^2, 0 \}$ and $t_k + t^{**}_{k} r_k^2 = \min \{ t_k + \frac14 r_k^2, t_{0,k} \}$, as well as a sequence $\{\M^k\}$ of Ricci flow spacetimes that satisfy the $\eps_{\can, k}$-canonical neighborhood assumption at scales $(\eps_{\can, k} r_k, 1)$ and maps $\psi_{1,k}$ belonging to $\delta_k$-necks at time $t_k$ and scale $\sqrt{1-3a_k} \, r_k$, but for which the conclusion of the lemma fails.
After passing to a subsequence, we may assume that $t^*_\infty := \lim_{k \to \infty} t^*_k$, $t^{**}_\infty := \lim_{k \to \infty} t^{**}_k$ and $a_\infty := \lim_{k \to \infty} a_k$ exist.

Choose $a^*_\infty \in [t^*_\infty, a_\infty]$ and $a^{**}_\infty \in [a_\infty, t^{**}_\infty ]$ minimal and maximal, respectively, such that for any $d > 0$ and any compact interval $[s_1, s_2] \subset (a^*_\infty, a^{**}_\infty)$ the following holds for large $k$ (possibly depending on $d, s_1, s_2$):
For all $x \in \psi_{1,k} ( S^2 \times (-d,d ))$, $t' \in [t_k + s_1 r_k^2, t_k + s_2 r_k^2]$ the point $x(t')$ is defined and we have
\begin{equation} \label{eq_rho_x_t_prime_110_10}
 \tfrac1{10} r_k \leq \rho (x(t')) \leq 10 r_k . 
\end{equation}
Note that by the remark in the beginning of the proof, this implies that $x(t')$ satisfies the canonical $\eps_{\can, k}$-canonical neighborhood assumption.
By (\ref{eq_rho_14_4}) and Lemma~\ref{lem_forward_backward_control} we know that $a^*_\infty < a_\infty$ if $a_\infty >  t^*_\infty$ and $a^{**}_\infty > a_\infty$ if $a_\infty < t^{**}_\infty$.

By the choices of $a^*_{\infty}, a^{**}_\infty$ we can find sequences $d_k \to \infty$, $a^*_k \in [-1, a_k]$ and $a^{**}_k \in [a_k, \frac14]$ with $\lim_{k \to \infty} a^*_k = a^*_\infty$ and $\lim_{k \to \infty} a^{**}_k = a^{**}_\infty$ such that the set
\[ P_k := \big( \psi_{1,k} (S^2 \times ( -d_k, d_k )) \big) \big( [t_k + a^*_k r_k^2, t_k + a^{**}_k r_k^2] \big) \]
is well defined and such that $\frac1{10} r_k \leq \rho \leq 10 r_k$ on $P_k$.
For every $k$ consider the parabolically rescaled flow $(g'_{k, s})_{s \in (a^*_k, a^{**}_k)}$ on $S^2 \times (-d_k,d_k)$ with
\begin{equation} \label{eq_g_prime_k_s}
 g'_{k, s} :=  r_k^{-2} \wh{g}_{t_k  + s r_k^2}, 
\end{equation}
where $\wh{g}_{t_k+s r_k^2}$ denotes the pullback of $g_{t_k+s r_k^2}$ under the composition 
of $\psi_{1,k}$ with the map 
$$
\psi_{1,k} \big( S^2 \times ( -d_k, d_k ) \big)  \lra P_k
$$
that is given by the time $(s-a_k)r_k^2$-flow of $\D_{\t}$.  

By (\ref{eq_rho_x_t_prime_110_10}) and the $\eps_{\can, k}$-canonical neighborhood assumption (see assertion \ref{item_p_x_a_scale_curvature_bound} of Lemma~\ref{lem_properties_kappa_solutions_cna}), we obtain that the curvature of $(g'_{k, s})_{s \in (a^*_k, a^{**}_k)}$, along with its covariant derivatives, is uniformly bounded.
Together with (\ref{eq_psi_1_lem_necks}), these bounds imply uniform $C^m$-bounds on the tensor fields $(g'_{k, s})$ themselves.
So, by passing to a subsequence, we obtain that the $(g'_{k, s})$ converge to a Ricci flow $(g'_{\infty, s})_{s \in (a^*_\infty, a^{**}_\infty)}$ on $S^2 \times \R$, which extends smoothly to the time-interval $[a^*_\infty, a^{**}_\infty]$.

The $\eps_{\can,k}$-canonical neighborhood assumption implies that all time-slices of this limit are final time-slices of $\kappa$-solutions.
By (\ref{eq_psi_1_lem_necks}) we know that $g'_{\infty, a_\infty} = g^{S^2 \times \R}_{a_\infty}$. 
Since $S^2 \times \R$ has two ends, $g'_{\infty, s}$ splits off an $\R$ factor for all $s \in (a^*_\infty, s^*_\infty)$ and must therefore be homothetic to a round cylinder.
It follows that  $g'_{\infty, s} = g^{S^2 \times \R}_{s}$ for all $s \in [a^*_\infty, a^{**}_\infty ]$.
Since this limit is unique, we obtain that the $(g'_{k,s})$ converge to $(g'_{\infty, s})$ even without passing to a subsequence.

As $\frac12 \leq \rho \leq 2$ on $(S^2 \times \R) \times (a^*_\infty, a^{**}_\infty)$, we obtain that for any $d  > 0$ and $[s_1,s_2] \in (a^*_\infty, a^{**}_\infty)$ we have $\frac14 r_k \leq \rho  \leq 4 r_k$ on $(\psi_{1,k} (S^2 \times (- d, d)))([s_1,s_2])$ for large $k$.
So by Lemma~\ref{lem_forward_backward_control} and the minimal and maximal choices of $a^*_\infty, a^{**}_\infty$, we have $a^*_\infty = t^*_\infty$ and $a^{**}_\infty = t^{**}_\infty$.
Moreover, after adjusting the sequence $d_k \to \infty$, we may assume that $a^*_k = t^*_k$ and $a^{**}_k = t^{**}_k$ for large $k$.

For large $k$ we now define $\psi_{2,k}$ by extending $\psi_{1,k}$ restricted to $S^2 \times (-\delta^{-1}_\#, \delta^{-1}_\#)$ forward and backward using the flow of $r^2_k \partial_{\mathfrak{t}}$.
Then we have $r^{-2}_k \psi^*_{2,k} g_k = g'_{k, s}$ on $(S^2 \times (-\delta^{-1}_\#, \delta^{-1}_\#)) \times [ t_k^*,t^{**}_k]$.
So it suffices to show that $g'_{k,s}$ converges to $g^{S^2 \times \R}$ on $(S^2 \times (-\delta^{-1}_\#, \delta^{-1}_\#)) \times [t_k^*,t^{**}_k]$ uniformly in the $C^{[\delta_\#^{-1}]}$-sense.
To see this, note that $g'_{k,s}$ from (\ref{eq_g_prime_k_s}) is uniformly bounded on $(S^2 \times (-\delta^{-1}_\#, \delta^{-1}_\#)) \times [-1,\frac14]$ in every $C^m$-norm and that we have uniform convergence of $g'_{k,s}$ to $g^{S^2 \times \R}$ on every subset of the form $(S^2 \times (-\delta^{-1}_\#, \delta^{-1}_\#)) \times [s_1, s_2]$ for $[s_1, s_2 ] \subset (t^*_\infty, t^{**}_\infty )$, in every $C^m$-norm.
\end{proof}

\bigskip

For notation and facts about the Bryant soliton, see Subsection~\ref{subsec_basics_Bryant} and Appendix~\ref{appx_Bryant_properties}.  In the following result, it is important that $\rho\geq 1$ on the normalized Bryant soliton.

\begin{lemma}[Propagating  Bryant-like geometry] 
\label{lem_promoting_Bryant}
If
\begin{multline*}
 \delta_\# > 0, \qquad
T < \infty, \qquad
\delta \leq \ov\delta (\delta_\#, T), \qquad 
\eps_{\can} \leq \ov\eps_{\can} (\delta_\#, T), \qquad
r \leq \ov{r}, 
\end{multline*}
then the following holds.

Assume that $\M$ is an $(\eps_{\can} r, t_0)$-complete Ricci flow spacetime that satisfies the $\eps_{\can}$-canonical neighborhood assumption at scales $(\eps_{\can} r, 1)$.
Let  $t \in [0,t_0]$ and consider a diffeomorphism onto its image
\[ \psi_1 : M_{\Bry} (\delta^{-1}) \times \{ 0 \} \to \M_t \]
with the property that
\begin{equation} \label{eq_psi_Bry_assumption_propagate}
 \big\Vert r^{-2} \psi_1^* g_t - g_{\Bry} \big\Vert_{C^{[\delta^{-1}]} ( M_{\Bry} (\delta^{-1}) \times \{ 0 \})} < \delta. 
\end{equation}
Then  there is an $r^2$-time equivariant and $\partial_{\t}$-preserving diffeomorphism onto its image
\[ \psi_2 : M_{\Bry} (\delta_\#^{-1}) \times [  t^*, t^{**}   ]  \longrightarrow \M_{[t - T r^2, t + T r^2] \cap [0,t_0]}, \]
where $t^* \leq 0 \leq t^{**}$ are chosen such that $t + t^* r^2 = \max \{ t - T r^2, 0 \}$ and $t + t^{**} r^2 = \min \{ t + T r^2, t_0 \}$.
The map $\psi_2$ has the property that $\psi_2 = \psi_1$ on $M_{\Bry} (\delta_\#^{ - 1}) \times \{ 0 \}$ and 
\[ \big\Vert r^{-2} \psi_2^* g - g_{\Bry} \big\Vert_{C^{[\delta_\#^{-1}]}} < \delta_\#, \]
where the norm is taken over the domain of $\psi_2$.
\end{lemma}

\begin{proof}
The proof is similar to the proof of Lemma~\ref{lem_time_slice_neck_implies_space_time_neck}. 

In the following, we may assume that $\ov{r}$ and $\ov\eps_{\can}$ are chosen  small enough such that any point with $x \in \M$ with $\frac1{10} r \leq \rho(x) \leq 10 r$ satisfies the $\eps_{\can}$-canonical neighborhood assumption.

Assuming
\[\delta\leq\ov\delta\]
 we have
\[ \tfrac14 r \leq \rho \quad \text{on} \quad \Image \psi_1, \qquad \tfrac14 r \leq \rho (\psi_1 (x_{\Bry})) \leq 4 r. \]

Assume now that the statement of the lemma was false for some fixed $\delta_\# > 0$, $T<\infty$.  
So there are a sequences $\{\eps_{\can,k}\}$, $\{\delta_k\}$, $\{ t_{0,k} \}$, $\{r_k\}$, $\{ t_k \}$, $\{ t^*_k \}$, $\{ t^{**}_k \}$ such that $\eps_{\can, k} \to 0$, $\delta_k \to 0$,  as well as a sequence $\{\M^k\}$ of $(\eps_{\can,k} r_k, t_{0,k} )$-complete Ricci flow spacetimes that satisfy the $\eps_{\can}$-canonical neighborhood assumption at scales $(\eps_{\can, k} r_k, 1)$ and a sequence of maps $\{\psi_{1,k}\}$ satisfying the hypotheses of the lemma, but for which the conclusion of the lemma fails for all $k$.
By passing to a subsequence, we may assume that $t^*_\infty := \lim_{k \to \infty} t^*_k$ and $t^{**}_\infty := \lim_{k \to \infty} t^{**}_k$ exist.

Choose $a^*_\infty \in [t^*_\infty, 0]$, $a^{**}_\infty \in [0,t^{**}_\infty]$ minimal and maximal, respectively, such that for any $d > 0$ and any compact interval $[s_1,s_2] \subset (a^*_\infty, a^{**}_\infty)$ the following holds for large $k$:
For every $x \in \psi_{1,k} (M_{\Bry} (d))$ and $t' \in [t_k + s_1 r_k^2, t_k + s_2 r_k^2]$ the point $x(t')$ is well defined and we have $\rho (x(t')) \geq \frac1{10} r_k$ and $ \rho ( (\psi_{1,k} (x_{\Bry}))(t')) \leq 10 r_k$.
Note that $a^*_\infty < 0$ if $t^*_\infty < 0$ and $a^{**}_\infty > 0$ if $t^{**}_\infty > 0$ due to Lemma~\ref{lem_forward_backward_control}.

As in the proof of Lemma~\ref{lem_time_slice_neck_implies_space_time_neck}, we can now find sequences $d_k \to \infty$ and $a^*_k \in [ t^*_k, 0]$, $a^{**}_k \in [0, t^{**}_k]$ with $\lim_{k \to \infty} a^*_k = a^*_\infty$ and $\lim_{k \to \infty} a^{**}_k = a^{**}_\infty$, such that the product domains
\[ P_k := \big( \psi_{1,k} \big( M_{\Bry} (d_k) \big) \big)  \big( [t_k + a^*_k r_k^2, t_k + a^{**}_k r_k^2] \big) \]
are well defined and such that
\begin{multline*}
 \rho \geq \tfrac1{10} r_k \quad \text{on} \quad P_k, \qquad \\
  \rho \big( \big(\psi_{1,k} (x_{\Bry}) \big)(t') \big) \leq 10 r_k \quad \text{for all} \quad t' \in [t_k + a^*_k r_k^2, t_k + a^{**}_k r_k^2] .
\end{multline*}
So $(\psi_{1,k} (x_{\Bry}))(t')$ satisfies the $\eps_{\can, k}$-canonical neighborhood assumption for all $t' \in [t_k + a^*_k r_k^2, t_k + a^{**}_k r_k^2]$.

For every $k$ consider the parabolically rescaled flow $(g'_{k, s})_{s \in [a^*_k, a^{**}_k]}$ on $M_{\Bry}(d_k )$ with  $g'_{k, s} :=  r_k^{-2} \wh{g}_{t_k  +s  r_k^2 }$,
where $\wh{g}_{t_k+r_k^2s}$ denotes the pullback of $g_{t_k+s r_k^2}$ under the composition 
of $\psi_{1,k}$ with the map 
$$
\M^k_{t_k}\supset\psi_{1,k} \big( M_{\Bry}(d_k ) \times\{0\} \big)\lra P_k,
$$
given by the time $r_k^2s$-flow of $\D_{\t}$.  
By the  $\eps_{\can, k}$-canonical neighborhood assumption at $(\psi_{1,k} (x_{\Bry}))(t')$ (see assertion \ref{item_p_x_a_scale_curvature_bound} of Lemma~\ref{lem_properties_kappa_solutions_cna}) and a distance distortion estimate on $P_k$, we obtain that the curvature of this flow, along with its derivatives, is uniformly bounded by a constant that may only depend on the spatial direction.
Together with (\ref{eq_psi_Bry_assumption_propagate}), these bounds imply uniform local $C^m$-bounds on the tensor fields $(g'_{k, s})$ themselves.

So, by passing to a subsequence, we obtain that $(g'_{k, s})$ converges to a Ricci flow $(g'_{\infty, s})_{s \in (a^*_\infty, a^{**}_\infty)}$ on $M_{\Bry}$ with uniformly bounded curvature, which extends smoothly to the time-interval $[a^*_\infty, a^{**}_\infty]$.
By the $\eps_{\can,k}$-canonical neighborhood assumption at $(\psi_{1,k} (x_{\Bry}))(t')$ and the compactness of $\kappa$-solutions (see assertions \ref{ass_C.1_a} and \ref{ass_C.1_b} of Lemma~\ref{lem_kappa_solution_properties_appendix}), we find that all time-slices of this limit are final time-slices of $\kappa$-solutions.
By (\ref{eq_psi_Bry_assumption_propagate}), we furthermore know that $g'_{\infty, 0} = g_{\Bry, 0}$.

We now claim that $g'_{\infty, s} = g_{\Bry, s}$ for all $s \in [a^*_\infty, a^{**}_\infty]$.
For $s \geq 0$, this follows from the uniqueness of Ricci flows with uniformly bounded curvature.
To verify this in the case $s < 0$, recall that there is a $\kappa$-solution $(g''_{s})_{s \in (-\infty, 0]}$ on $M_{\Bry}$ such that $g''_{0} = g'_{a^*_\infty}$.
Set $g'''_s := g'_{\infty, s }$ if $a^*_\infty \leq s \leq 0$ and $g'''_s := g^{\prime \prime}_{s - a^*_\infty}$ if $s < a^*_\infty$.
Then $(g'''_s)_{s \in (-\infty, 0]}$ is a smooth $\kappa$-solution (possibly after adjusting $\kappa$).
Since $\D_tR_{g'''}(x_{\Bry}, 0)=0$, it follows from Proposition~\ref{prop_dtR_0_Bryant} that $(M_{\Bry}, (g'''_s)_{s \in (-\infty, 0]}, x_{\Bry})$ is isometric to $(M_{\Bry}, (g_{\Bry, t})_{t \in (-\infty, 0]}, x_{\Bry})$.
Thus $g'_{\infty, s} = g_{\Bry, s}$ for all $s \in [a^*_\infty, a^{**}_\infty]$.
As in the proof of Lemma~\ref{lem_time_slice_neck_implies_space_time_neck}, the uniqueness of the limit implies that the $(g'_{k,s})$ converge to $(g'_{\infty, s})$ even without passing to a subsequence.

So $(g'_{\infty, s})_{s \in [a^*_\infty, a^{**}_\infty]}$ satisfies $\rho \geq 1$ everywhere and $\rho (x_{\Bry},s) = 1$ for all $s \in [a^*_\infty, a^{**}_\infty]$.
Therefore, by the minimal and maximal choices of $a^*_\infty$ and $a^{**}_\infty$ and Lemma~\ref{lem_forward_backward_control}, we obtain that $a^*_\infty = t^*_\infty$ and $a^{**}_\infty = t^{**}_\infty$.
Moreover, after possibly adjusting the sequence $d_k \to \infty$, we may assume that $a^*_k = t^*_k$ and $a^{**}_k = t^{**}_k$ for large $k$.
The claim now follows as in the proof of Lemma~\ref{lem_time_slice_neck_implies_space_time_neck}.
\end{proof}

\subsection{Identifying approximate Bryant structure}
In the next result, we exploit the rigidity theorems of Hamilton and Brendle to show that a large region must be well approximated by a Bryant soliton if the scale is nearly increasing at a point. 

\begin{lemma}\label{lem_d_rho_squared_small_almost_bryant}
If 
$$
\delta_\#>0\,,\qquad \delta\leq\ov\delta(\delta_\#)\,,\qquad \eps_{\can}\leq\ov\eps_{\can}(\delta_\#)\,,
$$
then the following holds.

If $\M$ is a Ricci flow spacetime satisfying the $\eps_{\can}$-canonical neighborhood assumption at $x\in \M$, and $\D_{\t} \rho^2(x)\geq -\delta$, then $(\M_t,x)$ is $\delta_\#$-close to $(M_{\Bry},g_{\Bry},x_{\Bry})$ at any scale $a \in ((1-\delta) \rho(x), (1+\delta) \rho(x))$.
\end{lemma}

Note that $\D_{\t} \rho^2$ is scaling invariant.

\begin{proof}
Suppose the lemma was  false for some $\delta_\#>0$.  Then there a sequence $\{\M^k\}$ of Ricci flow spacetimes satisfying the $\frac{1}{k}$-canonical neighborhood assumption at $x_k\in \M^k_{t_k}$, such that $\D_{\t} \rho^2(x_k)\geq -\frac{1}{k}$, but $(\M^k_{t_k},x_k)$ is not $\delta_\#$-close to $(M_{\Bry},g_{\Bry},x_{\Bry})$ at some scale $a_k \in ( (1-\frac1k) \rho(x_k), (1+\frac1k) \rho(x_k))$.

By the definition of the canonical neighborhood assumption, for every $k$ there is a pointed $\kappa_k$-solution $(\ov{M}_k, (\ov{g}_{k,t})_{t \in (-\infty, 0]}, \ov{x}_k)$  with $\rho (\ov x_k)=1$ and a diffeomorphism onto its image  
$$
\psi_k: B(\ov{x}_k, 0, k)\lra \M^k_{t_k}
$$
with $\psi_k(\ov{x}_k)=x_k$ such that for some $\lambda_k > 0$ with $\lambda_k / \rho(x_k) \to 1$ we have
\[ \big\Vert \lambda_k^{-2} \psi_k^* g_k - \ov{g}_k \big\Vert_{C^{k}(B(\ov x_k,k))} < k^{-1}. \]
So we also have
\[ \big\Vert a_k^{-2} \psi_k^* g_k - \ov{g}_k \big\Vert_{C^{k}(B(\ov x_k,k))}  \to 0. \]
Hence $\liminf_{k \to \infty} \D_t\rho^2( \ov{x}_k ) \geq 0$.  
Therefore $(M_k, (\ov{g}_{k,t})_{t \in (-\infty, 0]})$  cannot be a shrinking round spherical space form for large $k$.
So by assertions \ref{ass_C.1_a}, \ref{ass_C.1_b} and \ref{ass_C.1_e} of Lemma~\ref{lem_kappa_solution_properties_appendix}, after passing to a subsequence, $(\ov{M}_k, (\ov{g}_{k,t})_{t \in (-\infty, 0]}, \ov{x}_k)$  converges in the pointed smooth topology to a $\kappa$-solution $(\ov{M}_\infty, (\ov{g}_{\infty,t})_{t \in (-\infty, 0]}, \ov{x}_\infty)$  with $\D_t\rho^2(\ov x_\infty)=0$.  
By Proposition~\ref{prop_dtR_0_Bryant} it follows that $(\ov{M}_\infty, \lb (\ov{g}_{\infty,t})_{t \in (-\infty, 0]}, \lb \ov{x}_\infty)$  is isometric to a Bryant soliton.   
This is a contradiction.
\end{proof}

By combining Lemma~\ref{lem_d_rho_squared_small_almost_bryant} with Lemma~\ref{lem_promoting_Bryant}, we can deduce closeness to a Bryant soliton on a \emph{parabolic} region.

\begin{lemma}[Nearly increasing scale  implies Bryant-like geometry]
\label{lem_bryant_propagate}
If
\begin{gather*}
 \alpha, \delta > 0, \qquad
1 \leq J < \infty, \qquad
\beta \leq \ov\beta (\alpha, \delta, J),\\ \qquad 
\eps_{\can} \leq \ov\eps_{\can} (\alpha, \delta, J), \qquad  r\leq\ov r(\alpha)
\end{gather*}
then the following holds.

Let $0<r\leq 1$.
Assume that $\M$ is an $(\eps_{\can} r, t_0)$-complete Ricci flow spacetime that satisfies the $\eps_{\can}$-canonical neighborhood assumption at scales $(\eps_{\can} r, 1)$.

Let $t \in [ J  r^2 , t_0]$ and $x \in \M_t$.
Assume that $x$ survives until time $t-r^2$ and that 
\begin{align*}
\alpha r\leq \rho (x) \leq \alpha^{-1} r    \\
 \rho (x(t-r^2)) \leq \rho(x) + \beta r. 
\end{align*}

Let $a \in [\rho (x(t-r^2)) , \rho(x) + \beta r]$.
Then $(\M_{t'}, x(t'))$ is $\delta$-close to $(M_{\Bry}, g_{\Bry}, x_{\Bry})$ at scale $a$ for all $t' \in [t-r^2, t]$. 
Furthermore, there is an $a^2$-time-equivariant and $\partial_{\mathfrak{t}}$-preserving diffeomorphism onto its image
\[ \psi : \ov{M_{\Bry} (\delta^{-1} )} \times \big[{- J \cdot (a r^{-1})^{-2}, 0} \big] \longrightarrow \M \]
such that $\psi (x_{\Bry}, 0) = x$ and
\[ \big\Vert a^{-2} \psi^* g - g_{\Bry} \big\Vert_{C^{[\delta^{-1}]}} < \delta, \]
where the norm is taken over the domain of $\psi$.
\end{lemma}

\begin{proof}
Let $2 > L > 1$, $\delta' > 0$ be constants whose values will be determined in the course of this proof.
By Lemma~\ref{lem_lambda_thick_survives_backward}, and assuming
\[ \eps_{\can} \leq \ov\eps_{\can} (\alpha, L), \qquad r \leq \ov{r} (\alpha), \]
we obtain that for all $t' \in [t-r^2, t]$ we have
\[ \tfrac12 \alpha r \leq L^{-1} \rho(x)  \leq \rho (x(t')) \leq L \rho (x(t-r^2)). \] 
If moreover
\[ \beta \leq \ov\beta (\alpha, L), \]
then
\[ \rho(x(t-r^2)) \leq \rho(x) + \beta r \leq \rho(x) + (L-1) \alpha r \leq L \rho(x). \]
So $L^{-1} \rho(x) \leq \rho (x(t')) \leq L^2 \rho(x)$ and $a \in [L^{-1} \rho(x(t')), L^2 \rho(x(t'))]$ for all $t' \in [t - r^2 , t]$.
We also obtain that $x(t')$ satisfies the $\eps_{\can}$-canonical neighborhood assumption for all $t' \in [t-r^2, t]$, assuming
\[ \eps_{\can} \leq \ov\eps_{\can}, \qquad r \leq \ov{r}. \]
By the Mean Value Theorem, we can find a $t' \in [t-r^2, t]$ at which
\[ \partial_{\t} \rho^2 (x(t')) 
\geq - 2 \rho(x(t')) \cdot \frac{(L^2 - L^{-1}) \rho(x)}{r^2}
\geq - 2 \alpha^{-2} (L^2 - L^{-1}) L. \]
Therefore, if
\[ L \leq 1 + \ov{L}(\delta'), \qquad \eps_{\can} \leq \ov\eps_{\can} (\delta'),  \]
then Lemma~\ref{lem_d_rho_squared_small_almost_bryant} implies that $(\M_{t'}, x)$ is $\delta'$-close to $(M_{\Bry}, g_{\Bry}, x_{\Bry})$ at scale $a$.
Assuming
\[ \delta' \leq \ov{\delta}' (\alpha, \delta, J), \qquad
\eps_{\can} \leq \ov\eps_{\can} (\alpha, \delta, J), \qquad
r \leq \ov{r} (\alpha), \]
the claim now follows from Lemma~\ref{lem_promoting_Bryant}.
\end{proof}

\subsection{The geometry of comparison domains}
The results in this subsection analyze the structure of comparison domains (and related subsets) of spacetimes that satisfy completeness and canonical  neighborhood conditions, as well as some of the a priori assumptions \ref{item_time_step_r_comp_1}--\ref{item_eta_less_than_eta_lin_13},  as introduced in Section~\ref{sec_apas}.  
  
The first two  results --- the Bryant Slice Lemma \ref{lem_Bryant_slice} and the Bryant Slab Lemma \ref{lem_bryant_slab} --- describe the structure of comparison domains in approximate Bryant regions.  
These results  are helpful  in showing that neck-like boundaries of comparison domains and cuts are far apart (Lemma \ref{lem_boundary_far_from_cut}), and in facilitating the construction of the comparison domain in Section~\ref{sec_inductive_step_extension_comparison_domain}.

The Bryant Slice Lemma characterizes how a domain $X$ in a time-slice $\M_t$ that is   bounded by a central 2-sphere of a sufficiently precise neck intersects a domain $W \subset \M_t$ that is geometrically close to a Bryant soliton.
The domain $X$ will later be equal to either backward time-slice $\N_{t_j-}$ of a comparison domain or the domain $\Omega$ from Section~\ref{sec_inductive_step_extension_comparison_domain}.

\begin{lemma}[Bryant Slice Lemma] \label{lem_Bryant_slice}
If
\[ 
\delta_{\nn} \leq \ov\delta_{\nn}, \qquad
0 < \lambda < 1, \qquad
\Lambda \geq \underline{\Lambda}, \qquad
\delta \leq \ov\delta (\lambda, \Lambda ), \]
then the following holds for some $D_0 = D_0 (\lambda) < \infty$.

Consider a Ricci flow spacetime $\M$ and let $r > 0$ and $t \geq 0$.
Consider a  subset $X \subset \M_t$ such that the following holds:
\begin{enumerate}[label=(\roman*)]
\item \label{con_8.42_i} $X$ is a closed subset and is a domain with smooth boundary.
\item \label{con_8.42_ii} The boundary components of $X$ are central 2-spheres of $\delta_{\nn}$-necks at scale $r$.
\item \label{con_8.42_iii} $X$ contains all $\Lambda r$-thick points of $\M_t$.
\item \label{con_8.42_iv} Every component of $X$ contains a $\Lambda r$-thick point.
\end{enumerate}

Consider the image $W$ of a diffeomorphism
\[ \psi : W^* := \ov{M_{\Bry} (d)} \longrightarrow W \subset  \M_t, \]
such that $d \geq \delta^{-1}$ and
\[ \big\Vert (10 \lambda r)^{-2} \psi^* g_t - g_{\Bry} \big\Vert_{C^{[\delta^{-1}]} (W^*)} < \delta. \]
Then $\psi (x_{\Bry})$ is $11\lambda r$-thin.
Moreover, if $\C := W \setminus \Int X \neq \emptyset$, then
\begin{enumerate}[label=(\alph*)]
\item \label{ass_8.42_a} $\C$ is a 3-disk containing $\psi(x_{\Bry})$.
\item \label{ass_8.42_b} $\C$ is a component of $\M_t \setminus \Int X$, and $\partial \C \subset \partial X$.
\item \label{ass_8.42_c} $\C$ is $9\lambda r$-thick and $1.1r$-thin.
\item \label{ass_8.42_d} $\C \subset \psi (M_{\Bry} (D_0 (\lambda) ))\subset\Int W$.
\end{enumerate}
\end{lemma}

\begin{proof}
Assuming  
\[ \de \leq \ov\de(\lambda,\Lambda), \]
it follows from the definition of $W^*$ that $\D W$ is $\Lambda r$-thick, $W$ is $9\lambda r$-thick, and the image of the tip $\psi(x_{\Bry})$ is $11\lambda r$-thin.
The fact that $\partial W$ is $\Lambda r$-thick and assumption \ref{con_8.42_iii} imply that $\partial W \subset \Int X$.

Consider a boundary component $\Sigma \subset \partial X$ with $\Sigma\cap W \neq\emptyset$.
Let $U_\Sigma \subset \M_t$ be a $\delta_{\nn}$-neck at scale $r$ that has $\Sigma$ as a central 2-sphere.
If
\[ \de_{\nn} \leq \ov\de_{\nn}, \]
then we have $.99r < \rho < 1.01r$ on $U_\Sigma$.
Assuming 
\[ 
 \Lambda \geq 10, \]
we find that, then $U_{\Sigma}\cap \D W=\emptyset$ and hence $U_{\Sigma}\subset W$.

Next, if
\[ \delta \leq \ov\delta (\lambda), \]
then $ .98(10\lambda)^{-1} < \rho < 1.02(10\lambda)^{-1}$ on $\psi^{-1}(U_{\Sigma})$. 
Moreover, the 2-sphere $\Sigma^*:=\psi^{-1}(\Sigma)$ is isotopic within the set $\{.98(10\lambda)^{-1} \leq \rho \leq 1.02(10\lambda)^{-1}\} \subset W^*$ to the 2-sphere $\widehat{\Sigma}^*=\{ \rho=1.02(10\lambda)^{-1}\}$ in $W^*$.  

By Alexander's theorem, $\Sigma^*$ bounds a 3-disk $V_\Sigma^* \subset W^*$.
By the previous paragraph, we have $V_\Sigma^* \subset \wh{V}_\Sigma^* := \{ \rho \leq 1.02 (10\lambda)^{-1} \}$.
Thus, if
\[ \delta \leq \ov\delta (\lambda), \]
then $V_\Sigma := \psi (V^*_\Sigma) \subset \psi (\wh{V}^*_\Sigma)$ is $1.1 r$-thin and contains $\psi ( x_{\Bry})$.

Lastly, suppose that $\Sigma_1$, $\Sigma_2$ are distinct components of $\D X$ that intersect $W$.
Let $V_{\Sigma_1}$, $V_{\Sigma_2}$ be the corresponding 3-disk components, as defined in the discussion above.
Since $\psi (x_{\Bry}) \in V_{\Sigma_1}\cap V_{\Sigma_2}$, we may assume (after reindexing) that $V_{\Sigma_1}\subset V_{\Sigma_2}$.

If $X_0$ is the component of $X$ containing $\Sigma_1=\D V_{\Sigma_1}$, then it must be contained in $V_{\Sigma_2}$ since every arc leaving $V_{\Sigma_2}$ must intersect $\partial X \supset \partial V_{\Sigma_2}$.   
Thus $X_0$ is $1.1r$-thin, contradicting assumption \ref{con_8.42_iv} for
\[ \Lambda > 1.1. \]
Thus $W$ intersects at most one component of $\D X$.

Now suppose $\C:=W\setminus \Int X$ is nonempty.  Since $\D W \subset \Int X$, we have $\C\neq W$.  By the discussion above we see that $\D X\cap W$ consists of a single $2$-sphere component $\Sigma$, where $\Sigma$ bounds a $3$-disk $V_\Sigma$ which contains $\psi(x_{\Bry})$.  Thus $\C=V_\Sigma$, and assertions \ref{ass_8.42_a}--\ref{ass_8.42_c} now follow immediately.
For assertion \ref{ass_8.42_d} recall that $\C \subset \psi (\wh{V}^*_\Sigma) = \psi (\{ \rho \leq 1.02 (10\lambda)^{-1} \})$, which can easily be converted into the desired bound.
\end{proof}

\bigskip
Next we consider a parabolic region $W \subset \M_{[t_0,t_1]}$ inside a \emph{time-slab} of a Ricci flow spacetime that is geometrically close to an evolving Bryant soliton.
Moreover, we  consider two domains $X_0, X_1$ that are contained in the initial and final time-slices $\M_{t_0}, \M_{t_1}$ of this time-slab, respectively, and whose boundary components are central 2-spheres of sufficiently precise necks.
The Bryant Slab Lemma describes the complements of these domains in $W$ and characterizes their relative position.

\begin{lemma}[Bryant Slab Lemma]
\label{lem_bryant_slab}
If 
\[ 
\delta_{\nn} \leq \ov\delta_{\nn}, \qquad
0 < \lambda < 1, \qquad
\Lambda \geq \underline{\Lambda}, \qquad
\delta \leq \ov\delta (\lambda, \Lambda ), \]
then the following holds.

Consider a Ricci flow spacetime $\M$ and let $r > 0$ and $t_0 \geq 0$.
Set $t_1 := t_0 + r^2$.
For $i = 0,1$ let $X_i\subset \M_{t_i}$ be a closed subset that is a domain with boundary, satisfying conditions \ref{con_8.42_i}--\ref{con_8.42_iv} from Lemma \ref{lem_Bryant_slice}, and in addition:
\begin{enumerate}[label=(\roman*),start=5]
\item \label{con_8.42_v} $X_1 (t)$ is defined for all $t\in [t_{0},t_1]$, and $\partial X_1 (t_0)\subset\Int X_{0}$. \end{enumerate}
  
Consider a ``$\de$-good Bryant slab'' in $\M_{[t_{0},t_1]}$, i.e. the image $W$ of a map
\[ \psi : W^*=\ol{M_{\Bry}(d)}\times [-(10\la)^{-2},0] \longrightarrow \M_{[t_0,t_1]} \]
where $d \geq \de^{-1}$ and $\psi$  is a $(10 \lambda r)^2$-time equivariant and $\partial_{\mathfrak{t}}$-preserving diffeomorphism onto its image   and
\begin{equation*} \label{eq_Bryant_closeness_1}
 \big\Vert{ (10 \lambda r)^{-2} \psi^* g - g_{\Bry} }\big\Vert_{C^{[\delta^{-1}]} (W^*)} < \delta.
\end{equation*}

Set $\C_i:=W_{t_i}\setminus\Int X_i \subset \M_{t_i}$ for $i = 0,1$.
Then
\begin{enumerate}[itemsep=2pt,label=(\alph*), leftmargin=* ]
\item \label{ass_8.43_a} $\C_i(t)$ is well-defined and $9\lambda r$-thick.
\item \label{ass_8.43_b} If $\C_1\neq\emptyset$, then $\C_0 \subset \C_1 (t_0)$  and $\C_0 =\C_1 (t_0)\setminus\Int X_0$.
\end{enumerate}
\end{lemma}

\begin{proof}
Assuming
\[ 0 < \lambda < 1, \qquad
\Lambda \geq \underline{\Lambda}, \qquad
\delta_{\nn} \leq \ov\delta_{\nn}, \qquad
\delta \leq \ov\delta (\lambda, \Lambda )\,, \]
Lemma~\ref{lem_Bryant_slice} may be applied in the $t_i$-time-slice for $i = 0,1$ and  $W$ is $9\lambda r$-thick.  Since by definition $\C_i\subset W_{t_i}$,  assertion \ref{ass_8.43_a}   now follows from the fact that $W$ is a product domain.

We now verify assertion \ref{ass_8.43_b}.  
If $\C_1 \neq\emptyset$, then $\psi(x_{\Bry},0) \in \C_1$, by Lemma~\ref{lem_Bryant_slice}. 
So since $W$ is a product domain, we get that $\psi(x_{\Bry}, \lb -(10\lambda)^{-2}) \in \C_1(t_0)$.
If $\C_0 \neq\emptyset$, then both $\C_1 (t_0)$ and $\C_0$ are $3$-disks in $W_{t_0}$ containing $\psi(x_{\Bry},-(10\lambda)^{-2})$.
By assumption \ref{con_8.42_v} we have $\partial \C_1 (t_0) \subset \D X_1 (t_0) \subset \Int X_0$ and hence $\D\C_1 (t_0)$ is disjoint from $\C_0 \subset\M_{t_0}\setminus \Int X_0$.
Therefore $\C_0 \subset \C_1 (t_0)$.  This gives $\C_0=(W_{t_0}\setminus\Int X_0)\cap \C_1(t_0)=\C_1(t_0)\setminus\Int X_0$.
\end{proof}

We now show that a parabolic region $P(x,a,-b)$ lies in a  comparison domain $(\N, \lb \{ \N^j \}_{j=1}^J, \lb \{ t_j \}_{j=0}^J)$, provided the ball $B(x,a,b)$ lies in $\N$ and $P(x,a,-b)$ ``avoids the cuts''; see below for further discussion.

\begin{lemma}[Parabolic neighborhoods inside the comparison domain] \label{lem_parabolic_domain_in_N}
Consider a Ricci flow spacetime $\M$, a comparison domain $(\N, \lb \{ \N^j \}_{j=1}^J, \lb \{ t_j \}_{j=0}^J)$ in $\M$ and a set $\Cut = \Cut^1 \cup \ldots \cup \Cut^{J-1}$, where $\Cut^j$ is a collection of pairwise disjoint $3$-disks inside $\N_{t_j+}$ in such a way that each extension cap of $\N$ is contained in some $\DD \in \Cut$.

Let $x \in \M$, $a,b > 0$ and assume that $B(x, a) \subset \N$ and that $P(x, a,-b) \cap \DD = \emptyset$ for all $\DD \in \Cut$.
Then
\begin{equation} \label{eq_P_x_a_b_in_N_D}
 P(x, a, -b) \subset \N \setminus \cup_{\DD \in \Cut} \DD. 
\end{equation}
\end{lemma}

As the notation suggests, the set $\Cut$ will later denote the set of cuts of a comparison, according to Definition~\ref{def_comparison}.
However, we will  use Lemma~\ref{lem_parabolic_domain_in_N} at a stage of the proof when this comparison will not have been fully constructed.
More specifically, we will later consider a comparison domain defined over the time-interval $[0,t_{J+1}]$ and have to take $\Cut$ to be the union $\Cut \cup \Cut^{J}$.
Here $\Cut$ is the set of cuts of a comparison  that is only defined on the time-interval $[0, t_J]$ and $\Cut^{J}$ is a set of freshly constructed cuts at time $t_J$, which will not be part of a comparison yet.
 For this reason we have phrased Lemma~\ref{lem_parabolic_domain_in_N} --- and similarly Lemma~\ref{lem_boundary_far_from_cut} below --- without using the terminology of a comparison and have instead only listed the essential properties of $\Cut$.

\begin{proof}
Set $t := \mathfrak{t}(x)$.
Consider a point $y \in B(x, a)$ and choose $j$ minimal with the property that $y (t)$ is defined and $y(t) \in \N$ for all $t \in [t_j, t]$.
Assume that $t_j > 0$ and $t_j > t - b$.
Then $y (t_j) \in \N_{t_j+} \setminus \N_{t_j-}$.
So $y (t_j)$ is contained in an extension cap and therefore $y (t_j) \in \DD$ for some $\DD \in \Cut$ in contradiction to our assumption.
So $t_j = 0$ or $t_j \leq t-b$.
It follows that $P(x,a,-b) \subset \N$.
Combining this with the assumption of the lemma yields (\ref{eq_P_x_a_b_in_N_D}).
\end{proof}

The following result shows that any point near the neck-like boundary of a comparison domain is far from cuts, in the sense that there is a large backward parabolic region that is disjoint from the cuts.  
This result  plays an important role in Section~\ref{sec_extend_comparison}, where it allows us to isolate behavior occurring at the cuts from behavior that occurs near the neck-like boundary.  

\begin{lemma}[Boundaries and cuts are far apart] \label{lem_boundary_far_from_cut}
If
\begin{gather*}
 \eta_{\lin} > 0, \qquad 
 \delta_{\nn} \leq \ov\delta_{\nn}, \qquad 
 \lambda \leq \ov\lambda, \qquad
   D_{\cut} > 0, \qquad A_0  > 0, \qquad \Lambda\geq\ul\Lambda\, , \\ 
 \delta_{\bb} \leq \ov\delta_{\bb} (\lambda,D_{\cut}, A_0, \Lambda ), 
 \qquad \eps_{\can} \leq \ov\eps_{\can}(\lambda,D_{\cut}, A_0, \Lambda ), \qquad
 r_{\comp} \leq \ov{r}_{\comp}, 
\end{gather*}
then the following holds.

Suppose $0< T<\infty$, and consider Ricci flow spacetimes $\M, \M'$ that are $(\eps_{\can} r_{\comp}, \lb T)$-complete and that satisfy the $\eps_{\can}$-canonical neighborhood assumption at scales $(\eps_{\can} r_{\comp}, 1)$.
Let $(\N, \lb \{ \N^j \}_{j=1}^{J+1}, \lb \{ t_j \}_{j=0}^{J+1})$ be a comparison domain on the time-interval $[0,t_{J+1}]$, and $(\Cut, \lb \phi, \lb \{ \phi^j \}_{j=1}^J)$ be a comparison from $\M$ to $\M'$ defined on  this comparison domain over the interval $[0,t_J]$.  
Assume that $t_{J+1} \leq T$ and that this comparison domain and comparison satisfy a priori assumptions \ref{item_time_step_r_comp_1}--\ref{item_eta_less_than_eta_lin_13} for the tuple of parameters $(\eta_{\lin}, \lb \delta_{\nn}, \lb \lambda, \lb D_{\CAP}, \lb \Lambda, \lb  \delta_{\bb}, \lb \eps_{\can}, \lb r_{\comp})$.

Let  $\Cut^J$ be a set of pairwise disjoint $3$-disks in $\Int \N_{t_J+}$ such that each $\DD \in \Cut^J$ contains an extension cap of the comparison domain.
Assume that the diameter of each $\DD \in \Cut \cup \Cut^J$ is less than $D_{\cut} r_{\comp}$.

Suppose $x\in \N_t=\N_{t-}\cup\N_{t+}$ and $P(x,A_0\rho_1(x))\cap\DD\neq\emptyset$ for some $\DD\in\Cut\cup\Cut^J$, where $\DD\subset\M_{t_k} $.  

Then $B(x,A_0\rho_1(x))\subset \N_{t+}\cap\N_{t-}$ if $t>t_{k}$, and $B(x,A_0\rho_1(x))\subset \N_{t+}$ if $t=t_{k}$.
\end{lemma}

As in Lemma~\ref{lem_parabolic_domain_in_N} we have introduced a set $\Cut^J$ of ``synthetic'' cuts at time $t_J$ in order to avoid complications due to the possible lack of a map $\phi^{J+1}$ that extends the comparison $(\Cut, \phi, \{ \phi^j \}_{j=1}^J)$ past time $t_J$.

The sketch of the proof is as follows.  The cut $\DD$ contains an extension cap, which by a priori assumption \ref{item_geometry_cap_extension_5} and Lemma~\ref{lem_promoting_Bryant} implies that a large future parabolic region is Bryant-like.  Then the Bryant Slice and Slab Lemmas, applied inductively on time steps, imply that the comparison domain contains this Bryant-like  region for many time steps, which excludes neck-like boundary in the vicinity.

\begin{proof}
Pick  $y\in P(x,A_0\rho_1(x))\cap \DD$.  By Lemma~\ref{lem_bounded_curv_bounded_dist} and a priori assumption \ref{item_lambda_thick_2}, if
\[  \eps_{\can}\leq\ov\eps_{\can}(\lambda, A_0)\,,\] 
then 
\begin{equation}
\label{eqn_y_scale_comparable_x_scale}
C_1^{-1}\rho_1(x)\leq \rho_1(y)\leq C_1\rho_1(x)\,,
\end{equation}
for some $C_1=C_1(A_0) < \infty$.

By Definition \ref{def_comparison}(\ref{pr_7.2_3})  and our assumptions regarding $\Cut^J$, we know that $\DD$ contains an extension cap $\C$.  A priori assumption \ref{item_geometry_cap_extension_5} implies that there is a point $z\in \C$ such that $(\M_{t_j},z)$ is $\de_{\bb}$-close to a Bryant soliton at scale $10\lambda r_{\comp}$. 

Let $\delta_\#>0$ be a constant that will be chosen at the end of the proof.

Choose $l \in \{ 2, \ldots, J+1 \}$ such that $t \in [t_{l-1}, t_l)$ or $[t_{l-1}, t_l]$ if $l = J+1$.
Since $\DD \subset \M_{t_k}$, we have $k \leq l-1$.
By a priori assumption \ref{item_backward_time_slice_3}(e) and (\ref{eqn_y_scale_comparable_x_scale}) we have
\begin{equation}\label{eqn_time_bound_t_ell_t_k}
t_l-t_k\leq (A_0\rho_1(x))^2+r_{\comp}^2\leq \big( (A_0C_1\Lambda)^2+1 \big) r_{\comp}^2\,.
\end{equation}
Assuming 
\[ \de_{\bb} \leq\ov\de_{\bb} (\lambda, A_0, \Lambda,\delta_\#), \qquad
\eps_{\can} \leq \ov\eps_{\can} (\lambda, A_0, \Lambda,\delta_\#), 
\qquad r_{\comp}\leq\ov r_{\comp}, \]
we can use (\ref{eqn_time_bound_t_ell_t_k}), a priori assumption \ref{item_lambda_thick_2} and Lemma~\ref{lem_promoting_Bryant} to find a $(10 \lambda r_{\comp})^2$-time equivariant and $\partial_{\mathfrak{t}}$-preserving diffeomorphism 
\[ \psi : W^*:=\ol{M_{\Bry} (\delta_\#^{-1} )} \times \big[ {0, (t_{l}-t_{k}) \cdot (10 \lambda r_{\comp})^{-2}} \big] \longrightarrow \M \]
onto its image,   such that $\psi (x_{\Bry},0) = z$ and
\begin{equation*} 
 \big\Vert{ (10 \lambda r_{\comp})^{-2} \psi^* g - g_{\Bry} }\big\Vert_{C^{[\delta_\#^{-1}]} (W^*)} < \delta_\#.
\end{equation*}
Let $W:=\psi(W^*)$.  

In the following we will apply the Bryant Slice Lemma~\ref{lem_Bryant_slice} at time $t_j$ for $X = \N_{t_j-}$, using the time-slice $W_{t_j}$, where $k \leq j \leq l$.
We will also apply the Bryant Slab Lemma~\ref{lem_bryant_slab} for  $X_0 = \N_{t_{j-1}-}$ and $X_1 = \N_{t_{j}-}$, using the time slab $W_{[t_{j-1},t_j]}$, where $k+1 \leq j \leq l$.
Note that assumptions \ref{con_8.42_i}--\ref{con_8.42_iv} of the Bryant Slice Lemma hold due to a priori assumptions \ref{item_backward_time_slice_3}(a)--(c) and assumption \ref{con_8.42_v} of the Bryant Slab Lemma holds due to Definition~\ref{def_comparison_domain}(\ref{pr_7.1_3}).
If 
\[ \delta_{\nn} \leq \ov\delta_{\nn}, \qquad
 0 < \lambda < 1, \qquad
  \Lambda \geq \underline{\Lambda}, \qquad
\de_\#\leq \ov\de_\#(\lambda, A_0, \Lambda), \]
then the remaining assumptions of both the Bryant Slice and the Bryant Slab Lemma are satisfied.
This means, in particular, that the time-slice $W_{t_j}$ and the slab $W_{[t_{j-1},t_j]}$ satisfy the assumptions of the Bryant Slice/Slab Lemma for all $k \leq j \leq l$ and all $k+1 \leq j \leq l$, respectively.

\begin{claim*} \quad
\begin{enumerate}[label=(\alph*)]
\item \label{ass_claim_8.47_a} $W_{t_j}\subset\N_{t_j- }$ for all $k+1\leq j \leq l$. 
\item \label{ass_claim_8.47_b} $W_t\subset \N_{t+}\cap\N_{t-}$ if $t>t_k$, and $W_t\subset \N_{t+}$ if $t=t_k$.
\end{enumerate}
\end{claim*}
\begin{proof}
Let $\C_j := W_{t_j}\setminus\Int\N_{t_j-}$ for $k\leq j\leq l$.
By assertion \ref{ass_8.42_a} of the Bryant Slice Lemma we know that $\C_j$ is either empty or is a 3-disk in $\Int W_{t_j}$ for all $k \leq j \leq l$.
Furthermore, assertion \ref{ass_8.42_b} of the Bryant Slice Lemma implies that $\C_k = \C$.

We will now show by induction that $\C_j = \emptyset$ for all $k +1 \leq j \leq l$.
This will imply assertion \ref{ass_claim_8.47_a}.

To see this, observe first that if $\C_{k+1} \neq \emptyset$, then by the Bryant Slab Lemma we have $\C = \C_k \subset \C_{k+1} (t_k)$.
However, since $\C$ is an extension cap, we have $\C \subset \N_{t_k+}$, in contradiction to the fact that $\Int \C_{k+1} (t_k) \subset \M_{t_k} \setminus \N_{t_k+}$.

Next, assume that $k+2 \leq j \leq l$ and that $\C_{j-1} = \emptyset$, but $\C_j \neq \emptyset$.
Then, by the Bryant Slab Lemma, $\C_{j} (t)$ is defined and $9\lambda r_{\comp}$-thick for all $t \in [t_{j-1}, t_j]$.
Since $\C_{j-1} = \emptyset$, $W$ is a product domain and $\C_j\subset\Int W_{t_j}$, we have $\C_j (t_{j-1}) \subset \Int W_{t_{j-1}} \subset \N_{t_{j-1}-}$.
So $\C_j (t_{j-1})$ is a component of $\N_{t_{j-1}-} \setminus \Int \N_{t_{j-1}+}$ and $\partial \C_j (t_{j-1}) \subset \N_{t_{j-1}+}$.
This, however, contradicts a priori assumption \ref{item_discards_not_too_thick_4}, finishing the induction.

To see assertion  \ref{ass_claim_8.47_b}, observe that by assertion  \ref{ass_claim_8.47_a} for $j = l \geq k+1$ we have $W_t = W_{t_l} (t) \subset \N_{t_l-} (t)$.
As $t < t_l$ if $l \neq J+1$ and $\N_{t+}=\N_{t-}$ if $l=J+1$, this implies that $W_t \subset \N_{t+}$.
Assume now that $t > t_k$.
If $t > t_{l-1}$, then we trivially have $W_t \subset \N_{t+} = \N_{t-}$.
Lastly, if $t = t_{l-1} > t_k$, then $l-1 \geq k+1$ and therefore assertion  \ref{ass_claim_8.47_a} yields that $W_{t_{l-1}} \subset \N_{t_{l-1} -}$.
\end{proof}

We will now show that $B(x,A_0\rho_1(x))\subset W_t$.  
In combination with  assertion \ref{ass_claim_8.47_b} of the claim, this completes the proof of the lemma.

By the assumption of the lemma  we have $d_{t_k}(y,z)<  D_{\cut}r_{\comp}$.  
So if
\[ \delta_\#\leq\ov\delta_\#(\lambda,D_{\cut}), \]
then $y\in \DD\subset W_{t_k}$.  Recall that $\Ric>0$ on $(M_{\Bry},g_{\Bry})$.
So $g_{\Bry}$ is decreasing in time.  
Therefore, if
\[\delta_\#\leq\ov\delta_\# (\lambda,D_{\cut}), \]
then $d_t(y(t),z(t))\leq 2d_{t_k}(y,z)\leq 2D_{\cut}r_{\comp}$.  
Now by \ref{item_backward_time_slice_3}(e)
\begin{align*}
d_t(x,z(t))&\leq d_t(x,y(t))+d_t(y(t),z(t))\\
&< A_0\rho_1(x)+2D_{\cut}r_{\comp}\\
&\leq A_0C_1\rho_1(y)+2D_{\cut}r_{\comp}\\
&<(A_0C_1\Lambda+2D_{\cut})r_{\comp}\,.
\end{align*}
Therefore, assuming
\[\delta_\#\leq\ov\delta_\#(\lambda,D_{\cut},A_0,\Lambda ), \]
we have $B(x,A_0\rho_1(x)) \subset W_t$, as desired.
\end{proof}

The next lemma characterizes parabolic neighborhoods whose initial time-slices intersect a cut of a comparison.
It states that points that belong to such an initial time-slice, but not to the corresponding cut, must have large scale if certain parameters are chosen appropriately.
We also obtain that such an initial time-slice must be far from cuts that occur at earlier times.
The first assertion will follow from the fact that the geometry on and near a cut is geometrically sufficiently close to a Bryant soliton and the second assertion will be a consequence of Lemma~\ref{lem_boundary_far_from_cut}.

The results of the following lemma are specific for the proof in Subsection~\ref{subsec_verification_of_APAs}.
As in the previous lemmas, we will use a set $\Cut^J$ of ``synthetic'' cuts in time-$t_J$-slice.  Instead, we have listed the relevant properties of the cuts as assumptions of the lemma.

\begin{lemma} \label{lem_large_scale_near_neck}
For all $C_\#<\infty$, if 
\begin{gather*}
\delta_{\nn} \leq \ov\delta_{\nn}, \qquad 
\lambda\leq\ov\lambda,\qquad
D_{\cut} \geq \underline{D}_{\cut} (\lambda, C_\#), \qquad
\Lambda \geq \underline{\Lambda}, \\\qquad 
\delta_{\bb} \leq \ov\delta_{\bb} (\lambda, C_\#, D_{\cut}, A_0, \Lambda ), \qquad
\eps_{\can} \leq \ov\eps_{\can} (\lambda, D_{\cut}, A_0, \Lambda ), \\\qquad 
r_{\comp} \leq \ov{r}_{\comp} (C_\#) ,
\end{gather*}
then the following holds.

Suppose $0< T<\infty$, and consider Ricci flow spacetimes $\M, \M'$ that are $(\eps_{\can} r_{\comp}, \lb T)$-complete and that satisfy the $\eps_{\can}$-canonical neighborhood assumption at scales $(\eps_{\can} r_{\comp}, 1)$.
Let $(\N, \lb \{ \N^j \}_{j=1}^{J+1}, \lb \{ t_j \}_{j=0}^{J+1})$ be a comparison domain on the time-interval $[0,t_{J+1}]$, and $(\Cut, \lb \phi, \lb \{ \phi^j \}_{j=1}^J)$ be a comparison from $\M$ to $\M'$ defined on  this comparison domain over the time-interval $[0,t_J]$.  
Assume that $t_{J+1} \leq T$ and that this comparison domain and comparison satisfy a priori assumptions \ref{item_time_step_r_comp_1}--\ref{item_eta_less_than_eta_lin_13} for the tuple of parameters $(\eta_{\lin}, \lb \delta_{\nn}, \lb \lambda, \lb D_{\CAP}, \lb \Lambda, \lb  \delta_{\bb}, \lb \eps_{\can}, \lb r_{\comp})$.
Let $\Cut^J$ be a set of pairwise disjoint $3$-disks in $\N_{t_J+}$ such that each $\DD \in \Cut^J$ contains  exactly one extension cap of the comparison domain.

Assume that the diameter of every $\DD \in \Cut \cup \Cut^J$ is less than $D_{\cut} r_{\comp}$ and that the $\frac1{10} D_{\cut} r_{\comp}$-neighborhood of every extension cap is contained in some $\DD \in \Cut \cup \Cut^J$.

Let $x \in \N$ and $t := \mathfrak{t} (x)$.
Let $B_{t-T_0} := (B(x, A_0 \rho_1(x)))(t-T_0)$ be the initial time-slice of the parabolic neighborhood $P(x, \lb A_0 \rho_1 (x), \lb -T_0)$ for some $0 \leq T_0 \leq (A_0 \rho_1(x))^2$ and assume that $B_{t-T_0} \cap \DD_0 \neq \emptyset$ for some $\DD_0 \in \Cut \cup \Cut^J$.

Then
\[ \rho_1 \geq C_\# r_{\comp} \qquad \text{on} \qquad B_{t-T_0} \setminus \DD_0. \]
Moreover, for all $y \in B_{t-T_0}$ we have
\[ P (y, A_0 \rho_1(y)) \cap \DD = \emptyset \]
for all $\DD \in \Cut$ with $\DD \subset \M_{[0, \mathfrak{t}(y))}$.
\end{lemma}

\begin{proof}
Let $t := \mathfrak{t} (x)$ and choose $j \in \{ 1, \ldots, J \}$ such that $t_j = t - T_0$, so $B_{t-T_0}\cup \DD_0\subset \M_{t_j}$.
Let $\C_0$ be the extension cap that is contained in $\DD_0$.

By Lemma~\ref{lem_bounded_curv_bounded_dist} and a priori assumption \ref{item_lambda_thick_2}, and assuming
\[ \eps_{\can} \leq \ov\eps_{\can} (\lambda, A_0), \]
we find that the parabolic neighborhood $P(x, A_0 \rho_1 (x))$ is unscathed and that
\begin{equation} \label{eq_C1_bound_rho_1}
 C_1^{-1} \rho_1 (x) \leq \rho_1 \leq C_1 \rho_1 (x) 
\end{equation}
on $P(x, A_0 \rho_1 (x))$, where $C_1 = C_1 (A_0) < \infty$.
By a distance distortion estimate this implies that $B_{t-T_0} \subset B( x(t_j), A_1 \rho_1(x))$ for some $A_1 = A_1 (A_0) < \infty$.

Choose a point $z \in \partial \C_0 \subset \N_{t_j -} \cap \DD_0$.
By a priori assumption \ref{item_backward_time_slice_3}(a) and assuming
\[ \delta_{\nn} \leq \ov\delta_{\nn}, \]
we have $\frac12 r_{\comp} \leq \rho_1 (z) \leq 2 r_{\comp}$.
So, again by Lemma~\ref{lem_bounded_curv_bounded_dist}, and assuming
\[ \eps_{\can} \leq \ov\eps_{\can} (D_{\cut}), \]
we obtain that
\[ C_2^{-1} r_{\comp} \leq \rho_1 \leq C_2 r_{\comp} \qquad \text{on} \qquad \DD_0 \]
for some $C_2 = C_2 (D_{\cut}) < \infty$.
Combining this bound with (\ref{eq_C1_bound_rho_1}) and the fact that $B_{t-T_0} \cap \DD_0 \neq \emptyset$, we obtain that
\begin{equation} \label{eq_C12C2_rho_1}
 C_1^{-2} C_2^{-1} r_{\comp} \leq \rho_1 \leq C_1^2 C_2 r_{\comp} \qquad \text{on} \qquad B_{t-T_0}. 
\end{equation}
Therefore for all $y \in B_{t-T_0}$
\begin{equation} \label{eq_dtj_y_z_C1C2A1}
 d_{t_j} (y, z) \leq (2C_1^2 C_2 A_1 + D_{\cut}) r_{\comp} \leq C_3 r_{\comp}, 
\end{equation}
for some $C_3 = C_3 ( D_{\cut}, A_0) < \infty$.

By a priori assumption \ref{item_geometry_cap_extension_5}(c) there is a diffeomorphism
\[ \psi : M_{\Bry} (\delta_{\bb}^{-1} ) \longrightarrow W \subset \M_{t_j} \]
such that $\psi (x_{Bry}) \subset \C_0$ and
\[ \Vert (10 \lambda r_{\comp})^{-2} \psi^* g_{t_j} - g_{\Bry} \Vert_{C^{[\delta_{\bb}^{-1}]} ( M_{\Bry} (\delta_{\bb}^{-1} ) )} < \delta_{\bb}. \]
So by (\ref{eq_dtj_y_z_C1C2A1}), and the fact that $z \in \partial \C_0$ and that the diameter of $\C_0 \subset \DD_0$ is bounded by $D_{\cut} r_{\comp}$ we have
\begin{equation} \label{eq_BT0_W}
 B_{t-T_0} \subset W, 
\end{equation}
assuming that
\[ \delta_{\bb} \leq \ov\delta_{\bb} (\lambda, D_{\cut}, A_0). \]

Choose $D_\# = D_\# (\lambda, C_\#) < \infty$ such that $\rho > 20 \lambda C_\#$ on $M_{\Bry} \setminus M_{\Bry} ( D_\#)$ (see Lemma~\ref{lem_properties_Bryant}).
So if
\[ \delta_{\bb}  \leq \ov\delta_{\bb} (\lambda, C_\#) , \qquad
r_{\comp} \leq \ov{r}_{\comp} (C_\#), \]
then
\begin{equation} \label{eq_rho1_C_sharp}
  \rho_1 \geq C_\# r_{\comp} \qquad \text{on} \qquad W \setminus \psi ( M_{\Bry} ( D_\#)  ) . 
\end{equation}
If
\[ D_{\cut} \geq \underline{D}_{\cut} ( \lambda, D_\# (\lambda, C_\#)), \qquad
\delta_{\bb} \leq \ov\delta_{\bb}, \]
then $M_{\Bry} ( D_\# ) \subset \DD_0$.
Together with (\ref{eq_BT0_W}) and (\ref{eq_rho1_C_sharp}) this implies the first assertion of this lemma.

For the second assertion note that by (\ref{eq_C12C2_rho_1}) and (\ref{eq_dtj_y_z_C1C2A1}) we have
\[ B(y, C_1^2 C_2 C_3 \rho_1 (y)) \not\subset \N_{t_j-} \]
for all $y \in B_{t-T_0}$. 
So the second assertion follows from Lemma~\ref{lem_boundary_far_from_cut}, assuming
\begin{gather*}
 \delta_{\nn} \leq \ov\delta_{\nn}, \qquad 
 \lambda\leq\ov\lambda,\qquad 
\Lambda \geq \underline{\Lambda}, \qquad
\delta_{\bb} \leq \ov\delta_{\bb} (\lambda, D_{\cut}, A_0, \Lambda ), \qquad \\
\eps_{\can} \leq \ov\eps_{\can} (\lambda, D_{\cut}, A_0, \Lambda ), \qquad 
r_{\comp} \leq \ov{r}_{\comp}. 
\end{gather*}
This finishes the proof.
\end{proof}

\section{Semilocal maximum principle} \label{sec_semi_local_max}
In this section we will show that small Ricci-DeTurck perturbations satisfy a uniform decay estimate when weighted by a suitable function of  time and scale.  More precisely, we show that quantities of the form
\[ Q := e^{H (T - \mathfrak{t})} \rho^E_1 |h|  \]
satisfy a semi-local maximum principle as long as the Ricci-DeTurck perturbation $h$ is small enough, and the Ricci flow background satisfies appropriate geometric assumptions.  The estimates of this section are  based on a vanishing theorem  for solutions $h$ of the linearized Ricci-DeTurck equation on a $\kappa$-solution background, for which $|h| R^{-1-\chi}$ is uniformly bounded, where $\chi > 0$ (see Theorem~\ref{Thm_vanishing_lemma}).  The most important ingredient for the proof of this vanishing theorem is a maximum principle due to Chow and Anderson (see \cite{Anderson:2005cf}).

We first present the two main results of this section, Proposition~\ref{Prop_semi_local_max} and Proposition~\ref{Prop_interior_decay}.
The first result states that a Ricci-DeTurck perturbation decays by a factor of at least $100$ in the interior of a large enough neighborhood, in a weighted sense, as long as the solution is small enough.
The factor $100$ is chosen arbitrarily here and can be replaced by any number $> 1$.

\begin{proposition}[Semi-local maximum principle] \label{Prop_semi_local_max}
If
\[ E > 2, \qquad 
H \geq \underline{H}(E), \qquad
\eta_{\lin} \leq \ov\eta_{\lin} (E), \qquad
\eps_{\can} \leq \ov\eps_{\can} (E), \]
then there are constants $L = L(E), C= C(E) < \infty$ such that the following holds.

Let $\M$ be a Ricci flow spacetime and pick $x\in \M_t$.
Assume that $\M$ is $(\eps_{\can} \rho_1 (x), t)$-complete and satisfies the $\eps_{\can}$-canonical neighborhood assumption at scales $(\eps_{\can} \rho_1 (x), \lb 1)$.

Then the parabolic neighborhood $P:= P(x, L \rho_1(x))$ is unscathed and the following is true. 
Let  $h$ be a Ricci-DeTurck perturbation on $P$.
Assume that $|h| \leq \eta_{\lin}$ everywhere on $P$ and define the scalar function
\begin{equation} \label{eq_def_Q_semi_loc_max}
 Q :=e^{H (T - \mathfrak{t})} \rho_1^E |h|  
\end{equation}
on $P$, where $T \geq t $ is some arbitrary number.  

Then in the case $t  > (L  \rho_1 (x))^2$ (i.e. if $P$ does not intersect the time-$0$ slice) we have
\[ Q (x) \leq \frac1{100} \sup_{ P} Q . \]
In the case $t  \leq (L \rho_1 (x) )^2$ (i.e. if $P$ intersects the time-$0$ slice) we have
\[ Q(x) \leq \frac1{100} \sup_{ P} Q + C  \sup_{ P \cap \M_0} Q. \]
\end{proposition}

Note that the parabolic neighborhood $P$ may be defined on a time-interval of size less than $(L\rho_1 (x))^2$ if $P$ intersects the initial time-slice $\M_0$.
By performing a time shift, Proposition~\ref{Prop_semi_local_max} can be generalized to the case in which $P$ is defined on a time-interval of size less than $(L\rho_1 (x))^2$ that does not necessarily intersect $\M_0$.
This fact will be used in Section~\ref{sec_extend_comparison} when $P$ intersects a cut, i.e. a discontinuity locus of $h$, at some positive time.

We also remark that the constant $T$ in Proposition~\ref{Prop_semi_local_max} does not have any mathematical significance and could be eliminated from the statement.  
It is present in Proposition \ref{Prop_semi_local_max} only to conform with the notation later in the paper where it is used.

In the next result, we improve the interior estimate and replace the factor $100$ by an arbitrary factor.
As a trade-off, we need to choose the parabolic neighborhood on which $h$ and $Q$ are defined large enough; note however that we don't need to change the bound on $|h|$ appearing in the assumptions.

\begin{proposition}[Interior decay] \label{Prop_interior_decay}
If
\begin{gather*}
 E > 2, \qquad
H \geq \underline{H}(E), \qquad
\eta_{\lin} \leq \ov\eta_{\lin} (E), \qquad
\alpha > 0,  \\
A \geq \underline{A}(E, \alpha), \qquad
\eps_{\can} \leq \ov\eps_{\can} (E, \alpha),
\end{gather*}
then there is a constant $C = C(E) < \infty$ such that the following holds.

Let $\M$ be a Ricci flow spacetime and $x\in\M_t$.
Assume that $\M$ is $(\eps_{\can} \rho_1(x), t)$-complete and satisfies the $\eps_{\can}$-canonical neighborhood assumption at scales $(\eps_{\can} \rho_1 (x), \lb 1)$.

Consider the parabolic neighborhood $P := P(x, A \rho_1 (x))$ and let $h$ be a Ricci-DeTurck perturbation on $P$ such that $|h|\leq \eta_{\lin}$ everywhere.  Define $Q$ as in (\ref{eq_def_Q_semi_loc_max}).

Then in the case $t  > (Ar)^2$ (i.e. if $P$ does not intersect the time-$0$ slice) we have
\[ Q (x) \leq \alpha \sup_{ P} Q . \]
In the case $t  \leq (Ar)^2$ (i.e. if $P$ intersects the time-$0$ slice) we have
\[ Q(x) \leq \alpha \sup_{ P} Q + C \sup_{ P \cap \M_0} Q. \]
\end{proposition}

We remark that it follows from the proof that the parabolic neighborhood $P(x, \lb\ul{A}\rho_1(x))$ is unscathed, although we cannot guarantee this for $P(x,A\rho_1(x))$.  Due to the way the proposition will be applied later, it is more convenient to state the conditions using the possibly larger scale $A$.

The proofs of Propositions~\ref{Prop_semi_local_max} and \ref{Prop_interior_decay} are based on the following strong maximum principle for solutions of the linearized Ricci-DeTurck flow.  
This maximum principle is  a special case of a result of Anderson and Chow (cf \cite{Anderson:2005cf}).   
The proof of Anderson and Chow's result simplifies in this special case, which is why we have decided to include it in this paper.

\begin{lemma}[Strong maximum principle of Anderson-Chow] \label{lem_strong_max_AC}
Let $(M, \linebreak[1] (g_t)_{t \in (-T, 0]})$, $T > 0$, be a Ricci flow on a connected $3$-manifold $M$ such that $(M, g_t)$ has non-negative sectional curvature for all $t \in (- T, 0]$.

Consider a solution $(h_t)_{t \in [-T, 0]}$ of the linearized Ricci-DeTurck equation on $M$, i.e.
\[ \partial_t h_t = \triangle_{L, g_t} h_t \qquad \iff \qquad \nabla_{\partial_t} h_t = \triangle_{g_t} h_t + 2 \Rm_{g_t} (h_t) . \]
Assume that
\[ |h| \leq CR \qquad \text{on} \qquad M \times (- T, 0] \]
for some $C > 0$ and that $|h|(x_0, 0) = C R(x_0, 0)$ for some $x_0 \in M$.
Then
\[ |h| = CR \qquad \text{on} \qquad M \times (- T, 0]. \]
\end{lemma}

\begin{proof}
Using Kato's inequality it is not hard to see that wherever $|h| \neq 0$ we have
\[ \partial_t |h| \leq \triangle_{g_t} |h| + 2\frac{\Rm (h, h)}{|h|^2} \cdot |h|. \]
On the other hand, whenever $R > 0$ we have\
\[ \partial_t (CR) = \triangle_{g_t} (CR) + 2 \frac{|{\Ric}|^2}{R} \cdot C R. \]
So the claim follows by the strong maximum principle applied to $|h|-CR$ if we can show that for any symmetric 2-tensor $h$
\begin{equation} \label{eq_Rm_Ric_inequality}
 \frac{\Rm (h, h)}{|h|^2} \leq \frac{|{\Ric}|^2}{R}. 
\end{equation}

To see (\ref{eq_Rm_Ric_inequality}) let $h_{ij}\neq 0$ be a non-zero $3$-dimensional symmetric $2$-tensor and $\Rm_{ijkl}$ a $3$-dimensional algebraic curvature tensor with non-negative sectional curvature.
We denote by $\Ric_{ij}$ and $R$ its Ricci and scalar curvatures.
Without loss of generality, we may assume that $|h|=1$ and that $\Ric_{ij}$ is diagonal.
Then $\Rm_{ijkl}$ is only non-zero if $\{ i, j, k, l \}$ has cardinality $2$.
Set $a_1 := \Rm_{2332}, a_2 := \Rm_{1331}, a_3 := \Rm_{1221}$ and $x_i := h_{ii}$.
Then
\begin{align*}
\Rm (h, h) &= \Rm_{ijkl} h_{il} h_{jk} \\
& = - 2 a_1 h_{23}^2 - 2 a_2 h_{13}^2 - 2 a_3 h_{12}^2 + 2 a_1 h_{22} h_{33} + 2 a_2 h_{11} h_{33} + 2 a_3 h_{11} h_{22}\\
&\leq 2 ( a_1  x_2 x_3 + a_2  x_1 x_3 + a_3  x_1 x_2 )\,.
\end{align*}
On the other hand
\[ |{\Ric}|^2 = (a_2 + a_3)^2 + (a_1 + a_3)^2 + (a_1 + a_2)^2 \]
and
\[ R = 2 (a_1 + a_2 + a_3). \]
Since 
$x_1^2 + x_2^2 + x_3^2\leq |h|^2 = 1$ the next lemma implies
 (\ref{eq_Rm_Ric_inequality}).
\end{proof}

\begin{lemma}
If $x_1^2 + x_2^2 + x_3^2 \leq 1$ and $a_1, a_2, a_3 \geq 0$ and $a_1 + a_2 + a_3 > 0$, then
\begin{equation} \label{eq_A_C_inequality}
 a_1  x_2 x_3 + a_2  x_1 x_3 + a_3  x_1 x_2  \leq \frac{(a_2 + a_3)^2 + (a_1 + a_3)^2 + (a_1 + a_2)^2}{4(a_1 + a_2 + a_3)}, 
\end{equation}
\end{lemma}

\begin{proof}
Let $\lambda_1 \leq \lambda_2 \leq \lambda_3$ be the eigenvalues of the symmetric matrix
\[ A := \frac12 \left( \begin{matrix} 0 & a_3 & a_2 \\ a_3 & 0 & a_1 \\ a_2 & a_1 & 0 \end{matrix} \right) \]
and denote by $v_1, v_2, v_3 \in \R^3$ the corresponding orthonormal basis of eigenvectors.
The left-hand side of (\ref{eq_A_C_inequality}) is bounded from above by $\lambda_3$.

Since the trace of $A$ vanishes and its determinant equals $\frac14 a_1a_2a_3 \geq 0$, we must have $\lambda_1, \lambda_2 \leq 0$ and $\lambda_3 \geq 0$.
In the case $\lambda_3 = 0$ we are done.
So assume from now on that $\lambda_3 > 0$.
Consider the vector
\[ u := \left( \begin{matrix} 1 \\ 1 \\ 1 \end{matrix} \right) = c_1 v_1 + c_2 v_2 + c_3 v_3. \]
Since
\[ A u = \frac12 \left( \begin{matrix} a_2+a_3 \\  a_1+a_3 \\ a_1 + a_2 \end{matrix} \right) \qquad \text{and} \qquad u^T A u = a_1 + a_2 + a_3, \]
we obtain
\[ \frac{c_1^2 \lambda_1^2 + c_2^2 \lambda_2^2 + c_3^2 \lambda_3^2}{c_1^2 \lambda_1 + c_2^2 \lambda_2 + c_3^2 \lambda_3} = \frac{(a_2 + a_3)^2 + (a_1 + a_3)^2 + (a_1 + a_2)^2}{4(a_1 + a_2 + a_3)}. \]
Since $\lambda_1, \lambda_2 < 0$ and numerator and denominator of the first fraction are both positive, we obtain
\[ \lambda_3 = \frac{c_3^2 \lambda_3^2}{ c_3^2 \lambda_3} \leq \frac{(a_2 + a_3)^2 + (a_1 + a_3)^2 + (a_1 + a_2)^2}{4(a_1 + a_2 + a_3)}. \]
This is what we wanted to show.
\end{proof}

\begin{theorem}[Vanishing Theorem] \label{Thm_vanishing_lemma}
Consider a 3-dimensional $\kappa$-so\-lu\-tion $(M, \lb (g_t)_{t \in (- \infty, 0]})$ and a smooth, time-dependent tensor field $(h_t)_{t \in (-\infty, 0]}$ on $M$ that satisfies the linearized Ricci-DeTurck equation
\[ \partial_t h_t = \triangle_{L, g_t} h_t. \]
Assume that there are numbers $\chi > 0$ and $C < \infty$ such that
\begin{equation} \label{eq_h_R_chi}
 |h| \leq C R^{1+\chi} \qquad \text{on} \qquad M \times (- \infty, 0]. 
\end{equation}
Then $h \equiv 0$ everywhere.
\end{theorem}

\begin{proof}
Assume that $h_0 \not\equiv 0$.
Since $(M, (g_t)_{t \in (- \infty, 0]})$ has uniformly bounded curvature, we have
\[ |h| \leq C' R \]
for some $C' < \infty$.
Choose a sequence $(x_k, t_k) \in (- \infty, 0] \times M$ such that
\[ \lim_{k \to \infty} \frac{|h|(x_k, t_k)}{R(x_k, t_k)} = \sup_{M \times (- \infty, 0]} \frac{|h|}{R}. \]
It follows from (\ref{eq_h_R_chi}) that
\[ C R^{\chi} (x_k, t_k) \geq  \frac{|h|(x_k, t_k)}{R(x_k, t_k)} . \]
So there is a $c > 0$ such that $R(x_k, t_k) > c$ for all $k$.
Consider the sequence of pointed flows $(M, (g_{t + t_k})_{t \in (- \infty, 0]},  x_k)$.
After passing to a subsequence, this sequence converges to a pointed $\kappa$-solution $(M_\infty, \lb (g_{\infty,t})_{t \in (-\infty, 0]}, \lb x_\infty )$.
Similarly, consider the sequence of time-dependent tensor fields $h_k (\cdot,  t +t_k)$.
After passing to another subsequence, these tensor fields converge to a solution $(h_{\infty, t})_{t \in (- \infty, 0]}$ of the linearized Ricci-DeTurck flow on $M_\infty \times (- \infty, 0]$.
The bound (\ref{eq_h_R_chi}) carries over in the limit to
\begin{equation} \label{eq_h_infty_C_chi}
 |h_\infty| \leq C R^{1+\chi} 
\end{equation}
and by the choice of the points $(x_k, t_k)$ we obtain the extra property that
\[ \frac{|h_\infty|(x_\infty, 0)}{R(x_\infty, 0)} = \sup_{M_\infty  \times (-\infty, 0]} \frac{|h_\infty|}{R} = \sup_{M \times (- \infty, 0] } \frac{|h|}{R}=: C' > 0. \]

We can now apply the strong maximum principle, Lemma~\ref{lem_strong_max_AC}, and obtain that
\[ |h_\infty| \equiv C' R \qquad \text{on} \qquad  M_\infty \times (- \infty, 0]. \]
Combining this with (\ref{eq_h_infty_C_chi}) yields that on $M_\infty \times (-\infty, 0]$
\[ C' R\leq C R^{1+\chi}. \]
So $R$ is uniformly bounded from below on $M_\infty \times (- \infty, 0]$.
It follows that $(M, \lb (g_{\infty, t})_{t \in (- \infty, 0]})$ cannot be the round shrinking cylinder or a quotient thereof.
If $M_\infty$ was non-compact, then we can obtain the round shrinking cylinder as a pointed limit of $(M, (g_{\infty, t})_{t \in (- \infty, 0]})$, which contradicts the positive lower bound on $R$.
If, on the other hand, $M_\infty$ was compact, then the maximum principle applied to the evolution equation of $R$ would imply that $\min_{M_\infty} R(\cdot, t) \to 0$ as $t \to - \infty$, again contradicting the positive lower bound on $R$.
\end{proof}

\bigskip

\begin{proof}[Proof of Proposition~\ref{Prop_semi_local_max}]
Fix some $E > 2$ for the remainder of the proof.
By linearity of the desired bounds, we may assume for simplicity that $T = \mathfrak{t}(x)$.

Next, observe that, by bounded curvature at bounded distance, Lem\-ma~\ref{lem_bounded_curv_bounded_dist}, for any choice of $L < \infty$ we may choose $\eps_{\can} \leq \ov\eps_{\can} (L)$ small enough such the parabolic neighborhood $P(x, L \rho_1 (x))$ is unscathed and such that $\rho_1 > c_0 (L) \rho_1(x)$ on this parabolic neighborhood for some $c_0 = c_0 (L) > 0$.

Assume now that the statement was false (for fixed $E>2$).
Choose sequences $\eta_{\lin, k},\eps_{\can,k} \to 0$ and $H_{k} , L_k, C_k \to \infty$ such that $\eps_{\can,k}$ is small enough depending on $L_k$, as discussed in the preceding paragraph.   For each $k$ we can choose a Ricci flow spacetime $\M_k$, points $x_k\in \M_{k, t_k}$, an (unscathed) parabolic neighborhood $P_k := P(x_k, L_k \rho_1(x_k))$ and a Ricci-DeTurck perturbation $h_k$ on $P_k$ such that $|h_k| \leq \eta_{\lin, k}$ on $P_k$, which violate the conclusion of the proposition.  
Thus, setting
\[ Q_k(y)  := e^{H_k (t_k - \mathfrak{t}(y))} \rho_1^E(y) |h_k| (y)\qquad\text{for}\quad y\in P_k\,,\]
either $t_k := \mathfrak{t} (x_k) > (L_k\rho_1(x_k))^2$ and
\begin{equation} \label{eq_Qk_xk_large}
 Q_k (x_k) > \frac1{100} \sup_{P_k} Q_k 
\end{equation}
or $t_k = \mathfrak{t}(x_k) \leq (L_k\rho_1(x_k))^2$ and
\begin{equation} \label{eq_Qk_xk_large_initial}
 Q_k (x_k) > \frac1{100} \sup_{P_k} Q_k + C_{k} \sup_{P_k \cap \M_{k,0}} Q_k. 
\end{equation}

Let us rephrase the bounds (\ref{eq_Qk_xk_large}) and (\ref{eq_Qk_xk_large_initial}) in a more convenient form.
To do this, let $\alpha_k := |h_k (x_k)| \leq \eta_{\lin, k} \to 0$ and consider the tensor field $h'_k := \alpha_k^{-1} h_k$.
Then $h'_k$ is a solution to the rescaled Ricci-DeTurck equation (\ref{eq_rescaled_Ricci_de_Turck}) for $\alpha = \alpha_k$,
\begin{equation} \label{eq_h_prime_k_1_x_k}
 |h'_k|(x_k) = 1 
\end{equation}
and on $P_k$
\[ |h'_k| = \frac{|h_k|}{|h_k|(x_k)} = e^{- H_k (t_k - \mathfrak{t})} \bigg( \frac{ \rho_1}{\rho_1(x_k)} \bigg)^{-E} \cdot \frac{Q_k}{Q_k(x_k)}  .  \]
So by (\ref{eq_Qk_xk_large}) and (\ref{eq_Qk_xk_large_initial}) we have
\begin{alignat}{2}
  |h'_k| &\leq 100 e^{- H_k (t_k - \mathfrak{t})} \bigg( \frac{ \rho_1}{\rho_1(x_k)} \bigg)^{-E} \label{eq_sup_h_prime_k} \qquad &&\text{on} \qquad P_k \\
\intertext{and if $P_k\cap\M_{k,0}\neq\emptyset$, then}
  |h'_k| &\leq C_k^{-1} e^{- H_k (t_k - \mathfrak{t})} \bigg( \frac{ \rho_1}{\rho_1(x_k)} \bigg)^{-E} \qquad &&\text{on} \qquad P_k \cap \M_{k,0}. \label{eq_sup_h_prime_k_initial}
\end{alignat}

We now distinguish two cases.

\textit{Case 1: \quad $t_k  \geq c \rho_1^2(x_k)$ for all $k$ and some $c > 0$.}

  The metric $g_k$ restricted to $P_k$ can be expressed in terms of a classical Ricci flow $(g_{k, t})_{t \in [t_k - \Delta t_k, t_k]}$ on $B_k := B(x_k, L_k \rho_1(x_k))$, where 
  $$\Delta t_k := \min \{ t_k, (L_k \rho_1 (x_k))^2 \}\,.$$   
Let $r_k := \rho_1 (x_k)$ and $T_\infty := \limsup_{k \to \infty} r_k^{-2} \Delta t_k \geq c > 0$.  Consider the parabolically rescaled flows 
  $$(g'_{k ,t}:=r_k^{-2} g_{k, r_k^2 t + t_k})_{t \in [- r_k^{-2} \Delta t_k, 0]}\,.$$
 By bounded curvature at bounded distance, Lemma~\ref{lem_bounded_curv_bounded_dist}, and since $\eps_{\can, k} \to 0$, for any $s <\infty$, $T'<T_\infty$, for sufficiently large $k$ we find uniform bounds on the curvature on the curvature of $g'_{k, 0}$ on the $g'_{k , 0}$-ball  $B(x_k, 0, s)$ over the time-interval $[-T',0]$.

\textit{Case 1a: \quad We have $\liminf_{k \to \infty} \rho_1 (x_k) > 0$, and the injectivity radius satisfies $\liminf_{k\ra\infty}\injrad(g'_{k, 0},x_k)>0$.}

After passing to a subsequence, we may extract a smooth limiting  pointed flow $(M_\infty, (g_{\infty, t})_{t \in (-T_\infty, 0]}, x_\infty)$.
Due to (\ref{eq_sup_h_prime_k}) and the local gradient estimates from Lemma \ref{lem_loc_gradient_estimate},  the reparameterized tensor fields $(  r_k^{-2} h'_{k, r_k^2 t + t_k})_{t \in [ - r_k^{-2} \Delta t_k, 0]}$ converge, after passing to another subsequence, to a smooth solution $(h'_{\infty, t})_{t \in (-T_\infty, 0]}$ on $M_\infty$ of the linearized Ricci-DeTurck equation with background metric $(g_{\infty, t})_{t \in (-T_\infty, 0]}$ (see (\ref{eq_lin_RdT_Appendix}), such that
\begin{equation} \label{eq_h_prime_1_1}
|h'_{\infty}| (x_\infty, 0) = 1.
\end{equation}

Since $\lim_{k \to \infty} H_k \rho_1^2 (x_k) = \infty$, we can use the exponential factor in (\ref{eq_sup_h_prime_k}) to show that $h'_\infty \equiv 0$ on $M_\infty \times (- T_\infty, 0)$, which implies $h'_\infty(x_\infty,0)=0$.
This contradicts (\ref{eq_h_prime_1_1}).

\textit{Case 1a$\,^\prime$: \quad We have $\liminf_{k \to \infty} \rho_1 (x_k) > 0$, and the injectivity radius satisfies $\liminf_{k\ra\infty}\injrad(g'_k,x_k)=0$.}

For some $\wh{r} >0$  we may pull back $g'_k$ to the $\wh{r}$-ball in the tangent space at $x_k$ via the exponential map to reduce to Case~1a.
Note that in Case~1a it was not important that the time-slices of the limiting flow $(M_\infty, (g_{\infty, t})_{t \in (-T_\infty, 0]}, x_\infty)$ were complete.

\textit{Case 1b: \quad $\liminf_{k \to \infty} \rho_1 (x_k) = 0$, and the injectivity radius satisfies $\liminf_{k\ra\infty}\lb\injrad(g'_k(0),x_k)>0$.}

As explained in the beginning of Case~1a, by passing to a subsequence, we may assume that the pointed flows $(B_k, (g'_{k,t})_{k,t}, x_k)$ converge to a smooth pointed flow $(M_\infty, \lb (g_{\infty, t})_{t \in (-T_\infty, 0]}, \lb x_\infty)$ and, moreover, the tensor fields $(  r_k^{-2} h'_{k, r_k^2 t + t_k})_{t \in [ - r_k^{-2} \Delta t_k, 0]}$ converge to a smooth solution $(h'_{\infty, t})_{t \in (-T_\infty, 0]}$ on $M_\infty$ of the linearized Ricci-DeTurck equation with background metric $(g_{\infty, t})_{t \in (-T_\infty, 0]}$ (see (\ref{eq_lin_RdT_Appendix})), such that (\ref{eq_h_prime_1_1}) holds.

Using Lemma~\ref{lem_bounded_curv_bounded_dist} and the canonical neighborhood assumption, it follows that $R>0$ everywhere on $M_\infty \times (-T_\infty, 0]$.  
By assertion \ref{ass_C.1_a} of Lemma~\ref{lem_kappa_solution_properties_appendix} there is a $\kappa_0>0$ such that every $\kappa$-solution is either a shrinking round spherical space form or is a $\kappa_0$-solution.   
Therefore, in view of the injectivity radius bound, there is a $\kappa_1>0$ such that by the canonical neighborhood assumption  every time-slice $(M_\infty, g_{\infty, t})$, $t \in (-T_\infty, 0]$ is isometric to the final time-slice of a $\kappa_1$-solution.  
Since by assertion \ref{ass_C.1_e} of Lemma~\ref{lem_kappa_solution_properties_appendix} we have $\D_tR\geq 0$ on $\kappa$-solutions, we get that  $(M_\infty, (g_{\infty, t})_{t \in (-T_\infty, 0]}, x_\infty)$ has bounded curvature, so it is a $\kappa$-solution if $T_\infty = \infty$.

Passing (\ref{eq_sup_h_prime_k}) to the limit yields
\[ |h'_\infty| \leq 100\rho^{-E} \leq (C')^{E/2} R^{E/2} \qquad \text{on} \qquad M_\infty \times (- T_\infty, 0], \]
for some universal constant $C' < \infty$.

If $T_\infty = \infty$, then the Vanishing Theorem~\ref{Thm_vanishing_lemma} yields that $h'_\infty \equiv 0$, in contradiction to (\ref{eq_h_prime_1_1}).

Now suppose that $T_\infty < \infty$.  
We will show that for some constant $C''<\infty$ we have 
\begin{equation}
\label{eqn_linear_growth_near_t_infty}
|h'_{\infty}(x,t)| \leq C''(t+T_\infty)\,.
\end{equation} 
for all $x\in M_\infty$, $t\in(-T_\infty,0]$. 

As $(M_\infty, g_{\infty, 0})$ is isometric to the final time-slice of a $\kappa_1$-solution, and therefore has uniformly bounded curvature, we can find a constant $a_1 > 0$ such that for any $L'$ we have
\[ \rho > a_1 \rho_1 (x_k) \qquad \text{on} \qquad B(x_k, L' \rho_1 (x_k)), \]
as long as $k$ is chosen large enough.
So, by bounded curvature at bounded distance, Lemma~\ref{lem_bounded_curv_bounded_dist}, there is a constant $a_2 > 0$ such that for any $L' < \infty$ we have
\[  \rho > a_2 \rho_1 (x_k) \qquad \text{on} \qquad P(x_k, L' \rho_1 (x_k), - t_k ) \]
for large $k$.
By (\ref{eq_sup_h_prime_k}), (\ref{eq_sup_h_prime_k_initial}) and Proposition~\ref{prop_promote_RdT_forward_locally} we find a sequence $c_k \to 0$ and a constant  $C'' < \infty$ such that for any $L' < \infty$ we have
\begin{equation} \label{eq_h_k_close_to_initial}
 |h'_k|  < C'' \rho_1^{-2} (x_k) \cdot \mathfrak{t} + c_k  \qquad \text{on} \qquad P(x_k, L' \rho_1 (x_k), - t_k )  
\end{equation}
for large $k$.
Passing this bound to the limit implies (\ref{eqn_linear_growth_near_t_infty}).

Since $\sup|h'_\infty|<\infty$ this forces $h'_\infty\equiv 0$, again contradicting (\ref{eq_h_prime_1_1}).

\textit{Case 1b$\,^\prime$: \quad $\liminf_{k \to \infty} \rho_1 (x_k) = 0$, and the injectivity radius satisfies $\liminf_{k\ra\infty} \lb \injrad(g'_k(0),x_k)=0$.}

After passing to a subsequence, we may assume that  $\injrad(g'_k(0), \lb x_k)\ra 0$ as $k\ra \infty$.  
By Lemma~\ref{lem_kappa_solution_properties_appendix} the universal covers of the flows $(M_k,g'_{k,t})$ converge to shrinking round spheres on the time-interval $(-\infty,0]$.  We may now  pull back the tensor fields $h_k$ to the universal covers and reduce to Case~1b.

\textit{Case 2: \quad $\liminf_{k \to \infty} \rho_1^{-2} (x_k) t_k = 0$.}

 In this case, by combining the  curvature bounds from Lem\-ma~\ref{lem_bounded_curv_bounded_dist} with (\ref{eq_sup_h_prime_k}) and (\ref{eq_sup_h_prime_k_initial}), we can apply Proposition~\ref{prop_promote_RdT_forward_locally} to show that there is a sequence $c_k \to 0$ and constants  $C'', L' < \infty$ such that  (\ref{eq_h_k_close_to_initial}) holds for large $k$.
It follows that $\lim_{k \to \infty} |h'_k |(x_k) = 0$, in contradiction to (\ref{eq_h_prime_k_1_x_k}).
\end{proof}

\begin{proof}[Proof of Proposition~\ref{Prop_interior_decay}.]
The bound follows by iterating the bound from Proposition~\ref{Prop_semi_local_max}.

Assume that
\[ E > 2,  \qquad
H \geq \underline{H} (E), \qquad
\eta_{\lin} \leq \ov\eta_{\lin} (E), \qquad
\eps_{\can} \leq \ov\eps_{\can} (E), \]
and set $C = 2C(E)$ and $L = L(E)$ according to Proposition~\ref{Prop_semi_local_max}.
So Proposition~\ref{Prop_interior_decay} holds if $\alpha \geq \frac1{100}$.
Assume now by induction that $\alpha_0 < \frac1{100}$ and that Proposition~\ref{Prop_interior_decay} holds for $\alpha=100 \alpha_0$ under an assumption of the form
\[ A \geq A' := \underline{A} (E, 100 \alpha_0), \qquad \eps_{\can} \leq \ov\eps_{\can} (E, 100 \alpha_0). \]

Consider  the point $x \in \M$.
By Lemma~\ref{lem_containment_parabolic_nbhd} we can find a constant $A'' = A'' (L(E), A' (E, 100 \alpha_0)) < \infty$ such that if
\[ \eps_{\can} \leq \ov\eps_{\can} (L(E), A' (E, 100 \alpha_0)), \]
then the parabolic neighborhood $P (x, A' \rho_1 (x))$ is unscathed and we have
\[ P(y, A' \rho_1 (y)) \subset P(x, A'' \rho_1 (x)) \qquad \text{for all} \qquad y \in P(x, L\rho_1(x)). \]
Also, by bounded curvature at bounded distance, Lemma~\ref{lem_bounded_curv_bounded_dist}, assuming $\eps_{\can} \leq \ov\eps_{\can} (L(E))$, we know that $\rho_1 \geq c \rho_1 (x)$ on $P(x, L \rho_1 (x))$ for some $c = c(L(E)) > 0$.
 
Assume now that $A \geq A''$ and apply Proposition~\ref{Prop_interior_decay} at each $y \in P(x, L\rho_1(x))$ for $\alpha=100\alpha_0$.
Note that in order to do this, we need to ensure that $\M$ is $(\eps_{\can} \rho_1(y), \t (y))$-complete and satisfies the canonical neighborhood assumption at scales $(\eps_{\can} \rho_1 (y), 1)$.
This can always be guaranteed if we assume that $\eps_{\can} \leq c(L(E)) \eps_{\can} (E, 100 \alpha_0)$.
Proposition~\ref{Prop_interior_decay} for $\alpha=100\alpha_0$ gives us 
\[ \sup_{P(x, L \rho_1(x))} Q \leq 100 \alpha_0 \sup_{P(x, A'' \rho_1 (x))} Q + C \sup_{P(x, A'' \rho_1 (x)) \cap \M_0} Q. \]
Applying Proposition~\ref{Prop_semi_local_max} then implies (recall that we have replaced $C$ by $2C$)
\begin{multline*}
 Q(x) \leq \frac{100\alpha_0}{100} \sup_{P(x, L \rho_1 (x))} Q + \bigg( \frac{ C}{100} + \frac12 C \bigg) \sup_{P(x, A'' \rho_1 (x)) \cap \M_0} Q \\
 \leq \alpha_0 \sup_{P(x, A'' \rho_1 (x))} Q + C \sup_{P(x, A'' \rho_1 (x)) \cap \M_0} Q.
\end{multline*}
This finishes the induction.
\end{proof}

\section{Extending maps between Bryant solitons}
\label{sec_extending_map_between_bryant_solitons}
In this section we consider two regions that are close to Bryant solitons, at possibly different scales, and an almost isometry between annular subdomains inside these regions.
We will then prove that the scales of both Bryant soliton regions are almost equal and that the given almost isometry can be extended to an almost isometry, of possibly lesser accuracy, over the entire Bryant soliton regions.
An important aspect of the main result of this section is that the accuracy that is required from the given almost isometry depends only polynomially on the local scale --- or on the distance from the tip.  

Our main result, the Bryant Extension Proposition (Proposition \ref{prop_bryant_comparison_general}), will be needed in the proof of Proposition~\ref{Prop_performing_cap_extensions} in Section~\ref{sec_extend_comparison}.
In this proposition, we extend an almost isometry between two Ricci flow spacetime time-slices over an extension cap. 
By assumption, the accuracy of this  almost isometry improves at a large polynomial rate as we move away from the extension cap.
As long as this polynomial rate is sufficiently large, we can use Proposition~\ref{prop_bryant_comparison_general}  to construct  an extension of the almost isometry over the extension cap whose accuracy still improves at a large polynomial rate.  
This enables us to retain the fine geometric bounds needed to prolong our comparison. 

In this section we will use the notation $(M_{\Bry}, g_{\Bry}, x_{\Bry})$ for the pointed Bryant soliton with $\rho(x_{\Bry}) = 1$; for this and other notation related to the geometry of the Bryant soliton, we refer to  Subsection~\ref{subsec_basics_Bryant}.
We will also frequently use the curvature scale function $\rho : M_{\Bry} \to (0, \infty)$ as introduced in Definition~\ref{def_curvature_scale}.
Recall that $(M_{\Bry}, g_{\Bry})$ is an $O(3)$-invariant gradient steady soliton diffeomorphic to $\R^3$ and $\rho(x) \to \infty$ as $x \to \infty$.

We first present a version of the Bryant comparison result in a form that is most useful for its application in the proof of Proposition~\ref{Prop_performing_cap_extensions}.

\begin{proposition}[Bryant Extension] \label{prop_bryant_comparison_general}
If
\begin{gather*}
 E \geq \underline{E}, \qquad
C > 0, \qquad
\beta > 0, \qquad
D \geq \underline{D} (E, C, \beta), \qquad  \\
0 < b \leq C, \qquad 
0 < \delta \leq \ov\delta (E, C, \beta, D, b), 
\end{gather*}
then the following holds for any $D' > 0$.

Let $g$ and $g'$ be Riemannian metrics on $M_{\Bry} (D)$ and $M_{\Bry} (D')$, respectively, such that for some $\lambda \in [C^{-1}, C]$
\begin{equation} \label{eq_Bryant_closeness_Bryant_comparison}
 \big\Vert g - g_{\Bry} \Vert_{C^{[\delta^{-1}]} (M_{\Bry} (D))} , \quad \big\Vert \lambda^{ -2} g' - g_{\Bry} \Vert_{C^{[\delta^{-1}]} (M_{\Bry} (D'))} < \delta. 
\end{equation}
Consider a diffeomorphism onto its image $\phi : M_{\Bry} ( \frac12 D, D) \lb \to \lb M_{\Bry} (D')$ such that for $h := \phi^* g' - g$ we have for all $m = 0, \ldots, 4$
\[  \rho^E_g |\nabla^m_g h|_{g} \leq  b  \qquad \text{on} \qquad M_{\Bry} (\tfrac12 D, D)\,, \]
where $\rho_g$ denotes the scale function with respect to the metric $g$.
Then there is a diffeomorphism onto its image $\td\phi : M_{\Bry} (D) \to M_{\Bry} (D')$ such that the following holds:
\begin{enumerate}[label=(\alph*)]
\item \label{ass_10.1_a} $\td\phi = \phi$ on $M_{\Bry} (D-1, D)$.
\item \label{ass_10.1_b} For $\td{h} := \td\phi^* g' - g$ we have
\[ \rho^3_g \big| \td{h} \big|_{g} \leq  \beta b \qquad \text{on} \qquad M_{\Bry} ( D). \]
\end{enumerate}
\end{proposition}

We remark that there are several ways in which one could strengthen or sharpen this proposition.  
We chose the statement above, because it is adequate for our purposes and keeps the  complications in the proof to a minimum.
For example, the constant $\ul E$ in this proposition could be taken to be equal to $100$, or even smaller.
Also, the choice of the exponent $3$ in assertion \ref{ass_10.1_b} is arbitrary.  This exponent is needed in the proof of Proposition~\ref{Prop_performing_cap_extensions}, but it could be replaced by any other number, assuming that $E$ is chosen sufficiently large.

The Bryant Extension Proposition~\ref{prop_bryant_comparison_general} is a consequence of the following simpler result, on which we will focus for the larger part of this section.
A proof that Proposition~\ref{prop_bryant_comparison} implies Proposition~\ref{prop_bryant_comparison_general} is provided at the end of this section.

\begin{proposition}[Bryant Extension, simple form]
\label{prop_bryant_comparison}
There is a constant $C < \infty$ such that if
\[ 0 < \alpha < 1, \qquad E \geq \underline{E}, \qquad D \geq \underline{D} (\alpha), \]
then the following holds.
Assume that: 
\begin{enumerate}[label=(\roman*)]
\item $g_1 = g_{\Bry}$ and $g_2 = \lambda_2^2 g_{\Bry}$ is a rescaled Bryant soliton metric.
\item $\lambda_2 \in [\alpha, \alpha^{-1}]$.
\item $\phi : M_{\Bry} (\frac12 D , D) \to M_{\Bry}$ is a diffeomorphism onto its image.
\item For $h=\phi^*g_2-g_1$ and for some $b \leq \alpha^{-1}$ we have for all $m = 0, \ldots, 4$
\begin{equation*}
\label{eqn_nabla_h_bound}
|\nabla^m_{g_1} h|_{g_1} \leq  b D^{-E} \qquad \text{on} \qquad M_{\Bry} (\tfrac12 D, D).
\end{equation*}
\end{enumerate}
Then there a diffeomorphism onto its image $\td{\phi}: M_{\Bry} ( D) \ra M_{\Bry}$ such that: 
\ben[label=(\alph*)]
\item $\td{\phi}=\phi$ on $M_{\Bry} (D-1, D)$.
\item For $\td{h} := \td{\phi}^*g_2-g_1$ we have 
\begin{equation*}
\label{eqn_e'_bound}
 |h|_{g_1} \leq  b \cdot C \alpha^{-C} D^{-E+C} \qquad \text{on} \qquad M_{\Bry} (D).
\end{equation*}
\een
\end{proposition}

The strategy of the proof is as follows.
We first show that $\phi$ almost preserves the curvature operator and its first covariant derivative, up to an error that decays polynomially in $D$.
As the scale of a Bryant soliton can be expressed in terms of the curvature and its derivative, this will imply that the scale $\lambda_2$ of $g_2$ is close to the scale $1$ of $g_1$, up to an error that decays polynomially in $D$.
Similarly, we can argue that $\phi$ preserves the distance function to the tip $x_{\Bry}$ up to a polynomially decaying error.
Using this extra information, we can in turn argue that $\phi$ is sufficiently close to an isometric rotation of $(M_{\Bry}, g_{\Bry})$ around the tip $x_{\Bry}$, again up to an error that decays polynomially in $D$.
By an interpolation argument, we eventually extend $\phi$ to a map on $M_{\Bry} (D)$ that is equal to this isometric rotation sufficiently far away from the boundary.

The proof will use some standard properties geometric properties of the Bryant soliton, which are reviewed in appendix~\ref{appx_Bryant_properties}.
Recall that $x_{\Bry} \in M_{\Bry}$ denotes the tip, i.e. the center of rotational symmetry, of $M_{\Bry}$.
In the following we furthermore denote by $\sigma := d_{g_{\Bry}} ( \cdot, x_{\Bry})$ the distance function from the tip.

The remainder of this section will be devoted to the proof of Proposition \ref{prop_bryant_comparison}.
Until the end of the section we will let $g_1 = g_{\Bry}$, $g_2$, $\la_2$, etc, be as in the statement of this  proposition.   
Let $g_3=\phi^*g_2$.  
We begin with some estimates on the difference between geometric quantities for $g_1$ and $g_3$.

We will use the convention that $1<C<\infty$ denotes a generic universal constant, which may change from line to line.

\begin{lemma} \label{lem_difference_estimate}
If
\[ E \geq \underline{E}, \qquad D \geq \underline{D} (\alpha), \]
then the following holds.

Let $\cald=\nabla_{g_3}-\nabla_{g_1}$ be the difference tensor for the Levi-Civita connections of $g_3$, $g_1$, respectively.  
Then we have
 $$|T|_{g_1}\leq b \cdot C D^{-E} \qquad \text{on} \qquad  M_{\Bry} (\tfrac12 D, D), $$
where $T$ is any tensor field from the following list:
\begin{align*}
&\{\nabla_{g_1}^{k}(g_3 - g_1)\}_{0\leq k\leq 4 },\;\; \{\nabla_{g_1}^{k}\cald\}_{0\leq k\leq 3 }\\&\{ \nabla_{g_1}^k(\Rm_{g_3}-\Rm_{g_1}), \;\;
 \nabla_{g_1}^k(\Ric_{g_3}-\Ric_{g_1}),\;\; \nabla_{g_1}^k(R_{g_3}-R_{g_1})\}_{0\leq k\leq 2}\,.
\end{align*}
The bound also holds if we view $\Ric_{g_i}$, $i = 1, 3$, as a $(1,1)$-tensor.

\end{lemma}
\begin{proof} 
Consider a point $x \in M_{\Bry} (\frac12 D, D)$ and identify $T_x M_{\Bry}$ with $\R^3$ such that $g_{1,x}$ corresponds to the Euclidean inner product.
The tensors $h_x, \nabla_{g_1} h_x, \ldots, \nabla^4_{g_1} h_x$ and $\Rm_{g_1, x}, \ldots, \nabla^2_{g_1} \Rm_{g_1,x}$  and $T_x$ can be viewed as tensors on $\R^3$.
As $T$ can be written in the form of an algebraic expression involving the tensors $g_1, \lb (g_1 + h)^{-1}, \lb h, \lb \ldots, \lb \nabla^4_{g_1} h, \lb \Rm_{g_1}, \lb \ldots, \lb \nabla^2_{g_1} \Rm_{g_1}$, there is a smooth tensor-valued function $F$ such that
\[ T_x = F(h_x, \ldots, \nabla^4_{g_1} h_x, \Rm_{g_1,x}, \ldots, \nabla^2_{g_1} \Rm_{g_1,x}). \]
Note that
\[ F(0, \ldots, 0, \Rm_{g_1, x}, \ldots, \nabla^2_{g_1} \Rm_{g_1,x}) = 0. \]
So by (\ref{eq_higher_der_bound_Bry}) we have
\begin{equation*}
 | T_x |_{g_1} \leq C \big| ( h_x, \ldots, \nabla^4_{g_1} h_x ) \big|_{g_1} 
\leq C \big( |h_x|_{g_1} + \ldots + |\nabla^4_{g_1} h_x |_{g_1} \big) \leq  CbD^{-E}, 
\end{equation*}
as long as $E \geq \underline{E}$ and $D \geq \underline{D}(\alpha)$.
\end{proof}

We now prove that the scales of $g_1$ and $g_2$ are close, up to an error that decays polynomially in $D$.

\begin{lemma}[Scale detection] \label{lem_scale_detection_Bry}
If
\[ E \geq \underline{E}, \qquad D \geq \underline{D} (\alpha), \]
then we have
\[ |\lambda_2 - 1| \leq b \cdot C\al^{-C} D^{-E+4}. \]
\end{lemma}

\begin{proof}
Set $\lambda_1 := 1$.
Then $g_i = \lambda_i^2 g_{\Bry}$ for $i = 1,2$ and by rescaling (\ref{eqn_soliton_conserved}), (\ref{eqn_soliton_dr}) by $\lambda_i$ we obtain that for $i = 1,2$
\begin{equation} \label{eq_rescaled_soliton_eqs}
R_{g_i}+|\nabla_{g_i} f|_{g_i}^2\equiv  R_{g_i} (x_{\Bry}) = \la_i^{-2} R_{g_{\Bry}}(x_{\Bry}),\quad dR_{g_i}=2\Ric_{g_i}(\nabla _{g_i}f,\cdot).
\end{equation}
In the following, we will express these equations in terms of the metrics $g_1$, $g_2$, by combining the difference estimates from the previous lemma with some estimates on the geometry of the normalized Bryant soliton from Lemma~\ref{lem_properties_Bryant}.
It will then follow that $\lambda_1$ and $\lambda_2$ are close.

In the following, we will work on the annulus $M_{\Bry} (\frac12 D, D)$ and assume that $D > 2C_B$, where $C_B$ is the constant from Lemma~\ref{lem_properties_Bryant}.
Therefore $\sigma > \frac12 D > C_B$ on $M_{\Bry} (\frac12 D, D)$ and thus the bounds of Lemma~\ref{lem_properties_Bryant} apply for $g_1$.
We may also assume that $E \geq \underline{E}$ and $D \geq \underline{D} (\alpha)$ have  been chosen large enough so that $g_1$ and $g_3$ are $2$-bilipschitz on $M_{\Bry} (\frac12 D, D)$.

From (\ref{eq_Ric_Bry_lower_bound}) in Lemma \ref{lem_bryant_geometry} we obtain the following bound for the Ricci tensor, viewed as a quadratic form on $T^*M$, 
\[ \Ric_{g_1}>C_B^{-1} D^{-2} g_1. \]
Therefore, assuming $D$ large enough, the inverse $\Ric_{g_1}^{-1}$, viewed as a map $T^* M\ra T^* M$, is well-defined and satisfies 
\begin{equation} \label{eq_Ric_g1_inverse}
 \big| {\Ric_{g_1}^{-1}} \big|_{g_1} <CC_B D^2. 
\end{equation}
Hence by Lemma~\ref{lem_difference_estimate}, if  $E \geq \underline{E}$ and $D \geq \underline{D} (\alpha)$, then 
$$
\big| {\Ric_{g_1}^{-1}(\Ric_{g_3}-\Ric_{g_1})} \big|_{g_1} \leq \big|{\Ric_{g_1}^{-1}} \big|_{g_1} \big|{\Ric_{g_3}-\Ric_{g_1}} \big|_{g_1} \leq  b \cdot C D^{-E+2}.
$$
So if $E \geq \underline{E}$ and $D \geq \underline{D} (\alpha)$, then the inverse of
\[ I + \Ric_{g_1}^{-1}(\Ric_{g_3}-\Ric_{g_1}) = \Ric_{g_1}^{-1} \Ric_{g_3} \]
exists and we have
\[ \big| {\Ric_{g_3}^{-1} \Ric_{g_1} - I  }\big|_{g_1} \leq  b \cdot C  D^{-E+2}. \]
Therefore again by (\ref{eq_Ric_g1_inverse}),
\begin{equation} \label{eq_Ric_g_inverse_diff_bound}
 \big|{ \Ric_{g_3}^{-1} - \Ric_{g_1}^{-1} }\big|_{g_1} \leq \big| { \Ric_{g_3}^{-1} \Ric_{g_1} - I }  \big|_{g_1}  \big|{\Ric_{g_1}^{-1}} \big|_{g_1} \leq  b \cdot C \alpha^{-C} D^{-E+4}. 
\end{equation}

Using the second relation in (\ref{eq_rescaled_soliton_eqs}) we find that
\begin{multline} \label{eq_dfphi_df}
 d (f\circ\phi)- d f = 2 \Ric_{g_3}^{-1}(d R_{g_3})- 2 \Ric_{g_1}^{-1}(d R_{g_1})\\
= 2 (\Ric_{g_3}^{-1}-\Ric_{g_1}^{-1})(d R_{g_3}) + 2 \Ric_{g_1}^{-1} (d R_{g_3} - d R_{g_1} )  . 
\end{multline}
So, as $|dR_{g_3}|_{g_1} \leq C |dR_{g_3}|_{g_3} \leq C \lambda_2^{-3} \leq C \alpha^{-3}$, we obtain by (\ref{eq_dfphi_df}), (\ref{eq_Ric_g1_inverse}), (\ref{eq_Ric_g_inverse_diff_bound}) and Lemma~\ref{lem_difference_estimate} that
\[ |d(f\circ\phi)-d f |_{g_1} \leq  b \cdot C\al^{-C} D^{-E+4}. \]
It follows using (\ref{eq_nab_f_growth}) that
\begin{align*}
\big| |d(f  \circ\phi) |_{g_3}^2- |d f |_{g_1}^2 \big| &\leq \big| |d (f \circ \phi ) |_{g_3}^2 - |df|_{g_3}^2 \big| + \big| |df |^2_{g_3} - |df|^2_{g_1} \big| \\
& \leq \big| d (f \circ \phi ) - df \big|_{g_3} \cdot \big| d (f \circ \phi ) + df \big|_{g_3} + C |h|_{g_1} |df|^2_{g_1} \\
&\leq b \cdot C \alpha^{-C} D^{-E+4} \big( |d (f \circ \phi )|_{g_3} + |df|_{g_3} \big) + b \cdot C \alpha^{-C} D^{-E} \\
&\leq b \cdot C \alpha^{-C} D^{-E+4} \big( |d (f \circ \phi ) - df |_{g_3} + 2 |df|_{g_1 } \big) + b \cdot C \alpha^{-C} D^{-E} \\
&\leq b \cdot C \alpha^{-C} D^{-E+4}. 
\end{align*}
Combining this with (\ref{eq_rescaled_soliton_eqs}) and Lemma~\ref{lem_difference_estimate} yields
\begin{multline*}
|\la_2^{-2}-\la_1^{-2}| \cdot R_{g_{\Bry}}(x_{\Bry}) \leq |R_{g_3}-R_{g_1}|+ \big| |d(f\circ\phi)|_{g_3}^2- |d f|_{g_1}^2 \big| \\
\leq  b \cdot C\al^{-C}D^{-E+4}.
\end{multline*}
So the bound on $|\lambda_2 - 1|$ follows for large enough $D$, as $\lambda_1 = 1$.
\end{proof}

Next, we prove that $\phi$ nearly preserves the radial distance function $\sigma$, up to an error that decays polynomially in $D$.

\begin{lemma}[$\phi$ nearly preserves $\si$]
\label{lem_phi_nearly_preserves_si}
If
\begin{equation} \label{eq_parameters_radial_lemma}
 E \geq \underline{E}, \qquad D \geq \underline{D} (\alpha), 
\end{equation}
then we have for $k = 0, 1, 2$
\begin{equation} \label{eq_sigma_difference_lemma}
\big|{\nabla_{g_1}^k(\si\circ\phi-\si)}\big|_{g_1}\leq b \cdot C\al^{-C}D^{-E+C}.
\end{equation}
\end{lemma}

\begin{proof}
Let $F : (0, \infty) \to (0, \infty)$ be the function with the property that $R = F \circ \sigma$ on $(M_{\Bry}, g_{\Bry})$.
Consider the constant $C_B$ from Lemma~\ref{lem_properties_Bryant}.
By (\ref{eq_R_asymptotic_Bryant}), (\ref{eq_lower_partial_R_bound}) and (\ref{eq_higher_der_bound_Bry}) we have for $s > C_B$
\begin{multline} \label{eq_F_bounds}
 C_{B}^{-1} s^{-1} < F(s ) < C_{B} s^{-1}, \qquad C_{B}^{-1} s^{-2} < - F' (s) < C_{B} , \\ \qquad |F''(s)|, |F'''(s)| < C_{B}. 
\end{multline}
So there is a $c_0 > 0$ such that $F^{-1} ( (0, c_0) ) = (C_B, \infty)$ and such that there is an inverse $H : (0,c_0) \to (C_B, \infty)$ of $F |_{(C_B, \infty)}$.
A straight-forward application of the chain rule gives
\[ |H'(r )| < C_B r^{-2} , \qquad |H''(r )| < C r^{-6} , \qquad |H'''(r)| < C r^{- 10} . \]
(Note that these bounds are not optimal.)

Assume now that $E$ and $D$ have been chosen large enough, in the sense of (\ref{eq_parameters_radial_lemma}), that $\frac12 < \lambda_2 < 2$ by Lemma~\ref{lem_scale_detection_Bry} and that by (\ref{eq_F_bounds}) and Lemma~\ref{lem_difference_estimate} we have for $i = 1,3$
\[ (10 C_{B})^{-1} D^{-1} < R_{g_i} <  10 C_{B} D^{-1} < c_0 /10  \]
on $M_{\Bry} (\frac12 D, D)$.
Then on $M_{\Bry} (\frac12 D, D)$
\begin{equation} \label{eq_P_equation}
  \sigma \circ \phi  - \sigma  = P(R_{g_1}, R_{g_{3}} - R_{g_1}, \lambda_2),  
\end{equation}
where
\[ P (r_1, r_2, \lambda) :=  H ( \lambda^2 (r_1+r_2)) - H (r_1).  \]
Note that $P(r, 0 , 1) = 0$ for all $r \in (0, c_0)$ and that on $((10C_B)^{-1} D^{-1}, \lb 10C_B D^{-1})^2 \lb \times \lb (\frac12, 2)$  we have
\[ |\partial^k  P| \leq  C D^{10} \]
for $k = 1, 2, 3$.
So for $k = 0,1,2$ we have
\begin{equation*}
 |\partial_{1}^k P| (r_1, r_2, \lambda) \leq C D^{10} \big( |r_2| + |\lambda - 1| \big) \quad
\text{on}\quad ((10C_B)^{-1} D^{-1}, \lb 10C_B D^{-1})^2 \lb \times \lb (\tfrac12, 2)
\end{equation*}
So (\ref{eq_sigma_difference_lemma}) follows by differentiating (\ref{eq_P_equation}) and using Lemmas~\ref{lem_difference_estimate} and \ref{lem_scale_detection_Bry}.
\end{proof}

Recall that the Bryant soliton metric is a warped product $g_{\Bry }=d\si^2+w^2g_{S^2}$ on $M_{\Bry}\setminus\{x_{\Bry}\}$ and that $C_B^{-1} \sqrt{s} < w(s) < C_B \sqrt{s}$ for large $s$ (see Lemma~\ref{lem_properties_Bryant} for more details).
Fix some  $D$ that is sufficiently large such that $D > w(D)$.
We now let $g_4=d\si_4^2+w_4^2g_{S^2}$  be a warped product metric on $M_{\Bry} (D - \frac34 w(D) , D + \frac14 w(D))$ with
\[ \sigma_4 =\frac{\sigma - D}{w (D)} \]
and the warping function 
\[ w_4 = w_4 (\sigma)  = 1 + \sigma_4 = 1 + \frac{\sigma - D}{w (D) }. \]
Note that there is an isometry 
\[ \Phi : M_{\Bry} \big( D - \tfrac34 w(D) , D + \tfrac14 w(D) \big) \longrightarrow A_{1/4, 5/4} \subset \R^3 \]
to a  Euclidean annulus such that $1 + \sigma_4 (x) = |\Phi (x) |_{\R^3}$.
So
\[ \Phi \big( M_{\Bry} ( D - \tfrac12 w(D) , D  ) \big) = A_{1/2,1}. \]
Due to Lemma~\ref{lem_phi_nearly_preserves_si} we may assume in the following that $\phi ( M_{\Bry} (D  - \frac12 w(D), \lb D \lb)) \subset M_{\Bry} (D - \frac34 w(D), D + \frac14 w(D))$.
So $\phi$ induces a map
\[ \Phi \circ \phi \circ \Phi^{-1} : A_{1/2,1} \longrightarrow A_{1/4,5/4}. \]

We now show that $\phi$ restricted to $M_{\Bry} ( D - \frac12  w(D), D )$ almost preserves the metric $g_4$ and the  function $\sigma_4$.
This is equivalent to saying that $\Phi \circ \phi \circ \Phi^{-1}$ almost preserves the Euclidean metric and the radial distance function on $\R^3$.

\begin{lemma}
\label{lem_g4_control}
If  
\begin{equation} \label{eq_parameter_lemma_annulus}
 E \geq \underline{E}, \qquad D \geq \underline{D} (\alpha), 
\end{equation}
then for $k = 0, 1$
\begin{align}
\big|\nabla_{g_4}^k(\si_4\circ\phi-\si_4)\big|_{g_4}&\leq  b \cdot C\al^{-C}D^{-E+C}, \label{eq_phi_sigma4} \\
\big|\nabla_{g_4}^k(\phi^*g_4-g_4) \big|_{g_4}&\leq  b \cdot C\al^{-C}D^{-E+C}, \label{eq_phi_g4} 
\end{align}
on $M_{\Bry} ( D - \frac12 w(D), D )$.
\end{lemma}
\begin{proof}
Let us first consider the rescaled metric $\ov{g}_1 := w^{-2} (D) g_1$.
This metric is a warped product of the form
\[ \ov{g}_1 = d\sigma_4^2 + \ov{w}^2 g_{S^2}, \]
where
\[ \ov{w} = \frac{w}{w(D)}. \]
Note that for large $D$ the metric  $\ov g_1$ on $M_{\Bry} (D- \frac34 w(D), D + \frac14 w(D))$ is geometrically close to $S^2 \times (-\frac34,\frac14)$ equipped with the standard cylindrical metric.
More precisely, if we express $\ov{w} = \ov{w} (\sigma_4)$ as a function in $\sigma_4$, then by (\ref{eq_td_w_bounds}) in Lemma~\ref{lem_properties_Bryant} we have the following bounds when $\sigma_4 \in (-\frac34, \frac14)$
\begin{equation} \label{eq_w_5_bounds}
 | \ov{w} - 1| , \quad \bigg| \frac{d \ov{w}}{d\sigma_4} \bigg| , \quad \bigg| \frac{d^2 \ov{w}}{d\sigma^2_4} \bigg| \leq  C D^{-1/2}. 
\end{equation}

Let us now consider the map $\phi$.
We have
$$
\phi^*\ov{g}_1- \ov{g}_1= \lambda_2^{-2} w^{-2} (D) \big(\phi^*g_2-g_1+(1-\lambda_2^2)g_1 \big),
$$
Combining this with the scale detection Lemma~\ref{lem_scale_detection_Bry} gives us the following bound for $k = 0, 1$, assuming an estimate of the form (\ref{eq_parameter_lemma_annulus}):
\begin{equation}
\label{eqn_g5_g1}
\big| \nabla_{\ov{g}_1}^k \big( \phi^*\ov{g}_1- \ov{g}_1 \big) \big|_{\ov{g}_1}\leq  b \cdot C\al^{-C} D^{-E+C}.
\end{equation}
Note that here we have taken the covariant derivative with respect to $\ov{g}_1$, as opposed to $g_1$.
This change produces a factor of the order of $O(D^{k/2})$, which can be absorbed in the right-hand side.
Similarly, by rescaling (\ref{eq_sigma_difference_lemma}) in Lemma \ref{lem_phi_nearly_preserves_si} and assuming an estimate of the form (\ref{eq_parameter_lemma_annulus}), we obtain that for $k = 0,1,2$
\begin{equation}
\label{eqn_g5_si5}
\big| \nabla_{\ov{g}_1}^k(\si_4\circ\phi-\si_4) \big|_{\ov{g}_1}\leq b \cdot C\al^{-C} D^{-E+C} .
\end{equation}
This implies (\ref{eq_phi_sigma4}) for $k = 0$ immediately and for $k = 1$ after observing that $\ov{g}_1$ and $g_4$ are uniformly bilipschitz for large $D$.

So it remains to show (\ref{eq_phi_g4}).
The bound (\ref{eqn_g5_si5}) implies that for $k = 0, 1$
\begin{equation}
\label{eqn_phi_dsi5}
\big| \nabla_{\ov{g}_1}^k \big( \phi^*d\si_4^2-d\si_4^2 \big) \big|_{\ov{g}_1}\leq b \cdot C\al^{-C}D^{-E+C}.
\end{equation}
Combining (\ref{eqn_g5_g1})  and (\ref{eqn_phi_dsi5}), one gets
\begin{multline} \label{eq_nab_ovw_gS2}
 \big| \nabla_{\ov{g}_1}^k \big( \phi^*(\ov{w}^2g_{S^2})- \ov{w}^2g_{S^2} \big) \big|_{\ov{g}_1} = \big|\nabla_{\ov{g}_1}^k \big( (\phi^*\ov{g}_1-\ov{g}_1) - (\phi^*d\si_4^2-d\si_4^2) \big) \big|_{\ov{g}_1} \\
\leq b \cdot C\al^{-C}D^{-E+C}\,.
\end{multline}
Set now
\[ \chi := \frac{w_4^2}{\ov{w}^2} = \bigg( \frac{1  +  \sigma_4}{\ov{w}} \bigg)^2. 
\]
Let us first express $\chi (\sigma_4)$ as a function of $\sigma_4$.
Then by (\ref{eq_w_5_bounds}) we have for $k = 0,1,2$, as long as $- \frac34 < \sigma_4 < \frac14$,
\begin{equation} \label{eq_chi_bounds}
 | \chi  | , \quad \bigg| \frac{d \chi}{d\sigma_4} \bigg|,  \quad  \bigg| \frac{d^2 \chi}{d\sigma_4^2} \bigg|  \leq C. 
\end{equation}
It follows using (\ref{eqn_g5_si5}) that
\begin{align}
 \big| \chi \circ \sigma_4 \circ \phi  - \chi \circ \sigma_4 \big| &\leq  C D^{-1/2} | \sigma_4 \circ \phi - \sigma_4 | \leq  b \cdot C \alpha^{-C} D^{-E+C}, \label{eq_chi_phi_distortion} \\
 \big| \nabla_{\ov{g}_1} \big( \chi \circ \sigma_4 \circ \phi  - \chi \circ \sigma_4 \big) \big|_{\ov{g}_1} &\leq \big| (\chi' \circ \sigma_4 \circ \phi ) \, \phi^*d\sigma_4 -  ( \chi' \circ  \sigma_4 ) d\sigma_4  \big|_{\ov{g}_1} \notag \\
  &\leq \big| \chi' \circ \sigma_4 \circ \phi  \big| \cdot \big| \phi^* d\sigma_4 - d\sigma_4 \big|_{\ov{g}_1} \notag \\
&\qquad + \big| ( \chi' \circ \sigma_4 \circ \phi ) - \chi' (\sigma_4) \big|  \cdot |d\sigma_4|_{\ov{g}_1} \notag  \\
&\leq  b \cdot C \alpha^{-C} D^{-E+C}. \notag
\end{align}
So, assuming a bound of the form (\ref{eq_parameter_lemma_annulus}), we get using (\ref{eq_nab_ovw_gS2}) that for $k = 0,1$
\begin{align*}
\big| \nabla_{\ov{g}_1}^k &(\phi^*(w_4^2g_{S^2})-w_4^2g_{S^2}) |_{\ov{g}_1} \\
&= |\nabla_{\ov{g}_1 }^k(\phi^*((\chi \circ \sigma_4) \ov{w}^2g_{S^2})-(\chi \circ \sigma_4) \ov{w}^2g_{S^2}) |_{\ov{g}_1 } \\
&\leq \big| \nabla_{\ov{g}_1 }^k \big( ( \chi \circ \sigma_4 \circ \phi - \chi \circ \sigma_4 ) \phi^* (\ov{w}^2 g_{S^2}) \big) \big|_{\ov{g}_1 } \\
&\qquad + \big| \nabla_{\ov{g}_1 }^k \big( (\chi \circ \sigma_4) \big(  \phi^* (\ov{w}^2 g_{S^2}) - \ov{w}^2  g_{S^2} \big) \big) \big|_{\ov{g}_1 } 
\leq b \cdot C\al^{-C}D^{-E+C}.
\end{align*}
Combining this again with (\ref{eqn_phi_dsi5}) gives us that for $k = 0, 1$
\[ \big|\nabla_{\ov{g}_1}^k(\phi^*g_4-g_4) \big|_{\ov{g}_1} \leq b \cdot C\al^{-C}D^{-E+C} \]
This implies (\ref{eq_phi_g4}) for $k = 0$, as $g_4$ and $\ov{g}_1$ are uniformly bilipschitz for large $D$.
To see (\ref{eq_phi_g4}) for $k = 1$ note that due to (\ref{eq_chi_bounds}) we have $|\nabla_{g_4} - \nabla_{\ov{g}_1} |_{\ov{g}_1} \leq C$. 
\end{proof}

In the following lemma we extend the map $\Phi \circ \phi \circ \Phi^{-1} : A_{1/2, 1} \to A_{1/4, 5/4}$ to a map $\wh\phi$ on the unit ball $B_1 \subset \R^3$.

\begin{lemma}
\label{lem_phi4}
If  
\begin{equation} \label{eq_isometry_lemma_annulus}
 E \geq \underline{E}, \qquad D \geq \underline{D} (\alpha), 
\end{equation}
then there is a diffeomorphism onto its image $\wh\phi : B_1 \ra \R^3$ such that:
\begin{enumerate}[label=(\alph*)]
\item \label{ass_10.28_a} $|\wh\phi^*g_4-g_4 |_{g_4}\leq b \cdot C\al^{-C}D^{-E+C}$.
\item \label{ass_10.28_b} $|\sigma_4 \circ \wh\phi - \sigma_4|$, $|\wh\phi^*d\sigma_4 - d\sigma_4 |_{g_4} \leq  b \cdot C\al^{-C}D^{-E+C}$.
\item \label{ass_10.28_c} $\wh\phi \equiv \Phi \circ \phi \circ \Phi^{-1}$ on $A_{5/6, 1}$.
\item \label{ass_10.28_d} $\wh\phi \equiv \psi$ on $B_{4/6}$ for an orthogonal map $\psi \in O(3)$ of $\R^3$.
\end{enumerate}
\end{lemma}
\begin{proof}
By Lemma~\ref{lem_g4_control} and the fact that $g_4 \equiv \Phi^*g_{\R^3}$ we have for $k = 0,1$
\begin{align}
\big| \nabla_{\R^3}^k \big( (\Phi \circ \phi \circ \Phi^{-1})^*g_{\R^3}-g_{\R^3} \big) \big|_{\R^3}&\leq b \cdot C\al^{-C}D^{-E+C}, \label{eq_phi_R3_almost_isometry} \\
\big| \nabla_{\R^3}^k \big( r\circ(\Phi \circ \phi \circ \Phi^{-1})-r \big) \big|_{\R^3}&\leq b \cdot C\al^{-C}D^{-E+C}, \label{eq_phi_almost_rotation}
\end{align}
where $r (x) := | x|_{\R^3}$ denotes the radial distance function on $\R^3$.

From now on we will only work on $\R^{3}$.
To simplify notation, we will write $\phi$ instead of $\Phi \circ \phi \circ \Phi^{-1}$.
Expressing (\ref{eq_phi_R3_almost_isometry}) for $k = 1$ in Euclidean coordinates yields
\[ \bigg| \sum_{s=1}^3 \bigg( \frac{\partial^2 \phi^s}{\partial x^k \partial x^i} \, \frac{\partial \phi^s}{\partial x^j} + \frac{\partial^2 \phi^s}{\partial x^k \partial x^j} \,\frac{\partial \phi^s}{\partial x^i}  \bigg) \bigg|\leq b \cdot C\al^{-C}D^{-E+C}. \]
Permuting the indices $i, j, k$ cyclicly and using
\begin{multline*}
 2 \frac{\partial^2 \phi^s}{\partial x^i \partial x^j} \, \frac{\partial \phi^s}{\partial x^k} = \bigg( \frac{\partial^2 \phi^s}{\partial x^i \partial x^j} \, \frac{\partial \phi^s}{\partial x^k} + \frac{\partial^2 \phi^s}{\partial x^i \partial x^k} \,\frac{\partial \phi^s}{\partial x^j} \bigg) \\ 
 + 
\bigg( \frac{\partial^2 \phi^s}{\partial x^j \partial x^k} \, \frac{\partial \phi^s}{\partial x^i} + \frac{\partial^2 \phi^s}{\partial x^j \partial x^i} \,\frac{\partial \phi^s}{\partial x^k} \bigg) -
\bigg( \frac{\partial^2 \phi^s}{\partial x^k \partial x^i} \, \frac{\partial \phi^s}{\partial x^j} + \frac{\partial^2 \phi^s}{\partial x^k \partial x^j} \,\frac{\partial \phi^s}{\partial x^i} \bigg) 
\end{multline*}
gives us 
\[ \bigg| \sum_{s=1}^3  \frac{\partial^2 \phi^s}{\partial x^{i } \partial x^{j }} \, \frac{\partial \phi^s}{\partial x^{k }}  \bigg|\leq b \cdot C\al^{-C}D^{-E+C}. \]
Combining this with (\ref{eq_phi_R3_almost_isometry}) for $k = 0$ implies that under a condition of the form (\ref{eq_isometry_lemma_annulus})
\begin{equation} \label{eq_d2_bound_phi}
|d^2 \phi|_{\R^3}\leq b \cdot C\al^{-C}D^{-E+C}.
\end{equation}
Let now $x_0 \in A_{1/2, 1}$ be a point and consider the differential $(d \phi )_{x_0} : \R^3 \to \R^3$.
By (\ref{eq_phi_R3_almost_isometry}) there is a Euclidean isometry $\psi' : \R^3 \to \R^3$ with $\psi' (x_0) = \phi (x_0)$ and
\[ \big| (d\psi' )_{x_0} - (d \phi)_{x_0} \big|_{\R^3} \leq  b \cdot C \alpha^{-C} D^{-E+C}. \]
Combining this with (\ref{eq_phi_almost_rotation}) gives us
\[ \big| (d\psi')_{x_0} ((\nabla r)_{x_0}) - (\nabla r)_{\psi'(x_0)} \big|_{\R^3} \leq  b \cdot C \alpha^{-C} D^{-E+C}. \]
So again by (\ref{eq_phi_almost_rotation}) we have
\begin{align*}
 |\psi'(0) |_{\R^3} 
&= \big| \psi' (x_0) - |x_0|_{\R^3} (d\psi')_{x_0} \big( (\nabla r)_{x_0} \big) \big|_{\R^3}  \\
&\leq \big| \phi (x_0) - |x_0|_{\R^3} (\nabla r)_{\phi(x_0)} \big|_{\R^3} + b \cdot C \alpha^{-C} D^{-E+C} \\
&\leq \big| \phi (x_0) - |\phi(x_0)|_{\R^3} (\nabla r)_{\phi(x_0)} \big|_{\R^3} + b \cdot C \alpha^{-C} D^{-E+C} \\
&= b \cdot C \alpha^{-C} D^{-E+C} .
\end{align*}
Set now $\psi := \psi' - \psi' (0)$.
Then $\psi \in O(3)$ and for $k = 0,1$
\begin{equation*}
\big| d \psi (x_0)-d \phi (x_0)  \big|_{\R^3} \leq b \cdot C\al^{-C}D^{-E+C} 
\end{equation*}

Integrating (\ref{eq_d2_bound_phi}) along paths in $A_{1/2,1}$ starting from $x_0$ implies that under an assumption of the form (\ref{eq_isometry_lemma_annulus}) we have for all $x \in A_{1/2,1}$
\begin{multline}
\label{eqn_phi1_phi2_1}
\big| d \psi (x)-d \phi(x) \big|_{\R^3}
\leq \big| (d \psi )(x_0 )-d \phi)(x_0 ) \big|_{\R^3} \\ +10\sup_{A_{1/2,1}} \big|d^2 \psi -d^2\phi \big|_{\R^3}
\leq b \cdot C\al^{-C}D^{-E+C}.
\end{multline}
Integrating this bound once again along paths in $A_{1/2,1}$ yields that
\begin{multline}
\label{eqn_phi1_phi2_3}
|\psi (x)-\phi(x) |_{\R^3} 
\leq |\psi (x_0)-\phi(x_0)|_{\R^3} \\
+ 10 \sup_{A_{1/2,1}} \big| d \psi (x)-d \phi(x) \big|_{\R^3} 
 \leq b \cdot C\al^{-C}D^{-E+C}.
\end{multline}

We now let $\{\zeta_1,\zeta_2\}$ be a partition of unity on $A_{1/2,1}$ such that
$\zeta_1\equiv 1$ on  $A_{5/6, 1}$,  $\zeta_2\equiv 1$ in $B_{4/6}$, and $|\nabla_{\R^3} \zeta_i |_{\R^3}\leq C$.  
Let $\wh\phi := \zeta_1\phi+\zeta_2\psi$.  
Then assertions \ref{ass_10.28_c} and \ref{ass_10.28_d} hold immediately and assertion \ref{ass_10.28_a} follows from (\ref{eqn_phi1_phi2_1}) and (\ref{eqn_phi1_phi2_3}).
Assertion \ref{ass_10.28_b} follows from (\ref{eq_phi_almost_rotation}), the fact that $dr = d\sigma_4$ and that $r \circ \psi = r$.
\end{proof}

\begin{proof}[Proof of Proposition \ref{prop_bryant_comparison}]
We only need to translate the result of Lem\-ma~\ref{lem_phi4} back to $M_{\Bry}$.
By assertion \ref{ass_10.28_d} of Lemma~\ref{lem_phi4} and the fact that $M_{\Bry}$ is rotationally symmetric we can find an isometry $\td\psi : M_{\Bry} \to M_{\Bry}$ with $\td\psi (x_{\Bry}) = x_{\Bry}$ and $\td\psi = \Phi^{-1} \circ \psi \circ \Phi $.
Set $\td\phi := \Phi^{-1} \circ \wh\phi \circ \Phi$ on $M_{\Bry} (D - \frac12 w(D), D) $ and $\td\phi := \psi $ on the closure of $M_{\Bry} (D - \frac12 w(D))$.
By assertion \ref{ass_10.28_d} of Lemma~\ref{lem_phi4}, we know that $\td\phi$ is smooth.  
By assertion \ref{ass_10.28_b} the map $\td\phi$ is injective if $E \geq \underline{E}, D \geq \underline{D} (\alpha)$. 
So  it remains to bound $\td\phi^* g_2 - g_1$ on $M_{\Bry} (D - \frac12 w(D), D)$.
To do this, we first deal with the rescaling factor $\lambda_2$ using Lemma~\ref{lem_scale_detection_Bry}
\begin{multline*}
 \big| \td\phi^* g_2 - g_1 \big|_{g_1} 
 \leq  \big| \td\phi^* g_1 - g_1 \big|_{g_1} +  \big| (\lambda_2^2 - 1) \td\phi^* g_{1 } \big|_{g_1} \\
 \leq   \big| \td\phi^* g_1 - g_1 \big|_{g_1}  + b \cdot C \alpha^{-C} D^{-E+C}. 
\end{multline*}
So it remains to bound $\td\phi^* g_1 - g_1$.
For this purpose consider the rescaled metric $\ov{g}_1 = w^{-2} (D) g_1 = d\sigma_4^2 + \ov{w}^2 g_{S^2}$, as used in the proof of Lemma~\ref{lem_g4_control}, and observe that
\[ \ov{g}_1 = d \sigma_4^2   + \ov{w}^2 g_{S^2} = \frac{\ov{w}^2}{w^2_4} \big( d \sigma_4^2   + w_4^2 g_{S^2} \big) + \bigg( 1- \frac{\ov{w}^2}{w_4^2}  \bigg) d \sigma_4^2 . \]
Set $\chi \circ \sigma_4 := \frac{\ov{w}^2}{w_4^2}$ as in the proof of Lemma~\ref{lem_g4_control}.
As explained in this proof, we obtain using (\ref{eq_chi_phi_distortion}) and assertions \ref{ass_10.28_a} and \ref{ass_10.28_b} of Lemma~\ref{lem_phi4} that under an assumption of the form $E \geq \underline{E}$, $D \geq \underline{D} (\alpha)$
\begin{align*}
 \big| \td\phi^*  g_1 - g_1 \big|_{g_1} 
 &= \big| \td\phi^* \ov{g}_1 - \ov{g}_1 \big|_{\ov{g}_1}  \\
 &\leq (\chi \circ \sigma_4) \big| \td\phi^* g_4 - g_4 \big|_{\ov{g}_1} + \big| \chi \circ \sigma_4 \circ \td\phi - \chi \circ \sigma_4 \big| \cdot | \td\phi^* g_4|_{\ov{g}_1}  \\
 &\qquad +  \big| 1 - (\chi \circ \sigma_4 ) \big| \cdot \big| \td\phi^*  d \sigma_4^2 -  d \sigma_4^2 \big|_{\ov{g}_1} \\
 &\qquad + \big| \chi \circ \sigma_4 \circ \td\phi - \chi \circ \sigma_4 \big| \cdot | \td\phi^*  d \sigma_4^2 |_{\ov{g}_1}  \\
 & \leq  b \cdot C \alpha^{-C} D^{-E+C}. 
\end{align*}
This concludes the proof.
\end{proof}

\begin{proof}[Proof that Proposition~\ref{prop_bryant_comparison} implies Proposition~\ref{prop_bryant_comparison_general}]
Set $\lambda_2 :=  \lambda$, $g_1 \lb :=  \lb g_{\Bry}$ and $g_2 := \lambda_2^2 g_{\Bry}$.  
Assuming $\delta \leq \ov\delta (D)$, we have, using (\ref{eq_R_asymptotic_Bryant})
\begin{equation} \label{eq_rho_g_g_Bry}
 \rho_g \geq \tfrac12 \rho_{g_{\Bry}} \geq \tfrac1{4} C_B^{-1} D^{1/2} \qquad \text{on} \qquad M_{\Bry} (\tfrac12 D, D). 
\end{equation}
Now consider the map $\phi$ from Proposition~\ref{prop_bryant_comparison_general} and note that by the assumptions of this proposition and (\ref{eq_rho_g_g_Bry}) we have for $m = 0, \ldots, 4$
\begin{equation} \label{eq_phigp_g}
 \big| \nabla^m_g (\phi^* g' - g) \big|_g \leq   b \rho_g^{-E} \leq  4^{E} C_B^E \cdot  b D^{-E}. 
\end{equation}
We now claim that for $D \geq \underline{D} (E, C)$ and $\delta \leq \ov\delta (E, C, D, b)$ we have
\begin{equation} \label{eq_g2_g1_g1}
 \big| \nabla^m_{g_1} ( \phi^* g_2 - g_1 ) \big|_{g_1} \leq  C'_{1} (E) b D^{-E} 
\end{equation}
for all $m =0, \ldots, 4$ and some  constant $C'_1 = C'_1 (E) < \infty$.
To see this, assume first that $D \geq  \underline{D} (E, C)$ and $\delta \leq \ov\delta$ such that the pairs of metrics $\{g, \phi^* g'\}$, $\{g, g_1\}$, and $\{g', g_2\}$ are each 2-bilipschitz with respect to one another.
So $g_1, g, \phi^* g', \phi^* g_2$ are pairwise 8-bilipschitz.
As $\lambda \in (C^{-1}, C)$, we can find a constant $C'_2 = (C)  < \infty$ such that by (\ref{eq_Bryant_closeness_Bryant_comparison}) we have for all $m = 0, \ldots, 4$
\begin{align}
 \big| \nabla^m_{g_{1}} (g - g_1) \big|_{g_1} &\leq  C'_2 \delta,  \label{eq_g_g1_g1} \\
\big| \nabla^m_{\phi^* g_2} ( \phi^* g' - \phi^* g_2 ) \big|_{g_1} & \leq 8 \big| \nabla^m_{g_2} ( g' - g_2 ) \big|_{g_2} \circ \phi \leq  C'_2 \delta.  \label{eq_phigp_phig2}
\end{align}
We now argue similarly as in the proof of Lemma~\ref{lem_difference_estimate}.
The tensor $\nabla^m_{g_1} (\phi^* g_2 - g_1)$ can be written as an algebraic expression in terms of the tensors $g^{-1}, (\phi^* g_2)^{-1}, \nabla^{m'}_{g} ( \phi^* g' - g )$, $\nabla^{m'}_{g_1} (g - g_1)$ and $\nabla^{m'}_{\phi^* g_2} (\phi^* g' - \phi^* g_2)$, $m' \leq m$ (where we use $g_1$ as a background metric).
So, pointwise,
\begin{multline*}
 \nabla^m_{g_1} (\phi^* g_2 - g_1) = F \big(  \phi^* g' - g, \ldots, \nabla^{m}_{g} ( \phi^* g' - g ), \\
  g - g_1, \ldots, \nabla^{m}_{g_1} (g - g_1), 
\phi^* g' - \phi^* g_2, \ldots, \nabla^{m}_{\phi^* g_2} (\phi^* g' - \phi^* g_2) \big). 
\end{multline*}
for some smooth, tensor-valued function $F$.
By (\ref{eq_phigp_g}), (\ref{eq_g_g1_g1}) and (\ref{eq_phigp_phig2}), we therefore obtain (\ref{eq_g2_g1_g1}) as long as $2^E C_B^E b D^{-E}$ and $C'' \delta$ are sufficiently small.

So the conditions of Proposition~\ref{prop_bryant_comparison} are fulfilled for $\alpha  =\alpha (E, C):=  \min \{ C^{\prime -1}_1, \lb C^{-1} \} $.
Therefore, if 
$$
E \geq \underline{E}\,,\qquad D \geq \underline{D} (\alpha)\,,
$$
we obtain a diffeomorphism onto its image $\td\phi : M_{\Bry} (D) \to M_{\Bry}$ such that $\td\phi = \phi$ on $M_{\Bry} (D-1, D)$ and moreover there is a universal constant $C'_3 < \infty$ such that   
\[ \big| \td\phi^* g_2 - g_1 \big|_{g_1} \leq  b \cdot C'_3  \alpha^{-C'_{3}} D^{-E+C'_{3}}. \]
If $\delta \leq \ov{\delta} (E, b, \alpha, D)$, then we can assume that the metrics $g, g_1, \phi^* g', \phi^* g_2$ are pairwise sufficiently bilipschitz close to another such that we still have for some universal $C'_4 < \infty$
\[ \big| \td\phi^* g' - g \big|_{g } \leq b \cdot C'_4 \alpha^{-C'_4} D^{-E+C'_4}. \]
 By (\ref{eq_R_asymptotic_Bryant}) we have $C'_4 \alpha^{-C'_4} D^{-E+C'_4} \leq  \beta (10 C_{B} )^{-3/2} D^{-3/2} \leq \beta \rho_g^{-3}$ on $M_{\Bry} (\frac12 D, D)$, as long as 
$$
E \geq C'_4 + 4\,,\qquad D \geq \underline{D} (E, \alpha, \beta)\,,
\qquad
\delta \leq \ov\delta  (E, \alpha, \beta, D).
$$
This implies assumption \ref{ass_10.1_b} of Proposition~\ref{prop_bryant_comparison_general}.
Lastly, note that if 
$$
E \geq \underline{E}\,,\qquad D \geq \underline{D} (\alpha)\,,
$$ 
then $\td\phi$ is an immersion.
So since $$\td\phi ( M_{\Bry} (D-1, D) ) = \phi ( M_{\Bry} (D-1, D) ) \subset M_{\Bry} (D')$$ the image of $\td\phi$ must be contained in $M_{\Bry}(D')$ as well. 
\end{proof}

\section{Inductive step: extension of the comparison domain}
\label{sec_inductive_step_extension_comparison_domain}

\subsection{Statement of the main result} \label{subsec_Setup_comp_domain} 
Consider two Ricci flow spacetimes $\M$ and $\M'$.
The goal of this section is to extend a comparison domain $\N$ in $\M$ that is defined over a time-interval of the form $[0, t_J]$ by one time-step, to a comparison domain that is defined over the time-interval $[0, t_{J+1} = t_J + r_{\comp}^2]$.
In order to carry out this construction, we will assume the existence of a comparison from $\M$ to $\M'$ defined on $\N$ that together with $\N$ satisfies a priori assumptions \ref{item_time_step_r_comp_1}--\ref{item_eta_less_than_eta_lin_13} for some tuple of parameters. 
Assuming that these parameters are chosen appropriately, we will show that the extended comparison domain and the given comparison satisfy the same a priori assumptions for the same tuple of parameters.

The precise statement of the main result of this section is the following.
We remind the reader that we are using the notation for expressing parameter bounds explained in Section~\ref{sec_notation_terminology}.

\begin{proposition}[Extending the comparison domain] \label{prop_comp_domain_construction}
Suppose that
\begin{equation}
\begin{gathered}
\label{eqn_prop_comp_domain_construction_parameter_inequalities}
 \eta_{\lin} \leq \ov\eta_{\lin}, \qquad 
 \delta_{\nn} \leq \ov\delta_{\nn}, \qquad 
 \lambda \leq \ov\lambda (\delta_{\nn}), \qquad
 D_{\CAP} \geq \underline{D}_{\CAP} (\lambda),\\ 
\Lambda \geq \underline{\Lambda} (\delta_{\nn}, \lambda) , \qquad 
 \delta_{\bb} \leq \ov\delta_{\bb} ( \lambda,  \Lambda),\\ \qquad 
 \eps_{\can} \leq \ov\eps_{\can} (\delta_{\nn}, \lambda, \Lambda, \delta_{\bb}), \qquad 
 r_{\comp} \leq \ov{r}_{\comp} ( \lambda,  \Lambda)
\end{gathered}
\end{equation}
and assume that
\begin{enumerate}[label=(\roman*), leftmargin=* ]
\item \label{con_11.1_i} $\M, \M'$ are two $(\eps_{\can} r_{\comp}, T)$-complete Ricci flow spacetimes  that each satisfy the $\eps_{\can}$-canonical neighborhood assumption at scales \lb $(\eps_{\can} r_{\comp}, \lb 1)$.
\item \label{con_11.1_ii} $(\N, \{ \N^j  \}_{j = 1}^J, \{ t_j \}_{j=0}^J )$ is a comparison domain in $\M$ that is defined on the time-interval $[0, t_J]$.
We allow the case $J = 0$, in which this comparison domain is empty (see Definition~\ref{def_comparison_domain}).
\item \label{con_11.1_iii} $(\Cut, \phi, \{ \phi^j \}_{j = 1}^J )$ is a comparison from $\M$ to $\M'$ defined on  $(\N, \lb \{ \N^j  \}_{j = 1}^J, \lb \{ t_j \}_{j=0}^J )$ over the (same) time-interval $[0,t_J]$.  In the case $J=0$, this comparison is the trivial comparison (see the remark after Definition~\ref{def_comparison}).
\item \label{con_11.1_iv} $(\N, \{ \N^j  \}_{j = 1}^J, \{ t_j \}_{j=0}^J )$ and $(\Cut, \phi, \{ \phi^j \}_{j = 1}^J ))$ satisfy a priori assumptions \ref{item_time_step_r_comp_1}--\ref{item_eta_less_than_eta_lin_13} for the parameters $(\eta_{\lin}, \lb \delta_{\nn}, \lb \lambda, \lb D_{\CAP}, \lb \Lambda, \lb  \delta_{\bb}, \lb \eps_{\can}, \lb r_{\comp})$.
\item \label{con_11.1_v} $t_{J+1} := t_J + r_{\comp}^2 \leq T$.
\end{enumerate}

Then there is a subset $\N^{J+1} \subset \M_{[t_J, t_{J+1}]}$ such that $( \N \cup \N^{J+1},\lb \{ \N^j \}_{j = 1}^{J+1}, \lb \{ t_j \}_{j=0}^{J+1})$ is a comparison domain defined on the time-interval $[0, t_{J+1}]$ and such that $( \N \cup \N^{J+1}, \{ \N^j \}_{j = 1}^{J+1}, \{ t_j \}_{j=0}^{J+1})$ and $(\Cut, \lb \phi, \lb \{ \phi^j \}_{j = 1}^J )$ satisfy the a priori assumptions \ref{item_time_step_r_comp_1}--\ref{item_eta_less_than_eta_lin_13} for the same parameters $(\eta_{\lin}, \lb \delta_{\nn}, \lb \lambda, \lb D_{\CAP}, \lb \Lambda, \lb  \delta_{\bb}, \lb \eps_{\can}, \lb r_{\comp})$. 
\end{proposition}

We remind the reader that a priori assumptions \ref{item_time_step_r_comp_1}--\ref{item_eta_less_than_eta_lin_13} allow for the possibility that the comparison $(\Cut, \phi, \{ \phi^j \}_{j = 1}^J )$ is defined on a shorter time-interval than the underlying comparison domain (see Definition~\ref{def_a_priori_assumptions_1_7}).   
In particular,  \ref{item_geometry_cap_extension_5} and \ref{item_eta_less_than_eta_lin_13}  are only required to hold  over the time-interval on which the comparison is defined, which in the context of  Proposition~\ref{prop_comp_domain_construction} is $[0,t_J]$.  

We briefly explain the strategy of the proof of Proposition~\ref{prop_comp_domain_construction}, which will be carried out in the remainder of this section.
In Subsection~\ref{subsec_choosing_Omega}, we will first construct a domain $\Omega \subset \M_{t_{J+1}}$ such that the corresponding product domain $\Omega ([t_J, t_{J+1}]) \subset \M_{[t_J, t_{J+1}]}$ satisfies most of the a priori assumptions \ref{item_time_step_r_comp_1}--\ref{item_eta_less_than_eta_lin_13}.
The final time-slice $\N^{J+1}_{t_{J+1}}$ will later arise from $\Omega$ by adding certain components of its complement $\M_{t_{J+1}} \setminus \Omega$.
This is by far the most delicate part of the proof, because we need to accommodate both a priori assumption \ref{item_backward_time_slice_3}(d),  which forces certain components to be added to $\Omega$, and a priori assumption \ref{item_geometry_cap_extension_5}, which imposes strong restrictions whenever the addition of such components creates extension caps.
The precise criterion for which components of $\M_{t_{J+1}} \setminus \Omega$ will be added to $\Omega$, will be given in Subsection~\ref{subsec_def_of_N_Jp1} and some of the less problematic a priori assumptions will be verified in Subsection~\ref{subsec_verification_easy}.
The most important and complex step in our proof is Lemma~\ref{lem_APA5} in Subsection~\ref{subsec_construction_concluded}, which effectively states that cap extensions only arise when a priori assumption \ref{item_geometry_cap_extension_5} is satisfied.
For more details, we refer the reader to the explanations given before and after the statement of this lemma.

We make the standing assumption that hypotheses \ref{con_11.1_i}--\ref{con_11.1_v} of Proposition~\ref{prop_comp_domain_construction} hold for the remainder of this section. 
The construction of the domain $\N^{J+1}$ and the verification of its properties will proceed in several stages, with each stage requiring additional inequalities on the parameters.  The inequalities on the parameters imposed in the assumptions of lemmas or in discussions in between lemmas will be retained for the remainder of this section.
So the assertions of these lemmas or the conclusions of these discussions continue to hold until the end of this section.

We remind the reader that, while the dependence on the parameters may seem complex, it essentially suffices to observe that the parameter order, as discussed in Subsection~\ref{subsec_parameter_order}, is respected.
We will continue our practice of introducing parameter bounds in separate displayed equations,  to facilitate verification of the parameter dependences.

\subsection[Choosing an almost minimal domain containing all $\mathbf{\Lambda r_{\comp}}$-thick points]{Choosing an almost minimal domain containing all \texorpdfstring{$\Lambda r_{\comp}$}{{\textbackslash}Lambda r{\_}comp}-thick points} \label{subsec_choosing_Omega}
As a first step toward the construction of $\N^{J+1}$, we will  construct a precursor of its final time-slice $\N^{J+1}_{t_{J+1}}$ --- a subset $\Omega \subset \M_{t_{J+1}}$ bounded by  central $2$-spheres of $\delta_{\nn}$-necks at scale $r_{\comp}$ that contains all $\Lambda r_{\comp}$-thick points. 
The final time-slice $\N^{J+1}_{t_{J+1}}$ of $\N^{J+1}$ will later emerge from $\Omega$ by the addition of certain components of its complement inside $\M_{t_{J+1}}$.

Consider the collection $\mathcal{S}$ of all embedded $2$-spheres $\Sigma \subset \M_{t_{J+1}}$ that occur as central $2$-spheres of $\delta_{\nn}$-necks at scale $r_{\comp}$ in $\M_{t_{J+1}}$.

\begin{lemma}
\label{lem_choosing_s_prime}
We can find a subcollection $\mathcal{S}' \subset \mathcal{S}$ such that 
\begin{enumerate}[label=(\alph*)]
\item \label{ass_11.3_a} $d_{t_{J+1}} ( \Sigma_1, \Sigma_2 ) > 10 r_{\comp}$ for all distinct $\Sigma_1, \Sigma_2 \in \mathcal{S}'$.
\item \label{ass_11.3_b} For every $\Sigma \in \mathcal{S}$ there is an $\Sigma' \in \mathcal{S}'$ such that $d_{t_{J+1}} (\Sigma, \Sigma') < 100 r_{\comp}$.
\end{enumerate}
\end{lemma}

\begin{proof}
Let $\{ x_1, x_2, \ldots \} \subset \M_{t_{J+1}}$ be a countable dense subset.
We can successively construct a sequence of collections $\emptyset = \mathcal{S}'_0 \subset \mathcal{S}'_1 \subset \ldots \subset \mathcal{S}$ by the following algorithm: If $x_i$ is in an $ r_{\comp}$-neighborhood of some $\Sigma \in \mathcal{S}$ with the property that $d_{t_{J+1}} (\Sigma, \Sigma') > 10 r_{\comp}$ for all $\Sigma' \in \mathcal{S}'_{i-1}$, then we set $\mathcal{S}'_i := \mathcal{S}'_{i-1} \cup \{ \Sigma \}$.
Otherwise, we set $\mathcal{S}'_i := \mathcal{S}'_{i-1}$.

Set $\mathcal{S}' := \cup_{i =1}^\infty \mathcal{S}'_{i}$.
Then assertion \ref{ass_11.3_a} holds trivially and for assertion \ref{ass_11.3_b} observe that every $\Sigma \in \mathcal{S}$ is $r_{\comp}$-close to some $x_i$.
If $\mathcal{S}'_i = \mathcal{S}'_{i-1}$, then $d_{t_{J+1}} (\Sigma, \Sigma') \leq 10 r_{\comp}$ for some $\Sigma' \in \mathcal{S}'_{i-1}$ and if $\mathcal{S}'_i = \mathcal{S}'_{i-1} \cup \{ \Sigma' \}$, then $x_i$ is contained in an $r_{\comp}$-neighborhood of $\Sigma'$.
In both cases,  $d_{t_{J+1}} (\Sigma, \Sigma') < 100 r_{\comp}$.
\end{proof}

We now fix the collection $\mathcal{S}'$ for the remainder of this section.

\begin{lemma}
\label{lem_s_separates_thick_thin}
If 
\begin{multline*}
  \delta_{\nn} \leq \ov\delta_{\nn}, \qquad 
  \lambda \leq \ov{\lambda} (\delta_{\nn}), \qquad 
  \Lambda \geq \underline{\Lambda} (\delta_{\nn} ),  \qquad 
 \eps_{\can} \leq \ov\eps_{\can} (\delta_{\nn}), \qquad 
  r_{\comp} < 1,
\end{multline*}
then the collection $\S'$ separates the $100\lambda r_{\comp}$-thin points of $\M_{t_{J+1}}$ from the  $\Lambda r_{\comp}$-thick points.
\end{lemma}

\begin{proof}
Suppose that the assertion of the lemma was false.  
Then there is a continuous path $\gamma:[0,1]\ra \M_{t_{J+1}} \setminus \cup_{\Sigma \in \mathcal{S}'} \Sigma$ such that $\gamma(0)$ is $\Lambda r_{\comp}$-thick and $\gamma(1)$ is $100 \lambda r_{\comp}$-thin.
Without loss of generality, we may assume that $\gamma$ has been chosen almost minimal in the sense that any other such path has length at least $\length_{t_{J+1}} (\gamma)-r_{\comp}$.

We first argue that we may assume in the following that
\begin{equation} \label{eq_gamma_far_from_Sigma}
 d_{t_{J+1}} \big( \gamma ([0,1]), \Sigma' \big) > 1000  r_{\comp} \qquad \text{for all} \qquad \Sigma' \in \mathcal{S}', 
\end{equation}
Assume that $d_{t_{J+1}} (\gamma(s'), \Sigma') \leq 1000  r_{\comp}$ for some $s' \in [0,1]$ and some $\Sigma' \in \mathcal{S}'$.
Let $U \subset \M_{t_{J+1}}$ be a $\delta_{\nn}$-neck at scale $r_{\comp}$ that has $\Sigma'$ as a cross-sectional $2$-sphere.
If
\[ \delta_{\nn} \leq \ov\delta_{\nn}, \]
then $\gamma(s') \in U$.
Moreover, if 
\[ \delta_{\nn} \leq \ov\delta_{\nn}, \qquad
\lambda \leq \ov\lambda, \qquad
\Lambda \geq \underline\Lambda, \]
then no point on $U$ is $\Lambda r_{\comp}$-thick or $100 \lambda r_{\comp}$-thin and therefore $\gamma(0), \gamma(1) \not\in U$.
Let $\Sigma^* \subset U$ be a cross-sectional $2$-sphere of $U$, close to its boundary such that $\Sigma^*, \Sigma'$ bound a domain diffeomorphic to $S^2 \times [0,1]$ inside $U$ that contains $\gamma(s')$.
It follows that $\gamma |_{[0,s']}, \gamma |_{[s',1]}$ intersect $\Sigma^*$.
If
\[ \delta_{\nn} \leq \ov\delta_{\nn}, \]
then the diameter of $\Sigma^*$ is less than $10 r_{\comp}$ and $\Sigma^*$ may be chosen such that the distance between $\Sigma^*$ and $\gamma(s')$ is larger than $10 r_{\comp}$.
This implies that we can replace $\gamma$ by a path whose length is shorter than $\length_{t_{J+1}} (\gamma)-r_{\comp}$, in contradiction to its almost minimality.
Therefore, we may assume in the following that (\ref{eq_gamma_far_from_Sigma}) holds.

By the intermediate value theorem, assuming
\[ \lambda < \tfrac1{10}, \qquad 
\Lambda > 1, \qquad
r_{\comp} < 1, \]
we may pick $s\in [0,1]$ such that $x :=\gamma(s)$ has scale $\rho(x) = r_{\comp}$.
By the construction of $\mathcal{S}$ and (\ref{eq_gamma_far_from_Sigma}), assuming
\[ \delta_{\nn} \leq \ov\delta_{\nn}, \]
the point $x$ cannot be the center of a $\delta_{\nn}$-neck at scale $r_{\comp}$.
So assuming
\[ \eps_{\can} \leq \ov\eps_{\can} (\delta_{\nn}), \qquad
r_{\comp} < 1, \]
we can use Lemma~\ref{lem_neck_or_cap} to find a compact subset $V \subset \M_{t_{J+1}}$ with $x \in V$ that has connected boundary and on which $C_0^{-1} r_{\comp} < \rho < C_0 r_{\comp}$ holds, where $C_0 = C_0 (\delta_{\nn}) < \infty$.
So, assuming
\[ \lambda \leq (100 C_0 (\delta_{\nn}))^{-1}, \qquad
\Lambda \geq C_0 (\delta_{\nn}), \]
we can conclude that $\gamma(0), \gamma(1) \not\in V$.
Therefore, $V$ must have exactly one boundary component and this component is a central 2-sphere of a $\delta_{\nn}$-neck.

We claim that $\partial V$ is disjoint from all elements of $\mathcal{S}'$.
Assume by contradiction that $\partial V$ intersects some $\Sigma' \in \mathcal{S}'$.
If
\begin{equation} \label{eq_delta_nn_bound_V_SS}
 \delta_{\nn} \leq \ov\delta_{\nn}, 
\end{equation}
then we have $\frac12 r_{\comp} < \rho < 2 r_{\comp}$ on $\Sigma' \cap \partial V$.
Again, assuming a bound of the form (\ref{eq_delta_nn_bound_V_SS}), we find that $\partial V$ is a central 2-sphere of a neck at some scale of the interval $(\frac14 r_{\comp}, 4 r_{\comp})$.
So the intersection of $\partial V$ with $\gamma ([0,1])$ is not further than $40 r_{\comp}$ from the intersection with $\Sigma'$, in contradiction to (\ref{eq_gamma_far_from_Sigma}).

Choose now $s_1 \in [0, s)$ and $s_2 \in (s, 1]$ such that $\gamma(s_i) \in \partial V$.
By Lemma~\ref{lem_neck_or_cap} the path $\gamma |_{[s_1,s_2]}$ can be replaced by a continuous path inside $\partial V$ of length less than $\length_{t_{J+1}} (\gamma|_{[s_1,s_2]}) - r_{\comp}$, contradicting the minimality assumption of $\gamma$.
\end{proof}

Now let $\Omega \subset \M_{t_{J+1}}$ be the union of the closures of all components of
\[ \M_{t_{J+1}} \setminus \cup_{\Sigma \in \mathcal{S}'} \Sigma \]
that contain $\Lambda r_{\comp}$-thick points.
Then by the previous lemma, $\Omega$ is weakly $100 \lambda r_{\comp}$-thick.

\begin{lemma} \label{lem_Omega_survives_until_t_J}
If
\[  \eps_{\can} \leq \ov\eps_{\can} (\lambda), \]
then all points in $\Omega$ survive until time $t_J$.
\end{lemma}

\begin{proof}
This follows immediately from Lemma~\ref{lem_lambda_thick_survives_backward} and the fact that $\Omega$ is weakly $100 \lambda r_{\comp}$-thick.
\end{proof}

\begin{lemma}
\label{lem_thick_in_n_tj_minus}
If $J \geq 1$ and
\[
  \delta_{\nn} \leq \ov{\delta}_{\nn}, \qquad  
\Lambda \geq \underline{\Lambda} ,  \qquad 
  \eps_{\can} \leq \ov\eps_{\can}, \qquad
  r_{\comp} \leq \ov{r}_{\comp} (\Lambda), 
\]
then for every $\Lambda r_{\comp}$-thick point $x \in \M_{t_{J+1}}$ we have $x(t_J) \in \Int \N_{t_J -}$.
\end{lemma}

Recall that $x(t_J) \in \M_{t_J}$ denotes the image of $x$ under the time $- (t_{J+1}- t_J)$-flow of the time vector field $\partial_{\mathfrak{t}}$ (see Definition~\ref{def_points_in_RF_spacetimes}).

\begin{proof}
Assume that $x (t_J) \not\in \Int \N_{t_J-}$.
By a priori assumptions \ref{item_backward_time_slice_3}(a), (b), Lemma~\ref{lem_lambda_thick_survives_backward} and assuming that
\[ \delta_{\nn} \leq \ov\delta_{\nn}, \qquad
\Lambda \geq 2, \qquad
 \eps_{\can} \leq \ov\eps_{\can}, \qquad
r_{\comp} \leq \ov{r}_{\comp} (\Lambda), \]
we have $\rho(x) \leq 2 \Lambda r_{\comp}$.

Let $\delta_\# > 0$ be a constant whose value we will determine in the course of the proof.
Assuming
\[ \Lambda \geq 1, \qquad
\eps_{\can}  \leq \ov\eps_{\can} (\delta_\#), \qquad
 r_{\comp} \leq \ov{r}_{\comp} (\Lambda), \]
we can use Lemma~\ref{lem_bryant_propagate} (for $\alpha = 2\Lambda$) to argue that $(\M_{t_J}, x(t_J))$ is $\delta_\#$-close to $(M_{\Bry}, \lb g_{\Bry}, \lb x_{\Bry})$ at scale $\rho (x) > \Lambda r_{\comp}$.
Since $\rho$ is uniformly bounded from below on $(M_{\Bry}, g_{\Bry})$ and diverges at infinity, there is a universal constant $c > 0$ such that for
\[ \delta_\# \leq \ov\delta_\# \]
we can find a path $\gamma : [0,1] \to \M_{t_{J}}$ with $\gamma(0) = x(t_J)$, $\rho (\gamma(1)) > \Lambda r_{\comp}$ and $\rho (\gamma(s)) > c \Lambda r_{\comp}$ for all $s \in [0,1]$.
So by a priori assumption \ref{item_backward_time_slice_3}(b) we have $\gamma(1) \in \N_{t_J-}$.
If
\[ \delta_{\nn} \leq \ov\delta_{\nn}, \qquad
\Lambda \geq \underline{\Lambda}, \]
then by a priori assumption \ref{item_backward_time_slice_3}(a), the image $\gamma ([0,1])$ is disjoint from $\partial \N_{t_J-}$.
It follows that $x(t_J) = \gamma(0) \in \N_{t_J-}$.
\end{proof}

We remark that in the proof of Lemma~\ref{lem_thick_in_n_tj_minus}, the use of Lemma~\ref{lem_bryant_propagate}, which is based on the rigidity theorems of Hamilton and Brendle,  may be replaced by a longer but more elementary argument involving the maximum principle and the geometry of $\kappa$-solutions.

\begin{lemma}
\label{lem_omega_tj_in_n_tj_minus}
If $J \geq 1$ and
\begin{multline*}
 \delta_{\nn} \leq \ov\delta_{\nn},  \qquad 
 \lambda \leq \ov\lambda (\delta_{\nn}), \qquad  
\Lambda \geq \underline\Lambda (\delta_{\nn}),   \qquad 
 \eps_{\can} \leq \ov\eps_{\can} (\delta_{\nn}), \qquad 
 r_{\comp} \leq \ov{r}_{\comp} (\Lambda),
\end{multline*}
then $\Omega (t_J) \subset \Int \N_{t_J-}$.
\end{lemma}

\begin{proof}
Let $\Omega_0$ be the closure of a component of $\M_{t_{J+1}} \setminus \cup_{\Sigma \in \mathcal{S}'} \Sigma$ that contains a $\Lambda r_{\comp}$-thick point $x$.
Note that by definition of $\Omega$ we have $\Omega_0 \subset \Omega$ and the lemma follows if we can show that $\Omega_0 (t_J) \subset \Int \N_{t_j-}$ for all such $\Omega_0$.

Fix $\Omega_0$ and a $\Lambda r_{\comp}$-thick point $x \in \Omega_0$ for the remainder of the proof and assume by contradiction that $\Omega_0 (t_J) \not\subset \Int \N_{t_J-}$.
Suppose by contradiction that there is a point Let $z \in \Omega_0$ with the property that $z (t_J) \not\in \Int \N_{t_J -}$.
Choose a path $\gamma_{J+1} : [0,1] \to \Omega_0$ within $\Omega_0$ such that $x = \gamma_{J+1} (0)$ and $z = \gamma_{J+1} (1)$.
Without loss of generality, we may assume that we have chosen $z$ and $\gamma_{J+1}$ almost minimal in the sense that for any other such choice of $z', \gamma'_{J+1}$ we have
\begin{equation} \label{eq_almost_minimal_property}
 \length_{t_{J+1}} (\gamma'_{J+1}) > \length_{t_{J+1}} (\gamma_{J+1} ) - r_{\comp}. 
\end{equation}

By Lemma~\ref{lem_thick_in_n_tj_minus}, assuming
\[
  \delta_{\nn} \leq \ov{\delta}_{\nn}, \qquad  
\Lambda \geq \underline{\Lambda}, \qquad 
 \eps_{\can} \leq \ov\eps_{\can}, \qquad  
 r_{\comp} \leq \ov{r}_{\comp} (\Lambda), 
\]
we have $x(t_J) \in \N_{t_J-}$.
Denote by $\gamma_{J} : [0,1] \to \M_{t_{J}}$ the curve at time $t_J$ corresponding to $\gamma_{J+1}$ under the $(- r^2_{\comp})$-flow of the time vector field $\partial_{\mathfrak{t}}$, i.e. $\gamma_J (s) = (\gamma_{J+1} (s))(t_J)$.
This path exists due to Lemma~\ref{lem_Omega_survives_until_t_J}.
Since $\gamma_J (1) = z(t_J) \not\in \Int \N_{t_J-}$, we can find a parameter $s_0 \in [0,1]$ such that $\gamma_J (s_0) \in \partial \N_{t_J -}$ and $\gamma_J ([0,s_0)) \subset \Int \N_{t_J-}$.
By truncating $\gamma_J$ and $\gamma_{J+1}$, we may  assume without loss of generality that $s_0 = 1$ and therefore $z (t_{J}) = \gamma_J (1) \in \partial \N_{t_J-}$ and $\gamma_J ( [0,1) ) \subset \Int \N_{t_J-}$.
The almost minimality property (\ref{eq_almost_minimal_property}) of $z$ and $\gamma_{J+1}$ remains preserved under this truncation process.

Let $\Sigma_{J} \subset \partial \N_{t_J-}$ the boundary component that contains $z(t_J)$.
By a priori assumption \ref{item_backward_time_slice_3}(a), $\Sigma_J$ is a central $2$-sphere of a $\delta_{\nn}$-neck at scale $r_{\comp}$ in $\M_{t_J}$.
Let $\delta_\# > 0$ be a constant whose value we will determine later.
By Lemma~\ref{lem_time_slice_neck_implies_space_time_neck}, assuming 
\[ \delta_{\nn} \leq \ov\delta_{\nn} (\delta_\#), \qquad 
\eps_{\can} \leq \ov\eps_{\can} (\delta_\#), \qquad
r_{\comp} \leq \ov{r}_{\comp}, \]
this implies that all points on $\Sigma_J$ survive until time $t_{J} + \frac14 r_{\comp}^2$ and  $\Sigma_J (t_{J} + \frac14 r_{\comp}^2)$ is a central $2$-sphere of a $\delta_{\#}$-neck at scale $\frac12 r_{\comp}$.
So $\rho ( z( t_J + \frac14 r_{\comp}^2) )< 0.6 r_{\comp}$, assuming
\[ \delta_\# \leq \ov\delta_\# . \]
By Lemma~\ref{lem_lambda_thick_survives_backward}, this implies that $\rho (\gamma_{J+1} (1)) = \rho (z) < 0.7 r_{\comp}$, assuming
\[ \eps_{\can} \leq \ov\eps_{\can}, \qquad r_{\comp} < 1. \]

Recall that at the other endpoint of $\gamma_{J+1}$ we have $\rho (\gamma_{J+1} (0)) = \rho (x) > \Lambda r_{\comp}$.
So by the intermediate value theorem, assuming
\[ \Lambda > 1, \]
we can find a parameter $s \in (0,1)$ such that $y := \gamma_{J+1}(s)$ has scale $\rho(y) = r_{\comp}$.

Assuming
\[ \delta_{\nn} \leq \ov\delta_{\nn}, \qquad
\Lambda \geq \underline{\Lambda}, \]
we can conclude that $x, z$ cannot lie in $\delta_{\nn}$-necks at scale $r_{\comp}$ and therefore $d_{t_{J+1}} ( \{ x, z \}, \lb \partial \Omega_0) > 2000 r_{\comp}$.
So by the almost minimal choice of $\gamma_{J+1}$ we find, using the same argument as the one leading to (\ref{eq_gamma_far_from_Sigma}) in the proof of Lemma~\ref{lem_s_separates_thick_thin}, that 
\begin{equation} \label{eq_gamma_far_inside_Omega_0}
 d_{t_{J+1}} \big( \gamma ([0,1]), \partial \Omega_0 \big) > 1000  r_{\comp},
\end{equation}
assuming that
\[ \delta_{\nn} \leq \ov\delta_{\nn}. \]

As the interior of $\Omega_0$ is disjoint from all elements of $\mathcal{S}'$, we can use assertion \ref{ass_11.3_b} of Lemma~\ref{lem_choosing_s_prime} and  (\ref{eq_gamma_far_inside_Omega_0}) to conclude that the point $y$ cannot be a center of a $\delta_{\nn}$-neck at scale $r_{\comp}$, assuming
\[ \delta_{\nn} \leq \ov\delta_{\nn}. \]
We can hence apply Lemma~\ref{lem_neck_or_cap} and find a smooth domain $V \subset \M_{t_{J+1}}$ with $y \in V$.
Moreover, we have $C_0^{-1} (\delta_{\nn}) r_{\comp} <  \rho <  C_0 (\delta_{\nn}) r_{\comp}$ on $V$.
So by Lemma~\ref{lem_lambda_thick_survives_backward} and assuming
\[ \eps_{\can} \leq \ov\eps_{\can} (\delta_{\nn}), \]
all points on $V$ survive until time $t_J$ and
\begin{equation} \label{eq_rho_on_Vt_J}
 \rho \geq \tfrac12 C_0^{-1} (\delta_{\nn} ) r_{\comp} \qquad \text{on} \qquad V(t_J). 
\end{equation}
Also, if
\[ \Lambda > C_0 (\delta_{\nn}), \]
then $x \not\in V$.
In particular, this implies that $\partial V \neq \emptyset$.
By Lemma~\ref{lem_neck_or_cap} the boundary $\partial V$ is a central $2$-sphere of a $\delta_{\nn}$-neck.
Choose $s_1 \in [0,s)$ such that $\gamma_{J+1} (s_1) \in \partial V$.

We claim that
\begin{equation} \label{eq_SigmaJ1_in_V}
  z \in V.
\end{equation}
If not, then we can choose $s_2 \in (s, 1]$ such that $\gamma_{J+1} (s_2) \in \partial V$.
By Lemma~\ref{lem_neck_or_cap} we can connect $\gamma_{J+1} (s_1), \gamma_{J+1} (s_2)$ by a path $\gamma' : [s_1, s_2] \to \partial V \subset \Int \Omega_0$ whose length is less than $\length_{t_{J+1}} (\gamma |_{[s_1,s]}) - 100 r_{\comp}$.
The concatenation of $\gamma_{J+1} |_{[0,s_1]}$, $\gamma'$ and $\gamma_{J+1} |_{[s_2,1]}$ would have length less than $\length_{t_{J+1}} (\gamma_{J+1}) - 100 r_{\comp}$, contradicting the almost minimal choice of $\gamma_{J+1}$ and confirming (\ref{eq_SigmaJ1_in_V}).

Next, we argue that
\begin{equation} \label{eq_partial_V_in_N}
 (\partial V)(t_J) \subset \Int \N_{t_J-}. 
\end{equation}
Note that by our choice of $\gamma_{J+1}$ we have $(\gamma_{J+1} (s_1))(t_J) \in \Int \N_{t_J-}$.
So if (\ref{eq_partial_V_in_N}) was false, then $(\partial V)(t_J) \cap \partial \N_{t_J-} \neq \emptyset$.
Therefore, by Lemma~\ref{lem_neck_or_cap}  we would find a continuous curve $\gamma'' : [s_1,1] \to \partial V$ between $\gamma_{J+1} (s_1)$ and a point $z' \in \partial V$ with $z' (t_J) \in  \partial \N_{t_J-}$ such that $\length_{t_{J+1}} (\gamma'') \leq d_{t_{J+1}} (\gamma_{J+1} (s_1), \gamma_{J+1} (s)) - 100 r_{\comp}$.
The concatenation of $\gamma |_{[0,s_1]}$ with $\gamma''$ would then have length of at most
\begin{multline*}
 \length_{t_{J+1}}(\gamma_{J+1}) - d_{t_{J+1}} (\gamma_{J+1} (s_1), \gamma_{J+1} (s)) + \length_{t_{J+1}} (\gamma'') \\
  \leq \length_{t_{J+1}}(\gamma_{J+1}) - 100 r_{\comp}.
\end{multline*}
This, however, contradicts again the almost minimal choice of $\gamma_{J+1}$, confirming (\ref{eq_partial_V_in_N}).

The inclusion (\ref{eq_SigmaJ1_in_V}) implies that $z (t_J) \in V (t_J)$.
Let $\C^*$ be the component of $\M_{t_J} \setminus \Int \N_{t_J-}$ that is adjacent to $\Sigma_J$.
As $\C^*$ is path-connected and $z(t_J) \in \C^*$, we can conclude, using (\ref{eq_partial_V_in_N}), that $\C^* \subset V$.
By a priori assumption \ref{item_backward_time_slice_3}(d) there must be a $10\lambda r_{\comp}$-thin point in $\C^*$.
So if we choose
\[ \lambda < \tfrac1{20} C_0^{-1} (\delta_{\nn}), \]
then we obtain a contradiction to (\ref{eq_rho_on_Vt_J}).
\end{proof}

\subsection{The definition of \texorpdfstring{$\mathbf{\N^{J+1}}$}{N{\^{}}{\{}J+1{\}}}} \label{subsec_def_of_N_Jp1}
We will now enlarge $\Omega$ to a subset $\Omega^* \subset \M_{t_{J+1}}$ that will become the final time-slice $\N^{J+1}_{t_{J+1}}$ of the product domain $\N^{J+1}$.
The components $Z$ of the difference $\M_{t_{J+1}} \setminus \Int \Omega$ fall into (at least) one of the following four types:

\begin{enumerate}[label=(\Roman*)]
\item  \label{type_11.3_I}
$Z$  has non-empty boundary and all points on $Z$ are weakly $10 \lambda r_{\comp}$-thick (in particular $Z$ is not a closed component of $\M_{t_{J+1}}$).
\item \label{type_11.3_II}
	\begin{enumerate}
	\item $Z$ is diffeomorphic to a $3$-disk.
	\item $Z(t)$ is well-defined and $\lambda r_{\comp}$-thick for all $t\in [t_J,t_{J+1}]$. 
	\item $Z (t_J) \subset \N_{t_J-}$ if $J \geq 1$. 
	\end{enumerate}

\item  \label{type_11.3_III}
	\begin{enumerate}
	\item $Z$ is diffeomorphic to a $3$-disk.
	\item $Z(t)$ is well-defined and $\lambda r_{\comp}$-thick for all $t\in [t_J,t_{J+1}]$. 
	\item  $\C:=Z (t_J)\setminus\Int\N_{t_J-}$ is a component of $\M_{t_J}\setminus\Int\N_{t_J-}$, and there is a component $\C' \subset \M'_{t_J} \setminus \phi ( \Int \N_{t_J-})$ such that a priori assumptions \ref{item_geometry_cap_extension_5}(a)--(e) hold, that is:
		\begin{itemize}
		\item $\C$ and $\C'$ are $3$-disks.
		\item $\partial \C' = \phi_{t_J-} (\partial \C)$.
		\item   There is a point $x\in \C$ such that $(\M_{t_J},x)$ is $\de_{\bb}$-close to the pointed Bryant soliton $(M_{\Bry},g_{\Bry},x_{\Bry})$ at scale $10\lambda r_{\comp}$.
		\item There is a point $x'\in \M'_{t_J}$, at distance $\leq D_{\CAP} r_{\comp}$ from $\C'$, such that $(\M'_{t_j},x')$ is $\de_{\bb}$-close to the pointed Bryant soliton $(M_{\Bry}, \lb g_{\Bry}, \lb x_{\Bry})$  at some scale in the interval $[D_{\CAP}^{-1}r_{\comp},D_{\CAP}r_{\comp}]$.
		\item $\C$ and $\C'$ have diameter $\leq D_{\CAP}r_{\comp}$.
		\end{itemize}

	\end{enumerate}
\item \label{type_11.3_IV} None of the above.
\end{enumerate}

Let $\Omega^*$  be the union of $\Omega$ with all components $Z \subset \M_{t_{J+1}} \setminus \Int{\Omega}$ that are of type \ref{type_11.3_I}, \ref{type_11.3_II} or \ref{type_11.3_III}.   
Assuming
\[ \eps_{\can}\leq \ov\eps_{\can}(\lambda), \]
each component of type \ref{type_11.3_I}--\ref{type_11.3_III} survives until time $t_J$, either by definition or by Lemma \ref{lem_lambda_thick_survives_backward}. 
The subset $\Omega$ survives until time $t_J$ by Lemma~\ref{lem_Omega_survives_until_t_J}. 
Thus we may define $\N^{J+1}$ to be the product domain with final time-slice $\Omega^*$:
\begin{equation}
\label{eqn_n_j_plus_1}
 \N^{J+1} := \cup_{t\in [t_J,t_{J+1}]}\Omega^* (t). 
\end{equation}

To provide some motivation for the choice of $\Omega^*$, we point out that if $\Omega^*\subset\M_{t_{J+1}}$ is a manifold with boundary obtained from  $\Omega$ by adding some components of its complement, and $\N^{J+1}$ is defined by (\ref{eqn_n_j_plus_1}), then one can check that $(\N, \linebreak[1] \{ \N^j \}_{j=1}^{J+1}, \linebreak[1] \{ t_j \}_{j=0}^{J+1})$ and $(\Cut, \linebreak[1] \phi, \linebreak[1] \{ \phi^j \}_{j=1}^J)$  will only satisfy a priori assumptions \ref{item_time_step_r_comp_1}--\ref{item_eta_less_than_eta_lin_13}   if $\Omega^*$ includes all components of type \ref{type_11.3_I}--\ref{type_11.3_III}.  In this sense $\Omega^*$ is the ``minimal'' candidate for an extension of $\Omega$  that yields a comparison (domain) satisfying the a priori assumptions.

In the remainder of this section we will complete the proof of Proposition~\ref{prop_comp_domain_construction} by verifying that $(\N, \linebreak[1] \{ \N^j \}_{j=1}^{J+1}, \linebreak[1] \{ t_j \}_{j=0}^{J+1})$ is a comparison domain, and that $(\N, \linebreak[1] \{ \N^j \}_{j=1}^{J+1}, \linebreak[1] \{ t_j \}_{j=0}^{J+1})$ and $(\Cut, \linebreak[1] \phi, \linebreak[1] \{ \phi^j \}_{j=1}^J)$  satisfy a priori assumptions \ref{item_time_step_r_comp_1}--\ref{item_eta_less_than_eta_lin_13}.   
Most of the verification is straightforward, using the results already established.  
The main difficulty will be establishing the properties of extension caps, especially \ref{item_geometry_cap_extension_5}.  
The crucial fact here, which we will prove in Lemma~\ref{lem_APA5}, is that components $Z$ of type \ref{type_11.3_I} and \ref{type_11.3_II} satisfy $Z(t_J) \subset \N_{t_J-}$.
In other words, extension caps are only caused by components of type \ref{type_11.3_III}, which satisfy a priori assumption \ref{item_geometry_cap_extension_5}.

The main idea of the proof of Lemma~\ref{lem_APA5} will be to show that if $Z(t_J)\not\subset \N_{t_J-}$ for some component $Z$ of type \ref{type_11.3_I}, then a priori assumption \ref{item_geometry_cap_extension_5} would have forced an extension cap to have occurred at some earlier time.
For more details we refer to the reader to the overview preceding the proof of Lemma~\ref{lem_APA5} in Subsection~\ref{subsec_construction_concluded}.

\subsection{Verification of Proposition \ref{prop_comp_domain_construction}, except for Definition \ref{def_comparison_domain}(\ref{item_3_disk}) and \ref{item_geometry_cap_extension_5}. } \label{subsec_verification_easy}
We will now verify that $(\N \cup \N^{J+1}, \lb \{ \N^j \}_{j=1}^{J+1}, \lb \{ t_j \}^{J+1})$ satisfies properties (\ref{pr_7.1_1})--(\ref{pr_7.1_3}) of the definition of a comparison domain (Definition~\ref{def_comparison_domain}) and that $(\N \cup \N^{J+1}, \lb \{ \N^j \}_{j=1}^{J+1}, \lb \{ t_j \}^{J+1})$ and $(\Cut, \lb \phi, \lb \{ \phi^j \}_{j=1}^J )$ satisfy a priori assumptions \ref{item_time_step_r_comp_1}--\ref{item_discards_not_too_thick_4} and \ref{item_eta_less_than_eta_lin_13} (see Definition~\ref{def_a_priori_assumptions_1_7}).
Most of these properties and assumptions will follow fairly easily, apart from some technical points.
The remaining verification of Definition~\ref{def_comparison_domain}(\ref{pr_7.1_4}) and a priori assumption \ref{item_geometry_cap_extension_5} requires some deeper discussion, which we postpone to the next subsection.

We remind the reader that we assume inequalities of the form (\ref{eqn_prop_comp_domain_construction_parameter_inequalities}), such that the conclusions of the lemmas from the preceding subsections are valid. 

Property (\ref{pr_7.1_1}) of Definition~\ref{def_comparison_domain} holds by construction.  

Next, let us verify property (\ref{pr_7.1_2}) of Definition~\ref{def_comparison_domain}.
Since it is a union of $\Omega$ with connected components of its complement, $\Omega^*$ is a closed subset of $\M_{t_{J+1}}$, and is a domain with smooth boundary, where the boundary components are connected components of $\D\Omega$.  Since $\N^{J+1}_t=\Omega^*(t)$ is the image of $\Omega^*$ under the $(t-t_{J+1})$-flow of $\D_{\t}$, which is defined on a neighborhood of $\Omega^*$, it follows that $\N^{J+1}_t$ is a domain with smooth boundary for all $t \in [t_J, t_{J+1}]$.
Next, recall that $\Omega$ is weakly  $100\lambda r_{\comp}$-thick by Lemma~\ref{lem_s_separates_thick_thin}. By the definition of components of types \ref{type_11.3_I}--\ref{type_11.3_III} and Lemmas~\ref{lem_lambda_thick_survives_backward} and \ref{lem_bounded_curv_bounded_dist}, assuming 
$$
\lambda \leq \ov\lambda, \qquad
\eps_{\can}\leq \ov\eps_{\can}(\lambda)\,, \qquad
r_{\comp} \leq \ov{r}_{\comp},
$$ 
we find that for all $t\in [t_J,t_{J+1}]$:
\begin{enumerate}[label=(\Alph*), leftmargin=* ]
\item \label{prop_11.4_A} The time-slice $\N^{J+1}_t=\Omega^*(t)$ is $\lambda r_{\comp}$-thick.
\item \label{prop_11.4_B} For every $x\in\Omega^*$ the parabolic neighborhood $P(x, r_{\comp})$ is unscathed and is $c r_{\comp}$-thick, where $c = c(\lambda) > 0$.
\end{enumerate}

Now suppose that $\{y_k\}\subset \N^{J+1}$ and $y_k\ra y_\infty\in \M_{t_\infty}$.  Then $y_k=x_k(t_k)$ for some $x_k\in\Omega^*$, $t_k\in [t_J,t_{J+1}]$, and $t_k\ra t_\infty$.  
Clearly, $x_k(t_\infty)\ra y_\infty$.
So $\{x_k(t_\infty)\}$ is a Cauchy sequence in $\M_{t_\infty}$.  Therefore $\{x_k\}$ is Cauchy in $\M_{t_{J+1}}$ by (B) above and a distance distortion argument.  
Since $\Omega^*$ is closed and $\lambda r_{\comp}$-thick, it is complete, assuming
\[ \eps_{\can} \leq \ov\eps_{\can} (\lambda). \]
It follows that $\{x_k\}$ converges to some $x_\infty\in \Omega^*$, and $x_\infty(t_\infty)=y_\infty$.   Hence $\N^{J+1}$ is closed, and we have verified property (\ref{pr_7.1_2}) of Definition~\ref{def_comparison_domain}.   

We have $\D\N^{J+1}_{t_J}=(\D\Omega^*)(t_J)\subset \Omega(t_J)$.  Since $\Omega(t_J)\subset \Int\N_{t_J-}$ by Lemma~\ref{lem_omega_tj_in_n_tj_minus}, part (\ref{pr_7.1_3}) of Definition~\ref{def_comparison_domain} holds.  

We now turn to the a priori assumptions.

A priori assumption \ref{item_time_step_r_comp_1} is obvious.  
By \ref{prop_11.4_A} above, $\N^{J+1}$ is $\lambda r_{\comp}$-thick; so a priori assumption \ref{item_lambda_thick_2}  holds. 

Note that we need only verify a priori assumption \ref{item_backward_time_slice_3} for $\N_{t_{J+1}-}=\Omega^*$.   
A priori assumptions \ref{item_backward_time_slice_3}(a)--(c) follow directly from the construction of $\Omega^*$.
To see a priori assumption \ref{item_backward_time_slice_3}(d), consider a component $Z \subset \M_{t_{J+1}} \setminus \Int \N_{t_{J+1}-}=\M_{t_{J+1}}\setminus\Int\Omega^*$ with non-empty boundary.
Then, by construction, $Z$ is a type-\ref{type_11.3_IV} component of $\M_{t_{J+1}} \setminus \Omega$.
As $Z$ is not of type \ref{type_11.3_I}, it must contain a $10 \lambda r_{\comp}$-thin point.
 A priori assumption \ref{item_backward_time_slice_3}(e) holds since in Proposition \ref{prop_comp_domain_construction} the comparison is defined over the time-interval $[0,t_J]$, and does not include any cuts in $\M_{t_J}$.

Next, we verify a priori assumption \ref{item_discards_not_too_thick_4}.  
Let $\mathcal{C}$ be a $3$-disk component of $\N^J_{t_J} \setminus \Int \N^{J+1}_{t_J}$ (if $J\geq1$) or $\M_0 \setminus \Int \N^1_{0}$ (if $J=0$), such that  $\D\C\subset \N_{t_J}^{J+1}$.  Assume by contradiction that all points on $\mathcal{C}$ survive until time $t_{J+1}$ and that $\C(t)$ is $\lambda r_{\comp}$-thick for all $t\in[t_J,t_{J+1}]$.   Then $\C(t_{J+1})$ is contained in $\M_{t_{J+1}}\setminus \Int\Omega^*$ by the definition of $\N^{J+1}$.  Moreover,  $\D(\C(t_{J+1}))=(\D\C)(t_{J+1})$ is a $2$-sphere contained in $\D\Omega^*$, and hence an entire boundary component of $\Omega^*$.  It follows that $\C(t_{J+1})$ is a component of $\M_{t_{J+1}}\setminus\Int \Omega^*$ that is also a component of $\M_{t_{J+1}}\setminus\Int \Omega$ of type \ref{type_11.3_IV}.  However, it is also of type \ref{type_11.3_II}, which is a contradiction.

Lastly, we point out that by the hypotheses of Proposition~\ref{prop_comp_domain_construction}, we know that a priori assumption \ref{item_eta_less_than_eta_lin_13} holds for $(\N, \linebreak[1] \{ \N^j \}_{j=1}^{J+1}, \linebreak[1] \{ t_j \}_{j=0}^{J+1})$ and $(\Cut, \linebreak[1] \phi, \linebreak[1] \{ \phi^j \}_{j=1}^J)$ (recall that Definition~\ref{def_a_priori_assumptions_1_7} only requires the bound in a priori assumption \ref{item_eta_less_than_eta_lin_13} to hold in the time-interval $[0,t_J]$).

\subsection{Proof of Proposition~\ref{prop_comp_domain_construction}, concluded} \label{subsec_construction_concluded}
It remains to verify Definition~\ref{def_comparison_domain}(\ref{item_3_disk}) and a priori assumption \ref{item_geometry_cap_extension_5}.  

We first verify the ``if'' direction of \ref{item_geometry_cap_extension_5}.  To that end, suppose that $J \geq 1$ and that $\C$ is a component of $\M_{t_J}\setminus\Int \N_{t_J-}$ such that there is a component $\C'$ of $\M'_{t_J}\setminus \phi_{t_J-}(\Int \N_{t_J-})$ satisfying a priori assumptions \ref{item_geometry_cap_extension_5}(a)--(e); in other words:
\begin{enumerate}[label=(\alph*), leftmargin=* ]
\item \label{prop_11.5_a} $\C$ and $\C'$ are $3$-disks.
\item \label{prop_11.5_b} $\partial \C' = \phi_{t_J-} (\partial \C)$.
\item \label{prop_11.5_c}  There is a point $x\in \C$ such that $(\M_{t_J},x)$ is $\de_{\bb}$-close to the pointed Bryant soliton $(M_{\Bry},g_{\Bry},x_{\Bry})$ at scale $10\lambda r_{\comp}$.
\item \label{prop_11.5_d} There is a point $x' \in \M'_{t_J}$, at distance $\leq D_{\CAP} r_{\comp}$ from $\C'$ such that $(\M'_{t_j},x')$ is $\de_{\bb}$-close to the pointed Bryant soliton $(M_{\Bry}, \lb g_{\Bry}, \lb x_{\Bry})$  at some scale in the interval $[D_{\CAP}^{-1}r_{\comp},D_{\CAP}r_{\comp}]$.
\item \label{prop_11.5_e} $\C$ and $\C'$ have diameter $\leq D_{\CAP}r_{\comp}$.
\end{enumerate}
We now claim that, under suitable assumptions on the parameters, $\C$ is a component of $\N^{J+1}_{t_J} \setminus \Int \N^{J }_{t_J}$.
Since $\C$ is a $3$-disk by assumption, this will imply that $\C$ is an extension cap.

To see this, we will apply the Bryant Slab Lemma~\ref{lem_bryant_slab} for $X_0 = \N^{J}_{t_J}$ and $X_1 = \Omega$.
Note that assumptions \ref{con_8.42_i}--\ref{con_8.42_iv} of the Bryant Slab Lemma hold due to Definition~\ref{def_comparison_domain}(\ref{pr_7.1_1}), a priori assumptions \ref{item_backward_time_slice_3}(a)--(c) and by the construction of $\Omega$.
Assumption \ref{con_8.42_v} of the Bryant Slab Lemma holds due to Lemma~\ref{lem_omega_tj_in_n_tj_minus}.
So the Bryant Slab Lemma can be applied on the time-interval $[t_J, t_{J+1}]$ if
\begin{equation} \label{eq_Bryant_slab_constants} 
\delta_{\nn} \leq \ov\delta_{\nn} , \qquad
 0 < \lambda < 1, \qquad
\Lambda \geq \underline{\Lambda}, \qquad
\delta' \leq \ov{\delta}' (\lambda, \Lambda) 
\end{equation}
and if there is a map $\psi$ with $\psi (x_{\Bry}, - (10 \lambda)^{-2}) =x$ and a $\delta'$-good Bryant slab $W \subset \M_{[t_J, t_{J+1}]}$, as required in the Bryant Slab Lemma.
The existence of the map $\psi$ and the $\delta'$-good Bryant slab $W$ follows from \ref{prop_11.5_c} above and Lemma~\ref{lem_promoting_Bryant}, and assuming
\[ \delta_{\bb} \leq \ov\delta_{\bb} (\lambda, \delta'), \qquad
\eps_{\can} \leq \ov\eps_{\can} (\lambda, \delta'), \qquad
r_{\comp} \leq \ov{r}_{\comp}. \]
Under assumptions of the same form as (\ref{eq_Bryant_slab_constants}) we can also apply the Bryant Slice Lemma~\ref{lem_Bryant_slice} at time $t_i$, $i = J, J+1$, for $\psi = \psi_{t_i}$, $W = W_{t_i}$ and $X = X_{i-J}$.

Let $\C_0 :=W_{t_J}\setminus\Int X_0$ and $\C_1 := W_{t_{J+1}} \setminus \Int X_1$ be as in the Bryant Slab Lemma.  
By the Bryant Slice Lemma applied at time $t_{J+1}$ we know that $x(t_{J+1}) = \psi ( x_{\Bry}, 0)$ is $11 \lambda r_{\comp}$-thin.
So, by construction of $\Omega$, we have $x(t_{J+1}) \in \C_1 \neq \emptyset$.
By assertions \ref{ass_8.42_a}, \ref{ass_8.42_b} of the Bryant Slice Lemma we find that $\C_0 = \C$ and that $\C_{1}$ is a 3-disk component of $\M_{t_{J+1}} \setminus \Int \Omega$.
Assertions \ref{ass_8.43_a}, \ref{ass_8.43_b} of the Bryant Slab Lemma imply that $\C_1 (t)$ is $9\lambda r_{\comp}$-thick for all $t \in [t_J, t_{J+1}]$ and $\C_1 (t_{J}) \supset \C_0$ and $\C = \C_0 = \C_1 (t_J) \setminus \Int \N_{t_J}^J$.
It follows that $Z := \C_{1}$ is a component of type \ref{type_11.3_III}, and so $Z(t_J)\subset \Omega^*(t_J)=\N_{t_J}^{J+1}$.  
Thus $\C\subset \N^{J+1}_{t_J}\setminus\Int \N^J_{t_J}$, and since $\C$ is a component of $\M_{t_J}\setminus \Int \N^J_{t_J}$, it is also a component of $\N_{t_J}^{J+1} \setminus\Int \N_{t_J}^J$.  
Hence the ``if'' direction of \ref{item_geometry_cap_extension_5} holds.

In order to verify Definition \ref{def_comparison_domain}(\ref{item_3_disk}) and the ``only if'' direction of a priori assumption \ref{item_geometry_cap_extension_5}, we need the following fundamental result.

\begin{lemma}[Structure of extension caps] \label{lem_APA5}
 If
\begin{equation*}
\begin{gathered}
\eta_{\lin}\leq\ov{\eta}_{\lin}\,,\qquad 
\de_{\nn}\leq\ov\de_{\nn}\,, \qquad
\lambda\leq\ov\lambda, \qquad 
D_{\CAP} \geq \underline{D}_{\CAP} (\lambda), 
\qquad \Lambda \geq \underline{\Lambda} (\lambda) \,,\\ \qquad
\eps_{\can} \leq \ov\eps_{\can}(\la,\Lambda,\de_{\bb})\,, \qquad
 r_{\comp} \leq \ov{r}_{\comp} (\lambda)\,,
 \qquad\qquad\qquad 
\end{gathered}
\end{equation*}
then the following holds.

If $Z \subset \M_{t_{J+1}} \setminus \Int \Omega$ is a component of type \ref{type_11.3_I}, then $Z(t_J) \subset \N_{t_J-}$.
\end{lemma}

Before proceeding, we first explain how Lemma \ref{lem_APA5} completes the verification of Proposition \ref{prop_comp_domain_construction}.

For this purpose consider a  component $\C^* \subset \N^{J+1}_{t_J} \setminus \Int \N_{t_J-}$.
As $\Omega (t_J) \subset \Int \N_{t_J-}$, we have $\C^* \subset \N^{J+1}_{t_J} \setminus \Omega(t_J)$.
Thus $\C^* \subset \Int Z(t_J)$ for some component $Z \subset \M_{t_{J+1}} \setminus \Int \Omega$ of type \ref{type_11.3_I}, \ref{type_11.3_II} or \ref{type_11.3_III}.
By the above lemma and condition \ref{type_11.3_II}(c), $Z$ cannot be of type \ref{type_11.3_I} or \ref{type_11.3_II} and therefore must be of type \ref{type_11.3_III}.
Next, observe that $\C^* \subset Z(t_J) \setminus \Int \N_{t_J-} =: \C$ and $\C = Z(t_J) \setminus \Int \N_{t_J-} \subset \N^{J+1}_{t_J} \setminus \Int \N_{t_J-}$.
As $\C^*$ is a connected component of $\N^{J+1}_{t_J} \setminus \Int \N_{t_J-}$, it follows that $\C = \C^*$.

By \ref{type_11.3_III}(c) we know that $\C^* = \C$ is a 3-disk, which proves Definition~\ref{def_comparison_domain}(\ref{item_3_disk}).
The remaining statements of \ref{type_11.3_III}(c) imply that $\C^* = \C$ satisfies \ref{item_geometry_cap_extension_5}(a)--(e).

Next, we provide an outline of the proof of Lemma~\ref{lem_APA5}, neglecting several technicalities.

Assume by contradiction that $Z$ is a type \ref{type_11.3_I} component with $Z(t_J) \not\subset \N_{t_J-}$.
This means that $Z(t_J)$ contains a component $\C$ of the complement $\M_{t_J-} \setminus \Int \N_{t_J-}$.
As $Z$ consists of weakly $10 \lambda r_{\comp}$-thick points and $\C$ contains a $10 \lambda r_{\comp}$-thin point by a priori assumption \ref{item_backward_time_slice_3}(d), there must be a point in $\C$ whose scale increases over the time-interval $[t_J, t_{J+1}]$.
By Lemma~\ref{lem_bryant_propagate}, this is only possible if $Z$ and $\C$ lie in a large spacetime region $W \subset \M$ that is very close to a Bryant soliton.
More specifically, we may assume that this region is $9 \lambda r_{\comp}$-thick and defined over a long backward time-interval of the form $[t_{J - J_\#}, t_{J+1}]$, where $J_\# \gg 1$.

The existence of the component $\C$ and the Bryant like geometry on $W$ will then force the existence of a sequence of components $\C_j \subset \M_{t_j} \setminus \Int \N_{t_j -}$ for $j = J, J-1, \ldots, J - J_\#$, where $\C_J = \C$.
This will follow from a priori assumption \ref{item_discards_not_too_thick_4}, which forbids the discard of components that remain  $\lambda r_{\comp}$-thick during a time step.

Next, using the bilipschitz bound on the comparison map $\phi$ imposed by \ref{item_eta_less_than_eta_lin_13}, and the fact that $W$ is not too neck-like, we will find that for $t\in [t_{J - J_\#},t_J]$, the image $\phi_{t}(W\cap\N_{t})$ intersects a smoothly varying 3-disk region $\wh{W}'_t\subset\M_{t}'$ with scale and diameter comparable to $r_{\comp}$.

The union $\wh{W}' \subset \M'_{[t_{J-J_\#}, t_J]}$ of these regions forms a ``barrier region''  that will help us show the existence of a point $z' \in \wh{W}'_{t_{J-J_\#}}$ that survives until time $t_J$ and that has the property that $z'(t) \in \wh{W}'$ for all $t \in [t_{J-J_\#}, t_J]$.
The scale of $z'(t)$ will be controlled from above and below by a constant that is independent of $J_\#$.
Therefore, if we choose $J_\#$ large enough, then we can find a time-step $t_j\in [t_{J-J_\#+1}, t_{J-1}]$ such that the scale of $z(t)$ hardly decreases over the time-interval $[t_{j-1}, t_j]$.
Using again Lemma~\ref{lem_bryant_propagate} (this time in $\M'$), we will deduce that the geometry near $z'(t_j)$ is close to a Bryant soliton.
This means that  \ref{item_geometry_cap_extension_5} applies and would have forced $\C_j$ to be an extension cap, giving a contradiction.

\begin{proof}[Proof of Lemma \ref{lem_APA5}]
Fix a type \ref{type_11.3_I} component $Z \subset \M_{t_{J+1}} \setminus \Int \Omega$ for the remainder of the proof and assume that $Z(t_J) \not\subset \N_{t_J}$.
So $Z(t_J)$ intersects a component $\C$ of $\M_{t_J} \setminus \Int \N_{t_J-}$.
Because $Z(t_J)$ is a closed subset, its topological boundary in $\M_{t_J}$ is  $\D Z(t_J)$, and since $\partial Z(t_J) \subset \partial \Omega (t_J) \subset \Int \N_{t_J-}$, it is disjoint from $\M_{t_J} \setminus \Int \N_{t_J-}$.
The connectedness of $\C$ now implies that $\C \subset Z(t_J)$.

By a priori assumption \ref{item_backward_time_slice_3}(d) there is a $10 \lambda r_{\comp}$-thin point $x \in \C \subset Z(t_J)$.  
By the type \ref{type_11.3_I} property and the discussion in Subsection~\ref{subsec_def_of_N_Jp1}, we know that $x$ survives until time $t_{J+1}$ and that $x(t_{J+1})$ is weakly $10 \lambda r_{\comp}$-thick.
Moreover, by Lemma~\ref{lem_lambda_thick_survives_backward}, we find that $x(t_{J+1})$ is $11\lambda r_{\comp}$-thin, assuming
\[ \lambda \leq \tfrac1{10}, \qquad
\eps_{\can} \leq \ov\eps_{\can} (\lambda), \qquad
r_{\comp} \leq \tfrac1{10}. \]
We can therefore apply \ref{lem_bryant_propagate} to $x$ and obtain that a large spacetime neighborhood of $x(t_{J+1})$ is close to a Bryant soliton. 
More specifically, Lemma~\ref{lem_bryant_propagate} implies the following.
Let $\delta_\# > 0$ and $J_\# < \infty$ be constants whose values will be determined in the course of the proof.  
Then, under a condition of the form
\[ \lambda \leq \tfrac{1}{20}, \qquad
\eps_{\can} \leq \ov\eps_{\can} (\lambda, J_\#, \delta_\# ), \qquad
r_{\comp} \leq 1 ,
  \]
we can find a $(10 \lambda r_{\comp})^2$-time equivariant and $\partial_{\mathfrak{t}}$-preserving diffeomorphism 
\[ \psi : W^*:=\ol{M_{\Bry} (\delta_\#^{-1} )} \times \big[ {- \min \{ J ,J_\# +1 \} \cdot (10 \lambda)^{-2}, 0} \big] \longrightarrow \M \]
onto its image such that $\psi (x_{\Bry},0) = x (t_{J+1})$ and
\begin{equation} \label{eq_Bryant_closeness}
 \big\Vert{ (10 \lambda r_{\comp})^{-2} \psi^* g - g_{\Bry} }\big\Vert_{C^{[\delta_\#^{-1}]} (W^*)} < \delta_\#.
\end{equation}
Let $W=\psi(W^*)$.
Note that $ W^*$ has been chosen in such a way that its image $W$ has initial time-slice $t_{J- J_\#}$ if $J_\# \leq J -1$ and $t_1$ otherwise.

Next, we show that the existence of the component $\C$ of $\M_{t_J} \setminus \Int \N_{t_J-}$ forces the existence of components $\C_j \subset \M_{t_j} \setminus \Int\N_{t_j-}$ at a large number of earlier times $t_j \leq t_J$.
The existence of these components will be deduced using priori assumption \ref{item_discards_not_too_thick_4} and the Bryant-like geometry on $W$.

\begin{claim}[Cap hierarchy]
\label{cl_11_1}
If, in addition,
\begin{equation*}
\delta_{\nn} \leq \ov\delta_{\nn}  ,\qquad 
\lambda \leq \ov\lambda, \qquad
\Lambda \geq \underline{\Lambda}, \qquad \delta_\# \leq \ov{\delta}_\# ( \lambda, \Lambda, J_\#),
\end{equation*}
then  $J\geq J_\#+1$ and:
\begin{enumerate}[label=(\alph*)]
\item \label{ass_11.17_cl1_a} For all $J - J_\# \leq j \leq J$ the subset $\C_j:=W_{t_j}\setminus\Int \N_{t_j-}$ is a $3$-disk. 
\item \label{ass_11.17_cl1_b} For all $J - J_\# + 1 \leq j \leq J $ all points on $\C_j$ survive until time $t_{j-1}$ and $\C_{j-1} \subset \C_j (t_{j-1})$.
\item \label{ass_11.17_cl1_c} $\C = \C_J$.
\end{enumerate}
\end{claim}

\begin{proof}
In the following we will apply the Bryant Slice Lemma~\ref{lem_Bryant_slice} at time $t_j$ for $X = \N_{t_j-}$, where $J - J_\# \leq j \leq J$.
We will also apply the Bryant Slab Lemma~\ref{lem_bryant_slab} for  $X_0 = \N_{t_{j-1}-}$ and $X_1 = \N_{t_{j}-}$, where $J - J_\# + 1 \leq j \leq J$.
Note that assumptions \ref{con_8.42_i}--\ref{con_8.42_iv} of the Bryant Slice Lemma hold due to a priori assumptions \ref{item_backward_time_slice_3}(a)--(c) and assumption \ref{con_8.42_v} of the Bryant Slab Lemma holds due to Definition~\ref{def_comparison_domain}(\ref{pr_7.1_3}).
If 
\[ \delta_{\nn} \leq \ov\delta_{\nn}, \qquad
 0 < \lambda < 1, \qquad
  \Lambda \geq \underline{\Lambda}, \qquad
\de_\#\leq \ov\de_\#(J_\#,\lambda,\Lambda), \]
then the remaining assumptions of both the Bryant Slice and the Bryant Slab Lemma are satisfied.
This means, in particular, that the time-slice $W_{t_j}$ and the slab $W_{[t_{j-1},t_j]}$ satisfy the assumptions of the Bryant Slice Lemma and the Bryant Slab Lemma,  for all $J  - J_\# \leq j \leq J$  and $J  - J_\# +1\leq j \leq J$, respectively.

Since  $x\in \C\cap \left(W_{t_J}\setminus\Int\N_{t_J-}\right)$, we know by the Bryant Slice Lemma at time $t_J$ that $\C_J:=W_{t_J-}\setminus\Int\N_{t_J}$ is a $3$-disk and is a component of $\M_{t_J}\setminus\Int \N_{t_J-}$. 
Hence it coincides with $\C$, which proves assertion \ref{ass_11.17_cl1_c}.

Fix some $j$ with $J - J_\# \leq j \leq J$.
Assume inductively that $j \geq 1$ and that assertion \ref{ass_11.17_cl1_a} holds for all $j \leq j' \leq J$ and assertion \ref{ass_11.17_cl1_b} holds for all $j + 1 \leq j' \leq J$.
If $j = J - J_\#$, then $J \geq J_\# + 1$, as claimed, and assertions \ref{ass_11.17_cl1_a} and \ref{ass_11.17_cl1_b} hold. So assume in the following that $j > J - J_\#$.

By assertion \ref{ass_8.43_a} of the Bryant Slab Lemma, $\C_j (t)$ is defined and $9 \lambda r_{\comp}$-thick for all $t \in [t_{j-1}, t_j]$.
Moreover, the subset $\C_j(t_{j-1})$ is a $3$-disk component of $\M_{t_{j-1}}\setminus\Int \N_{t_{j-1}+}$ and $\D\C_j(t_{j-1})\subset \D\N_{t_{j-1}+}$.
It follows from a priori assumption \ref{item_discards_not_too_thick_4} that $j-1\geq 1$.  
Now suppose that $\C_{j-1}=W_{t_{j-1}}\setminus\Int \N_{t_{j-1}-}=\emptyset$.
It follows that $\C_j (t_{j-1}) \subset W_{t_{j-1}}\subset \Int \N_{t_{j-1}-}$.  
Therefore $\C_j(t_{j-1})\subset \N_{t_{j-1}-}\setminus\Int \N_{t_{j-1}+}$, and since it is a $3$-disk with boundary contained in $\D\N_{t_{j-1}+}$, it is a component of $\N_{t_{j-1}-}\setminus\Int \N_{t_{j-1}+}$.  This contradicts a priori assumption \ref{item_discards_not_too_thick_4}.  Thus $\C_{j-1}\neq\emptyset$ and by the Bryant Slice Lemma at time $t_{j-1}$ it must be a 3-disk.
So assertion \ref{ass_11.17_cl1_a} holds for $j-1$ and assertion \ref{ass_11.17_cl1_b} holds for $j$ by the Bryant Slab Lemma.

By induction we conclude that $J\geq J_\#+1$, and \ref{ass_11.17_cl1_a} and \ref{ass_11.17_cl1_b} hold.
\end{proof}

Next, we will construct the ``barrier'' region $\wh{W}'$ mentioned in the outline given above. We remark that in the following construction, we have to choose $\wh{W}$ larger than the reader may anticipate.
The reason is purely technical: Due to the fact that a priori assumption \ref{item_eta_less_than_eta_lin_13} only gives us $C^0$ bounds on the metric distortion of $\phi$, the weakness of the resulting scale distortion control (see Lemma~\ref{lem_scale_distortion}) forces us to work in a region whose boundary has scale a large multiple of $r_{\comp}$.

We will now construct the subset $\wh{W} \subset  W$.
For this purpose fix the (universal) constant $C_{\sd}$ from Lemma~\ref{lem_scale_distortion} and assume without loss of generality that $C_{\sd} > 100$.
Define
\[ \wh{W}^* \subset M_{\Bry} \times \big[ {- (J_\#+1) (10 \lambda)^{-2}, -(10 \lambda)^{-2}} \big] \]
to be the subset of points on which $\rho \leq 20 C_{\sd}^2 \cdot (10 \lambda)^{-1}$.
Then $\wh{W}^*$ is closed and connected, and its time-slices $\wh{W}^*_t$ are pairwise isometric $3$-disks for all $t\in [-(J_\#+1)(10\lambda)^{-2},-(10 \lambda)^{-2}]$. 
If
\[ \delta_\# \leq \ov\delta_\# (\lambda, J_\#), \]
then  
\[ \wh{W}^*  \subset W^*=\ol{M_{\Bry} (\delta_\#^{-1})} \times \big[{ - (J_\#+1) (10 \lambda)^{-2}, \lb -(10 \lambda)^{-2}} \big]\, . \]
So we may define
\[ \wh{W} := \psi (\wh{W}^*)\subset \M_{[t_{J-J_\#},t_{J}]}\,. \]
 Then, assuming 
$$
\delta_\# \leq \ov\delta_\# (\lambda, J_\#)\,, \qquad
r_{\comp} \leq \ov{r}_{\comp}\,,
$$
we obtain that  for all $t \in [t_{J-J_\#}, t_{J}]$, the time-slice $\wh{W}_t$ is a $3$-disk and 
\begin{equation}
\label{eqn_w_t_geometry}
\begin{gathered}
\text{$10 C_{\sd}^2 r_{\comp} < \rho = \rho_1  < 40 C_{\sd}^2 r_{\comp}$ on $\partial \wh{W}_t$}\,.\\
\text{$W_t\setminus\Int \wh{W}_t$ \quad is \quad $10 C_{\sd}^2 r_{\comp}$-thick}\\
\text{$\wh{W}_t$ \quad is \quad $40C_{\sd}^2 r_{\comp}$-thin}
\end{gathered}
\end{equation}

\begin{claim}
\label{cl_11_2}
If 
$$
\delta_{\nn} \leq \ov\delta_{\nn}, \qquad
\de_\#\leq \ov\de_\#(\lambda,\Lambda,J_\#), \qquad
\eps_{\can} \leq \ov\eps_{\can}, \qquad
r_{\comp} \leq \ov{r}_{\comp}, 
$$
then:
\begin{enumerate}[label=(\alph*)]
\item \label{ass_11.17_cl2_a} $W\subset \M\setminus\cup_{\DD\in\Cut}\DD$, and hence by Definition \ref{def_comparison}(\ref{pr_7.2_5}) the map $\phi$ is well-defined on $W_{[t_{J-J_\#},t_J]}\cap \N$. 
\item \label{ass_11.17_cl2_b} For all $J- J_\# < j \leq J$ and $t\in [t_{j-1},t_j]$, 
\begin{equation*} \label{eq_partial_W_NN_minus_cuts}
W_t\setminus \Int \wh{W}_t \subset \Int\N^j_t\,.
\end{equation*}
\item \label{ass_11.17_cl2_c} $\C_j \subset \Int \wh{W}_{t_j}$ for all $J-J_\# \leq j \leq J$.
\end{enumerate}
\end{claim}
\begin{proof}
 If 
$$
\de_\#\leq \ov\de_\#(\lambda,\Lambda,J_\#)\,,
$$
then for every $J-J_\# \leq j\leq J+1$ we get that $\D W_{t_j}$ is $\Lambda r_{\comp}$-thick.

Suppose that  $\DD\cap W_{t_j} \neq\emptyset$ for some  $\DD\in\Cut$.  Note that this implies that $J-J_\# \leq j<J$.  Since $
\DD$ is $\Lambda r_{\comp}$-thin by \ref{item_backward_time_slice_3}(e), it is disjoint from $\D W_{t_j}$. So since $\DD$ is connected by Definition~\ref{def_comparison}(\ref{pr_7.2_2}), we have $\DD\subset W_{t_j}$.  By Definition~\ref{def_comparison}(\ref{pr_7.2_3}) the cut $\DD$ contains an extension cap. However, this contradicts assertion \ref{ass_11.17_cl1_b} of Claim~\ref{cl_11_1}.  
So we have shown assertion \ref{ass_11.17_cl2_a} of this claim.

Now suppose that $J-J_\#<j\leq J $ and $t \in [t_{j-1}, t_j]$ or $j = J - J_\#$ and $t = t_j$.
As $\D W_{t_j}$ is $\Lambda r_{\comp}$-thick, it is contained in $\N_{t_j-}$ by \ref{item_backward_time_slice_3}(b). 
Thus $\partial W_t \subset \N^j$.
Moreover, if
\[ \delta_{\nn} \leq \ov\delta_{\nn}, \qquad
\eps_{\can} \leq \ov\eps_{\can}, \qquad
r_{\comp} \leq \ov{r}_{\comp}, \]
then we obtain from Lemma~\ref{lem_time_slice_neck_implies_space_time_neck} that $\partial \N^j_t$ is $2.1 r_{\comp}$-thin.
In view of the fact that $W_{t}\setminus\Int  \wh{W}_{t}$ is connected and $10 C_{\sd}^2 r_{\comp}$-thick by (\ref{eqn_w_t_geometry}), it is disjoint from $\D\N^j_{t}$ and hence contained in $\N_{t}$.
This proves assertion \ref{ass_11.17_cl2_b} and assertion \ref{ass_11.17_cl2_c} follows in the case $t = t_j$.
\end{proof}

Next we consider the image of $W_t\setminus\Int \wh{W}_t $ under $\phi$, and show that the boundary component $\phi_t(\D \wh{W}_t)$ is adjacent to a region with  controlled geometry.  
\begin{claim}
\label{cl_11_3}
Assuming
\begin{gather*}
\eta_{\lin}\leq\ov\eta_{\lin}\,,\qquad 
\de_{\nn} \leq\ov\de_{\nn} \,,\qquad
\lambda \leq \ov\lambda \,, \qquad
\Lambda\geq \underline\Lambda(\la)\,, \qquad 
\delta_{\#} \leq \ov\delta_{\#} (\lambda, J_\#) \,,\\ \qquad
\eps_{\can}\leq\ov\eps_{\can}(\la)\,, \qquad
r_{\comp}\leq \ov r_{\comp}(\la)\,, 
\end{gather*}
there is a constant $C_1=C_1(\la) < \infty$ with the following property.

There is a  subset $\wh{W}' \subset \M'_{[t_{J-J_\#}, t_J]}$ such that for every $t\in [t_{J-J_\#},t_{J}]$
\begin{enumerate}[label=(\alph*)]
\item \label{ass_11.17_cl3_a} $\wh{W}'_t$ is a $3$-disk.
\item \label{ass_11.17_cl3_b} $\wh{W}'_t\cap\phi_t(W_t\setminus\Int \wh{W}_t)=\phi_t(\D \wh{W}_t)=\D \wh{W}'_t$.
\item \label{ass_11.17_cl3_c} $\wh{W}'$ is compact and its relative topological boundary inside the time-slab $\M'_{[t_{J-J_\#}, t_J]}$ is equal to $ \cup_{t \in [t_{J-J_\#}, t_J]} \partial  \wh{W}'_t$.
\item \label{ass_11.17_cl3_d} $\wh{W}'_t$ is $C_1 r_{\comp}$-thin and $C_1^{-1}r_{\comp}$-thick and $\diam_{t} \wh{W}'_t \leq C_1 r_{\comp}$.
\item \label{ass_11.17_cl3_e} $\partial\wh{W}'_t$ is $10 C_{\sd} r_{\comp}$-thick.
\item \label{ass_11.17_cl3_f} For any $J-J_\# \leq j\leq J$ the difference $\C'_j := \wh{W}'_{t_j}\setminus \phi_{t_j-}(\Int \N_{t_j-})$ is a $3$-disk component of $\M'_{t_j} \setminus \phi_{t_j-} ( \Int \N_{t_j-})$ and we have $\D\C'_j=\phi_{t_j-}(\D\C_j)$.
\end{enumerate}
\end{claim}
\begin{proof}
Fix $t\in [t_{J-J_\#},t_{J}]$. 
By (\ref{eqn_w_t_geometry}), a priori assumptions \ref{item_lambda_thick_2}, \ref{item_backward_time_slice_3}(a), (c), \ref{item_eta_less_than_eta_lin_13} and Lemma~\ref{lem_scale_distortion}, assuming 
\begin{equation} \label{eq_scale_dist_application}
 \eta_{\lin} \leq \ov\eta_{\lin}, 
 \qquad \delta_{\nn} \leq \ov{\delta}_{\nn}, \qquad
  \Lambda \geq 2, \qquad 
  \eps_{\can} \leq \ov\eps_{\can} (\lambda) \,, \qquad
  r_{\comp}\leq\ov r_{\comp}\,, 
\end{equation}
 we have
\begin{equation} \label{eq_rho_bound_on_boundary_M_prime}
 10C_{\sd} r_{\comp} < \rho_1 = \rho < 40C_{\sd}^3 r_{\comp} \qquad \text{on} \qquad \phi_t ( \partial \wh{W}_t )\,, 
\end{equation}
and
\begin{equation*} \label{eq_diameter_bound}
 \diam \phi_t ( \partial \wh{W}_t ) < 10 \diam \partial \wh{W}_t < C'_1 r_{\comp},
\end{equation*}
where $C'_1 < \infty$ is a universal constant that can be determined in terms of $C_{\sd}$.

Choose $x \in \partial \wh{W}_t$.  Using (\ref{eqn_w_t_geometry}) and assuming 
\[  \delta_{\#} \leq \ov\delta_{\#} (\lambda, J_\#), \qquad
r_{\comp} \leq \ov{r}_{\comp}, 
\]
we can find a point $y \in W_t\setminus\Int \wh{W}_t$ with $\rho_1 (y) = \rho (y) = 80 C_{\sd}^4 r_{\comp}$ that can be connected to $x$ by a path of length at most $C'_2 r_{\comp}$ inside $W_t\setminus\Int \wh{W}_t$, for some $C'_2 = C'_2 (\lambda) < \infty$.
Let $x'=\phi_t(x)$, $y'=\phi_t (y)$.
Again by the scale distortion Lemma~\ref{lem_scale_distortion}, a priori assumptions \ref{item_lambda_thick_2}, \ref{item_backward_time_slice_3}(a), (c), \ref{item_eta_less_than_eta_lin_13}, and assuming a bound of the form (\ref{eq_scale_dist_application}), we conclude using (\ref{eq_rho_bound_on_boundary_M_prime}) that
\begin{equation} \label{eq_rho_yprime_xprime}
\rho(y') > 80 C_{\sd}^3 r_{\comp} > 2 \rho (x') 
\end{equation}
and $d_{\M'_t} (x', y') < 2 C'_2 r_{\comp}$. 
So there is a constant $\delta^\circ = \delta^\circ (\lambda ) > 0$ such that $x'$ cannot be a center of a $\delta^\circ$-neck in $\M'_{t}$.

Let us now apply Lemma~\ref{lem_neck_or_cap} to $x'$ for $\delta = \delta^\circ (\lambda)$.
In order to satisfy the assumptions of this lemma, we need to assume that
\[ \eps_{\can} \leq \ov\eps_{\can} (\lambda), \qquad
r_{\comp} \leq \ov{r}_{\comp}. \]
We obtain a constant $C_0 = C_0 (\delta^\circ (\lambda)) < \infty$ and a compact subset $V' \subset \M'_{t}$, containing $x'$ such that (compare with (\ref{eq_rho_bound_on_boundary_M_prime}))
\begin{equation} \label{eq_rho_bound_on_V_prime}
10 C_0^{-1} C_{\sd} r_{\comp} < \rho_1 = \rho < 40 C_0 C^3_{\sd} r_{\comp} \qquad \text{on} \qquad V' 
\end{equation}
and such that $\diam_t V' < 40 C_0 C_{\sd}^3 r_{\comp}$.
Moreover, we may assume that $\delta^\circ (\lambda)$ is chosen small enough such that $B(x', \max \{ C'_1, 2C'_2 \} r_{\comp} ) \subset V'$.
This implies that $y' \in \Int V'$ and 
\begin{equation} \label{eq_phi_wh_W_V_prime}
\phi_t (\partial \wh{W}_t) \subset \Int V'.
\end{equation}

We claim that $V'$ is a 3-disk.
To see this, we assume
\[ 
\delta_{\#} \leq \ov\delta_{\#} (\lambda, J_\#), \qquad
r_{\comp} \leq \ov{r}_{\comp} (\lambda),  
\]
such that $\partial W_t$ is $40 C_0 C_{\sd}^4 r_{\comp}$-thick and that $40 C_0 C_{\sd}^4 r_{\comp} \leq 1$.
So, again, by the scale distortion Lemma~\ref{lem_scale_distortion}, a priori assumptions \ref{item_lambda_thick_2}, \ref{item_backward_time_slice_3}(a), (c), \ref{item_eta_less_than_eta_lin_13}, and assuming (\ref{eq_scale_dist_application}), we obtain that $\phi_t (\partial W_t)$ is $40 C_0 C_{\sd}^3 r_{\comp}$-thick.
Thus by (\ref{eq_rho_bound_on_V_prime}) we have 
\begin{equation} \label{eq_ph_partial_W_V_prime}
 \phi_t (\partial W_t) \cap V' = \emptyset. 
\end{equation}
As $x'$ and $\phi_t (\partial W_t)$ lie in the same connected component of $\M'_t$, we must have $\partial V' \neq \emptyset$ and due to (\ref{eq_rho_yprime_xprime}), Lemma~\ref{lem_neck_or_cap} implies that $V'$ is a 3-disk.

By (\ref{eq_phi_wh_W_V_prime}) the 2-sphere $\phi_t (\partial \wh{W}_t)$ bounds a 3-disk $\wh{W}'_t \subset V'$.
We now repeat the construction above for all $t \in [t_{J-J_\#}, t_J]$ and set $\wh{W}' := \cup_{t \in [t_{J-J_\#}, t_J]} \wh{W}'_t$.
Then assertion \ref{ass_11.17_cl3_a} holds automatically.

Next, observe that $\wh{W}'_t$ and $\phi_t(W_t\setminus\Int(\wh{W}_t))$ are compact connected domains with smooth boundary that share a single boundary component $\phi(\D \wh{W}_t)$.  
Therefore assertion \ref{ass_11.17_cl3_b} of this claim can fail only if $\phi_t(W_t\setminus\Int \wh{W}_t) \subset \wh{W}'_t \subset V'$, which would contradict (\ref{eq_ph_partial_W_V_prime}).
Thus assertion \ref{ass_11.17_cl3_b} holds.

In order to show assertion \ref{ass_11.17_cl3_c}, it suffices to show that for all $t \in [t_{J-J_\#}, t_J]$, every point $q \in \Int \wh{W}'_t$ is not contained in the relative boundary of $\wh{W}'$ inside $\M'_{[t_{J-J_\#}, t_J]}$.
To see this, let $U \subset \M'_{[t_{J-J_\#}, t_J]}$ be a product domain containing $\wh{W}'_t$ that is open in $\M'_{[t_{J-J_\#}, t_J]}$.
By Claim~\ref{cl_11_2} and Definition \ref{def_comparison}, the map $\phi$ is well-defined and smooth on a neighborhood of $\cup_{\bar t\in {[t_{J-J_\#},t_J]}} W_{\ov{t}} \setminus\Int \wh{W}_{\ov{t}}$.
Therefore, after shrinking $U$ if necessary, we may assume that if $\bar t\in [t_{J-J_\#},t_J]$ is close to $t$, then $(\phi(W_{\bar t}\setminus\Int \wh{W}_{\bar t})\cap U_{\bar t} )(t)$ is defined and moves by smooth isotopy as $\ov{t}$ varies.
So by assertion \ref{ass_11.17_cl3_b} the 3-disk $(\wh{W}'_{\ov{t}} \cap U_{\ov{t}})(t)$ varies by smooth isotopy as well and therefore it contains a small neighborhood of $q$ inside $\M'_t$ for $\ov{t}$ close to $t$.
This implies that a small neighborhood of $q$ is contained in $\cup_{\ov{t} \in [t_{J-J_\#}, t_J]} \wh{W}'_{\ov{t}}$, which finishes the proof of assertion \ref{ass_11.17_cl3_c}.
The same argument implies that $\wh{W}'$ is a finite union of compact subsets and must therefore be compact.

Assertions \ref{ass_11.17_cl3_d} and \ref{ass_11.17_cl3_e} follow by construction of $\wh{W}'_t$ and (\ref{eq_rho_bound_on_V_prime}) and (\ref{eq_rho_bound_on_boundary_M_prime}), as long as $C_1 > 40 C_0 C_{\sd}^3$.

Lastly, consider assertion \ref{ass_11.17_cl3_f}.
Suppose that $J - J_\# \leq j \leq J$.

Recall that $\C_j= W_{t_j}\setminus\Int\N_{t_j-}$ is a $3$-disk by assertion \ref{ass_11.17_cl1_a} of Claim~\ref{cl_11_1}.  
By assertion \ref{ass_11.17_cl2_b} of Claim~\ref{cl_11_2} we have $W_{t_j}\setminus\Int \wh{W}_{t_j}\subset \Int\N_{t_j-}$.
So $\C_j\subset \Int \wh{W}_{t_j}$.  
Therefore, $\wh{W}_{t_j}\setminus\Int\C_j= \N_{t_j-} \cap \wh{W}_{t_j}$ is a compact connected manifold with boundary.

As $\phi_{t_j-} (\partial \wh{W}_{t_j} ) = \partial \wh{W}'_{t_j}$ and $\phi_{t_j-} :\N_{t_j-} \to \M'_{t_j}$ is injective, the image $\phi_{t_j-} ( \Int \wh{W}_{t_j} \cap \N_{t_j-})$ must either be contained in $\wh{W}'_{t_j}$ or in its complement.
Since $\phi_{t_j-}$ maps a neighborhood of $\D \wh{W}_{t_j}$ in $\wh{W}_{t_j}$ into $\wh{W}'_{t_j}$, we obtain by connectedness that $\phi_{t_j-} (\N_{t_j-} \cap \wh{W}_{t_j}) \subset \wh{W}'_{t_j}$. 
By the same argument, if $N_0$ is the component of $\N_{t_j-}$ that contains $\partial \wh{W}_{t_j}$, then $\phi_{t_j-} (N_0 \setminus \wh{W}_{t_j}  )$ is disjoint from $\wh{W}'_{t_j}$.

Assume now that there is a component $N_1 \neq N_0$ of $\N_{t_j-}$ with the property that $\phi_{t_j-} (N_1)$ intersects $\wh{W}'_{t_j}$.
Then again, since $\phi_{t_j-} (\partial \wh{W}_{t_j} ) = \partial \wh{W}'_{t_j}$ and $\phi_{t_j-}$ is injective, we must have $\phi_{t_j -} (N_1) \subset \wh{W}'_{t_j}$.
By a priori assumption \ref{item_backward_time_slice_3}(c), we know that $N_1$ must contain a $\Lambda r_{\comp}$-thick point.
So by Lemma~\ref{lem_scale_distortion}, a priori assumptions \ref{item_lambda_thick_2}, \ref{item_backward_time_slice_3}(a), (c), \ref{item_eta_less_than_eta_lin_13} and assuming (\ref{eq_scale_dist_application}),
these points must be mapped by $\phi_{t_j-}$ to $C_{\sd}^{-1}  \Lambda  r_{\comp}$-thick points in $\wh{W}'_{t_j}$.
Assuming
\[ \Lambda \geq \underline{\Lambda} (\lambda), \qquad
r_{\comp} \leq \ov{r}_{\comp} (\lambda), \]
this, however, contradicts assertion \ref{ass_11.17_cl3_d} of this claim.

Combining the conclusions of the last two paragraphs, we obtain that 
\[   \phi_{t_j-} (  \N_{t_j-} ) \cap  \wh{W}'_{t_j}
 =  \phi_{t_j-} (N_0) \cap \wh{W}'_{t_j} 
 = \phi_{t_j-} (N_0 \cap \wh{W}_{t_j})
 =  \phi_{t_j-} ( \N_{t_j-}  \cap   \wh{W}_{t_j} ) . \]
Since $\partial \wh{W}'_{t_j} = \phi_{t_j-} ( \partial \wh{W}_{t_j} )  \subset \phi_{t_j-} (\Int \N_{t_j-})$ we obtain from Alexander's theorem that $\C'_{t_j} = \wh{W}'_{t_j} \setminus \phi_{t_j-} ( \Int \N_{t_j-} ) = \wh{W}'_{t_j} \setminus \phi_{t_j-} (  \wh{W}_{t_j} \cap \Int \N_{t_j-} )$ is a 3-disk.
As $\C'_{t_j} \subset \Int \wh{W}'_{t_j}$, it is also a component of $\M'_{t_j} \setminus  \phi_{t_j-} (\Int \N_{t_j-})$.
This establishes assertion \ref{ass_11.17_cl3_f}.
\end{proof}

Choose $z' = \phi_{t_{J-J_\#}-} (z) \in \phi_{t_{J-J_\#}-}(\D\C_{J - J_\# })$.  We now show that $z'$ survives until time $t_J$ and at some time lies at the tip of an approximate Bryant soliton.

\begin{claim}
\label{cl_11_4}
If
\begin{gather*}
\eta_{\lin} \leq \ov\eta_{\lin} \,, \qquad
\delta_{\nn} \leq \ov\delta_{\nn} \,, \qquad
J_\#\geq \ul{J}_\#(\lambda, \de_{\bb})\,,\qquad 
\eps_{\can}\leq \ov\eps_{\can}( \lambda, \delta_{\bb}, J_\# )\,,\\
r_{\comp}\leq\ov r_{\comp} (\lambda) \,,\qquad
\end{gather*}
then: 
\begin{enumerate}[label=(\alph*)]
\item \label{ass_11.17_cl4_a} $z'(t)$ is defined and contained in $\wh{W}'_{t}$ for all $t\in[t_{J-J_\#},t_J]$. 
\item \label{ass_11.17_cl4_b} There is a $j_0\in \{J-J_\#,\ldots,J-1\}$ such that $(\M'_{t_{j_0}},z'(t_{j_0}))$ is $\de_{\bb}$-close to $(M_{\Bry},g_{\Bry},x_{\Bry})$ at scale $\rho(z'(t_{j_0})) < C_1(\lambda ) r_{\comp}$.
\end{enumerate}
\end{claim}

\begin{proof}
By the scale distortion Lemma~\ref{lem_scale_distortion} and a priori assumptions \ref{item_lambda_thick_2}, \ref{item_backward_time_slice_3}(a), (c), \ref{item_eta_less_than_eta_lin_13} we obtain, assuming
\[ \eta_{\lin} \leq \eta_{\lin}, \qquad
\delta_{\nn} \leq \ov\delta_{\nn}, \qquad
\Lambda \geq 2, \qquad
\eps_{\can} \leq \ov\eps_{\can} (\lambda) , \qquad
r_{\comp} \leq \ov{r}_{\comp}, \]
that $\rho(z') = \rho_1 (z' ) < C_{\sd} \rho_1 (z) < 2 C_{\sd} r_{\comp}$.

Next, recall that the scalar curvature on any $\kappa$-solution is pointwise non-decreasing in time (see assertion \ref{ass_C.1_e} of Lemma~\ref{lem_kappa_solution_properties_appendix}).
So assuming 
\[ \eps_{\can} \leq \ov\eps_{\can} (J_\#), \qquad
r_{\comp} \leq \ov{r}_{\comp}, \]
we obtain from Lemma~\ref{lem_rho_Rm_R} that at any $y \in \M'$ with $\eps_{\can} r_{\comp} < \rho (y) < 4 C_{\sd} r_{\comp}$ we have 
\begin{equation} \label{eq_dt_rho_square}
 \partial_{\mathfrak{t}} \rho^2 (y) = \partial_{\mathfrak{t}} \big( \tfrac13 R \big) < C^2_{\sd} J_\#^{-1}.
\end{equation}

Choose now $t^* \in [t_{J - J_\#}, t_J]$ maximal such that $z'(t)$ is well defined for all $t \in [t_{J - J_\#}, t^*)$.
By (\ref{eq_dt_rho_square}) we have $\rho^2 (z'(t)) < 5 C_{\sd}^2 r^2_{\comp}$ for all $t \in [t_{J-J_\# +1}, t^*)$ and therefore $\rho (z'(t)) < 10 C_{\sd} r_{\comp}$ for all such $t$.
Suppose that $t^* < t_J$.
As $\wh{W}'$ is compact, we must have $z'(t) \not\in \wh{W}'$ for $t$ close to $t^*$.
So there is a $t' \in [t_{J-J_\#}, t^*)$ such that $z(t')$ lies on the relative topological boundary of $\wh{W}'$ inside $\M'_{[t_{J-J_\#}, t_J]}$.
By assertions \ref{ass_11.17_cl3_c} and \ref{ass_11.17_cl3_e} of Claim~\ref{cl_11_3} this, however, implies that $\rho (z' (t')) > 10 C_{\sd} r_{\comp}$, contradicting our previous conclusion.
Therefore, $t^* = t_J$ and $z'(t) \in \wh{W}'$ for all $t \in [t_{J - J_\#}, t_J]$.

Since
\begin{equation*}
 \sum_{j=J-J_\#+1}^{J-1} \big( \rho (z'(t_j)) - \rho (z'(t_{j-1})) \big) = \rho (z' (t_{J-1})) - \rho (z' ) 
 > - \rho (z') > - 2C_{\sd} r_{\comp}, 
\end{equation*}
we can find a $j_0 \in \{ J - J_\# + 1, \ldots J - 1\}$ such that
\[  \rho (z'(t_{j_0})) - \rho (z'(t_{j_0-1})) > - \frac{2C_{\sd}}{J_\#-1} r_{\comp}. \]
Next, observe that $C_1^{-1}(\lambda) r_{\comp} < \rho (z'(t_{j_0})) < C_1 (\lambda) r_{\comp}$ by assertion \ref{ass_11.17_cl3_d} of Claim~\ref{cl_11_3}.
So by Lemma \ref{lem_bryant_propagate}, assuming 
$$
J_\#\geq \ul{J}_\#(\lambda, \de_{\bb}),  \qquad
\eps_{\can}\leq \ov\eps_{\can}(\lambda,\de_{\bb}), \qquad
r_{\comp} \leq \ov{r}_{\comp} (\lambda),
$$
we find that $(\M'_{t_{j_0}},z'(t_{j_0}))$ is $\de_{\bb}$-close to $(M_{\Bry},g_{\Bry},x_{\Bry})$ at scale $\rho(z'(t_{j_0}))$.
\end{proof}

Consider the component $\C'_{j_0}$ from assertion \ref{ass_11.17_cl3_f} of Claim~\ref{cl_11_3}.
We will now verify that \ref{item_geometry_cap_extension_5}(a)--(e) hold for $\C = \C_{j_0}$ and $\C' = \C'_{j_0}$, forcing the existence of an extension cap at time $t_{j_0}$.
A priori assumptions \ref{item_geometry_cap_extension_5}(a), (b) hold by Claim~\ref{cl_11_1}\ref{ass_11.17_cl1_a} and Claim~\ref{cl_11_3}\ref{ass_11.17_cl3_f}.
A priori assumption \ref{item_geometry_cap_extension_5}(c) is implied by (\ref{eq_Bryant_closeness}), assuming
\[ \delta_\# \leq \ov\delta_{\#} (\delta_{\bb}, J_\#). \]
A priori assumption \ref{item_geometry_cap_extension_5}(d) and the diameter bound on $\C'_{j_0}$ in a priori assumption \ref{item_geometry_cap_extension_5}(e) hold by Claim~\ref{cl_11_3}\ref{ass_11.17_cl3_d} and Claim~\ref{cl_11_4}\ref{ass_11.17_cl4_b}, as long as
\[ D_{\CAP} \geq C_1 (\lambda). \]
Lastly, the diameter bound on $\C_{j_0}$ in a priori assumption \ref{item_geometry_cap_extension_5}(e) follows from the fact that $\C_{j_0} \subset \wh{W}_{t_{j_0}}$ and by construction, assuming that
\[ D_{\CAP} \geq \underline{D}_{\CAP} (\lambda). \]

The conclusions of the previous paragraph combined with a priori assumption \ref{item_geometry_cap_extension_5} imply that $\C_{j_0}$ must be an extension cap, which implies that $\C_{j_0} \subset \N_{t_{j_0}+}$.
This, however, contradicts assertion \ref{ass_11.17_cl1_b} of Claim~\ref{cl_11_1} for $j = j_0 +1$, which finishes the proof.
\end{proof}

\section{Inductive step: extension of the comparison map} \label{sec_extend_comparison}
\subsection{Statement of the main result}
In this section we consider a comparison domain defined on the time-interval $[0, t_{J+1}]$, as constructed in Section~\ref{sec_inductive_step_extension_comparison_domain}, and a comparison defined on the time-interval $[0,t_J]$.
Our goal will be to extend the comparison to the time-interval $[0, t_{J+1}]$.
The following proposition will be the main result of this section.

\begin{proposition}[Extending  the comparison map by one step] \label{Prop_extend_comparison_by_one}
Suppose that
\begin{equation}
\begin{gathered} \label{eq_parameters_extend_comparison_by_one}
T > 0, \qquad
E \geq \underline{E}, \qquad
H \geq \underline{H} (E), \qquad
\eta_{\lin} \leq \ov\eta_{\lin} (E), \qquad \\
\nu \leq \ov\nu (T, E, H, \eta_{\lin}), \qquad
\delta_{\nn} \leq \ov\delta_{\nn} (T, E, H, \eta_{\lin} ), \qquad
\lambda \leq \ov\lambda, \qquad \\
D_{\CAP} > 0, \qquad
\eta_{\cut} \leq \ov\eta_{\cut} , \qquad
D_{\cut} \geq \underline{D}_{\cut} (T, E,H, \eta_{\lin}, \lambda, D_{\CAP}, \eta_{\cut} ), \qquad \\
W \geq \underline{W} (E, \lambda, D_{\cut}), \qquad 
A \geq \underline{A} (E, \lambda, W), \qquad
\Lambda \geq \underline{\Lambda} (\lambda, A), \qquad \\
\delta_{\bb} \leq \ov{\delta}_{\bb} (T, E, H, \eta_{\lin}, \lambda, D_{\CAP}, \eta_{\cut}, D_{\cut}, A, \Lambda), \qquad \\
\eps_{\can} \leq \ov\eps_{\can} (T, E, H, \eta_{\lin}, \lambda, D_{\CAP}, \eta_{\cut}, D_{\cut}, W, A, \Lambda), \qquad \\
r_{\comp} \leq \ov{r}_{\comp} (T, H, \lambda, D_{\cut}),
\end{gathered}
\end{equation}
and assume that:
\begin{enumerate}[label=(\roman*)]
\item \label{con_12.1_i} $\M, \M'$ are two $(\eps_{\can} r_{\comp}, T)$-complete Ricci flow spacetimes that each satisfy the $\eps_{\can}$-canonical neighborhood assumption at scales $(\eps_{\can} r_{\comp}, 1)$.  
\item \label{con_12.1_ii} $(\N, \{ \N^j \}_{j=1}^{J+1}, \{ t_j \}_{j=0}^{J+1})$ is a comparison domain in $\M$, which is defined over the time-interval $[0, t_{J+1}]$. 
We allow the case $J=0$.
\item \label{con_12.1_iii} $(\Cut, \phi, \{ \phi^j \}_{j=1}^J )$ is a comparison from $\M$ to $\M'$ defined on $(\N, \lb \{ \N^j \}_{j=1}^{J+1},  \lb \{ t_j \}_{j=0}^{J+1})$ over the time-interval $[0, t_J]$.
If $J = 0$, then this comparison is trivial, as explained in Definition~\ref{def_comparison}.
\item \label{con_12.1_iv} $(\N, \{ \N^j \}_{j=1}^{J+1}, \{ t_j \}_{j=0}^{J+1})$ and $(\Cut, \phi, \{ \phi^j \}_{j=1}^J )$ satisfy a priori assumptions \ref{item_time_step_r_comp_1}--\ref{item_eta_less_than_eta_lin_13} for the parameters $(\eta_{\lin}, \lb \delta_{\nn}, \lb \lambda, \lb D_{\CAP}, \lb \Lambda, \lb  \delta_{\bb}, \lb \eps_{\can}, \lb r_{\comp})$ and a priori assumptions \ref{item_q_less_than_q_bar_6}--\ref{item_apa_13} for the parameters $( T, \lb E, \lb H, \lb \eta_{\lin}, \lb \nu, \lb \lambda, \lb \eta_{\cut},  \lb  D_{\cut}, \lb W, \lb A, \lb r_{\comp})$.
\item \label{con_12.1_v} $t_{J+1} \leq T$.
\item \label{con_12.1_vi} If $J = 0$, then we assume in addition the existence of a map $\zeta : X \to \M'_0$ with the following properties.
First,  $X \subset \M_{0}$ is an open set that contains the $\delta_{\nn}^{-1}r_{\comp}$-tubular neighborhood around $\N_0$.
Second, $\zeta : X \to \M'_0$ is a diffeomorphism onto its image  that satisfies the following bounds on $X$:
\begin{align*}
 | \zeta^* g'_0 - g_0 | &\leq \eta_{\lin}, \\
 e^{HT} \rho_1^E | \zeta^* g'_0 - g_0 | &\leq \nu \ov{Q} = \nu \cdot 10^{-E-1} \eta_{\lin} r_{\comp}^E, \\
 e^{HT} \rho_1^3 | \zeta^* g'_0 - g_0 | &\leq \nu \ov{Q}^* = \nu \cdot 10^{-1} \eta_{\lin} (\lambda r_{\comp})^3. 
\end{align*}
Assume moreover that the $\eps_{\can}$-canonical neighborhood assumption holds at scales $(0,1)$ on the image $\zeta (X)$.
\end{enumerate}
Then, under the above assumptions, there are a set $\Cut^{J}$ of pairwise disjoint disks in $\M_{t_J}$, a time-preserving diffeomorphism onto its image $\phi^{J+1} : \N^{J+1} \to \M'$ and a continuous map 
\[ \ov\phi :  \N  \setminus \cup_{\DD \in \Cut \cup \Cut^{J}} \DD \to \M' \]
such that the following holds.

The tuple $(\Cut \cup \Cut^{J}, \lb \ov\phi, \lb \{ \phi^j \}_{j=1}^{J+1} )$ is a comparison from $\M$ to $\M'$ defined on $(\N , \{ \N^j \}_{j=1}^{J+1}, \{ t_j \}_{j=0}^{J+1})$ over the time-interval $[0, t_{J+1}]$.
This comparison and the corresponding domain still satisfy a priori assumptions \ref{item_time_step_r_comp_1}--\ref{item_eta_less_than_eta_lin_13} for the parameters $(\eta_{\lin}, \lb \delta_{\nn}, \lb \lambda, \lb D_{\CAP}, \lb \Lambda, \lb  \delta_{\bb}, \lb \eps_{\can}, \lb r_{\comp})$ and a priori assumptions \ref{item_q_less_than_q_bar_6}--\ref{item_apa_13} for the parameters $( T, \lb E, \lb H, \lb \eta_{\lin}, \lb \nu, \lb \lambda, \lb \eta_{\cut},  \lb  D_{\cut}, \lb W, \lb A, \lb r_{\comp})$.

Lastly, in the case $J=0$ we have $\phi^1_0 = \zeta |_{\N^1_0}$.
\end{proposition}

The proof of Proposition~\ref{Prop_extend_comparison_by_one} is divided into three steps, which are of rather different character. 
These are presented in Subsections~\ref{subsec_construction_cap_extensions}, \ref{subsec_extending_comparison} and \ref{subsec_verification_of_APAs}, respectively.

In the first step, we identify the set of disks $\Cut^J$, and construct the initial map $\phi^{J+1}_{t_J}$, at time $t_J$, so that it is defined on the union of $\N_{t_J-}$ with the extension caps, and agrees with $\phi_{t_J-}$ away from the cuts in $\Cut^J$.
Here we use the Bryant Extension Proposition, Proposition~\ref{prop_bryant_comparison_general}.

In the second step, we promote this extended map to a map $\phi^{J+1}$ that is defined on a time-interval of the form $[t_J, t^*]$, for some $t^* \in (t_J, t_{J+1}]$, by solving the harmonic map heat flow equation.
Unfortunately, at this point we cannot guarantee a priori that the harmonic map heat flow equation admits a solution on the \emph{entire} time-interval $[t_J, t_{J+1}]$, as it may develop a singularity at an earlier time.
However, we can rule out such a singularity as long as the solution satisfies certain uniform bounds.  In such a case we can indeed choose $t^* = t_{J+1}$.

In the third step, we verify that the map $\phi^{J+1}$, as constructed in the second step, satisfies a priori assumptions \ref{item_time_step_r_comp_1}--\ref{item_q_less_than_nu_q_bar_12}.
Our main focus will be on a priori assumptions \ref{item_eta_less_than_eta_lin_13}--\ref{item_q_star_less_than_q_star_bar}, as the remaining a priori assumptions follow relatively easily from our construction.
Once this is done, a priori assumption \ref{item_eta_less_than_eta_lin_13} provides sufficient control on the map $\phi^{J+1}$ to rule out the development of a singularity up to time $t^*$ and slightly after.
It thus follows a posteriori that $t^* = t_{J+1}$, which finishes the proof.

Readers interested in a more detailed description of the steps above will find further explanations embedded in Subsections~\ref{subsec_construction_cap_extensions}--\ref{subsec_verification_of_APAs}.

This section is organized as follows.
The intermediate results, Propositions~\ref{Prop_performing_cap_extensions} and \ref{Prop_extend_phi_to_t_star}, are presented in the next two subsections.
In order to reduce complexity, we have organized the discussion in each of these subsections to be independent from the remaining subsections; no assumptions are implicitly carried over to from one subsection to the next.
The last subsection (Subsection~\ref{subsec_verification_of_APAs}) contains the proof of the main proposition (Proposition~\ref{Prop_extend_comparison_by_one}).
This proof is linked to subsections~\ref{subsec_construction_cap_extensions} and \ref{subsec_extending_comparison} only via the intermediate results, Propositions~\ref{Prop_performing_cap_extensions} and \ref{Prop_extend_phi_to_t_star}, and does not depend on the details of their proofs.

As in Section \ref{sec_inductive_step_extension_comparison_domain}, we introduce parameter bounds in displayed equations.

\subsection{Extending the comparison over the extension caps} \label{subsec_construction_cap_extensions}
In this subsection, we consider a comparison domain $(\N, \{ \N^j \}_{j=1}^{J+1}, \{ t_j \}_{j=0}^{J+1})$, which is defined on the time-interval $[0,t_{J+1}]$, and a comparison $(\Cut, \lb \phi, \lb \{ \phi^j \}_{j=1}^J)$, which is defined on the time-interval $[0,t_J]$.
Based on this data, we will construct a collection of cuts $\Cut^{J}$ at time $t_J$ and a map $\wh\phi : \N_{t_J-} \cup \N_{t_J+} \to \M'_{t_J}$, which can be seen as an extension of $\phi_{t_J-}$ away from the cuts.
In Proposition~\ref{Prop_extend_phi_to_t_star}, which is the main result of the next subsection, the initial value $\phi_{t_J+}$ of the map $\phi^{J+1}$ will be taken to be the restriction of $\wh\phi$ to $\N_{t_J+}$.
In the proof of this proposition, it will turn out to be necessary that $\wh\phi$ is defined on a slightly larger domain than $\phi_{t_J+}$ due to technical reasons having to do with our process for
promoting $\phi_{t_J+}$ to later times $t > t_J$.

\begin{proposition}[Extending the comparison over the extension caps] \label{Prop_performing_cap_extensions}
Suppose that
\begin{equation} \label{eq_parameters_performing_cap_extensions}
\begin{gathered}
E \geq \underline{E} , \qquad 
 \eta_{\lin} \leq \ov\eta_{\lin}, \qquad 
 \delta_{\nn} \leq \ov\delta_{\nn}, \qquad 
\lambda \leq \ov\lambda, \qquad \\ 
D_{\cut} \geq \underline{D}_{\cut} (T, E, H, \eta_{\lin}, \lambda, D_{\CAP}, \eta_{\cut}) ,  \qquad
\Lambda \geq \underline\Lambda, \\ 
  \qquad  \delta_{\bb} \leq \ov\delta_{\bb} (T, E, H, \eta_{\lin}, \lambda, D_{\CAP}, \eta_{\cut}, D_{\cut}, A, \Lambda),   \\
   \qquad 
    \eps_{\can} \leq \ov\eps_{\can} (T, E, H, \eta_{\lin}, \lambda, D_{\CAP}, \eta_{\cut}, D_{\cut}, A, \Lambda),  \qquad  \\
    r_{\comp} \leq \ov{r}_{\comp} (T,H, \lambda, D_{\cut})
\end{gathered}
\end{equation}
and assume that assumptions \ref{con_12.1_i}--\ref{con_12.1_v} of Proposition~\ref{Prop_extend_comparison_by_one} hold and that $J \geq 1$.

Then there is a set of cuts $\Cut^{J}$ at time $t_J$, i.e. a family of pairwise disjoint $3$-disks in $\Int \N_{t_J+}$, and a diffeomorphism onto its image $\wh\phi : \N_{t_J-} \cup \N_{t_J+} \to \M'_{t_J}$ such that the following hold:
\begin{enumerate}[label=(\alph*)]
\item \label{ass_12.3_a} Each $\DD \in \Cut^{J}$ contains exactly one extension cap of the comparison domain $(\N, \lb \{ \N^j \}_{j=1}^{J+1}, \lb \{ t_j \}_{j=0}^{J+1})$ and each extension cap of this comparison domain that is in $\M_{t_J}$ is contained in one $\DD \in \Cut^{J}$.
\item \label{ass_12.3_b} $\wh\phi = \phi_{t_J-}$ on $\N_{t_J-} \setminus \cup_{\DD \in \Cut^{J}} \DD$.
\item \label{ass_12.3_c} Every cut $\DD \in \Cut^{J}$ has diameter $< D_{\cut} r_{\comp}$ and contains a $\frac1{10} D_{\cut} r_{\comp}$-neighborhood of the corresponding extension cap in $\DD$.
\item \label{ass_12.3_d} The associated perturbation $\wh{h} := \wh\phi^* g'_{t_J} - g_{t_J}$ satisfies $|\wh{h}| \leq \eta_{\lin}$ on $\N_{t_J-} \cup \N_{t_J+}$ and
\[ e^{H(T-t_J)} \rho_1^3 |\wh{h}| \leq  \eta_{\cut} \ov{Q}^* \]
on each $\DD \in \Cut^{J}$.
\item  \label{ass_12.3_e}
The $\eps_{\can}$-canonical neighborhood assumption holds at scales $(0,1)$ on the image $\wh\phi(\N_{t_J-} \cup \N_{t_J+})$.
\end{enumerate}
\end{proposition}

The main idea of the proof of this proposition is to use the Bryant Extension Proposition~\ref{prop_bryant_comparison_general} in order to construct the cuts $\DD \in \Cut^{J}$ and the map $\wh\phi$ on each $\DD$.  
The assumptions  of that  proposition hold due to a priori assumptions \ref{item_geometry_cap_extension_5} and \ref{item_q_less_than_q_bar_6}: the former implies that regions in $\M$ that are close to extension caps, as well as  the corresponding regions in $\M'$, are geometrically close to Bryant solitons;  the latter gives the bound $Q \leq \ov{Q}$ near each extension cap.

While the strategy of proof can be summarized in a relatively straightforward way, there are several technical issues that we need to address.
First, we need to argue that extension caps at time $t_J$ are positioned close enough to a tip of an almost Bryant soliton region and that those regions are far enough away from one another to allow a separate construction of $\wh\phi$ in a large neighborhood of each extension cap.
Second, we need to verify the condition under which a priori assumption \ref{item_q_less_than_q_bar_6} guarantees the bound $Q\leq \ov{Q}$.
Lastly, once the cuts $\DD$ and the extensions have been constructed on each $\DD$, we need to verify that the resulting map $\wh\phi$ satisfies all the desired properties, for example that it is a diffeomorphism onto its image.

\begin{proof}
In the following proof we will always assume, without further mention, that
\begin{equation} \label{eq_lambda_eta_small}
 \eta_{\lin}, \lambda,r_{\comp} < 10^{-2}
\end{equation}
and that 
\[ \delta_{\nn} \leq \ov\delta_{\nn} \]
is chosen small enough such that by a priori assumption \ref{item_backward_time_slice_3}(a)  we have
\begin{equation} \label{eq_rho_on_partial_N}
0.9 r_{\comp} < \rho_1 = \rho < 1.1 r_{\comp} \qquad \text{on} \qquad \partial \N_{t_J-}.
\end{equation}

By definition of the comparison domain $(\N, \{ \N^j \}_{j=1}^{J+1}, \{ t_j \}_{j=0}^{J+1})$ we know that $\N_{t_J+} \setminus \Int \N_{t_J-}$ is a disjoint union of (possibly infinitely many) extension caps $\C_i$, $i \in I$, which are $3$-disks.
A priori assumption \ref{item_geometry_cap_extension_5} implies the existence of components $\C'_i$, $i \in I$, of $\M'_{t_J} \setminus \phi_{t_J -} ( \Int \N_{t_J -})$ such that the following holds for all $i \in I$:
\begin{enumerate}[label=(\arabic*)]
\item \label{li_12.3_1} $\C'_i$ is a $3$-disk.
\item \label{li_12.3_2} $\phi_{t_J-} (\partial \C_i) = \partial \C'_i$.
\item \label{item_psi_control} \label{li_12.3_3}  There is a diffeomorphism $\psi_i :  M_{\Bry} (\delta_{\bb}^{-1}) \to W_i \subset \M_{t_J}$ such that $\psi_i (x_{\Bry}) \in \C_i$ and
\[ \big\Vert (10 \lambda r_{\comp})^{-2} \psi_i^* g_{t_J} - g_{\Bry} \big\Vert_{C^{[\delta^{-1}_{\bb}]}( M_{\Bry} (\delta_{\bb}^{-1}))} < \delta_{\bb}. \]
\item \label{li_12.3_4} \label{item_psi_prime_control} There is a diffeomorphism $\psi'_i :  M_{\Bry}(\delta_{\bb}^{-1}) \to W'_i \subset \M'_{t_J}$ such that $d_{t_J} (\psi'_i (x_{\Bry}), \lb \C'_i) \lb \leq \lb D_{\CAP} r_{\comp}$ and
\[ \big\Vert a_i^{-2} (\psi'_i)^* g'_{t_J} - g_{\Bry} \big\Vert_{C^{[\delta^{-1}_{\bb}]}( M_{\Bry} (\delta_{\bb}^{-1}))} < \delta_{\bb}\]
for some scale $a_i \in [D_{\CAP}^{-1} r_{\comp}, D_{\CAP} r_{\comp}]$.
\item \label{li_12.3_5} $\C_i$ and $\C'_i$ have diameter $\leq D_{\CAP} r_{\comp}$.
\end{enumerate}
Since $\phi_{t_J-} : \N_{t_J-} \to \phi (\N_{t_J-}) \subset \M'_{t_J}$ is a diffeomorphism onto its image, we obtain from  items \ref{li_12.3_1} and \ref{li_12.3_2} that the components $\C'_i$, $i \in I$, are pairwise distinct.

We will assume in the following  that
\[ \delta_{\bb} \leq \ov\delta_{\bb} \]
is chosen sufficiently small such that for all $i \in I$
\begin{enumerate}[label=(\arabic*), start = 6]
\item \label{item_length_curve_distortion} lengths of curves in $M_{\Bry} (\delta_{\bb}^{-1})$ are distorted by $\psi_i$ by a factor of at least $9 \lambda r_{\comp}$ and at most $11 \lambda r_{\comp}$.
\end{enumerate} 

We now fix a constant $D_\# < \infty$ whose value we will determine in the course of the proof.
This constant controls the size of the neighborhood around each extension cap $\C_i$ in which we will  carry out our construction of $\wh\phi$.
More specifically, each such neighborhood will be of the form $\psi_i ( M_{\Bry} (D_\#)) \supset \C_i$; in particular, its diameter will be approximately $D_\# \cdot 10 \lambda r_{\comp}$.
Outside these neighborhoods, we will set $\wh\phi := \phi_{t_J-}$ and we will choose the cuts $\Cut^{J}$ to be disks that are contained in the corresponding $\psi_i ( M_{\Bry} (D_\#))$.

As we proceed with the proof of Proposition \ref{Prop_performing_cap_extensions}, we will establish several claims, which hold under certain bounds on the parameters.
At any point in the proof we will assume that the parameter bounds of the preceding claims hold, so that we can apply the assertions of these claims without restating the parameter bounds.

We first show that, under certain assumptions on our parameters, the neighborhoods $\psi_i (M_{\Bry}(D_\#))$ are pairwise disjoint and the extension caps $\C_i$ lie in bounded domains of the form $\psi_i (M_{\Bry} (D_0(\lambda)))$.

\begin{claim}
\label{cl_12.2_1}
There is a constant $D_0 = D_0 (\lambda) < \infty$ such that if
\begin{multline*}
 \delta_{\nn} \leq \ov\delta_{\nn}, \qquad \Lambda \geq \underline{\Lambda}, \qquad  D_\# \geq D_0(\lambda), \qquad\delta_{\bb} \leq \ov\delta_{\bb} ( \lambda, \Lambda, D_\#), \\\qquad r_{\comp} \leq \ov{r}_{\comp} (D_\#), 
\end{multline*}
then $D_\# \leq \delta_{\bb}^{-1}$ and the images $\psi_i ( M_{\Bry} (D_\#) )$, $i \in I$, are pairwise disjoint.
Moreover, for all $i \in I$ we have
\begin{equation} \label{eq_CiinD0}
 \C_i \subset \psi_i (M_{\Bry} (D_0)), \qquad \psi_i (M_{\Bry} (D_\#)) \subset W_i \subset \N_{t_J-}  \cup \C_i
\end{equation} 
and
\begin{multline} \label{eq_psi_i_scale_distortion}
 9 \lambda r_{\comp} \rho(x) < \rho (\psi_i (x)) = \rho_1 (\psi_i (x)) < 11 \lambda r_{\comp} \rho(x) \\ \quad \text{for all} \quad x \in M_{\Bry} (D_\#). 
\end{multline}
\end{claim}

\begin{proof}
Fix some index $i \in I$. 
The bound (\ref{eq_psi_i_scale_distortion}) follows immediately from \ref{li_12.3_3}, provided that
\[ \delta_{\bb} \leq \ov\delta_{\bb} ( D_\#), \qquad r_{\comp} \leq \ov{r}_{\comp} (D_\#). \]
Next, we invoke the Bryant Slice Lemma~\ref{lem_Bryant_slice}, for $X = \N_{t_J+}$, assuming
\[ 
\delta_{\nn} \leq \ov\delta_{\nn}, \qquad 
\Lambda \geq \underline\Lambda, \qquad
\delta_{\bb} \leq \ov\delta_{\bb} (\lambda, \Lambda). \]
Assumptions \ref{con_8.42_i}--\ref{con_8.42_iv} of this lemma hold due to Definition~\ref{def_comparison_domain} and a priori assumptions \ref{item_backward_time_slice_3}(a)--(c).
The first inclusion in (\ref{eq_CiinD0}) is a restatement of assertion \ref{ass_8.42_d} of the Bryant Slice Lemma and the second string of inclusions is a consequence of assertion \ref{ass_11.17_cl4_a}.

Finally, assume that $\psi_{i_1} ( M_{\Bry} (D_\#)) \cap \psi_{i_2} ( M_{\Bry} (D_\#)) \neq \emptyset$ for some $i_1 \neq i_2$.
Then, assuming
\[ \delta_{\bb} \leq 10^{-2} D_\#^{-1}, \]
we must have $\C_{i_2} \subset \psi_{i_2} ( M_{\Bry} (D_\#)) \subset W_{i_1}$, contradicting the second string of inclusions of (\ref{eq_CiinD0}).
This finishes the proof of the claim.
\end{proof}

In the second claim we show that the neighborhoods $\psi_i (M_{\Bry} (D_\#))$ around the extension caps $\C_i$ are mapped by $\phi_{t_J-}$ into the regions $W'_i$, which are geometrically close to Bryant solitons. \

\begin{claim}
\label{cl_12.2_3}
If
\[  \delta_{\bb} \leq \ov\delta_{\bb} (\lambda, D_{\CAP}, D_\#), \]
then $D_\#<\delta_{\bb}^{-1}$ and $\phi_{t_J-} ( \psi_i ( M_{\Bry} (D_\#)) \setminus \Int \C_i ) \subset W'_i$ for all $i \in I$.
\end{claim}

\begin{proof}
Fix an index $i \in I$ and a point $y \in \psi_i (M_{\Bry} (D_\#)) \setminus \Int \C_i$.
By Property \ref{item_length_curve_distortion} above, we can find a continuous path $\gamma : [0,1] \to \N_{t_J-}$ between $y$ and a point $z \in \partial \C_i$ whose length is at most $11\lambda r_{\comp} \cdot 2D_\#$.
Assuming (\ref{eq_lambda_eta_small}), and using a priori assumption \ref{item_eta_less_than_eta_lin_13}, we find that the length of its image $\phi_{t_J-} \circ \gamma$ is bounded by $100 D_\#  \lambda r_{\comp}$.
So since $\phi_{t_J-} (\partial \C_i) = \partial \C'_i$, we have
\[ d_{t_J} (\phi_{t_J-} (y), \partial \C'_i) \leq 100 D_\# \lambda r_{\comp}. \]
On the other hand, by Properties \ref{li_12.3_4}, \ref{li_12.3_5} above we have
\[ \partial \C'_i \subset  B(\psi'_i(x_{\Bry}),  2D_{\CAP} r_{\comp}). \]
So if
\[ \delta_{\bb} \leq \ov\delta_{\bb} (\lambda, D_{\CAP}, D_\#), \]
then we obtain that $y \in W'_i$, as desired.
\end{proof}

This concludes our discussion on the relative positions of the components $\C_i, \C'_i$ and the images of the maps $\psi_i$ and $\psi'_i$.
We will now focus on the associated perturbation $h_{t_J-} = \phi^*_{t_J-} g'_{t_J} - g_{t_J}$.
In the next claim, and its proof, we use the bound $Q \leq \ov{Q}$, as asserted by a priori assumption \ref{item_q_less_than_q_bar_6}, to deduce a bound on the weighted norm $\rho^E |h_{t_J-}|$ on $\psi_i (M_{\Bry} (D_\#)) \setminus \C_i$.
Using a standard local derivative estimate, we will also deduce similar weighted bounds on covariant derivatives of the form $\nabla^m h_{t_J-}$.

\begin{claim}
\label{cl_12.2_4}
There is a constant $C = C (E) < \infty$ such that if
\begin{gather*}
  \eta_{\lin} \leq \ov\eta_{\lin}, \qquad
\delta_{\nn} \leq \ov\delta_{\nn}, \qquad 
\lambda \leq \ov\lambda, \qquad
\Lambda\geq \ul\Lambda \qquad
 \delta_{\bb} \leq \ov\delta_{\bb} (\lambda, D_{\cut}, A, \Lambda, D_\#), \qquad \\
 \eps_{\can} \leq \ov\eps_{\can} (\lambda, D_{\cut},A,\Lambda, D_\#), \qquad 
 r_{\comp} \leq \ov{r}_{\comp} , 
\end{gather*}
then for the associated perturbation $h_{t_J-} = \phi_{t_J-}^* g'_{t_J} - g_{t_J}$ the following holds for all $i \in I$ and all $m = 0, 1, \ldots, 4$:
\begin{equation} \label{eq_nab_h_R_bound}
 e^{H (T - t_J)} \rho^E |\nabla^m h_{t_J-}| \leq  C \lambda^{-m} r_{\comp}^{-m+E }  \qquad  \text{on} \qquad \psi_i (M_{\Bry} (D_0+1, D_\#-1)).
\end{equation}
\end{claim}

As mentioned earlier, the main idea of the proof of this claim is to invoke the bound $Q \leq \ov{Q}$ from a priori assumption \ref{item_q_less_than_q_bar_6}.
However, this bound is predicated on the remoteness of cuts.
In order to verify this remoteness, we will invoke Lemma~\ref{lem_boundary_far_from_cut}.

\begin{proof}
Fix an index $i \in I$ and a point $x \in \psi_i (M_{\Bry} (D_0 + 1, D_\# -1))$ for the remainder of this proof. Then by Claim~\ref{cl_12.2_1}, (\ref{eq_lambda_eta_small}) and Property \ref{item_length_curve_distortion} above, we have
\begin{equation} \label{eq_BxlambdainNN}
 B(x, \lambda r_{\comp} ) \subset \psi_i (M_{\Bry} (D_0, D_\#)) \subset \N_{t_J-}. 
\end{equation}
So for the corresponding parabolic neighborhood we have
\[ P(x, \lambda r_{\comp}) \subset \N^J_{(t_{J-1}, t_J]}  \subset \N \setminus \cup_{\DD \in \Cut} \DD\,,\]
since $\DD \subset \M_{[0,t_{J-1}]}$ for all $\DD \in \Cut$.

Our goal will be to use a priori assumption \ref{item_q_less_than_q_bar_6} to deduce the bound $Q \leq \ov{Q}$ on $P(x, \lambda r_{\comp})$.
So consider a point $y \in P(x, \lambda r_{\comp})$ and set $t' := \mathfrak{t} (y)$.
We now claim that for an appropriate choice of constants we have
\begin{equation} \label{eq_P_large_no_cuts}
  P(y, A \rho_1 (y)) \cap \DD = \emptyset \qquad \text{for all} \qquad \DD \in \Cut. 
\end{equation}
To see this choose a  point $z \in \partial \C_i \subset \partial \N_{t_J-}$ nearest to  $y(t_J)$.
Then, by (\ref{eq_BxlambdainNN}) and Properties \ref{li_12.3_3}, \ref{item_length_curve_distortion} above, 
\begin{equation} \label{eq_dist_x_z}
 d_{t_J} (y(t_J), z) < 11 D_\# \lambda r_{\comp}.
\end{equation}
Let $z' := z(t')$.
Since $t' \in (t_{J-1}, t_J]$, we can use the curvature bound on the product domain $\N^J$ from a priori assumption \ref{item_lambda_thick_2} to derive a distortion estimate of the minimizing geodesic between $y(t_J)$ and $z$ over the time-interval $[t', t_J]$.
Since $t_J-t'\leq (\lambda r_{\comp})^2$, we obtain that for some universal constant $C'_1 < \infty$
\begin{equation} \label{eq_dist_x_z_prime}
 d_{t'} (y, z') < 11 e^{C'_1} D_\# \lambda r_{\comp}.
\end{equation}
Next, let us apply bounded curvature at bounded distance, Lemma~\ref{lem_bounded_curv_bounded_dist}, at $z$, along with  (\ref{eq_dist_x_z}), while assuming
\[ \eps_{\can} \leq \ov\eps_{\can} ( D_\#) . \]
We obtain a constant $C'_2 = C'_2 ( D_\#) < \infty$ such that by (\ref{eq_rho_on_partial_N})
\[ C'_2 \rho_1 (y) \geq   \rho_1 (z) \geq 0.9  r_{\comp}. \]
Combining this with (\ref{eq_dist_x_z_prime}), yields that
\[ d_{t'} (y, z') < D' \rho_1(y) \]
for some $D' = D' ( D_\#) < \infty$.
So if $t' < t_J$, then $B(y, D' \rho_1 (y)) \not\subset \N_{t'-}$.
We can now apply Lemma~\ref{lem_boundary_far_from_cut} (Boundaries and cuts are far apart) along with a priori assumption \ref{item_cut_diameter_less_than_d_r_comp_11}, assuming
\begin{gather*}
 \delta_{\nn} \leq \ov\delta_{\nn}, \qquad
 \lambda \leq \ov\lambda, \qquad
 \Lambda \geq \underline\Lambda, \qquad
\delta_{\bb} \leq \ov\delta_{\bb} ( \lambda, D_{\cut}, A, \Lambda, D'(\lambda, D_\#)), \qquad \\
\eps_{\can} \leq \ov\eps_{\can} ( \lambda, D_{\cut}, A, \Lambda, D'(\lambda, D_\#)), 
 \qquad r_{\comp} \leq \ov{r}_{\comp}, 
\end{gather*}
and obtain (\ref{eq_P_large_no_cuts}).
The case $t' = t_J$ follows from the case $t' < t_J$ by continuity.

Using (\ref{eq_P_large_no_cuts}) and a priori assumption \ref{item_q_less_than_q_bar_6}, we can now deduce that 
\begin{equation} \label{eq_APA7_repeat}
 e^{H (T-t_J)} \rho_1^E (y) |h_{t'-}(y)| \leq Q(y) \leq \ov{Q} =  10^{-E - 1} \eta_{\lin} r_{\comp}^{E}  . 
\end{equation}
Next, we apply bounded curvature at bounded distance, Lemma~\ref{lem_bounded_curv_bounded_dist}, at $x$, along with a priori assumption \ref{item_lambda_thick_2}, while assuming
\[ \eps_{\can} \leq \ov\eps_{\can} (\lambda). \]
We obtain that there is a universal constant $C'_3 < \infty$ such that
\begin{equation} \label{eq_R_osc_P}
  \rho (x) = \rho_1 (x) \leq  C'_3 \rho_1(y) .
\end{equation}
The equality statement follows from (\ref{eq_psi_i_scale_distortion}).
Combining (\ref{eq_APA7_repeat}) with (\ref{eq_R_osc_P}) yields
\begin{equation} \label{eq_m_0_bound}
e^{H (T-t_J)} \rho^E (x) |h_{t_J-} (y)| \leq  10^{-E-1} C_3^{\prime E}  \eta_{\lin} r_{\comp}^{E}. 
\end{equation}
If $y=x$, then this bound implies (\ref{eq_nab_h_R_bound}) for $m = 0$.
The bounds on the higher derivatives follow from (\ref{eq_m_0_bound}) using (\ref{eq_R_osc_P}), a priori assumption \ref{item_eta_less_than_eta_lin_13}, Shi's estimates and standard local gradient estimates for the Ricci-DeTurck flow (see also Lemma~\ref{lem_loc_gradient_estimate}), assuming
\[ \eta_{\lin} \leq \ov\eta_{\lin}. \]
This finishes the proof.
\end{proof}

We will now apply the Bryant Extension Proposition~\ref{prop_bryant_comparison_general},  to the restrictions of the map $\phi_{t_J-}$ to each $W_i$, for suitably chosen $D_\#$.
The resulting maps, which will be denoted by $\td\phi_{i}$, will be only defined on the domains $\psi_i (M_{\Bry} (D_\#))$, but will be equal to $\phi_{t_J-}$ near the boundaries of these domains.

\begin{claim}
\label{cl_apply_bryant_comparison}
If
\begin{equation*} \label{eq_parameters_for_Bryant_comparison}
\begin{gathered}
 E \geq \underline{E},    \qquad 
 \eta_{\cut} \leq \ov\eta_{\cut}, \qquad
 D_\#  \geq \underline{D}_\# (T, E, H, \eta_{\lin}, \lambda, D_{\CAP}, \eta_{\cut} ),  \\
  \delta_{\bb} \leq \ov\delta_{\bb} (T, E, H, \eta_{\lin}, \lambda, D_{\CAP}, \eta_{\cut} , D_\# ),
\end{gathered}
\end{equation*}
then for each $i \in I$ there is a diffeomorphism onto its image
\begin{equation*} \label{eq_wh_phi_def}
 \td{\phi}_{i} : \psi_i (M_{\Bry}(D_\# -1)) \longrightarrow W'_i 
\end{equation*}
and a $3$-disk
\[ \DD_i := \psi_i \big( \ov{M_{\Bry} (D_\# -  2)} \big) \subset \M_{t_J} \]
such that the following holds:
\begin{enumerate}[label=(\alph*)]
\item \label{ass_12.3_cl4_a} $\C_i \subset \Int \DD_i$.
\item \label{ass_12.3_cl4_b} $\td{\phi}_{i} = \phi_{t_J-}$ on $\psi_i (M_{\Bry} (D_\# -1 )) \setminus \Int \DD_i$. 
\item \label{ass_12.3_cl4_c} The perturbation $\td{h}_i := \td{\phi}^*_i g'_{t_J} - g_{t_J}$ satisfies the following bounds on $\DD_i$
\[
\qquad |\td{h}_i | \leq  \eta_{\lin}, \qquad
 e^{H (T- t_J)} \rho_1^3 | \td{h}_{i}| \leq \eta_{\cut} \ov{Q}^* = \eta_{\cut} \cdot 10^{-1} \eta_{\lin} (\lambda r_{\comp})^3 . 
\]
\item \label{ass_12.3_cl4_d} $\td{\phi}_i (\DD_i) = \phi_{t_J-} (\DD_i \cap \N_{t_J-}) \cup \Int \C'_i$.
\item \label{ass_12.3_cl4_e} $\DD_i$ contains the $8 \lambda D_\# r_{\comp}$-tubular neighborhood around $\C_i$.
\end{enumerate}
\end{claim}

\begin{proof}
Fix some $i \in I$.
Set $b_i := a_i (10 \lambda r_{\comp})^{-1}$ and notice that Property \ref{item_psi_prime_control} from above gives
\[ b_i \in [ (10\lambda)^{-1} D_{\CAP}^{-1}, \lb (10\lambda)^{-1} D_{\CAP}]. \]
Assume that
\begin{equation} \label{eq_DsharpD0}
 D_\# \geq 2 (D_0 (\lambda) + 1) + 1 
\end{equation}
and consider the map
\[ \phi^{\circ}_i := \psi^{\prime -1}_i \circ \phi_{t_J-} \circ \psi_i : M_{\Bry} (\tfrac12 (D_\#-1), D_\# -1) \longrightarrow M_{\Bry} (\delta_{\bb}^{-1}), \]
which is well-defined by Claims~\ref{cl_12.2_1} and \ref{cl_12.2_3}.
Let $g_i^\circ := (10 \lambda r_{\comp})^{-2} \psi_i^* g_{t_J}$ and $g_i^{\prime \circ} := (10 \lambda r_{\comp})^{-2} \psi_i^{\prime *} g'_{t_J}$ be the pull-back metrics on $W_i$ and $W'_i$ to $M_{\Bry} (D_\#)$.
Notice that these pull-backs are close to $g_{\Bry}$ and $b_i^2 g_{\Bry}$, respectively, by Properties \ref{item_psi_control} and \ref{item_psi_prime_control} above.
Rescaling (\ref{eq_nab_h_R_bound}) from Claim~\ref{cl_12.2_4} by $(10\lambda r_{\comp})^{-1}$ yields for  $h^\circ_i := (\phi^\circ_i)^* g_i^{\prime\circ} -  g^{ \circ}_i$ 
\[ \rho^E \big|\nabla^m h^\circ_i \big|_{g^\circ_i} \leq  C (E) e^{-H (T-t_J)} (10\lambda)^{-E} \cdot 10^m \leq C (E) (10\lambda)^{-E} \cdot 10^4, \]
for all $m = 0,1, \ldots, 4$.
Here we have used $t_J \leq T$.
Note that $\rho$ is taken with respect to $g^\circ_i$.

We now apply the Bryant Extension Proposition~\ref{prop_bryant_comparison_general}  with $D=D_\#-1$,  $b = C(E) (10 \lambda)^{-E} \cdot 10^4$, $\beta = e^{-H (T-t_J)} \eta_{\cut} 10^{-4} \cdot \eta_{\lin}  \cdot b^{-1}$, $C = \max \{  (10\lambda)^{-1} D_{\CAP}, b \}$,  $\phi=\phi^{\circ}_i$, $g=g_i^\circ$, $g'=g_i^{\prime\circ}$.
We obtain that if
\begin{multline*}
 E \geq \underline{E}, \qquad
D_\# \geq \underline{D}_\# (T, E, H, \eta_{\lin}, \lambda, D_{\CAP}, \eta_{\cut} ), \qquad \\
\delta_{\bb} \leq \ov\delta_{\bb} (T, E, H, \eta_{\lin}, \lambda, D_{\CAP}, \eta_{\cut} , D_\#), 
\end{multline*}
then there is a smooth map $\td\phi^\circ_i : M_{\Bry} (D_\# - 1) \to M_{\Bry}(\delta_{\bb}^{-1})$ with
\begin{equation} \label{eq_tdphicircphicirc}
 \td\phi^\circ_i = \phi^\circ_i \qquad \text{on} \qquad M_{\Bry} (D_\# -2, D_\#-1),
 \end{equation}
such that for $\td{h}^\circ_i := (\td\phi^\circ_i)^* g^{\prime\circ}_i -  \psi^{\prime *} g^{ \circ}_i$ we have
\begin{equation} \label{eq_h_i_circ_bound}
 \rho^3 \big| \td{h}^\circ_i \big|_{g^\circ_i} \leq \eta_{\cut} \cdot e^{-H (T-t_J)} \cdot  10^{-4} \eta_{\lin} . 
\end{equation}
Now set $\td\phi_i := \psi'_i \circ \td\phi_i^\circ \circ \psi_i^{-1}$.
Then assertion \ref{ass_12.3_cl4_b} holds due to (\ref{eq_tdphicircphicirc}).
Rescaling (\ref{eq_h_i_circ_bound}) by $10 \lambda r_{\comp}$ implies the second bound in assertion \ref{ass_12.3_cl4_c}.
The first bound in assertion \ref{ass_12.3_cl4_c} follows from the second assuming
\[ \eta_{\cut} \leq \ov\eta_{\cut}, \qquad
\delta_{\bb} \leq \ov\delta_{\bb}. \]
Assertion \ref{ass_12.3_cl4_a} follows from Claim~\ref{cl_12.2_1} and (\ref{eq_DsharpD0}).

To see assertion \ref{ass_12.3_cl4_d} observe first that by assertion \ref{ass_12.3_cl4_b} and (\ref{eq_CiinD0}) from Claim~\ref{cl_12.2_1} we have
\[ \partial \big( \td{\phi}_i (\DD_i) \big) = \td{\phi}_i (\partial \DD_i) = \phi_{t_J-} (\partial \DD_i) = \partial \big( \phi_{t_J-} (\DD_i \cap \N_{t_J-}) \cup \Int \C'_i \big) . \]
So the smooth domains on both sides of the equation in assertion \ref{ass_12.3_cl4_d} share the same boundary and by assertion \ref{ass_12.3_cl4_b} these domains lie on the same side of this boundary.
So they have to agree.

Assertion \ref{ass_12.3_cl4_e} follows for
\[ D_\# \geq \underline{D}_\# (\lambda) , \qquad \delta_{\bb} \leq \ov\delta_{\bb} (\lambda) \]
from (\ref{eq_CiinD0}) in Claim~\ref{cl_12.2_1} and Property \ref{item_length_curve_distortion} from above.
\end{proof}

Next, we combine the maps $\td\phi_i$ and $\phi_{t_J-}$ to a map $\wh\phi : \N_{t_J-} \cup \N_{t_J+} \to \M'_{t_J}$. 
To do this, recall that by Claim~\ref{cl_12.2_1}, the subsets $\psi_i ( M_{\Bry} (D_\#))$, $i \in I$, are pairwise disjoint.
So the 3-disks $\DD_i$, $i \in I$, are pairwise disjoint as well.
Moreover, recall that by Claim~\ref{cl_12.2_1} and Claim~\ref{cl_apply_bryant_comparison}\ref{ass_12.3_cl4_a} we have
\[ \N_{t_J-} \cup \N_{t_J+} = \N_{t_J-} \cup_{i \in I} \Int \C_i = \N_{t_J-} \cup_{i \in I} \DD_i . \]
Therefore we can define $\wh\phi : \N_{t_J-} \cup \N_{t_J+} \to \M'_{t_J}$ as follows:
\[ \wh\phi := \begin{cases} \td{\phi}_{i}  \qquad & \text{on each} \qquad \DD_i , \qquad i \in I \\ \phi_{t_J-} \qquad & \text{on\phantom{ each}} \qquad \N_{t_J -} \setminus \cup_{i \in I} \DD_i \end{cases}  \]

\begin{claim}
$\wh\phi$ is a diffeomorphism onto its image.
\end{claim}

\begin{proof}
By assertions \ref{ass_12.3_cl4_a} and \ref{ass_12.3_cl4_b} of Claim~\ref{cl_apply_bryant_comparison} we know that $\wh\phi$ is smooth and has non-degenerate differential.
Next we argue that $\wh\phi$ is injective.
To see this, observe that the maps $\td\phi_{i}$, $i \in I$, and $\phi_{t_J-}$ are each injective.
So it suffices to show that the images $\wh{\phi}_i (\DD_i)$, $i \in I$, and $\phi_{t_J-} ( \N_{t_J-} \setminus \cup_{i \in I} \DD_i)$ are pairwise disjoint.
Using Claim~\ref{cl_apply_bryant_comparison}\ref{ass_12.3_cl4_d} and the fact that the 3-disks $\DD_i$, as well as the 3-disks $\C'_i$, $i \in I$, are pairwise disjoint, it follows immediately that the images $\td{\phi}_i (\DD_i)$, $i \in I$, are pairwise disjoint.
Similarly, using Claim~\ref{cl_apply_bryant_comparison}\ref{ass_12.3_cl4_d}, we have for all $i \in I$
\begin{multline*}
\td{\phi}_i (\DD_i) \cap \phi_{t_J-} ( \N_{t_J-} \setminus \cup_{i \in I} \DD_i ) \\
= \big( \phi_{t_J-} (\DD_i \cap \N_{t_J-}) \cup \Int \C'_i \big) \cap \phi_{t_J-} ( \N_{t_j-} \setminus \cup_{i \in I} \DD_i ) = \emptyset,
\end{multline*}
as desired.

So $\wh\phi$ is an injective smooth map with non-degenerate differential.
In order to see that $\wh\phi$ is even a diffeomorphism onto its image, it suffices to show that $\wh\phi^{-1}:\im\wh\phi\ra \N_{t_J-}\cup\N_{t_J+}$ is continuous, i.e. for any sequence $x_k \in \N_{t_J-} \cup \N_{t_J+}$ and any point $x_\infty \in \N_{t_J-} \cup \N_{t_J+}$ if $\lim_{k \to \infty} \wh\phi (x_k) = \wh\phi(x_\infty)$, then $\lim_{k \to \infty} x_k = x_\infty$ itself.
This can be seen as follows:
If $x_\infty$ lies in the interior of $\N_{t_J-} \cup \N_{t_J+}$, then we are done by the  inverse function theorem and the fact that $\wh\phi$ is injective and has non-degenerate differential.  So assume that 
\begin{equation} \label{eq_x_infty_far_from_W}
 x_\infty \in \partial ( \N_{t_J-} \cup \N_{t_J+})   = \partial \N_{t_J-} \setminus \cup_{i \in I} \partial \C_i = \partial \N_{t_J-} \setminus \cup_{i \in I} W_i. 
\end{equation}
The first equality follows from Definition~\ref{def_comparison_domain}(\ref{pr_7.1_3}) and the last equality follows from (\ref{eq_CiinD0}) in Claim~\ref{cl_12.2_1}.
If for some $k$ we have $x_k \in \DD_{i_k}$ for some $i_k \in I$, then by Claim~\ref{cl_apply_bryant_comparison}\ref{ass_12.3_cl4_d}, by the construction of $\DD_{i_k}$ and by a priori assumption \ref{item_eta_less_than_eta_lin_13} and (\ref{eq_lambda_eta_small}) a ball of uniform radius around $\wh\phi(x_k)$ must still be contained in  $\wh\phi (\psi_{i_k} (M_{\Bry} (D_\#))) \subset \wh\phi ( W_{i_k})$.
Therefore, by (\ref{eq_x_infty_far_from_W}), the distance $d_{t_J} (\wh\phi (x_\infty), \wh\phi (x_k) )$ must be bounded from below by a uniform constant.
It follows that for large $k$ we have $x_k \in \N_{t_J-} \setminus \cup_{i \in I} \DD_i$, and thus $\wh\phi (x_k) = \phi_{t_J-} (x_k)$ by Claim~\ref{cl_apply_bryant_comparison}\ref{ass_12.3_cl4_b}.
Since $\phi_{t_J-}$ is a diffeomorphism onto its image, we must have $\lim_{k \to \infty} x_k = x_\infty$, which proves our claim.
\end{proof}

Now let
\[ \Cut^{J} := \{ \DD_i \;\; : \;\; i \in I \}. \]
Then assertion \ref{ass_12.3_a} of this proposition holds due to Claim~\ref{cl_apply_bryant_comparison}\ref{ass_12.3_cl4_a}.
Assertion \ref{ass_12.3_b} holds by Claim~\ref{cl_apply_bryant_comparison}\ref{ass_12.3_cl4_b} and by the construction of $\wh\phi$.
Assertion \ref{ass_12.3_d} of the proposition follows from Claim~\ref{cl_apply_bryant_comparison}\ref{ass_12.3_cl4_c} and  priori assumption \ref{item_eta_less_than_eta_lin_13}.
For assertion \ref{ass_12.3_e} recall that by Claim~\ref{cl_apply_bryant_comparison}\ref{ass_12.3_cl4_d} we have $\wh\phi ( \N_{t_J-} \cup \N_{t_J+} ) = \phi_{t_J-} (\N_{t_J-}) \cup_{i \in I} \C'_i$ and $\C'_i \subset W'_i$ for all $i \in I$.
By a priori assumption \ref{item_eta_less_than_eta_lin_13} we know that the $\eps_{\can}$-canonical neighborhood assumption holds at scales $(0,1)$ on $\phi_{t_J-} (\N_{t_J-})$ and by Property~\ref{li_12.3_4} above we have $\rho > \frac12 D_{\CAP}^{-1} r_{\comp} > \eps_{\can} r_{\comp}$ on $W'_i$ for all $i \in I$, assuming
\[ \delta_{\bb} \leq \ov\delta_{\bb}, \qquad
\eps_{\can} \leq \ov\eps_{\can} (D_{\CAP}). \]
Therefore, the $\eps_{\can}$-canonical neighborhood assumption holds at scales $(0,1)$ on $W'_i$ as well, which implies assertion \ref{ass_12.3_e}.

Lastly, we argue that assertion \ref{ass_12.3_c} holds if we choose
\begin{equation} \label{eq_D_cut_D_sharp}
 D_{\cut} = 22 \lambda D_\#. 
\end{equation}
Fix some $i \in I$.
By Property~\ref{item_length_curve_distortion} from the beginning of this proof, we have $\diam \DD_i < 2 \cdot 11 \lambda D_\# r_{\comp}= D_{\cut} r_{\comp}$.
On the other hand, Claim~\ref{cl_apply_bryant_comparison}\ref{ass_12.3_cl4_e} states that $\DD_i$ contains a $8 \lambda D_\# r_{\comp}$-tubular neighborhood around $\C_i$ and $8 \lambda D_\# \geq \frac1{10} 22 \lambda D_\# = \frac1{10} D_{\cut}$.

Lastly, let us review the choice of parameters.
In the course of the proof, we have introduced the auxiliary parameter $D_\#$, which is related to $\lambda$ and $D_{\cut}$ via (\ref{eq_D_cut_D_sharp}).
Once $\lambda$ has been fixed, any lower bound on $D_\#$ implies a lower bound on $D_{\cut}$, as indicated in (\ref{eq_parameters_performing_cap_extensions}).
After fixing $D_{\cut}$, the auxiliary parameter $D_\#$ can be viewed as a constant of the form $D_\# (\lambda, D_{\cut})$.
This constant influences the choices of $\delta_{\bb}, \eps_{\can}, r_{\comp}$.
So these parameters are bounded in terms of $\lambda, D_{\cut}$, as shown in (\ref{eq_parameters_performing_cap_extensions}).

This completes the proof of Proposition~\ref{Prop_performing_cap_extensions}.
\end{proof}

\subsection{Extending the comparison map past time \texorpdfstring{$t_J$}{t\_J}} \label{subsec_extending_comparison}
The goal of this subsection is to evolve the map $\wh\phi$, as constructed in Proposition~\ref{Prop_performing_cap_extensions}, forward in time by the harmonic map heat flow.
More specifically, we consider again a comparison domain $(\N, \lb \{ \N^j \}_{j=1}^{J+1}, \lb \{ t_j \}_{j=1}^{J+1})$, defined over the time-interval $[0, t_{J+1}]$, and a comparison $( \Cut, \lb \phi, \lb \{ \phi^j \}_{j=1}^J )$ from $\M$ to $\M'$ defined over the time-interval $[0,t_J]$.
We moreover consider the map $\wh\phi : \N_{t_J-} \cup \N_{t_J+} \to \M'$ from Proposition~\ref{Prop_performing_cap_extensions}.
We will then promote the map $\wh\phi |_{\N_{t_J+}}$ to a map $\phi^{J+1} : \N^{J+1}_{[t_J, t^*]} \to \M'$, which is defined on a time-interval of the form $[t_J, t^*]$, where $t^* \in (t_J, t_{J+1}]$.
In this subsection we will not be able to guarantee that $t^* = t_{J+1}$ --- in fact $t^*$ may be quite close to $t_J$ --- since we will only solve the harmonic map heat flow until $|h|$ reaches a certain threshold.
However, we will find that if $|h|$ does not reach this threshold on the time-interval $[t_J, t^*]$, then in fact $t^* = t_{J+1}$.
In the next subsection, we will then deduce various bounds on $|h|$, which will imply that $|h|$ stays below this threshold.
Hence, it will follow that $t^* = t_{J+1}$ and so $\phi^{J+1}$ can indeed be used to extend the comparison $( \Cut, \phi, \{ \phi^j \}_{j=1}^J )$ to the time-interval $[0,t_{J+1}]$.

In the course of our construction, we will also discuss the case $J = 0$, i.e. the case in which $\phi^{J+1}$ is the comparison map in the first time-step.
In this case, the comparison $( \Cut, \phi, \{ \phi^j \}_{j=1}^J )$ is empty to start with and Proposition~\ref{Prop_performing_cap_extensions} does not apply.
Instead, we will assume in this case that $\wh\phi$ is the initial map $\zeta$, as introduced in the assumptions of Proposition~\ref{Prop_extend_comparison_by_one}.

Let us now state our main result of this subsection.

\begin{proposition}[Extending the comparison map until we lose control] \label{Prop_extend_phi_to_t_star}
If 
\begin{equation} \label{eq_parameters_extend_phi_to_t_star}
\begin{gathered}
E > 2, \qquad 
 F > 0, \qquad
H \geq \underline{H} (E), \qquad 
 \eta_{\lin} \leq \ov\eta_{\lin}(E), \qquad 
 \\
 \nu \leq \ov\nu (T, E, F, H, \eta_{\lin}), \qquad  
 \delta_{\nn} \leq \ov\delta_{\nn} (T, E, F, H,  \eta_{\lin}), \qquad 
 \lambda \leq \ov\lambda,  \\
  \delta_{\bb} \leq \ov{\delta}_{\bb} (T, E, F, H,  \eta_{\lin},  \lambda, D_{\cut}, A, \Lambda ), \qquad \\
  \eps_{\can} \leq \ov\eps_{\can} ( T, E, F, H,  \eta_{\lin},  \lambda, D_{\cut}, A,\Lambda ), \qquad 
    r_{\comp} \leq \ov{r}_{\comp} ,
\end{gathered}
\end{equation}
then the following holds.

Assume that assumptions \ref{con_12.1_i}--\ref{con_12.1_vi} of Proposition~\ref{Prop_extend_comparison_by_one} hold.

Recall that in the case $J=0$, assumption \ref{con_12.1_vi} imposes the existence of a domain $X \subset \M_{t_{J}}$ and map $\zeta : X \to \M'_{t_J}$ with certain properties.
In this case we set $\wh\phi := \zeta$.

In the case $J \geq 1$, we set $X := \N_{t_J-} \cup \N_{t_J+}$ and consider the set $\Cut^{J}$ and the map $\wh\phi : X \to \M'_{t_J}$ satisfying all assertions of Proposition~\ref{Prop_performing_cap_extensions}. 

Then there is some time $t^* \in (t_J, t_{J+1}]$ and a smooth, time-preserving diffeomorphism onto its image $\phi^{J+1} : \N^{J+1}_{[t_J, t^*]} \to \M'$ with $\phi^{J+1}_{t_J} = \widehat\phi |_{\N_{t_J+}}$ whose inverse $(\phi^{J+1})^{-1} : \phi^{J+1} (\N^{J+1}_{[t_J, t^*]}) \to \N^{J+1}_{[t_J, t^*]}$ evolves by harmonic map heat flow (see Definition~\ref{def_RF_spacetime_harm_map_hf}) and such that the following holds for the associated perturbation $h^{J+1} := (\phi^{J+1})^* g' - g$ (which is a Ricci-DeTurck flow):
\begin{enumerate}[label = (\alph*)]
\item \label{ass_12.24_a} $|h^{J+1}| \leq 10 \eta_{\lin}$ on $\N^{J+1}_{[t_J, t^*]}$. 
\item \label{ass_12.24_b} For any $t \in [t_J, t^*]$ and $x \in \N^{J+1}_{t}$ whose time-$t$ distance to $\partial \N^{J+1}_t$ is smaller than $F r_{\comp}$ we have  
\[ Q_+(x) = e^{H(T-\t(x))} \rho_1^E (x) |h^{J+1} (x)| < \ov{Q} = 10^{-E-1} \eta_{\lin} r_{\comp}^E\,. \]
\item \label{ass_12.24_c} If even $|h^{J+1}| \leq  \eta_{\lin}$ on $\N^{J+1}_{t^*}$, then $t^* = t_{J+1}$.
\item \label{ass_12.24_d} $\phi^{J+1} (\N^{J+1}_{[t_J, t^*]})$ is $\eps_{\can} r_{\comp}$-thick.
\end{enumerate}
\end{proposition}

We emphasize that we have introduced another auxiliary parameter, $F$, which we will choose in Subsection~\ref{subsec_verification_of_APAs}, depending only on $E$.
The bound in assertion \ref{ass_12.24_b}, which holds $F r_{\comp}$-close to the boundary of $\partial \N^{J+1}$, will be helpful later as we are not able to apply the semi-local maximum principle, Proposition~\ref{Prop_semi_local_max}, too close to the boundary.
For this purpose, we will later choose $F \geq L(E)$, where the latter is the constant from Proposition~\ref{Prop_semi_local_max}.

Let us now explain the main strategy of the proof of Proposition~\ref{Prop_extend_phi_to_t_star}.
Observe first that the parabolic domain $\N^{J+1} \subset \M$ is a product domain and the Ricci flow on it can be viewed as a conventional, non-singular Ricci flow.
A similar domain, which contains the image $\wh\phi (\N_{t_J+})$, can be found in $\M'$. 
So the proof of Proposition~\ref{Prop_extend_phi_to_t_star} can be reduced to a relatively standard short-time and long-time existence statement for the harmonic map heat flow between conventional Ricci flows on manifolds with boundary.   Rather than solving the harmonic map heat flow equation with a boundary condition, we found it technically simpler to use a ``grafting'' construction to eliminate the boundary.

 A large part of the following proof will be devoted to the characterization of the geometry near the boundary of $\N_{t_J+}$ and the boundary of its image $\wh\phi ( \N_{t_J+} )$, which will serve as a setup for the subsequent grafting construction.
More specifically, our goal will be to show that the boundary of $\N_{t_J  +}$ and its image $\wh\phi ( \N_{t_J+} )$ are contained in regions that look sufficiently neck-like on the time-interval $[t_J, t_{J+1}]$.
To achieve this, we will employ the following strategy.
A priori assumption \ref{item_backward_time_slice_3}(a)  provides neck structures near $\partial \N_{t_{J+1}-}$ at time $t_{J+1}$.
Using Lemma~\ref{lem_time_slice_neck_implies_space_time_neck}, these neck structures can be promoted backwards onto the time-interval $[t_J, t_{J+1}]$.
The newly constructed neck structure at time $t_J$, near $\partial \N_{t_J+}$, a priori assumption \ref{item_q_less_than_q_bar_6}  and the interior decay estimate, Proposition~\ref{Prop_interior_decay}, can then be used to identify $C^0$-neck structures near the boundary of $\wh\phi ( \N_{t_J+}) \subset \M'_{t_J}$. 
Using the canonical neighborhood assumption and the self-improving property of necks in $\kappa$-solutions, Lemma \ref{lem_C0_neck_smooth_neck}, these $C^0$-neck structures imply the existence of neck structures of higher regularity in $\M'_{t_J}$.
Lastly, we use Lemma~\ref{lem_time_slice_neck_implies_space_time_neck}, to promote these neck structures forward in $\M'$, onto the time-interval $[t_J, t_{J+1}]$. 

Based on this characterization of the boundary of $\N_{t_J+}$ and its image, we perform a grafting construction in the last phase of the proof.
This grafting construction involves cutting $\M_{[t_J,t_{J+1}]}$ and $\M'_{[t_J, t_{J+1}]}$ inside the previously identified neck regions, gluing on shrinking round half-cylinders, and passing to a map between the grafted spacetimes.  
We have thus reduced our discussion to standard existence results for the harmonic map heat flow between complete manifolds.  
We remark that our approach is facilitated by the fact that $\wh\phi$ is already defined on a larger neighborhood of $\N_{t_J+}$, therefore providing enough space for an interpolation between the metric on $\M_{[t_J, t_{J+1}]}$ and the cylindrical metric.

\begin{proof}[Proof of Proposition~\ref{Prop_extend_phi_to_t_star}]
Let $\delta_\# > 0$ be a constant whose value we will determine at the end of the proof.
It will only depend on $T, E, H$ and $\eta_{\lin}$ and influence only the parameters $\nu, \delta_{\nn}, \delta_{\bb}$ and $\eps_{\can}$.
So it lies between $\eta_{\lin}$ and $\nu$ in the parameter order introduced in Subsection~\ref{subsec_parameter_order}.
To avoid an accumulation of a large number of different constants, in what follows we will be using the standard practice of making a series of adjustments to the constant $\de_\#$.  This means, strictly speaking, that $\de_\#$ is not really a single constant, but takes on different values at different places in the proof, and the earlier values are adjusted as functions of the later values.

By a priori assumption \ref{item_backward_time_slice_3}(a), each boundary component $\Sigma \subset \partial \N^{J+1}_{t_{J+1}}$ is the central 2-sphere of a $\delta_{\nn}$-neck $U_\Sigma \subset \M_{t_{J+1}}$ at scale $r_{\comp}$.
Lemma~\ref{lem_time_slice_neck_implies_space_time_neck} implies that if
\[ \delta_{\nn} \leq \ov{\delta}_{\nn} (\delta_\#), \qquad \eps_{\can} \leq \ov{\eps}_{\can} (\delta_\#), \qquad r_{\comp} \leq \ov{r}_{\comp}, \]
then for each such $\Sigma$ there is a product domain $U^*_\Sigma \subset \M_{[t_J, t_{J+1}]}$  that contains $\Sigma$ and on which the flow is $\delta_\#$-close at scale $r_{\comp}$ to the round shrinking cylinder on the time-interval $[-1,0]$.
By this we mean the following: we can find an $r_{\comp}^2$-time-equivariant and $\partial_{\mathfrak{t}}$-preserving diffeomorphism
\[ \psi_\Sigma : S^2 \times \big( {- \delta_\#^{-1} , \delta_\#^{-1}  }\big) \times [ { - 1, 0} ] \longrightarrow U^*_\Sigma \]
such that $\Sigma = \psi_{\Sigma} ( S^2 \times \{ 0 \} \times \{ 0 \} )$ and
\begin{equation} \label{eq_delta_1_close_neck}
 \big\Vert r_{\comp}^{-2} \psi^*_\Sigma g - g^{S^2 \times \R} \big\Vert_{C^{[\delta_\#^{-1}]}} < \delta_\#. 
\end{equation}
Here $g^{S^2 \times \R}$ denotes the metric of the standard round shrinking cylinder spacetime and the norm is taken over the domain of $\psi_\Sigma$.

By (\ref{eq_delta_1_close_neck}) and assuming
\[ \delta_\# \leq \ov\delta_\#, \qquad
r_{\comp} \leq \ov{r}_{\comp}, \]
we have
\begin{equation} \label{eq_rho_bigger_1_on_U_t_J}
 1.9 r_{\comp}  < \rho_1 = \rho < 2.1  r_{\comp}   \qquad \text{on} \qquad U^*_{\Sigma, t_J}. 
\end{equation}
So, by a priori assumption \ref{item_backward_time_slice_3}(a), applied at time $t_J$, and assuming
\[ \delta_{\nn} \leq \ov\delta_{\nn}, \]
we find that $U^*_{\Sigma, t_J}$ is disjoint from $\partial \N^J_{t_J}$ if $J \geq 1$.
So, if $J\geq 1$, since $\Sigma (t_J) \subset \partial \N^{J+1}_{t_J} \subset \N_{t_J-}$, it follows that $U^*_{\Sigma, t_J}  \subset \N_{t_J-} \subset X$.

On the other hand, if $J=0$, and
\[ \delta_{\nn} \leq \ov\delta_{\nn}, \]
then $U^*_{\Sigma, t_J}$ has diameter $< 10 \delta_\#^{-1} r_{\comp}$.
So assuming
\[ \delta_{\nn} \leq \ov\delta_{\nn} (\delta_\#), \]
we have $U^*_{\Sigma, t_J} \subset X$.
So, in summary,
\begin{equation} \label{eq_U_star_in_NN_minus}
 U^*_{\Sigma, t_J}  \subset \N_{t_J-} \subset X \quad \text{if $J \geq 1$} \qquad \text{and} \qquad U^*_{\Sigma, t_J} \subset X \quad \text{if $J = 0$}.
 \end{equation}

Consider the Ricci-DeTurck perturbation $(h, \{ h^j \}_{j=1}^J)$ associated to the comparison $(\Cut, \phi,  \{ \phi^j \}_{j=1}^J )$ and let $Q$ be defined as in Definition~\ref{def_a_priori_assumptions_7_13} of the a priori assumptions \ref{item_q_less_than_q_bar_6}--\ref{item_apa_13}.
We will now use a priori assumption \ref{item_q_less_than_q_bar_6} to show that we have a bound on $Q$ in large parabolic neighborhoods near the boundary of $\N_{t_J+}$.
In Claim~\ref{cl_12.3_2}  this bound will later be used to obtain an improved bound on $Q$, and therefore on $h$, via the interior decay estimate, Proposition~\ref{Prop_interior_decay}.
For this purpose, let $A_\# < \infty$ be a constant whose value will be determined in the proof of Claim~\ref{cl_12.3_2} (depending only on $E$ and $\delta_\#$). 

\begin{claim}
\label{cl_12.3_1}
If $J \geq 1$ and
\begin{multline*} \label{eq_asspt_Q_bound_P_A_hash} 
 \delta_{\nn} \leq \ov\delta_{\nn} (A_\#), \qquad
 \lambda \leq \ov\lambda, \qquad
 \Lambda \geq \underline\Lambda, \qquad
 \delta_{\bb} \leq \ov{\delta}_{\bb} ( A_\#, \lambda, D_{\cut},  \Lambda, A ), \\ \qquad  
 \eps_{\can} \leq \ov{\eps}_{\can} (A_\#,  \lambda,  D_{\cut}, \Lambda, A), \quad 
 r_{\comp} \leq \ov{r}_{\comp}  ,
\end{multline*}
then for any $x \in \partial\N^{J+1}_{t_{J}}$ the parabolic neighborhood $P(x, A_\# r_{\comp})$ is unscathed,
\begin{equation} \label{eq_P_A_hash_in_N_minus_DD}
 P(x, A_\# r_{\comp}) \subset \N \setminus \cup_{\DD \in \Cut \cup \Cut^{J}} \DD 
\end{equation}
and we have the bound
\begin{equation} \label{eq_Q_bound_P_A_hash}
 Q \leq \ov{Q} \qquad \text{on} \qquad P(x, A_\# r_{\comp}). 
\end{equation}
\end{claim}

\begin{proof}
Choose a boundary component $\Sigma \subset \partial \N_{t_{J+1}}^{J+1}$ such that $x \in \Sigma(t_J) \subset \partial \N_{t_J}^{J+1}$.
So $x \in U^*_{\Sigma, t_J}$.
Assuming
\[ \delta_{\nn} \leq \ov\delta_{\nn} (A_\#), \]
we obtain by similar arguments as those that led to  (\ref{eq_U_star_in_NN_minus}) that 
\begin{equation} \label{eq_ball_in_N}
 B(x, A_\# r_{\comp} ) \subset \N_{t_J-}. 
\end{equation}

Next, using Lemma~\ref{lem_containment_parabolic_nbhd} along with (\ref{eq_rho_bigger_1_on_U_t_J}), and assuming
\[ \eps_{\can} \leq \ov\eps_{\can} (A_\# , A), \]
we can find  a constant $A' = A' (A_\#, A) < \infty$ with $A' \geq A$ such that $P(x,A'\rho_1(x))$ is unscathed and
\begin{equation} \label{eq_P_y_A_P_x_A_prime}
 P(y, A \rho_1 (y) ) \subset P(x, A' \rho_1 (x)) \qquad \text{for all} \qquad y \in P(x, A_\# r_{\comp}). 
\end{equation}

We now show that $P(x, A' \rho_1 (x))$ is disjoint from the cuts.
To do this, observe that for any $t' \in (t_J, t_{J+1}]$ we have $B(x(t'), A' \rho_1 (x)) \not\subset \N$.
So by Lemma~\ref{lem_boundary_far_from_cut}, along with (\ref{eq_rho_bigger_1_on_U_t_J}) and a priori assumption \ref{item_cut_diameter_less_than_d_r_comp_11}, assuming
\begin{gather*}
 \delta_{\nn} \leq \ov\delta_{\nn}, \qquad 
 \lambda \leq \ov\lambda, \qquad
 \Lambda \geq \underline\Lambda, \qquad
\delta_{\bb} \leq \ov\delta_{\bb} (\lambda, D_{\cut},  A' (A_\#, A), \Lambda), \qquad \\
\eps_{\can} \leq \ov\eps_{\can} (\lambda, D_{\cut}, A' (A_\#, A), \Lambda), \qquad 
r_{\comp} \leq \ov{r}_{\comp}, 
\end{gather*}
we find that $P(x, A' \rho_1(x) ) \cap \DD = \emptyset$ for all $\DD \in \Cut \cup \Cut^{J}$.
Combining this with (\ref{eq_ball_in_N}) gives us (\ref{eq_P_A_hash_in_N_minus_DD}) via Lemma~\ref{lem_parabolic_domain_in_N}.
Combining it further with (\ref{eq_P_y_A_P_x_A_prime}) and a priori assumption \ref{item_q_less_than_q_bar_6} yields (\ref{eq_Q_bound_P_A_hash}).
\end{proof}

Next we improve the estimate from Claim~\ref{cl_12.3_1} and use it to identify more precise necks in $\M'$.   
\begin{claim}
\label{cl_12.3_2}
If
\begin{equation} \label{eq_claim2_assumptions_parameters}
\begin{gathered}
 E > 2, \qquad 
 H \geq \underline{H} (E). \qquad
 \eta_{\lin} \leq \ov{\eta}_{\lin} (E), \qquad 
 A_\# \geq \underline{A}_{\#} (E, \delta_\#),  \\ 
 \nu \leq \ov\nu (E, \delta_\#),  \qquad
  \eps_{\can} \leq \ov{\eps}_{\can} (E, \delta_\#) , \qquad 
  r_{\comp} \leq \ov{r}_{\comp} , 
\end{gathered}
\end{equation}
then for any component $\Sigma \subset \partial\N_{t_{J+1}}^{J+1}$ we have
\begin{equation} \label{eq_h_bound_delta_hash}
 \big| \wh\phi^* g'_{t_J} - g_{t_J} \big| < \delta_\# \qquad \text{on} \qquad U^*_{\Sigma, t_J}, 
\end{equation}
and
\begin{equation} \label{eq_M_prime_neck_bilipschitz}
 \big\Vert r^{-2}_{\comp} \big( \wh\phi \circ \psi_{\Sigma, t_J} \big)^* g'_{t_J} - g^{S^2 \times \R}_{-1} \big\Vert_{C^0} < \delta_\#.
\end{equation}
\end{claim}

\begin{proof}
Consider first the case $J = 0$.
In this case, by (\ref{eq_U_star_in_NN_minus}) and assumption \ref{con_12.1_vi} of this proposition we have on $U^*_{\Sigma, t_0}$
\[ e^{H T} \rho_1^E \big| \wh{h} \big| \leq \nu \ov{Q} = \nu \cdot 10^{-E-1} \eta_{\lin} r_{\comp}^E \,; \]
  recall that $\wh{h} = \wh{\phi}^* g'_{t_0} - g_{t_0}$.
So (\ref{eq_h_bound_delta_hash}) follows from (\ref{eq_rho_bigger_1_on_U_t_J}), assuming
\[ \eta_{\lin} \leq 1, \qquad \nu \leq \ov\nu ( \delta_\# ). \]

Second, consider the case $J \geq 1$.
By (\ref{eq_P_A_hash_in_N_minus_DD}) and assuming
\[ A_\# \geq \underline{A}_\# (\delta_\#), \]
we have $U^*_{\Sigma, t_J} \cap \DD = \emptyset$ for all $\DD \in \Cut^{J}$.
So therefore  on $U^*_{\Sigma, t_J}$ we have $\wh\phi = \phi_{t_J-}$ and hence $\wh\phi^* g'_{t_J} - g_{t_J} = h_{t_J-}$.
We will now apply Proposition~\ref{Prop_interior_decay} at every point of $U^*_{\Sigma, t_J}$.
To do this, note that by (\ref{eq_P_A_hash_in_N_minus_DD}) the perturbation $h$ is defined and smooth on all of $P(x, A_\# r_{\comp})$ and by a priori assumption \ref{item_eta_less_than_eta_lin_13} and  (\ref{eq_Q_bound_P_A_hash}) we have $|h| \leq \eta_{\lin}$ and $Q \leq \ov{Q}$ everywhere on this parabolic neighborhood.
Moreover, if $P(x, A_\# r_{\comp})$ intersects the initial time-slice $\M_0$, then by a priori assumption \ref{item_q_less_than_nu_q_bar_12} we have $Q \leq \nu \ov{Q}$ on the intersection.
Lastly, note that the diameter of $U^*_{\Sigma, t_J}$ is bounded by $10 \delta_\#^{-1} r_{\comp}$ for sufficiently small $\delta_\#$.
So  assuming
\begin{gather*}
E > 2, \qquad
H \geq \underline{H} (E), \qquad
 \eta_{\lin} \leq \ov\eta_{\lin} (E), \qquad 
A_\# \geq \underline{A}_\# (E, \delta_\#), \\ 
\nu \leq \ov\nu (E, \delta_\#),  \qquad
\eps_{\can} \leq \ov\eps_{\can} (E, \delta_\#), \qquad
r_{\comp} \leq \ov{r}_{\comp},
\end{gather*}
we conclude by Proposition~\ref{Prop_interior_decay} that $Q < \delta_\# \ov{Q}$ on $U^*_{\Sigma, t_J}$.
Note that here we have used (\ref{eq_rho_bigger_1_on_U_t_J}) and we applied Proposition~\ref{Prop_interior_decay} centered at all points in $U^*_{\Sigma,t_J}$, with an appropriate choice for the radius $A$.

So on $U^*_{\Sigma, t_J}$
\[ e^{H (T-t_J)} \rho_1^E |h_{t_J-}| = Q < \delta_\# \ov{Q} = \delta_\# \cdot 10^{-E-1} \eta_{\lin} r_{\comp}^E. \]
Using (\ref{eq_rho_bigger_1_on_U_t_J}) and the fact that $t_J \leq T$, due to assumption \ref{con_12.1_v} of this proposition, we obtain (\ref{eq_h_bound_delta_hash}) assuming
\[ \eta_{\lin} \leq \ov\eta_{\lin}. \]

Finally, the bound (\ref{eq_M_prime_neck_bilipschitz}) follows by combining (\ref{eq_h_bound_delta_hash}) with (\ref{eq_delta_1_close_neck}) and adjusting (the earlier instance of) $\delta_\#$.  
\end{proof}

Next, we use (\ref{eq_M_prime_neck_bilipschitz}) to establish the existence of a $\delta_\#$-neck in $\M'$.

\begin{claim}
\label{cl_12.3_3}
If
\begin{equation} \label{eq_delta_nn_claim3}
 \delta_{\nn} \leq \ov\delta_{\nn}, 
\end{equation}
then following holds.
For any component $\Sigma \subset \partial\N_{t_{J+1}-}$ there is a $\delta_\#$-neck $U'_\Sigma \subset \M'_{t_J}$ at scale $2r_{\comp}$ that has a central $2$-sphere which intersects $\wh\phi (\Sigma (t_J)) \subset \M'_{t_J}$.
\end{claim}

\begin{proof}
Note that $\wh\phi (\Sigma (t_J)) = \phi^J_{t_J} (\Sigma (t_J))$, as $\DD \subset \Int \N_{t_j+ }$ for all $\DD \in \Cut^J$ (see Definition~\ref{def_comparison}).
The $\eps_{\can}$-canonical neighborhood assumption holds on $\wh\phi (\Sigma (t_J))$ by assumption \ref{con_12.1_vi} of this proposition (if $J=0$) and by assertion \ref{ass_12.3_e} of Proposition~\ref{Prop_performing_cap_extensions} (if $J \geq 1$).
The statement now follows from Lemma~\ref{lem_C0_neck_smooth_neck}, assuming
\[ \eps_{\can} \leq \ov\eps_{\can} (\delta_\#), \]
after possibly adjusting $\delta_\#$.
\end{proof}

By Lemma~\ref{lem_time_slice_neck_implies_space_time_neck} and assuming
\[ \eps_{\can} \leq \ov{\eps}_{\can} (\delta_\#), \qquad r_{\comp} \leq \ov{r}_{\comp}, \]
we obtain furthermore after adjusting $\delta_\#$:

\begin{claim}
\label{cl_12.3_4}
Assuming parameter bounds of the same form as in (\ref{eq_claim2_assumptions_parameters}) and (\ref{eq_delta_nn_claim3}), the following holds.

For any component $\Sigma \subset \D\N_{t_{J+1}-}$ there is a product domain $U^{ \prime *}_\Sigma \subset \M'_{[t_J, t_{J+1}]}$, on the time-interval $[t_J, t_{J+1}]$, with $\wh\phi ( \Sigma (t_J)) \subset U^{\prime *}_{\Sigma, t_J}$ on which the metric is $\delta_\#$-close at scale $r_{\comp}$ to the standard round shrinking cylinder.
More specifically, there is an $r_{\comp}^2$-time-equivariant and $\partial_{\mathfrak{t}}$-preserving diffeomorphism
\[ \psi'_\Sigma : S^2 \times \big( {- \delta_\#^{-1}, \delta_\#^{-1} }\big) \times [ { - 1, 0} ] \longrightarrow U^{ \prime *}_\Sigma \]
such that
\begin{equation} \label{eq_delta_close_to_neck_M_prime}
 \big\Vert r_{\comp}^{-2} \psi^{\prime *}_\Sigma g' - g^{S^2 \times \R} \big\Vert_{C^{[\delta_\#^{-1}]}} < \delta_\#. 
\end{equation}
We furthermore have
\[ \psi'_{\Sigma} (S^2 \times \{ 0 \} \times \{ - 1 \} ) \cap \wh\phi ( \Sigma (t_J)) \neq \emptyset. \]
\end{claim}

We now carry out the grafting construction.  We begin by  identifying product domains in the time slabs $\M_{[t_J,t_{J+1}]}$ and $\M'_{[t_J,t_{J+1}]}$ that will be used in the construction.

For $k=0,\ldots,5$ let $N_k$ be the (open) $100kr_{\comp}$-tubular neighborhood around $\N_{t_J+}$ in $\M_{t_J}$ and set  $N'_k := \wh\phi (N_k)$. Assuming
\[ \delta_\# \leq \ov\delta_\#, \qquad
\delta_{\nn} \leq \ov{\delta}_{\nn}, \]
we obtain from (\ref{eq_U_star_in_NN_minus}) and assumption \ref{con_12.1_vi} of this proposition that
\begin{equation} \label{eq_N_5_in_X}
\begin{aligned}
&N_0 \subset N_1 \subset \ldots \subset N_5 \subset X, \\ &N'_0 \subset N'_1 \subset \ldots \subset N'_5 \subset \wh\phi (X)\,. 
\end{aligned}
\end{equation}
Moreover, assuming
\[ \eta_{\lin} \leq 10^{-2}, \qquad \delta_\# \leq \ov\delta_\#, \]
Claim~\ref{cl_12.3_4}  and a priori assumption \ref{item_eta_less_than_eta_lin_13}  (if $J \geq 1$) or the assumptions from the proposition (if $J=0$) yield
\begin{equation} \label{eq_N_5_minus_N_0}
 N_5 \setminus N_0 \subset \bigcup_{\Sigma \subset \partial \N^{J+1}_{t_{J+1}}} U^*_{\Sigma, t_J}, \qquad N'_5 \setminus N'_0 \subset \bigcup_{\Sigma \subset \partial \N^{J+1}_{t_{J+1}}} U^{\prime *}_{\Sigma, t_J}. 
\end{equation}
By construction and by (\ref{eq_N_5_minus_N_0}), all points on $N_5 = \N_{t_J+} \cup (N_5 \setminus N_0)$ and $N'_5\setminus N'_0$ survive until time $t_{J+1}$.
A priori assumptions \ref{item_lambda_thick_2}, \ref{item_backward_time_slice_3}(a), (c), \ref{item_eta_less_than_eta_lin_13}, assertions \ref{ass_12.3_d} and \ref{ass_12.3_e} of Proposition~\ref{Prop_performing_cap_extensions}, (\ref{eq_N_5_in_X}), assumption \ref{con_12.1_vi} of this proposition and Lemma~\ref{lem_scale_distortion}, as well as (\ref{eq_delta_close_to_neck_M_prime}) and (\ref{eq_N_5_minus_N_0}), imply, assuming
\begin{multline*}
\delta_\# \leq \ov\delta_\#, \qquad
\eta_{\lin} \leq \ov\eta_{\lin}, \qquad 
\delta_{\nn} \leq \ov\delta_{\nn}, \qquad
\lambda \leq 1, \qquad
\Lambda \geq 2, \qquad
\eps_{\can} \leq \ov\eps_{\can} (\lambda) , \qquad \\
r_{\comp} \leq \ov{r}_{\comp}, 
\end{multline*}
that 
\[ \rho > C_{\sd}^{-1} \lambda r_{\comp} >  \eps_{\can} r_{\comp} \qquad \text{on} \qquad  N'_5 . \]

Let $t^*_1 \in [t_J, t_{J+1}]$ be maximal with the property that  $N'_5 (t)$ is defined and weakly $\frac12 C_{\sd}^{-1} \lambda r_{\comp}$-thick for all $t \in [t_J, t^*_1]$, where $C_{\sd}$ is the constant from Lemma \ref{lem_scale_distortion}.
Note here that $t^*_1$ is well-defined by the $(\eps_{\can}r_{\comp}, T)$-completeness of $\M'$ and Lemma~\ref{lem_forward_backward_control}, assuming
$$
\eps_{\can}\leq \ov\eps_{\can}(\lambda)\,.
$$  
We can now express the flows $g$ and $g'$ restricted to the product domains $N_5 ([t_J, t^*_1])$ and $N'_5 ([t_J, t^*_1])$ by conventional Ricci flows $ (N_5,(g_t)_{t\in [t_J,t^*_1]})$, $(N'_5,(g'_t)_{t\in[t_J,t^*_1]})$. 

\begin{claim}[Grafting on round half-cylinders]
\label{cl_12.3_5}
After adjusting $\de_\#$ there are   smoothly varying Riemannian metrics $(g^+_t)_{t \in [t_J, t^*_1]}$,  $(g^{\prime +}_t )_{t \in [t_J, t^*_1]}$ on smooth manifolds $N^+$ and $N^{\prime +}$, respectively, and a diffeomorphism $\phi^+ : N^+ \to N^{\prime +}$ such that:
\begin{enumerate}[label=(\alph*)]
\item \label{ass_12.24_cl5_a} $N_5$ and $N_5'$ can be viewed as open subsets of $N^+$ and $N^{'+}$, respectively.
\item \label{ass_12.24_cl5_b}  For all $t\in [t_J,t^*_1]$, we have $g^+_t=g_t$ on $N_1\subset N^+$ and $g^{\prime+}_t=g^\prime_t$ on $N'_1\subset N^{\prime+}$.
\item \label{ass_12.24_cl5_c} $g^+_t$, $g^{\prime+}_t$ are complete for all $t\in [t_J,t^*_1]$. 
\item \label{ass_12.24_cl5_d} For some constant $C=C(\la) < \infty$ we have
$$
|{\Rm_{g^+_t}}|\,,|{\Rm_{g^{\prime +}_t}}|\leq Cr_{\comp}^{-2}\,.
$$
\item \label{ass_12.24_cl5_e} $(g^+_t)_{t \in [t_J, t^*_1]}$,  $(g^{\prime +}_t )_{t \in [t_J, t^*_1]}$ are ``$\de_\#$-approximate Ricci flows'':
\begin{alignat*}{2}
- \delta_\# r_{\comp}^{-2} g^+_t &< \partial_{t} g^+_t + 2 \Ric_{g^+_t} &&< \delta_\# r_{\comp}^{-2} g^+_t\,, \\
 - \delta_\# r_{\comp}^{-2} g^{\prime+}_t &< \partial_{t} g^{\prime+}_t + 2 \Ric_{g^{\prime+}_t} &&< \delta_\# r_{\comp}^{-2} g^{\prime+}_t. 
\end{alignat*}
\item \label{ass_12.24_cl5_f} For some $C^* = C^*(\lambda) <\infty$,
\begin{align*}
|{\nabla^m_{g^+_t} \Rm (g^+_t)}|_{g^+_t}, |{\nabla^m_{g^{\prime+}_t} \Rm (g^{\prime+}_t)}|_{g^{\prime +}_t} &< C^*  r_{\comp}^{-2} (t-t_J)^{-m/2}  \\
|\nabla^{m_1}_{g^+_t} \partial_t^{m_2} g^+_t|_{g^+_t}, |\nabla^{m_1}_{g^{\prime+}_t} \partial_t^{m_2} g^{\prime+}_t|_{g^{\prime +}_t} &< C^*  r_{\comp}^{-2} (t-t_J)^{-(m_2-1 + m_1/2)} 
\end{align*} 
for all $t \in (t_J,t^*_1]$ and $m,m_1, m_2 = 0, \ldots, 100$.
\item \label{ass_12.24_cl5_g} There is a universal constant $C^{**} < \infty$ such that at every $x \in N^+$ with $d_{g^+_{t_J}}(x,N^+\setminus N_0)<\delta_\#^{-1} r_{\comp}$ we have
\begin{align*}
|{\nabla^m_{g^+_t} \Rm (g^+_t)}|_{g^+_t} (x) &< C^{**}  r_{\comp}^{-2 - m} \\
|\nabla^{m_1}_{g^+_t} \partial_t^{m_2} g^+_t|_{g^+_t} (x) &< C^{**}  r_{\comp}^{-m_1 - 2m_2}
\end{align*} 
for all $t \in (t_J,t^*_1]$ and $m,m_1, m_2 = 0, \ldots, 100$.
\item \label{ass_12.24_cl5_h} $\phi^+ = \wh\phi$ on $N_2$.
\item \label{ass_12.24_cl5_i} We have $|({\phi}^+)^*g^{\prime+}_{t_J}-g_{t_J}|_{g^+_t} <\delta_\#$ at every point $x\in N^+$ with $d_{g^+_{t_J}}(x,N^+\setminus N_0) <\delta_\#^{-1} r_{\comp}$.
\item \label{ass_12.24_cl5_j} $t^*_1 > t_J$ and if $t^*_1 < t_{J+1}$, then $\cup_{t \in [t_J, t^*_1]} N'_5 (t)$ must contain a $C_{\sd}^{-1} \lambda r_{\comp}$-thin point.
\end{enumerate}
\end{claim}

\begin{proof}
Using Lemma~\ref{lem_forward_backward_control}, we find that $t^*_1 > t_J$ and that if $t^*_1 < t_{J+1}$, then $ N'_5 ([t_J, t^*_1])$ must contain a $C_{\sd}^{-1} \lambda r_{\comp}$-thin point.  
This proves assertion \ref{ass_12.24_cl5_j}.

For each component $\Sigma\subset \D\N^{J+1}_{t_{J+1}}$,   we may pushforward $r_{\comp}^{-2}g^{S^2\times \R}$ under $\psi_{\Sigma}$ and $\psi'_{\Sigma}$ to obtain spacetime metrics on  $U^*_{\Sigma}$ and $U^{'*}_{\Sigma}$.  Using the product structure on $U^*_{\Sigma}$ and $U^{'*}_{\Sigma}$, these yield evolving metrics $(g^{\Sigma}_t)_{t\in [t_J,t_{J+1}]}$, $(g^{\prime \Sigma}_t)_{t\in [t_J,t_{J+1}]}$ on the initial time-slices $U^*_{\Sigma,t_J}$, $U^{'*}_{\Sigma,t_J}$, and hence also on $N_5\setminus N_0$, $N'_5\setminus N'_0$ by (\ref{eq_N_5_minus_N_0}).

By a standard interpolation argument we can construct smooth families of metrics $(\td{g}_t)_{t\in [t_J,t^*_1]}$, $(\td{g}')_{t\in [t_J,t^*_1]}$ on $N_5$ and  $N'_5$ such that $\td{g}_t = g_t$ and $\td{g}'_t = g'_t$ on $N_1$ and $N'_1$, respectively, and such that for every  component $\Sigma \subset \D\N^{J+1}_{t_{J+1}}$ we have $\td{g}_t=g^\Sigma_t$, $\td{g}'=g^{\prime\Sigma}_t$ on $(N_5\setminus N_2)\cap U^*_{\Sigma,t_J}$ and $(N'_5\setminus N'_2)\cap U^{\prime*}_{\Sigma,t_J}$, respectively.  Moreover, using (\ref{eq_delta_1_close_neck}) and (\ref{eq_delta_close_to_neck_M_prime}), and after possibly adjusting $\delta_\#$, we may assume that for every component $\Sigma\subset \D\N^{J+1}_{t_{J+1}}$,  $m_1,m_2<\delta_\#^{-1}$,  and $t\in [t_J,t^*_1]$, we have
\begin{equation}
\label{eqn_tilde_g_tilde_g_prime_spacetime}
\begin{aligned}
|\nabla^{m_1}_{g^\Sigma_{t_J}}\D^{m_2}_t(\td{g}_t-g^\Sigma_t)|_{g^\Sigma_t}&<\delta_\# r_{\comp}^{-m_1 - 2m_2} \,,\\ 
|\nabla^{m_1}_{g^{'\Sigma}_{t_J}}\D^{m_2}_t(\td{g}^\prime_t-g^{\prime\Sigma}_t)|_{g^{\prime\Sigma}_t} &<\delta_\# r_{\comp}^{-m_1-2m_2}
\end{aligned}
\end{equation}
on $(N_5\setminus N_0)\cap U^*_{\Sigma,t_J}$ and $(N'_5\setminus N'_0)\cap U^{\prime *}_{\Sigma,t_J}$.
So, after possibly adjusting $\delta_\#$ once again, we may assume that on $N_5$ and $N'_5$ we have
\begin{alignat*}{2}
- \delta_\# r_{\comp}^{-2} \td{g}_t &< \partial_{t} \td{g}_t + 2 \Ric_{\td{g}_t} &&< \delta_\# r_{\comp}^{-2} \td{g}_t\,, \\
 - \delta_\# r_{\comp}^{-2} \td{g}'_t &< \partial_{t} \td{g}'_t + 2 \Ric_{\td{g}'_t} &&< \delta_\# r_{\comp}^{-2} \td{g}'_t. 
\end{alignat*}
Since these flows are isometric to round shrinking cylindrical flows near the ends of $N_5$ and $N'_5$, we can attach round shrinking half-cylinders to these flows at each end.
This produces flows $(g^+_t)_{t \in [t_J, t^*_1]}$ and $(g^{\prime +}_t )_{t \in [t_J, t^*_1]}$ on $N^+ \supset N_4$ and $N^{\prime +} \supset N'_4$ satisfying assertions \ref{ass_12.24_cl5_a}--\ref{ass_12.24_cl5_e} of this claim.  
Assertion \ref{ass_12.24_cl5_g} also follows from (\ref{eq_delta_1_close_neck}), (\ref{eq_delta_close_to_neck_M_prime}) and (\ref{eqn_tilde_g_tilde_g_prime_spacetime}), after adjusting $\delta_\#$.
Assertion \ref{ass_12.24_cl5_f} follows from Shi's estimates in $N_1$, $N'_1$, $N^+\setminus N_2$, $N^{\prime+}\setminus N'_2$ and assertion \ref{ass_12.24_cl5_g}.

Since $\wh\phi$ is a $(1+\delta_\#)$-bilipschitz map on $N_5 \setminus N_2 \subset \cup_{\Sigma} U^*_{\Sigma, t_J}$ (see (\ref{eq_h_bound_delta_hash}) and (\ref{eq_N_5_minus_N_0})) and $\td{g}_{t_J}$ and $\td{g}'_{t_J}$ are isometric to subsets of  round cylinders on the interior of $N_5 \setminus N_2$ and $N'_5 \setminus N'_2$, respectively, we can use a smoothing procedure (see also Lemma~\ref{lem_smoothing_bilipschitz}) to construct a diffeomorphism onto its image $\td{\phi} : N_4 \to \td\phi (N_4) \subset N'_5$ such that $\td\phi = \wh\phi$ on $N_2$ and $\td\phi^*g^{\prime +}_{t_J}=\td\phi^* \td{g}'_{t_J} =  \td{g}_{t_J}=g^+_{t_J}$ on $N_4 \setminus N_3$ and such that, after adjusting $\delta_\#$,
\begin{equation}
\label{eqn_g_plus_g_prime_plus}
 \big| \td{\phi}^* g^{\prime +}_{t_J} - g^+_{t_J} \big| < \delta_\# \qquad \text{on} \qquad U^*_{\Sigma, t_J} \cap N_4, 
\end{equation}
for every component $\Sigma \subset \D\N^{J+1}_{t_{J+1}}$.
We can now extend the diffeomorphism $\td\phi : N_4 \to \td\phi (N_4)$ to a diffeomorphism $\phi^+ : N^+ \to N^{\prime +}$ such that it remains an isometry on $N^+ \setminus N_4$.  Adjusting $\delta_\#$ again, the map $\phi^+$ will satisfy assertions \ref{ass_12.24_cl5_h} and \ref{ass_12.24_cl5_i} of this claim, by (\ref{eqn_g_plus_g_prime_plus}) and the fact that $(\phi^+)^*g^{\prime+}_{t_J}=g^+_{t_J}$ on $N^+\setminus N_3$.
\end{proof}

We now construct the map $\phi^{J+1}$ by  solving the harmonic map heat flow equation starting from $(\phi^+)^{-1}$, where $\phi^+ : N^+ \to N^{\prime +}$ is the map constructed in Claim~\ref{cl_12.3_5}.

Using Claim~\ref{cl_12.3_5} and Proposition~\ref{prop_hh_flow_existence} from the appendix, we obtain that if
\[ \eta_{\lin} \leq \ov\eta_{\lin}, \qquad \delta_\# \leq \ov\delta_\# (\eta_{\lin}),  \]
then we can find a time $t^* \in (t_J, t^*_1]$ and a solution $(\chi_t)_{t \in [t_J, t^*]}$, $\chi_t : N^{\prime +} \to N^+$ to the harmonic map heat flow equation with respect to $(g^{\prime +}_t)_{t \in [t_J, t^*_1]}$ and $(g^{ +}_t )_{t \in [t_J, t^*_1]}$ with the following properties: 
\begin{enumerate}
\item \label{li_12.24_1} $\chi_{t_J} = (\phi^+)^{-1}$.
\item \label{li_12.24_2} $\chi_t$ is a diffeomorphism for all $t \in [t_J, t^*]$.
\item \label{li_12.24_3}  $|(\chi_t^{-1})^*g^{\prime +}_t-g^{+}_t|_{g^{+}_t}<2\eta_{\lin} $   for all $t \in [t_J, t^*]$.
\item \label{li_12.24_4}  If 
\begin{equation*} \label{eq_2_eta_lin_bound_d}
|(\chi_{t^*}^{-1})^*g^{\prime +}_{t^*}-g^{+}_{t^*}|_{g^{+}_{t^*}}<1.9\eta_{\lin}  
\end{equation*}
holds on  $N^{\prime+}$, then $t^*=t^*_1$.
\end{enumerate}

\medskip

We first show that 
\begin{equation} \label{eq_N0_in_chi_N_prime_1}
\chi_t^{-1}( N_0 )\subset  N'_1  \qquad \text{for all} \qquad t \in [t_J, t^*].
\end{equation}
After rescaling by $r_{\comp}^{-2}$, assuming 
\[  \eta_{\lin}\leq\ov\eta_{\lin}\,,\qquad \delta_\# \leq \ov\delta_\#, \]
we may apply assertion \ref{ass_12.24_cl5_g} of Claim~\ref{cl_12.3_5} and Proposition~\ref{prop_drift_control}, taking the constants $\de$, $A$ in the hypotheses to be  $\delta = 1$, $A = C^{**}$, to conclude that $\chi_t^{-1} ( \partial N_0 )\subset N^\prime_1$ for all $t \in [t_J, t^*]$.  
Here we have used a continuity argument,  the fact that $\chi_{t_J}^{-1} (\partial N_0) = \phi^+ (\partial N_0) = \wh\phi (\partial N_0) = \partial N'_0$ by assertion \ref{ass_12.24_cl5_h} of Claim~\ref{cl_12.3_5}, and Property~(\ref{li_12.24_3}) above, to retain the hypotheses of Proposition~\ref{prop_drift_control}. 
Therefore, since $\chi_t^{-1}$ is a smoothly varying diffeomorphism, (\ref{eq_N0_in_chi_N_prime_1}) follows.

Now set  $\ov\phi_t := \chi^{-1}_t |_{N_0}$ for $t\in [t_J,t^*]$.  Since $g^+_t = g_t$ on $N_0$ and $g^{\prime+}_t=g^\prime_t$ on $N^\prime_1$ by assertion \ref{ass_12.24_cl5_b} of Claim~\ref{cl_12.3_5}, 
we can view $(\ov\phi_t)_{t \in [t_J, t^*]}$ as a smooth, time-preserving diffeomorphism onto its image of the form 
$$
\phi^{J+1} : \N^{J+1}_{[t_J, t^*]} \longrightarrow N'_1([t_J, t^*]) \subset \M'\, , 
$$
whose inverse evolves by harmonic map heat flow equation with respect to $g'$ and $g$.

Consider the perturbation $h^{J+1} = (\phi^{J+1})^* g' - g$.
Assertion~\ref{ass_12.24_a} of this proposition follows from Property~(\ref{li_12.24_3}) above.

If  
\[ \eta_{\lin} \leq \ov\eta_{\lin}, \qquad
\delta_\#\leq\ov\delta_\#(T, E,H,\eta_{\lin})\,,\qquad  \]
then for every $x\in N^+$ with $d_{g^+_{t_J}}(x,N^+\setminus N_0)<\delta_\#^{-1} r_{\comp}$,  after adjusting $\delta_\#$, and by assertions \ref{ass_12.24_cl5_e}, \ref{ass_12.24_cl5_g}, \ref{ass_12.24_cl5_i} of Claim~\ref{cl_12.3_5} and Property~(\ref{li_12.24_3}) above, we may apply Proposition~\ref{prop_distortion_stays_small} to conclude that 
\begin{equation}
\label{eqn_h_bound_for_assertion_b}
|(\chi_{t}^{-1})^*g^{\prime+}_{t}-g^+_{t}|_{g^+_{t}}(x)<e^{-HT}10^{-E -1}r_{\comp}^E\rho_1^{-E}(x)\eta_{\lin}<\eta_{\lin}
\end{equation}
for all $t\in [t_J,t^*]$.
If 
\[ \delta_\# \leq \ov\delta_\#(F)\,,\]
then (\ref{eq_delta_1_close_neck}) implies that for every $t\in [t_J,t^*]$ and every $x\in \N^{J+1}_t$ such that $d_t(x,\D\N^{J+1}_t)<Fr_{\comp}$ we have $d_{g^+_{t_J}}(x(t_J),N^+\setminus N_0)<\delta_\#^{-1} r_{\comp}$.  
Hence 
\[
|h^{J+1}(x)|=|(\chi_{t}^{-1})^*g^{\prime+}_{t}-g^+_{t}|_{g^+_{t}}(x (t_J)) 
<e^{-HT}10^{-E}r_{\comp}^E\rho_1^{-E-1}(x)\eta_{\lin}
\]
by (\ref{eqn_h_bound_for_assertion_b}).   This yields
assertion \ref{ass_12.24_b} of this proposition.

Finally, we verify assertions \ref{ass_12.24_c} and \ref{ass_12.24_d} of this proposition.  
We first apply Lemma~\ref{lem_scale_distortion} and a priori assumptions \ref{item_lambda_thick_2}, \ref{item_backward_time_slice_3}(a), (c), \ref{item_eta_less_than_eta_lin_13}, assuming
\[
 \eta_{\lin} \leq \ov\eta_{\lin}, \qquad
\delta_{\nn} \leq \ov\delta_{\nn}, \qquad
\lambda \leq \ov\lambda, \qquad
\Lambda \geq 2, \qquad
\eps_{\can} \leq \ov\eps_{\can}(\lambda) , \qquad 
r_{\comp} \leq \ov{r}_{\comp}, 
\]
to find that for all $t \in [t_J, t^*]$ the following holds:
If the $\eps_{\can}$-canonical neighborhood assumption holds at scales $(0,1)$ on $\phi^{J+1}_t (\N^{J+1}_t)$, then $\phi^{J+1} (\N^{J+1}_t)$ is $C_{\sd} \lambda r_{\comp}$-thick.
By assertion \ref{ass_12.3_e} of Proposition~\ref{Prop_performing_cap_extensions}, this condition holds for $t = t_J$.
Therefore, assuming 
\[ \eps_{\can} < C_{\sd}^{-1} \lambda, \]
it holds for all $t \in [t_J, t^*]$ by continuity.
This shows assertion \ref{ass_12.24_d} of this proposition and the fact that $\phi^{J+1}_{t^*} (\N^{J+1}_{t^*})$ is $C_{\sd} \lambda r_{\comp}$-thick.

To see assertion \ref{ass_12.24_c} of this proposition, assume that $|h^{J+1}| \leq \eta_{\lin}$ on $\N^{J+1}_{t^*}$.
Then by the definition of $\phi^{J+1}$ we have the bound $|(\chi_{t^*}^{-1})^*g^{\prime+}_{t^*}-g^+_{t^*}|_{g^+_{t^*}}\leq\eta_{\lin}$ on $N_0$.  By (\ref{eqn_h_bound_for_assertion_b}) it follows that $|(\chi_{t^*}^{-1})^*g^{\prime+}_{t^*}-g^+_{t^*}|_{g^+_{t^*}}\leq\eta_{\lin}$  holds everywhere on $N^+$.
So Property~(\ref{li_12.24_4}) above implies $t^*=t^*_1$.
Combining this with the conclusion from the previous paragraph and applying assertion \ref{ass_12.24_cl5_j} of Claim~\ref{cl_12.3_5}, yields $t^* = t^*_1 = t_{J+1}$, as desired.
\end{proof}

\bigskip

\subsection{Proof of Proposition~\ref{Prop_extend_comparison_by_one}, concluded} \label{subsec_verification_of_APAs}
In this subsection we use the results of the previous subsections to prove our main Proposition~\ref{Prop_extend_comparison_by_one}.
More specifically, we will analyze the map $\phi^{J+1} : \N^{J+1}_{[t_J, t^*]} \to \M'$ that was constructed in Proposition~\ref{Prop_extend_phi_to_t_star}.
We will verify that this map satisfies a priori assumptions \ref{item_time_step_r_comp_1}--\ref{item_q_less_than_nu_q_bar_12} and show that $t^* = t_{J+1}$.
Therefore $\phi^{J+1}$ can be used to extend the comparison $( \phi, \{ \phi^j \}_{j=1}^J, \Cut )$ to the time-interval $[0, t_{J+1}]$.
This will finish the proof of Proposition~\ref{Prop_extend_comparison_by_one}.

Our proof can be roughly summarized as follows:
By the assumptions of Proposition~\ref{Prop_extend_comparison_by_one}  we may assume that a priori assumptions \ref{item_time_step_r_comp_1}--\ref{item_geometry_cap_extension_5} already hold until time $t_{J+1}$ and a priori assumptions \ref{item_eta_less_than_eta_lin_13} and \ref{item_q_less_than_q_bar_6}--\ref{item_q_less_than_nu_q_bar_12} already hold until time $t_J$. 
We will refer to this assumption as the ``induction hypothesis'' henceforth.
Using the induction hypothesis and the conclusions of Propositions~\ref{Prop_performing_cap_extensions}  and \ref{Prop_extend_phi_to_t_star}, we will then establish that a priori assumptions \ref{item_time_step_r_comp_1}--\ref{item_q_less_than_nu_q_bar_12} hold up to time $t^*$.
The only non-trivial assumptions in this step will be a priori assumptions \ref{item_eta_less_than_eta_lin_13}--\ref{item_q_star_less_than_q_star_bar}.
We will prove these assumptions using another continuity argument:
We will assume that relaxed versions of a priori assumptions \ref{item_q_less_than_q_bar_6}--\ref{item_q_star_less_than_q_star_bar} hold up to some almost maximal time $t^{**} \leq t^*$ and, based on these extra assumptions, we prove a priori assumptions \ref{item_eta_less_than_eta_lin_13}--\ref{item_q_star_less_than_q_star_bar} up to time $t^{**}$.
By a straightforward openness argument it therefore follows that $t^{**} = t^*$ and therefore that a priori assumptions \ref{item_eta_less_than_eta_lin_13}--\ref{item_q_star_less_than_q_star_bar} hold up to time $t^*$.
Eventually, the fact that a priori assumption \ref{item_eta_less_than_eta_lin_13} holds up to time $t^*$ and assertion \ref{ass_12.24_c} of Proposition~\ref{Prop_extend_phi_to_t_star} imply that we indeed have $t^* = t_{J+1}$.
This will finish our proof.

We remark that throughout this entire subsection, we will introduce global terminology and assumptions on the parameters, which will be understood to remain valid for the remainder of the subsection.
In particular, conditions on the parameters that can be found in the following lemmas will be assumed to hold for the remainder of the subsection so that the conclusions of these lemmas can be applied immediately.

This subsection is structured as follows:
We first set up our argument by recalling the important assumptions from Proposition~\ref{Prop_extend_comparison_by_one}.
In Lemma~\ref{lem_summary_APA_10_12}, we will then summarize and put into context the results of the constructions from Propositions~\ref{Prop_performing_cap_extensions} and \ref{Prop_extend_phi_to_t_star}.
Next, we introduce the relaxed versions of a priori assumptions \ref{item_q_less_than_q_bar_6}--\ref{item_q_star_less_than_q_star_bar} in equations (\ref{eq_relaxed_APA_6})--(\ref{eq_relaxed_APA_8}), which hold up to some time $t^{**} \leq t^*$.
In Lemma~\ref{lem_setup_of_t_star_star}, we show that $t^{**} > t_J$ and that if the strong versions of a priori assumptions \ref{item_q_less_than_q_bar_6}--\ref{item_q_star_less_than_q_star_bar} hold up to time $t^{**}$, then we must in fact have $t^{**} = t^*$.   Based on these relaxed versions of a priori assumptions \ref{item_q_less_than_q_bar_6}--\ref{item_q_star_less_than_q_star_bar}, we will establish a priori assumptions \ref{item_eta_less_than_eta_lin_13}--\ref{item_q_star_less_than_q_star_bar} in Lemmas~\ref{lem_verification_of_APA_6}, \ref{lem_verification_APA_7}, \ref{lem_verification_of_APA_8} and \ref{lem_verification_APA_9} --- one lemma per a priori assumption and in this order.
Lastly, we wrap up our discussion, argue that $t^{**} = t^* = t_{J+1}$ and verify the assertions of Proposition~\ref{Prop_extend_comparison_by_one}.

Further explanations of the arguments may be found after the statements of the Lemmas below.

In what follows, we will be considering the setup as described in assumptions \ref{con_12.1_i}--\ref{con_12.1_vi} of Proposition~\ref{Prop_extend_comparison_by_one}.
So, among other things, we assume that $\M, \M'$ are $(\eps_{\can} r_{\comp}, T)$-complete and satisfy the $\eps_{\can}$-canonical neighborhood assumption at scales $(\eps_{\can} r_{\comp}, 1)$.
We consider a comparison domain $(\N, \{ \N^j \}_{j=1}^{J+1}, \{ t_j \}_{j=0}^{J+1})$ over the time-interval $[0, t_{J+1}]$, for $J \geq 0$, and a comparison $(\Cut, \phi, \{ \phi^j \}_{j=1}^J )$ from $\M$ to $\M'$ defined on this comparison domain over the time-interval $[0,t_J]$.
If $J = 0$, then this comparison is trivial, as explained after Definition~\ref{def_comparison}. 
We also assume $(\N, \{ \N^j \}_{j=1}^{J+1}, \{ t_j \}_{j=0}^{J+1})$ and $(\Cut, \phi, \{ \phi^j \}_{j=1}^J )$ satisfy a priori assumptions \ref{item_time_step_r_comp_1}--\ref{item_eta_less_than_eta_lin_13} for the parameters $(\eta_{\lin}, \delta_{\nn}, \lambda, D_{\CAP}, \Lambda, \delta_{\bb}, \eps_{\can}, r_{\comp})$  and a priori assumptions \ref{item_q_less_than_q_bar_6}--\ref{item_apa_13} for the parameters $(T, \lb E, \lb H, \lb \eta_{\lin}, \lb \nu, \lb \lambda, \lb \eta_{\cut}, \lb D_{\cut}, \lb W, \lb A, \lb r_{\comp})$.
Moreover, we assume in the following that
\begin{equation} \label{eq_t_J_1_less_T_verification}
t_{J+1} \leq T.
\end{equation}

If $J \geq 1$, then assumptions \ref{con_12.1_i}--\ref{con_12.1_v} of Proposition~\ref{Prop_extend_comparison_by_one} allow us to apply Proposition~\ref{Prop_performing_cap_extensions}.
Doing so yields the map $\wh\phi : \N_{t_J-} \cup \N_{t_J+} \to \M'_{t_J}$ and the set of cuts $\Cut^J$ with the properties as explained in assertions \ref{ass_12.3_a}--\ref{ass_12.3_e} of this proposition.
If $J = 0$, then we skip this step.

Next, we fix an auxiliary constant $F < \infty$, whose value will be determined in the course of this subsection depending only on $E$.
We can then apply Proposition~\ref{Prop_extend_phi_to_t_star} for $\wh\phi :  X := \N_{t_J-} \cup \N_{t_J+} \to \M'_{t_J}$  from Proposition~\ref{Prop_performing_cap_extensions} (if $J \geq 1$) or  $\wh\phi = \zeta : X \to \M'_0$ from assumption \ref{con_12.1_vi} of Proposition~\ref{Prop_extend_comparison_by_one} (if $J = 0$).
Then Proposition~\ref{Prop_extend_phi_to_t_star} yields a time $t^* \in (t_J, t_{J+1}]$ and a map $\phi^{J+1} : \N^{J+1}_{[t_J, t^*]} \to \M'$ satisfying assertions \ref{ass_12.24_a}--\ref{ass_12.24_d} of this proposition.
Note that Propositions~\ref{Prop_performing_cap_extensions} and \ref{Prop_extend_phi_to_t_star} are only applicable if our parameters satisfy the bounds (\ref{eq_parameters_performing_cap_extensions}) and (\ref{eq_parameters_extend_phi_to_t_star}).
These bounds are implied by bounds of the form (\ref{eq_parameters_extend_comparison_by_one})  and
\begin{gather*}
\nu \leq \ov\nu (T, E, F, H, \eta_{\lin}), 
 \qquad
 \delta_{\nn} \leq \ov\delta_{\nn} (T, E, F, H,  \eta_{\lin}),  
 \\
\delta_{\bb} \leq \ov\delta_{\bb} (T, E, F, H, \eta_{\lin},  \lambda, D_{\cut}, A , \Lambda), \qquad 
\eps_{\can} \leq \ov\eps_{\can} (T, E, F, H, \eta_{\lin},  \lambda ,D_{\cut}, A, \Lambda).
\end{gather*}
(Note that assuming $F = F(E)$, these bounds also follow from bounds of the form (\ref{eq_parameters_extend_comparison_by_one}).)

In the following lemma we summarize the important properties of $\Cut^J$ and $\phi^{J+1}$ and we show how these objects can be used to extend the comparison $(\Cut, \phi, \{ \phi^j \}_{j=1}^J)$ to a comparison that is defined over the time-interval $[0, t^*]$.

\begin{lemma} \label{lem_summary_APA_10_12}
There is a unique map $\ov\phi : \N_{[0, t^*]} \setminus \cup_{\DD \in \Cut \cup \Cut^J} \DD \to \M'$ such that $(\Cut \cup \Cut^{J}, \lb \ov\phi, \lb \{ \phi^j \}_{j=0}^{J+1})$ is a comparison defined on the comparison domain $(\N_{[0 ,t^*]}, \lb  \{ \N^j, \N^{J+1}_{[t_J, t^*]} \}_{j=1}^J, \lb \{ t_j, t^* \}_{j=0}^J )$.
This comparison is an extension of the comparison $(\Cut , \phi, \{ \phi^j \}_{j=1}^{J})$ in the sense that
\[ \ov\phi = \phi \qquad \text{on} \qquad \cup_{j=1}^J \N^j  \setminus \cup_{\DD \in \Cut  \cup \Cut^J} \DD. \]

Furthermore, this extended comparison and the comparison domain $(\N_{[0 ,t^*]}, \lb  \{ \N^j, \lb \N^{J+1}_{[t_J, t^*]} \}_{j=1}^J, \lb \{ t_j, t^* \}_{j=0}^J )$ satisfy the following properties: 
\begin{enumerate}[label=(\alph*)]
\item \label{ass_12.47_a} They satisfy  a priori assumptions  \ref{item_h_derivative_bounds_9}--\ref{item_apa_13} for the parameters $( T, \lb E, \lb H, \lb \eta_{\lin}, \lb \nu, \lb \lambda, \lb \eta_{\cut},  \lb  D_{\cut}, \lb W, \lb A, \lb r_{\comp})$.
\item \label{ass_12.47_b} Let $(h, \{ h^j \}_{j=1}^{J+1})$ be the associated Ricci-DeTurck perturbation (note that $h^{J+1}$ is only defined over the time-interval $[t_J, t^*]$).
Then $|h| \leq \eta_{\lin}$ on $\cup_{j=1}^J \N^j  \setminus \cup_{\DD \in \Cut \cup \Cut^J} \DD$ and $|h^{J+1}| \leq  10 \eta_{\lin}$ on $\N_{[t_J,t^*]}^{J+1}$.
Moreover, $\phi^{J+1} (\N^{J+1}_{[t_J, t^*]})$ is $\eps_{\can} r_{\comp}$-thick.
\item \label{ass_12.47_c} For any $t \in [t_J, t^*]$ and $x \in \N^{J+1}_{t}$ with $d_t (x, \partial \N^{J+1}_t) < F r_{\comp}$ we have
\[ Q_+(x) = e^{H(T-t)} \rho_1^E (x) |h^{J+1}(x)| \leq \ov{Q} = 10^{-E-1} \eta_{\lin} r_{\comp}^E. \]
\item \label{ass_12.47_d} If even $|h| \leq  \eta_{\lin}$ on $\N^{J+1}_{t^*}$, then $t^* = t_{J+1}$.
\end{enumerate}
\end{lemma}

\begin{proof}
The construction of the map $\ov\phi$ and the verification of the properties of Definitions~\ref{def_comparison_domain} and \ref{def_comparison} are straightforward.
A priori assumptions \ref{item_h_derivative_bounds_9} and \ref{item_cut_diameter_less_than_d_r_comp_11} follow directly from the corresponding a priori assumptions of the induction hypothesis and Proposition~\ref{Prop_performing_cap_extensions}\ref{ass_12.3_d} and \ref{ass_12.3_c}.
A priori assumption \ref{item_q_less_than_nu_q_bar_12} follows directly from a priori assumption \ref{item_q_less_than_nu_q_bar_12} of the induction hypothesis (if $J \geq 1$) or from assumption \ref{con_12.1_vi} in Proposition~\ref{Prop_extend_comparison_by_one} (if $J = 0$).
Assertions \ref{ass_12.47_b}--\ref{ass_12.47_d} are just restatements of assertions \ref{ass_12.24_a}--\ref{ass_12.24_d} of Proposition~\ref{Prop_extend_phi_to_t_star} combined with the induction hypthesis.
\end{proof}

Note that by the assumptions of Proposition~\ref{Prop_extend_comparison_by_one}, the comparison domain $(\N, \lb \{ \N^j \}_{j=1}^{J+1}, \lb \{ t_j \}_{j=0}^{J+1})$ satisfies a priori assumptions \ref{item_time_step_r_comp_1}--\ref{item_geometry_cap_extension_5} for the parameters $(\eta_{\lin}, \lb \delta_{\nn}, \lb \lambda, \lb D_{\CAP}, \lb \Lambda, \lb  \delta_{\bb}, \lb \eps_{\can}, \lb r_{\comp})$.  For the remainder of this section, references to \ref{item_time_step_r_comp_1}--\ref{item_geometry_cap_extension_5} will implicitly refer to this larger comparison domain, rather than the comparison domain defined on the shorter interval $[0,t^*]$.

It remains to verify a priori assumptions \ref{item_eta_less_than_eta_lin_13}--\ref{item_q_star_less_than_q_star_bar}.
Once this has been accomplished, assertion \ref{ass_12.47_d} of Lemma~\ref{lem_summary_APA_10_12} will immediately imply that $t^* = t_{J+1}$.
So we have reduced our discussion to an analysis of the associated Ricci-DeTurck perturbation and its derived quantities $Q$ and $Q^*$.

We will verify a priori assumptions \ref{item_eta_less_than_eta_lin_13}--\ref{item_q_star_less_than_q_star_bar} via another continuity argument, which we will set up now.
Consider the comparison $(\Cut \cup \Cut^J, \ov\phi, \{ \phi^j \}^{J+1}_{j=0})$ from Lemma~\ref{lem_summary_APA_10_12} and let $(h, \{ h^j \}_{j=1}^{J+1})$ be the associated Ricci-DeTurck perturbation, as mentioned in assertion \ref{ass_12.47_b} of this lemma.
As in Definition~\ref{def_a_priori_assumptions_7_13} we define the quantities
\[ Q = e^{H (T- \mathfrak{t})} \rho_1^E |h|, \qquad Q^* = e^{H(T-\mathfrak{t})} \rho_1^3 |h| \]
and the extensions $Q_{\pm}$ and $Q^*_{\pm}$ to $\N_{t_j \pm}$.
Moreover, again as in Definition~\ref{def_a_priori_assumptions_7_13}, we set 
\[ \ov{Q} := 10^{-E-1} \eta_{\lin} r_{\comp}^{E}, \qquad \ov{Q}^* := 10^{-1} \eta_{\lin} (\lambda r_{\comp})^3. \]
Choose a time $t^{**} \in [t_J, t^*]$ such that the following conditions hold for all $x \in \N^{J+1}_{[t_J, t^{**}]} \setminus \cup_{\DD\in \Cut^J}\DD$:
\begin{alignat}{2}
Q (x) &\leq 10 \overline{Q} &\qquad &\text{whenever $P(x, 10A \rho_1 (x)) \cap \DD = \emptyset$} \label{eq_relaxed_APA_6} \\ &&&\text{for all $\DD \in \Cut \cup \Cut^{J}$} \notag \\
Q(x) &\leq 10W \overline{Q} &\qquad & 
 \label{eq_relaxed_APA_7} \\
Q^* (x) &\leq 10 \overline{Q}^* &\qquad &\text{whenever $B(x, 10 A \rho_1 (x)) \subset \N_{\t(x)-}$} \label{eq_relaxed_APA_8} 
\end{alignat}

Note that these conditions are relaxed versions of a priori assumptions \ref{item_q_less_than_q_bar_6}--\ref{item_q_star_less_than_q_star_bar}.
The main objective of this subsection will be to show --- under certain bounds on our parameters --- that assumptions (\ref{eq_relaxed_APA_6})--(\ref{eq_relaxed_APA_8}) imply a priori assumptions \ref{item_eta_less_than_eta_lin_13} and \ref{item_q_less_than_q_bar_6}--\ref{item_q_star_less_than_q_star_bar} up to time $t^{**}$.
The following lemma will help us conclude that it is possible to choose $t^{**} = t^*$ if a priori assumptions \ref{item_q_less_than_q_bar_6}--\ref{item_q_star_less_than_q_star_bar} have been established.

\begin{lemma} \label{lem_setup_of_t_star_star}
If
\begin{gather*}
E \geq \underline{E}, \qquad
 F \geq \underline{F},  \qquad
 \eta_{\lin} \leq \ov\eta_{\lin}, \qquad
\delta_{\nn} \leq \ov\delta_{\nn}, \qquad
\lambda \leq \ov\lambda, \qquad
\eta_{\cut} \leq \ov\eta_{\cut},  \\
D_{\cut} \geq \underline{D}_{\cut} (\lambda), \qquad 
W \geq \underline{W}(E, \lambda, D_{\cut}), \qquad 
A \geq \underline{A}, \qquad 
\Lambda \geq \underline\Lambda,  \\
\delta_{\bb} \leq \ov\delta_{\bb} (\lambda, D_{\cut}, A, \Lambda ), \qquad 
\eps_{\can} \leq \ov\eps_{\can} (\lambda, D_{\cut}, A, \Lambda ), \qquad
r_{\comp} \leq \ov{r}_{\comp} (\lambda), \qquad
\end{gather*}
then we can choose $t^{**} > t_J$.

Furthermore, there is a constant $\tau = \tau (T, E, H, \eta_{\lin}, \lambda, A, r_{\comp}) > 0$ with the following property.
If a priori assumptions \ref{item_q_less_than_q_bar_6}--\ref{item_q_star_less_than_q_star_bar} hold up to time $t^{**}$, meaning that for all $x \in \N^{J+1}_{[t_J, t^{**}]}\setminus \cup_{\DD\in \Cut^J}\DD$
\begin{alignat}{2}
Q(x) &\leq \overline{Q} &\qquad &\text{whenever $P(x,A \rho_1 (x)) \cap \DD = \emptyset$} \label{eq_relaxed_APA_6_modified} \\ &&&\text{for all $\DD \in \Cut \cup \Cut^{J}$} \notag \\
Q(x) &\leq W \overline{Q} &\qquad & \label{eq_relaxed_APA_7_modified} \\
Q^*(x) &\leq  \overline{Q}^* &\qquad &\text{whenever $B(x,  A \rho_1 (x)) \subset \N_{\mathfrak{t}(x)-}$} \label{eq_relaxed_APA_8_modified} 
\end{alignat}
then (\ref{eq_relaxed_APA_6})--(\ref{eq_relaxed_APA_8}) even hold for all $x \in \N^{J+1}_{[t_J, \min \{ t^{**} + \tau, t^* \}]}\setminus \cup_{\DD\in \Cut^J}\DD$.
\end{lemma}

In other words, if \ref{item_q_less_than_q_bar_6}--\ref{item_q_star_less_than_q_star_bar} hold up to time $t^{**}$, then we may replace $t^{**}$ by $\min \{ t^{**} + \tau, t^* \}$.
The important point here is that $\tau$ can be chosen independently of $t^{**}$.
In Lemmas~\ref{lem_verification_of_APA_6}, \ref{lem_verification_APA_7}, \ref{lem_verification_of_APA_8} and \ref{lem_verification_APA_9}
 below, we will show that a priori assumptions \ref{item_q_less_than_q_bar_6}--\ref{item_q_star_less_than_q_star_bar} indeed hold up to time $t^{**}$, regardless of the choice of $t^{**}$.
It will then follow by iterating Lemma \ref{lem_setup_of_t_star_star}  that we can choose $t^{**} = t^*$ and that a priori assumptions \ref{item_q_less_than_q_bar_6}--\ref{item_q_star_less_than_q_star_bar} hold up to time $t^{*}$.

The main idea of the proof of Lemma~\ref{lem_setup_of_t_star_star} is that the relaxed conditions (\ref{eq_relaxed_APA_6})--(\ref{eq_relaxed_APA_8}) hold in the neighborhood of any point at which the stricter conditions (\ref{eq_relaxed_APA_6_modified})--(\ref{eq_relaxed_APA_8_modified}) are satisfied.
Using the canonical neighborhood assumption, we will find a uniform lower bound on the size of such a neighborhood.
Extra care has to be taken near the cuts at time $t_J$.
Here we will use the a priori assumptions from our induction hypothesis along with the geometry of the cuts to deduce that (\ref{eq_relaxed_APA_7_modified}) and (\ref{eq_relaxed_APA_8_modified}) even hold on and near the cuts.

\begin{proof} 
We first show that for all $x \in \N_{t_J+}$ (which may possibly lie on a cut)
\begin{alignat}{2}
Q_+ (x) &\leq W \overline{Q} &\qquad & 
 \label{eq_relaxed_APA_7_tJ_plus} \\
Q^*_+ (x) &\leq  \overline{Q}^* &\qquad &\text{whenever $B(x,  A \rho_1 (x)) \subset \N_{t_J+}$} \label{eq_relaxed_APA_8_tJ_plus} 
\end{alignat}
Note that the condition in (\ref{eq_relaxed_APA_8_tJ_plus}) refers to the \emph{forward} time-slice, in contrast to (\ref{eq_relaxed_APA_8_modified}).

Let us first prove (\ref{eq_relaxed_APA_8_tJ_plus}).
If $x \in \DD \in \Cut^J$, then (\ref{eq_relaxed_APA_8_tJ_plus}) follows from a priori assumption \ref{item_h_derivative_bounds_9} (see assertion \ref{ass_12.47_a} of Lemma~\ref{lem_summary_APA_10_12}), assuming
\[ \eta_{\cut} \leq 1. \]
So assume that $x \in \N_{t_J+} \setminus \cup_{\DD \in \Cut^J} \DD$.
If $B(x,  A \rho_1 (x)) \subset \N_{t_J-}$, then (\ref{eq_relaxed_APA_8_tJ_plus}) follows from a priori assumption \ref{item_q_star_less_than_q_star_bar} of the induction hypothesis.
So assume that $B(x,  A \rho_1 (x)) \not\subset \N_{t_J-}$, but $B(x,  A \rho_1 (x)) \subset \N_{t_J+}$.
In other words, $B(x, A \rho_1 (x))$ intersects an extension cap $\C_0 \subset \N_{t_J+} \setminus \Int \N_{t_J-}$.
Choose $\DD_0 \in \Cut^J$ with $\C_0 \subset \DD_0$ and let $C_\# < \infty$ be a constant whose value we will determine in the course of the proof.
We now apply Lemma~\ref{lem_large_scale_near_neck} for $A_0 = A$ and $T_0 = 0$, assuming
\begin{multline*}
\delta_{\nn} \leq \ov\delta_{\nn}, \qquad
\lambda \leq \ov\lambda, \qquad
D_{\cut} \geq \underline{D}_{\cut} (\lambda, C_\#), \qquad
\Lambda \geq \underline{\Lambda}, \\ 
\delta_{\bb} \leq \ov\delta_{\bb} (\lambda, C_\#, D_{\cut}, A, \Lambda ), \qquad
\eps_{\can} \leq \ov\eps_{\can} (\lambda, D_{\cut}, A, \Lambda ), \qquad 
r_{\comp} \leq \ov{r}_{\comp} (C_\#)
\end{multline*}
Note that the assumptions of this lemma on the set $\Cut \cup \Cut^J$ hold due to a priori assumption \ref{item_cut_diameter_less_than_d_r_comp_11}, which holds due to assertion \ref{ass_12.47_a} of Lemma~\ref{lem_summary_APA_10_12}.
We find that
\begin{equation} \label{eq_rho_C_sharp_r_comp}
\rho_1 (x) \geq C_\# r_{\comp}
\end{equation}
and that $P(x, A \rho_1 (x)) \cap \DD = \emptyset$ for all $\DD \in \Cut$.
So by a priori assumption \ref{item_q_less_than_q_bar_6} of the induction hypothesis we have 
\[ e^{H(T-t_J)} \rho_1^E(x) |h (x)| = Q(x) \leq \ov{Q} = 10^{-E-1} \eta_{\lin} r_{\comp}^E. \]
Combining this with (\ref{eq_rho_C_sharp_r_comp}) yields
\begin{equation*}
 Q^* (x) = e^{H(T-t_J)} \rho_1^3 (x) |h(x)| \leq \rho_1^{3-E}(x) \cdot 10^{-E-1} \eta_{\lin} r_{\comp}^E 
\leq C_\#^{3-E} \eta_{\lin} r_{\comp}^3. 
\end{equation*}
It follows that $Q^*(x) \leq \ov{Q}^*$ if $C_\#^{3-E} \leq 10^{-1} \lambda^3$, which holds assuming
\[ E \geq 4, \qquad
C_\# \geq \underline{C}_\# (\lambda). \]
This finishes the proof of (\ref{eq_relaxed_APA_8_tJ_plus}).

To see the bound (\ref{eq_relaxed_APA_7_tJ_plus}), we only need to consider the case $x \in \DD \in \Cut^J$, due to a priori assumption \ref{item_q_less_than_w_q_bar_7} of the induction hypothesis.
Then, again by a priori assumption \ref{item_h_derivative_bounds_9} (see assertion \ref{ass_12.47_a} of Lemma~\ref{lem_summary_APA_10_12}) we have $Q^*_+ (x) \leq \eta_{\cut} \ov{Q}^*$.
By a priori assumptions \ref{item_geometry_cap_extension_5}  and \ref{item_cut_diameter_less_than_d_r_comp_11}, and assuming
\[ \delta_{\bb} \leq \ov{\delta}_{\bb} (\lambda, D_{\cut}), \]
we conclude that there is a constant $C' = C' (\lambda, D_{\cut}) < \infty$ such that $\rho_1 < C' r_{\comp}$ on $\DD$.
So      
\begin{multline*}
 Q_+ (x) = e^{H(T-t_J)} \rho_1^E(x) |h_{t_J+} (x)| 
 = \rho_1^{E-3} (x) Q^*_+ (x) \\
 \leq (C')^{E-3} r_{\comp}^{E-3} \cdot \eta_{\cut} \ov{Q}^*
 = (C')^{E-3} r_{\comp}^{E-3} \cdot \eta_{\cut}  \cdot 10^{-1} \eta_{\lin} (\lambda r_{\comp})^3 \\
 = (C')^{E-3} \lambda^{3} 10^{E} \eta_{\cut} \cdot \ov{Q}. 
\end{multline*}
It follows that $Q_+ (x) \leq W \ov{Q}$, assuming
\[ \eta_{\cut} \leq 1, \qquad
W \geq \underline{W} (E, \lambda, D_{\cut}). \]
This finishes the proof of (\ref{eq_relaxed_APA_7_tJ_plus}).

We will show that (\ref{eq_relaxed_APA_6})--(\ref{eq_relaxed_APA_8}) hold slightly beyond time $t^{**}$ if $t^{**} < t^*$.
The fact that we can choose $t^{**} > t_J$ will follow along the lines of the proof.

Let $\tau>0$ be a constant to be determined in the course of the proof.  
It  suffices to argue that if $t^{**} < t^*$ and if (\ref{eq_relaxed_APA_6_modified})--(\ref{eq_relaxed_APA_8_modified}) hold on $\N^{J+1}_{[t_J, t^{**}]} \setminus \cup_{\DD \in \Cut \cup \Cut^J} \DD$ and (\ref{eq_relaxed_APA_7_tJ_plus}), (\ref{eq_relaxed_APA_8_tJ_plus}) hold on $\N_{t_J+}$, then (\ref{eq_relaxed_APA_6})--(\ref{eq_relaxed_APA_8}) hold in $\N^{J+1}_{t'}$ whenever $t' \in (t^{**}, t^*]$ and $t' - t^{**}<\tau$ where $\tau\leq\ov\tau(T, E, H,  \eta_{\lin}, \lambda, A, r_{\comp})$.  
To that end, choose $t'\in (t^{**},t^*]$ with $t'-t^{**}<\tau$, and a point $x' \in \N^{J+1}_{t^\prime}$.
Since $\N^{J+1}_{[t_J,t*]}$ is a product domain, we have $x'=x(t')$ for some $x\in \N^{J+1}_{t^{**}}$.  

First suppose that $d_{t^{**}} (x, \partial \N^{J+1}_{t^{**}}) \leq  r_{\comp}$.   Assuming
\[ \tau\leq\ov\tau(\lambda,r_{\comp})\,,\]
then by a distance distortion estimate based on a priori assumption \ref{item_lambda_thick_2} and the fact that $\N^{J+1}_{[t_J,t^*]}$ is a product domain, we obtain $d_{t' } (x(t'), \lb \partial \N^{J+1}_{t' }) \lb <  \lb 10 r_{\comp}$.
So assuming
\[ F \geq 10, \]
assertion \ref{ass_12.47_c} of Lemma~\ref{lem_summary_APA_10_12} implies that (\ref{eq_relaxed_APA_6}) holds for $x(t')$.  
Thus if
\[ W \geq 1, \]
then (\ref{eq_relaxed_APA_7}) holds as well.
Next, by a priori assumption \ref{item_backward_time_slice_3}(a), Lemma~\ref{lem_time_slice_neck_implies_space_time_neck}, and assuming
\[ \delta_{\nn} \leq \ov\delta_{\nn}, \qquad
\eps_{\can} \leq \ov\eps_{\can}, \qquad
r_{\comp} \leq \ov{r}_{\comp}, \]
we obtain that $\rho(x(t')) > \frac12 r_{\comp}$.
So $B(x(t'), A \rho_1 (x(t'))) \not\subset \N$, assuming
\[ A > 20\,, \]
and thus (\ref{eq_relaxed_APA_8}) holds.

Now suppose that $d_{t^{**}} (x, \partial \N^{J+1}_{t^{**}}) >  r_{\comp}$.
By a priori assumption \ref{item_lambda_thick_2}, Lemma~\ref{lem_forward_backward_control}, and assuming
\[ \tau\leq\ov\tau(E, \lambda, r_{\comp})\,,\qquad \eps_{\can} \leq \ov\eps_{\can} (\lambda)\,, \]
we obtain that
\begin{equation} \label{eq_091E}
 (0.9)^{1/E} \rho_1(x) \leq \rho_1 \leq (1.1)^{1/E} \rho_1 (x) \qquad \text{on} \qquad P(x, \tau, \tau). 
\end{equation}
Thus on $P(x, \tau, \tau)$
\begin{equation} \label{eq_0911Qh}
 0. 9 \cdot e^{ H (T-\mathfrak{t})} \rho_1^E (x) \cdot |h| \leq Q \leq 1.1 \cdot e^{ H (T-\mathfrak{t})} \rho_1^E (x) \cdot |h|. 
\end{equation}
Assume now that $P(x(t'), 10 A \rho_1 (x(t'))) \cap \DD = \emptyset$ for all $\DD \in \Cut \cup \Cut^J$.
 By a priori assumption \ref{item_lambda_thick_2} and bounded curvature at bounded distance, Lemma~\ref{lem_bounded_curv_bounded_dist}, and assuming
\[ \eps_{\can} \leq \ov\eps_{\can} (\lambda, A), \]
we conclude that $P(x(t'), 10 A \rho_1 (x(t')))$ is unscathed.
We also obtain a curvature bound on this parabolic neighborhood, which implies via a distance distortion estimate that
\[ P(x, 9 A \rho_1 (x(t'))) \subset P(x(t'), 10 A \rho_1 (x(t'))), \]
assuming
\[ \tau \leq \ov\tau (\lambda, A, r_{\comp}). \]
Combining this with (\ref{eq_091E}) and a priori assumption \ref{item_lambda_thick_2}, and assuming
\[ E \geq 1, \qquad
A \geq \underline{A}, \qquad
\tau \leq \ov\tau (\lambda, A , r_{\comp}),   \]
we obtain that $P(y, A \rho_1 (y)) \subset P(x(t'), 10 A \rho_1 (x(t')))$ for all $y \in B(x, \tau)$.
This implies that for all such $y$ we have $P(y, A \rho_1 (y)) \cap \DD = \emptyset$ for all $\DD \in \Cut \cup \Cut^J$.
Therefore, by (\ref{eq_relaxed_APA_6_modified}) we have $Q  \leq \ov{Q}$ on $B(x, \tau)$.
So if
$$
\eta_{\lin} \leq \ov\eta_{\lin}\,,
$$
then we can use Proposition~\ref{prop_promote_RdT_forward_locally} together with a priori assumption \ref{item_lambda_thick_2}, (\ref{eq_091E}),  (\ref{eq_0911Qh}) and assertion \ref{ass_12.47_b} of Lemma~\ref{lem_summary_APA_10_12}, and assuming
\begin{equation}
\label{eqn_last_tau_conditions}
\tau\leq\ov\tau(T, E, H, \eta_{\lin}, \lambda, r_{\comp})\,,
\end{equation}
to get that (\ref{eq_relaxed_APA_6}) holds.

 Using similar arguments, Properties (\ref{eq_relaxed_APA_7}) and (\ref{eq_relaxed_APA_8}) can be verified at $x(t')$ as well, assuming a bound of the form
(\ref{eqn_last_tau_conditions}).
Note that if $t^{**} = t_J$, then we need to use the bounds (\ref{eq_relaxed_APA_7_tJ_plus}) and (\ref{eq_relaxed_APA_8_tJ_plus}).
\end{proof}

\bigskip
Assume for the remainder of this subsection that the parameter bounds of Lemma \ref{lem_setup_of_t_star_star} hold and that $t^{**} > t_J$.

In the following we will verify a priori assumptions \ref{item_eta_less_than_eta_lin_13}--\ref{item_q_star_less_than_q_star_bar} up to time $t^{**}$.
Whenever we say that ``a priori assumption (APA $x$) holds'', then we mean that $(\N_{[0,t^{**}]}, \lb \{ \N^j, \lb \N^{J+1}_{[t_J, t^{**}]} \}_{j=1}^J, \lb \{ t_j, t^{**} \}_{j=1}^J)$ and $(\Cut \cup \Cut^{J}, \lb \ov\phi |_{\N_{[0,t^{**}]}}, \lb \{ \phi^j, \phi^{J+1} |_{\N^{J+1}_{[t_J,t^{**}]}} \}_{j=1}^{J})$ satisfy a priori assumption (APA $x$) for the set of parameters $(T, \lb E, \lb H, \lb \eta_{\lin}, \lb \nu, \lb \lambda, \lb \eta_{\cut}, \lb D_{\cut}, \lb W, \lb A, \lb r_{\comp})$.
Note that it follows from assertion \ref{ass_12.47_a} of Lemma~\ref{lem_summary_APA_10_12}, that a priori assumptions \ref{item_h_derivative_bounds_9}--\ref{item_apa_13} hold.

Let us first verify a priori assumption \ref{item_eta_less_than_eta_lin_13}.

\begin{lemma}[Verification of \ref{item_eta_less_than_eta_lin_13}] \label{lem_verification_of_APA_6}
If
\begin{gather*}
 \delta_{\nn} \leq \ov\delta_{\nn}, \qquad \lambda \leq \ov\lambda,  \qquad 
 \Lambda \geq \underline{\Lambda} ( \lambda, A ),  \qquad \delta_{\bb} \leq \ov\delta_{\bb} (\lambda, D_{\cut}, A, \Lambda ), \\
   \qquad \eps_{\can} \leq \ov\eps_{\can} (\lambda, D_{\cut}, A, \Lambda ), \qquad  r_{\comp} \leq \ov{r}_{\comp} ,
\end{gather*}
then a priori assumption \ref{item_eta_less_than_eta_lin_13} holds.
In other words, we have $|h| \leq \eta_{\lin}$ on $\N_{[0, t^{**}]} \setminus \cup_{\DD \in \Cut \cup \Cut^J} \DD$ and the $\eps_{\can}$-canonical neighborhood assumption holds at scales $(0,1)$ on $\cup_{j=1}^J \phi^j (\N^j) \cup \phi^{J+1} (\N^{J+1}_{[t_J, t^{**}]})$.
\end{lemma}

We summarize the idea of the proof. 
It only remains to establish the bound $|h| \leq \eta_{\lin}$.
For points that are far enough away (compared to $A$) from the neck-like boundary of $ \N \cup \N^{J+1}$, we have $Q^* \leq 10 \ov{Q}^*$ from  (\ref{eq_relaxed_APA_8}), and  together with the lower bound $\rho_1 > \lambda r_{\comp}$ on $\N \cup \N^{J+1}$ from a priori assumption \ref{item_lambda_thick_2}, this implies $|h|\leq \eta_{\lin}$.  
On the other hand, points that are close to this boundary are far from the cuts, by Lemma~\ref{lem_boundary_far_from_cut}.
So at these points we may rely instead on the bound $Q \leq 10 \ov{Q}$ from (\ref{eq_relaxed_APA_6}).  This bound implies $|h| \leq \eta_{\lin}$ as long as $\rho_1 \geq  \frac1{10} r_{\comp}$, a fact which follows from the neck-like structure of the boundary of $\N^{J+1}$ and almost nonnegative curvature  (see Lemma~\ref{lem_thick_close_to_neck}).

\begin{proof}
The second part of a priori assumption \ref{item_eta_less_than_eta_lin_13} follows from assertion \ref{ass_12.47_b} of Lemma~\ref{lem_summary_APA_10_12} and the induction hypothesis.
So it remains to prove the bound $|h| \leq \eta_{\lin}$.
To this end consider a point  $x \in \N_{[0,t^{**}]} \setminus \cup_{\DD \in \Cut \cup \Cut^{J}} \DD$ and set $t := \mathfrak{t} (x)$.
Our goal will be to show $|h(x)| \leq \eta_{\lin}$.
In the case $t \in [0, t_J]$, we are done by a priori assumption \ref{item_eta_less_than_eta_lin_13} from our induction hypothesis, and the fact that $\N_{t_J+}\setminus \cup_{\DD \in \Cut \cup \Cut^{J}} \DD \subset \N_{t_J-}$.
So assume that $t \in (t_J, t^{**}]$ and therefore $x \in \N^{J+1}_{(t_J, t^{**}]}$.

We now distinguish the following two cases:

\textit{Case 1: \quad $B(x, 10 A \rho_1 (x) ) \subset \N^{J+1}_{t}=\N_{t-}$.}

In this case we can apply (\ref{eq_relaxed_APA_8}) and obtain that
\[ e^{H(T- t)} \rho_1^3 (x) |h(x)| = Q^*(x) \leq 10 \ov{Q}^* =  \eta_{\lin} (\lambda r_{\comp})^3. \]
Since by a priori assumption \ref{item_lambda_thick_2} and assumption (\ref{eq_t_J_1_less_T_verification}) we have $\rho_1 (x) > \lambda r_{\comp}$ and $t \leq t_{J+1} \leq T$, this implies $|h(x)| \leq  \eta_{\lin}$.

\textit{Case 2: \quad $B(x, 10 A \rho_1 (x) ) \not\subset \N^{J+1}_t$.}

Let us first apply Lemma~\ref{lem_boundary_far_from_cut} along with a priori assumption \ref{item_cut_diameter_less_than_d_r_comp_11}.
We obtain that if
\begin{gather*}
 \delta_{\nn} \leq \ov\delta_{\nn}, \qquad 
 \lambda\leq\ov\lambda, \qquad
 \Lambda \geq \underline\Lambda, \qquad
\delta_{\bb} \leq \ov\delta_{\bb} (\lambda, D_{\cut}, A, \Lambda), \qquad \\
\eps_{\can} \leq \ov\eps_{\can} (\lambda, D_{\cut}, A, \Lambda), \qquad 
r_{\comp} \leq \ov{r}_{\comp} , 
\end{gather*}
then $P(x, 10 A \rho_1 (x)) \cap \DD = \emptyset$ for all $\DD \in \Cut \cup \Cut^{J}$.
So by (\ref{eq_relaxed_APA_6}) we have  
\[ e^{H (T- t)} \rho_1^E (x) |h(x)| = Q(x) \leq 10 \ov{Q} = 10^{-E} \eta_{\lin}  r_{\comp}^E. \]
By assumption (\ref{eq_t_J_1_less_T_verification}) we have $t \leq t_{J+1} \leq T$.
So in order to show that $|h(x)| \leq \eta_{\lin}$, it suffices to verify the bound
\begin{equation} \label{eq_rho_less_10_r_comp}
 \rho_1 (x)  \geq \frac1{10} r_{\comp}. 
\end{equation}

To see that (\ref{eq_rho_less_10_r_comp}) holds, choose first some point $y \in \partial \N^{J+1}_t$ with $d_t (x,y) < 10 A \rho_1 (x)$.
Let $\Sigma \subset \partial \N^{J+1}_t$ be the (spherical) boundary component of $\N^{J+1}_t$ that contains $y$.
Consider the constant $\delta_0 > 0$ from Lemma~\ref{lem_thick_close_to_neck}.
If
\[ \delta_{\nn} \leq \ov\delta_{\nn}, \qquad \eps_{\can} \leq \ov\eps_{\can}, \qquad r_{\comp} \leq \ov{r}_{\comp}, \]
then by a priori assumption \ref{item_backward_time_slice_3}(a)  and Lemma~\ref{lem_time_slice_neck_implies_space_time_neck}, the component $\Sigma$ has to be a central $2$-sphere of a $\delta_0$-neck in $\M_t$ at scale $a r_{\comp}$ for some $a \in [1,2]$ and we must have
\[ 0.9 r_{\comp} < \rho_1 (y) < 2.1 r_{\comp}. \]
By bounded curvature at bounded distance, Lem\-ma~\ref{lem_bounded_curv_bounded_dist}, along with a priori assumption \ref{item_lambda_thick_2}, applied at $x$, and assuming
\[ \eps_{\can} \leq \ov\eps_{\can} (\lambda, A), \]
we find that  $\rho_1 (x) < C' \rho_1 (y) < 2.1 C' r_{\comp}$ for some $C' = C'(A) < \infty$.
So 
\[ d_t (x,y) < 10 A \rho_1 (x) < 21 C'  A r_{\comp}. \]

Let $Y_\# < \infty$  be a constant whose value we will fix at the end of the proof.
By a priori assumption \ref{item_backward_time_slice_3}(c) we can pick a $\Lambda r_{\comp}$-thick point $z \in \N_{t_{J+1}-}$  in the same component of $\N_{t_{J+1}-}$ as $\Sigma (t_{J+1})$.
By a priori assumption \ref{item_backward_time_slice_3}(a) and bounded curvature at bounded distance, Lemma~\ref{lem_bounded_curv_bounded_dist}, applied at all points on $\partial \N_{t_{J+1}-}$, and assuming
\[  \delta_{\nn} \leq \ov\delta_{\nn}, \qquad 
\Lambda \geq \underline\Lambda(Y_\#)\,,\qquad 
\eps_{\can} \leq \ov\eps_{\can} (Y_\#), \]
we obtain that  $d_{t_J} (z(t), \partial \N_{t_J+}) > Y_\# r_{\comp}$.

By a priori assumption \ref{item_lambda_thick_2} and a distance distortion estimate, it follows that then
\[ d_{t} (z(t), \partial \N_{t}^{J+1})  > e^{-C'' \lambda^{-2}} Y_\# r_{\comp} \]
for some universal constant $C'' < \infty$.
We can then apply Lemma~\ref{lem_thick_close_to_neck}, assuming that \[ \delta_{\nn} \leq \ov\delta_{\nn}, \qquad
 Y_\# \geq \underline{Y}_\# (\lambda, A), \qquad \eps_{\can} \leq \ov\eps_{\can} (A), \]
to show that $\rho_1 (x) \geq \frac1{10} a r_{\comp} \geq \frac1{10} r_{\comp}$.
So (\ref{eq_rho_less_10_r_comp}) holds.
\end{proof}

Next, we establish a priori assumption \ref{item_q_less_than_q_bar_6}.

\begin{lemma}[Verification of \ref{item_q_less_than_q_bar_6}] \label{lem_verification_APA_7}
If
\begin{gather*}
E \geq \underline{E}, \qquad
F \geq \underline{F} (E), \qquad
H \geq \underline{H} (E), \qquad
\eta_{\lin} \leq \ov\eta_{\lin} (E), \qquad \\
\nu \leq \ov\nu (E), \qquad
\delta_{\nn} \leq \ov\delta_{\nn} , \qquad 
\lambda \leq \ov\lambda, \qquad
A \geq \underline{A} (E, W), \qquad 
\Lambda \geq \underline\Lambda, \\
\delta_{\bb} \leq \ov\delta_{\bb} (E, \lambda, D_{\cut}, A, \Lambda), \qquad 
\eps_{\can} \leq \ov\eps_{\can} (E, \lambda, D_{\cut}, W, A, \Lambda), \qquad 
r_{\comp} \leq \ov{r}_{\comp},
\end{gather*}
then a priori assumption \ref{item_q_less_than_q_bar_6} holds.
In other words, for all $x \in \N_{[0,t^{**}]} \setminus \cup_{\DD \in \Cut \cup \Cut^J} \DD$ for which $P(x, A \rho_1 (x)) \cap \DD = \emptyset$ for all $\DD \in \Cut \cup \Cut^J$, we have
\begin{equation} \label{eq_Q_ov_Q_lemma_APA_7}
 Q (x) \leq \overline{Q} . 
\end{equation}
\end{lemma}

The strategy of the proof is the following:
Near the neck-like boundary of $\N$ the bound (\ref{eq_Q_ov_Q_lemma_APA_7}) is a direct consequence of assertion \ref{ass_12.47_c} of Lemma~\ref{lem_summary_APA_10_12}.
So it remains to consider points that are far away from this neck-like boundary.
If a relaxed bound of the form $Q \leq 10 \ov{Q}$ holds on a parabolic neighborhood of size comparable to $L(E)$ around such a point, either via a priori assumption \ref{item_q_less_than_q_bar_6} or (\ref{eq_relaxed_APA_6}), then we can use the semi-local maximum principle, Proposition~\ref{Prop_semi_local_max}, and a priori assumption \ref{item_q_less_than_nu_q_bar_12} to improve this bound by a factor of $\frac1{10}$.
On the other hand, points for which such a relaxed bound is absent in such a parabolic neighborhood must be close enough to a cut, and thus even farther from the neck-like boundary.  
In this case we can guarantee a bound of the form $Q\leq 10W\ov{Q}$ by either \ref{item_q_less_than_w_q_bar_7} or (\ref{eq_relaxed_APA_7}) on an even larger parabolic neighborhood, of size comparable to $A$.  
The bound (\ref{eq_Q_ov_Q_lemma_APA_7}) then follows from the interior decay estimate, Proposition~\ref{Prop_interior_decay} and a priori assumption \ref{item_q_less_than_nu_q_bar_12}, for large enough $A$.

\begin{proof}
Let $x \in \N_{[0, t^{**}]} \setminus \cup_{\DD \in \Cut \cup \Cut^J}$ and assume that $P(x, \lb A \rho_1 (x)) \lb \cap \lb \DD \lb = \lb \emptyset$ for all $\DD \in \Cut \cup \Cut^J$.
Set $t := \mathfrak{t} (x)$.
Our goal will be to show that $Q (x) \leq \ov{Q}$.
By a priori assumption \ref{item_q_less_than_q_bar_6} from our induction hypothesis, we only need to consider the case $t > t_J$ and $x \in \N^{J+1}_{(t_J, t^{**}]}$.

Let $L = L(E) < \infty$ be the constant from Proposition~\ref{Prop_semi_local_max}.
By Lemma~\ref{lem_containment_parabolic_nbhd} and a priori assumption \ref{item_lambda_thick_2}, and assuming
\[ \eps_{\can} \leq \ov\eps_{\can} (L(E), \lambda, A), \]
we can find a constant $A' = A' (L(E), A) < \infty$ with $A' \geq \max \{ A, L \}$ such that $P(x, A' \rho_1 (x))$ is unscathed and
\begin{equation} \label{eq_P_10_A_P_A_prime}
 P(y, 10A \rho_1 (y)) \subset P(x, A' \rho_1 (x)) 
 \qquad \text{for all} \qquad y \in P(x, L \rho_1 (x)). 
\end{equation}

We now distinguish two cases:

\textit{Case 1: \quad $B(x, L \rho_1 (x)) \not\subset \N^{J+1}_{t}$.}

The goal in this case will be to apply assertion \ref{ass_12.47_c} of Lemma~\ref{lem_summary_APA_10_12}.
To do this, we first need to bound $\rho_1(x)$ from above.
For this purpose, choose $z \in \partial \N^{J+1}_t$ such that $d_t (x,z) < L \rho_1 (x)$.
By a priori assumption \ref{item_backward_time_slice_3}(a) and Lemma~\ref{lem_time_slice_neck_implies_space_time_neck}, and assuming
\[ \delta_{\nn} \leq \ov\delta_{\nn} , \qquad
\eps_{\can} \leq \ov\eps_{\can} , \qquad
r_{\comp} \leq \ov{r}_{\comp}, \]
we know that $z$ is a center of a sufficiently precise neck $U \subset \M_t$ at scale $a r_{\comp}$ for some $a \in [1,2]$  such that $\rho_1 (z) < 2.1 r_{\comp}$.
By bounded curvature at bounded distance, Lemma~\ref{lem_bounded_curv_bounded_dist}, and assuming
\[ \eps_{\can} \leq \ov\eps_{\can} (L(E)), \]
we therefore obtain $\rho_1(x) < C r_{\comp}$ for some $C = C(L(E)) < \infty$.
Thus $$d_t (x, \partial \N_t^{J+1}) \leq d_t (x,z) < CL r_{\comp}\,.$$
We can now apply assertion \ref{ass_12.47_c} of Lemma~\ref{lem_summary_APA_10_12}, assuming
\[ F > C(L(E))\cdot L(E), \]
and  obtain that $Q (x) \leq \ov{Q}$.

\textit{Case 2: \quad $B(x, L \rho_1 (x)) \subset \N^{J+1}_{t}$.}

We distinguish two subcases.

\textit{Case 2a: \quad $P(x, A' \rho_1 (x)) \cap \DD = \emptyset$ for all $\DD \in \Cut \cup \Cut^J$.}

Recall that $P(x, L \rho_1 (x)) \subset P(x, A' \rho_1 (x))$.
So by Lemma~\ref{lem_parabolic_domain_in_N} we have
\[ P(x, L \rho_1 (x)) \subset  \N_{[0, t^{**}]}  \setminus \cup_{\DD \in \Cut \cup \Cut^J} \DD. \]

Using the assumption of Case 2a, (\ref{eq_P_10_A_P_A_prime}), a priori \ref{item_q_less_than_q_bar_6} from the induction hypothesis and (\ref{eq_relaxed_APA_6}), we obtain that $Q \leq 10 \ov{Q}$ on $P(x, L \rho_1 (x))$.
By Lemma~\ref{lem_verification_of_APA_6} we have $|h| \leq \eta_{\lin}$ on $P(x, L \rho_1 (x))$.
If $P(x, L \rho_1 (x))$ intersects the initial time-slice $\M_0$, then a priori assumption \ref{item_q_less_than_nu_q_bar_12} also implies that $Q \leq \nu \ov{Q}$ on $P(x, L \rho_1 (x)) \cap \M_0$. 
So by Proposition~\ref{Prop_semi_local_max}, a priori assumption \ref{item_lambda_thick_2}, and assuming
\begin{equation*}
 E > 2, \qquad 
 H \geq \underline{H} (E), \qquad 
 \eta_{\lin} \leq \ov\eta_{\lin} (E), \qquad 
 \nu \leq \ov\nu (E), \qquad  
 \eps_{\can} \leq \ov\eps_{\can} (E, \lambda), 
\end{equation*}
we obtain the improved estimate $Q(x) \leq \ov{Q}$.

\textit{Case 2b: \quad $P(x, A' \rho_1 (x)) \cap \DD \neq \emptyset$ for some $\DD \in \Cut \cup \Cut^J$.}

Applying Lemma~\ref{lem_boundary_far_from_cut} with $A_0=A'$ and a priori assumption \ref{item_cut_diameter_less_than_d_r_comp_11}, and assuming
\begin{gather*}
 \delta_{\nn} \leq \ov\delta_{\nn}, \qquad 
 \lambda\leq \ov\lambda,\qquad
 \Lambda \geq \underline\Lambda, \qquad
 \delta_{\bb} \leq \ov\delta_{\bb} (\lambda, D_{\cut},  A'(E,A), \Lambda), \\ \qquad \eps_{\can} \leq \ov\eps_{\can}  (\lambda, D_{\cut},  A'(E,A), \Lambda), \qquad 
 r_{\comp} \leq \ov{r}_{\comp}, 
\end{gather*}
we find that   $B(x, A \rho_1 (x)) \subset \N $.
Recall moreover that by assumption of the lemma we have $P(x, A \rho_1 (x)) \cap \DD = \emptyset$ for all $\DD \in \Cut \cup \Cut^J$.
Therefore, again by Lemma~\ref{lem_parabolic_domain_in_N}, we obtain that
\[ P(x, A \rho_1 (x)) \subset  \N_{[0, t^{**}]} \setminus \cup_{\DD \in \Cut \cup \Cut^J} \DD. \]
By  a priori assumption \ref{item_q_less_than_w_q_bar_7} from the induction hypothesis and (\ref{eq_relaxed_APA_7}), we have $Q  \leq 10 W \ov{Q}$ on $P(x, A \rho_1 (x))$.
We will now apply Proposition~\ref{Prop_interior_decay} to $P(x, \lb A \rho_1 (x))$ in order to improve this estimate at $x$.
To do this, observe that, by Lemma \ref{lem_verification_of_APA_6} we have $|h| \leq \eta_{\lin}$ on $P(x, A \rho_1 (x))$ and if $P(x, A \rho_1 (x))$ intersects the initial time-slice $\M_0$, then a priori assumption \ref{item_q_less_than_nu_q_bar_12} implies that $Q \leq \nu \ov{Q}$ on $P(x, A \rho_1 (x)) \cap \M_0$. 
We can therefore apply Proposition~\ref{Prop_interior_decay} to $P(x, A \rho_1 (x))$, along with a priori assumption \ref{item_lambda_thick_2}, assuming that
\begin{gather*}
 E > 2, \qquad 
 H \geq \underline{H} (E), \qquad 
 \eta_{\lin} \leq \ov\eta_{\lin} (E),  \qquad 
 \nu \leq \ov{\nu} (E), \qquad \\
 A \geq \underline{A} (E, W), \qquad 
 \eps_{\can} \leq \ov\eps_{\can} (E, \lambda, W), 
\end{gather*}
and conclude that $Q(x) \leq \ov{Q}$.
This finishes the proof.
\end{proof}

Next, we verify a priori assumption \ref{item_q_less_than_w_q_bar_7}.

\begin{lemma}[Verification of \ref{item_q_less_than_w_q_bar_7}] \label{lem_verification_of_APA_8}
If
\begin{gather*}
E \geq \underline{E}, \qquad 
H \geq \underline{H} (E), \qquad 
\eta_{\lin} \leq \ov\eta_{\lin} (E), \qquad 
\nu \leq \ov\nu (E), \\
\delta_{\nn} \leq \ov\delta_{\nn}, 
 \qquad \lambda \leq \ov\lambda, \qquad 
 W \geq \underline{W} (E, \lambda, D_{\cut}) \qquad 
 A \geq \underline{A}(E),   \\ 
 \Lambda \geq \underline\Lambda, \qquad 
 \delta_{\bb} \leq \ov\delta_{\bb} (\lambda, D_{\cut}, A, \Lambda ), \qquad 
 \eps_{\can} \leq \ov\eps_{\can} ( E, \lambda, D_{\cut}, A, \Lambda  ), \\
   \qquad r_{\comp} \leq \ov{r}_{\comp},
\end{gather*}
then a priori assumption \ref{item_q_less_than_w_q_bar_7} holds.
In other words, we have
\begin{equation} \label{eq_Q_W_Q_lemma}
 Q \leq W \overline{Q} \qquad \text{on} \qquad  \N_{[0, t^{**}]}  \setminus \cup_{\DD \in \Cut \cup \Cut^J} \DD. 
\end{equation}
\end{lemma}

Note that a main aspect of this lemma is that $W$ does not depend on $A$.
Otherwise the inequality (\ref{eq_Q_W_Q_lemma}) would follow easily from (\ref{eq_relaxed_APA_6}) and (\ref{eq_relaxed_APA_8}).
More specifically, at points whose distance to an extension cap is bounded in terms of $A$, we can only use (\ref{eq_relaxed_APA_8}) to obtain a bound on $Q$.
However, the ``conversion'' factor between $Q^*$ and $Q$ at such a point depends on $A$.
So the bound (\ref{eq_relaxed_APA_8}) cannot be used directly to verify (\ref{eq_Q_W_Q_lemma}).

The idea of the proof is the following.
We may focus on the time-slab $\N^{J+1}_{[t_J, t^{**}]}$ since the bound (\ref{eq_Q_W_Q_lemma}) follows from a priori assumption \ref{item_q_less_than_w_q_bar_7} of the induction hypothesis.
The bound (\ref{eq_Q_W_Q_lemma}) follows from (\ref{eq_relaxed_APA_6}) (the relaxed version of \ref{item_q_less_than_q_bar_6}) at points that are far away from the cuts, i.e. at distance comparable to $A$.
For points that are close to the cuts we distinguish two cases.
The strategy in the first case is to deduce (\ref{eq_Q_W_Q_lemma}) from  a priori assumption \ref{item_q_less_than_w_q_bar_7} and its relaxed version (\ref{eq_relaxed_APA_7}) via the semi-local maximum principle, Proposition~\ref{Prop_semi_local_max}.
This argument only works at points that are still sufficiently far away from the cuts, this time with separation comparable to $L(E)\ll A$.
In the second case we consider points that are close to cuts, comparable to $L(E)$.
At these points (\ref{eq_relaxed_APA_8}) (the relaxed version of \ref{item_q_star_less_than_q_star_bar}) guarantees a bound of the form $Q^* \leq 10 \ov{Q}^*$.
This bound translates into a bound on $Q$ and the conversion factor can be controlled in terms of $L(E), E, \lambda$ and $D_{\cut}$.
So (\ref{eq_Q_W_Q_lemma}) follows as long we choose $W$ larger than this conversion factor.

\begin{proof}
Consider a point $x \in \N^{J+1}_{[t_J, t^{**}]} \setminus \cup_{\DD \in \Cut^J} \DD$.
Note that the case when $t:=\t(x)=t_J$ follows from the induction hypothesis, so we assume in the following that $t>t_J$.

We distinguish the following cases.

\textit{Case 1: \quad $B(x, 10A \rho_1 (x)) \not\subset \N^{J+1}_{t}$.}

Then we can apply Lemma~\ref{lem_boundary_far_from_cut} along with a priori assumption \ref{item_cut_diameter_less_than_d_r_comp_11}, assuming that
\begin{gather*}
 \delta_{\nn} \leq \ov\delta_{\nn}, \qquad 
 \lambda \leq \ov\lambda, \qquad 
 \Lambda \geq \underline\Lambda , \qquad
\delta_{\bb} \leq \ov\delta_{\bb} (\lambda, D_{\cut}, A, \Lambda ), \qquad \\
\eps_{\can} \leq \ov\eps_{\can} (\lambda, D_{\cut}, A, \Lambda ), \qquad 
r_{\comp} \leq \ov{r}_{\comp}, 
\end{gather*}
and obtain that $P(x, 10 A \rho_1 (x)) \cap \DD = \emptyset$ for all $\DD \in \Cut \cup \Cut^J$.
So by (\ref{eq_relaxed_APA_6}) we have $Q (x) \leq 10 \ov{Q}$.
Therefore, $Q(x) \leq W \ov{Q}$, as long as
\[ W \geq 10. \]

\textit{Case 2: \quad $B(x, 10A \rho_1 (x)) \subset \N^{J+1}_{t}$.}

Choose $L = L(E)$ from Proposition~\ref{Prop_semi_local_max}.
We distinguish two subcases.

\textit{Case 2a: \quad $P(x, L \rho_1 (x)) \cap \DD = \emptyset$ for all $\DD \in \Cut \cup \Cut^J$.}

Assume that
\[ 10 A \geq L(E). \]
So $B(x, L \rho_1 (x)) \subset \N$ and thus by Lemma~\ref{lem_parabolic_domain_in_N}
\[ P(x, L \rho_1 (x)) \subset  \N_{[0, t^{**}]}^{J+1}  \setminus \cup_{\DD \in \Cut \cup \Cut^J} \DD. \]
Let us now apply Proposition~\ref{Prop_semi_local_max} to $P(x, L \rho_1 (x))$.
To do this, note that by Lemma \ref{lem_verification_of_APA_6}, a priori assumption \ref{item_q_less_than_w_q_bar_7} from the induction hypothesis and (\ref{eq_relaxed_APA_7}), we know that $|h| \leq \eta_{\lin}$ and $Q \leq 10 W \ov{Q}$ on $P(x, L \rho_1 (x))$.
If $P(x, L \rho_1 (x))$ intersects $\M_0$, then by a priori assumption \ref{item_q_less_than_nu_q_bar_12} we also have $Q \leq \nu \ov{Q}$ on the intersection.
Lastly, by a priori assumption \ref{item_lambda_thick_2} we have $\rho_1 (x) > \lambda r_{\comp}$.
So assuming
\begin{gather*}
 E > 2, \qquad 
 H \geq \underline{H} (E), \qquad 
 \nu \leq \ov\nu (E), \qquad 
 \eta_{\lin} \leq \ov\eta_{\lin} (E), \qquad 
 \eps_{\can} \leq \ov\eps_{\can} (E, \lambda), 
\end{gather*}
we obtain from Proposition~\ref{Prop_semi_local_max} that $Q(x) \leq W \ov{Q}$, as desired. 

\textit{Case 2b: \quad $P(x, L \rho_1 (x)) \cap \DD \neq \emptyset$ for some $\DD \in \Cut \cup \Cut^J$.}

By a priori assumptions \ref{item_geometry_cap_extension_5} and \ref{item_cut_diameter_less_than_d_r_comp_11}, and assuming
\[ \delta_{\bb} \leq \ov\delta_{\bb} ( \lambda,  D_{\cut}), \]
we can conclude that there is a constant $C' = C'(\lambda, D_{\cut}) < \infty$ such that $\rho_1 \leq C' r_{\comp}$ on $\DD$.
Next, by bounded curvature at bounded distance, Lemma~\ref{lem_bounded_curv_bounded_dist}, applied at $x$, a priori assumption \ref{item_lambda_thick_2}, and assuming that
\[ \eps_{\can} \leq \ov\eps_{\can} (L(E), \lambda), \]
we obtain a constant $C'' = C'' (L(E)) < \infty$ such that
\[ \rho_1 (x) \leq C'' C' r_{\comp}. \]
Since $B(x, 10A \rho_1 (x)) \subset \N^{J+1}_{t}$, we obtain from (\ref{eq_relaxed_APA_8}) that $Q^*(x) \leq 10 \ov{Q}^*$.
Assuming
\[ E \geq 3, \qquad \lambda \leq 1, \]
we can now convert this bound to a bound on $Q(x)$ as follows:
\begin{multline*}
 Q(x) = \rho_1^{E-3} (x) Q^*(x) 
 \leq \rho_1^{E-3} (x) 10 \ov{Q}^* \\
 \leq \big( C''(L(E))  C'(\lambda, D_{\cut})  r_{\comp} \big)^{E-3}  \eta_{\lin} (\lambda r_{\comp})^3 \\
 \leq 10^{E +1} \big( C'' (L(E)) C' (\lambda, D_{\cut}) \big)^{E-3} 10^{-E-1} \eta_{\lin} r_{\comp}^E  \\
 \leq 10^{E +1} \big( C'' (L(E)) C' (\lambda, D_{\cut}) \big)^{E-3} \ov{Q}. 
\end{multline*}
So $Q(x) \leq W \ov{Q}$, as long as
\[ W \geq \underline{W} (E, \lambda, D_{\cut}). \]
This finishes the proof.
\end{proof}

Lastly, we establish a priori assumption \ref{item_q_star_less_than_q_star_bar}.

\begin{lemma}[Verification of \ref{item_q_star_less_than_q_star_bar}] \label{lem_verification_APA_9}
If
\begin{gather*}
E \geq \underline{E}, \qquad
H \geq \underline{H}, \qquad
\eta_{\lin} \leq \ov\eta_{\lin}, \qquad
\nu \leq \ov\nu , \qquad
\delta_{\nn} \leq \ov\delta_{\nn}, \qquad 
\lambda \leq \ov\lambda, \qquad \\
\eta_{\cut} \leq \ov\eta_{\cut} , \qquad 
D_{\cut} \geq \underline{D}_{\cut} (\lambda), \qquad
A \geq \underline{A} (E, \lambda), \qquad 
\Lambda \geq \underline\Lambda, \\
\delta_{\bb} \leq \ov\delta_{\bb} (\lambda, D_{\cut}, A, \Lambda ), \qquad
\eps_{\can} \leq \ov\eps_{\can} (E,\lambda, D_{\cut}, A, \Lambda ), \qquad \\
r_{\comp} \leq \ov{r}_{\comp} (\lambda), \qquad
\end{gather*}
then a priori assumption \ref{item_q_star_less_than_q_star_bar} holds.
In other words, we have  
\begin{equation} \label{eq_Q_star_Q_lemma}
 Q^*(x) \leq \overline{Q}^* = 10^{-1} \eta_{\lin} (\lambda r_{\comp})^{3} 
\end{equation}
for all $x \in \N_{[0, t^{**}]} \setminus \cup_{\DD \in \Cut \cup \Cut^J} \DD$ for which $B(x, A \rho_1 (x)) \subset \N_{\mathfrak{t}(x)-}$.
\end{lemma}

Let us first summarize the strategy of the proof.
As in the previous proofs, the semi-local maximum principle, Proposition~\ref{Prop_semi_local_max}, can be used to deduce (\ref{eq_Q_star_Q_lemma}) from   a priori assumption \ref{item_q_star_less_than_q_star_bar} or its relaxed version (\ref{eq_relaxed_APA_8})  at points that are sufficiently far away from the cuts and the neck-like boundary of $\N$.
Now, consider points that are close to the neck-like boundary, but far enough (comparably to $A$) from this boundary such that the assertion does not become vacuous.
At such points we use the bound $Q \leq 10 \ov{Q}$ from  a priori assumption \ref{item_q_less_than_q_bar_6} and its relaxed version (\ref{eq_relaxed_APA_6}) and the interior decay estimate, Proposition~\ref{Prop_interior_decay}, to overcome the conversion factor between $Q$ and $Q^*$ for sufficiently large $A$.
Lastly, consider points that are close to a cut.
At such points we invoke the semi-local maximum principle, Proposition~\ref{Prop_semi_local_max} with initial condition, on a truncated parabolic neighborhood whose initial time-slice intersects the cut.
We then use a priori assumption \ref{item_h_derivative_bounds_9} or \ref{item_q_less_than_q_bar_6} to deduce a very good bound for $Q^*$ on this initial time-slice.
Proposition~\ref{Prop_semi_local_max} then implies (\ref{eq_Q_star_Q_lemma}).

\begin{proof}
Consider a point $x \in \N_{[0, t^{**}]} \setminus \cup_{\DD \in \Cut \cup \Cut^J} \DD$ such that $B(x,  \lb A \rho_1 (x)) \lb \subset \lb \N_{\mathfrak{t}(x)-}$ and set $t := \mathfrak{t}(x)$.
The case $t \leq t_J$ follows from a priori assumption \ref{item_q_star_less_than_q_star_bar} of the induction hypothesis.
So in the following we assume that $t > t_J$ and therefore that $B(x, A \rho_1 (x)) \subset \N^{J+1}_t$.  

Let $L=L(3)$ be the constant from Proposition~\ref{Prop_semi_local_max} (for $E=3$).
Using Lemma \ref{lem_containment_parabolic_nbhd}, a priori assumption \ref{item_lambda_thick_2}, and assuming that
\[ \eps_{\can} \leq \ov\eps_{\can} (L, \lambda, A), \]
we can find a constant $A' = A' (A) < \infty$ with $A' \geq A$ such that the parabolic neighborhood $P(x, A' \rho_1 (x))$ is unscathed and such that 
\begin{equation} \label{eq_PA_in_PA_prime}
 P(y, 10A \rho_1 (y)) \subset P(x, A' \rho_1 (x)) 
 \qquad \text{for all} \qquad y \in P(x, A \rho_1 (x)).
\end{equation}
Let us now distinguish three cases.

\textit{Case 1: \quad We have}
\[ P(x, A' \rho_1 (x),- (L \rho_1 (x))^2 ) \subset  \N_{[0, t^{**}]}  \setminus \cup_{\DD \in \Cut \cup \Cut^{J}} \DD \]
\textit{and $P(x, A' \rho_1 (x),- (L \rho_1 (x))^2 )$ does not intersect the initial time-slice $\M_0$.}

So, assuming
\[ A \geq L, \]
we have $P(x, L \rho_1 (x)) \subset  \N_{[0, t^{**}]} \setminus \cup_{\DD \in \Cut \cup \Cut^{J}} \DD$.
Using (\ref{eq_PA_in_PA_prime}), a priori assumption \ref{item_q_star_less_than_q_star_bar}, (\ref{eq_relaxed_APA_8}) and Lemma~\ref{lem_verification_of_APA_6}, we find that $Q^* \leq 10 \ov{Q}^*$ and $|h| \leq \ov\eta_{\lin}$ on $P(x, L \rho_1 (x))$.
Since  the exponent in the definition of $Q^*$ is  $3 > 2$, if
\[  H \geq \underline{H}, \qquad \eta_{\lin} \leq \ov\eta_{\lin} , \qquad \eps_{\can} \leq \ov\eps_{\can} (\lambda) , \]
we may apply the semi-local maximum principle, Proposition~\ref{Prop_semi_local_max}, to deduce that $Q^* (x) \leq \ov{Q}^*$, which finishes the proof in this case.
Note that here we have used a priori assumption \ref{item_lambda_thick_2}.

\textit{Case 2: \quad  We have}
\[ B(x,  A' \rho_1 (x) ) \not\subset  \N. \]

By Lemma~\ref{lem_boundary_far_from_cut} and a priori assumption \ref{item_cut_diameter_less_than_d_r_comp_11}, and assuming
\begin{gather*}
 \delta_{\nn} \leq \ov\delta_{\nn}, \qquad 
 \lambda\leq \ov\lambda,\qquad
 \Lambda \geq \underline\Lambda, \qquad
 \delta_{\bb} \leq \delta_{\bb} (\lambda, D_{\cut}, A' (A), \Lambda ),  \\ \qquad \eps_{\can} \leq \ov\eps_{\can} (\lambda, D_{\cut}, A' (A), \Lambda) , \qquad r_{\comp} \leq \ov{r}_{\comp}, 
\end{gather*}
we find that
\begin{equation} \label{eq_P_x_A_prime_DD}
 P(x, A' \rho_1 (x)) \cap \DD = \emptyset, \qquad \text{for all} \qquad \DD \in \Cut \cup \Cut^J. 
\end{equation}
So by Lemma~\ref{lem_parabolic_domain_in_N}
\begin{equation} \label{eq_P_in_NN}
 P(x, A \rho_1 (x)) \subset  \N_{[0, t^{**}]}  \setminus \cup_{\DD \in \Cut \cup \Cut^J} \DD. 
\end{equation}

Combining (\ref{eq_P_x_A_prime_DD}) with (\ref{eq_PA_in_PA_prime}), we obtain that for all $y \in P(x, \lb A \rho_1 (x))$
\[ P(y, 10A \rho_1 (y)) \cap \DD = \emptyset \qquad \text{for all} \qquad \DD \in \Cut \cup \Cut^J. \]
Therefore, by a priori assumption \ref{item_q_less_than_q_bar_6} and (\ref{eq_relaxed_APA_6}), we obtain that $Q \leq 10 \ov{Q}$ on $P(x, A \rho_1 (x))$.

Let us now convert this bound into a bound on $Q^*$.
There are two ways of doing this.
One way would be to use Lemma~\ref{lem_thick_close_to_neck}, as in the proof of  Lemma \ref{lem_verification_of_APA_6} leading up to (\ref{eq_rho_less_10_r_comp}) to show that $\rho_1(x) \geq \frac1{10} r_{\comp}$.
In the following, however, we will use a different strategy, as it is technically easier.

Assuming
\[ E \geq 3 \]
and using a priori assumption \ref{item_lambda_thick_2} and (\ref{eq_P_in_NN}), we have on $P(x, \lb  A \rho_1 (x))$  
\begin{multline*}
 Q^* 
= \rho_1^{3-E} Q 
\leq (\lambda r_{\comp})^{3-E} \cdot 10 \ov{Q}
= (\lambda r_{\comp})^{3-E} \cdot  10^{-E} \eta_{\lin} r_{\comp}^E \\
\leq \lambda^{-E} \eta_{\lin} (\lambda r_{\comp})^3 
= 10 \lambda^{-E} \ov{Q}^*.
\end{multline*}
We will now apply Proposition~\ref{Prop_interior_decay} to $Q^*$ on $P(x, A \rho_1 (x))$.
To do this, observe that by Lemma~\ref{lem_verification_of_APA_6} we have $|h| \leq \eta_{\lin}$ on $P(x, A \rho_1 (x))$.
In addition, if $P(x, A \rho_1(x))$ intersects the initial time-slice $\M_0$, then by a priori assumption \ref{item_q_less_than_nu_q_bar_12} we have $Q^* \leq \nu \ov{Q}^*$ on the intersection.
We also have $\rho_1 (x) > \lambda r_{\comp}$ by a priori assumption \ref{item_lambda_thick_2}.
So if  
\begin{gather*}
 H \geq \underline{H}, \qquad
\eta_{\lin} \leq \ov\eta_{\lin}, \qquad 
\nu \leq \ov\nu, \qquad
A \geq \underline{A}(E, \lambda), \qquad
\eps_{\can} \leq \ov\eps_{\can} (E, \lambda), 
\end{gather*}
then we obtain that $Q^* (x) \leq \ov{Q}^*$, as desired.

\textit{Case 3: \quad  We have}
\begin{equation*}
\label{eqn_case_3_def}
B(x,  A' \rho_1 (x) ) \subset  \N, 
\end{equation*}
\textit{and either}
\[ P(x, A' \rho_1 (x),- (L \rho_1 (x))^2 ) \not\subset  \N_{[0, t^{**}]}  \setminus \cup_{\DD \in \Cut \cup \Cut^{J}} \DD, \]
\textit{or $P(x, A' \rho_1 (x),- (L \rho_1 (x))^2 )$ intersects the initial time-slice $\M_0$.}

In the following we will use the notation
\[ \Cut^j = \{ \DD \in \Cut \cup \Cut^J \;\; : \;\; \DD \subset \M_{t_j} \}. \]
Choose $j_0 \in \{ 1, \ldots, J \}$ maximal with the property that
\[ P(x, A' \rho_1 (x), - (t - t_{j_0}) ) \not\subset   \N_{[0, t^{**}]}  \setminus \cup_{\DD \in \Cut \cup \Cut^{J}} \DD. \]
If no such $j_0$ exists, then set $j_0 := 0$.
By Lemma~\ref{lem_parabolic_domain_in_N} we have $P(x, \lb A' \rho_1 (x), \lb - (t - t')) \lb \subset \lb \N_{[0, t^{**}]} \lb \setminus \lb \cup_{\DD \in \Cut \cup \Cut^J} \DD$ for all $t' \in (t_{j_0}, t]$.
Letting $t' \to t_{j_0}$ and using the fact that $\N_{[0, t^{**}]}$ is a closed subset of $\M$, we obtain
\begin{equation} \label{eq_P_A_prime_in_NN}
 P(x, A' \rho_1 (x), - (t - t_{j_0}) ) \subset   \N_{[0, t^{**}]}  \setminus \cup_{\DD \in \Cut^{j_0+1} \cup \ldots \cup \Cut^{J}} \DD 
\end{equation}
and either $j_0 = 0$ or there is a cut $\DD_0 \in \Cut^{j_0}$ such that
\begin{equation} \label{eq_P_in_NN_up_to_j0}
 P(x, A' \rho_1 (x), - (t - t_{j_0}) ) \cap \DD_0 \neq \emptyset.
\end{equation}

Let
\[ B_{t_{j_0}} := \big( B(x, L \rho_1 (x)) \big) (t_{j_0}). \]
In other words, $B_{t_{j_0}}$ is the initial time-slice of the parabolic neighborhood $P(x, \lb L \rho_1 (x), \lb -(t-t_{j_0}))$.
Note that by (\ref{eq_P_in_NN_up_to_j0}) the perturbation $h$ is defined everywhere on $P(x, \lb L \rho_1 (x),\lb -(t-t_{j_0})) \setminus B_{t_{j_0}}$ and it can be smoothly extended to the entire parabolic neighborhood by setting $h = h_{t_{j_0}+}$ on $B_{t_{j_0}}$.
Similarly, we can extend $Q^*$ to the entire parabolic neighborhood $P(x, L \rho_1 (x), -(t-t_{j_0}))$ by setting $Q^* = Q^*_+$ on $B_{t_{j_0}}$.

We will now bound $Q^* = Q^*_+$ on $B_{t_{j_0}}$.
Let $y \in B_{t_{j_0}}$.
Then the two cases indicated above lead to the following three cases:

\textit{Case 3a: \quad We have $j_0 = 0$ and therefore $y \in \M_0$.}

In this case, by a priori assumption \ref{item_q_less_than_nu_q_bar_12} we have
\[ Q^* (y) \leq \nu \ov{Q}^*. \]

\textit{Case 3b: \quad We have $j_0\geq 1$ and $y \in \DD_0$. }

In this case, a priori assumption \ref{item_h_derivative_bounds_9} yields
\[ Q^* (y) \leq \eta_{\cut} \ov{Q}^*. \]

\textit{Case 3c: \quad   We have $j_0\geq 1$ and $y \not\in \DD_0$.}

Our strategy in this case is to use the bound on $Q(y)$ from a priori assumption \ref{item_q_less_than_q_bar_6}   and translate it into a bound on $Q^*(y)$.
In order to do this, we need to ensure that a priori assumption \ref{item_q_less_than_q_bar_6}  apply at $y$ (or slightly earlier) and that $\rho_1 (y)$ is sufficiently large so that the $Q$-bound implies a good bound on $Q^*$.

Let $C_\# < \infty$ be a constant whose value we will determined at the end of the proof.
We can now apply Lemma~\ref{lem_large_scale_near_neck} and a priori assumption \ref{item_cut_diameter_less_than_d_r_comp_11}, assuming
\begin{gather*}
\delta_{\nn} \leq \ov\delta_{\nn}, \qquad
D_{\cut} \geq \underline{D}_{\cut} (\lambda, C_\#), \qquad
\Lambda \geq \underline{\Lambda}, \qquad 
\delta_{\bb} \leq \ov\delta_{\bb} (\lambda, C_\#, D_{\cut}, A'(A), \Lambda),\\ \qquad
\eps_{\can} \leq \ov\eps_{\can} (\lambda, D_{\cut}, A'(A), \Lambda), \qquad 
r_{\comp} \leq \ov{r}_{\comp} (C_\#), 
\end{gather*}
to find that
\[ \rho_1 (y) \geq C_\# r_{\comp} \]
and $P(y, 2  A  \rho_1 (y)) \cap \DD = \emptyset$ for all $\DD \in \Cut^1 \cup \ldots \cup \Cut^{j_0-1}$.
So for any $t' \in [t_{j_0-1}, t_{j_0})$, sufficiently close to $t_{j_0}$, we have $P(y(t'),  A  \rho_1 (y(t'))) \cap \DD = \emptyset$ for all $\DD \in \Cut \cup \Cut^J$.
So by a priori assumption \ref{item_q_less_than_q_bar_6}  we have $Q(y(t')) \leq  \ov{Q}$.
Letting $t' \to t_{j_0}$ yields $Q(y) \leq  \ov{Q}$.

Assuming
\[ E \geq 4,  \]
we obtain  
\begin{multline*}
 Q^* (y) 
 = \rho_1^{3-E} (y) Q(y) 
 \leq (C_\# r_{\comp})^{3-E} \cdot  \ov{Q} \\
 = (C_\# r_{\comp})^{3-E} 10^{-E -1} \eta_{\lin} r_{\comp}^E 
 \leq  C_\#^{-1} \lambda^{-3} 10^{-E -1} \eta_{\lin} (\lambda r_{\comp})^3 \\
 \leq   C_\#^{-1} \lambda^{-3}  \ov{Q}^*. 
\end{multline*}

Summarizing the results of subcases 3a--3c, we obtain that
\[ Q^* \leq \big( \nu + \eta_{\cut}  + C_\#^{-1} \lambda^{-3} \big) \ov{Q}^* \qquad \text{on} \qquad B_{t_{j_0}}. \]
Similarly as in case 1, we can use (\ref{eq_PA_in_PA_prime}) and (\ref{eq_P_A_prime_in_NN})  together with  a priori assumption \ref{item_q_star_less_than_q_star_bar} and (\ref{eq_relaxed_APA_8}) to show that $Q^* \leq 10 \ov{Q}^*$ on $P(x, L \rho_1 (x), - (t-t_{j_0}))$.
By Lemma~\ref{lem_verification_of_APA_6} we have $|h| \leq \eta_{\lin}$ on $P(x, L \rho_1 (x), - (t-t_{j_0}))$.
We can now apply Proposition~\ref{Prop_semi_local_max} along with a priori assumption \ref{item_lambda_thick_2}, assuming
\begin{gather*}
  H \geq \underline{H}, \qquad
 \eta_{\lin} \leq \ov\eta_{\lin}, 
  \qquad \nu \leq \ov\nu , 
 \qquad C_\# \geq \underline{C}_\# (\lambda), \\
 \qquad \eta_{\cut} \leq \ov\eta_{\cut} , 
 \qquad \eps_{\can} \leq \ov\eps_{\can} (\lambda),
\end{gather*}
to show that $Q^* (x) \leq \ov{Q}^*$, as desired.
\end{proof}

\bigskip
We can finally finish the proof of Proposition~\ref{Prop_extend_comparison_by_one}.
Lemmas~\ref{lem_verification_of_APA_6}, \ref{lem_verification_APA_7}, \ref{lem_verification_of_APA_8} and \ref{lem_verification_APA_9} imply that $(\N_{[0,t^{**}]}, \lb \{ \N^j, \lb \N^{J+1}_{[t_J, t^{**}]} \}_{j=1}^J, \lb \{ t_j, t^{**} \}_{j=1}^J)$ and $(\Cut \cup \Cut^{J}, \lb \ov\phi |_{\N_{[0,t^{**}]}}, \lb \{ \phi^j, \phi^{J+1} |_{\N^{J+1}_{[t_J,t^{**}]}} \}_{j=1}^{J})$ satisfy a priori assumptions \ref{item_eta_less_than_eta_lin_13}--\ref{item_q_star_less_than_q_star_bar} whenever (\ref{eq_relaxed_APA_6})--(\ref{eq_relaxed_APA_8}) hold up to time $t^{**}$.
So by iterating Lemma~\ref{lem_setup_of_t_star_star}, we may choose  $t^{**} = t^*$.
Since a priori assumption \ref{item_eta_less_than_eta_lin_13} holds for the aforementioned comparison domain and comparison, we have $|h| \leq \eta_{\lin}$ on $\N^{J+1}_{t^*}$.
So by assertion \ref{ass_12.47_d} of Lemma~\ref{lem_summary_APA_10_12}, we obtain that $t^{**} = t^* = t_{J+1}$.
So $(\N, \lb \{ \N^j \}_{j=1}^{J+1}, \lb \{ t_j \}_{j=1}^{J+1})$ and $(\Cut \lb \cup \lb \Cut^{J}, \lb \ov\phi, \lb \{ \phi^j \}_{j=1}^{J+1})$ satisfy a priori assumptions \ref{item_time_step_r_comp_1} and \ref{item_eta_less_than_eta_lin_13}--\ref{item_q_star_less_than_q_star_bar}.
A priori assumptions   \ref{item_h_derivative_bounds_9}--\ref{item_apa_13} follow from assertion \ref{ass_12.47_a} of Lemma~\ref{lem_summary_APA_10_12}.
Recall that \ref{item_lambda_thick_2}--\ref{item_geometry_cap_extension_5} hold by the assumptions of Proposition~\ref{Prop_extend_comparison_by_one}.

Lastly note that the auxiliary parameter $F$ was assumed to be large depending only on $E$.
So it is straight forward to check that the assumptions of the parameters imposed in the course of this proof all follow from (\ref{eq_parameters_extend_comparison_by_one}).
This finishes the proof of Proposition~\ref{Prop_extend_comparison_by_one}.

\section{Proofs of the main results}
\label{sec_proofs_of_main_results}
In this section we will combine the main results of Sections \ref{sec_inductive_step_extension_comparison_domain} and \ref{sec_extend_comparison} to prove the main result of the paper, Theorem~\ref{Thm_existence_comparison_domain_comparison}.  
We then prove some corollaries, including several stability results and a uniqueness theorem, as presented in Subsection~\ref{subsec_statement_main_results}.

\begin{theorem}[Existence of comparison domain and comparison] \label{Thm_existence_comparison_domain_comparison}
If
\begin{gather*}
T > 0, \qquad 
E\geq \underline{E},\qquad
H \geq \underline{H}(E) , \qquad
\eta_{\lin} \leq \ov{\eta}_{\lin}  (E), \qquad\\
\nu \leq \ov\nu (T, H, \eta_{\lin}), \qquad 
\delta_{\nn} \leq \ov\delta_{\nn} (T,H, \eta_{\lin}), \qquad
\lambda \leq \ov\lambda (\delta_{\nn}), \qquad\\
D_{\CAP} \geq \underline{D}_{\CAP} (\lambda), \qquad
\eta_{\cut} \leq \ov\eta_{\cut}, \qquad 
D_{\cut} \geq \underline{D}_{\cut} ( D_{\CAP}, \eta_{\cut} ), \\
W \geq \underline{W} (D_{\cut}), \qquad
A \geq \underline{A} (W), \qquad 
\Lambda \geq \underline{\Lambda}(A), \qquad 
\delta_{\bb} \leq \ov\delta_{\bb} (\Lambda), \\  
\eps_{\can} \leq \ov\eps_{\can} (\delta_{\bb}), \qquad
r_{\comp} \leq \ov{r}_{\comp} (\Lambda),
\end{gather*}
then the following holds.

Consider two $(\eps_{\can} r_{\comp}, T)$-complete Ricci flow spacetimes $\M, \M'$ that each satisfy the $\eps_{\can}$-canonical neighborhood assumption at scales $(\eps_{\can} r_{\comp}, \lb 1)$.

Let $\zeta : \{ x \in \M_0 \;\; : \;\: \rho(x) > \lambda r_{\comp} \} \to \M'_0$ be a diffeomorphism onto its image that satisfies the following bounds:
\begin{align*}
| \zeta^* g'_0 - g_0 | &\leq \eta_{\lin}, \\
 e^{HT} \rho_1^E | \zeta^* g'_0 - g_0 | &\leq \nu \ov{Q} = \nu \cdot 10^{-E-1} \eta_{\lin} r_{\comp}^E, \\
 e^{HT} \rho_1^3 | \zeta^* g'_0 - g_0 | &\leq \nu \ov{Q}^* = \nu \cdot 10^{-1}\eta_{\lin} (\lambda r_{\comp})^3
 \end{align*}
Assume moreover that the $\eps_{\can}$-canonical neighborhood assumption holds at scales $(0,1)$ on the image of $\zeta$.

Then for any $J \geq 1$ with $J r_{\comp}^2 \leq T$ there is a comparison domain $(\N, \lb \{ \N^j \}_{j=1}^J, \lb \{ t_j \}_{j=0}^J )$ and a comparison $(\Cut, \phi, \{ \phi^j \}_{j=1}^J )$ from $\M$ to $\M'$ defined on this domain such that a priori assumptions \ref{item_time_step_r_comp_1}--\ref{item_eta_less_than_eta_lin_13} hold for the tuple of parameters $(\eta_{\lin}, \lb \delta_{\nn}, \lb \lambda, \lb D_{\CAP}, \lb \Lambda, \lb  \delta_{\bb}, \lb \eps_{\can}, \lb r_{\comp})$ and a priori assumptions \ref{item_q_less_than_q_bar_6}--\ref{item_apa_13} hold for the tuple of parameters $( T, \lb E, \lb H, \lb \eta_{\lin}, \lb \nu, \lb \lambda, \lb \eta_{\cut},  \lb  D_{\cut}, \lb W, \lb A, \lb r_{\comp})$.
Moreover, $\phi_{0+} = \phi^1_0 = \zeta |_{\N_{0}}$.
\end{theorem}

\begin{proof}[Proof of Theorem~\ref{Thm_existence_comparison_domain_comparison}.]
The theorem follows from Prop\-o\-si\-tions~\ref{prop_comp_domain_construction} and \ref{Prop_extend_comparison_by_one} by induction on $J$.
Both propositions can be applied under restrictions on the parameters that follow from the restrictions stated in the beginning of this theorem.  
Note that in the first step of the induction one applies Prop\-o\-si\-tion~\ref{prop_comp_domain_construction} to produce the first time slab $\N^1$ of the comparison domain.  By a priori assumption \ref{item_lambda_thick_2} we have $\N^1_0\subset X:=\{x\in \M_0\;\; : \;\;  \rho(x)>\lambda r_{\comp}\}$.  Assuming 
\[\lambda\leq\ov\lambda(\delta_{\nn})\,,\qquad \eps_{\can}\leq\ov\eps_{\can}(\delta_{\nn})\,,\]
by \ref{item_backward_time_slice_3} and Lemma~\ref{lem_bounded_curv_bounded_dist} it follows that the $\delta_{\nn}^{-1} r_{\comp}$-tubular neighborhood around $\N_0$ is contained in $X$.  
Hence the map $\zeta$ from the assumptions of Theorem~\ref{Thm_existence_comparison_domain_comparison} satisfies assumption \ref{con_12.1_vi} of Proposition~\ref{Prop_extend_comparison_by_one}.

Note that we have simplified the  restrictions on the parameters in the first part of this theorem by omitting arguments in parameter restrictions if they have already appeared in earlier restrictions.
This simplification does not change the nature of these restrictions.
For example, since we have imposed the restriction $H \geq \underline{H} (E)$, we can assume without loss of generality that $H \geq E$.
Therefore, it is not necessary to list $E$ in the restriction for $\nu \leq \ov\nu (H, \eta_{\lin}, T)$, as $\nu$ already depends on $H$.
\end{proof}

Next, we prove Theorem~\ref{Thm_comparing_RF_general_version}.
This theorem is similar to Theorem~\ref{Thm_existence_comparison_domain_comparison}; however the parameters associated with the a priori assumptions have been suppressed.
The proof of Theorem~\ref{Thm_comparing_RF_general_version} requires the following result.

\begin{lemma} \label{lem_MM_connected}
If
\[ \eps_{\can} \leq \ov\eps_{\can}, \qquad r \leq \ov{r}, \]
then the following holds.

Let $\M$ be an $(\eps_{\can} r, T)$-complete Ricci flow spacetime that satisfies the $\eps_{\can}$-canonical neighborhood assumption at scales $(\eps_{\can} r, \lb 1)$.
Let $x \in \M_{[0,T]}$ be a point with $\rho (x) > r$.
Then there is a continuous path $\gamma : [0,1] \to \M_{[0,T]}$ between $x$ and a point in $\M_0$ such that $\t \circ \gamma$ is non-increasing and such that $\rho (\gamma(s)) > .9 r$ for all $s \in [0,1]$.
\end{lemma}

\begin{proof}
A slightly different version of this statement, which would also be adequate for our needs here, was proven in \cite[Prop. 3.5]{Kleiner:2014wz}. 
For completeness, we provide an alternative argument.

Set $t_0 := \mathfrak{t}(x)$ and $r_0 := \rho_1 (x) > r$.
By Lemma~\ref{lem_lambda_thick_survives_backward}, assuming
\[ \eps_{\can} \leq \ov\eps_{\can}, \qquad r \leq \ov{r}, \]
we know that $x$ survives until time $\max \{ t_0 -r_0^2, 0 \}$ and $\rho_1( x(t) ) \geq  .95 r_0 > .95 r$ for all $t \in [\max \{ t_0 - r_0^2, 0 \}, t_0]$.
So if $t_0 \leq r_0^2$, then we are done.
Consider now the case $t_0 > r_0^2$.
If $r_0 \leq \frac12$ and $\rho (x(t_0-r_0^2)) \leq \rho (x)$, then we can use Lemma~\ref{lem_bryant_propagate}, assuming
\[ \eps_{\can} \leq \ov\eps_{\can}, \]
to show that $(\M_{t_0}, x (t_0 - r_0^2))$ is close enough to $(M_{\Bry}, g_{\Bry}, x_{\Bry})$ such that there is a point $y \in \M_{t_0 - r_0^2}$ with $\rho (y) > \rho(x)$ and such that $x (t_0 - r_0^2)$ can be connected with $y$ by a continuous path inside $\M_{t_0 - r_0^2}$ whose image only consists of $.9 r_0$-thick points.

So, summarizing our conclusions, each  $x \in \M_{[0,T]}$ can be connected with a point $y \in \M_{[0,T]}$ by  a path $\gamma : [0,1] \to \M_{[0,T]}$ such that $\t \circ \gamma$ is non-increasing and $\rho (\gamma(s)) > .9 r_0$ for all $s \in [0,1]$, and one of the following holds:
\begin{enumerate}
\item $y \in \M_0$,
\item $\rho(y) \geq .95\rho (x) > .95 \cdot \frac12$ and $\mathfrak{t} (y) = \mathfrak{t}(x) - \rho^2_1(x)$,
\item $\rho (y) > \rho(x)$ and $\mathfrak{t} (y) = \mathfrak{t}(x) - \rho^2_1(x)$,
\end{enumerate}
Iterating this process yields a sequence of points $x_0 = x, x_1, x_2, \ldots \in \M_{[0,T]}$ such that $x_i$ and $x_{i+1}$ can be connected by a path with the desired properties.
It now remains to show that this sequence terminates at some index $i$ and that $x_i \in \M_0$.
To see this, note that by (1)--(3) the sequence of times $\mathfrak{t}(x_i)$ is non-increasing and $\rho (x_i) > r$, assuming
\[ r \leq .95 \cdot \tfrac14. \]
Since $\t (x_{i+1}) = \mathfrak{t} (x_i) - \rho_1^2 (x_i) \leq \mathfrak{t} (x_i) - r^2$ in cases (2) and (3), the sequence must terminate after a finite number of steps.
\end{proof}

\bigskip

\begin{proof}[Proof of Theorem~\ref{Thm_comparing_RF_general_version}.]
Since we will invoke Theorem~\ref{Thm_existence_comparison_domain_comparison} below, in order to make the estimates in Theorem~\ref{Thm_comparing_RF_general_version} conform more closely with those in Theorem~\ref{Thm_existence_comparison_domain_comparison}, it will be convenient to prove the theorem for $E$ replaced by $E/2$, $\phi$ replaced by $\zeta$ and $\wh\phi$ replaced by $\phi$.
So we assume that $\zeta : U \to U'$ satisfies
\[ | \zeta^* g'_0 - g_0 | \leq \eps \cdot  r^{E} (|{\Rm}| + 1)^{E/2} \]
and our goal is to construct $\phi : \wh{U} \to \wh{U}'$ such that
\[ | \phi^* g'_0 - g_0 | \leq \delta \cdot  r^{E} (|{\Rm}| + 1)^{E/2}. \]

We will first prove a slightly weaker version of the theorem in which we allow $\eps_{\can}$ to also depend on $T$.
We will mention how we can remove this dependence at the end of this proof.

Fix $T$ and $E \geq \underline{E}$, where $\underline{E}$ is the constant from Theorem~\ref{Thm_existence_comparison_domain_comparison} and assume that $\underline{E} \geq 3$.
Based on these choices, fix constants $H$, $\eta_{\lin}$, $\nu, \delta_{\nn}$, $\lambda$, $D_{\CAP}$, $\eta_{\cut}$, $D_{\cut}$, $W$, $A$, $\Lambda$, $\delta_{\bb}$, $\eps_{\can}$ and $\ov{r}_{\comp}$ that satisfy the restrictions stated in Theorem~\ref{Thm_existence_comparison_domain_comparison}.
Without loss of generality, we may assume that $\ov{r}_{\comp} \leq 1$.
Choose $r_{\comp} := \alpha r \cdot \ov{r}_{\comp}$, where $0 < \alpha = \alpha (\delta, T,E) \leq 1$ is a constant whose value will be determined in the course of this proof.

We now verify the assumptions  of Theorem~\ref{Thm_existence_comparison_domain_comparison}.
In what follows, we will be imposing several upper bounds on the parameters $\alpha$ and $\eps$.
The upper bounds on $\alpha$ will only depend on $\delta, T, E$ and the upper bounds on $\eps$ will only depend on $\delta, T, E$ and $\alpha$.
As $\alpha$ will not be chosen depending on $\eps$, there will be no circular dependence.
At a number of steps in the following proof, we will also assume that the constants $\eta_{\lin}, \delta_{\nn}$ and $\eps_{\can}$ have been chosen smaller than some universal constant.

Assuming $$\eps \leq \eps_{\can} \cdot \alpha \ov{r}_{\comp}\,,$$ we get that $\M, \M'$ are $(\eps_{\can} r_{\comp}, T)$-complete and satisfy the $\eps_{\can}$-canonical neighborhood assumption at scales $(\eps_{\can} r_{\comp},  \lb 1)$.
If $$\eps \leq c_1 \lambda \cdot \alpha \ov{r}_{\comp} \,,$$ for some universal constant $c_1 > 0$, then $U^* := \{ x \in \M_0 \;\; : \;\; \rho (x) > \lambda r_{\comp} \} \subset U$.
So, without loss of generality, we can replace $\zeta$ by $\zeta |_{U^*}$.

Let us now verify the bounds on $h_0 := \zeta^* g'_0 - g_0$ in the assumptions of Theorem~\ref{Thm_existence_comparison_domain_comparison}.
For this purpose note that there is a universal constant $C_1 < \infty$ such that $C^{-1}_1 \rho^{-2}_1 \leq |{\Rm}| + 1 \leq C_1 \rho^{-2}_1$.
Now by assumption of this theorem and the fact that
\[ |{\Rm}| +1 \leq C_1 \rho_1^{-2} < C_1 \lambda^{-2} r_{\comp}^{-2} \leq C_1 (\lambda \alpha \ov{r}_{\comp})^{-2} r^{-2} \,,\]
on $U^*$, we have
\[ |h_0| \leq \eps \cdot r^E (|{\Rm}| +1)^{E/2}  \leq \eps  \cdot C_1^{E/2} (\lambda \alpha \ov{r}_{\comp})^{-E} \leq \eta_{\lin}, \]
as long as $$\eps \leq C_1^{-E/2} (\lambda \alpha \ov{r}_{\comp})^{E} \eta_{\lin}\,.$$
Similarly, we obtain that
\begin{multline*}
 e^{HT} \rho_1^E | h_0 | \leq C_1^{E/2} e^{HT} (|{\Rm}| +1)^{-E/2} |h_0| 
 \leq C_1^{E/2} e^{HT} \cdot \eps r^E  \\
\leq \nu \cdot 10^{-E-1} \eta_{\lin} (\alpha r \cdot \ov{r}_{\comp})^E, 
\end{multline*}
as long as $$\eps \leq C_1^{-E/2} e^{-HT} \cdot \nu \cdot 10^{-E-1} \eta_{\lin} (\alpha \ov{r}_{\comp})^E$$ and
\begin{multline*}
 e^{HT} \rho_1^3 | h_0 | 
\leq C_1^{3/2} e^{HT} (|{\Rm}| +1)^{-3/2} |h_0|  \\
\leq C_1^{E/2}  (\lambda \alpha \ov{r}_{\comp})^{3-E} r^{3-E} \cdot  e^{HT} (|{\Rm}| +1)^{-E/2} |h_0| \\
\leq C_1^{E/2} (\lambda \alpha \ov{r}_{\comp})^{3-E}  e^{HT} \cdot \eps r^{3} 
\leq \nu \cdot 10^{-1} (\lambda \alpha r \cdot \ov{r}_{\comp})^3, 
\end{multline*}
as long as $$\eps \leq C_1^{-E/2} e^{-HT} \cdot \nu \cdot 10^{-1} (\lambda \alpha \ov{r}_{\comp})^E\,.$$
Note that the three bounds that we have imposed on $\eps$ in this  paragraph  depend only on $\delta, T, E$, assuming that $\alpha$ can be chosen depending on these three constants.

Lastly, note by assumption of this theorem the $\eps_{\can}$-canonical neighborhood assumption holds on the image of $\zeta$.

We can therefore apply Theorem~\ref{Thm_existence_comparison_domain_comparison} and obtain a comparison domain $(\N, \lb \{ \N^j \}_{j=1}^J, \lb \{ t_j \}_{j=0}^J)$ and a comparison $(\Cut, \phi, \{ \phi^j \}_{j=1}^J)$ that satisfy the a priori assumptions \ref{item_time_step_r_comp_1}--\ref{item_eta_less_than_eta_lin_13} for the parameters $(\eta_{\lin}, \delta_{\nn}, \lambda, D_{\CAP}, \Lambda, \delta_{\bb}, \eps_{\can}, r_{\comp} )$ and a priori assumptions \ref{item_q_less_than_q_bar_6}--\ref{item_apa_13} for the parameters $( T, \lb E, \lb H, \lb \eta_{\lin}, \lb \nu, \lb \lambda, \lb \eta_{\cut},  \lb  D_{\cut}, \lb W, \lb A, \lb r_{\comp})$.
Without loss of generality, we may assume that $T$ is an integral multiple of $r^2_{\comp}$, i.e. $t_J = T$; otherwise we may decrease $r_{\comp}$ or increase $T$ slightly.

Let now $\wh{U} \subset \M_{[0,T]}$ be the set of $C_2 r_{\comp}$-thick points, where $C_2 = C_2 (\Lambda) < \infty$ is a constant whose value will be determined at the end of this paragraph.
We claim that $\wh{U} \subset \N \setminus \cup_{\DD \in \Cut} \DD$.
To see this consider $x \in \wh{U}$ and choose $j \in \{ 1, \ldots, J \}$ such that $x \in \M_{[t_{j-1}, t_j]}$.
Then by Lemma~\ref{lem_forward_backward_control}, if 
$$
C_3\geq \ul C_3\,, \qquad
C_2 \geq  \Lambda C_3\,,\qquad
\eps_{\can}\leq\ov\eps_{\can}^{\text{Lemma~\ref{lem_forward_backward_control}}}  \,, \qquad
r_{\comp}\leq \ov{r}_{\comp}^{\text{Lemma~\ref{lem_forward_backward_control}}}
$$ 
then $x$ survives until time $t_j$ and $x(t_j)$ is $\Lambda r_{\comp}$-thick.
(Here we have used the notation $\ov\eps_{\can}^{\text{Lemma~\ref{lem_forward_backward_control}}}$ and $\ov{r}_{\comp}^{\text{Lemma~\ref{lem_forward_backward_control}}}$ to avoid confusion with the upper bounds $\ov\eps_{\can}$ and $\ov{r}_{\comp}$ from Theorem~\ref{Thm_existence_comparison_domain_comparison}.)
So by a priori assumption \ref{item_backward_time_slice_3}(b) we have $x(t_j) \in \N_{t_j-}$ and thus $x \in \N$.
Lastly, by a priori assumption \ref{item_backward_time_slice_3}(e) we have $x \not \in \DD$ for all $\DD \in \Cut$.

By the choice of $\wh{U}$ we have
\[ |{\Rm}| + 1\geq C_1^{-1} \rho^{-2}_1 \geq C_1^{-1} C_2^{-2} r_{\comp}^{-2} 
= C_1^{-1} C_2^{-1} \ov{r}_{\comp}^{-2} \alpha^{-2} r^{-2} \]
on $\M_{[0,T]} \setminus \wh{U}$. 
So if $$\alpha \leq \frac12 C_1^{-1/2} C_2^{-1/2} \ov{r}_{\comp}\, ,$$ then $|{\Rm}| + 1 \geq 4 r^{-2} \geq r^{-2} + 1$ and therefore $|{\Rm}| \geq r^{-2}$ on  $\M_{[0,T]} \setminus \wh{U}$.
In the notation of Theorem~\ref{Thm_comparing_RF_general_version}, we can now set $\wh\phi := \phi |_{\wh{U}}$ and $\wh{U}' := \phi (\wh{U})$.

We now need to verify the upper bound on $|\wh{\phi}^* g' - g|$ on $\wh{U}$, in the notation of Theorem~\ref{Thm_comparing_RF_general_version}.
To this end, note that by a priori assumption \ref{item_q_less_than_w_q_bar_7} we have
\begin{multline*}
 |h| \leq e^{-H (T - \mathfrak{t})} \rho_1^{-E} \cdot W \cdot 10^{-E-1} \eta_{\lin} ( \alpha r \ov{r}_{\comp})^E \\
\leq C_1^{E/2}  (|{\Rm}| + 1)^{E/2} W \eta_{\lin} \alpha^E r^E \ov{r}_{\comp}^E 
\leq \delta \cdot (|{\Rm}| + 1)^{E/2} r^E 
\end{multline*}
as long as $$\alpha \leq \delta^{1/E} C_1^{-1/2}   W^{-1/E}\eta_{\lin}^{-1/E} \ov{r}_{\comp}\,.$$

We now show that $|{\Rm}| \geq r^{-2}$ on $\M'_{[0,T]} \setminus \wh{U}'$ for sufficiently small $\alpha$ and $\eps$ if we additionally assume that $|{\Rm}| \geq (\eps r)^{-2}$ on $\M'_0 \setminus U'$.
To see this, assume that $|{\Rm}| (x') < r^{-2}$ for some $x' \in \M'_{[0,T]} \setminus \wh{U}'$.
So $\rho (x') \geq \frac12 C_1^{-1/2} r$.
We can now apply Lemma~\ref{lem_MM_connected}, assuming that $\eps_{\can}$ is smaller than some universal constant, to find a continuous path $\gamma : [0,1] \to \M'_{[0,T]}$ between $x'$ and a point $y' \in \M'_0$ such that $\rho (\gamma(s)) > C_3^{-1}  r$ for all $s \in [0,1]$, where $C_3$ is some universal constant, and $\t' \circ \gamma$ is non-increasing.
On the other hand, we have $\rho_{1} < C_1^{1/2} |{\Rm}|^{-1/2} \leq  C_1^{1/2} \eps r$ on $\M'_0 \setminus U'$.
So if $$\eps \leq C_1^{-1/2} C_3^{-1} \,,$$ then $y' \in U'$.
Set $y := \zeta^{-1}  (y')$.

Our next goal is to show that $y' \in \wh{U}'_0$.
To see this, we will argue that $\rho (y) > C_2 r_{\comp}$.
By Lemma~\ref{lem_forward_backward_control}, and assuming that $\eps_{\can}$ is smaller than some universal constant, we can find a universal constant $c_2 > 0$ such that $\rho > \frac12 C_3^{-1} r $ on $B(y', c_2 r)$.
So, as in the last paragraph, we obtain that $B(y', c_2 r) \subset U'$, assuming $$\eps \leq \frac12 C_1^{-1/2} C_3^{-1} \,.$$
If $\rho (y') < 1$, then we can use the $\eps_{\can}$-canonical neighborhood assumption at $y'$ to deduce bounds on higher curvature derivatives on $B(y', c_2 r)$ (as in Lemma~\ref{lem_properties_kappa_solutions_cna}), assuming that $\eps_{\can}$ is smaller than some universal constant.
On the other hand, if $\rho (y') \geq 1$, then we obtain an improved bound of the form $\rho > c_3$ on $B(y', c_2 r)$ for some universal constant $c_3 > 0$ (via Lemma~\ref{lem_bounded_curv_bounded_dist}).
So using Lemma~\ref{lem_convergence_of_scales}, applied similarly as in the proof of Lemma~\ref{lem_scale_distortion}, we obtain a universal constant $c_4 > 0$ such that $\rho (y) > c_4 r$, assuming that $\eta_{\lin}, \eps_{\can}$ and $\alpha$ are smaller than some universal constant.
So if $$\alpha \leq c_4 C_2^{-1}\,,$$ then $\rho (y) > C_2 r_{\comp}$ and therefore by construction of $\wh{U}$, we have $y \in \wh{U}_0$.
It follows that $y' = \phi  (y) \in \wh{U}'_0$.

Choose $s_0 \in [0,1]$ minimal with the property that $\gamma ((s_0,1]) \subset \wh{U}'$.
As $y \in \wh{U}'$, we know that $s_0 < 1$ and since $\wh{U}'$ is open and $x' \not\in \wh{U}'$, we obtain $\gamma(s_0) \not\in \wh{U}'$.
For any $s \in (s_0, 1]$ we have $\phi^{-1}  (\gamma(s)) \in \wh{U} \subset \N$.
So by Lemma~\ref{lem_scale_distortion}, and assuming that $\eta_{\lin}, \delta_{\nn}, \eps_{\can}$ and $\alpha$ were chosen smaller than some universal constant, we obtain
\[ \rho_1 \big( \phi^{-1}  (\gamma(s)) \big) \geq C^{-1}_{\sd} \rho_1 (\gamma(s)) > C^{-1}_{\sd} \cdot C_3^{-1}  r. \]
Therefore, if $$\alpha \leq \frac12 C_{\sd}^{-1} C_2^{-1} C_3^{-1} \,,$$ then
\[ \rho \big( \phi^{-1}  (\gamma(s)) \big) > 2 C_2 r_{\comp}. \]
Using Lemma~\ref{lem_forward_backward_control}, Proposition~\ref{prop_drift_control} and the uniform lower bounds on the scales of $\gamma(s)$ and $\phi^{-1}  (\gamma(s))$, we obtain that $z:= \lim_{s \nearrow s_0} \phi^{-1}  (\gamma(s))$ exists.
It follows that $\rho (z) \geq 2 C_2 r_{\comp}$.
So $z \in \wh{U}$ and thus $ \gamma(s_0) = \phi  (z) \in \wh{U}'$, contradicting the choice of $s_0$.

This finishes the proof of the theorem if we allow $\eps_{\can}$ to depend on $T$.
To see that $\eps_{\can}$ can even be chosen independently of $T$, we revert back to the notation used in the theorem and we construct $\wh\phi$ successively on time-intervals of the form $[0,1]$, $[1,2]$, \ldots.
More specifically, given $E \geq \underline{E}$ set $\eps_{\can} := \eps_{\can} (1, E)$, (i.e. the value in the weaker version of the theorem for $T = 1$).
Now assume that $\delta > 0$ and $T < \infty$ are given and assume without loss of generality that $T$ is an integer.
Set inductively $\eps_0 := \delta$ and $\eps_{i} :=  \min \{ \eps (\eps_{i-1}, 1, E), \eps_{i -1} , 1\}$, where $\eps(\cdot,1,E)$ is as in the statement of the weaker version of the theorem in the $T=1$ case, as well as $r_0 := r$ and $r_i := \eps_i r_{i-1}$, for $i = 1, \ldots, T$.
Assume now that $|{\Rm}| \geq (\eps_1 \cdots \eps_T r )^{-2} = (\eps_T r_{T-1})^{-2}$ on $\M_0 \setminus U$ (and possibly also on $\M'_0 \setminus U'$) and 
\[ | \phi^* g'_0 - g_0 | \leq \eps_T \cdot r_{T}^{2E} (|{\Rm}| + 1)^E. \]
We can then apply the weaker version of the theorem for $r = r_{T-i}$ and $\delta = \eps_{T-i}$ to find a sequence of subsets $\wh{U}_i \subset \M_{[i-1, i]}$, $\wh{U}'_i \subset \M'_{[i-1,i]}$ and diffeomorphisms $\wh\phi_i : \wh{U}_i \to \wh{U}'_i$ such that
\[ | \wh\phi^*_i g' - g | \leq \eps_{T-i} \cdot r_{T-i}^{2E} (|{\Rm}| + 1)^E \]
and $|{\Rm}| \geq r_{T-i}^{-2} $ on $\M_{[i-1, i]}$ (and possibly also on $\M_{[i-1,i]} \setminus \wh{U}'_i$) for $i = 1, \ldots, T$.
Moreover, $\wh\phi_{i-1} = \wh\phi_{i}$ on $\wh{U}_{i-1} \cap \wh{U}_i \subset \M_i$. 
Then $\wh\phi$ can be constructed by combining the diffeomorphisms $\wh\phi_1, \ldots, \wh\phi_T$ on the open subset
\[ \big( \wh{U}_1 \cap \M_{[0,1)} \big) \cup (\wh{U}_1 \cap \wh{U}_2) \cup \big( \wh{U}_2 \cap \M_{(1,2)} \big) \cup (\wh{U}_2 \cap \wh{U}_3) \cup \ldots \cup \big( \wh{U}_{T} \cap \M_{(T-1,T]} \big). \]
This finishes the proof.
\end{proof}

\bigskip

\begin{proof}[Proof of Theorem~\ref{Thm_comparing_RF_weak_form}.]
The theorem is a consequence of Theorem~\ref{Thm_comparing_RF_general_version}.
To see this, assume $\delta \leq 1$, choose $E := \underline{E}$ and consider the constants $\eps_{\can} = \eps_{\can} (E)$ and $\eps = \eps ( \frac{1}{3^E} \delta, \lb T, \lb E)$ from Theorem~\ref{Thm_comparing_RF_general_version}.
Set $\eps' := \min \{ \delta^{2E} \eps, \delta \eps, \eps_{\can} \}$ and $r := \delta$.

We claim that Theorem~\ref{Thm_comparing_RF_weak_form} holds for $\eps = \eps'$.
By the assumption of this theorem we have
\[ | \phi^* g'_0 - g_0 | \leq \eps'   \leq \eps \cdot \delta^{2E}
\leq \eps \cdot  r^{2E}   (|{\Rm}| + 1)^{E}. \]
We also have $|{\Rm}| \geq \eps^{\prime -2} \geq (\eps r)^{-2}$ on $\M_0 \setminus U$.
So Theorem~\ref{Thm_comparing_RF_general_version} can be applied and yields the existence of a time-preserving diffeomorphism $\wh\phi : \wh{U} \to \wh{U}'$ such that $\wh\phi = \phi$ on $U \cap \wh{U}$ and
\begin{equation} \label{eq_phi_hat_weakform}
 | \wh\phi^* g' - g | \leq  \frac1{3^E} \delta  \cdot \delta^{2E} (|{\Rm}| + 1)^{E} <  \delta,  
\end{equation}
on $\wh{U} \cap \{ |{\Rm}| < 2 \delta^{-2} \}$ and $|{\Rm}| \geq \delta^{-2}$ on $\M_{[0,T]} \setminus \wh{U}$.
We can now replace $\wh{U}$ by $\wh{U} \cap \{ |{\Rm}| < 2 \delta^{-2} \}$ and then $\wh{U}'$ by $\wh\phi ( \wh{U} )$.
Then (\ref{eq_phi_hat_weakform}) holds on all of $\wh{U}$ and we still have $|{\Rm}| \geq \delta^{-2}$ on $\M_{[0,T]} \setminus \wh{U}$.
\end{proof}

\begin{proof}[Proof of Addendum to Theorem~\ref{Thm_comparing_RF_weak_form}]
The bounds on the higher derivatives follow from Lemma~\ref{lem_loc_gradient_estimate}, combined with Lemma~\ref{lem_forward_backward_control} and Shi's estimates, since $\wh\phi^* g' - g$ satisfies that Ricci-DeTurck equation.
\end{proof}

\bigskip

We now apply the stability theorem, Theorem~\ref{Thm_comparing_RF_general_version}, to prove Theorem~\ref{thm_uniquness}, which asserts  the uniqueness of the Ricci flow spacetimes with a given initial condition, under completeness and canonical neighborhood assumptions.

The idea of the proof is as follows.  We first apply Theorem~\ref{Thm_comparing_RF_general_version} 
to produce a sequence of maps $\wh\phi_i:U_i\ra \M'_{[0,T]}$ 
such that $\cup_i U_i=\M_{[0,T]}$ and the $|\wh\phi_i^*g'-g|\leq\de_i\ra 0$.  
We then show that the $\wh\phi_i$s converge locally smoothly to the desired diffeomorphism $\wh\phi$.  To do this, we appeal to the drift bound in Proposition~\ref{prop_drift_control} to propagate the region of convergence over time, and we use uniqueness of isometries of Riemannian manifolds to propagate the convergence within time-slices.

\bigskip

\begin{proof}[Proof of Theorem~\ref{thm_uniquness}.]
We will prove the theorem in the case $T < \infty$.
The case $T = \infty$ follows by letting $T \to \infty$.
Choose $\underline{E}$ and $\eps_{\can} := \eps_{\can} (\underline{E})$ according to Theorem~\ref{Thm_comparing_RF_general_version}.
Also, by parabolic rescaling, we may assume without loss of generality that $r = 1$.

By Theorem~\ref{Thm_comparing_RF_general_version} we can find a sequence of open subsets $U_1 \subset U_2 \subset \ldots \subset \M_{[0,T]}$ such that $\cup_{i=1}^\infty U_i = \M_{[0,T]}$ and a sequence of time-preserving diffeomorphisms onto their images $\wh\phi_i : U_i \to \M'_{[0,T]}$ that satisfy the harmonic map heat flow equation, such that $\wh\phi_i |_{U_i \cap \M_0} = \phi |_{U_i \cap \M_0}$,  $\t' \circ \wh\phi_i = \t$ and
\begin{equation} \label{eq_bilipschitz_u_proof}
 | \wh\phi^*_i g' - g | \leq \delta_i \to 0. 
\end{equation}

Let $Y$ be the set of points $x\in \M_{[0,T]}$ such that the pointwise limit $\phi_\infty (x) := \lim_{i \to \infty} \wh\phi_i (x)$ exists.  Let 
$$
X=\{x\in \M_{[0,T]}\;\; : \;\; B(x,r)\subset Y\;\text{for some}\;r>0\}\,,
$$
so $X$ is the set of points $x\in \M_{[0,T]}$ that belong to relative interior of $Y\cap \M_{\t(x)}$ in $\M_{\t(x)}$.  
Recall that $X_t=X\cap \M_t$ for $t\geq 0$.  
Our main goal is to show that $X = \M_{[0,T]}$ and that the pointwise limit  $\wh\phi_\infty$ is smooth, preserves the metric, and time vector field.  Obviously, $X_0=Y_0=\M_0$, since $\wh\phi_i=\phi$ on $U_i\cap \M_0$.

\setcounter{claim}{0}

\begin{claim} \label{cl_1.3_1}
For every $t\in [0,T]$:
\begin{enumerate}[label=(\alph*)]
\item \label{ass_1.3_cl1_a} $\wh\phi_\infty\mid_{X_t}\ra \M_t$ is a smooth isometric immersion.
\item \label{ass_1.3_cl1_b} $X_t$ is a union of connected components of $\M_t$.
\item \label{ass_1.3_cl1_c} For all $x\in X_t$, $\rho(\wh\phi_\infty(x))=\rho(x)$.
\een
\end{claim}
\begin{proof}
Suppose $t\in [0,T]$ and $z$ is in the closure of $X_t$.  Choose $r_0>0$ such that $\ol{B(z,6r_0)}$ is compact, and pick $x\in B(z,r_0) \cap X_t$. 
Hence $\ol{B(x,5r_0)}$ is compact and $z \in B(x, r_0)$.  There is a sequence $L_i\ra 1$ such that for large $i$ we have $\ol{B(x,5r_0)}\subset U_i$, and $L_i^{-1}g\leq \wh\phi_i^*g'\leq L_ig$ on $\ol{B(x,5r_0)}$.
An elementary Riemannian geometry argument gives, for large $i$, that
$\wh\phi_i(\ol{B(x,5r_0)})\supset B(\wh\phi_i(x),4r_0)$ and the restriction of $\wh\phi_i$ to $B(x,r_0)$ is $L_i$-bilipschitz with respect to the Riemannian distance functions on $\M_t$ and $\M'_t$.  Since $\wh\phi_i(x)\ra \wh\phi_\infty(x)$, for large $i$ we have
\begin{multline*}
\wh\phi_i(B(x,r_0))\subset B(\wh\phi_i(x),2r_0)\subset B(\wh\phi_\infty(x),3r_0)
\subset B(\wh\phi_i(x),4r_0)\subset \wh\phi_i(\ol{B(x,5r_0)})\,,
\end{multline*}
and therefore $\ol{B(\wh\phi_\infty(x),3r_0)}$ is compact.

Put $B:=B(x,r_0)$.  Suppose $\{\wh\phi_i |_B\}$ does not converge pointwise.  Then by the Arzela-Ascoli theorem, the sequence $\{\wh\phi_i |_B\}$ has two distinct subsequential limits $\psi,\psi':B\ra \M'_t$, and since $L_i\ra 1$, both maps  preserve the distance functions on $\M$ and $\M'$.  Hence $\psi$ and $\psi'$ are smooth Riemannian isometries.  
They agree on a neighborhood of $x$ in  $X_t$, because $x\in X$, and since $B$ is connected, they must coincide, contradicting $\psi\neq\psi'$.  Thus $B\subset X_t$, and the pointwise limit $\wh\phi_\infty$ is a smooth Riemannian isometry on $B$. 
This shows that the closure of $X_t$ is open, which implies assertion \ref{ass_1.3_cl1_b}.
Our proof also implies assertion \ref{ass_1.3_cl1_a}, which implies assertion \ref{ass_1.3_cl1_c}.
\end{proof}

\begin{claim} \label{cl_1.3_2}
There is a universal constant $c > 0$ such that for every $x \in Y_t = Y \cap \M_t$ and $\tau_x := c  \rho_1^2(x)$ the following holds for all $t' \in [ t - \tau_x, t + \tau_x] \cap [0,T]$: 
\begin{enumerate}[label=(\alph*)]
\item \label{ass_1.3_cl2_a} $x$ survives until $t'$, and $x(t') \in Y$.
\item \label{ass_1.3_cl2_b} $\wh\phi_\infty (x)$ survives until $t'$, and $(\wh\phi_\infty (x))(t') = \wh\phi_\infty (x(t'))$.
\end{enumerate}
\end{claim}

\begin{proof}
The claim follows from Proposition~\ref{prop_drift_control} via a continuity argument.
Let $x \in X_t$ and set $x' := \wh\phi_\infty (x)$.

Using assertion \ref{ass_1.3_cl1_c}  of Claim~\ref{cl_1.3_1} and Lemma~\ref{lem_forward_backward_control}, assuming that $\eps_{\can}$ is smaller than some universal constant, we can find a universal constant $c > 0$ such that for $r_0 := c^{1/2} \rho_1 (x)$ the following holds:
For all $t_0 \geq 0$ with $|t - t_0| \leq r_0^2$ the parabolic neighborhoods $P(x (t_0), r_0, 2r_0^2)$ and $P(x'(t_0), \lb 100 r_0, \lb 2r_0^2)$ are unscathed and $|{\Rm}| \leq r_0^{-2}$ on both.
By compactness, we moreover find a constant $1 \leq A_x < \infty$, which may depend on $x$, such that $|{\nabla^{m} \Rm}| \leq A_x r_0^{ -2-m}$ for $m = 0, 1, \ldots, 3$ on both parabolic neighborhoods.
Moreover, lengths of curves inside these parabolic neighborhoods are distorted by at most a factor of 2 under the Ricci flow.

For each $i$ choose $t^*_{-, i}$ minimal and $t^*_{+,i}$ maximal with $0 \leq t^*_{-,i} \leq t \leq t^*_{+,i} \leq T$ and $|t^*_{i, \pm} - t| \leq r_0^2$ such that we have $d_{t_0} (\wh\phi_i(x(t_0)) , (\wh\phi_i (x))(t_0)) < r_0$ for all $t_0 \in (t^*_{-,i}, t^*_{+,i})$.
Since $x \in Y_t$, we have $d_t (\wh\phi_i (x) , x') < r_0$ for large $i$.
So by the length distortion bound on $P(x'(t_0), \lb 100 r_0, \lb 2r_0^2)$ we have $d_{t_0} ((\wh\phi_i (x(t')))(t_0), x'(t_0)) < 4 r_0$ for all $t_0 \in (t^*_{-,i}, t]$ and $t' \in [t_0, t^*_{+,i})$ if $i$ is large (we use the convention $[t,t) = (t,t] = \{ t \}$ here).
By (\ref{eq_bilipschitz_u_proof}) and the distance distortion bounds on $P(x(t_0), r_0, 2r_0^2)$ and $P(x'(t_0), \lb 100 r_0, \lb 2r_0^2)$ we therefore obtain that for large $i$ and $t_0 \in (t^*_{-,i}, t]$
\[ \wh\phi_i \big( P(x(t_0), r_0, t^*_{+,i} - t_0) \big) \subset P( x'(t_0), \lb 100 r_0, \lb 2 r_0^2). \]

We can therefore apply Proposition~\ref{prop_drift_control} for $M = B(x(t_0),  r_0)$, $M' = B(x'(t_0), 100 r_0)$, $r = r_0 $ and $A = A_x $ along with (\ref{eq_bilipschitz_u_proof}) to find that there is a sequence $\eps_i \to 0$ such that for large $i$ we have
\begin{equation*} \label{eq_dt0_whphi}
 d_{t_0} \big( (\wh\phi_i (x(t')))(t_0), (\wh\phi_i (x))(t_0) \big) \leq \eps_i r_0 
\end{equation*}
for all $t_0 \in (t^*_{-,i}, t]$ and $t' \in [t, t^*_{+,i})$.
By the distance distortion bound on $P(x'(t_0), \lb 100 r_0, \lb 2r_0^2)$ this implies that
\begin{equation} \label{eq_dtprime_whphi}
 d_{t'} \big( \wh\phi_i (x(t')), (\wh\phi_i (x))(t') \big) \leq 2\eps_i r_0 
\end{equation}
for all such $t_0$ and $t'$.
The bound (\ref{eq_dtprime_whphi}) implies that for large $i$ we have $t^*_{-,i} = t - r_0^2$ or $t^*_{-,i} = 0$ and $t^*_{i,+} = t + r_0^2$ or $t^*_{i,+} = T$ due to their minimal and maximal choice.
So (\ref{eq_dtprime_whphi})  implies  assertions \ref{ass_1.3_cl2_a} and \ref{ass_1.3_cl2_b}.
\end{proof}

\begin{claim} \label{cl_1.3_3}
\begin{enumerate}[label=(\alph*)]
\item \label{ass_1.3_cl3_a} $X$ is (relatively) open and closed in $\M_{[0,T]}$.
\item \label{ass_1.3_cl3_b} $\wh\phi_\infty$ is smooth, and $(\wh\phi_\infty)_*(\D_{\t})=\D_{\t'}$.
\item \label{ass_1.3_cl3_c} $X=\M_{[0,T]}$.
\end{enumerate}
\end{claim}
\begin{proof}
Suppose $z\in \M_t$ and $z$ belongs to the closure of $X$.   For $r > 0$ sufficiently small, by assertion \ref{ass_1.3_cl1_b} of Claim~\ref{cl_1.3_1}, there exists $t'\in[t-r^2,t+r^2]\cap[0,T]$ such that $(B (z,r) )(t')$ is contained in $X$.  Shrinking $r$ if necessary, we may assume that for all $x\in ( B(z,r) )(t')$, we have $\tau_x\geq 2r^2$, where $\tau_x$ is as in Claim~\ref{cl_1.3_2}.  Thus by assertion \ref{ass_1.3_cl2_a} of Claim~\ref{cl_1.3_2} we conclude that $( B(x,r) ) (\ol{t})\subset X$ for all $\ol{t}\in [t-r^2,t+r^2]\cap[0,T]$.  
This implies the closure of $X$ in $\M_{[0,T]}$ is open, which implies assertion \ref{ass_1.3_cl3_a}.

By assertion \ref{ass_1.3_cl2_b} of Claim~\ref{cl_1.3_2}, it follows that $\wh\phi_\infty$ locally commutes with the flows of the time vector fields $\D_{\t}$ and $\D_{\t'}$ on $\M$ and $\M'$, respectively.  Combining this with assertion \ref{ass_1.3_cl1_a}  of Claim~\ref{cl_1.3_1}, we obtain assertion \ref{ass_1.3_cl3_b}.  By assertion \ref{ass_1.3_cl3_a}, it follows that $X$ is a union of connected components of $\M_{[0,T]}$.  Assertion \ref{ass_1.3_cl3_c} now follows from Lemma~\ref{lem_MM_connected}, assuming that $\eps_{\can}$ is smaller than some universal constant, and the fact that $\M_0 \subset X$.    
\end{proof}

By Claim~\ref{cl_1.3_3} we have constructed a smooth map $\wh\phi_\infty : \M_{[0,T]} \to \M'_{[0,T]}$ such that
\begin{equation} \label{eq_RF_spacetime_isometry}
\wh\phi^*_\infty g' = g, \qquad
\wh\phi_\infty |_{\M_0} = \phi, \qquad
(\wh\phi_\infty)_* \partial_{\t} = \partial_{\t'}, \qquad
\t' \circ \wh\phi_\infty = \t.
\end{equation}

We now claim that the map $\wh\phi_\infty$ is uniquely characterized by (\ref{eq_RF_spacetime_isometry}).
To see this, consider two such maps $\wh\phi_\infty, \wh\phi'_\infty$.
As both maps satisfy the harmonic map heat flow equation, we can apply the conclusions of our proof to up to this point to the sequences $U_i = \M_{[0,T]}$ and $\wh\phi_{2i-1} = \wh\phi_\infty$, $\wh\phi_{2i} = \wh\phi'_\infty$.
It follows that $\wh\phi_{i}$ converges pointwise, and therefore we must have $\wh\phi_\infty = \wh\phi'_\infty$ as asserted.

It remains to show that $\wh\phi_\infty$ is bijective.
To see this, we can interchange the roles of $\M$ and $\M'$ and apply our discussion to obtain a map $\wh\psi_\infty : \M'_{[0,T]} \to \M_{[0,T]}$ such that
\[
\wh\psi^*_\infty g = g', \qquad
\wh\psi_\infty |_{\M'_0} = \psi := \phi^{-1}, \qquad
(\wh\psi_\infty)_* \partial_{\t'} = \partial_{\t}, \qquad
\t \circ \wh\psi_\infty = \t'.
\]
Now consider the composition $\alpha := \wh\psi_\infty \circ \wh\phi_\infty$ such that
\[ \alpha^* g = g, \qquad
\alpha |_{\M_0} = \id_{\M_0}, \qquad
\alpha_* \partial_{\t} = \partial_{\t}, \qquad
\t \circ \alpha = \t, \]
By the uniqueness property, as discussed in the previous paragraph (for $\M = \M'$), we obtain that $\wh\psi_\infty \circ \wh\phi_\infty = \alpha = \id_{\M_{[0,T]}}$.
Similarly, we obtain that $\wh\phi_\infty \circ \wh\psi_\infty  = \id_{\M'_{[0,T]}}$.
This shows that $\wh\phi_\infty$ is bijective, finishing the proof.
\end{proof}

\appendix

\section{Ricci-DeTurck flow and harmonic map heat flow} \label{appx_Ricci_deT}

In this section we discuss the main estimates for harmonic map heat flow and  Ricci-DeTurck flow that will be needed in the paper.  While the  general methodology is fairly standard, we were unable to find  suitable references in the general PDE literature for these results.

\subsection{The main equations} \label{subsec_hmhf_main_equations}
In this subsection we derive the general equations for harmonic map heat flow with time dependent metrics on the source and target, and the associated Ricci-DeTurck flow.  Most of the ideas presented in this subsection go back to DeTurck \cite{DeTurck:1983jp}  and Hamilton \cite{hamilton_formation}.

Consider two $n$-dimensional manifolds $M, M'$, each equipped with a smooth family of Riemannian metrics $(g_t)_{t \in [0,T]}$, $(g'_t)_{t \in [0,T]}$.
Let moreover $(\chi_t)_{t \in [0,T]}$, $\chi_t : M' \to M$ be a smooth family of maps.

\begin{definition} \label{def_harmonic_map_heat_flow_appendix}
We say that the family $(\chi_t)_{t \in [0,T]}$ moves by {\bf harmonic map heat flow between $(M', g'_t)$ and $(M,g_t)$} if the family satisfies the following evolution equation:
\begin{equation} \label{eq_hmhf_def_appx}
 \partial_t \chi_t = \triangle_{g'_t, g_t} \chi_t = \sum_{i=1}^n \big( \nabla^{g_t}_{d\chi_t (e_i)} d\chi_t (e_i) - d\chi_t ( \nabla^{g'_t}_{e_i} e_i ) \big),  
\end{equation}
where $\{ e_i \}_{i=1}^n$ is a local frame field on $M'$ that is orthonormal with respect to $g'_t$.
\end{definition}

Assume now for the remainder of this subsection that all the maps $\chi_t$ are diffeomorphisms and consider their inverses $\chi_t^{-1}$.
Let
\begin{equation} \label{eq_associated_perturbation_appendix}
 h_t := \big( \chi^{-1}_t \big)^* g'_t - g_t 
\end{equation}
be the {\bf associated perturbation}.
The pullback $( \chi^{-1}_t )^* g'_t = g_t + h_t$ evolves by the following equation
\begin{align}
 \partial_t \big( \big( \chi^{-1}_t \big)^* g'_t \big) &= \big( \chi^{-1}_t \big)^* \partial_t g'_t - \mathcal{L}_{\partial_t \chi_t \circ \chi_t^{-1}} \big( \big( \chi^{-1}_t \big)^* g'_t \big) \label{eq_evolution_pullback} \\
 &= \big( \chi^{-1}_t \big)^* \big(  \partial_t g'_t + 2 \Ric(g'_t) \big) - 2 \big( \chi^{-1}_t \big)^* \Ric (g'_t) \notag \\
 &\qquad\qquad - \mathcal{L}_{\partial_t \chi_t \circ \chi_t^{-1}} \big( \big( \chi^{-1}_t \big)^* g'_t \big) \notag \\
 &= \big( \chi^{-1}_t \big)^* \big(  \partial_t g'_t  + 2 \Ric(g'_t) \big) \notag \\
 &\qquad\qquad - 2  \Ric (g_t + h_t) - \mathcal{L}_{\partial_t \chi_t \circ \chi_t^{-1}} ( g_t + h_t )  \notag \\
 &= \big( \chi^{-1}_t \big)^* \big( \partial_t g'_t + 2 \Ric(g'_t)  \big) - 2 \Ric (g_t) + \mathcal{X}_t , \notag
\end{align}
where $\mathcal{X}_t$ can be expressed as follows (in the following identity, covariant derivatives and curvature quantities are taken with respect to $g_t$ and the time-index $t$ is suppressed)
\begin{align*}
 \mathcal{X}_{ij} &=    (g+h)^{pq} \big( \nabla^2_{pq} h_{ij} + R_{pij}^{\quad u}  h_{uq} + R_{pji}^{\quad u} h_{uq} - R_{ip q}^{\quad u} h_{uj} - R_{j pq}^{\quad u} h_{iu } \big) \\
&\quad - \frac12 (g + h)^{pq} (g+h)^{uv} \big({ - \nabla_i h_{pu} \nabla_j h_{qv} - 2 \nabla_u h_{ip} \nabla_q h_{jv}} \\ 
&\qquad\qquad\qquad + 2 \nabla_u h_{ip} \nabla_v h_{jq}  + 2 \nabla_p h_{iv} \nabla_j h_{qu} 
+ 2 \nabla_i h_{pu} \nabla_q h_{jv} \big) .
\end{align*}
We will now use (\ref{eq_evolution_pullback}) to derive an evolution equation for $h_t$.
First observe that
\[ \partial_t h_t = \partial_t \big( \big( \chi^{-1}_t \big)^* g'_t \big) - \partial_t g_t. \]
Similarly as in Uhlenbeck's trick, we define (we will suppress the time-index again wherever it interferes with the index notation)
\[  (\nabla_{\partial_t} h_t)_{ij} = (\partial_t h_t)_{ij}  - \frac12 g^{pq} \big( h_{pj}  \partial_t g_{qi} + h_{ip}  \partial_t g_{qj} ). \]
Then
\begin{align} 
 (\nabla_{\partial_t} h_t)_{ij} 
 &= \big( \big( \chi^{-1}_t \big)^* \big( \partial_t g'_t + 2 \Ric(g'_t) \big) \big)_{ij} - \big(  \partial_t g_t + 2\Ric(g_t)  \big)_{ij}  \label{eq_evolution_h} \\
 &\qquad - \frac12 g^{pq} \big( h_{p j}  \partial_t g_{q i} + h_{ip }  \partial_t g_{qj} )  + \mathcal{X}_{ij}  \notag \\
 &= \big( \big( \chi^{-1}_t \big)^* \big( \partial_t g'_t  + 2 \Ric(g'_t) \big) \big)_{ij} - \big(  \partial_t g_t + 2 \Ric(g_t) \big)_{ij} \notag \\
&\qquad - \frac12 g^{pq} \big(  \partial_t g_t + 2 \Ric(g_t) \big)_{pi} h_{qj} \notag \\
&\qquad- \frac12 g^{pq} \big( \partial_t g_t + 2 \Ric(g_t) \big)_{pj} h_{iq} + \mathcal{Y}_{ij}, \notag 
\end{align}
where
\begin{align*}
 \mathcal{Y}_{ij} &= \mathcal{X}_{ij} + g^{pq} \Ric_{ip} h_{qj} + g^{pq} \Ric_{jp} h_{qi} \\
  &= (g+h)^{pq} \big( \nabla^2_{pq} h_{ij} + R_{pij}^{\quad u}  h_{uq} + R_{pji}^{\quad u} h_{uq} \big) \\
& \qquad\qquad\qquad + \big( g^{pq} - (g+h)^{pq} \big) \big( R_{ipq}^{\quad u} h_{uj} + R_{jpq}^{\quad u} h_{iu } \big)  \\
&\quad - \frac12 (g + h)^{pq} (g+h)^{uv} \big( - \nabla_i h_{pu} \nabla_j h_{qv} - 2 \nabla_u h_{ip} \nabla_q h_{jv} \\ 
&\qquad\qquad\qquad + 2 \nabla_u h_{ip} \nabla_v h_{jq}  + 2 \nabla_{p } h_{iv} \nabla_j h_{q u} 
+ 2 \nabla_i h_{pu} \nabla_q h_{jv} \big).
\end{align*}

In the following, we will focus on the case in which $h_t$ is small and in which the families of metrics $(g_t)_{t \in [0,T]}$ and $(g'_t)_{t \in [0,T]}$ almost satisfy the Ricci flow equation in the following sense.
For parameters $0 < \eta < 0.1$ and $\delta > 0$ we assume that for all $t \in [0,T]$
\[ - \eta g_t \leq h_t \leq \eta g_t \]
and
\begin{equation} \label{eq_eta_almost_RF}
 - \delta g'_t \leq \partial_t g'_t  + 2 \Ric(g'_t) \leq \delta g'_t, \qquad - \delta g_t \leq  \partial_t g_t + 2 \Ric(g_t) \leq \delta g_t. 
\end{equation}
If we now multiply (\ref{eq_evolution_h}) by $2g^{iu } g^{jv } h_{uv }$, then we obtain that for some dimensional constant $C_0 < \infty$:
\begin{multline*} \label{eq_evol_h2_with_gradient}
  \partial_t |h|^2 \leq (g+h)^{ij} \nabla_{ij}^2 |h|^2 - 2(g+h)^{ij} g^{pq} g^{uv} \nabla_i h_{pu} \nabla_j h_{qv}  \\
  +  C_0 \delta \cdot |h| + C_0 |{\Rm}_{g}| \cdot |h|^2 + C_0 |h| \cdot |\nabla h|^2.  
\end{multline*}
We will later consider the case  $\eta < \min \{ 0.1, C_0^{-1} \}$.
Note that then 
\begin{equation} \label{eq_RdT_norm_evolution}
  \partial_t |h|^2 \leq   (g+h)^{ij} \nabla_{ij}^2 |h|^2  + C_0 \delta \cdot |h|  + C_0 |{\Rm}_{g}|  \cdot |h|^2 .
\end{equation}

Next, let us consider the case in which $(g_t)_{t \in [0,T]}$ and $(g'_t)_{t \in [0,T]}$ both satisfy the (exact) Ricci flow equation.
Then (\ref{eq_evolution_pullback}) implies the {\bf Ricci-DeTurck equation} for the pullback  metric $g_t + h_t = (\chi^{-1}_t)^* g'_t$.
\begin{equation*}
 \partial_t ( g_t + h_t ) = - 2 \Ric  (  g_t + h_t )  - \mathcal{L}_{X_{g_t} ( g_t + h_t )}  ( g_t + h_t ), 
\end{equation*}
where the vector field $X_{g_t} (g_t + h_t)$ is defined by
\begin{equation} \label{eq_def_X_appendix}
 X_{g} (g^*) :=  \triangle_{g^*, g} \id_M =  \sum_{i=1}^n \big( \nabla^{g}_{e_i} e_i - \nabla^{g^*}_{e_i} e_i \big),
\end{equation}
for a local frame $\{ e_i \}_{i=1}^n$ that is orthonormal with respect to $g^*$.
Note that
\begin{equation} \label{eq_X_is_dt_chi}
 X_{g_t} (g_t + h_t) = \partial_t \chi_t \circ \chi^{-1}_t. 
\end{equation}

From an analytical point of view, (\ref{eq_evolution_h}) implies that the Ricci-DeTurck equation can be expressed as follows in terms of the perturbations $h_t$ (also referred to as the {\bf Ricci-DeTurck perturbation equation} here):
\begin{equation} \label{eq_Ricci_de_Turck}
 \nabla_{\partial_t} h_t = \triangle_{g_t} h_t + 2\Rm_{g_t} (h_t) + \mathcal{Q}_{g_t} [ h_t]. 
\end{equation}
Here the expression on the left-hand side now denotes the conventional Uhlenbeck trick:
\[ (\nabla_{\partial_t} h_t)_{ij} = (\partial_t h_t)_{ij} + g^{pq}_t \big(  ( h_t)_{pj} \Ric_{qi} + ( h_t)_{ip}  \Ric_{qj} \big) \]
Moreover,
\[ (\Rm_{g_t} (h_t))_{ij} = g^{pq} R_{p ij}^{\quad u} h_{qu} \]
and $\mathcal{Q}_{g_t} [h_t] = \mathcal{Q}^{(1)}_{g_t} [h_t]$ where
\begin{align} \label{eq_QQ_formula}
  \big( \mathcal{Q}^{(\alpha)}_{g_t} [h_t] \big)_{ij} &=    \big( (g+ \alpha h)^{pq} - g^{pq} \big) \big( \nabla^2_{pq} h_{ij} + R_{pij}^{\quad u} h_{uq} + R_{pji}^{\quad u} h_{uq} \big) \\
& \qquad\qquad + \big( g^{pq} - (g+\alpha h)^{pq} \big) \big( R_{ipq}^{\quad u} h_{uj} + R_{jpq}^{\quad u} h_{iu } \big) \notag \\
&\quad - \frac{\alpha}2 (g + \alpha h)^{pq} (g + \alpha h)^{uv} \big({ - \nabla_i h_{pu} \nabla_j h_{qv} - 2 \nabla_u h_{ip} \nabla_q h_{jv} } \notag \\ 
&\qquad\qquad + 2 \nabla_u h_{ip} \nabla_v h_{jq}  + 2 \nabla_{p } h_{iv} \nabla_j h_{q u} 
+ 2 \nabla_i h_{pu} \nabla_q h_{jv} \big). \notag
\end{align}
In this paper, we also consider the {\bf rescaled Ricci-DeTurck  equation} for perturbations of the form $\td{h}_t := \alpha^{-1} h_t$ (we will be interested in the case $\alpha \leq 1$ mostly):
\begin{equation} \label{eq_rescaled_Ricci_de_Turck}
  \nabla_{\partial_t} \td{h}_t = \triangle_{g_t} \td{h}_t + 2 \Rm_{g_t} (\td{h}_t) + \mathcal{Q}^{(\alpha)}_{g_t} [ \td{h}_t]. 
\end{equation}
Note that $Q^{(0)} [ h_t ] = 0$.
So for $\alpha \to 0$, the equation (\ref{eq_rescaled_Ricci_de_Turck}) converges to the {\bf linearized Ricci-DeTurck equation}
\begin{equation} \label{eq_lin_RdT_Appendix}
 \nabla_{\partial_t} \td{h}_t = \triangle_{g_t} \td{h}_t + 2 \Rm_{g_t} (\td{h}_t) . 
\end{equation}

\subsection{Local derivative estimates}
In the following, we will derive local bounds on derivatives of the Ricci-DeTurck equation and the harmonic map heat flow equation.
Let us first consider the Ricci-DeTurck perturbation equation.
We obtain the following local derivative bounds.

\begin{lemma}[Local derivative estimates for Ricci-DeTurck flow] \label{lem_loc_gradient_estimate}
For any $m, n \geq 1$ there are constants $\eta_{m} = \eta_{m} (n) > 0$ and $C_m = C_{m} (n) < \infty$ such that the following holds.

Consider a Ricci flow $(g_t)_{t \in [0, r^2]}$ on an $n$-dimensional manifold $M$.
Let $p \in M$ be a point, $r > 0$ and assume that the time-$0$ ball $B(p,0, r) \subset M$ is relatively compact and that $|{\nabla^{m'} {\Rm}}| \leq r^{-2-m'}$ on $B(p,0,r) \times [0, r^2]$ for all $m' = 0, \ldots, m +2$.

Consider a solution $(h_t)_{t \in [0,r^2]}$ on $(M, (g_t)_{t \in [0,r^2]})$ to the Ricci-De\-Turck perturbation equation (\ref{eq_Ricci_de_Turck}).
Then the following holds:

\begin{enumerate}[label=(\alph*)]
\item If
\[ H := \sup_{B(x,0,r) \times [0, r^2]} |h_t|_{g_t} \leq \eta_{m}, \]
then
\[ |\nabla^{m'} h_{t }|_{g_t} \leq C_{m} H t^{-m'/2} \]
on $B(p,0,r/2) \times (0, r^2]$ for all $m' = 1, \ldots, m$.
\item If
\[ H_0 := \sup_{B(x,0,r) \times [0, r^2]}  |h_t|_{g_t}  + \max_{0 \leq m' \leq m+1} \sup_{B(x,0,r)} r^{m'} |\nabla^{m'} h_{t}|_{g_t} \leq \eta_{m}, \]
then
\[ |\nabla^{m'} h_{t}|_{g_t}\leq C_{m} H_0 r^{-m'} \]
on $B(p,0,r/2) \times [ 0, r^2]$ for all $m' = 1, \ldots, m$.
\end{enumerate}
\end{lemma}

\begin{proof}
This follows directly from \cite[Proposition~2.5]{Bamler:2014cf}, \cite[Lemma~4.4]{Appleton:2016ub} and (\ref{eq_Ricci_de_Turck}).
\end{proof}

Next, we discuss similar local derivative bounds for the harmonic map heat flow.
To this end, consider families of metrics $(g_t)_{[0,T]}$ on $M$ and $(g'_t)_{[0,T]}$ on $M'$ and a solution $(\chi_t :M'\ra M)_{t \in [0,T]}$ to the harmonic map heat flow equation (\ref{eq_hmhf_def_appx}) between $M'$ and $M$.
Choose local coordinates $(x^1, \ldots, x^n)$ on $U \subset M$ and $(y^1, \ldots, y^n)$ on $V \subset M'$ such that $\chi_t (V) \subset U$ for all $t \in [0,T]$.
Express the families of metric $(g_t)_{t \in [0,T]}$ on $U$ and $(g'_t)_{[0,T]}$ on $V$ as
\[ g_t = g_{t, ij} \, dx^i dx^j, \qquad g'_t = g'_{t, ij} \, dy^i dy^j. \]
The maps $\chi_t$ can be expressed on $V$ as an $n$-tuple of functions $(\chi^1_t (y^1, \lb \ldots,\lb y^n), \lb \ldots, \lb \chi^n_t (y^1,\lb \ldots,\lb y^n))$ and the harmonic map heat flow equation takes the form (we again suppress the $t$-index)
\begin{multline*}
 \partial_t \chi^k = g^{\prime ij} \frac{\partial^2 \chi^k}{\partial y^i \partial y^j}  
 - g^{\prime ij} g^{\prime uv} \bigg( 2 \frac{\partial g'_{iu}}{\partial y^j} - \frac{\partial g'_{ij}}{\partial y^u} \bigg) \frac{\partial \chi^k}{\partial y^v}  \\
 + g^{\prime ij} g^{k l} (\chi^1, \ldots, \chi^n) \frac{\partial \chi^u}{\partial y^i} \frac{\partial \chi^v}{\partial y^j} \bigg( 2  \frac{\partial g_{ul}}{\partial x^v} - \frac{\partial g_{uv}}{\partial x^l} \bigg)
\end{multline*}
Using this notation, we can now state the following local regularity result.

\begin{lemma}[Local gradient bounds for harmonic map heat flow] \label{lem_hmhf_loc_gradient}
For any $m, n \geq 1$ and $A < \infty$ there are constants $\alpha_m = \alpha_m (A, n)$ and $C_m = C_m (A, n) < \infty$ such that the following holds.

Choose $r > 0$ such that $r^2 \leq T$ and $p \in V$ and assume that the Euclidean ball $B(q,  r) \subset V$ is relatively compact. 
Assume that on $U \times [0, r^2]$ and $V \times [0,r^2]$ we have
\begin{equation*} \label{eq_gij_bound_2}
 |\partial^{m_1} \partial_t^{m_2} ( g_{ij} - \delta_{ij} )|  \leq \alpha_m r^{-m_1 - 2m_2}, \quad 
|\partial^{m_1} \partial_t^{m_2} (g'_{ij} - \delta_{ij} )| \leq \alpha_m r^{-m_1 - 2m_2}
\end{equation*}
for all $0 \leq m_1 + 2 m_2 \leq m+1$ (here ``$\partial^{m_1}$'' denotes spatial derivatives).
Assume moreover that there is a $p \in U$ such that $\chi_t (B(q,r)) \subset B(p,Ar)$ for all $t \in [0,r^2]$.

Then the following holds:
\begin{enumerate}[label=(\alph*)]
\item \label{ass_A.16_a} We have
\begin{equation} \label{eq_gradient_chi_bound}
 |\partial^{m_1} \partial_t^{m_2} \chi^k |  \leq C_m t^{-(m_1+2m_2-1)/2} . 
\end{equation}
on $B(q,r / 2) \times (0, r^2]$  for all $0 < m_1 + 2m_2 \leq m$.
\item \label{ass_A.16_b} If moreover
for all $0 <  m_1 \leq m + 1 $ we have
\[ |\partial^{m_1} \chi^k | \leq A r^{1-m_1}  \qquad \text{on} \qquad B(p,r) \times \{ 0 \}  \]
(for $t = 0$), then we even have
\[  |\partial^{m_1} \partial_t^{m_2} \chi^k |  \leq C_m r^{-(m_1+2m_2-1)}  \]
 on $B(q,r / 2) \times  [0, r^2]$ for all $0 <  m_1 + 2m_2 \leq m$.
\end{enumerate}
\end{lemma}

\begin{proof}
Without loss of generality, we may assume via parabolic rescaling and translating that $r = 1$ and $p=q=0$.
The constant $\alpha_m$ will be chosen in the course of this proof.
We will always assume that $\alpha_m < 0.1$.

Let $\beta > 0$ be a constant whose value we determine in the course of this proof.
Set $\td\chi^k_t := \beta \cdot \chi^k_t$.
Then $\td\chi$ satisfies an equation of the form
\begin{multline*}
 \partial_t \td\chi = g^{\prime ij} \partial^2_{ij} \td\chi + f_1 (x^1, \ldots, x^n, t) * \partial \td\chi \\
 + \alpha_m \beta^{-1} \cdot f_2 ( x^1, \ldots, x^n, \beta^{-1} \td\chi^1, \ldots, \beta^{-1} \td\chi^n, t) * \partial \td\chi * \partial \td\chi,
\end{multline*}
where $f_1, f_2$ are functions with
\begin{equation} \label{eq_fi_bound}
|\partial^{m_1} \partial^{m_2}_t f_i | \leq C(m,n)
\end{equation}
on $B(p, 0, 1) \times [0, 1]$ for all $0 \leq m_1 + 2m_2 \leq m$ (note that $f_2$ has $2n$ spatial components).
Assume for the remainder of the proof that $\alpha_m \leq \beta^{m+1}$.
Then
\[ f_3 ( x^1, \ldots, x^n,  z^1, \ldots,  z^n, t) := \alpha_m \beta^{-1}  f_2 ( x^1, \ldots, x^n, \beta^{-1} z^1, \ldots, \beta^{-1} z^n, t) \]
also satisfies a bound of the form (\ref{eq_fi_bound}).

Next, note that
\[ |\td\chi | \leq \beta A \qquad \text{on} \qquad B(p, 0, 1) \times [0, t_0]. \]
So if $\beta$ is chosen small, depending on $A$, $m$ and $n$, then we can again use \cite[Proposition~2.5]{Bamler:2014cf} in assertion \ref{ass_A.16_a} to derive bounds for $ |\partial^{m_1} \partial_t^{m_2} \td\chi^k_{ij} |$ on $B(p,0,1/2) \times (0, 1]$ that depend only on $A, m, n$. 
These bounds imply (\ref{eq_gradient_chi_bound}).
For assertion \ref{ass_A.16_b} we can use \cite[Lemma~4.4]{Appleton:2016ub}.
\end{proof}

Using this local gradient estimate, we can now prove the following drift bound.

\begin{lemma}[Drift bound in local coordinates] \label{lem_drift_local_coo}
For every $n \geq 1$ and $A < \infty$ there are constants $\tau = \tau (A, n), \alpha = \alpha (A, n) > 0$ and $C = C(A, n) < \infty$ such that the following holds.

Let $r > 0$ and let $(g_t)_{t \in [0,T]}, (g'_t)_{t \in [0,r^2]}$ be smooth families of Riemannian metrics on $n$-dimensional manifolds $M, M'$.
Assume that $(\chi_t)_{t \in [0,r^2]}$ is a solution to the harmonic map heat flow equation (\ref{eq_hmhf_def_appx}) with the property that $\chi_t$ is $A$-Lipschitz for all $t \in [0, r^2]$.

Let $q \in M'$ and $p := \chi_0 (q) \in M$.

Assume that we have the bounds $|{\nabla^m \Rm (g_t)}|, |{\nabla^m \Rm (g'_t)}| \leq \alpha r^{-2-m}$ for $m = 0, \ldots, 3$ and $|{\nabla^m \partial_t g_t}|, \lb |{\nabla^m \partial_t g'_t}| \leq \alpha r^{-2-m}$ for $m = 0,1$ on $B(p,0,r) \times [0,r^2]$ and $B(q,0,r) \times [0,r^2]$.
Assume moreover that $B(p,0,r)$ and $B(q,0,r)$ are relatively compact in $M$ and $M'$, respectively.

Then $d_0 ( \chi_t (q), p ) \leq C t^{1/2}$ for all $t \in [0, \tau r^2]$.
\end{lemma}

\begin{proof}
Without loss of generality, we may assume that $r = 1$.
Choose $t^* \in [0, 1]$ maximal with the property that $d_0 (\chi_t (q), p) \leq \frac1{10}$ for all $t \in [0, t^*]$.
Obviously, $t^* > 0$.
In the following we will find a lower bound on $t^*$ in terms of $A, n$.

Assuming $\alpha$ to be sufficiently small, we can use the $A$-Lipschitz bound on $\chi_t$ to conclude that we have for all $t \in [0, t^*]$
\[ \chi_t ( B(q, 0, (2A)^{-1}) ) \subset B(p, 0, 1) \subset U. \]

We will now apply Lemma~\ref{lem_hmhf_loc_gradient} for $r = r^* := \frac12 (2A)^{-1}$.
To do this, consider the exponential map $\exp_{q, g'_0} : T_q M' \supset B(0,r^*) \to B(q,0,r^*) \subset M'$ based at $q$ with respect to the metric $g'_0$.
Then the family of pullback metrics $(\exp_{q, g'_0})^* g'_t$ on $B(0, r^*)$ satisfies a bound of the form (\ref{eq_gradient_chi_bound}) for $r = r^*$ and $\alpha$ replaced by $C(A,n) \alpha$.
A similar bound holds for the family of pullback $(\exp_{p, g_0})^* g_t$ on $B(0,1)$.
The family of maps $\chi_t \circ \exp_{q, g'_0} : B(0,r^*) \to B(p,0,1)$ can be lifted to a family of maps $\td\chi_t : B(0, r^*) \to B(0,1)$ with $\chi_t \circ \exp_{q, g'_0} = \exp_{p,g_0} \td\chi_t$ and $\td\chi_0 (0) = 0$.

We can now apply Lemma~\ref{lem_hmhf_loc_gradient} for $\td\chi_t$ and assuming that $\alpha$ is sufficiently small, and obtain that
\[  | \partial_t \td\chi_t | \leq C' t^{-1/2} \]
for some $C' = C' (A,n) < \infty$.
Integrating this bound yields $$d_0 (\chi_t (q), p)) \leq d_0 (\td\chi_t(0),0) \leq C t^{1/2}$$ for all $t \in [0, t^*]$, where $C = C(A,n) < \infty$.

Set $\tau := \min \{ (100C)^{-2}, 1 \}$.
If $t^* < \tau$, then $d_0 (\chi_t (q), p)) < \frac1{10}$ for all $t \in [0, t^*]$, in contradiction to the maximal choice of $t^*$.
So $t^* \geq \tau$, which finishes the proof.
\end{proof}

\subsection{Short-time existence}
In this subsection, we prove our main short-time existence result, Proposition~\ref{prop_hh_flow_existence}, for the harmonic map heat flow.
The main technical challenges come from the fact that we will work in the non-compact setting and that the background metrics on domain and target are time-dependent and may not strictly satisfy the Ricci flow equation.

We first derive the following bound for solutions of the harmonic map heat flow, which is a consequence of (\ref{eq_RdT_norm_evolution}).

\begin{lemma} \label{lem_hh_flow_C0_bound}
For every $n \geq 1$ there is a constant $\ov\eta_n > 0$ such that  for any $0 < \eta_0 < \eta_1 < \ov\eta_n $ and $0 < \delta, C < \infty$ there is a constant $\tau = \tau (\eta_0, \eta_1, \delta, C, n) > 0$ such that the following holds.

Consider smooth families of metrics $(g_t)_{t \in [0, T]}$ and $(g'_t)_{t \in [0,T]}$ on $n$-dimensional manifolds $M$ and $M'$ such that (\ref{eq_eta_almost_RF}) holds.
Assume moreover that $(M, g_t)$ and $(M', g'_t)$ are complete and $|{\Rm (g_t)}|, \linebreak[1] |{\Rm (g'_t)}| \leq C$ for all $t \in [0,T]$ and that $|\nabla^{g_t} \partial_t g_t|$ is uniformly bounded on $M \times [0,T]$ (by some constant that may be independent of $C$).

Let $(\chi_t)_{t \in [0,T]}$ be a smooth family of diffeomorphisms between $M'$ and $M$ moving by harmonic map heat flow (\ref{eq_hmhf_def_appx}) and set $h_t := ( \chi^{-1}_t )^* g'_t - g_t$.
Assume that $|h_0| \leq \eta_0$ and that  $|\partial_t h_t | < C' t^{-1/2}$ on $M \times (0,T]$ for some finite constant $C'$.

Then for all $t \in [0, \min \{\tau, T \}]$ we have $|h_t| \leq \eta_1$.
\end{lemma}

Note that in this lemma the constants $\eta_0, \eta_1, \delta$ can be chosen independently of $C$.

\begin{proof}
By (\ref{eq_RdT_norm_evolution}) we have
\[   \partial_t |h|^2 \leq   (g+h)^{ij} \nabla_{ij}^2 |h|^2  + C_0^2 \delta^2 + |h|^2  + C_0 C  \cdot |h|^2,  \]
as long as $|h| \leq \ov\eta_n$ for some universal $\ov\eta_n > 0$.
So by the weak maximum principle applied to (\ref{eq_RdT_norm_evolution}) we obtain
\[ |h_t|^2 \leq \eta_0^2 e^{(C_0C +1) t} + \frac{C_0^2 \delta^2}{C_0 C +1} ( e^{(C_0C +1) t} - 1 ). \]
Note that for the application of the weak maximum principle we need to use the fact that $\partial_t g_t$ and $|\nabla^{g_t} \partial_t g_t|$ are  uniformly bounded on $M \times [0,T]$.
The bound on the first quantity follows from (\ref{eq_eta_almost_RF}) and the curvature bound and the second quantity is bounded by assumption.

The lemma now follows immediately by a continuity argument. 
Observe here that the condition $|h_t| \leq \ov\eta_n$ always holds on a slightly larger time-interval than the condition $|h_t| \leq \eta_1$, due to the bound on $|\partial_t h_t |$ and the fact that $C' t^{-1/2}$ is integrable.
\end{proof}

\bigskip
We first discuss the existence theory of the harmonic map heat flow in the case in which the domain $M'$ is compact and we will derive a lower bound on the time of existence.

\begin{lemma}[Short-time existence of harmonic map heat flow, compact case] \label{lem_existenc_hmhf_compact}
For every $n \geq 1$ and $C < \infty$ there are constants $\tau = \tau (C, n) > 0$ and $C^* = C^* (C, n) < \infty$ such that the following holds.

Let $(g_t)_{t \in [0,T]}, (g'_t)_{t \in [0,T]}$ be smooth families of Riemannian metrics on $n$-dimensional manifolds $M, M'$ and $\ov\chi : M' \to M$ a smooth map such that:
\begin{enumerate}[label = (\roman*)]
\item \label{con_A.22_i} $(M, g_0)$ is complete and $M'$ is compact.
\item \label{con_A.22_ii} 
$|{\nabla^m_{g_t} \Rm (g_t)}|, |{\nabla^m_{g'_t} \Rm (g'_t)}|, |\nabla^{m}_{g_t} \partial_t g_t|, |\nabla^{m}_{g'_t} \partial_t g'_t| \leq C $ on $M$ and $M'$ for all $t \in [0,T]$ and $m = 0, \ldots, 3$.
\item \label{con_A.22_iii} $\ov{\chi}$ is $C$-Lipschitz.
\end{enumerate}
Then the harmonic map heat flow equation
\begin{equation} \label{eq_hmhf_equation_compact}
 \partial_t \chi_t = \triangle_{g'_t, g_t} \chi_t, \qquad \chi_0 = \ov\chi 
\end{equation}
has a smooth solution on the time-interval $[0, \min \{ \tau, T \})$ and $\chi_t$ is $C^*$-Lipschitz for all $t \in [0, \min \{ \tau, T \}]$.
\end{lemma}

\begin{proof}
By standard parabolic theory, we find that (\ref{eq_hmhf_equation_compact}) has a solution $(\chi_t)_{t \in [0,T^*)}$ for some maximal $0 < T^* \leq T$.
If $T^* < T$, then this solution does not extend smoothly until time $T$.
It remains to deduce a lower bound on $T^*$ and a Lipschitz bound that only depend on $C, n$.

As explained in \cite{Bamler:2015hd}, the norm of the differential $d\chi_t \lb  \in \lb  C^\infty (M' ; \lb  T^*M' \lb \otimes \lb \chi^*_t TM)$ satisfies an evolution inequality of the form
\begin{multline*}
 \partial_t |d\chi_t|^2 \leq \triangle_{g'_t} |d\chi_t|^2 + \partial_t g'_t * d\chi_t * d\chi_t + \Ric (g'_t) * d\chi_t * d\chi_t \\
+ \partial_t g_t * d\chi_t * d\chi_t + \Rm(g_t) * d\chi_t * d\chi_t * d\chi_t * d\chi_t  . 
\end{multline*}
So for some $C' = C'(C, n) < \infty$ we have
\[  \partial_t |d\chi_t|^2 \leq \triangle_{g'_t} |d\chi_t|^2 + C' |d\chi_t|^2 + C' |d\chi_t|^4. \]
So, using assumption \ref{con_A.22_iii} and the weak maximum principle, we can find constants $\tau = \tau (C, n) > 0$ and $C'' = C'' (C, n) < \infty$ such that $| d\chi_t |^2 \leq C''$ for all $t \in [0, \min \{ \tau, T^* \})$.
So $\chi_t$ remains $C^*$-Lipschitz for all $t \in [0, \min \{ \tau, T^* \})$ for some $C^* = C^* (C, n) < \infty$.

We can now use Lemma~\ref{lem_drift_local_coo} followed by \ref{lem_hmhf_loc_gradient} to derive bounds on higher derivative of $\chi_t$ that are independent of $t$.
This shows that $\chi_t$ extends smoothly to time $\min \{ \tau, T^* \}$.
We therefore obtain a contradiction to the maximality  of $T^*$ in the case in which $T^* < \min \{ \tau, T \}$.
\end{proof}

Next, we remove the compactness assumption on $M'$, but assume that the injectivity radius of $M'$ is positive.

\begin{lemma}[Short-time existence of harmonic map heat flow, non-compact case, positive injectivity radius] \label{lem_existence_hmhf_noncompact_inj_positive}
Lemma~\ref{lem_existenc_hmhf_compact} continues to hold if we modify the assumptions by replacing \ref{con_A.22_i} and \ref{con_A.22_ii} by
\begin{enumerate}[label = (\roman*$\,^\prime$)]
\item \label{con_A.24_ip} $(M, g_0)$ and $(M', g'_0)$ are complete.
\item \label{con_A.24_iip} $|{\nabla^m_{g_t} \Rm (g_t)}|, |{\nabla^m_{g'_t} \Rm (g'_t)}|, |\nabla^{m}_{g_t} \partial_t g_t|, |\nabla^{m}_{g'_t} \partial_t g'_t| \leq C $ on $M$ and $M'$ for all $t \in [0,T]$ and $m = 0, \ldots, 7 $.
\end{enumerate}
and if we assume in addition that
\begin{enumerate}[label = (\roman*$\,^\prime$), start=4]
\item \label{con_A.24_ivp} The injectivity radius of $(M, g_0)$ and $(M', g'_0)$ is uniformly bounded from below by a positive constant.
\end{enumerate}
\end{lemma}

\begin{proof}
We will reduce the non-compact case to the compact case via a standard doubling construction.
By \cite{Cheeger:1991tx} we can find a sequence $N^{(1)} \Subset  N^{(2)} \Subset  \ldots \Subset  M'$ of domains with smooth boundary such that $\cup_{i=1}^{\infty} \Int N^{(i)} = M'$ and such that the second fundamental form of $\partial N^{(i)}$ is bounded by some constant $C' = C'(C, n) < \infty$.
Let $M^{\prime (i)}$ be the manifold that arises by identifying two copies of $N^{(i)}$ along their boundary and define $\ov\chi^{(i)} : M^{\prime (i)} \to M$ to be equal to $\ov\chi |_{N^{(i)}}$ on each copy of $N^{(i)}$.
By a smoothing construction, and using assumption (iv$\,^\prime$), we can find families of metrics $(g^{\prime (i)}_t)_{t \in [0,T]}$ on $M^{\prime (i)}$ that agree with $g^{(i)}$ away from a $1$-tubular neighborhood of $\partial N^{(i)}$ and such that the bounds in assumption (ii$\,^\prime$) continue to hold for all $m = 0, \ldots, 3$.
Moreover, by modifying $\ov\chi^{(i)}$ in a $1$-tubular neighborhood of $\partial N^{(i)}$, we can construct maps $\ov\chi^{\prime (i)} : M^{\prime (i)} \to M$ that are $C''$-Lipschitz, for some $C'' < \infty$ that is independent of $i$.
Note that $C''$ may, however, depend on the injectivity radius bound in assumption \ref{con_A.24_ivp}.

Using Lemma~\ref{lem_existenc_hmhf_compact}, we can evolve $\ov\chi^{\prime (i)}$ by the harmonic map heat flow to some family $(\chi^{\prime (i)}_t)_{t \in [0,T^*]}$ for some $T^* > 0$ that is independent of $i$, but may depend on the injectivity radius bound in assumption \ref{con_A.24_ivp}.
Moreover, the maps $\chi^{\prime (i)}_t$ are $C^{\prime *}$-Lipschitz for some uniform $C^{\prime *} < \infty$.

Using Lemmas~\ref{lem_drift_local_coo} and \ref{lem_hmhf_loc_gradient}, we obtain uniform local derivative bounds on the families $(\chi^{\prime (i)}_t)_{t \in [0,T^*]}$.
So after passing to a subsequence, these families converge to a solution $\chi_t : M' \to M$ of the harmonic map heat flow on the time-interval $[0, T^*]$.

By the same maximum principle argument as used in the proof of Lemma~\ref{lem_existenc_hmhf_compact}, we obtain a Lipschitz bound on $\chi_t$ of the form $C^* (C, n)$ that holds up to  time $\min \{ \tau, T \}$, where $\tau = \tau (C, n) > 0$ is a constant that does not depend on the injectivity radius bound in \ref{con_A.24_ivp}.
Assume now that $T^* < T$ is chosen maximal with the property that the harmonic map heat flow exists on $[0,T^*)$.
If $T^* < \min \{ \tau, T \}$, then we can argue as in the proof of Lemma~\ref{lem_existenc_hmhf_compact} that the flow extends smoothly to time $T^*$ and then restart the flow at time $T^*$.
This would contradict the maximal choice of $T^*$.
Therefore, $T^* \geq \min \{ \tau, T \}$.
\end{proof}

Using a similar construction, we can remove the assumption on the positivity of the injectivity radius.

\begin{lemma}\label{ex_hh_flow_version_1}
Lemma~\ref{lem_existence_hmhf_noncompact_inj_positive} continues to hold if we remove assumption (iv$\,^\prime$) and replace assumption (ii$\,^\prime$) by
\begin{enumerate}[label = (\roman*$\,^{\prime\prime}$), start=2]
\item $|{\nabla^m_{g_t} \Rm (g_t)}|, |{\nabla^m_{g'_t} \Rm (g'_t)}|, |\nabla^{m}_{g_t} \partial_t g_t|, |\nabla^{m}_{g'_t} \partial_t g'_t| \leq C $ on $M$ and $M'$ for all $t \in [0,T]$ and $m = 0, \ldots, 10$.
\end{enumerate}
\end{lemma}

\begin{proof}
The solution $(\chi_t)$ arises again via a limit argument with the help of Lemma~\ref{lem_existence_hmhf_noncompact_inj_positive}.
For this purpose we represent $(M, g_0)$ and $(M', g'_0)$ as a limit of Riemannian manifolds with positive injectivity radius.
The method used here can also be found in \cite{Chen:2006wia}.

Choose $p' \in M'$ and $p := \ov\chi (p') \in M$ and denote by $r := d_0 (p, \cdot)$ and $r' := d_0 ( p', \cdot )$ the distance functions to the respective basepoints.
Due to assumptions \ref{con_A.24_ip} and \ref{con_A.24_iip} we have $\injrad > c e^{- C'_1 r}$ on $M$ and $M'$ for some $c = c(C) > 0$ and $C'_1 = C'_1 (C) < \infty$.

Let $i \geq 1$. 
By mollification of  the functions $\zeta^*_i := \max \{ 0,  r - i - 1 \}$ and $\zeta^{\prime *}_i := \max \{ 0, C r' - i - 1 \}$ (for example by application of the heat flow for some uniform time and composition with a cutoff function),  we obtain approximations $\zeta_i \in C^\infty (M)$ and $\zeta'_i \in C^\infty (M')$ such that
\begin{enumerate}
\item $\zeta_i \equiv 0$ on $B(p, 0,  i)$ and $\zeta'_{i} \equiv 0$ on $B(p', 0,  C^{-1} i)$.
\item $\zeta_i > r -  i $ and $\zeta'_i > C r' - i$.
\item $\zeta_i \circ \ov\chi < \zeta'_i + 10$.
\item $|\nabla^m \zeta_i |, |\nabla^m \zeta'_i | < C'$ for all $m = 1, \ldots, 9$, for some $C'_2 = C'_2 (C) < \infty$.
\end{enumerate}
Set $g^{(i)}_t := \exp ( 2 C'_1 \zeta_i) g_t$ and $g^{\prime (i)}_t := \exp ( 2 C'_1 \zeta'_i) g'_t$.
By property (4), assumption (ii$\,^\prime$) of Lemma~\ref{lem_existence_hmhf_noncompact_inj_positive} holds  for $g_t$ and $g'_t$ replaced by $g^{(i )}_t$ and $g^{\prime (i )}_t$ and $C$ replaced by some constant $C'_3 = C'_3 (C) < \infty$.
Moreover, the injectivity radius on $(M, g^{(i)}_0)$ and $(M', g^{\prime (i)}_0)$ is uniformly bounded from below, by a constant that may depend on $i$.
By property (3), the map $\ov\chi$ is moreover $C'_4$-Lipschitz for some $C'_4 (C) < \infty$.

We can now use Lemma~\ref{lem_existence_hmhf_noncompact_inj_positive} to solve the harmonic map heat flow starting from $\ov\chi$ with the background metrics $(g^{\prime (i)}_t)_{t \in [0,T]}$ and $(g^{ (i)}_t)_{t \in [0,T]}$, on a time-interval of the form $[0, \min \{ \tau, T \}]$ for some $\tau = \tau (C, n) > 0$.
Similarly as in the proof of Lemma~\ref{lem_existence_hmhf_noncompact_inj_positive}, these solutions then subsequentially converge to the desired solution of the harmonic map heat flow with background metrics $(g'_t)_{t \in [0,T]}$ and $(g_t)_{t \in [0,T]}$.
\end{proof}

Using Lemmas~\ref{lem_hh_flow_C0_bound} and \ref{ex_hh_flow_version_1}, we can finally prove the main short-time existence result that is used in Section~\ref{sec_extend_comparison}.

\begin{proposition}[Short-time existence of harmonic map heat flow, general form] \label{prop_hh_flow_existence}
For every $n \geq 1$ there is a constant $\ov\eta_n > 0$ such that for any $0 < \eta_0 < \eta_1 < \ov\eta_n $ and $0 < \delta, C < \infty$ there is a constant $\tau = \tau (\eta_0, \eta_1, \delta, C, n) > 0$ such that the following holds.

Let $(g_t)_{t \in [0,T]}, (g'_t)_{t \in [0,T]}$ be smooth families of Riemannian metrics on $n$-dimensional manifolds $M, M'$ and consider a smooth map $\ov\chi : M' \to M$ such that the following holds for some $C' < \infty$.
\begin{enumerate}[label=(\roman*)]
\item \label{con_A.26_i} $(M, g_0)$ and $(M, g'_0)$ are complete.
\item \label{con_A.26_ii} $|{\Rm (g_t)}|, |{\Rm (g'_t)}| \leq C$ on $M$ and $M'$ for all $t \in [0,T]$.
\item \label{con_A.26_iii} $|{\nabla^m_{g_t} \Rm (g_t)}|, |{\nabla^m_{g'_t} \Rm (g'_t)}|, |\nabla_{g_t} \partial_t g_t|, |\nabla^{m}_{g'_t} \partial_t g'_t| \leq C' t^{-m/2} $ on $M$ and $M'$ for all $t \in (0,T]$ and $m = 0, \ldots, 10$.
\item \label{con_A.26_iv} $-\delta g_t \leq \partial_t g_t + 2 \Ric (g_t) \leq \delta g_t$ for all $t \in [0,T]$.
\item \label{con_A.26_v} $-\delta g'_t \leq \partial_t g'_t + 2 \Ric (g'_t) \leq \delta g'_t$ for all $t \in [0,T]$.
\item \label{con_A.26_vi} $\ov\chi$ is a diffeomorphism and
\[  \big| \big( \ov\chi^{-1} \big)^* g'_0 - g_0 \big| \leq \eta_0 . \]
\end{enumerate}
Then the harmonic map heat flow equation
\begin{equation*} \label{eq_hh_flow_existence_equation}
 \partial_t \chi_t = \triangle_{g'_t, g_t} \chi_t, \qquad \chi_0 = \ov\chi 
\end{equation*}
has a smooth solution on the time-interval $[0,T^*]$ for some $\min \{ \tau, T \} \leq T^* \leq T$ and for $h_t := (\chi_t^{-1})^* g'_t - g_t$ we have $|h_t| \leq \eta_1$ for all $t \in [0,T^*]$.
For all $t \in [0,T^*]$, the map $\chi_t$ is a diffeomorphism.
Moreover, if $|h_{T^*}| \leq \eta' < \eta_1$, then $T^* = T$.
\end{proposition}

\begin{proof}
Fix some sequence $\theta_i \to  0$.
By assumptions \ref{con_A.26_iii} and \ref{con_A.26_vi}, we can find a sequence $\eta_0^{(i)} \to \eta_0$ such that
\[  \big| \big( \ov\chi^{-1} \big)^* g'_{\theta_i} - g_{\theta_i} \big| \leq \eta_0^{(i)} . \]
Let $\tau = \tau (\frac12 (\eta_0 + \eta_1), \eta_1, \delta, C, n )$ be the constant from Lemma~\ref{lem_hh_flow_C0_bound}.
Fix some large $i$ and assume that $\eta_0^{(i)} \leq \frac12 (\eta_0 + \eta_1) $.
By Lemma~\ref{ex_hh_flow_version_1} and assumptions \ref{con_A.26_i}--\ref{con_A.26_iii}, \ref{con_A.26_vi}, we can solve the harmonic map heat flow equation
\begin{equation} \label{eq_RdT_start_later}
 \partial_t \chi^{(i)}_t = \triangle_{g'_t, g_t} \chi^{(i)}_t, \qquad \chi^{(i)}_{\theta_i} = \ov\chi .
\end{equation}
on a time-interval of the form $[\theta_i, T_i]$, where $\theta_i < T_i \leq T$.
Assume that $T_i \leq T$ is chosen maximally with the property that (\ref{eq_RdT_start_later}) has a solution on $[\theta_i, T_i)$ and that for $h^{(i)}_t := (\chi_t^{-1})^* g'_t - g_t$ we have $|h_t| \leq \eta_1$ for all $t \in [\theta_i, T_i)$.
By Lemmas~\ref{lem_hh_flow_C0_bound} and \ref{ex_hh_flow_version_1} we find that $T_i \geq \min \{ \theta_i + \tau, T \}$.
Note that Lemma~\ref{lem_hh_flow_C0_bound} requires a  bound on $|\partial_t h_t|$, which we can obtain from assumption \ref{con_A.26_iii} and Lemmas~\ref{lem_drift_local_coo} and \ref{lem_hmhf_loc_gradient}.
The same bound combined with assumption \ref{con_A.26_vi} also allows us to argue that $\chi^{(i)}_t$ is a diffeomorphism for all $t \in [\theta_i, T_i)$.

We now show that we have smooth convergence of the $(\chi^{(i)}_t)$ to a harmonic map heat flow $(\chi_t)$ on $[0,\min \{ \tau, T\})$, after passing to a subsequence.
Consider some point $q \in M'$ and set $p := \ov\chi (q)$.
By smoothness of $(g_t)$ and $(g'_t)$, we can find some constant $r_{q} > 0$ such that Lemmas~\ref{lem_drift_local_coo} and \ref{lem_hmhf_loc_gradient} are applicable at scale $r_q$ near $q$ and $p$ and at time $\theta_i$, for all $i$. 
So, after passing to a subsequence, $\chi^{i}_t$ converges locally in the $C^{10}$-sense in the neighborhood of every point $(q,0) \in M' \times \{ 0 \}$.
Moreover, by assumption \ref{con_A.26_iii}, for any $t_0 > 0$, Lemma~\ref{lem_drift_local_coo} is applicable at any point of $M \times [t_0, T_i]$ and $M' \times [t_0, T_i]$ at some uniform scale $r_0 = r_0 (t_0, C', n) > 0$, which is independent of $i$.
Iterating this fact, and using Lemma~\ref{lem_hmhf_loc_gradient}, yields local bounds on $\chi^{(i)}_t$ for any $t \in [0, \tau)$.
So after passing to a subsequence, the families of maps $(\chi^{(i)}_t)$ indeed converge to a solution of the harmonic map heat flow $(\chi_t)$ on $[0,\min \{ \tau, T\})$, in the $C^{10}$-sense.
Repeated application of Lemma~\ref{lem_hmhf_loc_gradient} yields higher derivative bounds and implies that the convergence is smooth, after passing to another subsequence.
The bounds in part \ref{ass_A.16_b} of Lemma~\ref{lem_hmhf_loc_gradient} imply that $\chi_0 = \ov\chi$.
The fact that $\chi_t$ is a diffeomorphism for all $t \in [0, \min \{ \tau, T \})$ follows from the fact that it is a limit of diffeomorphisms $\chi^{(i)}_t$ that are uniformly bilipschitz.

It remains to show the last assertion.
So assume by contradiction that $|h_{T^*}| \leq \eta' < \eta_1$, but $T^* < T$.
Then by Lemma~\ref{ex_hh_flow_version_1} we can extend the flow past time $T^*$ and by Lemma~\ref{lem_hh_flow_C0_bound} with $\eta_0 = \eta'$ we have $|h_t | \leq \eta_1$ for $t$ close to $T^*$, contradicting our choice of $T^*$.
\end{proof}

\subsection{Further results}
In the following, we will prove several analytical results that are needed in Section~\ref{sec_extend_comparison}.
The results will mostly build on the computations of Subsection~\ref{subsec_hmhf_main_equations}.

The following proposition provides a bound on the drift of a solution to the harmonic map heat flow, whenever the associated perturbation $h$ is small.

\begin{proposition}[Drift control]
\label{prop_drift_control}
For any $n \geq 1$, $\delta > 0$ and $A < \infty$ there is a constant $\eta = \eta (\delta, A, n) > 0$ such that the following holds.

Let $r > 0$ and $T \leq Ar^2$ and consider Ricci flows $(g_t)_{t \in [0, T]}$ and $(g'_t)_{t \in [0, T]}$ on $n$-dimensional manifolds $M$ and $M'$.
Let $(\phi_t)_{t \in [0, T]}$, $\phi_t : M \to M'$, be a smooth family of diffeomorphisms onto their images whose inverses $\phi_t^{-1} : \phi_t (M) \to M'$ satisfy the harmonic map heat flow equation
\[ \partial_t \phi^{-1}_t = \triangle_{g'_t, g_t} \phi^{-1}_t. \]
Let $x \in M$ and assume that for $x' := \phi_0 (x)$:
\begin{enumerate}[label=(\roman*)]
\item \label{con_A.29_i} $B(x, 0, r)$ is relatively compact in $M$.
\item \label{con_A.29_ii} $|\nabla^m {\Rm}| \leq A r^{-2-m} $ on $B(x, 0, r) \times [0, T]$ for all $m = 0,1,2,3$.
\item \label{con_A.29_iii} $|{\Rm}| \leq A r^{-2}$ on $B(x', 0, r) \times [0, T]$.
\item \label{con_A.29_iv} $- \eta g'_t \leq h_t = \phi^*_t g_t - g'_t \leq \eta g'_t$ for all $t \in [0, T]$.
\end{enumerate}
Then for all $t \in [0,T]$
\[ d_0 (\phi_t (x), \phi_0 (x)) < \delta r. \]
\end{proposition}

We note that this proposition is similar to Lemma~\ref{lem_drift_local_coo}.
In fact, this lemma could be used in lieu of Proposition~\ref{prop_drift_control}.
Nevertheless, we have included this proposition since its proof is somewhat shorter and does not use local coordinates.

\begin{proof}
By parabolic rescaling, we may assume without loss of generality that $r = 1$.

Using Lemma~\ref{lem_loc_gradient_estimate}, we obtain that if $\eta$ is sufficiently small depending on $A, n$, then for all $t \in [0, T]$ we have
\begin{equation*} \label{eq_nab_h_drift_proof}
 |\nabla h_t| (x) \leq C_1 \eta t^{-1/2}, 
\end{equation*}
where $C_1 = C_1(A,n) < \infty$.
So by (\ref{eq_X_is_dt_chi}) and (\ref{eq_def_X_appendix}) we obtain
\[ | \partial_t \phi^{-1}_t (\phi_t (x)) | = |X_{g_t} (g_t + h_t)| \leq C_2 \eta t^{-1/2}, \]
where $C_2 = C_2 (A,n) < \infty$.
As $\partial_t \phi_t = - d\phi_t ( \partial_t \phi^{-1}_t \circ \phi_t )$, we obtain that 
\[ | \partial_t \phi_t (x) | \leq C_3 \eta t^{-1/2}, \]
where $C_3 = C_3 (A,n) < \infty$.
Integrating this bound, and taking into account the distance distortion in $M'$ via assumption \ref{con_A.29_iii}, we obtain that $d_0 (\phi_t (x), \phi_0 (x)) < C_4 \eta t^{1/2}$ for some $C_4 = C_4 (A, n) < \infty$, as long as $C_4 \eta t^{1/2} < 1$.
So the proposition follows if $\eta \leq  C_{4}^{-1} \min \{ \delta, (2A)^{-1/2} \}$.
\end{proof}

Next, we derive short-time bounds for solutions to the Ricci-DeTurck equation.
To do this, we first establish the following barrier-type estimate.

\begin{lemma} \label{lem_promote_u_bound}
For any $n \geq 1$ and $A < \infty$ there is a constant $C = C(A, n) < \infty$ such that the following holds.

Let $r > 0$.
Consider smooth families of metrics $(g_t)_{t \in [0, r^2]}$, $(g_t+h_t)_{t \in [0,r^2]}$ on an $n$-dimensional manifold $M$, a point $x \in M$ and a smooth function $u \in C^\infty (B(x,r) \times [0,r^2])$ such that:
\begin{enumerate}[label=(\roman*)]
\item \label{con_A.30_i} $B(x,0,r)$ is relatively compact in $M$.
\item  \label{con_A.30_ii} $\frac12 g_t \leq h_t \leq 2 g_t$ on $B(x,0,r)$ for all $t \in [0,r^2]$.
\item \label{con_A.30_iii} $u$ satisfies the inequality
\begin{equation} \label{eq_u_equation_h_g}
 \partial_t u \leq (g_t + h_t)^{ij} \nabla^{2 , g_t}_{ij} u.  
\end{equation}
\item \label{con_A.30_iv} $|{\Rm (g_t)}|_{g_t}, |\partial_t g_t |_{g_t}  \leq A r^{-2}$ on $B(x,0,r)$ for all $t \in [0, r^2]$.
\item \label{con_A.30_v} $|\nabla^{g_t} \partial_t g_t |_{g_t} \leq A r^{-2} t^{-1/2}$ on $B(x,0,r)$ for all $t \in [0, r^2]$.
\item \label{con_A.30_vi} $|u| \leq 1$ on $B(x,0,r) \times [0,r^2]$.
\end{enumerate}
Then for all $t \in [0,r^2]$ we have
\begin{equation} \label{eq_u_promote_bound}
 u(x,t) \leq C  t r^{-2} + \sup_{B(x,0,r)} u_0. 
\end{equation}
\end{lemma}

\begin{proof}
Without loss of generality we may assume that $r = 1$.
Fix a function $f : [0, 1] \to [0, \infty)$ such that $f \equiv 0$ on $[0, \frac14]$ and $f (\frac12) > 1$ and $f' \geq 0$ everywhere.
Set $w(y) := f (d_{0} (x,y))$ for all $y \in B(x,0,1)$.
By Hessian comparison and assumptions \ref{con_A.30_i}, \ref{con_A.30_ii} there is a constant $C'_1 = C'_1 (n, A) < \infty$ such that $\nabla^{2,g_0} w \leq C'_1 g_0$ on $B(x,0,1)$ in the barrier sense.
By a local smoothing procedure (see for example \cite{Greene:1972ke}) we can construct a smooth, non-negative $w' \in C^\infty (B(x,0,1))$ such that for some $C'_2 = C'_2 (n,A) < \infty$ we have $|\nabla^{g_0} w'| \leq C'_2$, $\nabla^{2,g_0} w' \leq C'_2 g_0$, $w' \equiv 0$ on $B(x, 0, \frac12)$ and $w' > 1$ on $B(x,0,1) \setminus B(x,0,\frac12)$.

For any vector $v \in T_y M$, $y \in B(x,0,1)$, we have by assumptions \ref{con_A.30_ii}, \ref{con_A.30_iii}
\begin{multline*}
 \bigg| \frac{d}{dt} \nabla^{2,g_t}_{ v,v} w' \bigg| = \bigg| dw' \Big( \Big( \frac{d}{dt} \nabla^{g_t} \Big)(v,v) \Big) \bigg| \\
 = |d w'|_{g_t} \big| g_t^{-1} \big( (\nabla^{g_t}_v \partial_t g_t)(v, \cdot) - \tfrac12 (\nabla^{g_t}_\cdot \partial_t g_t)(v,v) \big) \big|_{g_t} \leq C'_3 t^{-1/2}
\end{multline*}
for some $C'_3 = C'_3 (n,A) < \infty$.
Integrating this bound over $t$ and tracing in $v$, we conclude, using assumption \ref{con_A.30_ii}, that there is a constant $C'_4 = C'_4 (n, A) < \infty$ such that
\begin{equation} \label{eq_laplace_w_bound}
 (g_t + h_t)^{ij} \nabla^{2, g_t}_{ij} w' < C'_4 \qquad \text{on} \qquad B(x,0,1) \quad \text{for all} \quad t \in [0,1]. 
\end{equation}

We now show that for any $\eps > 0$ we have
\begin{equation} \label{eq_max_principle_u_claim}
 u_t < w' + C'_4 t + \sup_{B(x,0,1)} u_0 + \eps 
\end{equation}
on $B(x,0,1)$ for all $t \in [0,1]$.
Evaluating this bound at $x$ and letting $\eps \to 0$ will then imply (\ref{eq_u_promote_bound}).

Note that (\ref{eq_max_principle_u_claim}) trivially holds for $t = 0$ and for $t > 0$ it can only fail on $B(x, 0, \frac12)$ due to assumption \ref{con_A.30_vi},  since $w' > 1$ on $B(x,0,1) \setminus B(x,0,\frac12)$.
Assume by contradiction that (\ref{eq_max_principle_u_claim}) fails for some $t \in [0, 1]$.
As $B(x, 0, \frac12)$ is relatively compact in $M$, we may assume that $t$ is chosen minimal with this property.
Then $t > 0$ and there is a point $y \in B(x, 0, \frac12)$ at which equality holds in (\ref{eq_max_principle_u_claim}).
It follows that at $y$ we have, using (\ref{eq_laplace_w_bound})
\[ \partial_t u_t - (g_t + h_t)^{ij} \nabla^{2, g_t}_{ij} u_t \geq C'_4 - (g_t + h_t)^{ij} \nabla^{2, g_t}_{ij} w' > 0. \]
This, however, contradicts (\ref{eq_u_equation_h_g}).

Therefore (\ref{eq_max_principle_u_claim}) holds on $B(x,0,1)$ for all $t \in [0,1]$ and $\eps > 0$, which finishes the proof. 
\end{proof}

Using Lemma~\ref{lem_promote_u_bound}, we can prove the following short-time bounds for the Ricci-DeTurck flow.

\begin{proposition}[Short-time bounds for Ricci-DeTurck flow] \label{prop_promote_RdT_forward_locally}
For any $n \geq 1$ there is a constant $\eta_0 = \eta_0 (n) > 0$ and for any $A < \infty$ there is a constant $C= C(A, n) < \infty$ such that the following holds.

Let $(g_t)_{t \in [0, r^2]}$, $r > 0$, be a Ricci flow on an $n$-dimensional manifold $M$ and $(h_t)_{t \in [0,r^2]}$ a Ricci-DeTurck flow with background metric $(g_t)_{t \in [0, r^2]}$.
Let $x \in M$ be a point and assume that
\begin{enumerate}[label=(\roman*)]
\item $B(x,0,r)$ is relatively compact in $M$.
\item $|{\Rm (g_t)}| \leq A r^{-2}$ on $B(x,0,r)$ for all $t \in [0, r^2]$.
\item $|h_t| \leq \eta \leq \eta_0$ on $B(x,0,r)$ for all $t \in [0, r^2]$.
\end{enumerate}
Then for all $t \in [0,r^2]$ we have
\[ |h(x,t)|^2 \leq C  \eta^2 t r^{-2} + \sup_{B(x,0,r)} |h_0|^2 . \]

\end{proposition}

\begin{proof}
By (\ref{eq_RdT_norm_evolution}), if $\eta$ is smaller than some dimensional constant, then
\[ \partial_t |h_t|^2 \leq (g_t + h_t)^{ij} \nabla_{ij}^2 |h_t|^2 + C'_1 |h_t|^2 \]
for some constant $C'_1 = C'_1(A,n) < \infty$.
So the proposition follows from Lemma~\ref{lem_promote_u_bound} by setting $u_t := \eta^{-2} e^{- C'_1 t} |h_t|^2$.
Note that assumption \ref{con_A.30_ii} in Lemma~\ref{lem_promote_u_bound} is guaranteed if we choose $\eta_0$ sufficiently small and assumption \ref{con_A.30_v} follows using Shi's estimates.
\end{proof}

Lastly, we prove that solutions to the Ricci-DeTurck perturbation equation remain small in a parabolic neighborhood if they are small on a larger ball at time $0$.
The following proposition also holds for perturbations that arise from almost Ricci flows, as discussed in Subsection~\ref{subsec_hmhf_main_equations}.

\begin{proposition}[Smallness of $h$ at time 0 implies smallness of $h$ at later times]
\label{prop_distortion_stays_small}
For any $n \geq 1$ there is a constant $\eta = \eta(n) > 0$ such that for any $\eps > 0$ and $A < \infty$ there is a constant $0 < \delta = \delta ( \eps, A, n) < 1$ such that the following holds.

Let $r > 0$.
Consider smooth families of metrics $(g_t)_{t \in [0,r^2]}$ and $(g'_t)_{t \in [0,r^2]}$ on $n$-dimensional manifolds $M$ and $M'$, respectively, as well as a smooth family of diffeomorphisms $(\chi_t)_{t \in [0,r^2]}$ between $M'$ and $M$ that satisfies the harmonic map heat flow equation
\[ \partial_t \chi_t = \triangle_{g'_t, g_t} \chi_t. \]
Set $h_t := (\chi^{-1}_t)^* g'_t - g_t$ as in (\ref{eq_associated_perturbation_appendix}) and assume that for some $x \in M$ the following bounds hold on $B(x, 0, \delta^{-1} r)$ for all $t \in [0,r^2]$:
\begin{enumerate}[label=(\roman*)]
\item \label{con_A.36_i} $B(x, 0, \delta^{-1} r)$ is relatively compact in $M$.
\item \label{con_A.36_ii} $|{\Rm}(g_t)| \leq A r^{-2}$.
\item \label{con_A.36_iii} $|\nabla^{g_t} \partial_t g_t| \leq A r^{-1}$.
\item \label{con_A.36_iv} $- \delta g_t \leq \partial_t g_t + 2 \Ric(g_t) \leq \delta g_t$.
\item \label{con_A.36_v} $- \delta g'_t \leq \partial_t g'_t + 2 \Ric(g'_t) \leq \delta g'_t$.
\item \label{con_A.36_vi} $|h_t| \leq \eta$.
\item \label{con_A.36_vii} $|h_0| \leq \delta $.
\end{enumerate}
Then $|h(x,t)| < \eps$ for all $t \in [0,r^2]$.
\end{proposition}

\begin{proof}
By parabolic rescaling we may assume $r = 1$.
The constant $\delta$ will be chosen in the course of the proof.
In the following, we will always assume that $\delta \leq 1$.

By assumptions \ref{con_A.36_ii}, \ref{con_A.36_iv}--\ref{con_A.36_vi} and (\ref{eq_RdT_norm_evolution}), and assuming that $\eta$ is smaller than some dimensional constant, we can find a constant $C_1 = C_1 (A, n) < \infty$ such that $h_t$ satisfies the evolution inequality
\begin{equation*} \label{eq_h_PDE_condensed}
 \partial_t |h_t|^2 \leq   (g_t+h_t)^{ij} \nabla_{ij}^2 |h_t|^2  + C_1 \delta   + C_1   |h_t|^2 . 
\end{equation*}
on $B(x, 0, \delta^{-1} ) \times [0,1]$.
So
\begin{equation} \label{eq_u_def_appendix}
 u_t := e^{-C_1 t} \big( |h_t|^2 +  \delta \big) 
\end{equation}
satisfies the evolution inequality
\[ \partial_t u_t \leq   (g_t+h_t)^{ij} \nabla_{ij}^2 u. \]

We will now derive a bound on $u$ by an argument that is analogous to the proof of Lemma~\ref{lem_promote_u_bound}.
Fix a non-decreasing function $f : [0, 1] \to [0, \infty)$ such that $f \equiv 0$ on $[0, \frac14 ]$ and $f(\frac12) > 1$.
Set $w (y) := f (\delta \cdot d (x, y) )$.
By Hessian comparison and assumption \ref{con_A.36_ii} we obtain that at any $y \in B(x, 0, \delta^{-1})$ in the barrier sense, for $d := d_{0} (x,y)$,
\[ \nabla^{2, g_0} w \leq C_3 \bigg( \delta^2 f'' ( \delta \cdot d ) + C_2 \frac{\cosh (C_{2} d)} {\sinh (C_2 d)} \cdot \delta f' ( \delta \cdot d ) \bigg) g_0 \leq C_4 \cdot \delta, \]
where $C_i = C_i (A, n) < \infty$ for $i = 2,3,4$.

As $|\nabla^{g_0} w | \leq C_{5 } \delta$ for some $C_{5} = C_{5} (A, n) < \infty$, we can argue similarly as in the proof of Lemma~\ref{lem_promote_u_bound}, using assumption \ref{con_A.36_iii}, that there is a smoothing $w' \in C^\infty (B (x, 0, \delta^{-1}))$ of $w$ such that $w' > 1$ on $B(x, 0,  \delta^{-1}) \setminus B(x, 0, \frac12 \delta^{-1})$ and
\[  (g_t + h_t)^{ij} \nabla^{2, g_t}_{ij} w' < C_{6} \cdot \delta \qquad \text{on} \qquad B(x,0,\delta^{-1}) \quad \text{for all} \quad t \in [0,1],  \]
where $C_{6} = C_{6} (A, n) < \infty$.
Compare this inequality with (\ref{eq_laplace_w_bound}).
We can now argue similarly as for (\ref{eq_max_principle_u_claim}) to show that for all $\eps' > 0$ and $t \in [0,1]$
\[  u_t < w' + C_{6} \delta \cdot t + \sup_{B(x,0,\delta^{-1})} u_0 + \eps'.  \]

So by assumption \ref{con_A.36_vii} and letting $\eps' \to 0$, we obtain that for all $t \in [0,1]$
\[ u (x,t) \leq C_5 \delta + \delta + C_1 \delta \leq C_6 \delta. \]
The proposition now follows using (\ref{eq_u_def_appendix}) if $\delta$ is chosen small depending on $A, n$.
\end{proof}

\section{Properties of Bryant solitons} \label{appx_Bryant_properties}
In this appendix we discuss properties of the (normalized) Bryant soliton $(M_{\Bry}, \lb g_{\Bry})$ that are needed in this paper.
In the following we denote by $x_{\Bry} \in M_{\Bry}$ the tip of the Bryant soliton, i.e. the center of rotational symmetry, and denote by $\si :=d_{g_{\Bry}}(\cdot,x_{\Bry})$ the distance function from the tip.

\begin{lemma}[Properties of the Bryant soliton] \label{lem_properties_Bryant}
\label{lem_bryant_geometry}
There is a rotationally  symmetric potential function $f \in C^\infty (M_{\Bry})$ such that 
\begin{align}
\label{eq_soliton_equation} \Ric + \nabla^2 f &\equiv 0 \\
\label{eqn_soliton_conserved} R+|\nabla f|^2 & \equiv R(x_{\Bry}) \\
\label{eqn_soliton_dr} dR &=2\Ric(\nabla f,\cdot)
\end{align}
Moreover, there is a constant $C_B < \infty$ such that the following holds:
If $\si> C_B$, then
\begin{align}
C_B^{-1} \sigma^{-1} &< R < C_B \sigma^{-1} \label{eq_R_asymptotic_Bryant} \\
\Ric &> C_B^{-1}  \si^{-2} g_{\Bry} \quad\text{as quadratic forms}\label{eq_Ric_Bry_lower_bound} \\
- \partial_\sigma R &> C_B^{-1} \sigma^{-2} \label{eq_lower_partial_R_bound} \\
|\nabla f| &< C_B  \label{eq_nab_f_growth} \\
|\nabla R|, |\nabla^2 \Rm |, |\nabla^3 \Rm | &< C_B \label{eq_higher_der_bound_Bry}
\end{align}
The metric $g_{\Bry}$ is a warped product of the form $g_{\Bry} = d\sigma^2 + w^2 (\sigma) g_{S^2}$ such that for $\sigma > C_B$
\begin{equation} \label{eq_w_bound_sqrt}
 C_B^{-1} \sqrt{\sigma} < w (\sigma) < C_B \sqrt{\sigma}. 
\end{equation}
Moreover, for any $\sigma_0 > C_B$  if we consider the normalized function and parameter
\[ \ov{w} := \frac{w}{w(\sigma_0)}, \qquad \ov{\sigma} := \frac{\sigma - \sigma_0}{w(\sigma_0)}, \]
 and if we express $\ov{w}$ in terms of $\ov\sigma$, then for all $\ov\sigma \in [-1, 1]$
\begin{equation} \label{eq_td_w_bounds}
 |\ov{w} (\ov\sigma ) - 1 |,  \quad \bigg| \frac{d \ov{w}}{d \ov\sigma} \bigg| , \quad \bigg| \frac{d^2 \ov{w}}{d \ov\sigma^2} \bigg| < C_B \sigma_0^{-1/2}. 
\end{equation}
In other words, $w^{-2} (\sigma_0) g_{\Bry}$ is geometrically $C^2$-close to a piece of a round cylinder on $M_{\Bry} (\sigma_0 -  w(\sigma_0), \sigma_0 + w(\sigma_0))$.
\end{lemma}

\begin{proof}
Identities (\ref{eqn_soliton_conserved}) and (\ref{eqn_soliton_dr}) are standard bounds for a complete gradient steady soliton with bounded curvature, where $f$ satisfies the steady gradient soliton potential equation (\ref{eq_soliton_equation}).
For (\ref{eqn_soliton_conserved}) observe that the left-hand side is constant and $|\nabla f| = 0$ at $x_{\Bry}$.
Identity (\ref{eq_nab_f_growth}) is a direct consequence of (\ref{eqn_soliton_conserved}).

By \cite[Theorem~1]{Bryant:gVzfj4nx} we know that $w \sim c_1\sqrt{\sigma}$ for large $\sigma$ and that the radial and orbital sectional curvatures, $K_R, K_O$ behave like
\[ K_R \sim c_2 \sigma^{-2}, \qquad K_O \sim c_3 \sigma^{-1}. \]
Here $c_1, c_2, c_3$ are positive constants that depend on the normalization of $g_{\Bry}$.
The bounds (\ref{eq_R_asymptotic_Bryant}), (\ref{eq_w_bound_sqrt}) and (\ref{eq_Ric_Bry_lower_bound}) follow immediately.
It also follows that $R$ decays to $0$ at infinity and, therefore, by (\ref{eqn_soliton_conserved}) we have $|\nabla f|^2 \to R(x_{\Bry})$ at infinity.
Combining this with (\ref{eqn_soliton_dr}) yields (\ref{eq_lower_partial_R_bound}).
The bound (\ref{eq_higher_der_bound_Bry}) follows by Shi's estimates.

Lastly, by the Jacobi equation, we obtain
\[ w'' = - K_R w \sim - c_1 c_3 \sigma^{-3/2}. \]
Integrating this bound and using (\ref{eq_w_bound_sqrt}), we obtain that $w' \lb \sim \lb 2 c_1 c_3 \sigma^{-1/2}$.
Rescaling both bounds by $w(\sigma_0)$ implies the bounds on the second and third term in (\ref{eq_td_w_bounds}).
The first bound follows by integration over $[-1,1]$ and observing that $\ov{w} (0) = 1$.
\end{proof}

\section{Properties of \texorpdfstring{$\kappa$}{{\textbackslash}kappa}-solutions}
\label{app_kappa_solution_properties}
In this section we discuss properties of $\kappa$-solutions (Definition~\ref{def_kappa_solution}) that are needed in the paper.  We remind the reader that we are using the curvature scale function $\rho$ from Definition~\ref{def_curvature_scale}.

\begin{lemma}\label{lem_kappa_solution_properties_appendix}
3-dimensional $\kappa$-solutions have the following properties:
\begin{enumerate}[label=(\alph*)]
\item \label{ass_C.1_a} There is a  $ \kappa_0>0$ such that every 3-dimensional $\kappa$-solution $(M^3, (g_t)_{t \in (-\infty,0]})$ is either a $\kappa_0$-solution or  a shrinking round spherical space form.
\item \label{ass_C.1_b} (Compactness)  For any $\ul\kappa>0$, the collection of pointed $\kappa$-solutions $(M^3, \lb (g_t)_{t \in (-\infty,0]}, \lb x)$ with $\kappa\geq\ul\kappa$ and $R(x,0)=1$ is compact in the pointed smooth topology.
\item \label{ass_C.1_c} For every $A<\infty$ there is a constant $C=C(A)<\infty$ such that for any point $(x,t)$ in a $\kappa$-solution $(M^3, (g_t)_{t \in (-\infty,0]})$, we have
$$
C^{-1}\rho(x,t)\leq \rho\leq C\rho(x,t)
$$
on the parabolic neighborhood $P(x,t,A\rho(x,t))$.
\item \label{ass_C.1_d} For every $k,l$, $A<\infty$ there is a constant $C=C(A,k,l)<\infty$ such that for any point $(x,t)$ in a $\kappa$-solution $(M^3, (g_t)_{t \in (-\infty,0]})$ we have $$|\D^k_t\nabla^l \Rm|\leq C_{k,l} \rho^{-2-2k-l} (x,t) $$ on the parabolic neighborhood $P(x,t,A\rho(x,t))$. 
\item \label{ass_C.1_e} $\D_tR\geq 0$ and $-C\leq \D_tR^{-1}\leq 0$ on every $\kappa$-solution for some universal constant $C < \infty$.
\item \label{ass_C.1_f} The shrinking round cylinder is the only 3-dimensional $\kappa$-solution with more than one end.
\end{enumerate}
\end{lemma}
\begin{proof}
For assertions \ref{ass_C.1_a}, \ref{ass_C.1_b}, \ref{ass_C.1_f}, and the first part of \ref{ass_C.1_e}, see \cite[Section 11]{Perelman:2002um} or \cite[Sections 38-51]{Kleiner:2008fh}.  Assertions \ref{ass_C.1_c}, \ref{ass_C.1_d}, and the second part of \ref{ass_C.1_e} follow immediately from the compactness assertion \ref{ass_C.1_b}.
\end{proof}

The following lemma  is a variation on the geometric definition of canonical neighborhoods used by Perelman in \cite[Subsections 1.5, 4.1]{Perelman:2003tka}.

\begin{lemma} \label{lem_neck_or_cap_kappa_solution}
For every $\delta > 0$ there is a constant $C_0=C_0(\delta) < \infty$ such that the following holds.

If $(M, (g_t)_{t \in (-\infty,0]})$ is a 3-dimensional $\kappa$-solution and $(x,t) \in M \times (-\infty, 0]$, then one of the following holds:
\begin{enumerate}[label=(\alph*)]
\item \label{ass_C.2_a} The point $(x,t)$ is the center of a $\delta$-neck at scale $\rho(x,t)$.
\item \label{ass_C.2_b} There is a compact, connected domain $V \subset M$ with connected (possibly empty) boundary such that the following holds:
\begin{enumerate}[label=(\arabic*)]
\item  $B(x,t,\delta^{-1} \rho(x)) \subset V$.
\item $\rho (y_1,t) <  C_0 \rho(y_2,t)$ for all $y_1, y_2 \in V$.
\item $\diam_t V < C_0 \rho(x,t)$.
\item If $\D V \neq \emptyset$, then: 
	\ben
	\item $\partial V$ is a central $2$-sphere of a $\delta$-neck at time $t$.
	\item Either $V$ is  a 3-disk and $(M,g_t)$ has  strictly positive sectional curvature, or $V$ is diffeomorphic  to a twisted interval bundle over $\R P^2$ and $(M, g_t)$ is a $\Z_2$-quotient of the shrinking round cylinder.
	\item Any two points $z_1, z_2 \in \partial V$ can be connected by a continuous path inside $\partial V$ whose length is less than \[ \min \{ d_t (z_1, x),  d_t (x, z_2) \} - 110 \rho (x,t). \]
	\een
\end{enumerate}
\end{enumerate}
\end{lemma}
\begin{proof}
Suppose the lemma were false for some $\delta>0$.  
Then there exists a sequence $(M_k, (g_{k,t})_{t \in (-\infty,0]}, x_k)$ of pointed $\kappa_{k}$-solutions and a sequence $C_k\ra\infty$ such the conclusion of the lemma fails at time $0$ for all $k$, where $C_0$ is replaced by $C_k$.  Without loss of generality, we may assume that $\rho(x_k,0)=1$  for all $k$.  
Since the conclusion holds for shrinking round spherical space forms when $C_k$ is sufficiently large, we may assume by assertion \ref{ass_C.1_a} of Lemma~\ref{lem_kappa_solution_properties_appendix} that $\kappa_k \geq \kappa_0 > 0$ for large $k$.
So by assertion \ref{ass_C.1_b} of Lemma~\ref{lem_kappa_solution_properties_appendix}, we may assume that, after passing to a subsequence, the sequence $\{(M_k , (g_{k,t})_{t \in (-\infty,0]},x_k)\}$ converges to a pointed $\kappa$-solution $(M_\infty, (g_{\infty, t})_{t \in (-\infty, 0]},x_\infty)$ in the pointed smooth topology.  
Note that $M_\infty$ must non-compact, and $(M_\infty, \lb (g_t)_{t \in (-\infty, 0]})$ cannot be a shrinking round cylinder, since otherwise assertion \ref{ass_C.2_a} or \ref{ass_C.2_b} will hold for large $k$, contradicting our assumptions.  
Now by \cite[Lemma 59.1]{Kleiner:2008fh}, its proof, and the discussing preceding the statement of that lemma,  there is a compact manifold with boundary $V_\infty\subset M_\infty$ such that $B(x_\infty, 0, 2 \delta^{-1} \rho(x_\infty, 0)) \subset V_\infty$, the boundary $\D V_\infty$ is the central $2$-sphere of a $\frac12 \delta$-neck at time $t$, and either $V_\infty$ is diffeomorphic to a $3$-disk and $(M_\infty , (g_t)_{t \in (-\infty, 0]})$ has strictly positive sectional curvature, or $V_\infty$ is diffeomorphic to a twisted interval bundle over $\R P^2$ and $(M_\infty, (g_t)_{t \in (-\infty, 0]})$ is isometric to a $\Z_2$-quotient of a shrinking round cylinder.  Now for large $k$, the domain $V_\infty$ yields a compact domain with boundary satisfying assertions (i)-(iii).  This is a contradiction.  
\end{proof}

\begin{proposition} \label{prop_dtR_0_Bryant}
Let $(M, (g_t)_{t \in (- \infty, 0]} )$ be $3$-dimensional $\kappa$-solution.
If  $\partial_t R (p, 0) = 0$ for some $p \in M$, then modulo parabolic rescaling there is a pointed isometry of Ricci flows 
$$
(M, (g_t)_{t \in (- \infty, 0]},p )\ra (M_{\Bry},(g_{\Bry,t})_{t \in (- \infty, 0]},x_{\Bry} )\,.
$$
\end{proposition}
\begin{proof}
The fact that $(M, (g_t)_{t \in (- \infty, 0]} )$ is a steady gradient soliton was shown by Hamilton in \cite{hamilton_eternal_solutions}, where he analyzed the equality case of his matrix Harnack inequality for the Ricci flow (see \cite{Hamilton:1993em}).  Below we have included an alternate proof of this that incorporates several simplifications.   Brendle showed that up to homothety the Bryant soliton is the only $\kappa$-solution that is a gradient steady soliton (see \cite{Brendle:2013dd}).  Finally, for gradient steady solitons we have
$$
\D_tR=dR(\nabla f)=2\Ric(\nabla f,\nabla f)\,,
$$
where $f$ is the soliton potential.  Since $\Ric>0$ on the Bryant soliton, we have $\nabla f(p)=0$.   Because $x_{\Bry}$ is the unique critical point of the soliton potential of $g_{\Bry}$, this forces the homothety $(M,g_0) \ra (M_{\Bry},g_{\Bry}) $ to map $p $ to $x_{\Bry}$. 
\end{proof}

\bigskip

In the remainder of this appendix, we will give a simplified proof of the first part of Proposition~\ref{prop_dtR_0_Bryant}, which was shown by Hamilton in \cite{hamilton_eternal_solutions}.  The proof is based on his matrix Harnack inequality 
(see \cite{Hamilton:1993em}) and Brendle's strong maximum principle in vector bundles (see \cite[sec 2]{Brendle:2008de}).   The reader may also consult a more general treatment of Hamilton's Harnack inequality due to Brendle (see \cite{Brendle:2009fd}), as we will mainly rely on the terminology  developed in this work.  As a preparation we briefly recall the main ideas of Hamilton's proof.
The bound $\partial_t R \geq 0$ follows from the following theorem after passing to the limit $T \to \infty$.

\begin{theorem} \label{Thm:dtRHamilton}
Let $(M, (g_t)_{t \in (- T, 0]} )$ be a $3$-dimensional Ricci flow with complete time-slices and bounded, non-negative sectional curvature.
Then
\begin{equation} \label{eq:dtRHamilton}
\partial_t R \geq - \frac{R}{T+t}.
\end{equation}
\end{theorem}

The proof of this bound follows from the following matrix Harnack estimate:
Consider the bundle $E = TM \oplus \Lambda_2 TM$ over $M$.
We introduce the following (time-dependent) generalized curvature quantity $S \in \Sym_2 E^*$:
\begin{multline*}
 S_t ( (x, u_1 \wedge u_2 ), (y, v_1 \wedge v_2) ) = W(x, y) + P (u_1 \wedge u_2, y) + P( v_1 \wedge v_2, x) \\ + R( u_1, u_2, v_2, v_1),
\end{multline*}
where
\begin{multline*}
 W (x,y) = (\triangle \Ric) (x, y) - \frac12 \nabla^2_{x, y} R + 2 \sum_{i, j= 1}^3 R (x, e_i, e_j, y) \Ric (e_i, e_j) \\
 - \sum_{i = 1}^3 \Ric( x, e_i ) \Ric( e_i, y) .
\end{multline*}
and
\[ P (u_1, u_2, y) = (\nabla_{u_1} \Ric) (u_2, y) - (\nabla_{u_2} \Ric) (u_1, y). \]
Hamilton (see also \cite{Brendle:2009fd} for the terminology used here) observed that this generalized curvature quantity satisfies an evolution equation that is similar to the evolution equation for the curvature tensor:
\begin{equation*} \label{eq:evolS}
 \td{D}_{\partial_t} S = \td\triangle S + Q(S).
\end{equation*}
Here $\td\triangle = \sum_{i = 1}^3 \td\nabla_{e_i} \td\nabla_{e_i}$ is the connection Laplacian on $\Sym_2 E^*$ with respect to the connection $\td\nabla$ that is induced by the following connection on $E$:
\[ \td{\nabla}_z ( X, \alpha )  = ( \nabla_z X, \nabla_z \alpha + \Ric(z) \wedge X ) \]
and
\begin{multline*}
 \td{D}_{\partial_t} (X, U_1 \wedge U_2) = \big({ ( \partial_t X + \Ric(X) - \tfrac12 (U_1 \wedge U_2) ( \nabla R, \cdot ), \partial_t (U_1 \wedge U_2)} \\ {+ \Ric (U_1) \wedge U_2 + U_1 \wedge \Ric (U_2) }\big). 
\end{multline*}
($X, U_1, U_2$ and $\alpha$ denote time-dependent local sections of $TM$ and $\Lambda_2 TM$, respectively.)
The quadratic part $Q(S)$ is non-negative definite, whenever $S$ is non-negative definite.
By a more general approach, which takes into account the case in which $S$ is indefinite, but bounded from below, Hamilton deduces the following lower bound for the quadratic form $S$:
For all $(x, \alpha) \in E$ we have
\begin{equation} \label{eq_S_bounded_1_T}
 S ( (x, \alpha), (x, \alpha) ) \geq - \frac1{2(T-t)} \Ric (x,x). 
\end{equation}
This implies that
\[ W (x, x) \geq - \frac1{2(T-t)} \Ric (x,x). \]
Tracing this equation in $x$ yields
\begin{equation} \label{eq_trace_W}
 \frac12 \partial_t R = \triangle R - \frac12 \triangle R + |{\Ric}|^2 = \tr W \geq - \frac{1}{2(T-t)} R, 
\end{equation}
which implies (\ref{eq:dtRHamilton}).

We can now present the proof of Proposition~\ref{prop_dtR_0_Bryant}.

\begin{proof}[Proof of Proposition~\ref{prop_dtR_0_Bryant}]
It remains to consider the case in which $\partial_t R(p, \lb0) \lb = \lb 0$ for some $p \in M$.
We first argue that the all sectional curvatures on $M \times (-\infty, 0]$ are positive.
If not, then by a standard strong maximum principle argument, this implies that the flow locally splits off an $\R$-factor.
It follows that the universal covering flow is homothetic to the round shrinking cylinder, in contradiction to $\partial_t R(p, 0) = 0$.

Letting $T \to \infty$ in (\ref{eq_S_bounded_1_T}) we obtain that $S$, and hence $W$ are non-negative definite everywhere on $M \times (- \infty, 0]$.
As $\partial_t R(p, 0) = 0$, we obtain from (\ref{eq_trace_W}) that $W(p,0) = 0$.
So $S(p,0)$ has nullity of at least 3.
On the  other hand, $S(p,0)$ restricted to $0 \oplus \Lambda_2 TM$ is strictly positive definite, as the sectional curvatures at $(p,0)$ are positive.

So the nullity of $S(p,0)$ is equal to $3$.
We can now apply the strong maximum principle due to Brendle (see \cite[sec 2]{Brendle:2008de}) and conclude that for all $(q,t) \in M \times (-\infty, 0]$ the nullity of $S(q,t)$ is $3$ and the nullspace $N_{q,t}$ of $S(q,t)$ forms a time-dependent subbundle in $E$ that is invariant under parallel transport with respect to $\td\nabla$ (in space) and $\td{D}_{\partial_t}$ (in time).

Next, observe that since all sectional curvatures on $M \times (-\infty, 0]$ are positive, the subbundle $0 \oplus \Lambda_2 TM \subset M$ intersects $N$ only in the origin over  every $(q, t) \in M \times (- \infty, 0]$.
So, at every time $t \leq 0$, the vector bundle $E$ is the direct sum of the subbundles $N_{\cdot,0}$ and $0 \oplus \Lambda_2 TM$.
It follows that there is a smooth, time-dependent, section $(F_t)_{t \in (-\infty, 0]}$ of the endomorphism bundle $\End ( TM, \Lambda_2 TM)$ such that for all $t \leq 0$
\[ N_{\cdot, t} = \{ (x, F_t (x)) \;\; : \;\; x \in TM \}. \] 
Let us now express the fact that $N$ is parallel with respect to $\td\nabla$ in terms of $F$, at some fixed time $t \leq 0$.
To do this, let $q \in M$ and $w \in T_q M$ and consider a locally defined vector field $X$  such that at $q$
\[ 0 = \td\nabla_z (X, F(X)) = \big( \nabla_z X, (\nabla_z F)(X) + F(\nabla_z X) + \Ric (z) \wedge X \big). \]
It follows that $\nabla_z X = 0$ and
\begin{equation} \label{eq_nab_F}
 (\nabla_z F)(X)  + \Ric (z) \wedge X = 0. 
\end{equation}
Let $A := \tr_{12} F$ be the trace of the first two factors of $F$ viewed as a section of $T^*M \otimes TM \otimes TM$.
Tracing (\ref{eq_nab_F}) yields
\begin{equation*}
 \nabla_z A - 2 \Ric(z)  = \nabla_z A + \sum_{i=1}^3 \big( \langle \Ric(z), e_i \rangle e_i - \langle e_i, e_i \rangle \Ric(z) \big)   = 0. 
\end{equation*}
So
\[ (\mathcal{L}_{A} g) (x, y) = \langle \nabla_x A, y \rangle + \langle x, \nabla_y A \rangle = 4 \Ric (x, y), \]
which implies that $(M, g_t)$ is a steady soliton.
As $\langle \nabla_x A, y \rangle = 2 \Ric (x, y)$ is symmetric in $x, y$, the vector field $A$ is a gradient vector field if $M$ is simply connected.

We can now apply Brendle's result (see \cite{Brendle:2013dd}) and conclude that the universal cover of $(M, g_t)$ is homothetic to $(M_{\Bry}, g_{\Bry})$ for all $t \leq 0$.
Since all isometries of $M_{\Bry}$ leave $x_{\Bry}$ invariant it follows that $(M, g_t)$ is homothetic to $(M_{\Bry}, g_{\Bry})$ for all $t \leq 0$.
So by uniqueness of Ricci flows with bounded curvature, the flow $(M, (g_t)_{t \in (- \infty, 0]})$ has to be homothetic to the Bryant soliton.
\end{proof}

\section{Smoothing maps}

\begin{lemma}[Smoothing bilipschitz maps between cylinders] \label{lem_smoothing_bilipschitz}
For every $\eps > 0$ there is a constant $\delta > 0$ such that the following holds.

Let $\phi : S^2 \times (0, 3) \to S^2 \times \R$ be a $(1+\delta)$-bilipschitz map, where both cylinders are considered to be round and of the same scale.
Then there is a $(1+\eps)$-bilipschitz map $\td\phi : S^2 \times (0,3) \to S^2 \times \R$ such that $\td\phi = \phi$ on $S^1 \times (0,1)$ and such that $\td\phi |_{S^2 \times (1,2)}$ is an isometry.
\end{lemma}

\begin{proof}
Let $\alpha > 0$ be a small constant whose value we will determine in the course of the proof, depending on $\eps$.
A limit argument implies that if
\begin{equation} \label{eq_delta_less_alpha_mollification}
 \delta \leq \ov\delta (\alpha),
\end{equation}
then there is an isometric embedding  $\psi : S^2 \times (.1,2.9) \to S^2 \times \R$ such that $d (\phi (x), \chi (x)) \lb < \lb \alpha$ for all $x \in S^2 \times (.1, 2.9)$.
After replacing $\phi$ with $\ov\chi^{-1} \circ \phi$, where $\ov\chi : S^2 \times \R \to S^2 \times \R$ is the isometric extension of $\chi$, we may assume without loss of generality that
\begin{equation} \label{eq_phi_close_to_isometry_alpha}
 d (\phi (x), x) < \alpha \qquad \text{for all} \qquad x \in S^2 \times (.1, 2.9). 
\end{equation}

Next, we will carry out a mollification procedure on $S^2 \times (0,3)$ producing a family of bilipschitz maps $\phi'_\beta : S^2 \times (0,2.9) \to S^2 \times \R$ such that $\phi'_\beta = \phi$ on $S^2 \times (0, 1)$ and such that $\phi'_\beta$ has improved regularity on $S^2 \times (1.5,2.9)$.  
This would be a completely standard mollification procedure, except for the fact that the scale of the mollification varies slowly.
For this purpose, we fix a smooth cutoff function $\rho : S^2 \times (0,3) \to [0,1]$, depending only on the $(0,3)$-factor, such that $\rho \equiv 0$ on $S^2 \times (0, 1)$ and $\rho \equiv 1$ on $S^2 \times (1.5,3)$.
Let moreover, $0 < \beta < .01$ be a constant whose value we will fix in the course of the proof.
The function $\beta \rho$ will determine the scale at which we mollify $\phi$.

Let $X:=S^2\times \R\subset\R^3\times\R= \R^4$ be the standard embedding.  Our mollification construction is similar to that in \cite{Karcher_mollification}.
However, in our case we can simplify the construction by using the embedding $X \subset \R^4$ and the nearest point projection $\proj_{X} : \R^4 \setminus (\{ (0,0,0 ) \} \times \R) \to X$.
Let $\psi : [0, \infty) \to [0,1]$ be a smooth cutoff function such that $\psi  \equiv 1$ on $[0, \frac12]$ and $\psi \equiv 0$ on $[1, \infty)$.
Set
\[ a := \int_{\R^3} \psi (|v|) dv. \]
Then we can define $\phi'_\beta : S^2 \times (0,2.8) \to S^2 \times \R$ as follows:
\begin{equation} \label{eq_phi_prime_beta_def}
 \phi'_\beta (x) := \proj_{X} \bigg( \int_{T_xX}  \phi \big( \exp_x (\beta \rho v) \big) a^{-1}\psi (|v|) dv \bigg). 
\end{equation}

\begin{claim*}

\mbox{}
\begin{enumerate}[label=(\alph*)]
\item \label{ass_D.1_cl_a} $\phi'_\beta$ is smooth.
\item \label{ass_D.1_cl_b} $\phi'_\beta \equiv \phi$ on $\{ \rho = 0 \}$ along with all higher derivatives.
\item \label{ass_D.1_cl_c} $d (\phi'_\beta (x), x) < 3(\alpha + 2 \beta)$, for all $x \in S^2 \times (.1, 2.9)$, assuming $\alpha, \delta < .01$.
\item \label{ass_D.1_cl_d} For any $\eps'> 0$, the following holds if $\beta \leq \ov\beta(\eps')$, $\delta \leq \ov\delta (\eps')$ and $\alpha \leq \ov\alpha (\eps', \beta)$.
The map $\phi'_\beta$ is $(1+\eps')$-bilipschitz and for all $x \in S^2 \times (1.5, 2.8)$ we have
\[ \big| (d\phi'_{\beta})_x - (d \id_{S^2 \times \R})_x \big| < \eps'. \]
Here we compare both differentials within the ambient space $\R^4$.
\end{enumerate}
\end{claim*}

\begin{proof}
Assertion \ref{ass_D.1_cl_a} follows from the definition of $\phi'_\beta$ and assertion \ref{ass_D.1_cl_b} holds since all derivatives of $\rho$ vanish on $\{ \rho = 0 \}$.

For assertion \ref{ass_D.1_cl_c} observe that $ \int_{T_xX}  \phi ( \exp_x (\beta \rho v) ) a^{-1}\psi (|v|) dv$ is the center of mass of a distribution that is supported on 
\[ \phi (B_X(x, \beta )) \subset B_{\R^4}(\phi(x), 2 \beta) \subset B_{\R^4}(x, \alpha + 2\beta ). \]
Due to the convexity of the latter ball, the center of mass must be contained in the same ball, and hence the nearest point projection lies in $B_{\R^4}(x,2(\alpha+2\beta))$.  Since $\beta<.01$ we have $B_{\R^4}(x,2(\alpha+2\beta))\cap X\subset B_X(x,3(\alpha+2\beta))$. 

We now prove assertion \ref{ass_D.1_cl_d} using a contradiction argument.
Assume that assertion \ref{ass_D.1_cl_d} was false for some fixed $\eps' > 0$.
Choose a sequence $\delta_i, \alpha_i, \beta_i \to 0$ such that $\alpha_i / \beta_i \to 0$.
Then we can find a sequence of $(1+\delta_i)$-bilipschitz maps satisfying (\ref{eq_phi_close_to_isometry_alpha}) for $\alpha = \alpha_i$ and points $x_i \in S^2 \times (.1, 2.8)$ such that one of the following holds:
\begin{enumerate}[label=(\Alph*)]
\item We have
\begin{equation} \label{eq_not_bilipschitz_eps_p}
 \big| |(d\phi'_{i, \beta_i})_{x_i} (v_i) | - 1 \big| \geq \eps' 
\end{equation}
 for some unit tangent vector $v_i$ at $x_i$.
 \item We have $x_i \in S^2 \times (1.5,2.9)$ and
\begin{equation} \label{eq_not_straight}
  \big| (d\phi'_{i, \beta_i})_{x_i} - (d \id_{S^2 \times \R})_{x_i} \big| \geq \eps'. 
\end{equation}
 \end{enumerate}
By assertion \ref{ass_D.1_cl_b} we have $\rho (x_i) > 0$ for large $i$ in case (A) and $\rho (x_i) = 1$ in the case (B), by the definition of $\rho$.
Moreover, by passing to a subsequence, we may assume that one of the above cases holds for all $i$.

Consider the rescaled metric $g_i := (\beta_i \rho(x_i))^{-2} g_{S^2 \times \R}$.
Then the sequences of pointed manifolds $\{ (S^2 \times (0,3), \lb g_i, \lb x_i)\}$ and $\{(S^2 \times (0,3), \lb g_i, \lb \phi_i (x_i))\}$ converge in the pointed smooth topology to pointed Euclidean space.
Moreover, with respect to the corresponding rescaling, the maps $\phi_i $ converge in the pointed topology to maps $\phi_\infty  : \R^3 \to \R^3$, after passing to a subsequence.
As the $\phi_i$ are $(1+\delta_i)$-bilipschitz, their limit $\phi_\infty$ must be a Euclidean isometry.
Furthermore, note that in case (B) we have $d_{g_i} ( \phi_i (y), y) \leq \alpha_i / \beta_i \to 0$, so in this case we even have $\phi_\infty = \id_{\R^3}$.

On the other hand, due to the mollification procedure (\ref{eq_phi_prime_beta_def}), the maps $\phi'_{i, \beta_i}$ converge smoothly to a map $\phi'_\infty$ with  
\[ \phi'_\infty (x) =  \int_{\R^3} a^{-1} \psi (|v|) \phi_\infty ( x +  v) dv. \]
This implies that $\phi'_\infty$ is also a Euclidean isometry, and in case (B) we even have $\phi'_\infty = \id_{\R^3}$.
This, however, contradicts the smooth convergence and (\ref{eq_not_bilipschitz_eps_p}) or (\ref{eq_not_straight}).
\end{proof}

We now apply a standard gluing procedure on $S^2 \times (1.5,2)$ to obtain a map $\td\phi : S^2 \times (0,3) \to S^2 \times \R$ that agrees with $\phi'_\beta$ on $S^2 \times (0,1.5)$ and with $\id_{S^2 \times \R}$ on $S^2 \times (2, 3)$.
In order to ensure that this map is $(1+\eps)$-Lipschitz, we use assertions \ref{ass_D.1_cl_c} and \ref{ass_D.1_cl_d} of the Claim and assume that 
\[ \alpha + 2 \beta \leq c (\eps), \qquad \eps' \leq c(\eps) \]
for some constant $c (\eps) > 0$.
We now verify that we can choose $\alpha, \beta, \delta$ such that these bounds, the conditions of assertions \ref{ass_D.1_cl_c} and \ref{ass_D.1_cl_d} of the Claim and (\ref{eq_phi_close_to_isometry_alpha}) hold.
Choose $\eps' \leq c (\eps)$ and then $\beta \leq \min \{ \ov\beta (\eps'), \frac14 c \}$, where $\ov\beta (\eps')$ denotes the upper bound from assertion \ref{ass_D.1_cl_d} of the Claim.
Next, choose $\alpha \leq \min \{ \ov\alpha (\eps', \beta), \frac14 c \}$, where $\ov\alpha (\eps', \beta)$ denotes the upper bound from assertion \ref{ass_D.1_cl_d} of the Claim.
Lastly, we choose $\delta \leq \ov\delta (\eps')$ according to assertion \ref{ass_D.1_cl_d} of the Claim and $\delta \leq \ov\delta(\alpha)$ according to (\ref{eq_delta_less_alpha_mollification}).
\end{proof}

\bibliography{uniqueness}{}
\bibliographystyle{amsalpha}

\end{document}